\documentclass{book}

\usepackage[utf8]{inputenc}
\usepackage[T1]{fontenc}
\usepackage{lmodern}
\usepackage[french]{babel}
\usepackage{amsmath}
\usepackage{amsthm}
\usepackage{amssymb}
\usepackage{units}
\usepackage{amscd}
\usepackage{geometry}
\usepackage{graphicx}
\usepackage{tikz}
\usepackage{tikz-cd}
\usepackage{pst-plot}
\usepackage{enumitem}
\usetikzlibrary{arrows,automata}
\usepackage{graphicx}
\usepackage{url}
\usepackage{pgfplots}
\usepackage{hyperref}
\usepackage{titlesec}

{\newtheorem{theointro}{Théorème}

\newtheorem{conjectureintro}{Conjecture}}
{\newtheorem{theo}{Théorème}[section]}
{\newtheorem{corollaire}[theo]{Corollaire}}
{}
{\newtheorem*{problem2}{Problème}}
{\newtheorem{conjecture}[theo]{Conjecture}}
{\newtheorem{question}[theo]{Question}}
{\theoremstyle{definition} \newtheorem{defin}[theo]{Définition}
							\newtheorem{prop}[theo]{Proposition}
							\newtheorem{lemme}[theo]{Lemme}}
{\theoremstyle{remark} \newtheorem{remarque}[theo]{Remarque}
						\newtheorem{exemple}[theo]{Exemple}}
{\theoremstyle{proof}\newtheorem{reduction}{Réduction}}

	\newcommand{\Id}{Id}	
\newcommand{\Set}{\text{Set}}
\newcommand{\op}{op}

\newcommand{\sS}{\text{sSet}}
\newcommand{\Hom}{\text{Hom}}

\newcommand{\Top}{\text{Top}}
\newcommand{\Real}[1]{{}||#1||}
\newcommand{\RealP}[1]{{}\Real{#1}_P}
\newcommand{\Sing}{\text{Sing}}
\newcommand{\Sk}{Sk}
\newcommand{\pr}{pr}
\newcommand{\colim}{\operatornamewithlimits{colim}}
\newcommand{\Ex}{\text{Ex}}
\newcommand{\sd}{\text{sd}}
\newcommand{\Exi}{\text{Ex}^{\infty}}
\newcommand{\lv}{\text{l.v}}
\newcommand{\Map}{\text{Map}}
\newcommand{\Fun}{\text{Fun}}

\newcommand{\nd}{\text{n.d.}}

\newcommand{\holim}{\text{holim}}

\newcommand{\ev}{\text{ev}}

\newcommand{\Hol}{\text{Holink}}

\newcommand{\OSk}{\text{OSk}}

\newcommand{\Z}{\mathbb{Z}}
\newcommand{\R}{\mathbb{R}}

\newcommand{\fil}[1]{(#1,\varphi_{#1})}
\newcommand{\A}{\mathcal{A}}
\newcommand{\B}{\mathcal{B}}
\newcommand{\C}{\mathcal{C}}
\newcommand{\D}{\mathcal{D}}
\newcommand{\Sat}{\text{Sat}}
\newcommand{\sk}{\text{Sk}}
\newcommand{\N}{\mathbb{N}}
\newcommand{\Diag}{\text{Diag}}
\newcommand{\Colim}{\text{Colim}}
\newcommand{\An}{\text{An}}
\newcommand{\diag}{\text{diag}}
\newcommand{\HOM}{\underline{\Hom}}
\newcommand{\bisS}{\text{bi}{-}\sS}
\newcommand{\coker}{\text{coker}}
\newcommand{\DiagR}{\Diag^{\text{réd}}}
\newcommand{\RealNP}[1]{\Real{#1}_{N(P)}}
\newcommand{\Int}{\text{Int}}
\newcommand{\Coprod}{\operatornamewithlimits{\coprod}}
\newcommand{\tr}{\text{tr}}
\newcommand{\sSU}{\sS^{\mathcal{U}}}
\newcommand{\sSTop}{\sS^{\Top}}
\newcommand{\cof}{\text{cof}}
\newcommand{\Ho}{\text{Ho}}
\newcommand{\U}{\mathcal{U}}
\newcommand{\Strat}{\text{Strat}}
\newcommand{\FStrat}{\text{FStrat}}
\newcommand{\sStrat}{\text{sStrat}}
\newcommand{\E}{\mathcal{E}}
\newcommand{\filstrat}[2]{(#1,#2,\varphi_{#1})}
\newcommand{\strat}[1]{\filstrat{#1}{P_{#1}}}
\newcommand{\Poset}{\text{EnsOrd}}
\newcommand{\RealFStrat}[1]{\Real{#1}_{\FStrat}}
\newcommand{\sSJoyal}{\sS^{\text{Joyal}}}
\newcommand{\sSJK}{\sS^{\text{Joyal-Kan}}}
\newcommand{\Pvar}{\text{PVar}}
\newcommand{\codim}{\text{codim}}
\newcommand{\naif}{\text{Naïf}}

\pgfplotsset{compat=1.14}

\makeatletter
\def\thickhrulefill{\leavevmode \leaders \hrule height 1ex \hfill \kern \z@}
\def\@makechapterhead#1{%
  {\parindent \z@ \raggedright
    \reset@font
    \hrule
    \vspace*{10\p@}%
    \par
    \Large \scshape \@chapapp{} \Huge\bfseries \thechapter
    \par\nobreak
    \vspace*{10\p@}%
    \par
    \hrule
    \vspace*{20\p@}
    \Huge \bfseries #1\par\nobreak
    \vskip 40\p@
  }}

\titleformat{\section}
  {\normalfont\Large\bfseries}{\thesection}{1em}{}[{\titlerule[0.5pt]}]

\def\sectionEtoile#1{%
  \par\bigskip\bigskip
  \par\nobreak\noindent
  \refstepcounter{section}%
  \addcontentsline{toc}{section}{#1}%
  \reset@font 
  { \Large \bfseries
    #1}%
    \vspace*{7\p@}%
    \hrule 
    \vspace*{5\p@}%
  \par
  \medskip
}

\makeatother

\begin{document}
\begin{titlepage}



\vskip-3.5cm

\begin{center}
\includegraphics[width=2cm]{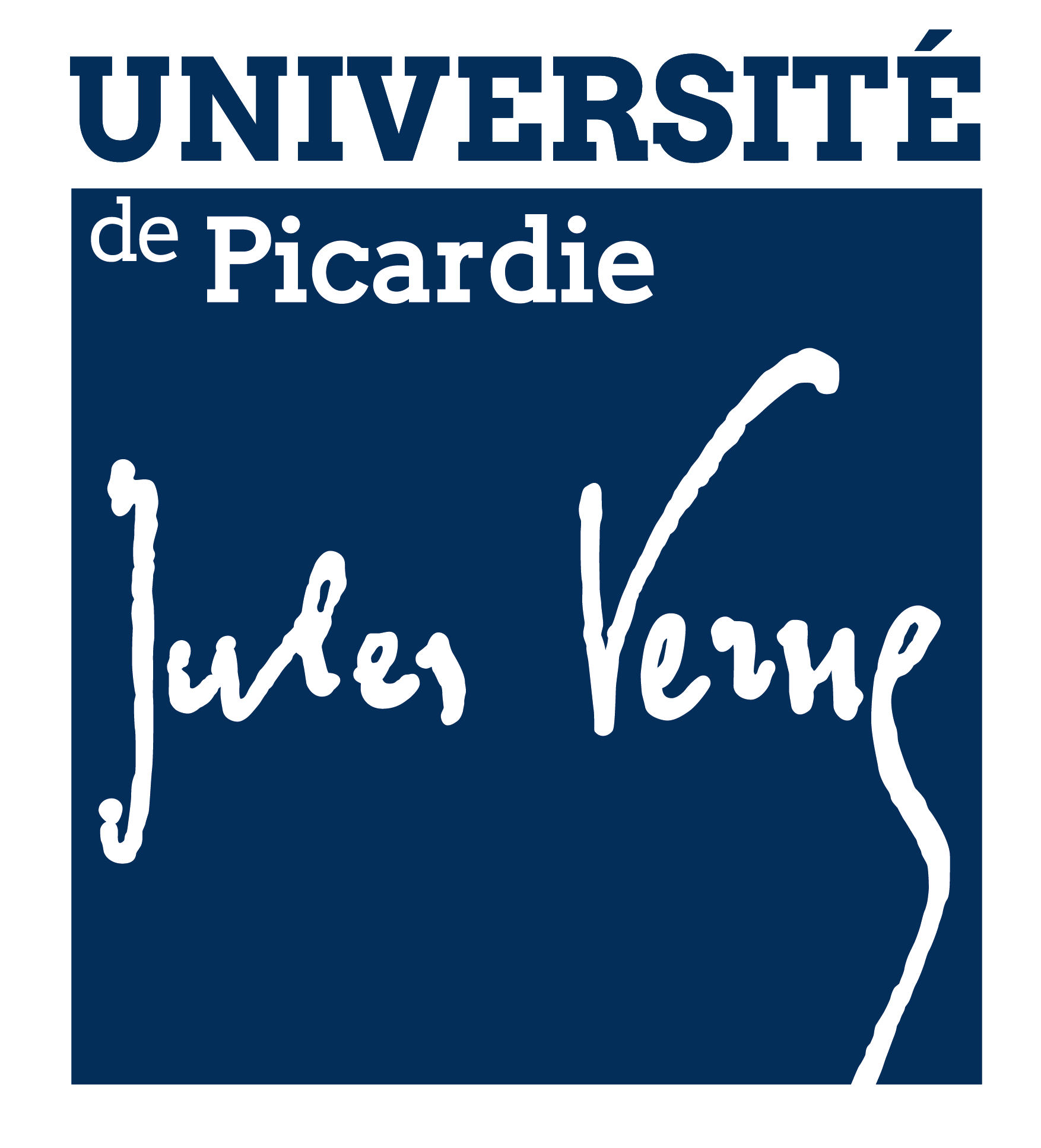}  
\end{center}

\centerline{\Large \bfseries Thèse de Doctorat}
\vskip 0.8cm
\centerline{ \it \bfseries Mention: \bfseries Math\'ematiques}
\vskip 0.8cm
\centerline{ présentée à  \it \bfseries l'Ecole Doctorale en Sciences, Technologie, Santé (ED 585)}
\vskip 0.8cm
\centerline{ \large \bf de l'Université de Picardie Jules Verne}
\vskip 0.8cm
\centerline{ par }
\vskip 0.8cm
\centerline{\Large \bf Sylvain Douteau}
\vskip 0.8cm
\centerline{ pour obtenir le grade de Docteur de l'Université de Picardie Jules Verne}
\vskip 0.8cm
\begin{center}
\fbox{
  \centerline{  \Large \it \bfseries  Étude homotopique des espaces stratifiés}
}
\end{center}
\vskip 0.5cm
\centerline{\bf Soutenue le 4 juillet 2019, apr\`es avis des rapporteurs, devant le jury d'examen}
  \vskip 0.5 cm
\hskip-2cm {\small 
\begin{center}
\begin{tabular}{lll}
M$^{\rm me}$ Kathryn Hess & Professeur, EPFL & Rapporteur \\
M. Bertrand Toën	&  DR, Univ. de Toulouse & Rapporteur \\
M. Serge Bouc &	DR, UPJV  & Examinateur \\
M$^{\rm me}$ Muriel Livernet \hspace{0.5cm} & Professeur, Univ. Paris Diderot\hspace{2cm} & Examinatrice \\
M. Ivan Marin  & Professeur, UPJV   & Examinateur \\
M. Geoffrey Powell  & DR, Univ. d'Angers & Examinateur \\
M. Jon Woolf  & Professor, Univ. of Liverpool  & Examinateur \\
M. David Chataur  & Professeur, UPJV & Directeur de thèse \\
\end{tabular}
\end{center}
}

\vskip 0.5cm
\hskip 0.5cm  
\hfill\includegraphics[width=2.5cm]{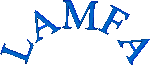}\hfill
\thispagestyle{empty}
\end{titlepage}
\newpage
~
\thispagestyle{empty}
\clearpage
\newpage
\pagenumbering{roman}
\section*{Remerciements}

Je tiens avant toute chose à remercier David - la coutume veut que l'on commence par remercier son directeur de thèse, et le mien s'est efforcé de me transmettre les bonnes pratiques du mathématicien. Cette thèse a été une expérience profondément enrichissante, en grande partie grâce à nos échanges. Je te remercie pour ta disponibilité, la patience dont tu as fait preuve face à mon entêtement et la confiance que tu m'as accordée tout au long de cette thèse. Elle a été un moteur puissant. Merci aussi à toi et à Stéphanie de m'avoir accueilli chez vous à plusieurs reprises. Tu ne seras bientôt plus mon directeur, mais j'espère que tu resteras un ami et un collaborateur.
\\~\\

Je tiens aussi à remercier mes rapporteurs, Kathryn Hess et Bertrand Toën. Chacun sait le travail que représente la relecture minutieuse d'une thèse, et je les remercie de m'avoir fait l'honneur d'accepter ce rôle.
I would also like to thank the examiners Muriel Livernet, Geoffrey Powell, Jon Woolf, Serge Bouc, and Ivan Marin for accepting to be part of the comittee. Jon, your talk at the fields institute sparked my interest for stratified homotopy theory, and the conversations I had with you and Stephen were all enlightening. Serge, tu as assisté à presque tous les exposés que j'ai eu l'occasion de donner à Amiens, merci d'avoir accepté de m'écouter parler une fois de plus.
Ivan, merci d'avoir accepté cette responsabilité supplémentaire.
\\~\\

Merci aussi à tous les membres du LAMFA. Merci aux habitués de la bistouille, et à ceux de la place du don, pour les moments de partages informels, parfois même mathématiques.

Merci aussi à Christelle et Isabelle. Vous m'avez l'une et l'autre plus d'une fois aidé dans des démarches que je ne comprenais pas, et ainsi rendu ma vie de doctorant beaucoup plus facile.

Merci aux gens avec qui j'ai partagé des mathématiques. C'est en donnant des exposés que j'ai appris ce qu'étaient les ensembles simpliciaux et les catégories modèles, et je suis redevable envers tous ceux qui m'ont donné cette opportunité d'apprendre.

Merci en particulier à Ivan, pour m'avoir appris à voir au delà des statuts.
J'espère qu'on continuera à profiter de la place du Don, en débattant des avantages relatifs des complexes simpliciaux par rapport aux ensembles simpliciaux, et que tu finiras par retenir le nombre de satellites géostationnaires.

Merci à Radu pour m'avoir aidé à trouver un tableau pour ma soutenance, et pour sa grande disponibilité en général.

Merci aussi à tous les doctorants du LAMFA. L'ambiance en BC 01 a été un des plaisirs de ma thèse, et plus d'une fois c'est la perspective du tarot de la mi-journée qui m'a motivé à venir jusqu'à la fac. Merci à tout ceux qui ont accepté le rôle du canard en plastique, c'est quand même vraiment plus efficace quand il répond.
I would also like to thank in particular those of you who will not be able to read the last sentence because you don't speak french. Thanks to you, my English is a bit better than it would have been.
Merci en particulier à Alexandre et Anne-Sophie, pour les déjeuner et les soirées en trio. Je n'ai même pas eu le temps de me sentir seul en arrivant à Amiens. Merci Alexandre d'avoir supporté mes flots de paroles ininterrompus, et d'avoir toujours été partant pour aller boire une bière.
\\~\\
\newpage
Merci à Louis, grâce à qui je ne me sentirai seul nul part. Merci à Joran, Axelle et Oscar dont le canapé est devenu une étape obligée en retour de conférence. Merci à Sevan, pour les sessions jeux hors du temps. Merci à Matthieu pour les discussions philosophiques au bout desquels on arrive parfois à se mettre d'accord sur une définition.
Merci aussi à la bande de la K-Fêt, et à tout ceux que je suis toujours heureux de recroiser, à Fermanville ou lors de mes passages à Paris.
\\~\\

Merci à toute ma famille. Collectivement, d'abord, parce c'est une tribu qui m'a entouré, pendant ma thèse et depuis toujours.

Merci à mon père pour son pragmatisme contagieux, et à ma mère pour son affection débordante et ses efforts pour la contrôler. Vous m'avez poussé à faire ce que j'aimais même lorsque vous ne "compreniez pas tout", et vous avez fait en sorte que je me sente toujours soutenu. Merci aussi de m'avoir aidé à organiser le pot, pour le plus grand plaisir de tous.

Merci à Vincent et Laura pour les apéros à rallonge, les soirées jeux et les suggestions séries pertinentes. C'est toujours un moment de vacances de passer vous voir. 

Merci à Yveline et Sebastien pour leur relecture de ma thèse. J'espère que vous ne m'en voudrez pas pour les erreurs restantes. Merci aussi de m'avoir écouté parler de maths, pâte à modeler en mains. Vous pourriez presque faire mes exposés à ma place.

Merci à Magali pour les café-phones et à Pascal pour les dégustations de Whisky. Deux sources de réconfort liquide qui ont toujours su tomber à point nommé. Comment ne pas mentionner aussi les talents de Magali pour la tartiflette exotique.

Merci à Corinne de toujours me payer un café lorsque je vais la voir.  

Merci aussi à Juliette, Marc, Mathieu, Arthur, Baptiste, Elsa, Marie, Émilie, Valentine et Sarah de toujours m'accueillir comme si je n'étais jamais parti.
\\~\\

Merci enfin à Anne-Sophie. Nous sommes arrivés à Amiens en même temps et nous étions voisins de bureaux. Aujourd'hui nous partageons un tableau. Cette thèse, c'est aussi notre rencontre, dont ce texte porte des traces. Je suis heureux qu'on ait pu partager ce chapitre de nos vies, et j'ai hâte d'écrire la suite avec toi.
\clearpage
\newpage
~
\clearpage
\newpage

\tableofcontents 
\chapter*{Introduction}
\chaptermark{Introduction}
\pagenumbering{arabic}
\addcontentsline{toc}{chapter}{Introduction}
\section*{Espaces singuliers et stratifications}
\phantomsection
\addcontentsline{toc}{section}{Espaces singuliers et stratifications}
L'étude des variétés, de leurs symétries et de leurs relations constitue l'essence de la géométrie. Très souvent, cette étude est
restreinte à des objets "réguliers", ou à des situations "génériques", cependant cette restriction au contexte régulier n'est en rien naturelle. Il suffit en effet de manipulations élémentaires pour sortir du cadre régulier. 

Par exemple, considérons une sphère privée de ces deux pôles. 
Il s'agit d'une variété différentiable ouverte. Son compactifié d'Alexandroff est obtenu en identifiant les deux pôles.
On obtient ainsi le tore pincé qui est un objet presque régulier : il possède un point singulier au niveau du pincement, mais c'est une variété en dehors de ce point. C'est cette classe d'objets que l'on étudie dans ce texte.

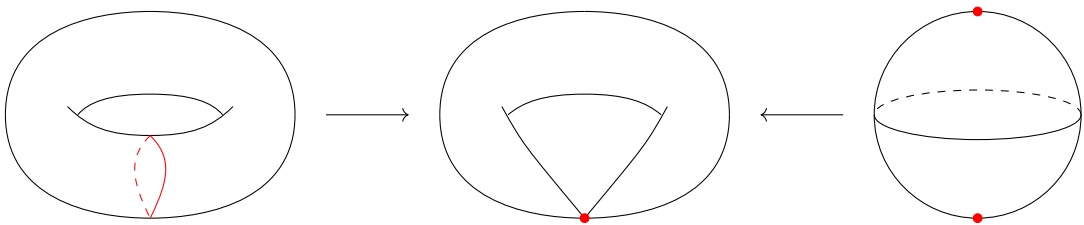
\begin{figure}\label{TorePince}
\begin{tikzpicture}[scale=0.55]
\draw (-3.5,0) .. controls (-3.5,2) and (-1.5,2.5) .. (0,2.5);
\draw[xscale=-1] (-3.5,0) .. controls (-3.5,2) and (-1.5,2.5) .. (0,2.5);
\draw[rotate=180] (-3.5,0) .. controls (-3.5,2) and (-1.5,2.5) .. (0,2.5);
\draw[yscale=-1] (-3.5,0) .. controls (-3.5,2) and (-1.5,2.5) .. (0,2.5);

\draw (-2,.2) .. controls (-1.5,-0.3) and (-1,-0.5) .. (0,-.5) .. controls (1,-0.5) and (1.5,-0.3) .. (2,0.2);

\draw (-1.75,0) .. controls (-1.5,0.3) and (-1,0.5) .. (0,.5) .. controls (1,0.5) and (1.5,0.3) .. (1.75,0);
\draw[->] (4.25,0)-- (6.25,0);

\draw[red](0,-0.5)..controls (0.5,-1) and (0.5,-1.5)..(0,-2.5);
\draw[red, dashed](0,-0.5)..controls (-0.5,-1) and (-0.5,-1.5)..(0,-2.5);

\draw[shift= {(10.5 cm, 0 cm)}] (-3.5,0) .. controls (-3.5,2) and (-1.5,2.5) .. (0,2.5);
\draw[shift= {(10.5 cm, 0 cm)}, xscale=-1] (-3.5,0) .. controls (-3.5,2) and (-1.5,2.5) .. (0,2.5);
\draw[shift= {(10.5 cm, 0 cm)},rotate=180] (-3.5,0) .. controls (-3.5,2) and (-1.5,2.5) .. (0,2.5);
\draw[shift= {(10.5 cm, 0 cm)},yscale=-1] (-3.5,0) .. controls (-3.5,2) and (-1.5,2.5) .. (0,2.5);

\draw[shift= {(10.5 cm, 0 cm)}] (-2,.2) .. controls (-1.5,-0.7) .. (0,-2.5);
\draw[shift= {(10.5 cm, 0 cm)}] (2,.2) .. controls (1.5,-0.7) .. (0,-2.5);
\draw[shift= {(10.5 cm, 0 cm)}] (-1.85,0) .. controls (-1.5,0.3) and (-1,0.5) .. (0,.5) .. controls (1,0.5) and (1.5,0.3) .. (1.85,0);
\filldraw[shift= {(10.5 cm, 0 cm)},red] (0,-2.5) circle (3pt);

\draw[<-] (14.75,0)--(16.75,0);

  \draw[shift= {(20 cm, 0cm)}] (0,0) circle (2.5);
  \draw[shift= {(20 cm, 0cm)}] (-2.5,0) arc (180:360:2.5 and 0.6);
  \draw[shift= {(20 cm, 0cm)},dashed] (2.5,0) arc (0:180:2.5 and 0.6);
  \filldraw[shift= {(20 cm, 0 cm)},red] (0,-2.5) circle (3pt);
  \filldraw[shift= {(20 cm, 0 cm)},red] (0,2.5) circle (3pt);

\end{tikzpicture}
\caption{Le tore pincé, comme compactifié de la sphère épointée et comme contraction du tore le long d'un méridien.} \end{figure}
Les espaces singuliers apparaissent naturellement lors de l'étude des espaces réguliers, leur ajout au cadre géométrique est donc crucial.

La géométrie algébrique est féconde en exemples d'espaces singuliers. Dans ce domaine, on étudie les solutions de systèmes d'équations polynomiales à coefficients dans un anneau. Les solutions d'un tel système s'incarnent géométriquement comme un sous-espace d'un espace affine - ou projectif quand les équations sont homogènes.
Et si l'anneau de base est le corps des nombres réels ou complexes, ces sous-espaces portent une topologie induite par celle de l'espace ambiant. On est dans une situation fascinante où sont subtilement liées les propriétés algébriques des idéaux définissant les équations d'une part, et les propriétés topologiques et géométriques de l'espace des solutions d'autre part.
Lorsque le système satisfait une certaine condition de régularité - si la matrice Jacobienne du système est de rang maximal - l'espace des solutions est une variété algébrique lisse, et une variété différentiable. Cette condition est générique, et quand elle n'est pas vérifiée, on obtient des variétés algébriques singulières. Par exemple, dans le plan projectif complexe $\mathbb{C}P^2$, la courbe algébrique plane donnée par l'équation 
\begin{equation*}
y^3-x^3-x^2z=0
\end{equation*}
est singulière en le point $[0:0:1]$. Cette cubique est par ailleurs homéomorphe au tore pincé que nous avons déjà rencontré plus tôt. 

Considérons un système d'équations polynomiales. Il définit une certaine variété algébrique $V$. Le lieu des points où la matrice Jacobienne n'est pas de rang maximal définit la sous variété algébrique de $V$ des points singuliers, notée $\Sing(V)$. Celle-ci vérifie $\dim(\Sing(V))<\dim(V)$. Le complémentaire des points singuliers $V\setminus\Sing(V)$ est une variété lisse. En itérant ce procédé, on obtient une filtration finie de la variété $V$ :
\begin{equation*}
\dots\subset\Sing(\Sing(V))\subset \Sing(V)\subset V
\end{equation*}
En raffinant cette filtration, H. Whitney \cite{WhitneyStratification} obtient une filtration de $V$ :
\begin{equation*}
V_0\subseteq V_1\subseteq\dots\subseteq V.
\end{equation*}

Cette filtration vérifie les propriétés suivantes.
\begin{itemize}
\item Chaque sous espace $V_i\setminus V_{i-1}$ est une union disjointe de sous-variétés algébriques lisses de $V$, les \textbf{strates}.
\item Ces strates satisfont des conditions de recollements.
\end{itemize}
Un tel découpage est une \textbf{stratification} de $V$.
Puisque les strates sont régulières on peut les étudier comme des variétés lisses et en déduire des propriétés globales grâce aux conditions de recollement.

On peut abstraire le procédé précédent pour l'appliquer à d'autres contextes, en conservant le découpage en strates régulières. Les conditions de recollements deviennent des conditions de trivialité locale (voir aussi \cite{ThomEnsemblesMorphismesStratifies}). C'est la notion de \textbf{pseudo-variété}.

Ces objets dont les singularités sont contrôlées apparaissent naturellement lors de manipulations géométriques élémentaires sur les variétés. On en présente quelques exemples ici.

\begin{itemize}
\item En géométrie, on considère souvent des familles de variétés. Si une telle famille dépend d'un paramètre, il est naturel de déterminer les limites géométriques au bord de l'espace des paramètres. Considérons par exemple la famille de cubiques projectives complexes :
\begin{equation*}
y^3-x^3-x^2z -\epsilon(x^2z+xz^2)=0,
\end{equation*}
où $\epsilon$ varie dans $\mathbb{C}^*$. Alors pour toute valeur de $\epsilon\not=0$, l'espace des solutions est lisse. Mais comme on l'a vu précédemment, quand $\epsilon=0$, on obtient une cubique singulière, le tore pincé.
\item  Le quotient d'une variété lisse $M$ par un groupe de symétries $G$ n'a aucune raison d'être régulier dès que l'action n'est pas libre. Cependant, si $G$ est un groupe de Lie compact qui agit sur une variété différentiable $M$, $M/G$ est une pseudo-variété. Une stratification adaptée est obtenue en découpant $M/G$ selon les stabilisateurs de l'action de $G$.
\item Un autre procédé standard en géométrie est celui d'éclatement d'une variété suivant une sous-variété. Le procédé inverse est celui de contraction. Si $M$ est une variété différentiable et $N\subseteq M$ est une sous variété différentiable, l'espace quotient $M/N$ - l'espace obtenu en contractant $N$ en un seul point - est généralement singulier. De nouveau, il s'agit d'une pseudo-variété. Elle possède deux strates : un point singulier, correspondant à l'image de $N$ dans le quotient $M/N$, et le complémentaire de ce point singulier, homéomorphe à $M\setminus N$. Notre fil conducteur, le tore pincé, est un exemple d'une telle contraction. C'est le résultat de la contraction de deux pôles opposés sur la sphère. On peut aussi l'obtenir comme la contraction d'un tore le long d'un cercle méridien.
\end{itemize}
\section*{De l'usage des espaces singuliers en géométrie et en topologie}
\phantomsection
\addcontentsline{toc}{section}{De l'usage des espaces singuliers en géométrie et en topologie}

Ainsi, les espaces singuliers sont incontournables - même lorsqu'on s'attache à étudier les variétés - et la notion de stratification apparait comme un moyen de ramener l'étude des objets singuliers à des considérations sur les strates, qui, elles, restent régulières. Mais les pseudo-variétés apparaissent aussi comme la généralisation appropriée pour répondre à des questions naturelles sur les variétés. Illustrons ceci par quelques exemples.
\begin{itemize}
\item 
Le produit d'intersection en homologie, comme défini par H. Poincaré en 1895 \cite{AnalysisSitus} est un outil extrêmement puissant. Il permet de construire des invariants d'une grande portée, tels que la signature et la forme d'entrelacement, et sa définition en termes d'intersections de sous-variétés triangulées permet de le calculer de façon élémentaire.
Cependant, pour pouvoir calculer le produit d'intersection de deux classes d'homologie de cette façon, celles-ci doivent pouvoir être représentées par des sous-variétés triangulées. S.Lefschetz \cite{Lefschetz} contourne ce problème en définissant le produit d'intersection de certains sous-objets, qu'il appelle circuits. Ces derniers sont en fait des pseudo-variétés triangulées, et permettent de représenter toutes les classes d'homologies d'une variété triangulée. 
Néanmoins, N. Steenrod pose la question suivante. 

Est-ce que toutes les classes d'homologies d'une variété triangulée sont réalisables par des sous-variétés triangulées? Autrement dit, peut-on se débarrasser des pseudo-variétés en les désingularisant?

R. Thom répond négativement à cette question \cite{ThomProprietesGlobalesVarietesDifferentiables}. Si $M$ est une variété triangulée et $\alpha\in H_n(M,\mathbb{Z})$ est une classe d'homologie alors il existe toujours un multiple de $\alpha$ qui est réalisable par une sous-variété. Cependant, la classe $\alpha$ elle-même n'est pas toujours réalisable par une sous-variété. Il est à noter qu'une approche de l'homologie singulière purement géométrique, en terme de pseudo-variétés, a été développée par M. Kreck \cite{Kreck}\cite{KreckBook}.
\item La signature d'une variété topologique est un invariant central en topologie géométrique. Dans les années 1960-1970, à la recherche d'invariants complets pour classifier les variétés topologiques à homéomorphismes près, J. Morgan et D. Sullivan ont étendu la signature à une classe d'objets singuliers appelés $\Z/n$-variétés. L'enjeu de ces travaux étaient d'étudier les propriétés arithmétiques des L-classes, des invariants intimement liés à la signature (voir \cite{MorganSullivan}).

L'extension de la signature à des objets singuliers a conduit D. Sullivan à conjecturer l'existence d'une théorie homologique adaptée aux espaces singuliers. Cette théorie devait permettre  d'obtenir la dualité de Poincaré pour les pseudo-variétés, et devait fournir un contexte dans lequel interpréter la signature. 

C'est ce programme qui a été réalisé par M. Goresky et  R. MacPherson, lorsqu'ils ont construit la cohomologie d'intersection des pseudo-variétés et montré un théorème de dualité de Poincaré (voir \cite{IntersectionHomologyI}). Ceci leur a permis d'étendre la signature à une large classe de pseudo-variétés. P. Siegel \cite{SiegelWittSpaces} a ensuite poursuivi ce programme en montrant que la signature des espaces de Witt - une sous classe des pseudo-variétés - est un invariant de bordisme.  
Ce qui lui a permis notamment d'obtenir une
démonstration élégante du théorème de Novikov d'additivité de la signature.

Les travaux de P. Siegel ont aussi permis de définir de nouvelles théories cohomologiques généralisées telles que les théories de bordismes d'espaces de Witt, d'IP-bordisme et de L-bordisme (voir par exemple \cite{BanaglMcClure}).
\item La cohomologie des variétés Grassmannienne est le point de départ pour la définition des classes caractéristiques de fibrés vectoriels \cite{CharacteristicClasses}. Dans le cas complexe, le calcul de sa structure d'algèbre peut se mener de façon élégante à l'aide de la dualité de Poincaré en calculant le produit d'intersection en homologie. Une base naturelle de l'homologie singulière des Grassmanniennes complexes est donnée par la décomposition cellulaire de Schubert. Chacune de ces cellules a pour adhérence une sous-variété algébrique, potentiellement singulière. Le calcul du produit d'intersection se fait en utilsant ces cycles, à la manière de S. Lefschetz. En fait, ce calcul s'effectue dans l'anneau de Chow de $G(k,n)$ (voir \cite{3264}).
\end{itemize}
Les exemples précédents illustrent l'utilité d'étendre des théories aux objets singuliers. Les exemples suivants montrent que la notion de stratification peut s'avérer utile même en l'absence de singularités.
\begin{itemize}
\item En théorie de Morse, les propriétés d'une variété lisse $M$ sont étudiées à travers une application lisse
\begin{equation*}
f\colon M\to \R.
\end{equation*}
Si l'application $f$ est de Morse Smale, elle permet de définir une stratification de $M$. Les strates sont obtenues en considérant le flot du gradient de $f$. La stratification ainsi obtenue vérifie les conditions de Whitney, c'est à dire qu'elle fait de $M$ une pseudo-variété. Cette stratification est un outil pour calculer l'homologie de $M$ (voir \cite{Nicolaescu}).
\item La notion de stabilité différentiable a été introduite et étudiée par R. Thom et J. Mather. Si $M$ et $N$ sont deux variétés différentiables et $f\colon M\to N$ est une application lisse entre elles, on dit que $f$ est différentiellement stable si il existe $U$, un voisinage de $f$ dans $C^{\infty}(M,N)$ tel que pour tout $g\in U$, on peut trouver deux difféomorphismes $\phi, \psi$ tels que
\begin{equation*}
f=\phi\circ g\circ \psi.
\end{equation*}
Une conséquence de la théorie de Morse est que le sous-ensemble des applications différentiellement 
stables est ouvert et dense
quand $\dim(N)=1$, et H. Whitney a démontré le même résultat dans le cas $\dim(M)=\dim(N)=2$ \cite{WhitneyMappingsEuclideanSpaces}. Cependant, R. Thom a fourni un 
contre-exemple lorsque $\dim(M)=\dim(N)=16$ (voir \cite{Levine}). Il est naturel de se demander sous quelles conditions ce résultat est vrai. J. Mather a montré que la validité de ce résultat ne dépendait que des dimensions de $M$ et de $N$ et que le résultat était vrai en toute dimension si on suppose seulement que $\phi$ et $\psi$ sont des homéomorphismes (On renvoie à l'article introductif \cite{GoreskyIntro} pour plus de détails sur le sujet). La résolution de ce problème passe par une stratification de l'espace $C^{\infty}(M,N)$. Dans le cas particulier où $N=\R$, la strate régulière correspond aux fonctions de Morse. 
En travaillant sur cette stratification, J. Cerf a montré des résultats profonds sur la topologie des espaces de difféomorphismes des variétés différentiables \cite{Cerf}.
\end{itemize}
\pagebreak
\section*{À propos de la cohomologie d'intersection}
\phantomsection
\addcontentsline{toc}{section}{A propos de la cohomologie d'intersection}

Les espaces stratifiés apparaissent d'abord comme un outil. Leur ubiquité nous amène naturellement à rechercher une classification. La cohomologie d'intersection apparait alors comme un invariant approprié. Avant tout, c'est l'invariant qui permet d'étendre aux pseudo-variétés les propriétés cohomologiques des variétés. Elle est, par construction, dépendante de la stratification - et donc contient de l'information par rapport à cette structure supplémentaire - mais dans le cas des pseudo-variétés, c'est un invariant topologique. Au moins pour cette sous-classe d'objets, c'est un analogue stratifié de la cohomologie singulière. 

Algébriquement, la cohomologie singulière peut être décrite comme la cohomologie d'un faisceau localement constant (dans le cas de coefficients constants, il s'agit du faisceau constant). On sait que de tels faisceaux correspondent aux représentations du groupe fondamental. 
La cohomologie d'intersection, elle, est obtenue comme la cohomologie d'un faisceau constructible. On peut alors se demander quel est le bon analogue du groupe fondamental dans ce contexte. D. Treumann \cite{Treumann} a montré que c'est la catégorie des chemins sortants qui joue ce rôle pour les pseudo-variétés (le résultat montré par D. Treumann est en fait une généralisation 2-catégorique d'un résultat non publié de R. MacPherson), et J. Woolf a étendu ce résultat aux espaces homotopiquement stratifiés \cite{WoolfFundamentalCategory}. Depuis, l'étude de cette catégorie, ainsi que de ces généralisations supérieures, est devenue un volet important de la théorie des espaces stratifiés (voir par exemple \cite[Appendix A]{HigherAlgebra}). Ces travaux visent à comprendre la théorie des faisceaux pervers de façon catégorique, voire homotopique.

Géométriquement, tout comme la cohomologie singulière, la cohomologie d'intersection peut être caractérisée par une courte liste de propriétés \cite{King}.
\begin{itemize}
\item Elle coïncide avec la cohomologie singulière sur les variétés topologiques.
\item Elle vérifie la propriété de Mayer-Vietoris. C'est à dire que étant donnée une décomposition d'un espace stratifié en deux ouverts $X=U\cup V$, il existe une suite exacte longue reliant les cohomologies d'intersections de $X$, $U$, $V$ et $U\cap V$.
\item Elle vérifie la formule du cône. Si $X$ est un espace stratifié compact, de dimension $n$, son cône ouvert $c(X)$ est l'espace stratifié de dimension $n+1$ obtenu comme le quotient 
\begin{equation*}
c(X)=\frac{X\times [0,1[}{X\times \{0\}}.
\end{equation*}
Pour toute perversité $p$ (une fonction annexe faisant partie de la définition de la cohomologie d'intersection), on a un isomorphisme
\begin{equation*}
I_pH^k(c(X),\R)=\left\{\begin{array}{cl}
0 &\text{ si $k\geq n-p$}\\
I_pH^k(X,\R) &\text{ si $k\leq n-p$}
\end{array}
\right.
\end{equation*}
\item Elle est naturelle par rapport aux applications continues qui respectent la stratification, et invariante par homotopie stratifiée. Ici, on dit qu'une application $f\colon X\to Y$ respecte la stratification si $X$ et $Y$ sont stratifiés et si pour tout strate $S\subseteq Y$, la pré-image $f^{-1}(S)$ est une union de strates de $X$. Une homotopie stratifiée est une homotopie qui est aussi une application stratifiée.
\end{itemize}
En exploitant le formalisme géométrique de H. King, G. Friedman et J. McClure \cite{FriedmanMcClure} étendent les définitions classiques de cap et cup produit à la cohomologie d'intersection. Ils obtiennent ainsi une description géométrique de la dualité de Poincaré, comme cap produit avec une classe fondamentale. Ces constructions cohomologiques ont été relevées au niveau des complexes de chaines et de cochaines par D. Chataur M. Saralegui et D. Tanré dans \cite{DavidMemoire} et \cite{DavidPoincare}.

\section*{Sur l'étude des espaces stratifiés}
\phantomsection
\addcontentsline{toc}{section}{Sur l'étude des espaces stratifiés}

On remarque que la formule du cône implique directement que la cohomologie d'intersection n'est pas un invariant du type d'homotopie de l'espace sous-jacent. En effet, un cône est toujours contractile, mais la cohomologie d'intersection d'un cône n'est pas toujours triviale. La cohomologie d'intersection est donc un invariant du type d'homotopie stratifié, et non du type d'homotopie. Motivé par ce constat, M. Goresky et R. MacPherson \cite{Borel} formulent plusieurs problèmes. Il est naturel de commencer par chercher une catégorie adaptée à l'étude de la cohomologie d'intersection \cite[Problem $\# 4$]{Borel}.

\begin{problem2}
Find the most general category of spaces and maps (perhaps with additional data) on which intersection cohomology is functorial.
\end{problem2}

Un certain nombre d'approches ont été suivies pour répondre à cette question.

En particulier, D. Chataur, M. Saralegui et D. Tanré introduisent la catégorie des ensembles de faces filtrés (Filtered face sets) \cite{DavidMemoire}. Les objets de cette catégorie portent une filtration, et sont des ensembles simpliciaux sans dégénérescence. Le foncteur de cohomologie d'intersection qu'ils définissent permet ensuite d'étudier la théorie de l'homotopie rationnelle des espaces stratifiés. 

Plus généralement, on s'intéresse à un contexte catégorique pour les espaces stratifiés, adapté à la notion d'homotopie stratifiée.

Poursuivant une notion d'espace stratifié régulier se prêtant plus naturellement à l'homotopie, F. Quinn propose la classe des espaces homotopiquement stratifiés \cite{Quinn}. Les strates d'un tel espace ne sont plus supposées être des variétés, et les conditions de trivialités locales sont remplacées par des conditions homotopiques. Si $X$ est un espace stratifié et $S,T$ sont deux strates de $X$ on appelle entrelacs homotopique de $S$ et $T$ l'espace des chemins $\gamma\colon [0,1]\to X$ tels que $\gamma(0)\in S$ et $\gamma(t)\in T$ pour tout $t>0$. Alors, $X$ est homotopiquement stratifié si la restriction en $0$ induit une fibration depuis l'entrelacs homotopique vers $S$, pour toute paire de strates $S$ et $T$. L'étude des propriétés homotopiques de ces objets a été poursuivie par D. Miller \cite{Miller}, qui a notamment prouvé qu'une équivalence d'homotopie stratifiée entre de tels espaces homotopiquement stratifiés était exactement une application stratifiée induisant des équivalences faibles sur toutes les strates et tous les entrelacs homotopiques. 

Une approche différente pour étudier l'homotopie des espaces stratifiés a été proposée par \cite{HigherAlgebra}, \cite{WoolfFundamentalCategory}, \cite{TamakiTanaka}. Plutôt que de se concentrer sur les objets "réguliers" de la théorie des espaces stratifiés, on considère une catégorie de "tous" les espaces stratifiés, pour laquelle les notions de morphismes et d'homotopies respectant la stratification sont plus élémentaires. Dans ce contexte, un espace stratifié est simplement défini comme un espace topologique $X$ muni d'une application continue $\varphi_X\colon X\to P_X$ vers un ensemble ordonné de strates, $P_X$, muni de la topologie induite par l'ordre. Les strates y sont données par les pré-images des éléments de $P$, et un morphisme stratifié est simplement un diagramme commutatif
\begin{equation*}
\begin{tikzcd}
X
\arrow[swap]{d}{\varphi_X}
\arrow{r}{f}
&Y
\arrow{d}{\varphi_Y}
\\
P_X
\arrow{r}{\bar{f}}
&P_Y
\end{tikzcd}
\end{equation*}
On retrouve les espaces stratifiés "réguliers" dans ce contexte en considérant les espaces coniquement stratifiés. C'est dans ce contexte que la notion de $\infty$-catégorie des chemins sortants est définie.

\section*{Vers une catégorie homotopique d'espaces stratifiés}
\phantomsection
\addcontentsline{toc}{section}{Vers une catégorie homotopique d'espaces stratifiés}

On dispose de catégories d'espaces stratifiés dans lesquels on sait interpréter la notion d'homotopie stratifiée. Dans ce contexte, on peut réexaminer le problème \cite[Problem $\# 11$]{Borel} de M. Goresky et R. MacPherson.
 
\begin{problem2}
Is there a category of spaces, maps and homotopies and a "classifying space" $B$ so that $IH_i(X)$ can be interpreted as homotopy classes of maps from $X$ to $B$?
\end{problem2}
 
La notion de catégorie modèle se prête particulièrement bien aux problèmes de représentabilités. En effet, on dispose de résultats généraux garantissant l'existence de classifiants \cite{JardineRepresentability}. Alternativement, on peut travailler avec des $\infty$-catégories.

Par exemple, des travaux de D. Ayala, J. Francis et N. Rozenblyum fournissent une $\infty$-catégorie pour la théorie de l'homotopie de certains espaces stratifiés réguliers \cite{AyalaFrancisTanaka}. 
 
Plus récemment, S. Nand-Lal, dans sa thèse \cite{NandLal}, a transporté la structure de modèle de Joyal sur les ensembles simpliciaux vers la catégorie des espaces stratifiés. Il a montré qu'une certaine sous-catégorie d'objets fibrants, contenant notamment les espaces homotopiquement stratifiés de F. Quinn ainsi que les pseudo-variétés, portait une structure de modèle.

Indépendamment, P. Haine a construit une localisation de la structure de Joyal sur la catégorie des ensembles simpliciaux au dessus d'un ensemble ordonné \cite{Haine}. Il a ainsi obtenu la catégorie de modèle simpliciale sous-jacente à la $\infty$-catégorie considérée dans \cite{Exodromy} dont les objets encodent les $\infty$-catégorie de chemins sortants des objets stratifiés. 

C'est aussi à ce problème qu'on s'attaque dans cette thèse, en proposant une structure de modèle sur la catégorie des espaces stratifiés.

Dans le cadre simplicial, un résultat de représentabilité de la cohomologie d'intersection prolongerait une série de travaux par D. Chataur, M.Saralegui et D.Tanré.

 Au delà de la question de la représentabilité, il est naturel de se demander si des théories cohomologiques d'intersections généralisées existent, et en quel sens elles se comparent à la cohomologie d'intersection. C. Debord et J-M. Lescure \cite{DebordLescure} ont défini un analogue de la K-théorie pour les pseudo-variétés, et M.Banagl a proposé une construction générale qui fournit des candidats de théories homologiques d'intersection généralisés rationnelles \cite{BanaglRationalGeneralizedHomology}.
Il s'agit de réponses partielles à la question \cite[Problem $\# 1$]{Borel}

\begin{problem2}
Is there an "intersection homology version" of cobordism theory or K theory (or homotopy theory)? \texttt{[...]}
\end{problem2}

\section*{Résultats principaux}
\phantomsection
\addcontentsline{toc}{section}{Résultats principaux}

Dans le contexte non-stratifié, les invariants homotopiques des espaces jouent un rôle important dans la classification des variétés de grande dimension ($>4$).
Ces invariants sont définis sur les classes d'équivalences d'espaces à homotopie près. Pour les calculer, on choisit des représentants appropriés. Ce procédé est analogue au remplacement par une résolution projective ou injective en algèbre homologique. La théorie des catégories modèles, construite par D. Quillen \cite{QuillenHomotopicalAlgebra} donne corps à cette analogie. Elle fournit un contexte catégorique dans lequel traiter des problèmes de nature homotopique; c'est ce point de vue que l'on adopte dans cette thèse.

Nous allons construire des structures de modèles pour les espaces stratifiés et étudier certains de leurs invariants.
Dans un souci de généralité, on ne restreint pas cette étude aux espaces stratifiés "réguliers" mais on considère plutôt la catégorie des espaces stratifiés au dessus d'un ensemble ordonné. Cette définition, bien que présente implicitement dans la littérature depuis H. Whitney, a été formalisée par J. Lurie \cite{HigherAlgebra}. 

Dans ce contexte, deux applications stratifiées $f,g\colon (X,\varphi_X\colon X\to P_X)\to(Y,\varphi_Y\colon Y\to P_Y)$ sont homotopes au sens stratifiés si il existe une homotopie stratifiée $H\colon (X\times [0,1],\varphi_X\circ\pr_X)\to \fil{Y}$ entre $f$ et $g$.
On observe que la définition d'homotopie stratifiée a une conséquence immédiate. Si deux espaces stratifiés sont homotopiquement équivalents au sens stratifié, leurs ensembles ordonnés de strates sont isomorphes. On peut donc ramener l'étude au cas d'un ensemble ordonné fixé, $P$, au moins dans un premier temps. Lorsque l'ensemble ordonné est fixé, on parlera d'objets filtrés plutôt que d'objets stratifiés. 

En suivant l'approche classique, qui est de travailler avec des ensembles simpliciaux pour étudier le type d'homotopie des espaces, on est amené à considérer une catégorie d'ensembles simpliciaux filtrés $\sS_P$. Les objets y sont des ensembles simpliciaux $X$, muni d'un morphisme $X\to N(P)$ vers le nerf de l'ensemble ordonné $P$, et les morphismes sont les morphismes $X\to Y$ faisant commuter le diagramme
\begin{equation*}
\begin{tikzcd}
X
\arrow{rr}
\arrow{dr}
&&
Y
\arrow{dl}
\\
&N(P)
\end{tikzcd}
\end{equation*}
Les chapitres \ref{ConstructionCMFSSetP} et \ref{ChapitreStructureSimplicialeSSetP} ainsi que les annexes \ref{ChapitreCaracterisationFibrationsAnnexe} et \ref{ChapitreCaracterisationMorphismeXExXAnnexe} sont entièrement consacrés à construire et à caractériser une structure de modèle sur $\sS_P$. On y construit notamment un foncteur de remplacement fibrant, $\Exi_P$, ainsi que des invariants du type d'homotopie filtré, les groupes d'homotopie filtrés $s\pi_n$. Le théorème suivant est ainsi une synthèse des théorèmes \ref{ExistenceCMFCisinski}, \ref{TheoDescriptionExpliciteSSetP} et \ref{EquivalenceFaibleIsoGroupeHomotopie}.
\begin{theointro}
La catégorie $\sS_P$ munie des classes de flèches suivantes, est une catégorie modèle simpliciale, à engendrement cofibrant, et propre.
\begin{itemize}
\item Les cofibrations sont les monomorphismes.
\item Les fibrations sont les morphismes ayant la propriété de relèvement à droite par rapport aux inclusions de cornets admissibles.
\item Les équivalences faibles entre objets fibrants sont les morphismes induisant des isomorphismes entre tout les groupes d'homotopie filtrés.
\end{itemize}
\end{theointro}
Muni de cette catégorie modèle, on exploite une adjonction simpliciale 
\begin{equation*}
\RealP{-}\colon \sS_P\leftrightarrow \Top_P\colon \Sing_P
\end{equation*}
pour déduire des résultats sur la théorie homotopique des espaces stratifiés. Par un théorème de J. Lurie \cite[Théorème A.6.4]{HigherAlgebra}, si $\fil{X}$ est un espace coniquement stratifié, alors $\Sing_P\fil{X}$ est un objet fibrant de $\sS_P$. D'autre part, par un résultat de S. Nand-Lal \cite{NandLal}, si $\fil{X}$ est un espace métrique homotopiquement stratifié, alors $\Sing_P\fil{X}$ est un objet fibrant de $\sS_P$. On en déduit alors une version stratifiée du théorème de Whitehead. (voir les théorèmes \ref{PremierTheoremeWhitehead} et \ref{DeuxiemeTheoremeWhitehead} et les remarques \ref{PremierTheoremeWhiteheadAHomotopiePres} et \ref{RemarqueNandLalWhitehead})
\begin{theointro}
Soient $\fil{X}$ et $\fil{Y}$ deux espaces filtrés, et $f\colon \fil{X}\to\fil{Y}$ un morphisme entre eux. Si 
\begin{itemize}
\item il existe des ensembles simpliciaux filtrés $\fil{A}$ et $\fil{B}$ et des équivalences d'homotopie filtrées $g\colon \RealP{\fil{A}}\to\fil{X}$ et $h\colon \RealP{\fil{B}}\to\fil{Y}$,
\item les espaces filtrés $\fil{X}$ et $\fil{Y}$ sont coniquement stratifiés, ou métriques homotopiquement stratifiés,
\end{itemize}
alors, les assertions suivantes sont équivalentes.
\begin{itemize}
\item L'application $f$ est une équivalence d'homotopie filtrée.
\item L'application $f$ induit des isomorphismes sur tous les groupes d'homotopie filtrés.
\item L'application $f$ induit des équivalences faibles sur toutes les strates et tous les entrelacs.
\end{itemize}
\end{theointro}

Au chapitre \ref{ChapitreExemples}, on examine une série d'exemples de nature géométriques. On étudie d'une part des espaces stratifiés construit à partir de fibrés localement triviaux, et d'autre part des stratifications induites par des plongements de sous-variétés. Ces exemples illustrent le potentiel d'application des groupes d'homotopie stratifiés à des situations géométriques.

On se tourne ensuite vers la définition d'une catégorie modèle pour les espaces filtrés. Dans le chapitre \ref{ChapitreCMFTopP}, on exploite une catégorie modèle à engendrement cofibrant sur une catégorie de diagrammes simpliciaux , $\DiagR_P$, pour définir une structure de modèle sur la catégorie $\Top_P$. Ces deux catégories sont reliées par une adjonction
\begin{equation*}
\Colim\colon \DiagR_P\leftrightarrow \Top_P\colon D.
\end{equation*}
Par construction, les foncteurs de groupes d'homotopie filtrés se factorisent par le foncteur $D$. Par une application du théorème de transport \cite[Corollary 3.3.4]{Hess}, on obtient la catégorie modèle suivante (Théorème \ref{CategorieModeleTopP})
\begin{theointro}
La catégorie $\Top_P$, munie des classes de flèches suivantes est une catégorie modèle à engendrement cofibrant.
\begin{itemize}
\item Les fibrations sont les morphismes $f$ tels que $D(f)$ est une fibration de $\DiagR_P$.
\item Les équivalences faibles sont les morphismes induisant des isomorphismes sur tout les groupes d'homotopie filtrés
\item Les cofibrations sont définies par propriétés de relèvement par rapport aux fibrations triviales.
\end{itemize}
\end{theointro}

La comparaison avec le cas non-filtrée pousse à comparer les catégories modèles $\sS_P$ et $\Top_P$. Cependant, on observe que la paire de foncteurs $(\RealP{-},\Sing_P)$ ne forme par une paire de Quillen. Néanmoins, au chapitre \ref{ChapitreKanQuillen}, à l'aide d'une catégorie modèle intermédiaire, $\sSTop_P$, on arrive au résultat de comparaison suivant (Corollaire \ref{CorollaireAdjonctionKanQuillen} et Théorème \ref{RealisationSingPreserventEquivalencesFaibles}).

\begin{theointro}\label{TheoIntroKanQuillen}
Il existe une adjonction de Quillen entre $\sS_P$ et $\Top_P$ donnée par
\begin{equation*}
\RealP{\sd_P(-)}\colon \sS_P\leftrightarrow \Top_P\colon \Ex_P\Sing_P.
\end{equation*}
De plus, les adjonctions $(\RealP{\sd_P(-)},\Ex_P\Sing_P)$ et $(\RealP{-},\Sing_P)$ préservent les équivalences faibles.
\end{theointro}
Finalement, au chapitre \ref{ChapitreCMFStrat}, on revient au problème initial qui était de comprendre la théorie de l'homotopie des espaces stratifiés. On exploite l'existence d'une bi-fibration de Grothendieck $\Strat\to \Poset$ pour construire des catégories modèles des espaces stratifiés et des ensembles simpliciaux stratifiés. On obtient ainsi une version du théorème \ref{TheoIntroKanQuillen} pour les objets stratifiés (Théorèmes \ref{TheoremeCMFStrat} et \ref{TheoremeRealStratSingStratPreserventEquivalenceFaible})
\begin{theointro}
Il existe une catégorie modèle pour les espaces stratifiés, $\Strat$ ainsi qu'une catégorie modèle pour les ensembles simpliciaux stratifiés $\sStrat$. De plus, ces catégories sont liées par une adjonction qui préserve les équivalences faibles.
\begin{equation*}
\Real{-}_{\Strat}\colon\sStrat\leftrightarrow \Strat\colon \Sing_{\Strat}
\end{equation*}
\end{theointro}

\section*{Quelques conjectures}
\phantomsection
\addcontentsline{toc}{section}{Quelques conjectures}
A l'issue de l'écriture de ce texte, plusieurs questions naturelles restent ouvertes. On a construit ici des catégories modèles pour les espaces filtrés $\Top_P$ ainsi que pour les ensembles simpliciaux filtrés $\sS_P$ - ainsi que pour leur contreparties stratifiées $\Strat$ et $\sStrat$ - et on a montré qu'elles étaient liées par une adjonction de Quillen. Dans le cas où $P=\{*\}$ cette adjonction coïncide avec l'adjonction de Kan-Quillen classique entre $\sS$ et $\Top$, et est en particulier une équivalence de Quillen. Il est alors naturel de se demander si l'adjonction entre $\Top_P$ et $\sS_P$ est une équivalence de Quillen en général. Dans cette direction, dans le chapitre \ref{ChapitreKanQuillen}, on décompose cette adjonction en deux adjonctions de Quillen :
\begin{equation*}
\sd_P\colon \sSU_P\leftrightarrow \sSTop_P\colon \Ex_P
\end{equation*} 
et
\begin{equation*}
\RealP{-}\colon \sSTop_P\leftrightarrow \Top_P\colon \Sing_P,
\end{equation*}
qu'on étudie séparément. Après avoir ramené à deux propriétés atomiques l'affirmation que chacune de ces adjonctions est une équivalences de Quillen, on pose la conjecture suivante (Voir les conjectures \ref{ConjectureRealNPSingNP} et \ref{ConjectureSdPExP}).
\begin{conjectureintro}
Les catégories modèles $\Strat$ et $\sStrat$ sont Quillen équivalentes.
\end{conjectureintro}
Toujours au chapitre \ref{ChapitreKanQuillen}, on étudie aussi les objets cofibrants de $\Top_P$. Contrairement au cas non filtré, la réalisation d'un ensemble simplicial filtré n'est pas nécessairement fibrante; en particulier, les pseudo-variétés PL ne sont pas cofibrantes en général. Cependant, la subdivision filtrée fournit un remplacement cofibrant. En étudiant la topologie de ce remplacement cofibrant dans le cas élémentaire où $P=\{0<1\}$, on arrive à montrer que la catégorie homotopique naïve des pseudo-variétés PL filtrées au dessus de $P$ est une sous catégorie pleine de $\Ho(\Top_P)$. Bien que faisant intervenir des constructions topologiques fines, les méthodes employées pour prouver ce résultat semblent s'adapter au cas d'un ensemble ordonné arbitraire. En notant $\Pvar\subset \Strat$ la sous catégorie pleine des pseudo-variétés PL, et $\Pvar/{\sim}$ sa localisation par rapport aux équivalences d'homotopie stratifiées, on formule donc la conjecture suivante. (Voir la conjecture \ref{ConjectureSousCategorieFibre})
\begin{conjectureintro}
L'inclusion de la sous catégorie pleine $\Pvar\hookrightarrow \Strat$, induit un foncteur pleinement fidèle
\begin{equation*}
\Pvar/{\sim}\to \Ho(\Strat)
\end{equation*}
\end{conjectureintro}

\section*{Guide du lecteur}
\phantomsection
\addcontentsline{toc}{section}{Guide du lecteur}

Le sujet de cette thèse est d'étudier l'homotopie des espaces stratifiés. Aussi, ce texte porte sur la définition et l'étude de plusieurs catégories modèles pour les espaces stratifiés.

Dans le chapitre \ref{ChapterEspacesStratifies}, on rappelle les diverses incarnations de la notion classique de pseudo-variété ainsi qu'une définition de l'homologie d'intersection. On présente ensuite la notion d'espace stratifié au dessus d'un ensemble ordonné et on donne un dictionnaire permettant de traduire les diverses notions de pseudo-variétés en ces termes. 

Dans le chapitre \ref{ChapterRappelHomotopie}, on aborde la notion de catégorie modèle à travers une étude du théorème de Whitehead. On présente la catégorie modèle des ensembles simpliciaux ainsi que celle des espaces topologiques, et on rappelle que l'adjonction de Kan-Quillen induit une équivalence entre leurs catégories homotopiques.

Dans le chapitre \ref{ConstructionCMFSSetP}, on étudie la catégorie des ensembles simpliciaux filtrés, $\sS_P$. Lorsque $P$ est un ensemble ordonné, les objets de $\sS_P$ sont des ensembles simpliciaux $X$ munis d'un morphisme $\varphi_X\colon X\to N(P)$, où $N(P)$ est le nerf de $P$. Cette catégorie joue un rôle analogue à celui que la catégorie des ensembles simpliciaux joue pour l'homotopie des espaces non-filtrés. On montre que la construction générale de D-C. Cisinski \cite{Cisinski} s'applique pour fournir une structure de modèle sur $\sS_P$. Cette catégorie modèle est propre et à engendrement cofibrant, et la dernière section du chapitre est consacrée à prouver qu'un ensemble générateur des cofibrations triviales est donné par les inclusions de cornets admissibles. Cette caractérisation des fibrations "à la Kan" est clé pour l'étude de cette catégorie modèle, et repose sur les foncteurs de subdivision filtrée $\sd_P$ et d'extension filtrée $\Ex_P$ ainsi que sur des résultats techniques prouvés en annexes \ref{ChapitreCaracterisationFibrationsAnnexe} et \ref{ChapitreCaracterisationMorphismeXExXAnnexe}.

Dans le chapitre \ref{ChapitreStructureSimplicialeSSetP}, on montre que la catégorie modèle $\sS_P$ est simpliciale. Ceci permet notamment de définir des invariants algébriques du types d'homotopie filtré : les groupes d'homotopies filtrés $s\pi_n$. On montre que ces invariants caractérisent les équivalences faibles entre objets fibrants.

Dans le chapitre \ref{ChapitreGroupesHomotopiesFiltresEspaces}, on exploite une adjonction entre la catégorie des espaces filtrés et celle des ensembles simpliciaux filtrés 
\begin{equation*}
\RealP{-}\colon \sS_P\leftrightarrow \Top_P\colon \Sing_P,
\end{equation*}
pour extraire de l'information sur les espaces filtrés. On étend la définition des groupes d'homotopie filtrés aux espaces filtrés de façon compatible au foncteur $\Sing_P$. On en déduit une version stratifiée du théorème de Whitehead (Théorèmes \ref{PremierTheoremeWhitehead} et \ref{DeuxiemeTheoremeWhitehead}).

Le chapitre \ref{ChapitreExemples} est une collection d'exemples d'espaces filtrés et de calculs de groupes d'homotopie filtrés. On y examine notamment des espaces filtrés provenant de fibrés localement triviaux et de plongements. En s'intéressant au cas des noeuds, on illustre le contenu des groupes d'homotopie filtrés à travers un invariant complet des noeuds.

Au chapitre \ref{ChapitreCMFTopP}, on définit une catégorie modèle des espaces filtrés. La structure de modèle est obtenue par transport depuis une catégorie de diagrammes simpliciaux intervenant dans la définition des groupes d'homotopie filtrés. On introduit aussi la catégorie modèle des espaces fortement filtrés et on montre qu'elle est Quillen équivalente à celle des espaces filtrés. Les espaces fortement filtrés sont des espaces topologiques munis d'une application continue vers la réalisation du nerf de l'ensemble ordonné $P$. On montre ensuite que la catégorie des espaces filtrés est Quillen équivalente à celle des diagrammes simpliciaux. 

Au chapitre \ref{ChapitreKanQuillen}, on exhibe une adjonction de Kan-Quillen entre la catégorie modèle des ensembles simpliciaux filtrés et celle des espaces fortement filtrés
\begin{equation*}
\RealNP{\sd_P(-)}\colon \sS_P\leftrightarrow\Top_{N(P)}\colon \Ex_P\Sing_{N(P)}.
\end{equation*}
Pour montrer qu'il s'agit d'une adjonction de Quillen, on exhibe une catégorie modèle intermédiaire, $\sSTop_P$. Celle-ci est obtenue par transport sur $\sS_P$ de la structure de modèle de $\Top_{N(P)}$, et le transport repose sur une construction explicite de remplacements cofibrants dans la catégorie $\Top_{N(P)}$. On examine finalement la possibilité que l'adjonction $(\RealNP{\sd_P(-)},\Ex_P\Sing_{N(P)})$ soit une équivalence de Quillen.

Finalement, au chapitre \ref{ChapitreCMFStrat}, on exploite la notion de bifibration de Quillen pour construire des structures de modèles sur des catégories d'espaces stratifiés $\Strat$ et d'ensembles simpliciaux stratifiés $\sStrat$. On discute aussi des liens avec d'autres travaux sur le sujet.

\chapter{Espaces stratifiés}
\label{ChapterEspacesStratifies}

Ce premier chapitre est une introduction à la notion d'espace stratifié. Dans la première section, on présente les différentes définitions classiques de pseudo-variétés. Ce sont les objets réguliers de la théorie des espaces stratifiés, et les premier exemples historiques pour lesquels la notion de stratification a été considérée.

Dans la section \ref{SectionHomologieIntersection}, on aborde l'homologie d'intersection. On choisit - parmi les nombreuses définitions possibles - de la définir via la notion de simplexe filtré en suivant \cite{DavidInvarianceTopologique}. Ce formalisme a pour vertu de s'interpréter facilement en terme d'ensembles simpliciaux filtrés (voir le chapitre \ref{ConstructionCMFSSetP}). On rappelle aussi plusieurs propriétés importantes de l'homologie d'intersection, notamment sa naturalité par rapport aux morphismes stratifiés ainsi que son invariance par homotopie stratifiée.

Finalement, dans la section \ref{SectionChap1TopP} on aborde la notion d'espace stratifié au dessus d'un ensemble ordonné. Cette notion est motivée par la condition de frontière que vérifient tout les exemples d'espaces stratifiés considérés dans la section \ref{SectionPseudoVar}. On étudie la définition d'équivalence d'homotopie stratifiée entre de tels espaces et on remarque que pour comprendre les équivalences d'homotopie stratifiées, il suffit de s'intéresser aux paires d'espaces ayant les mêmes ensembles de strates. Ceci nous amène dans la section \ref{SectionHomotopieStratifieesEtTopP} à définir la catégorie des espaces filtrés au dessus d'un ensemble ordonné fixé.
Finalement, dans la section \ref{SectionPseudoVarTopP}, on montre que les notions d'espace localement conique et de pseudo-variété se traduisent naturellement dans la catégorie des espaces stratifiés.

\section{Pseudo-variétés}
\label{SectionPseudoVar}
Dans cette section on introduit la notion de pseudo-variété et quelques une de ces variantes. Les pseudo-variétés apparaissent naturellement dans diverses contexte géométrique, comme limites ou quotients. Elles sont construites suivant un processus itératif, on découpe l'espace suivant une strate régulière - qui est une variété - et une partie singulière. On raffine ensuite ce découpage en décomposant la partie singulière. La décomposition ainsi obtenue est une partition en variétés (différentiables, topologiques ou PL suivant le contexte), avec des conditions de recollements. On commence cette section par introduire la notion la plus générale de pseudo-variété topologique différentiable avant de passer aux définitions plus classiques, mais plus restrictives, de pseudo-variété PL, de Whitney et de Thom-Mather. On rappellera aussi la notion d'espaces homotopiquement stratifiés de F. Quinn.

\subsection{Pseudo-variétés topologiques}
\label{SectionPseudoVarTop}
\begin{defin}[{\cite[Definition 2.2.1]{Friedman}}]\label{DefinitionEspaceFiltreClassique}
Un espace filtré de dimension formelle $n$ est un espace topologique séparé, $X$, muni d'une suite croissante de sous-espaces fermés
\begin{equation*}
\emptyset=X^{-1}\subseteq X^0\subseteq X^1\subseteq\dots\subseteq X^{n-1}\subseteq X^n=X
\end{equation*}
\end{defin}

\begin{remarque}\label{RemarqueFiltreClassiqueEgalFiltre}
On note que la définition classique \ref{DefinitionEspaceFiltreClassique} correspond exactement à la donnée d'un espace filtré $(X,\varphi_X\colon X\to \{0<1<\dots<n-1<n\})$ au sens de la définition \ref{DefinitionTopP} (voir aussi la définition \ref{DefinitionChap1TopP}) avec l'hypothèse supplémentaire que $X$ est séparé.
\end{remarque}

\begin{defin}
Si $X$ est un espace filtré de dimension formelle $n$, on appelle strate de dimension $i$ de $X$ les composantes connexes de $X^i\setminus X^{i-1}$. Les strates de dimensions $i<n$ sont dites singulières, et celles de dimension $n$ sont dites régulières.
\end{defin}

\begin{defin}\label{DefinitionConeOuvertFiltreClassique}
Si $X$ est un espace topologique, le cône ouvert de $X$, $c(X)$, est l'espace topologique défini comme
\begin{equation*}
c(X)=\frac{X\times [0,1[}{X\times\{0\}}.
\end{equation*}
Par convention, on pose $c(\emptyset)=\{*\}$.
Si $X$ est muni d'une filtration, le cône ouvert de $X$ est muni de la filtration définie par
\begin{equation*}
c(X)^{k+1}=c(X^{k})
\end{equation*}
\end{defin}

\begin{defin}[{\cite[Definition 2.3.1]{Friedman}}]
\label{DefinitionLocalementConique}
Un espace filtré $X$, de dimension formelle $n$ est localement conique si pour tout $0\leq i\leq n$ et pour tout $x\in X_i$, il existe un espace filtré compact $L$, des voisinages de $x$, $U\subset X_i$ et $N\subset X$, et un homéomorphisme
\begin{equation*}
h\colon U\times c(L)\to N,
\end{equation*}
tel que pour tout $0\leq k\leq n-i-1$, 
\begin{equation*}
h(U\times c(L^k))\simeq X^{i+k+1}\cap N.
\end{equation*}
Dans ce cas, $L$ est un entrelacs de $x$.
On dira que $X$ est un CS espace si pour tout $x\in X^i\setminus X^{i-1}$, on peut choisir le voisinage $U\subseteq X^i$ tel que $U\simeq \R^i$.
\end{defin}

\begin{remarque}
Si $X$ est un CS espace de dimension formelle $n$, alors pour tout $0\leq i\leq n$, les strates de dimensions $i$ de $X$ sont des variétés topologiques de dimension $i$.
\end{remarque}

On peut maintenant définir les pseudo-variétés topologiques
\begin{defin}[{\cite[Definition 2.3.10]{Friedman}}]\label{DefinPseudoVarTopologique1}
Un CS espace récursif est un CS espace tel que tout point admet un entrelacs qui est lui même un CS espace récursif (éventuellement vide).
Un CS espace récursif de dimension $n$, $X$, est une pseudo-variété si $X^{n}\setminus X^{n-1}$ est dense dans $X$.
\end{defin}
Alternativement, on peut donner une définition équivalente et plus directe des pseudo-variétés topologiques.

\begin{defin}\label{DefinitionAlternativePseudoVarTop}
Soit $X$ un espace filtré de dimension formelle $n$. L'espace $X$ est une pseudo-variété topologique de dimension $n$ si pour tout $x\in X^i\setminus X^{i-1}$, il existe un voisinage ouvert de $x$, $N\subseteq X$, une pseudo-variété $L$ de dimension $n-i-1$ (vide si et seulement si $i=n$) et un homéomorphisme
\begin{equation*}
h\colon \R^i\times c(L)\to N
\end{equation*}
tel que pour tout $0\leq k\leq n-i$,
\begin{equation*}
h(\R^i\times c(L^{k-1}))\simeq X^{i+k}.
\end{equation*}
\end{defin}

\begin{remarque}
On voit que en particulier, toutes les strates d'une pseudo-variété de dimension $n$ sont des variétés topologiques de dimension $\leq n$. 
\end{remarque}
\begin{remarque}
Il est standard dans la littérature d'ajouter l'hypothèse $X^{n-2}=X^{n-1}$ à la définition de pseudo-variété. Cette hypothèse est équivalente à demander que $X$ ne possède pas de strate de codimension $1$, et est nécessaire pour que l'homologie et la cohomologie d'intersection vérifient de bonnes propriétés. Cependant, cette hypothèse n'est pas nécessaire pour développer la théorie des pseudo-variétés (voir \cite{Friedman}), et on ne l'inclut donc pas. Une pseudo-variété qui vérifiera $X^{n-1}=X^{n-2}$ sera dite classique.
\end{remarque}

\begin{exemple}
Soit $M$ une variété topologique à bord de dimension $n$. On muni $M$ de la filtration 
\begin{equation*}
\emptyset\subseteq\dots\subseteq\emptyset\subseteq M^{n-1}\subseteq M^n=M
\end{equation*}
avec $M^{n-1}=\partial(M)$. Alors, $M$ muni de cette filtration est une pseudo-variété. Ce n'est cependant pas une pseudo-variété classique si $\partial(M)\not=\emptyset$ puisque dans ce cas $M$ possède une strate de codimension $1$.
\end{exemple}

\begin{exemple}
Soient $C=S^1\times [0,1]$ un cylindre, et $M=S^1\times [-1,1]/{\sim}$, avec $(x,t)\sim(-x,-t)$ pour tout $x\in S^1$, $t\in [-1,1]$, un ruban de Möbius. On définit des filtrations sur $C$ et $M$ par
\begin{equation*}
\emptyset=\emptyset\subset \begin{red}S^1\times\{0\}\end{red}\subseteq S^1\times [0,1]=C,
\end{equation*}
et
\begin{equation*}
\emptyset=\emptyset\subset\begin{red}S^1\times\{0\}/{\sim}\end{red}\subseteq S^1\times [-1,1]/{\sim}=M.
\end{equation*}
Alors, $M$ et $C$ sont des pseudo-variétés de dimension $2$. Voir la figure \ref{FigureMobiusCylindreIntro}
\end{exemple}

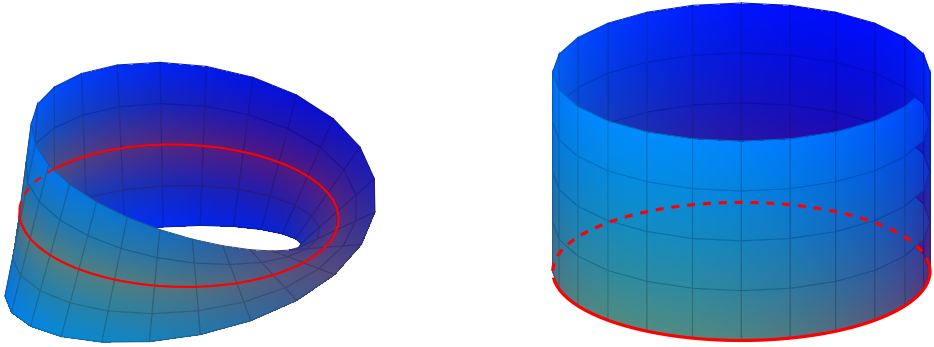
\begin{figure}
\centering
\begin{tikzpicture}
\begin{axis}[
    hide axis,
    view={40}{40}
]
\addplot3 [
    surf, shader=faceted interp,
    mesh/color input=explicit mathparse,
    samples y=5,
    domain=0:360,
    y domain=-0.5:0.5,
    point meta={symbolic={
    		1.3*(0.5-2*abs(z))*(0.5-2*abs(z)), 
	   		0.3-0.2*y, 
    		0.5+2*abs(z) 
    	}
    }
] (
    {(1+0.5*y*cos(x/2)))*cos(x)},
    {(1+0.5*y*cos(x/2)))*sin(x)},
    {0.5*y*sin(x/2)});
\addplot3 [
    samples=50,
    domain=-145:180, 
    samples y=0,
    thick, red
] (
    {cos(x)},
    {sin(x)},
    {0});
\addplot3 [
	dashed,
    samples=20,
    domain=-180:-145, 
    samples y=0,
    thick, red
] (
    {cos(x)},
    {sin(x)},
    {0});
\end{axis}

\end{tikzpicture}
\hspace{60 pt}
\begin{tikzpicture}
\begin{axis}[
    hide axis,
    view={30}{30}
]
\addplot3 [
    surf, shader=faceted interp,
    mesh/color input=explicit mathparse,
    samples y=5,
    domain=0:360,
    y domain=0:0.5,
    point meta={symbolic={
    		1.3*(0.5-2*abs(z))*(0.5-2*abs(z)), 
	   		0.3-0.2*y, 
    		0.5+2*abs(z) 
    	}
    }
] (
    {1.25*cos(x)},
    {1.25*sin(x)},
    {0.5*y});
\addplot3 [
    samples=50,
    domain=-145:35, 
    samples y=0,
    very thick, red
] (
    {1.25*cos(x)},
    {1.25*sin(x)},
    {0});
\addplot3 [
	dashed,
    samples=50,
    domain=35:215, 
    samples y=0,
    very thick, red
] (
    {1.25*cos(x)},
    {1.25*sin(x)},
    {0});
\end{axis}
\end{tikzpicture}
\caption{Les pseudo-variétés $M$ et $C$. La strate singulière est représenté en rouge.
}\label{FigureMobiusCylindreIntro}
\end{figure}

\begin{figure}[h!]
\centering
\includegraphics[width= 100pt]{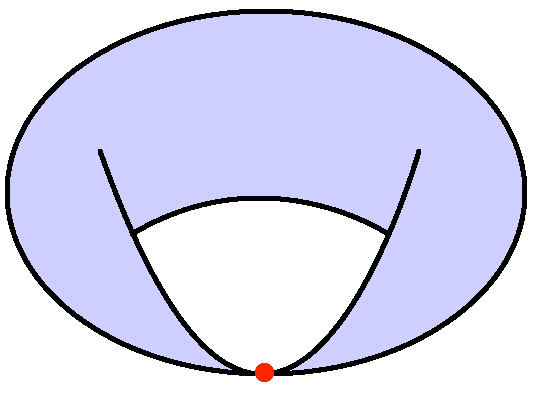}
\caption{Le tore pincé muni de sa filtration usuelle. Le point singulier est représenté en rouge.}
\label{FigureTorePince}
\end{figure}

\begin{exemple}\label{ExemplePseudoVarSingularitesIsolees}
Les pseudo-variétés à singularités isolées forment une classe d'exemples distinguée parmi les pseudo-variétés. Ce sont des pseudo-variétés $X$ ne possédant que des strates singulières de dimension $0$, les points singuliers isolés. La filtration d'un tel espace est donc donnée par
\begin{equation*}
\emptyset\subset X^0=X^1=\dots=X^{n-1}\subset X^n.
\end{equation*}
Ces objets sont les exemples les plus simples de pseudo-variétés et la théorie pour les étudier est très développée. On obtient de tels exemples en considérant une variété à bord $(M,\partial(M))$ et en prenant le quotient $M/\partial(M)$. L'intérieur de $M$ fournit donc la strate régulière, et le point singulier est l'image dans le quotient de $\partial(M)$. Réciproquement, si $X$ est une pseudo-variété à singularités isolées, on peut la construire à partir d'une variété à bord comme suit. Pour tout point singulier $x\in X^0$, on choisit un voisinage distingué $N_x\simeq c(L_x)$ de façon à ce que ces voisinages soient deux à deux disjoints. Comme $X$ est une pseudo-variété, les entrelacs $L_x$ sont aussi des pseudo-variétés. De plus, comme $X$ n'a que des singularités en dimension $0$, on en déduit que les $L_x$ n'ont pas de singularité, ce sont donc des variétés. Notons maintenant $c_{<1/2}(L_x)$ le sous-cône ouvert de $c(L_x)$ induit par l'inclusion 
\begin{equation*}
L_x\times [0,1/2[\subset L_x\times [0,1[.
\end{equation*}
On considère alors l'espace
\begin{equation*}
M=X\setminus \left(\coprod_{x\in X^0}c_{<1/2}(L_x)\right),
\end{equation*}
et on note
\begin{equation*}
\partial(M)=\coprod_{x\in X^0 }L_x.
\end{equation*}
Les inclusions $L_x\simeq L_x\times \{1/2\}\hookrightarrow L_x\times [1/2,1[$ induise une inclusion 
\begin{equation*}
\partial(M)\hookrightarrow M,
\end{equation*}
et $M\setminus \partial(M)$ est un ouvert de $X$ ne contenant pas les points singuliers. C'est donc un ouvert de la variété $X\setminus X^0$, c'est à dire une variété. Par construction, $\partial(M)$ est une variété et elle admet un collier dans $M$. La paire $(M,\partial(M))$ est donc une variété à bord. Finalement, on retrouve $X$ en contractant séparement chacune des composantes de bord $L_x$. On note que les $L_x$ ne sont a priori pas  connexes.
\end{exemple}

\begin{exemple}
Soit $TP=S^1\times S^1/{\sim}$ où $(x,p)\sim (y,p)$ pour tout $x,y\in S^1$ et où $p$ est un point fixé de $S^1$. L'espace $TP$ est communément appelé le tore pincé et est représenté figure \ref{FigureTorePince}. C'est une pseudo-variété classique si on le munit de la filtration
\begin{equation*}
\emptyset\subset \begin{red}S^1\times\{p\}/{\sim}\end{red}=S^1\times\{p\}/{\sim}\subset S^1\times S^1/{\sim}=TP
\end{equation*}
En effet, pour voir que $TP$ est une pseudo-variété, il suffit de montrer que tout point admet un voisinage conique. Notons $s\in TP$ le point singulier, correspondant au quotient $S^1\times \{p\}/{\sim}$. On a $TP_0=TP_1=\{s\}$. D'autre part, $TP\setminus \{s\}\simeq S^1\times S^1\setminus S^1\times \{p\}$ est une variété topologique de dimension $2$. Ainsi, tout point $x\in TP_2\setminus TP_1$ admet un voisinage homéomorphe à $\R^2\simeq \R^2\times c(\emptyset)$. Par ailleurs, un voisinage de $s$ est homéomorphe à l'union de deux copies de $c(S^1)$ le long de leur sommet. Mais un tel espace est homéomorphe à $c(S^1\coprod S^1)$, et donc $S^1\coprod S^1$ (muni de la filtration triviale) est un entrelacs de $s$. Finalement, $TP$ est une pseudo-variété, et comme $TP_1=TP_0$, c'est une pseudo-variété classique. C'est aussi un exemple de pseudo-variété à singularités isolées, et on peut l'obtenir comme quotient du cylindre $S^1\times [0,1]$ par son bord $S^1\times \{0,1\}$ (voir l'exemple \ref{ExemplePseudoVarSingularitesIsolees}).
\end{exemple}

\begin{figure}[h]
\centering
\begin{tikzpicture}[scale = 1.5]
\filldraw[blue,blue,opacity=0.5](-2,0)--(-2,2)--(2,2)--(2,0)--(-2,0);
\draw[green,thick](-2,0)--(-0.005,0);
\draw[green,thick](0.005,0)--(2,0);
\draw[green,thick](0,0.005)--(0,2);
\draw[green,thick](0,-0.005)--(0,-2);
\node at (0,0) [circle,fill,inner sep=1.5pt,red]{};
\end{tikzpicture}
\caption{L'espace filtré localement conique $X$ muni de la filtration $\textcolor{red}{X^0}\subset \textcolor{green}{X^1}\subset \textcolor{blue}{X^2}$.}
\label{FigureExempleLocalementConiqueIntro}
\end{figure}

\begin{exemple}
On considère $X$ le sous espace de $\R^2$, défini comme
\begin{equation*}
X=\{(x,y)\ |\ y\geq 0 \text{ ou } x=0\}.
\end{equation*}
On se demande si $X$ est une pseudo-variété. Pour répondre à cette question, on doit exhiber une filtration de $X$ pour laquelle toutes les strates sont des variétés. On constate que $X$ est défini comme l'union de deux sous espaces de $\R^2$, la droite $\{x=0\}$ et le demi plan $\{y\geq 0\}$. Comme les strates de $X$ devront être des variétés, on est amené à décomposer le demi plan fermé en deux : le demi plan ouvert $\{y>0\}$ et la droite $\{y=0\}$. Ceci suggère de considérer la filtration
\begin{equation*}
\emptyset=X^0\subset \{x=0\}\cup \{y=0\}\subset X^2=X
\end{equation*}
Cependant, la strate de dimension $1$, $X^1\setminus X^0= \{x=0\}\cup \{y=0\}$ n'est pas une variété car c'est l'union de deux droites qui s'intersectent en $(0,0)$. On est alors amené à considérer la filtration
\begin{equation*}
\emptyset\subset \{(0,0)\}\subset \{x=0\}\cup \{y=0\}\subset X^2=X.
\end{equation*}
L'espace $X$ muni de cette filtration est représenté Figure \ref{FigureExempleLocalementConiqueIntro}.
Montrons que $X$ muni de cette filtration est localement conique. On examine d'abord les strates de $X$ :
\begin{itemize}
\item Les strates de dimension $2$ sont les ouverts $\{(x,y)\ | x<0, y>0\}$ et $\{(x,y)\ |\ x>0, y>0\}$.
\item Les strates de dimension $1$ sont les sous espaces 
\begin{equation*}
\{(0,y)\ |\ y>0\}, \{(0,y)\ |\ y<0\}, \{(x,0)\ |\ x<0\}\text{ et } \{(x,0)\ |\ x>0\}.
\end{equation*}
\item La strate de dimension $0$ est $\{(0,0)\}$.
\end{itemize}  Soit $(x,y)\in X$. Si $(x,y)\in X^2\setminus X^1$, comme $X^2\setminus X^1$ est une union d'ouverts de $\R^2$, $(x,y)$ admet un voisinage homéomorphe à $\R^2\simeq \R^2\times c(\emptyset)$. Si $(x,y)\in X^2\setminus X^1$, plusieurs cas se présentent.
\begin{itemize}
\item  Si $y=0$ alors $(x,0)\in \{(x,0)\ |\ x>0\}$ ou $\{(x,0)\ |\ x<0\}$. Par symétrie, on suppose $x>0$. Alors, $N=]x/2,2x[\times [0,1[$ est un voisinage ouvert de $(x,0)$, et porte la filtration
\begin{equation*}
\emptyset=\emptyset \subset ]x/2,2x[\times\{0\}\subset N^2=N.
\end{equation*}
En particulier, on a un homéomorphisme :
\begin{equation*}
N\simeq ]x/2,2x[\times c(\{*\})\simeq \R\times c(\{*\})
\end{equation*}
et
\begin{equation*}
N\cap \{y=0\}=]x/2,2x[\times \{0\}\simeq \R\times c(\emptyset).
\end{equation*}
L'espace filtré $X$ est donc localement conique en $(x,0)$ et $\{*\}$ est un entrelacs.
\item Si $y<0$, alors $x=0$, et $N=\{0\}\times]2y,y/2[$ est un voisinage ouvert de $(0,y)$ dans $X$. On a
\begin{equation*}
N\simeq c(\emptyset)\times ]2y,y/2[.
\end{equation*}
L'espace filtré $X$ est donc localement conique en $(0,y)$, et $\emptyset$ est un entrelacs.
\item Si $y>0$, alors $x=0$ et un voisinage de $(0,y)$ dans $X$ est donné par $N=]-1,1[\times ]y/2,2y[$. Le voisinage $N$ porte la filtration
\begin{equation*}
\emptyset=\emptyset\subset \{0\}\times ]y/2,2y[\subset N^2=N.
\end{equation*}
Comme précédemment, on a un homéomorphisme qui préserve la filtration :
\begin{equation*}
N\simeq c\left(\{*\}\coprod\{*\}\right)]y/2,2y[\times .
\end{equation*}
et $X$ est donc localement conique en $(0,y)$, et $\{*\}\coprod\{*\}$ est un entrelacs.
\end{itemize}
finalement, si $(x,y)=(0,0)\in X^0$, un voisinage de $(0,0)$ est donné par
\begin{equation*}
N=\{(x,y)\ |\ y\geq 0,\ x^2+y^2\leq 1\}\cup \{(0,y)\ |\ -1\leq y\leq 0\}
\end{equation*}
On remarque que $N$ est homéomorphe au cône de $L$ où $L$ est donné par
\begin{equation*}
L=\{(x,y)\ |\ y\geq 0\ x^2+y^2=1\}\coprod \{(0,-1)\}.
\end{equation*}
De plus, si on muni $L$ de la stratification donnée par
\begin{equation*}
\emptyset\subset \{(0,-1),(0,1),(-1,0),(1,0)\}\subset L^1=L
\end{equation*}
alors l'homéomorphisme entre $N$ et $L$ préserve la filtration. Ainsi, $X$ est localement conique en $(0,0)$, et $X$ est donc un espace localement conique. En examinant les entrelacs, on peut en fait montrer que $X$ est un CS espace récursif. De plus, toutes les strates de $X$ sont des variétés. Cependant $X$ n'est pas une pseudo-variété. En effet, la clôture de $X^2\setminus X^1$ dans $X$ n'est pas $X$ tout entier, mais seulement le demi-plan supérieur. Alternativement, on remarque que l'entrelacs d'un point $(0,y)\in X$ avec $y<0$ est vide, alors que la strate contenant $(0,y)$ n'est pas de dimension maximale. Ceci montre que la notion de pseudo-variété est strictement plus forte que la notion d'espace localement conique.
\end{exemple}

\subsection{Pseudo-variétés PL}

Il existe une définition classique et très élégante des pseudo-variétés dans le cadre de la topologie linéaire par morceaux (PL). On renvoie à \cite{RourkeSanderson} pour une introduction à la géométrie PL.

\begin{defin}\label{DefinPseudoVarPLClassique}
Un complexe simplicial est une pseudo-variété PL classique de dimension $n$ si
\begin{itemize}
\item tout simplexe est une face d'un simplexe de dimension $n$,
\item tout simplexe de dimension $n-1$ est une face d'exactement deux simplexes de dimension $n$.
\end{itemize}
\end{defin}

Éclaircissons le lien entre pseudo-variété PL et pseudo-variété topologique. La définition de pseudo-variété topologique (Définition \ref{DefinPseudoVarTopologique1}) peut aisément être généralisée au cas PL. Un espace filtré PL de dimension $n$ est un espace PL, $X$, munie d'une filtration par des sous espaces PL
\begin{equation*}
\emptyset\subseteq X^0\subseteq \dots\subseteq X^{n-1}\subseteq X^n=X.
\end{equation*}
De même, on généralise la notion d'espaces localement coniques et de CS espaces au cas PL en demandant que les homéomorphismes soient tous des homéomorphismes PL. Finalement, on arrive à une définition moderne de pseudo-variété PL \cite[Definition 2.5.13]{Friedman}.

\begin{defin}\label{DefinPseudoVarPL}
Une pseudo-variété PL est une pseudo-variété topologique, $X$, telle que
\begin{itemize}
\item l'espace $X$ est un espace filtré PL,
\item les strates de $X$ sont des variétés PL,
\item tout point $x\in X$ admet un voisinage distingué
\begin{equation*}
N\simeq \mathbb{R}^i\times c(L)
\end{equation*}
tel que l'entrelacs $L$ est un CS espace récursif PL et l'homéomorphisme filtré 
\begin{equation*}
N\to \mathbb{R}^i\times  c(L)
\end{equation*}
est un homéomorphisme PL.
\end{itemize}
\end{defin}

On obtient une définition équivalente en adaptant au cas PL la définition \ref{DefinitionAlternativePseudoVarTop}.
On a finalement le résultat suivant \cite[Corollary 2.5.21]{Friedman}.

\begin{prop}
Une pseudo-variété PL classique (au sens de la définition \ref{DefinPseudoVarPLClassique}) est une pseudo-variété PL (au sens de la définition \ref{DefinPseudoVarPL}).
\end{prop}

\subsection{Stratifications de Whitney}

H. Whitney s'est intéressé à la topologie des variétés algébriques réelles et complexes. Celles ci ne sont pas toujours des variétés différentiables, mais elle peuvent toujours être considérée plongée dans une variété algébrique lisse de plus grande dimension. Ainsi, la notion de stratification qu'il a introduit est une notion qui fait intervenir la topologie d'un espace ambiant. On présente ici la définition de stratification qu'il a introduit dans \cite{WhitneyStratification}, en suivant le texte introductif de B. Kloeckner \cite{Kloeckner}. 

On rappelle les conditions de Whitney \ref{ConditionWhitneyA} et \ref{ConditionWhitneyB}

\begin{defin}
Soient $M$ une variété différentiable de dimension $n$ et $R$ et $S$ deux sous variétés de $M$ (non fermées en général). Le couple $(R,S)$ satisfait la condition \ref{ConditionWhitneyA} de Whitney si pour tout $x\in R$ et pour toute suite $(y_k)$ de $S$ tels que 
\begin{itemize}
\item la suite des $y_k$ converge vers $x$,
\item la suite d'espaces tangents $T_{y_k}S$ converge vers un sous-espace vectoriel $\tau\subseteq T_xM$,
\end{itemize}
la condition suivante est vérifiée.
\begin{equation}\label{ConditionWhitneyA}
T_xR\subseteq \tau \tag{A}.
\end{equation}
Le couple $(R,S)$ satisfait la condition \ref{ConditionWhitneyB} de Whitney si pour tout point $x\in R$ et pour toute suite $(y_k)$ de $S$ vérifiant les conditions précédentes, et tels qu'il existe une carte de $M$, $\phi\colon U\to \mathbb{R}^n$, définie sur un voisinage de $x$, telle que les droites $\phi(x)\phi(y_k)$ convergent vers une droite $l\in \mathbb{R}^n$,
la condition suivante est vérifiée.
\begin{equation}\label{ConditionWhitneyB}
l\subseteq d_x\phi(\tau) \tag{B}.
\end{equation}
\end{defin}

\begin{defin}
Soit $M$ une variété différentiable et $X\subseteq M$ un sous-espace. Une stratification de Whitney de $X$ est une partition localement finie $\mathcal{S}=(X^{s})_{s\in S}$ de $X$ en strates, telle que chaque strate est une sous-variété de $M$, et telle que toute paire de strates $(X^r,X^s)$ vérifie les conditions de Whitney \ref{ConditionWhitneyA} et \ref{ConditionWhitneyB}.
\end{defin}

\begin{remarque}
C'est un résultat classique que la condition \ref{ConditionWhitneyB} implique la condition \ref{ConditionWhitneyA}. Ainsi, il suffit de vérifier la condition \ref{ConditionWhitneyB} pour montrer qu'une stratification est de Whiney. De plus, on constate que cette condition est non-vide sur la paire $(X^r,X^s)$ seulement si $X^r\cap \bar{X^s}\not=\emptyset$. On reviendra sur cette condition plus loin dans ce chapitre.
\end{remarque}

Le résultat clé de Whitney, qui motive l'étude des espaces stratifiés, est le suivant.

\begin{theo}[{\cite{WhitneyStratification}}]
Soit $V$ une variété algébrique réelle ou complexe, alors $V$ admet une stratification de Whitney.
\end{theo}

\begin{remarque}
Si $V$ est une variété algébrique, elle peut admettre plusieurs stratifications de Whitney distinctes. Néanmoins si $\mathcal{S}$ et $\mathcal{S'}$ sont deux stratifications de Whitney de $V$, il existe une stratification de Whitney $\mathcal{S}''$ qui raffine les deux.
\end{remarque}
\subsection{Pseudo-variétés de Thom et Mather}
R. Thom et J. Mather se sont inspirés de la définition de H. Whitney pour définir une stratification indépendamment d'un plongement. On renvoie au texte introductif \cite{Kloeckner} pour plus de précisions.

\begin{defin}
Soit $X$ un espace topologique muni d'une partition localement finie $\mathcal{S}=(X^s)_{s\in S}$ telle que chaque strate $X^s$ est une variété différentiable. Un système de tubes est une famille de triplets $T=(T_s,\pi_s,\rho_s)_{s\in S}$, où
\begin{itemize}
\item $T_s$ est un voisinage ouvert de $X^s$ dans $X$,
\item $\pi_s\colon T_s\to X^s$ est une rétraction de l'inclusion $X^s\hookrightarrow T_s$,
\item $\rho_s\colon T_s\to \mathbb{R}^{+}$ est une fonction continue appelée fonction distance,
\end{itemize}
et vérifiant 
\begin{itemize}
\item $\rho_s^{-1}(0)=X^s$,
\item pour toute paire de strates $(X^r,X^s)$ telle que $X^r\subset \bar{X^s}$, la restriction de $(\pi_r,\rho_r)$
\begin{equation*}
T_r\cap X^s\to X^r\times \mathbb{R}^{+},
\end{equation*}
est une submersion lisse,
\item pour tout $x\in T_s\cap T_r$, $\pi_s(x)\in T_r$ et 
\begin{equation*}
\pi_r\pi_s(x)=\pi_r(x) \text{ et } \rho_r(\pi_s(x))=\rho_s(x).
\end{equation*}
\end{itemize}
Le couple $(\mathcal{S},T)$ est appelé stratification de Mather.
\end{defin}

\begin{defin}
Une pseudo-variété de Thom-Mather est un espace topologique muni d'une stratification de Mather.
\end{defin}

La notion de stratification de Mather est effectivement une généralisation de la notion de stratification de Whitney, comme le montre le théorème suivant \cite{Mather}.

\begin{theo}\label{TheoremeWhitneyImpliqueThomMather}
Soit $M$ une variété différentiable et $X\subseteq M$ un sous-espace. Si le sous-espace $X$ admet une stratification de Whitney, il admet une stratification de Mather.
\end{theo}

Ce théorème, ainsi que la proposition \ref{PropositionThomMatherPseudoVar} (voir \cite{Mather}) permet de garantir que toutes les notions de pseudo-variétés considérées jusqu'ici sont des cas particuliers de pseudo-variétés topologiques.

\begin{prop}\label{PropositionThomMatherPseudoVar}
Une pseudo-variété de Thom-Mather est une pseudo-variété topologique.
\end{prop}

On note que le théorème \ref{TheoremeWhitneyImpliqueThomMather} admet une réciproque. Ainsi, la classe des pseudo-variétés de Thom-Mather coïncide avec celle des pseudo-variétés de Whitney.

\begin{theo}[{\cite{Teufel}}]
Pour toute pseudo-variété de Thom-Mather $X$, il existe $n\geq 0$ et un plongement $f\colon X\to \mathbb{R}^n$ tels que $f(X)$ admet une stratification de Whitney.
\end{theo}

\subsection{Espaces homotopiquement stratifiés de Quinn}

F. Quinn a introduit dans \cite{Quinn} la notion d'espaces homotopiquement stratifiés. Cette notion - plus faible que les notions de pseudo-variétés vues précédemment - a de meilleures propriétés vis à vis des homotopies stratifiées. Les strates n'y sont plus supposées être des variétés, et les conditions de trivialités locales sont remplacées par des conditions homotopiques. L'énoncé de la définition nécessite quelques définitions intermédiaires.

\begin{defin}
Soient $X$ un espace topologique et $Y\subset X$ un sous-espace. Le sous-espace $Y$ est dit docile dans $X$ si il existe un voisinage $U$ de $Y$ dans $X$ et une homotopie $H\colon U\times [0,1]\to X$ tels que
\begin{itemize}
\item la restriction $H_{|U\times\{0\}}$ est égale à l'inclusion $U\hookrightarrow X$,
\item l'image de $U$ par la restriction $H_{|U\times\{1 \}}$ est égale à $Y$,
\item Pour tout $t\in [0,1]$, la restriction $H_{|Y\times\{t\}}$ est égale à l'inclusion $Y\hookrightarrow X$,
\item Pour tout $t<1$, on a $
H_{|U\times\{t\}}\left(U\setminus Y\right)\cap Y=\emptyset $
\end{itemize}
\end{defin}

\begin{defin}
Soient $X$ un espace topologique et $Y\subset X$ un sous espace. L'entrelacs homotopique de $Y$ dans $X$ est l'espace de chemin donné par
\begin{equation*}
\Hol(X,Y)=\{\gamma\colon [0,1]\to X\ |\ \gamma(0)\in Y,\ \gamma(t)\in X\setminus Y,\ \forall t>0\}\subset C^{0}([0,1],X).
\end{equation*}
L'évaluation en $0$ fournit une application continue
\begin{equation*}
q\colon \Hol(X,Y)\to Y.
\end{equation*} 
\end{defin}

Muni de ces deux définitions, on arrive alors à la définition de F. Quinn \cite[Definition 3.1]{Quinn}

\begin{defin}
Un espace métrique, filtré de dimension $n$ (voir définition \ref{DefinitionEspaceFiltreClassique}) est homotopiquement stratifié si pour tout $i<k\leq n$, 
\begin{itemize}
\item l'union de strates $X^i\setminus X^{i-1}$ est docile dans $X^i\setminus X^{i-1}\cup X^k\setminus X^{k-1}$,
\item l'évaluation en $0$,
\begin{equation*}
\Hol(X^i\setminus X^{i-1}\cup X^k\setminus X^{k-1},X^i\setminus X^{i-1})\to X^i\setminus X^{i-1}
\end{equation*}
est une fibration.
\end{itemize}
\end{defin}

\section{Homologie d'intersection}
\label{SectionHomologieIntersection}
Pour plus de détails sur l'homologie d'intersection, on pourra consulter les ouvrages \cite{Borel}, \cite{KirwanWoolf} ou \cite{Friedman}. L'approche présentée ici est basée sur \cite{DavidInvarianceTopologique}
\subsection{Simplexes singuliers filtrés}

\begin{defin}
Soit $X$ un espace topologique. Un simplexe singulier de $X$ est une application continue
\begin{equation*}
\sigma\colon \Delta^n\to X,
\end{equation*}
où $\Delta^n$ est le $n$-simplexe standard,
\begin{equation*}
\Delta^n=\{(t_0,\dots,t_n)\ |\ \sum_{i=0}^nt_i=1, \ t_i\geq 0\ \forall 0\leq i \leq n\}\subset \mathbb{R}^{n+1}.
\end{equation*}
\end{defin}

\begin{defin}\label{DefinitionSimplexeSingulierFiltre}
Soit $X$ un espace filtré de dimension $d$ (voir la définition \ref{DefinitionEspaceFiltreClassique}). Un simplexe singulier filtré de $X$ est un simplexe singulier 
\begin{equation*}
\sigma\colon\Delta^n\to X,
\end{equation*}
tel que pour tout $0\leq k\leq d$, 
\begin{equation*}
\sigma^{-1}(X^k)=\Delta^{n_k}\subseteq \Delta^n,
\end{equation*}
où $n_k\leq n$ est un entier et l'inclusion $\Delta^{n_k}\to \Delta^n$ est donnée par
\begin{align*}
\Delta^{n_k}&\to\Delta^{n}\\
(t_0,\dots,t_{n_k})&\mapsto (t_0,\dots,t_{n_k},0,\dots,0)
\end{align*}
On autorise aussi le cas $\sigma^{-1}(X^k)=\emptyset$. Dans ce cas, $\Delta^{n_k}=\emptyset$ et on pose la convention $n_k=-\infty$.
\end{defin}

\begin{exemple}
On considère l'espace filtré représenté figure \ref{FigureSimplexesFiltres}. Cet espace est de dimension $2$ et possède $4$ strates. Une strate régulière, de dimension $2$, en gris, deux strates singulières de dimension $1$ en bleu et vert, et un point singulier, en rouge. Considérons la première paire de simplexes au centre de la figure \ref{FigureSimplexesFiltres}. Le simplexe $[a,b]$ n'est pas filtré car la préimage de la strate verte n'est pas une face de $\Delta^1$. Le simplexe $[d,e,f]$ quant à lui n'est pas filtré car la préimage de la strate bleue est $[d,f]$ qui n'est pas la face initiale de $\Delta^2$. Plus explicitement, l'inclusion $[d,f]\to [d,e,f]$ correspond à l'inclusion
\begin{align*}
\Delta^1&\to \Delta^2\\
(t_0,t_1)&\mapsto (t_0,0,t_1)
\end{align*}
et ne vérifie donc pas la condition de la définition \ref{DefinitionSimplexeSingulierFiltre}. Considérons maintenant la deuxième paire de simplexes de la figure \ref{FigureSimplexesFiltres}, à droite. Ce sont bien des simplexes filtrés. On explicite leur filtration. Le $1$-simplexe $[a,b]$ est filtré comme suit
\begin{equation*}
\emptyset=\emptyset\subset \Delta^0\subset \Delta^1,
\end{equation*}
et le $2$-simplexe $[c,d,e]$ est filtré comme suit
\begin{equation*}
\emptyset\subset\Delta^0\subset\Delta^1\subset \Delta^2.
\end{equation*}
\end{exemple}

\begin{figure}[h!]
\centering
\includegraphics[height=80 pt]{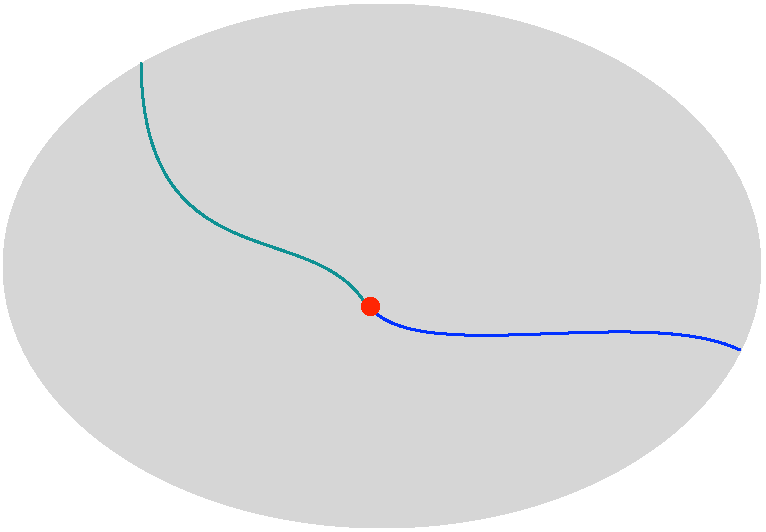}
\hspace{20pt}
\includegraphics[height=80 pt]{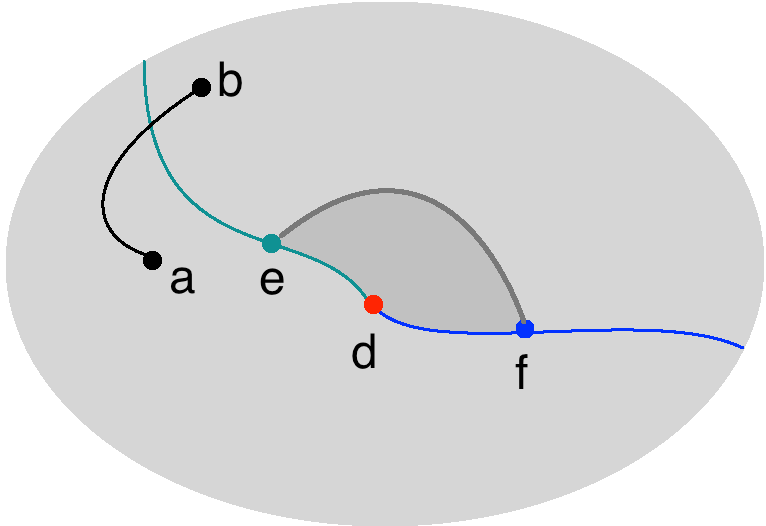}
\hspace{20pt}
\includegraphics[height=80 pt]{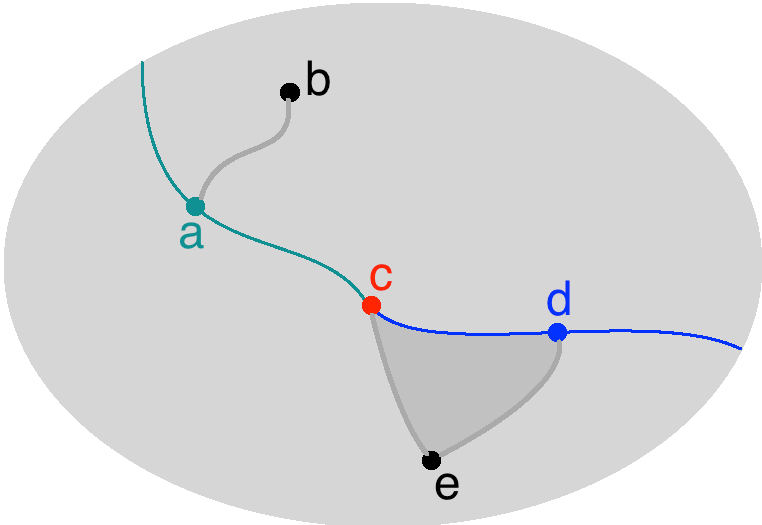}
\caption{Un espace filtré et des simplexes non-filtrés (au centre) et filtrés (à droite).}
\label{FigureSimplexesFiltres}
\end{figure}

\begin{defin}
Soient $G$ un groupe abélien et $X$ un espace filtré. On définit le complexe des chaines singulières filtrées, $C_*(X,G)$, comme suit. Pour tout $n\geq 0$, le groupe $C_n(X,G)$ est défini comme un ensemble de combinaisons linéaires formelles :
\begin{equation*}
C_n(X,G)=\left\{\sum\lambda_i\sigma_i\ |\ \forall i\ \lambda_i\in G, \ \sigma_i\colon \Delta^n\to X \right\}
\end{equation*}
où les sommes $\sum\lambda_i\sigma_i$ sont finies et les simplexes singuliers $\sigma_i\colon \Delta^n\to X\}$ sont filtrés. Pour $n\geq 0$, on définit les applications faces $\partial^{n}\colon C_{n+1}(X,G)\to C_n(X,G)$ sur les simplexes singuliers filtrés par la formule
\begin{equation*}
\partial(\sigma\colon \Delta^n\to X)=\sum_{i=0}^n(-1)^i\sigma\circ D^i,
\end{equation*}
et on étend ensuite par linéarité. Ici, pour $0\leq i\leq n$, l'application $D^i$ correspond à l'inclusion de $\Delta^{n-1}$ dans $\Delta^n$ comme la $i$-ème face.
\end{defin}

\subsection{Perversités et degré pervers}
L'idée clé pour définir l'homologie d'intersection est d'interdire les intersections "trop grandes" entre les simplexes singuliers et les strates singulières. Il y a a priori une ambiguité dans le choix de ce que signifie "trop grande". Lever cette ambiguité c'est choisir une fonction annexe - la perversité - pour contrôler cette intersection.

On rappelle que si $\emptyset\subseteq X^0\subseteq\dots\subseteq X^{d}=X$ est un espace filtré, les strates de $X$ sont les composantes connexes de $X^i\setminus X^{i-1}$ pour $0\leq i\leq d$.

\begin{defin}
Soit $X$ un espace filtré, notons $\mathcal{S}$ son ensemble de strates. Une perversité sur $X$ est une fonction
\begin{equation*}
\bar{p}\colon\mathcal{S}\to \mathbb{Z}.
\end{equation*}
telle que $\bar{p}(S)=0$ dès que $S$ est une strate régulière de $X$ (c'est à dire une composante connexe de $X^d\setminus X^{d-1}$).
\end{defin}

On voudra ensuite comparer la dimension de l'intersection entre un simplexe singulier et une strate de $X$ avec la valeur de la perversité en cette strate. Pour le faire de façon systématique, il est utile d'introduire la notion de degré pervers.

\begin{defin}
Soient $X$ un espace filtré de dimension $d$ et $\sigma\colon \Delta^n\to X$ un simplexe singulier filtré. Le simplexe $\sigma$ induit une filtration sur $\Delta^n$ de la forme suivante (voir la définition \ref{DefinitionSimplexeSingulierFiltre}).
\begin{equation*}
\emptyset\subseteq \Delta^{n_0}\subseteq\dots\subseteq \Delta^{n_{d-1}}\subseteq \Delta^{n_{d}}=\Delta^n.
\end{equation*}
Soit $S$ une strate de $X$ de dimension (formelle) $k$ (c'est à dire une composante connexe de $X^k\setminus X^{k-1}$, le degré pervers de $\sigma$ le long de $S$ est défini comme
\begin{equation*}
\Real{\sigma}_{S}=\left\{
\begin{array}{cl}
-\infty &\text{ si $\sigma(\Delta^n)\cap S=\emptyset$}\\
n_k &\text{ sinon}
\end{array}
\right.
\end{equation*}
Ainsi, le degré pervers d'un simplexe définit une application
\begin{equation*}
\Real{\sigma}\colon \mathcal{S}\to\mathbb{Z}.
\end{equation*}
\end{defin}

On remarque que la filtration de $X$ induit une troisième application $\mathcal{S}\to\mathbb{Z}$ donnée par la codimension.
\begin{align*}
\codim\colon \mathcal{S}&\to\mathbb{Z}\\
S&\mapsto d-\dim(S)
\end{align*}

On peut maintenant formuler la définition d'un simplexe admissible par rapport à une perversité $\bar{p}$. Cette notion correspond à l'intuition que le simplexe intersecte les parties singulières de $X$ de façon contrôlée par $\bar{p}$.

\begin{defin}
Soit $\sigma\colon \Delta^n\to X$ un simplexe singulier filtré. Le simplexe $\sigma$ est admissible par rapport à la perversité $\bar{p}$ (on dit aussi $\bar{p}$-admissible), si la somme des fonctions
\begin{equation*}
n-\Real{\sigma}-\codim+\bar{p}\colon \mathcal{S}\to \mathbb{Z}
\end{equation*}
est positive (ici $n$ est la fonction constante égale à $n=\dim(\sigma)$). Autrement dit, si pour toute strate de $X$, $S$, on a
\begin{equation}\label{EquationSimplexeAdmissible}
\Real{\sigma}_S\leq n-\codim(S)+\bar{p}(S)
\end{equation}
\end{defin}

\begin{remarque}
La définition d'un simplexe singulier filtré $\bar{p}$-admissible fait intervenir les dimensions formelles de l'espace filtré $X$ et des strates $S$. Ainsi, cette définition fait sens dans le contexte très général des espaces filtrés au sens de la définition \ref{DefinitionEspaceFiltreClassique}. Cependant, l'intuition géométrique derrière cette définition se comprend mieux dans le cas où $X$ est une pseudo-variété. Dans ce cas les dimensions formelles correspondent à des dimensions topologiques, et on peut réinterpréter la condition d'admissibilité comme suit. La condition la plus naturelle à imposer aux simplexes singuliers est de demander qu'ils soient transverses aux strates singulières. Au niveau des dimensions, ceci impose pour toute strate singulière $S$ que
\begin{equation*}
\dim(\sigma\cap S)\leq \dim(\sigma)+\dim(S)-\dim(X)=\dim(\sigma)-\codim(S).
\end{equation*}
Dans la définition précédente $\Real{\sigma}_S$ mesure la dimension  $\dim(\sigma\cap S)$, et la perversité apparait alors comme un paramètre qui permet de relâcher ($\bar{p}(S)\geq 0$) ou de renforcer ($\bar{p}(S)\leq 0$) la condition de transversalité.
\end{remarque}

\subsection{Définition de l'homologie d'intersection}

Fixons un espace filtré $X$, $\bar{p}\colon \mathcal{S}\to\mathbb{Z}$ une perversité sur $X$ et $G$ un groupe abélien. On va définir un sous-complexe de chaine $C^{\bar{p}}_*(X,G)\subset C_*(X,G)$ en considérant une condition d'admissibilité adaptée aux chaines.

\begin{defin}
Soit $\xi=\sum_i\lambda_i\sigma_i\in C_n(X,G)$ une chaine de simplexes singuliers filtrés. La chaine $\xi$ est $\bar{p}$-admissible si tout les simplexes singuliers filtrés apparaissant dans sa décomposition (les $\sigma_i$) sont $\bar{p}$-admissibles. Elle est de $\bar{p}$-intersection si
\begin{itemize}
\item elle est $\bar{p}$-admissible,
\item son bord $\partial(\xi)$ est $\bar{p}$-admissible.
\end{itemize}
\end{defin}

\begin{defin}\label{DefinitionHomologieIntersection}
On définit le complexe des chaines singulières de $\bar{p}$-intersection, $C^{\bar{p}}_*(X,G)$ comme le sous complexe de $C_*(X,G)$ engendré par les chaines de $\bar{p}$-intersections. L'homologie d'intersection de $X$ à coefficients dans $G$ par rapport à la perversité $\bar{p}$ est définie comme l'homologie de $C^{\bar{p}}_*(X,G)$
\begin{equation*}
I^{\bar{p}}H_k(X,G)=H_k(C^{\bar{p}}_*(X,G))=\frac{\ker\left(\partial\colon C^{\bar{p}}_k(X,G)\to C^{\bar{p}}_{k-1}(X,G)\right)}{\text{Im}\left(\partial\colon C^{\bar{p}}_{k+1}(X,G)\to C^{\bar{p}}_{k}(X,G)\right)}
\end{equation*}
\end{defin}

\begin{remarque}\label{RemarquePerversiteGenerales}
La définition \ref{DefinitionHomologieIntersection} n'est pas la seule définition possible de l'homologie d'intersection. Le fait que la définition donnée ici coïncide avec la définition originale de M. Goresky et R. MacPherson \cite{IntersectionHomologyI} lorsque cette dernière est définie est un théorème de D. Chataur M. Saralegui et D. Tanré \cite[Theorem A]{DavidInvarianceTopologique}. Lorsque les perversités deviennent trop grandes (lorsque $\bar{p}(S)>\codim(S)-2$ pour une strate $S$), il devient nécessaire de modifier la définition \ref{DefinitionHomologieIntersection} pour que l'homologie d'intersection conserve des propriétés satisfaisantes (notamment la dualité de Poincaré). En effet, dans le cas $\bar{p}(S)\geq codim(S)-1$, la condition \ref{EquationSimplexeAdmissible} autorise un simplexe admissible à avoir une de ces faces entièrement incluse dans une strate singulière. On renvoie à \cite{DavidInvarianceTopologique} pour plus de détails sur le sujet.
\end{remarque}

\subsection{Propriétés de l'homologie d'intersection}
Pour éviter les complications entrainées par l'utilisation de perversités trop générales (voir la remarque \ref{RemarquePerversiteGenerales}), on suppose que $\bar{p}(S)\leq \codim(S)-2$ pour toute strate singulière $S$. 

La première propriété à laquelle on s'intéresse est celle de la naturalité de l'homologie d'intersection. Il existe de nombreuses variantes pour définir une classe de morphismes ayant de bonnes prorpiétés vis à vis de l'homologie d'intersection. On reprend ici la définition  \cite[Definition 1.6]{DavidInvarianceTopologique}.

\begin{defin}\label{DefinitionCodimStratifiee}
Soient $X$ et $Y$ deux espaces filtrés. Une application continue $f\colon X\to Y$ est codim-stratifiée si pour toute strate $S$ de $X$, il existe $S^{f}$ une strate de $Y$, telle que $f(S)\subset S^{f}$, et $\codim_Y(S^f)\leq \codim_X(S)$.
\end{defin}

\begin{remarque}
On parle ici d'application codim-stratifiée et non d'application stratifiée comme le font les auteurs dans \cite{DavidInvarianceTopologique}. La raison de cette distinction est qu'on définira plus tard des applications stratifiées en omettant la condition sur la codimension.
\end{remarque}

\begin{prop}[Naturalité par rapport aux applications stratifiées]
\label{PropositionNaturaliteHomologieIntersection}
Soient $X$ et $Y$ deux espaces filtrés, $f\colon X\to Y$ une application codim-stratifiée et $\bar{p}$, $\bar{q}$ des perversités sur $X$ et $Y$ respectivement. Si pour toute strate de $X$, $S$ on a 
\begin{equation*}
\bar{p}(S)\leq \bar{q}(S^f),
\end{equation*}
alors l'application $f$ induit un morphisme de complexes de chaines
\begin{align*}
C^{\bar{p}}_*(X,G)&\xrightarrow{f_*} C^{\bar{q}}_*(X,G)\\
\sigma\colon \Delta^n\to X&\mapsto f\circ\sigma\colon \Delta^n\to Y
\end{align*}
En particulier, on obtient une application en homologie 
\begin{equation*}
f_*\colon I^{\bar{p}}H_*(X,G)\to I^{\bar{q}}H_*(Y,G)
\end{equation*}
\end{prop}
On renvoie à \cite[Proposition 3.11]{DavidInvarianceTopologique} pour une preuve complète de ce résultat. L'idée clé est que si $f$, $\bar{p}$ et $\bar{q}$ sont comme dans l'énoncé, alors pour tout simplexe $\bar{p}$-admissible de $X$, $\sigma$, $f\circ\sigma$ est un simplexe $\bar{q}$-admissible de $Y$.

\begin{defin}
Soient $X$ un espace filtré, $\bar{p}$ une perversité sur $X$ et $U\subset X$ un sous espace. Alors on définit la filtration induite sur $U$ par 
\begin{equation*}
\emptyset\subseteq U\cap X^0\subseteq \dots\subseteq U\cap X^{d-1}\subseteq U\cap X^{d}=U.
\end{equation*}
De plus, les strates de $U$ sont obtenues comme les composantes connexes des intersections $S\cap U$ où $S$ est une strate de $X$. Ainsi, on étend la définition de $\bar{p}$ à $U$ en posant $\bar{p}(S\cap U)=\bar{p}(S)$.
\end{defin}

\begin{prop}[Suite exacte longue de Mayer-Vietoris]
Soient $X$ un espace filtré, $U,V\subset X$ deux ouverts tels que $X=U\cup V$ et $\bar{p}$ une perversité sur $X$. Alors on a une suite exacte longue en homologie d'intersection
\begin{equation*}
\dots\to I^{\bar{p}}H_i(U\cap V,G)\to I^{\bar{p}}H_i(U,G)\oplus I^{\bar{p}}H_i(V,G)\to I^{\bar{p}}H_i(X,G)\to I^{\bar{p}}H_{i-1}(U\cap V,G)\to\dots
\end{equation*}
\end{prop}

On renvoie à \cite[Proposition 4.1]{DavidInvarianceTopologique} pour une preuve.

Soit $X$ un espace filtré. On a défini dans la section \ref{SectionPseudoVarTop} le cône ouvert de $X$, $c(X)$ (Définition \ref{DefinitionConeOuvertFiltreClassique}). 
L'espace filtré $c(X)$ contient une strate de dimension $0$, le sommet du cône noté $v$, et des strates de dimension $\geq 1$, de la forme $S\times ]0,1[$ où $S$ est une strate de $X$. Ainsi, étant donnée une perversité $\bar{p}$ sur $X$ si on veut l'étendre en une perversité sur $c(X)$ il suffit de définir la valeur sur la strate $v$. Par abus de notation, on appellera la perversité correspondante sur $c(X)$ $\bar{p}$, bien que celle-ci dépende du choix de la valeur $\bar{p}(v)$. Avec ces notations, on a l'isomorphisme canonique suivant \cite[Proposition 5.2]{DavidInvarianceTopologique}.

\begin{prop}[Formule du cône]
Soient $X$ un espace filtré compact de dimension $d$ et $\bar{p}$ une perversité sur $X$. On suppose que $\bar{p}(v)\leq d-1$. On a un isomorphisme 
\begin{equation*}
I^{\bar{p}}H_k(c(X),G)\simeq \left\{\begin{array}{cl}
I^{\bar{p}}H_k(X,G) &\text{ si $k< d-\bar{p}(v)$}\\
0 &\text{ si $k\geq d-\bar{p}(v)$}
\end{array}\right.
\end{equation*}
où l'isomorphisme dans le cas $k<d-\bar{p}(v)$ est induit par l'inclusion
\begin{equation*}
X\simeq X\times\{1/2\}\hookrightarrow c(X).
\end{equation*}
voir la propriété de naturalité \ref{PropositionNaturaliteHomologieIntersection}.
\end{prop}

La formule du cône permet de constater immédiatement que l'homologie d'intersection n'est pas invariante par homotopie. En effet, pour tout espace topologique $X$, le cône ouvert $c(X)$ est homotopiquement équivalent à un point, alors que son homologie d'intersection $I^{\bar{p}}H_*(c(X),G)$ est en général non-triviale d'après la proposition précédente. Cependant, l'homologie d'intersection est invariante par homotopie stratifiée. On rend cette affirmation plus précise à travers les définitions suivantes.

\begin{defin}
Soit $X$ un espace filtré de dimension $d$ et $K$ un espace topologique. On munit le produit $X\times K$ de la filtration
\begin{equation*}
\emptyset\subseteq X^0\times K\subseteq \dots\subseteq X^{d-1}\times K\subset X^d\times K=X\times K.
\end{equation*}
\end{defin}

\begin{defin}
Soit $f,g\colon X\to Y$ deux applications codim-stratifiées entre espaces filtrés. Une homotopie stratifiée entre $f$ et $g$ est une application codim-stratifiée 
\begin{equation*}
H\colon X\times [0,1]\to Y,
\end{equation*}
telle que $H(-,0)=f$ et $H(-,1)=g$.
On dit alors que $f$ et $g$ sont homotopes au sens stratifié. Si $f\colon X\to Y$ et $g\colon Y\to X$ sont deux applications codim-stratifiées telles que $f\circ g$ et $\Id_Y$ d'une part et $g\circ f$ et $\Id_X$ d'autre part sont homotopes au sens stratifiés, alors on dit que $f$ et $g$ sont des équivalences d'homotopie stratifiées inverses l'une de l'autre.
\end{defin}

On a alors la propriété suivante \cite[Proposition 3.13]{DavidInvarianceTopologique}.

\begin{prop}[Invariance par homotopie stratifiée]
Soient $f,g\colon X\to Y$ deux applications codim-stratifiées, homotopes au sens stratifié, et $\bar{p}, \bar{q}$ sont des perversités sur $X$ et $Y$ vérifiant les conditions de la proposition \ref{PropositionNaturaliteHomologieIntersection} par rapport aux applications $f$ et $g$. Alors, les applications induites en homologie d'intersection sont égales.
\begin{equation*}
f_*=g_*\colon I^{\bar{p}}H_*(X,G)\to I^{\bar{p}}H_*(Y,G).
\end{equation*}
En particulier, si $f\colon X\to Y$ est une équivalence d'homotopie stratifiée, $f$ induit un isomorphisme en homologie d'intersection
\begin{equation*}
f_*\colon I^{\bar{p}}H_*(X,G)\simeq I^{\bar{p}}H_*(Y,G).
\end{equation*}
\end{prop}

\begin{remarque}
En fait, l'hypothèse sur les codimensions dans la définition de morphismes codim-stratifiés peut être complètement omise, on obtient ainsi la notion de morphismes stratifiés. Dans ce cas, il faut modifier la condition de la proposition sur la naturalité de l'homologie d'intersection (Proposition \ref{PropositionNaturaliteHomologieIntersection}) en 
\begin{equation*}
\bar{p}(S)-codim(S)\leq \bar{q}(S^f)-codim(S^f).
\end{equation*}
On remarque aussi que si $f$ et $g$ sont codim-stratifiées et que $H$ est une homotopie stratifiée entre $f$ et $g$, alors $H$ vérifie automatiquement l'hypothèse sur les codimensions. En effet, si $S$ est une strate de $X$ de codimension $c$, et $S^f$ est la strate de $Y$ correspondante, alors $S\times [0,1]$ est une strate de $X\times [0,1]$ de codimension $c$, et $S^H=S^f$. Ainsi, la définition d'homotopie stratifiée se généralise immédiatement en omettant toute mention de la codimension, et si deux morphismes $f$ et $g$ sont codim-stratifiés et homotope par une homotopie stratifiée, on a automatiquement que l'homotopie entre $f$ et $g$ est codim-stratifiée, et on retrouve donc la définition \cite[Definition 1.9]{DavidInvarianceTopologique}
\end{remarque}

\section{Espaces stratifiés au dessus d'un ensemble ordonné}
\label{SectionChap1TopP}
\subsection{La condition de frontière et l'ensemble ordonné des strates}
\label{SectionFrontiere}
Les espaces stratifiés que nous avons considéré dans la section \ref{SectionPseudoVar} vérifiaient tous une condition topologique particulière appelée condition de frontière, que l'on définit ici.

\begin{defin}[Condition de frontière]
Soit $X$ un espace topologique partitionné en strates $(X^s)_{s\in \mathcal{S}}$. La stratification vérifie la condition de frontière si pour toute paire de strates $(X^r,X^s)$ la condition suivante est vérifiée
\begin{equation*}
X^r\cap \bar{X^s}\not =\emptyset \text{ si et seulement si } X^r\subseteq \bar{X^s}.
\end{equation*}
\end{defin}

On définit alors la relation binaire sur $\mathcal{S}$, $r\leq s\Leftrightarrow X^r\subseteq \bar{X^s}$. On a immédiatement la proposition suivante. On renvoie à \cite[Proposition 1.2]{Kloeckner} pour une preuve élémentaire.

\begin{prop}
Soit $X$ un espace stratifié vérifiant la condition de frontière et dont les strates sont localement fermées. La relation $\leq$ sur l'ensemble des strates de $X$ est une relation d'ordre.
\end{prop}

Par cohérence avec les notations du reste de ce texte, on notera $P_X$ l'ensemble ordonné des strates de $X$. Considérons maintenant la notion d'application stratifiée de la définition \ref{DefinitionCodimStratifiee}, en omettant les hypothèses sur les codimensions. Si $X$ et $Y$ sont deux espaces stratifiés, vérifiant la condition de frontière, et dont les strates sont localement fermées, une application stratifiée $f\colon X\to Y$ induit un morphisme d'ensemble $P_X\to P_Y$. De plus, si $S$ et $R$ sont des strates de $X$ vérifiant $S\subseteq \bar{R}$, on a $f(S)\subseteq f(\bar{R})\subseteq\overline{f(R)}$. On en déduit que $S^f\subseteq \overline{R^f}$, et le morphisme induit par $f$, $P_X\to P_Y$ est un morphisme d'ensembles ordonnés. Ceci nous amène naturellement à la définition abstraite d'espaces stratifiés de la section suivante.

\subsection{La catégorie des espaces stratifiés}

\begin{defin}
Soit $P$ un ensemble ordonné. On le munit de la topologie de l'ordre (aussi appelée topologie d'Alexendroff) dont les ouverts sont les ensembles $U$ vérifiant la propriété suivante.
\begin{equation*}
\forall p\in U, \left( \ q\geq p\Rightarrow q\in U\right).
\end{equation*}
\end{defin}

\begin{defin}
Un espace stratifié est la donnée de
\begin{itemize}
\item un espace topologique $X$,
\item un ensemble ordonné $P_X$,
\item une application continue $\varphi_X\colon X\to P_X$.
\end{itemize}
Les strates de $X$ sont les préimages $\varphi_X^{-1}(p)\subset X$, où $p\in P_X$.
Un morphisme stratifié entre deux espaces stratifiés $\strat{X}$ et $\strat{Y}$ est la donnée d'une application continue $f\colon X\to Y$ et d'un morphisme d'ensembles ordonnés $\alpha\colon P_X\to P_Y$ tel que le diagramme suivant commute.
\begin{equation*}
\begin{tikzcd}
X
\arrow{r}{f}
\arrow[swap]{d}{\varphi_X}
&Y
\arrow{d}{\varphi_Y}
\\
P_X
\arrow[swap]{r}{\alpha}
&P_Y
\end{tikzcd}
\end{equation*}
La catégorie $\Strat$ est la catégorie des espaces stratifiés et des applications stratifiées.
\end{defin}

\begin{remarque}
Bien qu'elle puisse paraitre plus générale, la notion de morphisme stratifié qu'on a donné ici coïncide avec celle qu'on a étudié plus tôt dans le cas espaces stratifiés vérifiant la condition de frontière. En effet, pour ces espaces une application $f\colon X\to Y$ est stratifiée si pour toute strate $S$ de $X$, $f(S)$ est inclus dans une unique strate de $Y$, notée $S^f$. Mais alors, la correspondance $S\mapsto S^f$ définit un morphisme $\alpha\colon P_X\to P_Y$, et la condition de frontière garantit que ce morphisme préserve l'ordre. En partie pour cette raison, on identifiera souvent le morphisme de $\Strat$ $(f,\alpha)$ avec l'application continue $f$.
\end{remarque}

\subsection{Homotopies stratifiées et la catégorie des espaces filtrés}
\label{SectionHomotopieStratifieesEtTopP}
\begin{defin}
Soient $\strat{X}$ et $\strat{Y}$ deux espaces stratifiés et 
\begin{equation*}
(f,\alpha),(g,\beta)\colon \strat{X}\to\strat{Y}
\end{equation*}
deux applications stratifiées. Elles sont homotopes au sens stratifié si $\alpha=\beta$ et si il existe une application continue $H\colon X\times [0,1]\to Y$ telle que $H(-,0)=f$, $H(-,1)=g$ et telle que $(H,\alpha)$ est une application stratifiée de $(X\times [0,1],P_X, \varphi_X\circ \pr_X)$ vers $\strat{Y}$.
\end{defin}

\begin{remarque}
On peut reformuler la définition précédente comme suit. Deux applications stratifiées $f$ et $g$ sont homotopes au sens stratifié si il existe une homotopie $H$ entre $f$ et $g$ qui est elle même une application stratifiée. Pour donner un sens à l'affirmation "$H$ est une application stratifiée", il faut fixer une stratification du cylindre $X\times [0,1]$. On considère en fait la composée
\begin{equation*}
X\times [0,1]\xrightarrow{\pr_X}X\xrightarrow{\varphi_X}P_X.
\end{equation*}
Cette stratification implique en particulier que si $x\in X$ est un point, alors le chemin $t\mapsto H(x,t)$ est dans la même strate de $Y$ pour tout $0\leq t\leq 1$. L'existence d'une telle homotopie implique en particulier que les applications $f$ et $g$ induise le même morphisme sur les ensembles ordonnés de strates, ou autrement dit que $\alpha=\beta$.
\end{remarque}

\begin{defin}
Une application stratifiée 
\begin{equation*}
(f,\alpha)\colon \strat{X}\to\strat{Y}
\end{equation*}
est une équivalence d'homotopie stratifiée si il existe une application stratifiée
\begin{equation*}
(g,\beta)\colon \strat{Y}\to\strat{X}
\end{equation*}
telle que $(f,\alpha)\circ(g,\beta)$ et $(g,\beta)\circ(f,\alpha)$ sont respectivement homotopes au sens stratifié à $(\Id_X,\Id_{P_X})$ et $(\Id_Y,\Id_{P_Y})$. Dans ce cas, on dit que $(f,\alpha)$ et $(g,\beta)$ sont des équivalences d'homotopie stratifiées inverse l'une de l'autre.
\end{defin}

\begin{remarque}
En particulier, si $(f,\alpha)\colon \strat{X}\to\strat{Y}$ est une équivalence d'homotopie stratifiée, $f$ est une équivalence d'homotopie entre $X$ et $Y$. La réciproque est fausse. Si $f$ est une équivalence d'homotopie, alors $(f,\alpha)$ n'est pas nécessairement une équivalence d'homotopie stratifiée. La remarque suivante donne une obstruction simple, et on trouvera des exemples plus élaborés dans le chapitre \ref{ChapitreExemples}.
\end{remarque}

\begin{remarque}
Si $(f,\alpha)$ et $(g,\beta)$ sont des équivalences d'homotopies stratifiées inverses l'une de l'autre, alors $\alpha$ et $\beta$ sont des isomorphismes d'ensembles ordonnés inverses l'un de l'autre. En effet, par définition d'une homotopie stratifiée, si il existe une homotopie stratifiée entre $(f,\alpha)\circ (g,\beta)=(f\circ g,\alpha\circ \beta)$ et $(\Id_Y,\Id_{P_Y})$, on doit donc avoir $\alpha\circ \beta= \Id_{P_Y}$, de même pour la composition $\beta\circ\alpha$.
\end{remarque}

Ainsi si $(f,\alpha)$ est une équivalence d'homotopie stratifiée entre $\strat{X}$ et $\strat{Y}$, quitte à remplacer $\varphi_Y$ par $\alpha^{-1}\circ\varphi_Y$, on peut supposer que $P_Y=P_X$ et que $\alpha=\Id_{P_X}$. 
Puisqu'on s'intéresse aux propriétés homotopiques des espaces stratifiés, ceci suggère fortement de travailler dans une catégorie où les espaces sont tous stratifiés au dessus d'un même ensemble ordonné, $P$, et où le morphisme induit sur $P$ est l'identité. De façon à éviter toute confusion entre les deux contextes, on parlera dans le cas où $P$ est fixé d'espaces \textbf{filtrés} au dessus de $P$. On donne la définition précise suivante.

\begin{defin}\label{DefinitionChap1TopP}
Soit $P$ un ensemble ordonné. Un espace filtré au dessus de $P$ est la donnée de
\begin{itemize}
\item un espace topologique $X$,
\item une application continue $\varphi_X\colon X\to P$.
\end{itemize}
Une application filtrée entre deux espace filtrés $\fil{X}$ et $\fil{Y}$ est une application continue $f\colon X\to Y$ telle que le diagramme suivant commute
\begin{equation*}
\begin{tikzcd}
X
\arrow{rr}{f}
\arrow[swap]{dr}{\varphi_X}
&&Y
\arrow{dl}{\varphi_Y}
\\
&P
\end{tikzcd}
\end{equation*}
La catégorie $\Top_P$ est la catégorie des espaces filtrés au dessus de $P$ et des applications filtrées.
\end{defin}

\subsection{Espaces localement coniques et pseudo-variétés topologiques dans la catégorie $\Strat$}
\label{SectionPseudoVarTopP}
\begin{defin}\label{DefinitionVoisinageStratifie}
Soit $\strat{X}$ un espace stratifié, et $x\in X$ un point de $X$. Un voisinage stratifié de $x$ est la donnée d'un espace stratifié $\strat{U}$ et d'un morphisme stratifié
\begin{equation*}
(i,\alpha)\colon \strat{U}\to\strat{X},
\end{equation*}
tels que $i\colon U\to X$ est l'inclusion d'un voisinage  ouvert de $x$ dans $X$, et tels que $\alpha\colon P_U\to P_X$ est l'inclusion d'un sous-ensemble ordonné.
\end{defin}

\begin{remarque}
La définition précédente peut sembler arbitrairement compliquée. La notion de voisinage est élémentaire en topologie, et on pourrait définir naïvement un voisinage d'un point $x$ de $\strat{X}$ comme un voisinage $U$ de $x$ dans $X$ muni de la stratification $(\varphi_X)_{|U}\colon U\to P_X$. Le problème est que cette restriction de $\varphi_X$ n'a aucune raison d'être surjective ($\varphi_X$ elle même n'est pas supposée surjective), et en définissant un voisinage ainsi, on conserverait une information globale : l'ensemble ordonné $P_X$. La définition \ref{DefinitionVoisinageStratifie} évite partiellement cet écueil en autorisant un choix sur l'ensemble ordonné $P_U$. Ceci permet à la fois d'utiliser les voisinages pour décrire des propriétés locales et de considérer que $\strat{X}$ est un ouvert de $\strat{X}$, même si la stratification $\varphi_X$ n'est pas surjective.
\end{remarque}

\begin{defin}
Soit $P$ un ensemble ordonné. On définit $c(P)$, le cône de $P$, comme l'ensemble ordonné obtenu en adjoignant à $P$ un élément minimal, noté $-\infty$.
\end{defin}

\begin{defin}
Soit $\strat{X}$ un espace stratifié. On définit $c\strat{X}$, le cône stratifié de $\strat{X}$ comme
\begin{equation*}
c\strat{X}=(c(X),c(P_X),c(\varphi_X)),
\end{equation*}
où $c(X)$ est le cône ouvert de l'espace topologique $X$ (voir la définition \ref{DefinitionConeOuvertFiltreClassique}), et la stratification $c(\varphi_X)$ est définie par
\begin{align*}
\c(\varphi_X)\colon c(X)&\to c(P)\\
(x,t)&\mapsto \varphi_X(x), \text{ si $t>0$}\\
v&\mapsto -\infty
\end{align*}
On rappelle que $v$ désigne le sommet du cône, et que tout point du cône est de la forme $(x,t)$, où $x\in X$ et $t\in [0,1]$, avec $(x,0)=(x',0)=v\in c(X)$ pour tout $x,x'\in X$.
\end{defin}

\begin{defin}
Soient $\strat{X}$ un espace stratifié et $M$ un espace topologique non-filtré. On définit le produit de $\strat{X}$ et de $M$ comme l'espace stratifié
\begin{equation*}
M\times \strat{X}= (M\times X,P_X,\varphi_X\circ\pr_X),
\end{equation*}
où $\pr_X\colon M\times X\to X$ est la projection sur $X$.
\end{defin}

On peut alors donner une définition d'un espace coniquement stratifié dans $\Strat$.

\begin{defin}\label{DefinitionConiquementStratifieChap1}
Soit $\strat{X}$ un espace stratifié. C'est un espace coniquement stratifié si pour tout $x\in X$, il existe un voisinage stratifié $\strat{N}$ de $x$ dans $\strat{X}$, un espace stratifié $\strat{L}$, un voisinage $U$ de $x$ dans $\varphi_X^{-1}(\varphi_X(x))$ ainsi qu'un homéomorphisme stratifié
\begin{equation*}
(h,\alpha)\colon \strat{N}\to U\times c(\strat{L})
\end{equation*}
\end{defin}

\begin{remarque}\label{RemarqueConiquementStratifieLocalementConique}
Cette définition est équivalente à la définition de ce que J. Lurie appelle les "conically stratified spaces" \cite[Definition A.5.5]{HigherAlgebra}. Par ailleurs, on a vu avec la remarque \ref{RemarqueFiltreClassiqueEgalFiltre} qu'un espace filtré au sens classique correspondait à un espace stratifié au dessus de l'ensemble ordonné $\{0<1<\dots<n\}$. Ainsi, la définition \ref{DefinitionConiquementStratifieChap1} est aussi équivalente à la définition \ref{DefinitionLocalementConique}.
\end{remarque}

On peut aussi donner une définition de pseudo-variété topologique dans la catégorie $\Strat$ en s'inspirant de la définition \ref{DefinitionAlternativePseudoVarTop}.

\begin{defin}
Soit $\strat{X}$ un espace stratifié. C'est une pseudo-variété stratifiée de dimension $n$ si $X$ est séparé et si pour tout $x\in X$, il existe $0\leq i\leq n$, un voisinage stratifié $\strat{N}$ de $x$ dans $\strat{X}$, une pseudo-variété stratifiée de dimension $n-i-1$ $\strat{L}$ (vide si et seulement si $i=n$), ainsi qu'un homéomorphisme stratifié
\begin{equation}\label{EquationHomeoPseudoVarStrat}
(h,\alpha)\colon \strat{N}\to \mathbb{R}^i\times c(\strat{L})
\end{equation}
\end{defin}

\begin{prop}
Une pseudo-variété stratifiée est une pseudo-variété topologique.
\end{prop}

\begin{proof}
Soient $\strat{X}$ une pseudo-variété stratifiée, et $x\in X$. Notons $S\subseteq X$ la strate de $\strat{X}$ contenant $x$. La restriction de l'homéomorphisme stratifiée \ref{EquationHomeoPseudoVarStrat} à $S\cap N$ fournit un homéomorphisme
\begin{equation*}
h_{|S}\colon S\cap N\to \mathbb{R}^i.
\end{equation*}
On en déduit que la strate contenant $x$ est une variété topologique de dimension $i$. Ceci permet de définir une filtration sur $X$
\begin{equation*}
\emptyset\subseteq X^0\subseteq \dots\subseteq X^n=X
\end{equation*}
en prenant pour $X^k$ l'union des strates de $\strat{X}$ qui sont des variétés de dimension $\leq k$. Les strates de cette filtration coïncide par construction avec les strates de $\strat{X}$. Ainsi, d'après la remarque \ref{RemarqueConiquementStratifieLocalementConique}, l'espace filtré $X$ est un CS espaces (voir la définition \ref{DefinPseudoVarTopologique1}) récursif. De plus, comme les entrelacs de strates de dimension $<n$ sont non vides, l'union des strates de dimension $n$ est dense dans $X$, et $X$ est donc une pseudo-variété.
\end{proof}

\chapter{Théorie de l'homotopie classique}
Ce chapitre est une introduction à la théorie de l'homotopie classique. A travers le théorème de Whitehead (Théorème \ref{TheoremeWhiteheadClassique}), on reprend quelques idées importantes qui ont menées à la définition des catégories modèles. On discute ensuite de l'équivalence entre la théorie homotopique des espaces topologiques et celle des ensembles simpliciaux. L'ordre choisi pour présenter les résultats reflète l'ordre dans lequel on procédera pour explorer la théorie homotopique des espaces stratifiés dans les chapitres suivants.

La première section de ce chapitre est centrée sur le théorème de Whitehead (Théorème \ref{TheoremeWhiteheadClassique}). 
Dans la section \ref{SectionHomotopiesGroupesHomotopiesChap2}, on définit la notion d'homotopie ainsi que les groupes d'homotopie des espaces topologique. Dans la section \ref{SectionEnsemblesSimpliciauxChap2}, on présente les ensembles simpliciaux ainsi que l'adjonction entre la catégorie des espaces topologiques et celle des ensembles simpliciaux. On discute ensuite dans les sections \ref{SectionHomotopieGroupesHomotopieSimpliciauxChap2} et \ref{SectionComplexesDeKan} des analogues simpliciaux aux groupes d'homotopie, et des complexes de Kan - les ensembles simpliciaux pour lesquels ces invariants ont de bonnes propriétés. On conclut cette section par une esquisse de preuve du théorème de Whitehead (voir le théorème \ref{TheoremeWhiteheadSimplicialChap2}).

Dans la section \ref{SectionCategoriesModeles}, on introduit la notion de catégorie modèle. Après avoir rappelé les axiomes (section \ref{SectionAxiomesCMF}), on étudie la notion d'homotopie dans une catégorie modèle (section \ref{SectionHomotopiesCMF}), et on présente la catégorie homotopique d'une catégorie modèle (section \ref{SectionCategorieHomotopiqueCMF}). Dans la section \ref{SectionCMFEngendrementCofibrant}, on présente le cas particulier des catégories modèles à engendrement cofibrant. Dans ces catégories, les fibrations sont caractérisées par leurs propriétés de relèvement par rapport à un ensemble de morphismes connus. Ceci permet de ramener beaucoup de preuves à l'étude de cet ensemble de flèches par un argument dû à D. Quillen appelé "argument du petit objet". Celui-ci est rappelé à la fin de la section. (On note que, à l'exception de $\sSTop_P$ pour laquelle la question est ouverte, toutes les catégories modèles que nous étudierons dans ce texte seront à engendrement cofibrant.) Dans la section \ref{SectionCMFSimpliciale}, on introduit la notion de catégorie modèle simpliciale.

Dans la section \ref{SectionAdjonctionKanQuillenChap2}, on montre comment la notion de catégorie modèle permet de montrer que les théories homotopiques des espaces topologiques et des ensembles simpliciaux sont équivalentes. On définit la catégorie modèle des ensembles simpliciaux dans la section \ref{SectionCMFsSChap2} et celle des espaces topologique dans la section \ref{SectionCMFTopChap2}. Finalement, dans la section \ref{SubsectionAdjonctionKanQuillenChap2}, on présente le théorème classique de Kan-Quillen affirmant que l'adjonction 
\begin{equation*}
\Real{-}\colon \sS\leftrightarrow\Top\colon \Sing
\end{equation*}
induit une équivalence entre les catégories homotopiques (Théorème \ref{TheoremeEquivalenceKanQuillenChap2}).
\label{ChapterRappelHomotopie}

\section{Théorèmes de Whitehead}

Dans cette section, on donne les définitions nécessaires pour comprendre et prouver le théorème classique de H. Whitehead \cite{WhiteheadCombinatorial1}.

\begin{theo}\label{TheoremeWhiteheadClassique}
Soient $X$ et $Y$ deux CW-complexes et $f\colon X\to Y$ une application continue. L'application $f$ est une équivalence d'homotopie si et seulement si, pour tout $n\geq 0$ et pour tout $x\in X$, $f$ induit des isomorphismes
\begin{equation*}
f_*\colon \pi_n(X,x)\to\pi_n(Y,f(x))
\end{equation*}
\end{theo}

\subsection{Homotopies et groupes d'homotopie}
\label{SectionHomotopiesGroupesHomotopiesChap2}
\begin{defin}
Soient $X$ et $Y$ deux espaces topologiques et $f,g\colon X\to Y$ deux applications continues. On dit que $f$ et $g$ sont homotopes si il existe une application continue
\begin{equation*}
H\colon X\times[0,1]\to Y,
\end{equation*}
telle que $f=\left(x\mapsto H(x,0)\right)$ et $g=\left(x\mapsto H(x,1)\right)$, dans ce cas, on dit que $f$ et $g$ sont homotopes par $H$ et on note $f\sim g$.
Soit $f\colon X\to Y$ une application continue. On dit que $f$ est une équivalence d'homotopie si il existe une application continue $g\colon Y\to X$ telle que les composées $f\circ g$ et $g\circ f$ sont respectivement homotopes à $\Id_Y$ et $\Id_X$. Dans ce cas, on dit que $f$ et $g$ sont des équivalences d'homotopie inverses l'une de l'autre.
\end{defin}

\begin{remarque}
Si $(X,x)$ et $(Y,y)$ sont des espaces pointés (c'est à dire qu'on a choisit des points $x\in X$ et $y\in Y$), deux applications $f,g\colon X\to Y$ vérifiant $f(x)=y$ et $g(x)=y$ sont homotopes au sens pointé si il existe une homotopie entre $f$ et $g$, $H\colon X\times [0,1]\to Y$ telle que $H(x,t)=y$ pour tout $0\leq t\leq 1$. On omettra par la suite de préciser au sens pointé en on notera les applications $f,g\colon (X,x)\to (Y,y)$ pour lever l'ambiguïté.
\end{remarque}

\begin{prop}
La relation d'homotopie ${\sim}$ est une relation d'équivalence.
\end{prop}

\begin{proof}
\begin{itemize}
\item Soit $f\colon X\to Y$ une application continue. La composée $X\times [0,1]\xrightarrow{\pr_X} X\xrightarrow{f}Y$ fournit une homotopie entre $f$ et $f$, ${\sim}$ est donc réflexive.
\item Soient $f,g\colon X\to Y$ deux applications continues et $H\colon X\times [0,1]\to Y$ une homotopie entre $f$ et $g$. Alors l'application
\begin{align*}
G\colon X\times [0,1]&\to Y\\
(x,t)&\mapsto H(x,1-t)
\end{align*}
est une homotopie entre $g$ et $f$, la relation ${\sim}$ est donc symétrique.
\item Soient $f,g,h\colon X\to Y$ des applications continues et $H_1,H_2\colon X\times [0,1]\to Y$ des homotopies respectivement entre $f$ et $g$ et entre $g$ et $h$. On définit l'homotopie
\begin{align*}
G\colon X\times [0,1]&\to Y\\
(x,t)&\mapsto\left\{\begin{array}{cl}
H_1(x,2t)&\text{ si $0\leq t\leq 1/2$}\\
H_2(x,2t-1) &\text{ si $1/2<t\leq 1$}
\end{array}\right.
\end{align*}
Alors, $G$ est une homotopie entre $f$ et $h$. La relation ${\sim}$ est donc transitive.
\end{itemize}
\end{proof}

\begin{defin}
Soient $X$ et $Y$ deux espaces topologiques. On définit l'ensemble des classes d'homotopies d'applications de $X$ vers $Y$ comme le quotient 
\begin{equation*}
[X,Y]=C^0(X,Y)/{\sim}.
\end{equation*}
Si $(X,x)$ et $(Y,y)$ sont des espaces pointés, on note $[(X,x),(Y,y)]$ les classes d'homotopie pointée correspondantes.
\end{defin}

\begin{defin}
Soit $X$ un espace topologique. Son ensemble de composantes connexes par arcs est défini comme
\begin{equation*}
\pi_0(X)=[*,X]
\end{equation*}
où $*$ est un point.
\end{defin}

\begin{defin}\label{DefinitionGroupeHomotopieClassique}
Soit $(X,x)$ un espace topologique pointé et $n\geq 0$. On définit le $n$-ième groupe d'homotopie de $X$ comme
\begin{equation*}
\pi_n(X,x)=[(S^n,*),(X,x)],
\end{equation*}
où $S^n\subset \mathbb{R}^{n+1}$ est la sphère unité et $*\in S^n$ est un point.
\end{defin}

\begin{remarque}
Si $n=1$ on retrouve la définition du groupe fondamental de $X$.
\end{remarque}

\begin{remarque}
Le $0$-ième groupe d'homotopie de $(X,x)$, $\pi_0(X,x)$ n'est pas un groupe, c'est l'ensemble des composantes connexes par arcs de $X$, pointé par la composante connexe par arcs de $x$ dans $X$. Si $n\geq 1$, le $n$-ième groupe d'homotopie de $(X,x)$ est un groupe, et il est abélien dès que $n\geq 2$.
\end{remarque}

\begin{prop}\label{PropositionCheminIsomorphismeGroupesHomotopies}
Soient $X$ un espace topologique, $x,x'\in X$ et $\gamma\colon [0,1]\to X$ un chemin vérifiant $\gamma(0)=x$ et $\gamma(1)=x'$. Alors, pour tout $n\geq 1$, $\gamma$ induit un morphisme de groupes
\begin{equation*}
\gamma_*\colon \pi_n(X,x)\to\pi_n(X,x')
\end{equation*}
\end{prop}

\begin{prop}\label{PropositionHomotopieEgalesGroupesHomotopiesChap2}
Soit $f\colon (X,x)\to (Y,y)$ une application continue et $n\geq 0$. L'application $f$ induit un morphisme 
\begin{equation*}
f_*\colon \pi_n(X,x)\to \pi_n(Y,y),
\end{equation*}
qui est un morphisme de groupes dès que $n\geq 0$. Si $g\colon (X,x)\to (Y,y)$ est une seconde application continue et que $f$ et $g$ sont homotopes, alors pour tout $n\geq 0$, $f_*=g_*$.
\end{prop}

\begin{remarque}\label{RemarqueSensDirectWhiteheadChap2}
Les propositions \ref{PropositionHomotopieEgalesGroupesHomotopiesChap2} et \ref{PropositionCheminIsomorphismeGroupesHomotopies} impliquent le sens direct du théorème de Whitehead \ref{TheoremeWhiteheadClassique}. En effet, si $f\colon X\to Y$ est une équivalence d'homotopie, il existe une équivalence d'homotopie inverse $g\colon Y\to X$. Mais alors, comme $f\circ g$ et $\Id_Y$ sont homotopes, $f_*\circ g_*$ est un isomorphisme de groupes. De même, comme $g\circ f$ et $\Id_X$ sont homotopes, $g_*\circ f_*$ est un isomorphisme de groupes. Ainsi, $f_*$ et $g_*$ sont des isomorphismes inverses l'un de l'autre. En particulier, $f$ induit des isomorphismes entre tous les groupes d'homotopie.
\end{remarque}

\subsection{Ensembles simpliciaux}
\label{SectionEnsemblesSimpliciauxChap2}
Dans le cas simplement connexe - c'est à dire lorsque $\pi_0(X)=\pi_0(Y)=*$ et $\pi_1(X,x)=\pi_1(Y,y)=0$ - le théorème de Whitehead est souvent reformulé sous la forme suivante

\begin{theo}
Soient $X$ et $Y$ deux CW-complexes simplement connexes et $f\colon X\to Y$ une application continue. C'est une équivalence d'homotopie si et seulement si pour tout $n\geq 0$, $f$ induit des isomorphismes en homologie,
\begin{equation*}
f_*\colon H_n(X)\to H_n(Y).
\end{equation*}
\end{theo}

Ceci suggère que l'homologie d'un espace contient une grande partie de l'information sur son type d'homotopie, au moins dans le cas simplement connexe. 
On rappelle que pour $n\geq 0$ le simplexe standard $\Delta^n$ est l'espace topologique
\begin{equation*}
\Delta^n=\{(t_0,\dots,t_n)\in \mathbb{R}^{n+1}\ |\ \forall i, \ t_i\geq 0\ \sum_{i=0}^nt_i=1\}.
\end{equation*}
On sait que l'homologie singulière est définie à partir de l'ensemble des applications continues de la forme
\begin{equation*}
\sigma\colon \Delta^n\to X,
\end{equation*}
où $\Delta^n$, $n\geq 0$ est un simplexe standard, que l'on appelle simplexes singuliers. 
Puisque cet ensemble de simplexes permet de calculer l'homologie de $X$, il est naturel de se demander quelle information il contient sur $X$. En particulier, peut on calculer les groupes d'homotopie de $X$ à partir de son ensemble de simplexes singuliers? Pour répondre à cette question, détaillons plus précisément les propriétés de cet ensemble, que l'on notera $\Sing(X)$.

On constate d'abord que cet ensemble porte une graduation naturelle, par la dimension des simplexes standards. Ainsi, pour tout $n\geq 0$, on notera
\begin{equation*}
\Sing(X)_n=\{\sigma\colon \Delta^n\to X\}.
\end{equation*}
Ensuite, si $\sigma\in \Sing(X)_n$ est un simplexe singulier avec $n\geq 1$, on peut restreindre $\sigma$ à l'une des faces de $\Delta^n$. Énumérons les faces de $\Delta^n$. Soit $0\leq i\leq n$, on définit les inclusions de faces $D^i$ comme suit
\begin{align*}
D^i\colon\Delta^{n-1}&\to\Delta^n\\
(t_0,\dots,t_{n-1})&\mapsto(t_0,\dots,t_{i-1},0,t_i,\dots,t{n-1})
\end{align*}
Alors, les faces de codimension $1$ de $\Delta^n$ sont de la forme $D^i(\Delta^{n-1})$ pour $0\leq i\leq n$. Mais alors, pour tout $0\leq i\leq n$, la précomposition avec $D^i$ fournit une application
\begin{align*}
\Sing(X)_n&\xrightarrow{-\circ D^i}\Sing(X)_{n-1}\\
\left(\sigma\colon\Delta^n\to X\right)&\mapsto \left(\sigma\circ D^i\colon \Delta^{n-1}\to X\right)
\end{align*}
On notera $d_i$ cette application. De la même façon, pour $0\leq i\leq n$, on considère les contractions
\begin{align*}
S^i\colon \Delta^{n+1}&\to\Delta^n\\
(t_0,\dots,t_{n+1})&\mapsto (t_0,\dots,t_i+t_{i+1},t_{i+2},\dots,t_{n+1})
\end{align*}
Les précompositions avec $S^i$ induisent alors des applications
\begin{align*}
\Sing(X)_n&\xrightarrow{-\circ S^i}\Sing(X)_{n+1}\\
\left(\sigma\colon \Delta^n\to X\right)&\mapsto\left(\sigma\circ S^i\colon \Delta^{n+1}\to X\right)
\end{align*}
qu'on appellera dégénérescences. De même, en considérant n'importe quelle composition de $D^i$ et de $S^j$, $\alpha\colon \Delta^m\to \Delta^n$, on obtient une application
\begin{equation}\label{EquationMorphismeInduitAlphaSing}
\Sing(X)_n\xrightarrow{-\circ\alpha}\Sing(X)_m.
\end{equation}

Ainsi, on peut décrire la structure portée par $\Sing(X)$ en disant qu'il s'agit d'un ensemble gradué tel que pour toute composition de faces et de dégénérescences, $\alpha\colon \Delta^m\to \Delta^n$, on dispose d'un morphisme \ref{EquationMorphismeInduitAlphaSing}. Cependant, pour que la description soit complète, il est nécessaire de décrire les relations que vérifient ces différents morphismes. Une approche est de définir explicitement les relations vérifiées par les $D^i$ et $S^j$ (Voir par exemple \cite{FriedmanSimplicialSets}), mais on peut aussi procéder de façon plus conceptuelle. 

Soit $n\geq 0$, pour $0\leq i\leq n$, on note $e_i$ le $i+1$-ème sommet de $\Delta^n$. Autrement dit, $e_i=(t_0,\dots,t_{n})$, avec $t_i=1$ et $t_j=0$ pour tout $j\not=i$. Les applications $D^i$ et $S^j$ étant linéaires, il suffit de décrire l'image des sommets pour les décrire complètement, il en va de même pour leurs compositions qui seront, elles aussi, linéaires. Ainsi, on se ramène à l'étude d'applications entre ensembles :
\begin{equation}\label{EquationAlphaEntreEnsembles}
\alpha\colon \{0,\dots,m\}\to\{0,\dots,n\}.
\end{equation}
De plus, on remarque que les applications $D^i$ et $S^j$ - et donc leurs composées aussi - préservent l'ordre des sommets. Ainsi, il suffit de considérer les applications \ref{EquationAlphaEntreEnsembles} vérifiant pour tout $0\leq k\leq l\leq m$, $\alpha(k)\leq \alpha(l)$. 

\begin{defin}
On définit la catégorie $\Delta$ dont les objets sont les entiers, et dont les morphismes entre $m$ et $n$ sont les applications croissantes de $\{0,\dots,m\}$ vers $\{0,\dots,n\}$. Plus explicitement :
\begin{equation*}
\Hom_{\Delta}(m,n)=\{\alpha\colon \{0,\dots,m\}\to\{0,\dots,n\}\ |\ \forall\ k,l\in \{0,\dots,m\},\ k\leq l\ \Rightarrow \alpha(k)\leq \alpha(l)\}
\end{equation*}
\end{defin}

\begin{defin}
On définit la catégorie des ensembles simpliciaux comme la catégorie de foncteurs :
\begin{equation*}
\sS=\Fun(\Delta^{\op},\Set).
\end{equation*}
\end{defin}

\begin{remarque}
Un ensemble simplicial $K$ est donc une collection d'ensembles $K_n$, $n\geq 0$, et de morphismes $\{d_i\}$, $\{s_i\}$ de la forme
\begin{equation*}
d_i\colon K_n \to K_{n-1} \text{ et } s_i\colon K_n\to K_{n+1}.
\end{equation*}
Satisfaisant certaines propriétés. Si $K$ et $L$ sont des ensembles simpliciaux, un morphisme $f\colon K\to L$ est une collection d'applications
\begin{equation*}
f_n\colon K_n\to L_n
\end{equation*}
vérifiant $f_n\circ d^K_i=d^L_i\circ f_n$ et $f_n\circ s^K_i=s^L_i\circ f_n$, pour tout $0\leq i\leq n$.
\end{remarque}

\begin{defin}
L'association $X\mapsto \Sing(X)$ définit un foncteur
\begin{equation*}
\Sing\colon \Top\to \sS.
\end{equation*}
Si $f\colon X\to Y$ est une application continue, l'image de $f$ par $\Sing$ est définie comme
\begin{align*}
\Sing(X)&\to\Sing(Y)\\
\left(\sigma\colon\Delta^n\to X\right)&\mapsto \left(f\circ\sigma\colon \Delta^n\to Y\right)
\end{align*}
\end{defin}

\begin{remarque}\label{RemarqueYoneda1}
Il existe un foncteur pleinement fidèle $\Delta\to \sS$ (le plongement de Yoneda)
de la forme
\begin{align*}
\Delta&\to\sS\\
n&\mapsto \left(m\mapsto \Hom_{\Delta}(m,n)\right)
\end{align*}
Moins abstraitement, pour tout entier $n\geq 0$, on note $\Delta^n=\Hom_{\Delta}(-,n)$ l'image du plongement de Yoneda. Alors, pour tout $k\geq 0$, $(\Delta^n)_k$ est l'ensemble des dégénérescences de faces de $\Delta^n$ de dimension $k$. Ainsi, l'ensemble simplicial $\Delta^n$ encode l'ensemble des faces et dégénérescences du simplexe standard $\Delta^n$. C'est pour cette raison qu'on utilise la même notation. Pour tout $n\geq 0$ on dira que $\Delta^n\in \sS$ est un simplexe.
\end{remarque}

\begin{remarque}\label{RemarqueYonedaEnsemblesSimpliciauxChap2}
Toujours par le lemme de Yoneda, si $K$ est un ensemble simplicial, alors 
\begin{equation*}
K_n=\Hom_{\sS}(\Delta^n,K).
\end{equation*}
De plus, on a 
\begin{equation*}
K\sim\colim_{\sigma\in K}\Delta^n.
\end{equation*}
Autrement dit, tout ensemble simplicial est isomorphe à la colimite de ses simplexes. Ce constat permet de définir des foncteurs $F\colon \sS\to \mathcal{C}$, où $\mathcal{C}$ est une catégorie cocomplète, en les définissant sur $\Delta\subset\sS$ et en posant
\begin{equation*}
F(K)=\colim_{\sigma\in K}F(\Delta^n).
\end{equation*}
On appelle ce procédé "extension de Kan à gauche", ou extension par colimites.
\end{remarque}

\begin{defin}
Soit $\Delta^n\in \sS$ un simplexe, on définit sa réalisation comme le simplexe standard $\Delta^n$. On le notera $\Real{\Delta^n}$ à partir de maintenant pour lever l'ambiguïté. Ceci permet de définir le foncteur de réalisation
\begin{align*}
\Real{-}\colon\sS &\to\Top\\
K&\mapsto \colim_{\sigma\in K}\Real{\Delta^n}
\end{align*}
\end{defin}

\begin{prop}\label{PropositionAdjonctionSingRealChap2}
La paire de foncteurs $(\Real{-},\Sing)$ est une paire de foncteurs adjoints.
\end{prop}

\begin{proof}
Soient $K$ un ensemble simplicial et $Y$ un espace topologique. Par définition de $\Real{K}$ comme une colimite, un morphisme $f\colon \Real{K}\to Y$ correspond à la donnée d'une collection de morphismes, compatibles aux applications faces et dégénérescences,
de la forme
\begin{equation*}
f_{\sigma}\colon \Real{\Delta^n}\to Y,
\end{equation*}
où $\sigma\in K_n$. Par définition de $\Sing(Y)$, les $f_{\sigma}$ sont des éléments de $\Sing(Y)_n$. Ainsi, par la remarque \ref{RemarqueYoneda1}, on a une collection de morphismes $\Delta^n\xrightarrow{f_{\sigma}} \Sing(Y)$, compatibles aux faces et dégénérescences. Ceci permet de définir
\begin{equation*}
K=\colim_{\sigma\in K}\Delta^n\xrightarrow{\colim f_{\sigma}}\Sing(Y).
\end{equation*}
L'application induite par l'association $f\mapsto \colim f_{\sigma}$ est clairement bijective, les foncteurs $\Real{-}$ et $\Sing$ sont donc adjoints.
\end{proof}

Il est temps de donner la définition de CW-complexe \cite{WhiteheadCombinatorial1}.

\begin{defin}
Un espace topologique $X$ est un complexe cellulaire si il est séparé et si il existe une partition de $X$ en cellules ouvertes, $X=\coprod e^n_{\alpha}$, telle que pour toute cellule $e^n_{\alpha}$ il existe une application continue
\begin{equation*}
f_{\alpha}\colon \Real{\Delta^n}\to \bar{e^n_{\alpha}}\subset X
\end{equation*}
vérifiant
\begin{itemize}
\item  
$(f_{\alpha})_{|\Real{\Delta^n}\setminus\Real{\partial\Delta^n}}\colon \Real{\Delta^n}\setminus\Real{\partial\Delta^n}\to \bar{e^n_{\alpha}}$
est un homéomorphisme sur son image, d'image $e^n_{\alpha}$,
\item l'image $f_{\alpha}(\Real{\partial(\Delta^n)})$ est incluse dans l'union des cellules de $X$ de dimension $\leq n-1$.
\end{itemize}
Un sous-espace $Y\subset X$ est un sous-complexe cellulaire si $Y$ est une union de cellules de $X$, et que pour tout $e^n_{\alpha}\subset Y$, $\bar{e^n_{\alpha}}\subset Y$.
L'espace topologique $X$ est un CW-complexe si c'est un complexe cellulaire muni de la topologie faible - c'est à dire qu'un sous ensemble $U\subset X$ est ouvert dans $X$ si et seulement si, pour toute cellule $e^n_{\alpha}\subset X$, $U\cap e^n_{\alpha}$ est ouvert dans $e^n_{\alpha}$ - et si pour tout point $x\in X$, il existe un sous-complexe fini $Y_x\subset X$ contenant $x$.
\end{defin}

\begin{remarque}
La définition de CW-complexe est technique, mais elle décrit une intuition élémentaire. Un CW-complexe est un espace topologique obtenu comme recollement de cellules ouvertes.
\end{remarque}

Grâce au théorème suivant, dû à J. Milnor \cite{MilnorGeometricRealization}, on dispose d'une famille d'exemples de CW-complexes. On rappelle que si $K$ est un ensemble simplicial, un simplexe $\sigma\in K_n$ est dit dégénéré si il existe $\tau\in K_{n-1}$, tel que $\sigma=s^i(\tau)$ pour un certain $0\leq i\leq n-1$. Un simplexe est non-dégénéré si il n'est pas dégénéré.

\begin{theo}\label{TheoremeEnsembleSimplicialCWComplexe}
Soit $K$ un ensemble simplicial, sa réalisation $\Real{K}$ est un CW-complexe dont les cellules de dimension $n$ sont les simplexes non-dégénérés de $K$.
\end{theo}

\begin{remarque}
Il existe une réciproque partielle au théorème \ref{TheoremeEnsembleSimplicialCWComplexe}. Si $X$ est homotopiquement équivalent à un CW-complexe, alors $X$ est homotopiquement équivalent à la réalisation d'un ensemble simplicial. Cette réciproque est une conséquence du théorème de Whitehead (Théorème \ref{TheoremeWhiteheadClassique}) et des résultats que l'on verra plus tard dans ce chapitre.
\end{remarque}

\subsection{Homotopies et groupes d'homotopie pour les ensembles simpliciaux}
\label{SectionHomotopieGroupesHomotopieSimpliciauxChap2}
\begin{defin}\label{DefinitionProduitEnsemblesSimpliciauxChap2}
Soient $K$ et $L$ deux ensembles simpliciaux. On définit l'ensemble simplicial produit, $K\times L$ par
\begin{equation*}
(K\times L)_n=K_n\times L_n
\end{equation*}
pour tout $n\geq 0$. De plus, pour $0\leq i\leq n$, on pose
\begin{equation*}
d_i=d^K_i\times d^L_i\colon K_n\times L_n\to K_{n-1}\times L_{n-1}
\end{equation*}
et
\begin{equation*}
s_i=s^K_i\times s^L_i\colon K_n\times L_n\to K_{n+1}\times L_{n+1}
\end{equation*}
Ceci définit un foncteur
\begin{equation*}
-\times -\colon \sS\times\sS\to\sS
\end{equation*}
\end{defin}

\begin{defin}
Soient $f,g\colon K\to L$ deux morphismes entre ensembles simpliciaux. On dit que $f$ et $g$ sont homotopes si il existe un morphisme $H\colon K\times\Delta^1\to L$ tel qu'on ait les égalités
\begin{equation*}
H\circ i_0=f \text{ et } H\circ i_1= g,
\end{equation*}
où $i_0,i_1\colon K\simeq K\times\Delta^0\to K\times \Delta^1$ correspondent aux deux inclusions $\Delta^0\to \Delta^1$.
\end{defin}

\begin{remarque}
On a un homéomorphisme $\Real{\Delta^1}\simeq [0,1]$, et \cite{MilnorGeometricRealization} prouve que $\Real{K\times\Delta^1}\simeq \Real{K}\times \Real{\Delta^1}$. On en déduit que si $f,g\colon K\to L$ sont homotopes, alors leurs réalisations $\Real{f},\Real{g}\colon \Real{K}\to\Real{L}$ sont homotopes. Autrement dit, le foncteur $\Real{-}$ préserve les homotopies.
\end{remarque}

\begin{remarque}\label{RemarqueHomotopieSimplicialPasEquivalence}
La relation d'homotopie entre applications simpliciales n'est pas une relation d'équivalence, elle n'est par exemple pas symétrique en général. On parlera donc d'une homotopie de $f$ vers $g$ ou de $g$ vers $f$ lorsqu'il sera nécessaire de préciser. Cependant, on peut définir la relation d'équivalence engendrée par la relation d'homotopie. Autrement dit, on dira que $f,g\colon K\to L$ sont dans la même classe d'homotopie  si il existe une suite d'applications simpliciale $f^0,\dots,f^n\colon K\to L$ telles que $f^0=f$, $f^n=g$, et telle que pour tout $0\leq i\leq n-1$, il existe une homotopie de $f_i$ vers $f_{i+1}$ ou une homotopie de $f_{i+1}$ vers $f_i$. Dans ce cas, on notera $f\sim g$. Comme dans le cas topologique, on note 
\begin{equation*}
[K,L]=\frac{\Hom(K,L)}{{\sim}}
\end{equation*}
\end{remarque}

On souhaite définir les groupes d'homotopie pour les ensembles simpliciaux. Pour imiter la définition \ref{DefinitionGroupeHomotopieClassique}, il suffit de trouver un équivalent de la sphère unité. On remarque que si $n\geq 0$, $\partial(\Delta^{n+1})$, l'union des faces de dimension $n$ de $\Delta^{n+1}$ est un modèle pour la sphère unité. En effet, on a $\Real{\partial(\Delta^{n+1})}\simeq S^n$. On est donc amené à poser la définition (naïve) suivante.

\begin{defin}\label{DefinitionGroupeHomotopieNaif}
Soit $K$ un ensemble simplicial, $x\in K_0$ un sommet et $n\geq 0$ un entier, le $n$-ième groupe d'homotopie naïf de $(K,x)$ est défini comme
\begin{equation*}
\pi_n^{\naif}(K,x)=[(\partial(\Delta^{n+1}),*),(K,x)],
\end{equation*}
où $*\in \left(\partial(\Delta^{n+1})\right)_0$.
\end{defin}

On a alors immédiatement la propriété suivante

\begin{prop}\label{PropositionPiSingEgalPiTop}
Soit $(X,x)$ un espace topologique pointé, on a un isomorphisme naturel
\begin{equation*}
\pi_n(X,x)=\pi_n^{\naif}(\Sing(X),x),
\end{equation*}
où $x\in X$ est vu comme un $0$-simplexe de $\Sing(X)$, $x\colon \Delta^0\to X$.
\end{prop}

\begin{proof}
Comme $\Real{\partial(\Delta^{n+1})}\simeq S^n$, l'adjonction $(\Real{-},\Sing)$ fournit une bijection entre ensembles 
\begin{equation}\label{EquationIsomorphismeHomAdjonctionRealSing}
\Hom_{\Top}(S^n,X)\simeq \Hom_{\sS}(\partial(\Delta^{n+1}),\Sing(X)).
\end{equation}
De plus, l'adjonction $(\Real{-},\Sing)$ préserve les homotopies. On en déduit que le morphisme \ref{EquationIsomorphismeHomAdjonctionRealSing} induit un isomorphisme entre  $\pi_n(X,x)$ et $\pi_n^{\naif}(\Sing(X),x)$.
\end{proof}

\begin{remarque}\label{RemarqueGroupeHomotopieNaifFail}
La proposition \ref{PropositionPiSingEgalPiTop} garantit que la définition \ref{DefinitionGroupeHomotopieNaif} est compatible avec la définition des groupes d'homotopie d'un espace topologique (Définition \ref{DefinitionGroupeHomotopieClassique}). Cependant la définition \ref{DefinitionGroupeHomotopieNaif} n'est pas complètement satisfaisante. Prenons par exemple $K=\partial(\Delta^2)$. On note ses sommets $e_0,e_1,e_2$. Alors, les simplexes non-dégénérés de $K$ sont les simplexes $[e_0],[e_1],[e_2], [e_0,e_2],[e_0,e_1]$ et $[e_1,e_2]$. Calculons le premier groupe d'homotopie naïf de $K$. On a
\begin{equation*}
\Hom_{\sS}((\partial(\Delta^{2}),e_0),(K,e_0))=\{\Id_K,f,g,h\},
\end{equation*}
avec $f\colon \sigma\mapsto [e_0]$, $g$ définie comme suit

\begin{align*}
g([e_0])= [e_0], \ g([e_1])=[e_1], \ g([e_2])=[e_1]\\
g([e_0,e_1])=[e_0,e_1], \ g([e_0,e_2])=[e_0,e_1],\ g([e_1,e_2])=[e_1,e_1],
\end{align*}
et $h$ définie comme $g$ en échangeant les rôles de $e_1$ et $e_2$. On vérifie que $f\sim g\sim h,\not\sim \Id_K$. Ainsi, le premier groupe d'homotopie de $K$, 
\begin{equation*}
\pi_1^{\naif}(K,e_0)=\{\Id,f\}
\end{equation*}
contient deux éléments. D'autre part, on sait que $\Real{K}\simeq S^1$, et $\pi_1(S^1,*)\simeq \mathbb{Z}$. Les foncteurs $\pi_n^{\naif}$ ne sont donc pas compatibles à la réalisation en général. Par ailleur, $\pi_n^{\naif}(K)$ n'est pas un groupe en général.
\end{remarque}

\subsection{Complexes de Kan et théorèmes de Whitehead}
\label{SectionComplexesDeKan}
On a vu dans la remarque \ref{RemarqueGroupeHomotopieNaifFail} que les foncteurs $\pi_n^{\naif}$ avaient de bonnes propriétés par rapport aux ensembles simpliciaux de la forme $\Sing(X)$ mais de mauvaises propriétés en général. Explicitions les propriétés qui distinguent l'ensemble simplicial $\Sing(X)$. 

\begin{defin}
Soit $n\geq 0$ et $0\leq k\leq n$. Le cornet $\Lambda^n_k\subset \Delta^n$ est le sous-ensemble simplicial de $\Delta^n$ engendré par les faces propres de $\Delta^n$ différentes de $D^k(\Delta^{n-1})$. Alternativement, on décrit $\Lambda^n_k$ comme l'union
\begin{equation*}
\Lambda^n_k=\bigcup_{j\not= k}D^j(\Delta^{n-1})\subset \Delta^n
\end{equation*}
\end{defin}

\begin{defin}\label{DefinitionComplexeDeKan}
Un complexe de Kan est un ensemble simplicial $K$ tel que pour toute inclusion de cornet $\Lambda^n_k\to \Delta^n$, $n\geq 1$, $0\leq k\leq n$, et tout morphisme $\alpha\colon\Lambda^n_k\to K$, il existe un morphisme $\beta\colon \Delta^n\to K$ tel que $\beta_{|\Lambda^n_k}=\alpha$. On formule souvent cette condition sous la forme d'un diagramme commutatif.
\begin{equation*}
\begin{tikzcd}
\Lambda^n_k
\arrow{d}
\arrow{r}{\alpha}
&K
\\
\Delta^n
\arrow[swap,dashrightarrow]{ur}{\beta}
\end{tikzcd}
\end{equation*}
\end{defin}

\begin{prop}
Soit $X$ un espace topologique, l'ensemble simplicial $\Sing(X)$ est un complexe de Kan.
\end{prop}

\begin{proof}
Soit $X$ un espace topologique, $\Lambda^n_k\to \Delta^n$ une inclusion de cornet et $\alpha\colon \Lambda^n_k\to \Sing(X)$ un morphisme. Par adjonction, $\alpha$ correspond à un unique morphisme $\widetilde{\alpha}\colon \Real{\Lambda^n_k}\to X$. On a alors le diagramme suivant
\begin{equation*}
\begin{tikzcd}
\Real{\partial(\Lambda^n_k)}
\arrow{r}{\widetilde{\alpha}}
\arrow{d}
&X\\
\Real{\Delta^n}
\arrow[swap, dashrightarrow]{ur}{\beta}
\end{tikzcd}
\end{equation*}
Mais alors, si $s\colon\Real{\Delta^n}\to\Real{\Lambda^n_k}$ est une section de l'inclusion $\Real{\Lambda^n_k}\to\Real{\Delta^n}$, la composée $\beta=\alpha\circ s$ fournit une solution au problème de relèvement, par adjonction. On construit aisément une telle section, ce qui conclut la preuve.
\end{proof}

La proposition classique suivante justifie de travailler avec les complexes de Kan. 

\begin{prop}
Soient $K$ et $L$ deux ensembles simpliciaux tels que $L$ est un complexe de Kan. La relation d'homotopie sur $\Hom(K,L)$ est une relation d'équivalence.
\end{prop}

En fait, les complexes de Kan sont les objets appropriés pour étudier l'homotopie des ensembles simpliciaux. On a le théorème suivant, dont on verra une preuve dans les sections suivantes, qui est une version simpliciale du théorème de Whitehead.

\begin{theo}\label{TheoremeWhiteheadSimplicialChap2}
Soient $K$ et $L$ deux complexes de Kan et $f\colon K\to L$ une application simpliciale. L'application $f$ est une équivalence d'homotopie simpliciale si et seulement si $f$ induit des isomorphismes sur tous les groupes d'homotopie (naïfs).
\end{theo}

On peut ensuite déduire le théorème de Whitehead classique, (Théorème \ref{TheoremeWhiteheadClassique}) du théorème \ref{TheoremeWhiteheadSimplicialChap2}. On esquisse ici la stratégie de preuve.

\begin{proof}[Démonstration du théorème \ref{TheoremeWhiteheadClassique}]
Soient $X,Y$ deux CW-complexes et $f\colon X\to Y$ une application continue. Si $f$ est une équivalence d'homotopie, on a vu à la remarque \ref{RemarqueSensDirectWhiteheadChap2} que $f$ induisait des isomorphismes sur tout les groupes d'homotopie, ce qui montre le sens direct du théorème \ref{TheoremeWhiteheadClassique}. Le coeur de la preuve de la réciproque est contenu dans le diagramme commutatif suivant.

\begin{equation}\label{DiagrammePreuveWhiteheadChap2}
\begin{tikzcd}[column sep = huge]
\phantom{X}
&\Real{\Sing(Y)}
\arrow{r}{\Real{\widetilde{g}}}
&\Real{\Sing(X)}
\arrow{r}{\Real{\Sing(f)}}
\arrow{d}{\ev_X}
&\Real{\Sing(Y)}
\arrow{d}{\ev_Y}
\\
X
\arrow{d}{i_X}
\arrow{r}{h}
&Y
\arrow{r}{g}
\arrow{u}{i_Y}
&X
\arrow{r}{f}
&Y
\\
\Real{\Sing(X)}
\arrow{r}{\Real{\widetilde{h}}}
&\Real{\Sing(Y)}
\arrow{u}{\ev_Y}
\arrow{r}{\Real{\Sing(g)}}
&\Real{\Sing(X)}
\arrow{u}{\ev_X}
\end{tikzcd}
\end{equation}
Si $f$ induit des isomorphismes sur tout les groupes d'homotopies, c'est aussi le cas de $\Sing(f)$, par la proposition \ref{PropositionPiSingEgalPiTop}. Mais alors, par le théorème \ref{TheoremeWhiteheadSimplicialChap2}, $\Sing(f)$ est une équivalence d'homotopie simpliciale et admet donc un inverse à homotopie près, $\widetilde{g}\colon \Sing(Y)\to \Sing(X)$. Comme la réalisation préserve les homotopies, $\Real{\widetilde{g}}$ est un inverse à homotopie près de $\Real{\Sing(f)}$. En utilisant l'unité $i_Y$ et la counité $\ev_X$ de l'adjonction $(\Real{-},\Sing)$, on définit $g$ comme la composition 
\begin{equation*}
g=ev_X\circ \Real{\widetilde{g}}\circ i_Y\colon Y\to X.
\end{equation*}
On vérifie ensuite qu'on peut construire une homotopie entre $f\circ g$ et $\Id_Y$ à partir de l'homotopie entre $\Real{\Sing(f)}\circ\Real{\widetilde{g}}$ et $\Id_{\Real{\Sing(Y)}}$. L'application $f$ est donc un inverse à gauche de $g$ à homotopie près, et on en déduit en particulier que $g$ induit des isomorphismes entre tous les groupes d'homotopies. En réappliquant la construction à $g$, on obtient un inverse à droite de $g$, $h$, à homotopie près. Comme $g$ admet un inverse à gauche et à droite, c'est une équivalence d'homotopie, et on en déduit que $f$ est une équivalence d'homotopie inverse à $g$.
\end{proof}

\begin{remarque}
La preuve du théorème \ref{TheoremeWhiteheadClassique} donnée par H. Whitehead dans \cite{WhiteheadCombinatorial1} ne fait pas intervenir les ensembles simpliciaux. En particulier, elle est différente de celle qu'on a présentée ici. On a choisi de présenter cette preuve plus moderne pour deux raisons. D'abord, la notion de catégorie modèle que l'on présente dans la section suivante est apparue comme la formalisation de l'idée qu'il était possible d'étudier l'homotopie des espaces à partir des ensembles simpliciaux - idée qui est à l'oeuvre dans cette preuve. Ensuite, parce que la stratégie de preuve qu'on a présenté ici est exactement celle qu'on utilisera au chapitre \ref{ChapitreGroupesHomotopiesFiltresEspaces} pour prouver la version filtrée du théorème de Whitehead (Théorème \ref{PremierTheoremeWhitehead}.
\end{remarque}

\section{Catégories modèles}
\label{SectionCategoriesModeles}
\subsection{Les axiomes d'une catégorie modèle}
\label{SectionAxiomesCMF}
\begin{defin}
Soient $\C$ une catégorie, et $i\colon A\to B$, $p\colon X\to Y$ deux morphismes de $\C$. On dit que $i$ a la propriété de relèvement à gauche (\textbf{L}eft \textbf{L}ifting \textbf{P}roperty) par rapport à $p$ si pour toute paire de morphisme $\alpha\colon A\to X$, $\beta\colon B\to Y$ telle que le diagramme suivant commute
\begin{equation*}
\begin{tikzcd}
A
\arrow{r}{\alpha}
\arrow[swap]{d}{i}
& X
\arrow{d}{p}
\\
B
\arrow[swap]{r}{\beta}
\arrow[dashrightarrow]{ur}{h}
&Y
\end{tikzcd}
\end{equation*}
il existe un morphisme $h\colon B\to X$ faisant commuter les deux triangles. Dans ce cas, on dit aussi que $p$ a la propriété de relèvement à droite (\textbf{R}ight \textbf{L}ifting \textbf{P}roperty) par rapport à $i$.
\end{defin}

\begin{defin}
Soient $X,Y$ deux objets d'une catégorie $\C$. On dit que $X$ est un rétract de $Y$ si il existe des morphismes $i\colon X\to Y$ et $r\colon Y\to X$ tels que $r\circ i=\Id_X$. Soient $f\colon X\to X'$ et $g\colon Y\to Y'$ deux morphismes de $\C$. On dit que le morphisme $f$ est un rétract de $g$ si c'est un rétract de $g$ dans la catégorie des morphismes de $\C$. Plus explicitement, $f$ est un rétract de $g$ si il existe des morphismes $i\colon X\to Y$, $r\colon Y\to X$, $i'\colon X'\to Y'$ et $r'\colon Y'\to X'$ tels que $r\circ i=\Id_X$, $r'\circ i'=\Id_{X'}$ et tels que le diagramme suivant commute.
\begin{equation*}
\begin{tikzcd}
X
\arrow{r}{i}
\arrow[swap]{d}{f}
&Y
\arrow{r}{r}
\arrow{d}{g}
&X
\arrow{d}{f}
\\
X'
\arrow[swap]{r}{i'}
& Y'
\arrow[swap]{r}{r'}
&X'
\end{tikzcd}
\end{equation*}
\end{defin}

\begin{defin}
Soit $\C$ une catégorie munie de trois classes de flèches stables par compositions et contenant les identités : une classe de cofibrations, une classe de fibrations et une classe d'équivalences faibles. Un morphisme appartenant à la fois à la classe des fibrations et à celle des équivalences faibles est une fibration triviale, et un morphisme appartenant à la fois à la classe des cofibrations et à celle des équivalences faibles est une cofibration triviale. C'est une catégorie modèle si les axiomes suivants sont vérifiés.
\begin{enumerate}
\item \label{AxiomeMC1} La catégorie $\C$ est stable par limites finies et colimites finies.
\item \label{AxiomeMC2}Si $f\colon X\to Y$ et $g\colon Y\to Z$ sont deux morphismes de $\C$ tels que deux des trois morphismes $f,g$ et $g\circ f$ sont des équivalences faibles, alors le troisième morphisme est une équivalence faible.
\item \label{AxiomeMC3}Soient $f$ et $g$ sont deux morphismes de $\C$ tels que $f$ est un rétract de $g$. Si $g$ est une fibration, une cofibration, ou une équivalence faible, alors $f$ aussi.
\item \label{AxiomeMC4}Les fibrations triviales ont la propriété de relèvement à droite par rapport aux cofibrations, et les cofibrations triviales ont la propriété de relèvement à gauche par rapport aux fibrations. Autrement dit, étant donnés une cofibration $i\colon A\to B$, une fibration $p\colon X\to Y$ et un diagramme commutatif
\begin{equation*}
\begin{tikzcd}
A
\arrow{r}{\alpha}
\arrow[swap]{d}{i}
& X
\arrow{d}{p}
\\
B
\arrow[swap]{r}{\beta}
\arrow[dashrightarrow]{ur}{h}
&Y
\end{tikzcd}
\end{equation*}
un relèvement $h$ existe dès que $i$ ou $p$ est une équivalence faible.
\item \label{AxiomeMC5}	Tout morphisme de $\C$, $f$ admet deux factorisations $f=q\circ i$ et $f=p\circ j$, où $q$ est une fibration triviale, $i$ est une cofibration, $p$ est une fibration et $j$ est une cofibration triviale.
\end{enumerate}
Si $\C$ est une catégorie, la donnée de trois classes de flèches vérifiant les axiomes \ref{AxiomeMC2}, \ref{AxiomeMC3}, \ref{AxiomeMC4} et \ref{AxiomeMC5} est une structure de modèle.
\end{defin}

L'axiome \ref{AxiomeMC4} affirme que les cofibrations triviales ont la propriété de relèvement à gauche par rapport aux fibrations (et de même pour les cofibrations et fibrations triviales). En fait, les cofibrations triviales sont complétement caractérisée par cette propriété. Ce constate est le contenu de la proposition classique suivante.

\begin{prop}\label{PropositionRelevementDetermineClasses}
Soit $\C$ une catégorie modèle.
\begin{itemize}
\item Les cofibrations sont les morphismes ayant la LLP par rapport aux fibrations triviales.
\item Les cofibrations triviales sont les morphismes ayant la LLP par rapport aux fibrations.
\item Les fibrations sont les morphismes ayant la RLP par rapport aux cofibrations triviales.
\item Les fibrations triviales sont les morphismes ayant la RLP par rapport aux cofibrations.
\end{itemize}
\end{prop}

\begin{remarque}\label{RemarqueIlSuffitDeDeuxClasses}
Si $\C$ est une catégorie modèle, il suffit donc de connaitre deux classes parmi les classes de fibrations, cofibrations et équivalences faibles pour déduire la troisième. Étant données la classe de cofibrations (resp. de fibrations) et d'équivalences faibles, la classes des fibrations (resp. cofibrations) est caractérisée par ses propriétés de relèvement. Étant données les classes de cofibrations et de fibrations, les équivalences faibles sont les morphismes admettant une factorisation sous la forme d'une cofibration triviale suivie d'une fibration triviale. En effet, si $f$ est un morphisme de $\C$, par application de l'axiome \ref{AxiomeMC5} il existe une factorisation $f=q\circ i$ où $q$ est une fibration triviale et $i$ est une cofibration. Par l'axiome \ref{AxiomeMC2} (souvent appelé axiome de deux sur trois), $i$ est une équivalence faible si et seulement si $f$ est une équivalence faible. Et on a donc la caractérisation voulue.
De plus, comme la classe des cofibrations et celle des fibrations triviales se déterminent mutuellement, il suffit par exemple de donner les classes de fibrations et de fibrations triviales pour décrire complétement une catégorie modèle.
\end{remarque}

La proposition \ref{PropositionRelevementDetermineClasses} permet de caractériser les classes d'une catégorie modèle à partir de leurs propriétés de relèvement. Ceci permet notamment de montrer la proposition suivante.

\begin{prop}\label{PropositionCofibrationsStablesEverythingChap2}
Les classes de cofibrations et de cofibrations triviales d'une catégorie modèle sont stables par unions disjointes, sommes amalgamées et compositions transfinies. 
\end{prop}

\begin{proof}
Ces trois affirmations se prouvent de la même façon. Soit $i\colon A\to B$ est un morphisme obtenu comme union disjointe, somme amalgamée ou comme composition transfinie de cofibrations (triviales), et $p\colon X\to Y$ une fibration (triviale). On considère le problème de relèvement suivant.
\begin{equation*}
\begin{tikzcd}
A
\arrow{r}
\arrow[swap]{d}{i}
&X
\arrow{d}{p}
\\
B
\arrow{r}
\arrow[dashrightarrow]{ur}{h}
&Y
\end{tikzcd}
\end{equation*}
Les propriétés universelles définissant le morphisme $i\colon A\to B$ garantissent qu'un relèvement $h$ existe. Par exemple, si $i\colon A\to B$ est une union disjointe de cofibrations (triviales)
\begin{equation*}
\coprod_{\alpha} i_{\alpha}\colon \coprod_{\alpha} A_{\alpha}\to \coprod_{\alpha} B_{\alpha}
\end{equation*}
Alors pour tout $\alpha$, il existe un relèvement $h_{\alpha}\colon B_{\alpha}\to X$ ce qui permet alors de définir un relèvement
\begin{equation*}
h=\coprod h_{\alpha}\colon \coprod B_{\alpha}\to X.
\end{equation*}
Mais alors, comme $i\colon A\to B$ a la propriété de relèvement à gauche par rapport à toutes les fibrations (triviales), c'est une cofibration (triviale).
\end{proof}

\begin{defin}
Soit $\C$ une catégorie modèle. Par l'axiome \ref{AxiomeMC1}, elle admet un objet initial, qu'on note $\emptyset$, et un objet final, qu'on note $*$. Soit $X$ un objet de $\C$. C'est un objet fibrant si le morphisme $X\to *$ est une fibration. C'est un objet cofibrant si le morphisme $\emptyset\to X$ est une cofibration.
\end{defin}

\subsection{Homotopies dans une catégorie modèle}
\label{SectionHomotopiesCMF}
\begin{defin}
Soit $X$ un objet d'une catégorie modèle $\C$. Un objet cylindre de $X$ dans $\C$ est la donnée d'un objet $X\times I$ (pas nécessairement un produit) et de morphismes :
\begin{equation*}
\begin{tikzcd}
X\coprod X\arrow{r}{i}& X\times I\arrow{r}{q} &X,
\end{tikzcd}
\end{equation*}
tels que la composée est égale à $\Id_X\coprod \Id_X$, et tels que $p$ est une équivalence faible.

Soit $Y$ un objet d'une catégorie modèle $\C$. Un objet chemin de $Y$ dans $\C$ est la donnée d'un objet $Y^{I}$ et de morphismes
\begin{equation*}
\begin{tikzcd}
Y\arrow{r}{j} &Y^{I}\arrow{r}{p} &Y\times Y,
\end{tikzcd}
\end{equation*}
tels que la composée est égale à l'application diagonale, $(\Id_Y,\Id_Y)$, et tels que $j$ est une équivalence faible.
\end{defin}

\begin{remarque}
Comme le suggèrent les notations, on peut obtenir des objets cylindres et chemins en appliquant l'axiome de factorisation \ref{AxiomeMC5}. 
\end{remarque}

\begin{defin}\label{DefinitionHomotopieGaucheDroite}
Soient $f,g\colon X\to Y$ deux morphismes. On dit que $f$ et $g$ sont homotopes à gauche (noté $f\sim_l g$) si il existe un objet cylindre pour $X$, 
\begin{equation*}
\begin{tikzcd}
X\coprod X\arrow{r}{i_0\coprod i_1} &X\times I\arrow{r} & X,
\end{tikzcd}
\end{equation*} 
et un morphisme
\begin{equation*}
\begin{tikzcd}
X \times I\arrow{r}{H} &Y,
\end{tikzcd}
\end{equation*}
tel que $H\circ i_0=f$ et $H\circ i_1=g$.
On dit que $f$ et $g$ sont homotopes à droite (noté $f\sim_r g$) si il existe un objet chemin pour $Y$, 
\begin{equation*}
\begin{tikzcd}
Y\arrow{r} &Y^I\arrow{r}{(p_0,p_1)} &Y\times Y
\end{tikzcd}
\end{equation*} 
et un morphisme
\begin{equation*}
\begin{tikzcd}
X \arrow{r}{G} &Y^I,
\end{tikzcd}
\end{equation*}
tel que $p_0\circ G=f$ et $p_1\circ G=g$. 
\end{defin}

\begin{remarque}
La définition \ref{DefinitionHomotopieGaucheDroite} montre que l'on peut définir une notion d'homotopie dans une catégorie modèle arbitraire. Cependant, les notions d'homotopies à gauche et à droite se comportent mal en général. En effet, les notions d'homotopies à gauche et à droite ne coïncident pas et les relations $\sim_l$ et $\sim_r$ ne sont pas des relations d'équivalences. De plus, on ne peut pas fixer un objet cylindre (ou un objet chemin) pour chaque objet a priori, et calculer la relation $\sim_l$ seulement par rapport à ce cylindre (voir par exemple la remarque \ref{RemarqueHomotopieSimplicialPasEquivalence}). Heureusement, lorsqu'on se restreint aux objets à la fois fibrants et cofibrants, la situation s'améliore.
\end{remarque}

\begin{prop}\label{PropBonObjetCylindre}
Soient $A$ un objet cofibrant et $X$ un objet fibrant. Alors les relations $\sim_l$ et $\sim_r$ sont des relations d'équivalences et coïncident sur $\Hom_{\C}(A,X)$. On notera simplement $\sim$ cette relation. De plus, deux morphismes $f,g\colon A\to X$ sont homotopes (à gauche) si et seulement si il existe une homotopie entre eux $A\times I\to X$, où $A\times I$ est un objet cylindre \textbf{fixé} (il faut que $A\times I$ soit un bon objet cylindre, voir \cite[Definition 4.2]{DwyerSpalinski}). 
\end{prop}

On a immédiatement une version catégorie modèle du théorème de Whitehead (voir \cite[Lemma 4.24]{DwyerSpalinski}).

\begin{theo}\label{TheoremeFibrantCofibrantEquivalencesFaiblesHomotopiesChap2}
Soient $f\colon A\to X$ un morphisme de $\C$, où $A$ et $X$ sont des objets fibrants et cofibrants. Alors $f$ est une équivalence faible si et seulement si il existe $g\colon X\to A$ tel que $f\circ g\sim \Id_X$ et $g\circ f\sim \Id_A$.
\end{theo}

\subsection{La catégorie homotopique d'une catégorie modèle}
\label{SectionCategorieHomotopiqueCMF}
\begin{defin}
Soient $\C$ une catégorie modèle et $W$ sa classe d'équivalence faible. La catégorie homotopique de $\C$, notée $\Ho(\C)$ et munie d'un foncteur $\gamma\colon \C\to \Ho(\C)$, est définie par la propriété universelle suivante. Pour tout foncteur $F\colon\C\to \mathcal{D}$ envoyant les équivalences faibles de $\C$ sur des isomorphismes de $\mathcal{D}$, il existe un unique foncteur $F'\colon \Ho(\C)\to \mathcal{D}$ tel que le diagramme suivant commute.
\begin{equation*}
\begin{tikzcd}
\C
\arrow[swap]{d}{\gamma}
\arrow{r}{F}
&\mathcal{D}
\\
\Ho(\C)
\arrow[dashrightarrow,swap]{ur}{F'}
\end{tikzcd}
\end{equation*}
\end{defin}

\begin{defin}
Soit $\C$ une catégorie modèle, et $\C_{c,f}$ la sous-catégorie pleine de $\C$ contenant les objets à la fois cofibrants et fibrants. On note $\C_{c,f}/{\sim}$ la catégorie dont les objets sont les mêmes que ceux de $\C_{c,f}$, et dont les morphismes sont les quotients
\begin{equation*}
\Hom_{\C}(A,X)/{\sim}.
\end{equation*}
\end{defin}

\begin{remarque}
Le quotient $\C/{\sim}$ est parfois appelé la catégorie homotopique naïve de $\C$. La catégorie $\C_{c,f}/{\sim}$ est alors la sous-catégorie homotopique naïve des objets cofibrants-fibrants.
\end{remarque}

\begin{theo}\label{TheoremeCategorieHomotopique}
Soit $\C$ une catégorie modèle. La catégorie homotopique de $\C$ existe, et l'inclusion $\C_{c,f}\to\C$ induit une équivalence de catégories
\begin{equation*}
\C_{c,f}/{\sim}\to \Ho(\C).
\end{equation*}
\end{theo}

\begin{proof}
On renvoie à \cite[Theorem 6.2]{DwyerSpalinski} pour une preuve complète. L'idée de la preuve est la suivante. Pour chaque objet $X$ de $\C$, on factorise le morphisme $X\to *$ à l'aide de l'axiome \ref{AxiomeMC5} pour obtenir
\begin{equation*}
\begin{tikzcd}
X\arrow{r}{j} &R(X)\arrow{r}{p} &*
\end{tikzcd}
\end{equation*}
où $j$ est une cofibration triviale et où $p$ est une fibration. On remplace ainsi $X$ par un objet fibrant $R(X)$ qui lui est faiblement équivalent. Puis, grâce à l'autre moitié de l'axiome \ref{AxiomeMC5}, on obtient une factorisation
\begin{equation*}
\begin{tikzcd}
\emptyset\arrow{r}{i} &Q(R(X))\arrow{r}{q} &R(X),
\end{tikzcd}
\end{equation*}
où $i$ est une cofibration et où $q$ est une fibration triviale. Comme les fibrations sont stables par composition, $Q(R(X))$ est un objet fibrant. Il est aussi cofibrant car $i$ est une cofibration, et il est faiblement équivalent à $X$ par le zigzag d'équivalences faibles
\begin{equation}\label{EquationZigZagQRX}
\begin{tikzcd}
X\arrow{r}{j} &R(X) & Q(R(X))\arrow[swap]{l}{q}.
\end{tikzcd}
\end{equation}
Le reste de la preuve consiste à montrer que tout morphisme $f\colon X\to Y$ correspond à un morphisme $\bar{f}\colon Q(R(X))\to Q(R(Y))$ et que ces deux morphismes peuvent être comparés à l'aide du zigzag \ref{EquationZigZagQRX}
\end{proof}

Il est alors naturel de se demander comment construire un foncteur $\Ho(\C)\to \Ho(\D)$ à partir d'un foncteur entre catégories modèles $F\colon\C\to \mathcal{D}$. Étant donné le théorème \ref{TheoremeCategorieHomotopique}, le foncteur $F$ devra préserver certaines propriétés liés aux objets cofibrants et fibrants.

\begin{defin}
Soient $\C$ et $\D$ deux catégories modèles et 
\begin{equation*}
F\colon \C\leftrightarrow \D\colon G
\end{equation*}
une paire de foncteurs adjoints. 
\begin{itemize}
\item L'adjonction $(F,G)$ est une adjonction de Quillen si $F$ envoie les cofibrations de $\C$ sur des cofibrations de $\D$ et $G$ envoie les fibrations de $\D$ sur des fibrations de $\C$.
\item C'est une équivalence de Quillen si de plus pour tout objet cofibrant $A$ de $\C$, tout objet fibrant $X$, de $\D$, et tout morphisme $f\colon F(A)\to X$, le morphisme $f$ est une équivalence faible de $\D$ si et seulement si son image par l'adjonction $\widetilde{f}\colon A\to G(X)$ est une équivalence faible de $\C$.
\end{itemize}
\end{defin}

On a alors le théorème suivant (voir \cite[Theorem 9.7]{DwyerSpalinski} pour une preuve).

\begin{theo}
Soit $F\colon \C\leftrightarrow \D\colon G$ une adjonction entre les catégories modèles $\C$ et $\D$. Si l'adjonction $(F,G)$ est une adjonction de Quillen, elle induit une paire de foncteurs adjoints
\begin{equation*}
\mathbf{L}F\colon \Ho(\C)\leftrightarrow \Ho(\D)\colon \mathbf{R}G.
\end{equation*}
Si de plus c'est une équivalence de Quillen, les foncteurs $\mathbf{L}F$ et $\mathbf{R}G$ sont des équivalences de catégories inverses l'une de l'autre.
\end{theo}

\begin{remarque}
Ainsi, la notion d'équivalence de Quillen permet de comparer des théories homotopiques. Si deux catégories modèles sont Quillen équivalents, elles ont des catégories homotopiques équivalentes et sont donc des modèles de la même théorie homotopique.
\end{remarque}

\begin{prop}\label{PropositionAdjonctionQuillenPreserveEquivalencesFaibles}
Soit $F\colon \C\leftrightarrow \D\colon G$ une adjonction de Quillen. Alors $F$ préserve les équivalences faibles entre objets cofibrants et $G$ préserve les équivalences faibles entre objets fibrants
\end{prop}

\begin{proof}
Soit $f\colon A\to B$ un équivalence faible de $\C$ entre objets cofibrants. On considère le morphisme
\begin{equation*}
f\coprod\Id_B\colon A\coprod B\to B.
\end{equation*}
Par l'axiome \ref{AxiomeMC5}, on peut factoriser ce morphisme en une cofibration suivie d'une fibration triviale
\begin{equation*}
\begin{tikzcd}
A\coprod B
\arrow{r}{i}
&C
\arrow{r}{q}
&B
\end{tikzcd}.
\end{equation*}
D'autre part, les inclusions $i_A\colon A\to A\coprod B$ et $i_B\colon B\to A\coprod B$ peuvent être vues comme des sommes amalgamées
\begin{equation*}
\begin{tikzcd}
\emptyset
\arrow{r}
\arrow{d}
&A
\arrow{d}{i_A}
\\
B
\arrow[swap]{r}{i_B}
&A\coprod B
\end{tikzcd}
\end{equation*}
En particulier, comme $A$ et $B$ sont cofibrants, les inclusions $i_A,i_B\colon A,B\to A\coprod B$ sont des cofibrations. Notons $s\colon B\to C$ la composition
\begin{equation*}
B\to A\coprod B\to C.
\end{equation*}
Par construction, c'est une section de $q\colon C\to B$. Comme  $q$ est une équivalence faible et $q\circ s=\Id_Y$, par l'axiome de deux sur trois, on déduit que $s\colon B\to C$ est une équivalence faible. C'est aussi une composition de cofibrations, donc $s$ est une cofibration triviale. D'autre part, $f$ est égale à la composition
\begin{equation*}
\begin{tikzcd}
A\arrow{r}{i\circ i_A} &C\arrow{r}{q} & B.
\end{tikzcd}
\end{equation*}
Comme $q$ et $f$ sont des équivalences faibles, $i\circ i_A$ est une équivalence faible par deux sur trois. Comme $i$ et $i_A$ sont des cofibrations, $i\circ i_A$ est une cofibration triviale. Finalement, $F(f)$ est égale à la composition
\begin{equation*}
\begin{tikzcd}
F(A)\arrow{r}{F(i\circ i_A)}&F(C)\arrow{r}{F(q)} &F(B)
\end{tikzcd}
\end{equation*}
Comme $F$ préserve les cofibrations triviales, $F(i_A)$ est une cofibration triviale. D'autre part, $F(s)$ est une section de $F(q)$, et une cofibration triviale. Par deux sur trois, $F(q)$ est une équivalence faible, et donc $F(f)$ est une équivalence faible. La preuve pour $G$ est duale.
\end{proof}

\subsection{Catégories modèles à engendrement cofibrant}
\label{SectionCMFEngendrementCofibrant}

On a vu à la proposition \ref{PropositionRelevementDetermineClasses} que les classes de fibrations (triviales) et de cofibrations (triviales) étaient complètement déterminées par leur propriétés de relèvement les unes par rapport aux autres. De plus, on a vu à la remarque \ref{RemarqueIlSuffitDeDeuxClasses} qu'il suffisait de décrire les fibrations et les fibrations triviales pour décrire complètement une catégorie de modèle. On souhaite décrire celles-ci par des propriétés de relèvement, non pas par rapport à toutes les cofibrations (triviales), mais par rapport à un \textbf{ensemble} bien choisi de cofibrations (triviales). C'est la notion de catégorie modèle à engendrement cofibrant.

\begin{defin}
Une catégorie modèle $\C$ est à engendrement cofibrant si il existe deux ensembles de morphismes $I$ et $J$ tels que
\begin{itemize}
\item un morphisme est une fibration si et seulement si il a la propriété de relèvement à droite par rapport à tous les morphismes de $J$.
\item un morphisme est une fibration triviale si et seulement si il a la propriété de relèvement à droite par rapport à tous les morphismes de $I$.
\end{itemize}
\end{defin}

\begin{remarque}\label{RemarqueEngendrementCofibrantIEtJ}
Ainsi, pour décrire une catégorie modèle à engendrement cofibrant, il suffit de donner les deux ensembles de morphismes $I$ et $J$. Cependant, étant donné deux ensembles de morphismes, les classes de "fibrations" et de "fibrations triviales" qu'ils définissent ne sont en général pas les classes d'une catégorie modèle. Il existe  des théorèmes de reconnaissance, voir par exemple \cite[Theorem 11.3.1]{Hirschhorn}.
\end{remarque}

Un des avantages majeurs des catégories à engendrement cofibrant est qu'elles permettent l'argument du petit objet. Celui-ci permet notamment d'obtenir les factorisations de l'axiome \ref{AxiomeMC5} sous forme fonctorielle. On explicite cette construction ici. 

\begin{defin}
Un objet $A$ d'une catégorie $\C$ est petit si pour toute suite d'objet de $\C$
\begin{equation*}
X^0\to X^1\to\dots\to X^n\to\dots,
\end{equation*}
tout morphisme $f\colon A\to X^{\infty}=\colim_iX^i$, se factorise sous la forme
\begin{equation*}
\begin{tikzcd}
A
\arrow{rr}{f}
\arrow[swap]{dr}{f_n}
&&X^{\infty}
\\
&X^{n}
\arrow{ur}
\end{tikzcd}
\end{equation*}
pour un certain $n$.
\end{defin}

Supposons que $\C$ est une catégorie modèle à engendrement cofibrant avec les ensembles $I$ et $J$, et que $\C$ est complète et cocomplète. On suppose de plus que pour tout morphisme de $I$, $\alpha\colon A\to B$, l'objet $A$ est petit. Soit $f\colon X\to Y$ un morphisme de $\C$, on va factoriser $f$ en un cofibration suivie d'une fibration triviale (ou en une cofibration triviale suivie d'une fibration en échangeant $I$ et $J$).
Cette construction passe par la définition d'une suite d'objets $Z^n$ comme suit :
\begin{equation}\label{EquationDiagrammeArgumentPetitObjet}
\begin{tikzcd}
X
\arrow{r}{i_0}
\arrow[swap]{ddrr}{f}
&Z^1
\arrow{r}{i_1}
\arrow{ddr}{p^1}
&Z^2
\arrow{r}{i_2}
\arrow{dd}{p^2}
&\dots
\arrow{r}{i_{n-1}}
&Z^{n}
\arrow{ddll}{p^n}
\arrow{r}{i_n}
&\dots
\\
\\
&&Y
\end{tikzcd}
\end{equation}
telle que pour tout $n$, la composition $X\to Z^n\to Y$ est égale à $f$. On procède inductivement. On pose $Z^0=X$ et $p^0=f\colon Z^0\to Y$. Puis pour $n\geq 0$, on définit $S^n$ l'ensemble des diagrammes commutatifs de la forme suivante,
\begin{equation*}
\begin{tikzcd}
A
\arrow{r}{a}
\arrow[swap]{d}{\alpha}
&Z^n
\arrow{d}{p^n}
\\
B
\arrow[swap]{r}{b}
&Y
\end{tikzcd}
\end{equation*}
où $\alpha\colon A\to B$ est un morphisme de $I$.
Puis on définit $Z^{n+1}$ comme la somme amalgamée
\begin{equation*}
\begin{tikzcd}
\coprod_{S^n}A\arrow{r}{\coprod a}
\arrow[swap]{d}{\coprod\alpha}
&Z^n
\arrow{d}{i_n}
\\
\coprod_{S^n}B
\arrow{r}
&Z^{n+1}
\end{tikzcd}
\end{equation*}
Par définition, chacun des morphismes $\alpha\colon A\to B$ apparaissant dans $S^n$ sont des cofibrations. Les cofibrations étant stables par union disjointe, le morphisme
\begin{equation*}
\begin{tikzcd}
\coprod_{S^n}A
\arrow{r}{\coprod\alpha}
&
\coprod_{S^n}B
\end{tikzcd}
\end{equation*}
est une cofibration. Comme les cofibrations sont stables par sommes amalgamées, on en déduit que le morphisme $i_n$ est une cofibration. Par ailleurs, par construction de $S^n$, on dispose d'un morphisme
\begin{equation*}
\begin{tikzcd}
\coprod_{S^n}B\arrow{r}{\coprod b} & Y
\end{tikzcd}
\end{equation*}
Comme on a aussi un morphisme $p^n\colon Z^n\to Y$, par la propriété universelle de la somme amalgamée, on obtient un morphisme $p^{n+1}\colon Z^{n+1}\to Y$. On obtient donc le diagramme \ref{EquationDiagrammeArgumentPetitObjet}, où tous les morphismes horizontaux sont des cofibrations. Notons $j_n=i_{n-1}\circ\dots\circ i_0\colon X\to Z^n$, et $j$ l'inclusion
\begin{equation*}
j\colon X\to \colim Z^n=Z^{\infty}.
\end{equation*}
Alors, $f$ se factorise sous la forme
\begin{equation*}
\begin{tikzcd}
X
\arrow{rr}{j}
\arrow[swap]{dr}{f}
&& Z^{\infty}
\arrow{dl}{p^{\infty}}
\\
&Y
\end{tikzcd}
\end{equation*}
où $p^{\infty}=\colim p^n\colon \colim Z^n\to Y$. De plus, $j$ est une cofibration car c'est une composition de cofibrations. Il reste à montre que $p^{\infty}$ est une fibration. Considérons un problème de relèvement
\begin{equation*}
\begin{tikzcd}
A
\arrow{r}{a}
\arrow[swap]{d}{\alpha}
&Z^{\infty}
\arrow{d}{p^{\infty}}
\\
B
\arrow{r}
&Y
\end{tikzcd}
\end{equation*}
où $\alpha\colon A\to B\in I$.
Comme $A$ est petit, il existe un entier $n\geq 0$ tel que $a$ se factorise sous la forme
\begin{equation*}
\begin{tikzcd}
A\arrow{r}{a_n} &Z^n\arrow{r} & Z^{\infty}.
\end{tikzcd}
\end{equation*}
Mais, alors, par construction, il existe un relèvement dans $Z^{n+1}$.
\begin{equation*}
\begin{tikzcd}
A\arrow{r}{a_n}
\arrow{dd}
&Z^{n}
\arrow{d}
\\
&Z^{n+1}
\arrow{d}{p^{n+1}}
\\
B
\arrow{r}
\arrow[dashrightarrow]{ur}
&Y
\end{tikzcd}
\end{equation*}
Ainsi, le morphisme $p^{\infty}$ admet la propriété de relèvement à droite par rapport à tous les morphismes de $I$, c'est donc une fibration triviale.

\subsection{Catégories modèles simpliciales}
\label{SectionCMFSimpliciale}
\begin{defin}
Une catégorie $\C$ est une catégorie simpliciale si elle est munie d'un foncteur
\begin{equation*}
\Map\colon\C^{\op}\times \C\to \sS,
\end{equation*}
vérifiant
\begin{itemize}
\item pour toute paire d'objets de $\C$, $X,Y$, $\Map(X,Y)_0=\Hom_{\C}(X,Y)$,
\item pour tout objet de $\C$, $X$, le foncteur 
\begin{equation*}
\Map(X,-)\colon \C\to \sS
\end{equation*} 
admet un adjoint à gauche 
\begin{equation*}
X\otimes -\colon \sS\to \C
\end{equation*}
associatif au sens suivant : si $K,L$ dont des ensembles simpliciaux, on a un isomorphisme naturel en $X$, $K$ et $L$
\begin{equation*}
X\otimes \left(K\times L\right)\simeq \left(X\otimes K\right)\otimes L,
\end{equation*}
\item pour tout objet de $\C$, $Y$, le foncteur 
\begin{equation*}
\Map(-,Y)\colon \C^{\op}\to\sS
\end{equation*}
admet un adjoint à droite
\begin{equation*}
Y^{-}\colon \sS\to \C^{\op}.
\end{equation*}
\end{itemize}
\end{defin}

\begin{remarque}\label{RemarqueMapCategorieSimpliciale}
Alternativement, on peut définir une catégorie simpliciale comme une catégorie munie d'un foncteur
\begin{equation*}
-\otimes -\colon \C\times\sS\to \C,
\end{equation*}
vérifiant certaines hypothèses. Dans ce cas, pour $X,Y$ des objets de $\C$, on définit $\Map(X,Y)$ comme suit.
\begin{equation*}
\Map(X,Y)_n=\Hom_{\C}(X\otimes\Delta^n,Y).
\end{equation*}
Dans tout les cas, cette relation est vérifiée
(Voir \cite[Lemma II.2.4]{GoerssJardine}).
\end{remarque}

Soient $p\colon X\to Y$ et $i\colon A\to B$ deux morphismes d'une catégorie simpliciale $\C$. Le foncteur $\Map$ induit le diagramme commutatif suivant,

\begin{equation*}
\begin{tikzcd}
\Map(B,X)
\arrow[dashrightarrow]{dr}{(i^*,p_*)}
\arrow[bend left= 20]{drr}{(\Id_B,p_*)}
\arrow[bend right = 20, swap]{ddr}{(i^*,\Id_X)}
\\
&\Map(A,X)\times_{\Map(A,Y)}\Map(B,Y)
\arrow{r}
\arrow{d}
&\Map(B,Y)
\arrow{d}{(i^*,\Id_Y)}
\\
&\Map(A,X)
\arrow[swap]{r}{(\Id_A,p_*)}
&\Map(A,Y)
\end{tikzcd}
\end{equation*}
où l'ensemble simplicial $\Map(A,X)\times_{\Map(A,Y)}\Map(B,Y)$ est défini comme le produit fibré, et le morphisme $(i^*,p_*)$ est obtenu par la propriété universelle du produit fibré.

\begin{defin}
Une catégorie modèle $\C$ est une catégorie modèle simpliciale si c'est une catégorie simpliciale, et si pour toute fibration $p\colon X\to Y$ et toute cofibration $i\colon A\to B$, le morphisme d'ensembles simpliciaux
\begin{equation}\label{EquationAxiomeCMFSimpliciale}
(i^*,p_*)\colon \Map(B,X)\to \Map(A,X)\times_{\Map(A,Y)}\Map(B,Y)
\end{equation}
est une fibration de $\sS$, qui est triviale dès que $i$ ou $p$ est une équivalence faible.
\end{defin}

\begin{remarque}
La définition d'une catégorie modèle simpliciale fait intervenir la structure de modèle de Kan sur la catégorie $\sS$, que l'on définira dans la section suivante. On note pour l'instant que les objets fibrants de cette structure de modèle sont les complexes de Kan.
\end{remarque}

\begin{prop}
Soient $\C$ une catégorie modèle simpliciale et $A,X$ deux objets de $\C$ tels que $A$ est cofibrant et $X$ est fibrant. Alors, $\Map(A,X)$ est un complexe de Kan.
\end{prop}

\begin{proof}
On considère la cofibration $i\colon\emptyset\to A$ ainsi que la fibration $p\colon X\to *$. Alors, le morphisme \ref{EquationAxiomeCMFSimpliciale}
\begin{equation*}
(i^*,p_*)\colon\Map(A,X)\to \Map(\emptyset,X)\times_{\Map(\emptyset,*)}\Map(A,*)
\end{equation*}
est une fibration. De plus, par la remarque \ref{RemarqueMapCategorieSimpliciale}, on obtient 
\begin{equation*}
\Map(\emptyset,X)\simeq \Map(\emptyset,*)\simeq \Map(A,*)\simeq *
\end{equation*}
On en déduit que $\Map(A,X)\to *$ est une fibration. Autrement dit, $\Map(A,X)$ est un complexe de Kan.
\end{proof}

Dans une catégorie modèle simpliciale les ensembles simpliciaux de morphismes, $\Map(X,Y)$, et les ensembles de classes d'homotopie sont reliés.

\begin{prop}
Soient $\C$ une catégorie modèle simpliciale et $A,X$ deux objets de $\C$ tels que $A$ est cofibrant et $X$ est fibrant. Alors, on a
\begin{equation*}
\pi_0(\Map(A,X))=\Hom_{\C}(A,X)/{\sim}
\end{equation*}
\end{prop}

\begin{proof}
Soient $\C$ une catégorie modèle simpliciale, et $A,X$ deux objets de $\C$ tels que $A$ est cofibrant et $X$ est fibrant. Alors, par la proprosition \ref{PropBonObjetCylindre}, deux morphismes $f,g\colon A\to X$ sont homotopes si et seulement si il existe une homotopie entre eux, $H\colon A\times I\to X$, où $A\times I$ est un bon cylindre de $A$ fixé. On vérifie que
\begin{equation*}
A\otimes\partial(\Delta^1)\to A\otimes \Delta^1\to A
\end{equation*}
fournit un bon cylindre de $A$ (voir \cite[Lemma II.3.5]{GoerssJardine}). D'autre part, on sait que
\begin{equation*}
\Map(A,X)_0=\Hom_{\C}(A,X).
\end{equation*}
et deux morphismes $f,g\in \Map(A,X)_0$ correspondent à la même classe dans $\pi_0(\Map(A,X))$ si et seulement si il existe une homotopie
\begin{equation*}
H\colon \Delta^1\to \Map(A,X),
\end{equation*}
telle que $H\circ i_0=f$ et $H\circ i_1=g$. Par la remarque \ref{RemarqueYonedaEnsemblesSimpliciauxChap2}, une telle homotopie correspond à un $1$-simplexe de $\Map(A,X)$, et par la remarque \ref{RemarqueMapCategorieSimpliciale}, ces simplexes correspondent aux morphismes de la forme
\begin{equation*}
A\otimes\Delta^1\to X.
\end{equation*} 
Finalement, deux morphismes $f,g\in \Hom(A,X)$ sont homotopes si et seulement si ils sont dans la même composante connexe de $\Map(A,X)$.
\end{proof}

\section{L'adjonction de Kan-Quillen}
\label{SectionAdjonctionKanQuillenChap2}
\subsection{La catégorie modèle des ensembles simpliciaux}
\label{SectionCMFsSChap2}
\begin{theo}\label{TheoremeStructureKanSSetChap2}
La catégorie des ensembles simpliciaux $\sS$ admet une structure de modèle à engendrement cofibrant où les ensembles générateurs des cofibrations, $I$, et des cofibrations triviales, $J$, sont les suivants.
\begin{itemize}
\item $I=\{\partial(\Delta^n)\to \Delta^n\ |\ n\geq 0\}$,
\item $J=\{\Lambda^n_k\to \Delta^n\ |\ n\geq 1,\ 0\leq k\leq n\}$
\end{itemize}
\end{theo}

\begin{remarque}
On a vu à la remarque \ref{RemarqueEngendrementCofibrantIEtJ} qu'il suffisait de décrire les ensembles générateurs $I$ et $J$ pour décrire complètement une catégorie modèle. Il est cependant utile d'expliciter les différentes classes de morphismes obtenues à partir de $I$ et $J$. On constate d'abord que les objets fibrants sont les complexes de Kan (Définition \ref{DefinitionComplexeDeKan}), et que les fibrations sont les fibrations de Kan. On a aussi la propriété suivante.
\end{remarque}

\begin{prop}\label{PropositionCofibrationsSSetMonomorphismes}
Les cofibrations de $\sS$ sont les monomorphismes.
\end{prop}

\begin{proof}
Soit $i\colon X\to Y$ un monomorphisme de $\sS$, on peut obtenir $Y$ en "rajoutant" des simplexes à $X$. Plus précisément, on note $\sk_n(Y)\subset Y$ le sous-ensemble simplicial de $Y$ engendré par les simplexes de $Y$ de dimension $\leq n$. On a alors la suite d'inclusions
\begin{equation*}
X=X\cup\sk_{-1}(Y)\hookrightarrow X\cup\sk_{0}(Y)\hookrightarrow X\cup\sk_1(Y)\hookrightarrow \dots \hookrightarrow Y.
\end{equation*}
Pour $n\geq -1$, chacune de ces inclusions est une somme amalgamée de la forme
\begin{equation}\label{EquationCofibrationMonomorphismeChap2}
\begin{tikzcd}
\coprod \partial(\Delta^{n+1})
\arrow{r}
\arrow{d}
&X\cup \sk_n(Y)
\arrow{d}
\\
\coprod \Delta^{n+1}
\arrow{r}
&X\cup \sk_{n+1}(Y)
\end{tikzcd}
\end{equation}
où l'union est prise sur tout les $n+1$-simplexes non-dégénérés de $Y$ n'appartenant pas à $X$. Mais alors, pour tout $n$, l'inclusion $X\cup\sk_n(Y)\to X\cup\sk_{n+1}(Y)$ est une cofibration, car c'est une somme amalgamée d'une union disjointe de cofibration. Il suit que la composition transfinie \ref{EquationCofibrationMonomorphismeChap2} est une cofibration (voir la proposition \ref{PropositionCofibrationsStablesEverythingChap2}).

Réciproquement, supposons que $f\colon X\to Y$ est une cofibration. Par l'argument du petit objet, on peut factoriser $f$ sous la forme
\begin{equation*}
\begin{tikzcd}
X
\arrow{rr}{i}
\arrow[swap]{dr}{f}
&& Z
\arrow{dl}{q}
\\
&Y
\end{tikzcd}
\end{equation*}
avec $q$ une fibration triviale et $i$ une composition transfinie de sommes amalgamées d'unions disjointes de morphismes de la forme $\partial(\Delta^n)\to \Delta^n$ (voir la fin de la section \ref{SectionCMFEngendrementCofibrant}). En particulier, $i$ est un monomorphisme. Par ailleurs, par hypothèse $f$ est une cofibration. Il existe donc une solution $h$ au problème de relèvement suivant
\begin{equation*}
\begin{tikzcd}
X
\arrow[swap]{d}{f}
\arrow{r}{i}
&Z
\arrow{d}{q}
\\
Y
\arrow[swap]{r}{\Id_Y}
\arrow[dashrightarrow]{ur}{h}
&Y
\end{tikzcd}
\end{equation*}
Mais alors, $f$ est un rétract de $i$.
\begin{equation*}
\begin{tikzcd}
X
\arrow{r}{\Id_X}
\arrow[swap]{d}{f}
&X
\arrow{r}{\Id_X}
\arrow{d}{i}
&X
\arrow{d}{f}
\\
Y
\arrow{r}{h}
&Z
\arrow{r}{q}
&Y
\end{tikzcd}
\end{equation*}
Comme la classe des monomorphismes est stable par rétract, on en déduit que $f$ est un monomorphisme.
\end{proof}

\begin{remarque}
La preuve précédente fait apparaitre un fait clé de la théorie des catégories modèles à engendrement cofibrant : La classe de morphismes obtenues en appliquant les opérations d'union disjointes, de sommes amalgamées, de compositions transfinies et de rétract à l'ensemble $I$ est la classe des cofibrations (de même pour $J$ et la classe des cofibrations triviales).
\end{remarque}

La classe des équivalences faibles est plus difficile à décrire explicitement, mais on peut cependant en donner une caractérisation \cite[\S II.3]{QuillenHomotopicalAlgebra}. On rappelle que si $X,Y$ sont deux ensembles simpliciaux, $[X,Y]$ désigne l'ensemble des classes d'homotopies d'applications simpliciales de $X$ vers $Y$.

\begin{prop}\label{PropositionCaractérisationEquivalenceFaibleChap2}
Soit $f\colon X\to Y$ un morphisme entre ensembles simpliciaux. C'est une équivalence faible si et seulement si, pour tout complexe de Kan $Z$, le morphisme induit par $f$, 
\begin{equation*}
f^*\colon [Y,Z]\to [X,Z],
\end{equation*}
est un isomorphisme.
\end{prop}

\begin{proof}[Démonstration du théorème \ref{TheoremeStructureKanSSetChap2}]
La catégorie $\sS$ est complète et cocomplète (on peut calculer les limites dimension par dimension dans $\Set$) donc l'axiome \ref{AxiomeMC1} est vérifié. Par définition des fibrations (triviales) et des cofibrations (triviales) les propriétés de relèvement de l'axiome \ref{AxiomeMC4} sont satisfaites. Par l'argument du petit objet (voir la section \ref{SectionCMFEngendrementCofibrant}), l'axiome \ref{AxiomeMC5} de factorisation est vérifié. Les propriétés de relèvement étant stables par rétract, l'axiome \ref{AxiomeMC3} est vérifié. Il reste donc à montrer que les équivalences faibles satisfont l'axiome \ref{AxiomeMC2}, de deux sur trois. Il existe plusieurs stratégies distinctes pour montrer ce résultat. 

On peut passer par la caractérisation donnée en proposition \ref{PropositionCaractérisationEquivalenceFaibleChap2}. La preuve de cette proposition fait intervenir la notion de fibrations minimales, une notion spécifique aux cas des ensembles simpliciaux (voir la preuve originale de D. Quillen \cite{QuillenHomotopicalAlgebra}). 

Une autre stratégie est de montrer que $f\colon X\to Y$ est une équivalence faible de $\sS$ si et seulement si 
\begin{equation*}
\Real{f}\colon \Real{X}\to\Real{Y}
\end{equation*}
est une équivalence faible de $\Top$, voir par exemple  \cite[Theorem I.11.3]{GoerssJardine}. (Plus exactement, on définit les équivalences faibles de $\sS$ comme les morphismes qui se réalisent en des équivalences faibles de $\Top$, et on montre ensuite que la catégorie modèle obtenue est la même que celle décrite ici). Cette approche a le défaut de ne pas être interne à la catégorie des ensembles simpliciaux, et de nécessiter la construction préalable de la catégorie modèle des espaces topologiques.

Une troisième stratégie est de définir explicitement un foncteur de remplacement fibrant $\Ex^{\infty}\colon \sS\to\sS$. Ceci revient à définir, pour tout ensemble simplicial $X$, un ensemble simplicial fibrant (un complexe de Kan) $\Ex^{\infty}(X)$ ainsi qu'un morphisme
\begin{equation*}
\beta_{X}\colon X\to \Ex^{\infty}(X)
\end{equation*}
tels que le morphisme $\beta_X$ est une cofibration triviale pour tout $X$. (On demande en fait que $\beta\colon \Id\to \Ex^{\infty}$ soit une transformation naturelle). Soit $f\colon X\to Y$, alors on a le diagramme commutatif
\begin{equation*}
\begin{tikzcd}
X
\arrow[swap]{d}{\beta_X}
\arrow{r}{f}
&Y
\arrow{d}{\beta_Y}
\\
\Ex^{\infty}(X)
\arrow[swap]{r}{\Ex^{\infty}(f)}
&\Ex^{\infty}(Y)
\end{tikzcd}
\end{equation*}
On montre alors que $f$ est une équivalence faible si et seulement si $\Ex^{\infty}(f)$ est une équivalence faible (ce qui est aussi une conséquence de l'axiome \ref{AxiomeMC2} que l'on souhaite montrer). Or $\Ex^{\infty}(X)$ et $\Ex^{\infty}(Y)$ sont fibrants, par hypothèse, et cofibrants (car tous les ensembles simpliciaux sont cofibrants, $\emptyset\to X$ est toujours un monomorphisme) donc par le théorème \ref{TheoremeFibrantCofibrantEquivalencesFaiblesHomotopiesChap2}, $\Ex^{\infty}(f)$ est une équivalence faible si et seulement si c'est une équivalence d'homotopie. Comme les équivalences d'homotopies vérifient la propriété de deux sur trois, on en déduit que l'axiome \ref{AxiomeMC2} est vérifié. (On renvoie à \cite{Cisinski} ou encore à \cite{SeanMoss} pour des preuves complètes).
\end{proof}

\begin{remarque}
La description de la structure de modèle que l'on a donné ici est celle de D. Quillen \cite{QuillenHomotopicalAlgebra}. Les trois stratégies de preuves que l'on a mentionné dans la preuve précédente fournissent la même catégorie modèle, mais à partir de descriptions différentes. On peut montrer a posteriori que les classes de cofibrations, de fibrations et d'équivalences faibles coïncident, mais il est nécessaire a priori de montrer que toutes ces descriptions correspondent bien à des catégories modèles. Suivant la description choisie, l'étape difficile n'est pas toujours de vérifier l'axiome \ref{AxiomeMC2}. Par exemple, si on définit les équivalences faibles de $\sS$ comme les morphismes dont la réalisation est une équivalence faible, l'axiome de deux sur trois est automatiquement vérifié. Par ailleurs, pour la stratégie de preuve de D-C. Cisinski, on obtient une structure de modèle par application d'un résultat général \cite[Théorème 1.3.22]{Cisinski}, et la difficulté réside dans la caractérisation des fibrations. C'est cette approche-ci qu'on suivra au chapitre \ref{ConstructionCMFSSetP}, et de même, la partie difficile sera la caractérisation des fibrations. (Voir le théorème \ref{TheoDescriptionExpliciteSSetP} ainsi que les annexes \ref{ChapitreCaracterisationFibrationsAnnexe} et \ref{ChapitreCaracterisationMorphismeXExXAnnexe}).
\end{remarque}

La catégorie des ensembles simpliciaux $\sS$ est de plus une catégorie simpliciale. Le foncteur $-\otimes-$ est simplement le produit $-\times-$ de la définition \ref{DefinitionProduitEnsemblesSimpliciauxChap2}, et étant donné $A$ et $X$ deux ensembles simpliciaux, on a 
\begin{equation*}
\Map(A,X)_n=\Hom_{\sS}(A\times\Delta^n,X)
\end{equation*}
On a alors le résultat suivant \cite{QuillenHomotopicalAlgebra}.
\begin{theo}
La catégorie modèle $\sS$ est une catégorie modèle simpliciale.
\end{theo}

\begin{remarque}
On peut aussi caractériser les équivalences faibles de $\sS$ comme suit. Un morphisme $f\colon X\to Y$ de $\sS$, où $X$ et $Y$ sont fibrants, est une équivalence faible si et seulement si $f$ induit des isomorphismes sur les groupes d'homotopies naïfs (voir la définition \ref{DefinitionGroupeHomotopieNaif}). En fait, cette caractérisation vient du constat suivant : $f$ est une équivalence faible si et seulement si $\Real{f}$ induit des isomorphismes sur tous les groupes d'homotopies. D. Kan \cite{KanCombinatorialDefinitionHomotopyGroups} a montré que les groupes d'homotopie naïfs de $X$ et ceux de $\Real{X}$ coïncident, ce qui prouve la caractérisation voulue.
\end{remarque}

\begin{remarque}
On dispose maintenant d'une preuve complète du théorème de Whitehead (Théorème \ref{TheoremeWhiteheadClassique}). On a donné à la section \ref{SectionComplexesDeKan} une preuve dépendant d'une version simpliciale du théorème de Whitehead (Théorème \ref{TheoremeWhiteheadSimplicialChap2}). Comme $\sS$ est une catégorie modèle, le théorème \ref{TheoremeWhiteheadSimplicialChap2} est un cas particulier du théorème \ref{TheoremeFibrantCofibrantEquivalencesFaiblesHomotopiesChap2}, qui garantit que les équivalences faibles entre objets fibrants et cofibrants d'une catégorie modèle sont exactement les équivalences d'homotopies. On note qu'il n'est pas nécessaire de construire une structure de modèle sur la catégorie des espaces topologiques, $\Top$, pour obtenir ce résultat. Ceci illustre l'idée qu'il est possible de prouver des résultats sur la théorie homotopique des espaces à l'aide de celle des ensembles simpliciaux. C'est pour expliquer pourquoi une telle traduction est possible qu'il est nécessaire de définir la catégorie modèle des espaces topologiques. 
\end{remarque}

\subsection{La catégorie modèle des espaces topologiques}
\label{SectionCMFTopChap2}
\begin{theo}[\cite{QuillenHomotopicalAlgebra}]
La catégorie $\Top$ munie des classes de morphismes suivantes est une catégorie modèle.
\begin{itemize}
\item Les fibrations sont les morphismes ayant la propriété de relèvement à droite par rapport aux réalisations d'inclusions de cornets $\Real{\Lambda^n_k}\to\Real{\Delta^n}$, $n\geq 1$, $0\leq k\leq n$.
\item Les équivalences faibles sont les morphismes $f\colon X\to Y$ tel que $f$ induit un isomorphisme
\begin{equation*}
f_*\colon \pi_0(X)\to\pi_0(Y),
\end{equation*}
et, pour tout $x\in X$ et pour tout $n\geq 1$, $f$ induit des isomorphismes
\begin{equation*}
f_*\colon \pi_n(X,x)\to\pi_n(Y,f(x)).
\end{equation*}
\item Les cofibrations sont les morphismes ayant la propriétés de relèvement à droite par rapport aux fibrations triviales.
\end{itemize}
\end{theo}

On dispose de plusieurs caractérisation équivalente de cette catégorie modèle. En voici quelques unes.

\begin{prop}
Les fibrations de $\Top$ sont les fibrations de Serre
\end{prop}

\begin{proof}
Les fibrations de Serre sont définies par une propriété de relèvement par rapport aux inclusions $D^n\to D^n\times [0,1]$, où $D^n$ est la boule unité fermée de dimension $n$. On vérifie aisément qu'on a des homéomorphismes $D^n\times [0,1]\simeq \Real{\Delta^{n+1}}$ se restreignant en des homéomorphismes $D^n\simeq \Real{\Lambda^n_k}$. Ainsi, les conditions de relèvement que vérifient les fibrations de Serre et les fibrations de la structure de modèle sur $\Top$ sont les mêmes.
\end{proof}

\begin{theo}[\cite{QuillenHomotopicalAlgebra}]
La catégorie $\Top$ est à engendrement cofibrant, et les ensembles générateurs sont donnés par
\begin{itemize}
\item $I=\{\Real{\partial(\Delta^n)}\to\Real{\Delta^n}\ |\ n\geq 0\}$,
\item $J=\{\Real{\Lambda^n_k}\to\Real{\Delta^n}\ |\ n\geq 1, \ 0\leq k\leq n\}$.
\end{itemize}
\end{theo}

On note que les cofibrations (triviales) génératrices sont exactement les réalisations des cofibrations (triviales) génératrices de $\sS$. Ainsi, en utilisant l'adjonction $(\Real{-},\Sing)$, (voir la proposition \ref{PropositionAdjonctionSingRealChap2}), on obtient

\begin{prop}\label{PropositionSingCaracteriseFibrationsChap2}
Soit $f\colon X\to Y$ un morphisme de $\Top$. C'est une fibration (triviale) de $\Top$ si et seulement si $\Sing(f)$ est une fibration (triviale) de $\sS$.
\end{prop}

\begin{proof}
La preuve est identique pour les fibrations et les fibrations triviales. Montrons le pour les fibrations. Soit $f\colon X\to Y$ un morphisme de $\Top$, et $\Lambda^n_k\to \Delta^n$ une inclusion de cornet. Un problème de relèvement pour $f$,
\begin{equation*}
\begin{tikzcd}
\Real{\Lambda^n_k}
\arrow{r}{a}
\arrow{d}
&X
\arrow{d}{f}
\\
\Real{\Delta^n}
\arrow[swap]{r}{b}
\arrow[dashrightarrow]{ur}{h}
&Y
\end{tikzcd}
\end{equation*}
correspond, par adjonction, à un problème de relèvement pour $\Sing(f)$.
\begin{equation*}
\begin{tikzcd}
\Lambda^n_k
\arrow{r}{\widetilde{a}}
\arrow{d}
&\Sing(X)
\arrow{d}{\Sing(f)}
\\
\Delta^n
\arrow[swap]{r}{\widetilde{b}}
\arrow[dashrightarrow]{ur}{\widetilde{h}}
&\Sing(Y)
\end{tikzcd}
\end{equation*}
Mais alors, la donnée d'un relèvement pour l'un de ces deux problèmes fournit un relèvement pour l'autre, par adjonction. Ainsi, $f$ est une fibration si et seulement si $\Sing(f)$ est une fibration.
\end{proof}

\begin{prop}\label{PropositionCWCofibrant}
Tout CW-complexe est cofibrant.
\end{prop}

\begin{proof}
Soit $X$ un CW-complexe. Pour tout $n\geq -1$, on note $X_n$ l'union des cellules de $X$ de dimension $\leq n$ ($X_{-1}=\emptyset$). Alors, l'inclusion $X_n\to X_{n+1}$ correspond à la somme amalgamée
\begin{equation*}
\begin{tikzcd}
\coprod \Real{\partial(\Delta^{n+1})}
\arrow{r}
\arrow{d}
&X_n
\arrow{d}
\\
\coprod\Real{\Delta^{n+1}}
\arrow{r}
&X_{n+1}
\end{tikzcd}
\end{equation*}
Comme l'inclusion $\Real{\partial(\Delta^{n+1})}\to\Real{\Delta^n}$ est une cofibration de $\Top$, et que les cofibrations sont stables par unions disjointes et sommes amalgamées, l'inclusion $X_n\to X_{n+1}$ est une cofibration. Mais alors, la composition transfinie
\begin{equation*}
\emptyset=X_{-1}\to X_0\to X_1\to\dots\to X
\end{equation*}
est une cofibration, et $X$ est cofibrant.
\end{proof}

\begin{remarque}
La réciproque est plus difficile à montrer. Si $X$ est un espace cofibrant, par l'argument du petit objet, en factorisant le morphisme $\emptyset\to X$ on obtient que $X$ est le rétract d'un espace $Z^{\infty}=\colim Z^n$, où $Z^{-1}=\emptyset$, et pour tout $n\geq -1$, $Z^{n+1}$ est obtenu à partir de $Z^n$ comme la somme amalgamée
\begin{equation*}
\begin{tikzcd}
\coprod_{(\alpha,\beta)} \Real{\partial(\Delta^{k_{\alpha}})}
\arrow{r}{\coprod\alpha}
\arrow{d}
&Z^n
\arrow{d}
\\
\coprod_{(\alpha,\beta)}\Real{\Delta^{k_{\alpha}}}
\arrow{r}
&Z^{n+1}
\end{tikzcd}
\end{equation*}
L'espace $Z^{\infty}$ est donc naturellement décomposé en cellules. Cependant, rien ne garantit que pour toute cellule $f_{\alpha}\colon \Real{\Delta^{k_{\alpha}}}\to Z^{\infty}$, $f_{\alpha}(\Real{\partial(\Delta^{k_{\alpha}})})$ est inclus dans une union de cellules de dimensions $\leq k_{\alpha}$. Néanmoins, par le théorème \ref{TheoremeEnsembleSimplicialCWComplexe}, on sait que la réalisation d'un ensemble simplicial est un CW-complexe, donc un objet cofibrant par la proposition \ref{PropositionCWCofibrant}, et on verra dans la section suivante que pour tout espace topologique $X$, la counité de l'adjonction $(\Real{-},\Sing)$,
\begin{equation*}
\Real{\Sing(X)}\to X,
\end{equation*}
est une équivalence faible. Finalement, si $X$ est cofibrant, la counité est une équivalence d'homotopie, et $X$ est homotopiquement équivalent à un CW complexe.
\end{remarque}

\begin{remarque}
La propriété \ref{PropositionCWCofibrant} fournit une preuve immédiate du théorème de Whitehead. Tous les espaces topologiques sont fibrants et les CW-complexes sont des objets cofibrants. Les équivalences faibles coincident avec les équivalences d'homotopies entre objets cofibrants-fibrants (Théorème \ref{TheoremeFibrantCofibrantEquivalencesFaiblesHomotopiesChap2}).
\end{remarque}

\subsection{L'adjonction de Kan-Quillen}
\label{SubsectionAdjonctionKanQuillenChap2}
On a vu à la proposition \ref{PropositionAdjonctionSingRealChap2} que les catégories $\Top$ et $\sS$ étaient liée par une adjonction. Cette adjonction, que l'on appelle adjonction de Kan-Quillen permet de comparer les deux catégories modèles.

\begin{prop}
L'adjonction
\begin{equation*}
\Real{-}\colon \sS\leftrightarrow\Top\colon \Sing
\end{equation*}
est une adjonction de Quillen.
\end{prop}

\begin{proof}
On a vu à la proposition \ref{PropositionSingCaracteriseFibrationsChap2} que le foncteur $\Sing$ préserve les fibrations. Montrons que le foncteur $\Real{-}$ préserve les cofibrations. Soit $X\to Y$ une cofibration de $\sS$. C'est un monomorphisme, et d'après la preuve de la proposition \ref{PropositionCofibrationsSSetMonomorphismes}, on peut écrire ce morphisme comme la composition transfinie
\begin{equation*}
X=X\cup\sk_{-1}(Y)\to X\cup\sk_0(Y)\to X\cup\sk_1(Y)\to\dots\to Y
\end{equation*}
où pour tout $n\geq -1$, l'inclusion $X\cup\sk_n(Y)\to X\cup\sk_{n+1}(Y)$ est de la forme
\begin{equation*}
\begin{tikzcd}
\coprod \partial(\Delta^{n+1})
\arrow{r}
\arrow{d}
&X\cup \sk_n(Y)
\arrow{d}
\\
\coprod \Delta^{n+1}
\arrow{r}
&X\cup \sk_{n+1}(Y)
\end{tikzcd}
\end{equation*}
La réalisation étant un adjoint à gauche, elle commute avec les colimites. Le morphisme $\Real{X}\to\Real{Y}$ est donc la composition transfinie
\begin{equation*}
\Real{X}=\Real{X\cup\sk_{-1}(Y)}\to \Real{X\cup\sk_0(Y)}\to \Real{X\cup\sk_1(Y)}\to\dots\to \Real{Y},
\end{equation*}
où pour $n\geq -1$, les inclusions successives sont les sommes amalgamées
\begin{equation*}
\begin{tikzcd}
\coprod \Real{\partial(\Delta^{n+1})}
\arrow{r}
\arrow{d}
&\Real{X\cup \sk_n(Y)}
\arrow{d}
\\
\coprod \Real{\Delta^{n+1}}
\arrow{r}
&\Real{X\cup \sk_{n+1}(Y)}
\end{tikzcd}
\end{equation*}
Comme $\Real{\partial(\Delta^n)}\to\Real{\Delta^n}$ est une cofibration de $\Top$ pour tout $n\geq -1$, et comme les cofibrations sont stables par unions disjointes, sommes amalgamées et compositions transfinies, on en déduit que $\Real{X}\to\Real{Y}$ est une cofibration de $\Top$.
\end{proof}

\begin{prop}
Les foncteurs $\Real{-}\colon \sS\to \Top$ et $\Sing\colon \Top\to \sS$ préservent les équivalences faibles.
\end{prop}

\begin{proof}
Tous les ensembles simpliciaux sont cofibrants, et tous les espaces topologiques sont fibrants. Comme l'adjonction $(\Real{-},\Sing)$ est une adjonction de Quillen, c'est une application de la proposition \ref{PropositionAdjonctionQuillenPreserveEquivalencesFaibles}.
\end{proof}

On peut finalement énoncer le théorème suivant. 

\begin{theo}[\cite{QuillenHomotopicalAlgebra}]\label{TheoremeEquivalenceKanQuillenChap2}
L'adjonction
\begin{equation*}
\Real{-}\colon \sS\leftrightarrow\Top\colon \Sing
\end{equation*}
est une équivalence de Quillen.
\end{theo}

\begin{proof}
Soient $K$ un ensemble simplicial, $X$ un espace topologique et $f\colon \Real{K}\to X$ une application continue. Notons $\widetilde{f}\colon K\to \Sing(X)$ son image par l'adjonction $(\Real{-},\Sing)$. On considère les diagrammes commutatifs suivants, où $\eta$ et $\epsilon$ sont respectivement l'unité et la counité de l'adjonction $(\Real{-},\Sing)$.
\begin{equation}\label{DiagrammeKanQuillenChap21}
\begin{tikzcd}
&\Real{\Sing(X)}
\arrow{dr}{\epsilon_X}
\\
\Real{K}
\arrow[swap]{rr}{f}
\arrow{ur}{\Real{\widetilde{f}}}
&&X
\end{tikzcd}
\end{equation}
\begin{equation}\label{DiagrammeKanQuillenChap22}
\begin{tikzcd}
K
\arrow{rr}{\widetilde{f}}
\arrow[swap]{dr}{\eta_K}
&&\Sing(X)\\
&\Sing(\Real{K})
\arrow[swap]{ur}{\Sing(f)}
\end{tikzcd}
\end{equation}
Supposons que $f\colon \Real{K}\to X$ est une équivalence faible. Comme le foncteur $\Sing$ préserve les équivalences faibles, par l'axiome de deux sur trois dans le diagramme \ref{DiagrammeKanQuillenChap22}, il suffit de montrer que $\eta_K$ est une équivalence faible pour en déduire que $\widetilde{f}$ est une équivalence faible. De même, si $\widetilde{f}$ est une équivalence faible, par deux sur trois dans le diagramme \ref{DiagrammeKanQuillenChap21}, il suffit de montrer que $\epsilon_X$ est une équivalence faible pour en déduire que $f$ est une équivalence faible. Ces deux résultats (Lemmes \ref{LemmeUniteEquivalenceFaibleChap2} et \ref{LemmeCouniteEquivalenceFaibleChap2}) ont été montrés par J. Milnor \cite{MilnorGeometricRealization}.
\end{proof}

\begin{lemme}\label{LemmeUniteEquivalenceFaibleChap2}
Soit $K$ un ensemble simplicial. L'unité $\eta_K\colon K\to \Sing(\Real{K})$ est une équivalence faible de $\sS$.
\end{lemme}

\begin{lemme}\label{LemmeCouniteEquivalenceFaibleChap2}
Soit $X$ un espace topologique. La counité $\epsilon_X\colon \Real{\Sing(X)}\to X$ est une équivalence faible de $\Top$.
\end{lemme}
\chapter[La catégorie modèle des ensembles simpliciaux filtrés]{La catégorie modèle des ensembles simpliciaux filtrés}
\label{ConstructionCMFSSetP}
\chaptermark{Ensembles simpliciaux filtrés}

Dans ce chapitre, on construit une catégorie modèle pour la théorie homotopique des espaces filtrés. Par analogie avec le cas non-filtré, il est naturel de passer par l'étude d'une catégorie d'objets simpliciaux. Ainsi,  on étudie la catégorie modèle des ensembles simpliciaux filtrés. 

Dans la section \ref{SectionCategorieSSetP}, on définira la catégorie $\sS_P$ des ensembles simpliciaux filtrés, et on verra comment la notion d'homotopie stratifiée se traduit pour ses objets. 

Ensuite, dans la section \ref{SectionConstructionCMFsSetP}, on appliquera les méthodes de Cisinski \cite{Cisinski} pour définir une structure de modèle sur $\sS_P$. C'est le théorème \ref{ExistenceCMFCisinski}. 

Par construction, cette catégorie sera engendrée de façon cofibrante, mais les résultats de Cisinski ne permettent pas de façon générale de caractériser les classes de cofibrations triviales génératrices. La section \ref{SectionCaracterisationFibrations} permet d'obtenir une description des fibrations "à la Kan" en terme de propriétés de relèvement par rapport à des cornets. Cette caractérisation passe par la définition d'un foncteur de remplacement fibrant $\Exi_P$. On obtient finalement une description explicite de la catégorie modèle $\sSU_P$, c'est le contenu du théorème \ref{TheoDescriptionExpliciteSSetP}. La preuve de ce dernier résultat utilise des résultats techniques sur les catégories de préfaisceaux qui sont prouvés dans les annexes \ref{ChapitreCaracterisationFibrationsAnnexe} et \ref{ChapitreCaracterisationMorphismeXExXAnnexe}.

\section{La catégorie des ensembles simpliciaux filtrés}
\label{SectionCategorieSSetP}
\subsection{Ensembles simpliciaux filtrés}
\begin{remarque}
Dans tout ce texte, un ensemble ordonné est un ensemble muni d'un ordre \textit{partiel}.
\end{remarque}
\begin{defin}
Soit $P$ un ensemble ordonné. Son nerf, noté $N(P)$ est l'ensemble simplicial dont les simplexes sont donnés par des suites croissantes d'éléments de $P$ :
\begin{equation*}
N(P)_n=\{(p_0,\dots,p_n)\ | p_0\leq\dots\leq p_n\in P\}
\end{equation*}
Les applications faces sont obtenues par effacement d'un des termes de la suite :
\begin{align*}
d_i\colon N(P)_n&\to N(P)_{n-1}\\
(p_0,\dots,p_n)&\mapsto (p_0,\dots,\widehat{p_i},\dots,p_n)
\end{align*}

et les applications dégénerescences sont obtenues par répétition d'un des termes de la suite :
\begin{align*}
s_i\colon N(P)_n&\to N(P)_{n+1}\\
(p_0,\dots,p_{i-1},p_i,p_{i+1},\dots,p_n)&\mapsto (p_0,\dots,p_{i-1},p_i,p_i,p_{i+1},\dots,p_n)
\end{align*}
\end{defin}

\begin{defin}
Un ensemble simplicial filtré au dessus de $P$ est la donnée :
\begin{itemize}
\item d'un ensemble simplicial $X$,
\item d'une filtration $\varphi\colon X\to N(P)$.
\end{itemize}
Une application filtrée entre $(X,\varphi\colon X\to N(P))$ et $(Y,\psi\colon Y\to N(P))$ est la donnée d'une application $f\colon X\to Y$ telle que le triangle suivant est commutatif.
\begin{equation*}
\begin{tikzcd}
X
\arrow[swap]{dr}{\varphi}
\arrow{rr}{f}
&\phantom{X}
&Y
\arrow{dl}{\psi}
\\
\phantom{X}
&N(P)
\end{tikzcd}
\end{equation*}
La donnée des ensembles simpliciaux filtrés et des applications filtrées fournit une description de la catégorie $\sS_P$ des ensembles simpliciaux filtrés au dessus de $P$.
\end{defin}

\begin{defin}
On appelle simplexe filtré tout objet de $\sS_P$ de la forme $(\Delta^N,\varphi\colon \Delta^N\to N(P))$, où $\Delta^N$, $N\geq 0$ est le $N$-simplexe et $\varphi$ est une filtration quelconque. Par commodité, on identifie cet objet $(\Delta^N,\varphi)$ avec le simplexe de $N(P)$, $\varphi(\Delta^{N})=[p_0,\dots,p_N]$. On note $\Delta(P)$ la sous-catégorie pleine de $\sS_P$ formée des simplexes filtrés.

\end{defin}

\begin{prop}\label{CategoriePrefaisceaux}
La catégorie $\sS_P$ est équivalente à la catégorie des préfaisceaux sur $\Delta(P)$. Cette équivalence est réalisée par le foncteur suivant :

\begin{align*}
\sS_P&\to \Fun(\Delta(P)^{\op},\Set)\\
(X\to N(P))&\mapsto \Hom_{\sS_P}(-,X)
\end{align*}
\end{prop}

\begin{proof}
Soit $Y\colon\Delta(P)^{\op}\to\Set$ un foncteur, et soit $N\geq 0$. On pose 
\begin{equation*}
Z_N=\coprod_{\varphi\colon \Delta^N\to N(P)} Y(\Delta^N,\varphi).
\end{equation*}
Soient $N\geq 1$, $0\leq i\leq N$ et $(\Delta^N,\varphi\colon\Delta^N\to N(P))$ un simplexe filtré. L'inclusion $\Delta^{N-1}\subset \Delta^N$ comme la $i$-ème face donne l'application filtrée :
\begin{equation*}
\begin{tikzcd}
\Delta^{N-1}
\arrow[swap]{dr}{\varphi\circ D^i}
\arrow{rr}{D^i}
&\phantom{X}
&\Delta^N
\arrow{dl}{\varphi}
\\
\phantom{X}
&N(P)
\end{tikzcd}
\end{equation*}

de même, pour $N\geq 0$, la $i$-ème dégénérescence $\Delta^{N+1}\to \Delta^N$ donne l'application filtrée 

\begin{equation*}
\begin{tikzcd}
\Delta^{N+1}
\arrow[swap]{dr}{\varphi\circ S^i}
\arrow{rr}{S^i}
&\phantom{X}
&\Delta^N
\arrow{dl}{\varphi}
\\
\phantom{X}
&N(P)
\end{tikzcd}
\end{equation*}

Ainsi, en appliquant le foncteur contravariant $Y$, on obtient
\begin{align*}
d_i\colon Y(\Delta^N,\varphi)&\to Y(\Delta^{N-1},\varphi\circ D^i)\\ s_i\colon Y(\Delta^N,\varphi)&\to Y(\Delta^{N+1},\varphi\circ S^i).
\end{align*}

Puis, en passant à l'union sur toutes les filtrations de $\Delta^N$, on obtient :
\begin{align*}
d_i\colon Z_N&\to Z_{N-1}\\ 
s_i\colon Z_N&\to Z_{N+1}.
\end{align*}
Finalement, $Z$ est un ensemble simplicial. Pour produire la filtration $\psi\colon Z\to N(P)$, on pose $\psi(Y(\Delta^N,\varphi))=\varphi(\Delta^N)$.
\end{proof}

\begin{corollaire}
Soit $\fil{X}$ un ensemble simplicial filtré. Alors on a
\begin{equation*}
\fil{X}\simeq\colim\limits_{\sigma\in\Hom(\Delta^\varphi,\fil{X})}\Delta^{\varphi}
\end{equation*}
\end{corollaire}

\subsection{Foncteurs utiles}

\begin{prop}
Il existe une adjonction 
\begin{equation*}
U\colon \sS_P\leftrightarrow \sS\colon F
\end{equation*}
où $U(X,\varphi\colon X\to N(P))=X$, et $F(K)=(K\times N(P),\pr_{N(P)}\colon K\times N(P)\to N(P))$.
\end{prop}

\begin{proof}
Soient $(X,\varphi\colon X\to N(P))$ un ensemble simplicial filtré et $K$ un ensemble simplicial. On a les isomorphismes suivants
\begin{align*}
\Hom_{\sS_P}((X,\varphi),(K\times N(P),\pr_{N(P)}))& \to \Hom_{\sS}(X,K)\\
(f\colon X\to K\times N(P))&\mapsto (\pr_{K}\circ f\colon X\to K)\\
&\text{et}\\
\Hom_{\sS}((X,K)& \to \Hom_{\sS_P}((X,\varphi),(K\times N(P),\pr_{N(P)}))\\
(g\colon X\to K)&\mapsto (g\times\varphi\colon X\to K\times N(P))
\end{align*}
\end{proof}

\begin{defin}\label{ProduitFiltre}
Soient $(X,\varphi_X)$ et $(Y,\varphi_Y)$ deux ensembles simpliciaux filtrés. On définit leur produit filtré, $(X,\varphi_X)\times_{N(P)} (Y,\varphi_Y)$ comme le produit fibré suivant :
\begin{equation*}
\begin{tikzcd}
X\times_{N(P)} Y
\arrow[swap]{d}{\pr_X}
\arrow{r}{\pr_Y}
&Y
\arrow{d}{\varphi_Y}
\\
X
\arrow[swap]{r}{\varphi_X}
&N(P)
\end{tikzcd}
\end{equation*}
où la filtration sur $X\times_{N(P)} Y$ est la composition $\varphi_X\circ \pr_X = \varphi_Y\circ \pr_Y$.
\end{defin}

\begin{defin}\label{TenseurSimplicial}
Soient $(X,\varphi)$ un ensemble simplicial filtré et $K$ un ensemble simplicial. On définit $K\otimes (X,\varphi)$ comme le produit filtré $F(K)\times_{N(P)}(X,\varphi)$. Cette définition donne lieu à un bifoncteur :
\begin{equation*}
-\otimes-\colon \sS\times \sS_P\to \sS_P
\end{equation*}
\end{defin}

\begin{remarque}
Avec cette définition, $F(K)$ devient $K\otimes(N(P),\Id_{N(P)})$.
\end{remarque}

\subsection{Homotopies filtrées}

\begin{defin}\label{DefHomotopiesFiltrees}
Soient $f,g\colon (X,\varphi_X)\to(Y,\varphi_Y)$ deux applications filtrées. Une homotopie filtrée élémentaire entre $f$ et $g$ est une application filtrée 
$H\colon \Delta^1\otimes (X,\varphi_X)\to (Y,\varphi_Y)$
 telle que le diagramme suivant commute 
\begin{equation*}
\begin{tikzcd}
X
\arrow[swap]{dr}{i_0\otimes \Id_X}
\arrow[bend left = 12]{drr}{f}
&\phantom{X}
&\phantom{X}
\\
\phantom{X}
&\Delta^1\otimes X
\arrow{r}{H}
&Y
\\
X
\arrow{ur}{i_0\otimes\Id_X}
\arrow[swap,bend right = 12]{urr}{g}
\end{tikzcd}
\end{equation*}
où $i_0\colon \Delta^0\to \Delta^1$ est l'inclusion $\{0\}\hookrightarrow [0,1]$, et où on fait l'identification $(X,\varphi_X)\simeq \Delta^0\otimes (X,\varphi_X)$. Plus généralement, on considère la relation d'équivalence $\sim_P$ sur l'ensemble $\Hom_{\sS_P}((X,\varphi_X),(Y,\varphi_Y))$ engendrée par $f\sim_P g$ s'il existe une homotopie filtrée élémentaire entre $f$ et $g$. On dit que deux applications filtrées $f$ et $g$ sont homotopes au sens filtré (ou par une homotopie filtrée) si $f\sim_P g$.
\end{defin}

\begin{defin}
Une application filtrée $f\colon (X,\varphi_X)\to(Y,\varphi_Y)$ est une équivalence d'homotopie filtrée s'il existe une application filtrée $g\colon (Y,\varphi_Y)\to (X,\varphi_X)$ tel que $g\circ f\sim_P\Id_X$ et $f\circ g\sim_P\Id_Y$. Dans ce cas, on dit que $(X,\varphi_X)$ et $(Y,\varphi_Y)$ sont filtrés homotopiquement équivalents.
\end{defin}

\begin{remarque}
En particulier, une équivalence d'homotopie filtrée $f\colon (X,\varphi_X)\to (Y,\varphi_Y)$ fournit une équivalence d'homotopie (non-filtrée) entre $X$ et $Y$. La réciproque est fausse. En effet, si $f\colon (X,\varphi_X)\to (Y,\varphi_Y)$ est une application filtrée telle que l'application entre ensembles simpliciaux $f\colon X\to Y$ est une équivalence d'homotopie, alors il existe $g\colon Y\to X$ telle que $f\circ g\sim \Id_Y$ et $g\circ f\sim \Id_X$. Cependant, $g$ n'est pas nécessairement une application filtrée. De plus, même si $g$ préserve les filtrations, en général $f\circ g\not\sim_P \Id_Y$ et $g\circ f\not\sim_P \Id_X$.
\end{remarque}

\section{Construction d'une structure de modèle}
\label{SectionConstructionCMFsSetP}
\subsection{Cornets admissibles}

\begin{defin}
Soit $\Delta^{\varphi}=(\Delta^N,\varphi)$ un simplexe filtré. On appelle $k$-ème cornet de $\Delta^{\varphi}$, noté $\Lambda^{\varphi}_k$ l'ensemble simplicial filtré $(\Lambda^N_k,\varphi_{|\Lambda^N_k})$. Où $\Lambda^N_k=\cup_{i\not = k}d_i\Delta^N$ est le $k$-ème cornet de $\Delta^N$.
\end{defin}

\begin{prop}\label{PropCornetAdmissible}
Soient $\Delta^{\varphi}=(\Delta^N,\varphi)$ un simplexe filtré, et $0\leq k\leq N$ un entier. Les assertions suivantes sont équivalentes :
\begin{enumerate}
\item il existe $m\in \{k-1,k\}$ tel que $\varphi=\varphi\circ D_k\circ S_m$ où $D_k$ et $S_m$ sont définies comme dans la preuve de \ref{CategoriePrefaisceaux}. 
\item $\varphi(\Delta^N)=[p_0,\dots,p_N]$ avec $p_k=p_{k+1}$ ou $p_k=p_{k-1}$.
\item L'inclusion $\Lambda^{\varphi}_k\to \Delta^{\varphi}$ est une équivalence d'homotopie filtrée élémentaire. 	
\end{enumerate}
\end{prop}

\begin{proof}
($1\Rightarrow 2$) Soit $\Delta^{\varphi}$ et $k$ tels que $\varphi=\varphi\circ D_k\circ S_m$ avec $m=k$ ou $m=k-1$. On a le diagramme commutatif suivant :
\begin{equation*}
\begin{tikzcd}
\Delta^N
\arrow[swap]{dr}{\varphi}
\arrow{r}{S_m}
&\Delta^{N-1}
\arrow{r}{D_k}
&\Delta^{N}
\arrow{dl}{\varphi}
\\
\phantom{X}
&N(P)
\end{tikzcd}
\end{equation*}
Notons $e_0,\dots,e_N$ les $N+1$ sommets de $\Delta^N$. Dans le cas $m=k$, on a $D_k(S_k(e_k))=e_{k+1}$. On en déduit que $p_{k+1}=\varphi(e_{k+1})=\varphi(D_k(S_k(e_k)))=\varphi(e_k)=p_k$. De même, si $m=k-1$, on obtient $p_k=p_{k-1}$.

($2\Rightarrow 1$) Soient $\Delta^N, \varphi, k$ tels que $\varphi(\Delta^N)=[p_0,\dots,p_N]$. La composition $D_k\circ S_{k}$ peut se décrire sur les sommets de $\Delta^N$ par :
\begin{align*}
(\Delta^N)_0&\to (\Delta^N)_0\\
e_i&\mapsto \left\{\begin{array}{cl}
e_i & \text{ si $i\not =k$}\\
e_{k+1} &\text{ si $i=k$}
\end{array}\right.
\end{align*}
En particulier, si $p_{k+1}=p_k$, on a $\varphi(D_k(S_k(e_i)))=p_i$ pour tout $0\leq i\leq N$. Comme les simplexes de $N(P)$ sont entièrement déterminés par leurs sommets, on en déduit que dans ce cas $\varphi=\varphi\circ D_k\circ S_k$. Le même raisonnement dans le cas où $p_k=p_{k-1}$ aboutit à $\varphi=\varphi D_k\circ S_{k-1}$.

($2\Rightarrow 3$) On suppose que $\varphi(\Delta^N)=[p_0,\dots,p_N]$ avec $p_k=p_{k+1}$ (le cas $p_k=p_{k-1}$ est symétrique). Considérons la composition 
\begin{equation*}
(\Delta^N,\varphi\circ D_{k+1}\circ S_k)\xrightarrow{S_k} (\Delta^{N-1},\varphi\circ D_{k+1})\xrightarrow{D_{k+1}}(\Delta^N,\varphi).
\end{equation*}
D'une part, on a $\varphi\circ D_{k+1}\circ S_k=\varphi$, par les mêmes arguments que précédemment. Donc la composition est un morphisme 
\begin{equation*}
(\Delta^N,\varphi)\xrightarrow{D_{k+1}\circ S_k}(\Delta^N,\varphi).
\end{equation*} 
D'autre part, on remarque que $D_{k+1}\circ S_k(\Delta^N)\subseteq d_{k+1}\Delta^N\subseteq \Lambda^N_k$. En particulier, on a une application filtrée bien définie :
\begin{equation*}
(\Delta^N,\varphi)\xrightarrow{D_{k+1}\circ S_k} (\Lambda^N_k,\varphi_{|\Lambda^N_k}).
\end{equation*}
On va montrer que c'est un inverse à homotopie filtrée près
de l'inclusion $(\Lambda^N_k,\varphi_{|\Lambda^N_k})\to (\Delta^N,\varphi)$.
On définit, sur les sommets de $\Delta^N$, l'homotopie filtrée suivante :
\begin{align*}
H\colon \Delta^1\otimes(\Delta^N,\varphi)&\to (\Delta^N,\varphi) \\
(\epsilon,e_i)&\mapsto \left\{\begin{array}{cl}
e_i & \text{ si $\epsilon = 1$ ou $i\not = k+1$}\\
e_{k} & \text{ si $\epsilon =0$ et $i=k+1$}
\end{array}\right. 
\end{align*}
L'homotopie est bien définie car les simplexes de $\Delta^N$, et de $\Delta^1\times\Delta^N$ sont uniquement déterminés par leur sommets, et elle est filtrée car $\varphi(e_k)=\varphi(e_{k+1})$, par hypothèse. On constate donc que $H$ fournit une homotopie filtrée élémentaire entre $D_{k+1}\circ S_k$ et $\Id_{\Delta^N}$.
Montrons maintenant que $H(\Delta^1\times\Lambda^N_k)\subseteq \Lambda^N_k$. Il suffit pour cela de montrer que le simplexe $[e_0,\dots,e_N]\subset \Delta^N$ n'est pas atteint par $H_{|\Delta^1\times \Lambda^N_k}$. Soit $\sigma\in \Delta^1\times \Lambda^N_k$ un $N$-simplexe. Alors on a $\sigma\in \Delta^1\times d_{l}\Delta^N$ pour un certain $l\not =k$. On peut écrire $\sigma$ sous une des deux formes suivantes :
\begin{align*}
\sigma=[(0,e_0),\dots, \widehat{(0,e_l)},\dots,(0,e_m),(1,e_m),\dots,(1,e_N)]\\
\text{ou}\\
\sigma=[(0,e_0),\dots,(0,e_m),(1,e_m),\dots,\widehat{(1,e_l)},\dots,(1,e_N)]
\end{align*}
On calcule $H(\sigma)$. En général, on a $H(\epsilon,e_i)\in \{e_i,e_k\}$. En particulier, comme $l\not = k$, $H(\sigma)$ n'atteint pas $e_l$, et donc $H(\sigma)\subseteq d_l(\Delta^N)\subseteq\Lambda^N_k$. Finalement, on obtient l'existence de la restriction suivante :
\begin{align*}
H_{|\Delta^1\times\Lambda^N_k}\colon \Delta^1\otimes(\Lambda^N_k,\varphi_{|\Lambda^N_k})&\to (\Lambda^N_k,\varphi_{|\Lambda^N_k}) \\
(\epsilon,e_i)&\mapsto \left\{\begin{array}{cl}
e_i & \text{ si $\epsilon = 1$ ou $i\not = k+1$}\\
e_{k} & \text{ si $\epsilon =0$ et $i=k+1$}
\end{array}\right. 
\end{align*}
En particulier, $H_{|\Delta^1\times\Lambda^N_k}$ fournit l'homotopie filtrée élémentaire entre $(D_{k+1}\circ S_k)_{|\Lambda^N_k}$ et $\Id_{\Lambda^N_k}$.

($3\Rightarrow 2$) Notons $j\colon (\Lambda^N_k,\varphi_{|\Lambda^N_k})\to (\Delta^N,\varphi)$ l'inclusion du cornet. On suppose qu'il existe une application filtrée $r\colon (\Delta^N,\varphi)\to(\Lambda^N_k,\varphi_{|\Lambda^N_k})$ et des homotopies filtrées élémentaires $H$ et $G$ entre $r\circ j$ et $\Id_{\Lambda^N_k}$ et entre $j\circ r$ et $\Id_{\Delta^N}$ respectivement.  Si $N=1$, $\Lambda^1_k$ est un $0$-simplexe et $p_0=p_1$ car $r$ est filtrée.
Pour $N\geq 2$, on suppose que $p_k\not=p_{k+1},p_{k-1}$ et par symétrie, on suppose que $H_{|\{0\}\times\Lambda^N_k}=r\circ j$ et $H_{|\{1\}\times \Lambda^N_k}=\Id_{\Lambda^N_k}$. En particulier, cette hypothèse implique que pour tout $0\leq i\leq N$, $r\circ j (e_i)=e_{i'}$ avec $i'\leq i$. 
De plus, comme il n'existe pas de section $s\colon \Delta^N\to \Lambda^N_k$ telle que $s\circ j=\Id_{\Lambda^N_k}$, cela implique qu'il doit exister $i$ tel que $r\circ j(e_i)\not =e_{i}$.
On définit :
\begin{equation*}
l=\min\{i\ |\ r\circ j(e_i)\not = e_{i}\}
\end{equation*}
Comme $r\circ j$ est filtrée, cela implique que $p_k=\varphi(e_k)=\varphi\circ r\circ j (e_k)$. D'autre part, comme on a supposé que $p_k\not= p_{k+1},p_{k-1}$, on a $\varphi(e_i)=p_k\Leftrightarrow i=k$. On en déduit que $r\circ j(e_k)=e_k$, et donc que $l\not = k$.
Ensuite, on observe que $r\circ j(e_l)=e_{l'}$ avec $l'<l$. Mais $r\circ j (e_{l-1})=e_{l-1}$, par définition de $l$, et $r\circ j$ envoie le segment $[e_{l-1},e_{l}]$ sur un segment $[e_{l-1},e_{l'}]$ ce qui impose que $l'\geq l-1$. Finalement, on otient $l'=l-1$, et par le même raisonnement que précédemment, on en déduit qu'en particulier $l-1\not = k$. On calcule l'image par $H$ du simplexe 
\begin{equation*}
\sigma=[(0,e_0),\dots,(0,e_{l-2}),\widehat{(0,e_{l-1})},(0,e_l)(1,e_l),\dots,(1,e_N)]\subseteq \Delta^1\times d_{l-1}\Delta^N\subseteq \Delta^1\times \Lambda^N_k
\end{equation*}
\begin{equation*}
H(\sigma)=[e_0,\dots, e_{l-2},\widehat{e_{l-1}},e_{l-1},e_l,\dots,e_N]=\Delta^N\not\subseteq \Lambda^N_k
\end{equation*}
On en déduit que $p_k=p_{k+1}$ ou $p_k=p_{k-1}$.

\end{proof}

\begin{defin}
L'inclusion d'un cornet $\Lambda^{\varphi}_k\to \Delta^{\varphi}$ est dite admissible, si les conditions équivalentes de la proposition \ref{PropCornetAdmissible} sont vérifiées.
\end{defin}

\begin{remarque}
Dans la proposition \ref{PropCornetAdmissible}, l'adjectif élémentaire peut en fait être omis de la condition 3. En effet, parmi les conséquences des résultats prouvés dans la suite de ce chapitre, on obtient que les cornets admissibles sont les seuls cornets qui sont aussi des équivalences faibles filtrées. En particulier, ce sont les seules à être des équivalences d'homotopie filtrées. Cependant, la formulation de la proposition \ref{PropCornetAdmissible} permet d'écrire une preuve élémentaire.
\end{remarque}

\subsection{Propriétés de relèvement et saturation}

\begin{defin}Soient $f\colon A\to B$, et $g\colon X\to Y$ deux morphismes d'une catégorie $\mathcal{C}$. On dit que $f$ admet la propriété de relèvement à gauche (Left Lifting Property) par rapport à $g$ et que $g$ admet la propriété de relèvement à droite (Right Lifting Property) par rapport à $f$ si pour tout carré commutatif dans $\mathcal{C}$ 
\begin{equation*}
\begin{tikzcd}
A
\arrow{r}
\arrow[swap]{d}{f}
&X
\arrow{d}{g}
\\
B\arrow{r}
\arrow[dashed]{ur}{h}
&Y
\end{tikzcd} 
\end{equation*}
il existe un morphisme $h\colon B\to X$ qui fait commuter les deux triangles. 

Soit $\Lambda$ une classe de morphismes de $\mathcal{C}$. On dit que $g$ admet la RLP par rapport à $\Lambda$ si pour tout $f$ dans $\Lambda$, $g$ admet la RLP par rapport à $f$. On note :
\begin{equation*}
r(\Lambda)=\{g\ |\ \text{$g$ admet la RLP par rapport à $\Lambda$}\}.
\end{equation*}

De façon symétrique, si $\Gamma$ est une classe de morphismes de $\mathcal{C}$, on dit que $f$ admet la LLP par rapport à $\Gamma$ si $f$ admet la LLP par rapport à $g$ pour tout $g$ dans $\Gamma$. On note :
\begin{equation*}
l(\Gamma)=\{ f\ |\ \text{$f$ admet la LLP par rapport à $\Gamma$}\}.
\end{equation*}
\end{defin}

\begin{defin} Soit $\Lambda$ une classe de morphisme de $\mathcal{C}$. La classe $\Lambda$ est saturée si elle vérifie les conditions suivantes :
\begin{itemize}
\item $\Lambda$ contient les isomorphismes de $\mathcal{C}$.
\item $\Lambda$ est stable par image directe. C'est à dire que pour tout carré cocartésien 
\begin{equation*}
\begin{tikzcd}
A
\arrow{r}
\arrow{d}{f}
&C
\arrow{d}{g}
\\
B
\arrow{r}
&D
\end{tikzcd}
\end{equation*}
si $f$ est dans $\Lambda$, alors $g$ est dans $\Lambda$.
\item $\Lambda$ est stable par rétraction. C'est à dire que étant donné $f\colon A\to B$ et $f'\colon A'\to B'$ tel qu'il existe un diagramme commutatif 
\begin{equation*}
\begin{tikzcd}
A'
\arrow{r}{i}
\arrow{d}{f'}
&A
\arrow{r}{r}
\arrow{d}{f}
&A'
\arrow{d}{f'}
\\
B'
\arrow{r}{j}
&B
\arrow{r}{s}
&B'
\end{tikzcd}
\end{equation*}
avec $r\circ i= \Id_{A'}$ et $s\circ j=\Id_{B'}$, Si $f$ est dans $\Lambda$ alors $f'$ est dans $\Lambda$.
\item $\Lambda$ est stable par somme directe. C'est à dire que si $f_s\colon A_s\to B_s, s\in S$ est une collection de morphismes de $\Lambda$, alors leur somme directe 
\begin{equation*}
\coprod_{s\in S} f_s\colon \coprod_{s\in S} A_s\to\coprod_{s\in S} B_s
\end{equation*}
est dans $\Lambda$.
\item $\Lambda$ est stable par composition dénombrable. C'est à dire que si $f_i\colon A_i\to A_{i+1}, i\in \mathbb{N}$ est une collection de morphismes de $\Lambda$, la composition dénombrable
\begin{equation*}
A_0\to\lim_{\to} A_i
\end{equation*}
est dans $\Lambda$.
\end{itemize}
\end{defin}

\begin{defin}
Soit $\mathcal{A}$ une classe de flèche dans $\mathcal{C}$, on appelle classe saturée engendrée par $\mathcal{A}$, et on note $\text{Sat}(\mathcal{A})$, la plus petite classe saturée de $\mathcal{C}$ contenant $\mathcal{A}$.
\end{defin}

\begin{prop}\label{SatureLifting}
Soit $\mathcal{A}$ un ensemble de morphismes de $\sS_P$, tel que pour tout morphisme de $\mathcal{A}$, $f\colon (A,\varphi)\to (B,\psi)$, l'ensemble simplicial $A$ ne contient qu'un nombre fini de simplexes non-dégénérés. Alors, $\text{Sat}(\mathcal{A})=l(r(\mathcal{A}))$.
\end{prop}

\begin{proof}
Soit $\mathcal{A}$ un ensemble de morphismes de $\sS_P$ vérifiant l'hypothèse de la proposition \ref{SatureLifting}. Montrons que $l(r(\mathcal{A}))$ est une classe saturée.
Tout isomorphisme de $\sS_P$ appartient à $l(r(\mathcal{A}))$ pour tout ensemble de morphisme $\mathcal{A}$.
Soit $f\colon (A,\varphi_A)\to (B,\varphi_B)$ un morphisme de $l(r(\mathcal{A}))$, et soit $g$ son image directe, obtenue comme la somme amalgamée suivante 
\begin{equation*}
\begin{tikzcd}
(A,\varphi_A)
\arrow{r}
\arrow{d}{f}
&(C,\varphi_C)
\arrow{d}{g}
\\
(B,\varphi_B)
\arrow{r}
&(D,\varphi_D)
\end{tikzcd}
\end{equation*}
Montrons que $g\in l(r(\mathcal{A}))$. Considérons un problème de relèvement pour $g$, avec $h\in r(\mathcal{A})$.
\begin{equation*}
\begin{tikzcd}
(C,\varphi_C)
\arrow[swap]{d}{g}
\arrow{r}{\alpha}
&(X,\varphi_X)
\arrow{d}{h}
\\
(D,\varphi_D)
\arrow[swap]{r}{\beta}
\arrow[dashed]{ur}
&(Y,\varphi_Y)
\end{tikzcd}
\end{equation*}

En reprenant la définition de $g$, on obtient le diagramme suivant

\begin{equation*}
\begin{tikzcd}
(A,\varphi_A)
\arrow{r}
\arrow{d}{f}
&(C,\varphi_C)
\arrow{r}{\alpha}
&(X,\varphi_X)
\arrow{d}{h}
\\
(B,\varphi_B)
\arrow[swap]{r}
\arrow[dashed,near start]{urr}{l}
&(D,\varphi_D)
\arrow[crossing over, leftarrow, near end]{u}{g}
\arrow[swap]{r}{\beta}
\arrow[dashed,swap]{ur}{l'}
&(Y,\varphi_Y)
\end{tikzcd}
\end{equation*}
Le carré extérieur de ce diagramme correspond à un problème de relèvement pour $f$. Comme $f\in l(r(\mathcal{A}))$ il existe un relèvement $l$ faisant commuter le diagramme. Mais, par la propriété universelle de la somme amalgamée, $\alpha$ et $l$ induisent un unique morphisme $l'\colon (D,\varphi_D)\to (X,\varphi_X)$. On en déduit que $g\in l(r(\mathcal{A}))$ et donc que $l(r(\mathcal{A}))$ est stable par image directe. On montre de la même façon que la classe $l(r(\mathcal{A}))$ est stable par somme directe et par composition transfinie en utilisant les propriétés universelles qui les définissent. Montrons que $l(r(\mathcal{A}))$ est stable par rétracte. Soit $f\in l(r(\mathcal{A}))$, et $f'$ un rétracte de $f$. 
Considérons le problème de relèvement suivant pour $f'$, où $g\in r(A)$
\begin{equation*}
\begin{tikzcd}
\fil{A'}
\arrow{d}{f'}
\arrow{r}{\alpha}
&\fil{X}
\arrow{d}{g}
\\
\fil{B'}
\arrow[dashed]{ur}
\arrow{r}{\beta}
&\fil{Y}
\end{tikzcd}
\end{equation*}
En utilisant le fait que $f'$ est un rétracte de $f$, on obtient le diagramme suivant :
\begin{equation*}
\begin{tikzcd}
\fil{A'}
\arrow{r}{i}
\arrow{d}{f'}
&\fil{A}
\arrow{r}{r}
\arrow{d}{f}
&\fil{A'}
\arrow{r}{\alpha}
&\fil{X}
\arrow{d}{g}
\\
\fil{B'}
\arrow[swap]{r}{j}
&\fil{B}
\arrow[swap]{r}{s}
\arrow[dashed, near start]{urr}{l}
&\fil{B'}
\arrow[swap]{r}{\beta}
\arrow[leftarrow, crossing over, near end]{u}{f'}
\arrow[dashed,swap]{ur}{l'}
&\fil{Y}
\end{tikzcd}
\end{equation*}
Les compositions $\alpha\circ r$ et $\beta\circ s$ donnent un problème de relèvement pour $f$. Comme $f\in l(r(\mathcal{A})$ et $g\in r(\mathcal{A})$, il existe un relèvement $l\colon \fil{B}\to\fil{X}$ faisant commuter le diagramme. Puis, on pose $l'=l\circ j$. Comme $s\circ j=\Id_{B'}$, on a bien $g\circ l'=g\circ l\circ j=\beta\circ s\circ j=\beta$. De même, $l'\circ f'=l\circ j\circ f'=l\circ f\circ i=\alpha\circ r\circ i=\alpha$. Finalement, nous avons montré que $l(r(\mathcal{A}))$ est une classe saturée. En particulier, on a $\text{Sat}(\mathcal{A})\subseteq l(r(\mathcal{A}))$. 

Prouvons l'autre inclusion, soit $f\colon \fil{X}\to\fil{Y}$ un morphisme de $l(r(\mathcal{A}))$. Montrons, par l'argument du petit objet, que $f$ peut s'écrire comme le rétracte d'une composition d'image directe de sommes directes de morphismes de $\mathcal{A}$. Considérons l'ensemble $S_1$ de tous les triplets $s=(a_s,\alpha_s,\beta_s)$ où $a_s\colon \fil{A_s}\to \fil{B_s}$ est un morphisme de $\mathcal{A}$, et où les trois morphismes font partie d'un diagramme commutatif de la forme :
\begin{equation*}
\begin{tikzcd}
\fil{A_s}
\arrow{r}{\alpha_s}
\arrow[swap]{d}{a_s}
&\fil{X}
\arrow{d}{f}
\\
\fil{B_s}
\arrow[swap]{r}{\beta_s}
&\fil{Y}
\end{tikzcd}
\end{equation*}
On définit l'ensemble simplicial filtré $\fil{Z_1}$ comme la somme amalgamée suivante :
\begin{equation*}
\begin{tikzcd}
\coprod_{s\in S_1}\fil{A_s}
\arrow{r}{\coprod\alpha_s}
\arrow[swap]{d}{\coprod a_s}
&\fil{X}
\arrow{d}{j_1}
\\
\coprod_{s\in S_1}\fil{B_s}
\arrow{r}
&\fil{Z_1}
\end{tikzcd}
\end{equation*}

Par la propriété universelle de la somme amalgamée, on obtient le morphisme $f_1$.
\begin{equation*}
\begin{tikzcd}
\coprod_{s\in S_1}\fil{A_s}
\arrow{r}{\coprod\alpha_s}
\arrow[swap]{d}{\coprod a_s}
&\fil{X}
\arrow{d}{j_1}
\arrow[bend left=12]{ddr}{f}
&\phantom{\fil{Y}}
\\
\coprod_{s\in S_1}\fil{B_s}
\arrow{r}
\arrow[bend right=6,swap]{drr}{\coprod \beta_s}
&\fil{Z_1}
\arrow{dr}{f_1}
&\phantom{\fil{Y}}
\\
\phantom{X}
&\phantom{X}
&\fil{Y}
\end{tikzcd}
\end{equation*}
On a donc obtenu une factorisation $f=f_1\circ j_1$. De plus, on constate que $j_1$ est dans $\text{Sat}(\mathcal{A})$ car elle est obtenue comme l'image directe de la somme directe $\coprod a_s$ de morphisme de $\mathcal{A}$. On itère ensuite la construction comme suit. On suppose construit $i_n=j_n\circ\dots\circ j_1\colon\fil{X}\to \fil{Z_1}\to\dots\to Z_n$, et $f_n\colon \fil{Z_n}\to\fil{Y_n}$ tels que $f_n\circ j_n= f$. Et on définit l'ensemble $S_n$ comme l'ensemble des triplets $s=(a_s,\alpha_s,\beta_s)$ tel $\alpha_s\colon \fil{A_s}\to\fil{B_s}$ est dans $\mathcal{A}$ et les trois morphismes font partie d'un diagramme commutatif de la forme
\begin{equation*}
\begin{tikzcd}
\fil{A_s}
\arrow{r}{\alpha_s}
\arrow[swap]{d}{a_s}
&\fil{Z_n}
\arrow{d}{f_n}
\\
\fil{B_s}
\arrow[swap]{r}{\beta_s}
&\fil{Y}
\end{tikzcd}
\end{equation*}
De la même façon que précédemment, on construit $Z_{n+1}$ comme une somme amalgamée, et on obtient les applications $j_{n+1}$ et $f_{n+1}$ :
\begin{equation*}
\begin{tikzcd}
\coprod_{s\in S_n}\fil{A_s}
\arrow{r}{\coprod\alpha_s}
\arrow[swap]{d}{\coprod a_s}
&\fil{Z_n}
\arrow{d}{j_{n+1}}
\arrow[bend left=12]{ddr}{f}
&\phantom{\fil{Y}}
\\
\coprod_{s\in S_n}\fil{B_s}
\arrow{r}
\arrow[bend right=6,swap]{drr}{\coprod \beta_s}
&\fil{Z_{n+1}}
\arrow{dr}{f_{n+1}}
&\phantom{\fil{Y}}
\\
\phantom{X}
&\phantom{X}
&\fil{Y}
\end{tikzcd}
\end{equation*}
Comme précédemment, $j_{n+1}$ est dans $\text{Sat}(\mathcal{A})$. En particulier, on en déduit que $i_{n+1}=j_{n+1}\circ\dots\circ j_1$ est dans $\text{Sat}(\mathcal{A})$ comme composition de morphismes de $\text{Sat}(\mathcal{A})$. On passe ensuite à la limite inductive, et on définit 
\begin{equation*}
\fil{Z_{\infty}}=\lim_{\to} \fil{Z_n}
\end{equation*}
On obtient une factorisation $f=f_\infty\circ i_\infty$, où $i_\infty$ est la composition dénombrable des $j_n$. En particulier, $i_\infty$ est dans $\text{Sat}(\mathcal{A}))$. Nous allons montrer que $f_{\infty}$ est dans $r(\mathcal{A})$.
Considérons un problème de relèvement pour $f_{\infty}$, où $a\colon \fil{A}\to \fil{B}$ est dans $\mathcal{A}$.
\begin{equation*}
\begin{tikzcd}
\fil{A}
\arrow[swap]{d}{a}
\arrow{r}{\alpha}
&\fil{Z_{\infty}}
\arrow{d}{f_{\infty}}
\\
\fil{B}
\arrow[swap]{r}{\beta}
&\fil{Y}
\end{tikzcd}
\end{equation*}
Comme $a$ est un morphisme de $\mathcal{A}$, $A$ ne doit avoir par hypothèse qu'un nombre fini de simplexes non-dégénérés. En particulier, toute application simpliciale $A\to \lim\limits_{\to} Z_n$ doit se factoriser comme $A\to Z_m\to \lim\limits_{\to}Z_n$ pour un certain $m$. Finalement, on est dans la situation suivante :
\begin{equation*}
\begin{tikzcd}
\fil{A}
\arrow{rr}{\alpha}
\arrow{dr}{\alpha_m}
\arrow[swap]{ddd}{a}
&\phantom{X}
&\fil{Z_{\infty}}
\arrow{ddd}{f_{\infty}}
\\
\phantom{X}
&\fil{Z_m}
\arrow[swap]{d}{j_m}
\arrow{ur}{\iota_m}
&\phantom{X}
\\
\phantom{X}
&\fil{Z_{m+1}}
\arrow{dr}{f_{m+1}}
\arrow[swap]{uur}{\iota_{m+1}}
&\phantom{X}
\\
\fil{B}
\arrow[dashed]{ur}{l}
\arrow{rr}{\beta}
&\phantom{X}
&\fil{Y}
\end{tikzcd}
\end{equation*}
Le triplet $(a,\alpha_m, \beta)$ correspond à un problème de relèvement pour $f_m=f_{m+1}\circ j_m\colon \fil{Z_m}\to\fil{Y}$, et est dans $S_m$. En particulier,  par définition de $Z_{m+1}$, il existe une application $l$ faisant commuter le diagramme. On en déduit que la composition $\iota_{m+1}\circ l\colon \fil{B}\to \fil{Z_{\infty}}$ fournit le relèvement souhaité. On a montré que $f_{\infty}$ est dans $r(\mathcal{A})$.
Considérons maintenant le problème de relèvement suivant.
\begin{equation*}
\begin{tikzcd}
\fil{X}
\arrow{r}{j_{\infty}}
\arrow{d}{f}
&\fil{Z_{\infty}}
\arrow{d}{f_{\infty}}
\\
\fil{Y}
\arrow[dashed]{ur}{l}
\arrow{r}{\Id_Y}
&\fil{Y}
\end{tikzcd}
\end{equation*}
Par hypothèse, $f$ est dans $l(r(\mathcal{A}))$, et on a montré que $f_{\infty}$ était dans $r(\mathcal{A})$. En particulier, il existe un relèvement $l$ faisant commuter le diagramme. On peut donc écrire $f$ comme un rétracte de $j_{\infty}$
\begin{equation*}
\begin{tikzcd}
\fil{X}
\arrow{r}{\Id_X}
\arrow{d}{f}
&\fil{X}
\arrow{r}{\Id_X}
\arrow{d}{j_{\infty}}
&\fil{X}
\arrow{d}{f}
\\
\fil{Y}
\arrow{r}{l}
&\fil{Z_{\infty}}
\arrow{r}{f_{\infty}}
&\fil{Y}
\end{tikzcd}
\end{equation*}
Donc, $f$ est le rétracte d'un morphisme de $\text{Sat}(\mathcal{A})$, et donc $f$ est dans $\text{Sat}(\mathcal{A})$. Finalement, $l(r(A))\subseteq \text{Sat}(\mathcal{A})$, et donc $l(r(A))=\text{Sat}(\mathcal{A})$.
\end{proof}

\subsection{Extensions anodines}

\begin{defin}\label{DefinitionExtensionAnodine}
Soit $\Lambda$ une classe de (mono-)morphismes dans $\sS_P$. On dit que $\Lambda$ est une classe d'extensions anodines si les conditions suivantes sont satisfaites :
\begin{itemize}
\item (An0) Il existe un ensemble $\mathcal{A}$ de morphismes dans $\sS_P$ tel que $\Lambda= l(r(\mathcal{A}))$.
\item (An1) Si $(X,\varphi)\hookrightarrow (Y,\psi)$ est une inclusion filtrée et $\epsilon=0$ ou $1$ est un sommet de $\Delta^1$, alors les applications 
\begin{equation*}
\Delta^1\otimes (X,\varphi)\cup \{\epsilon\}\otimes (Y,\varphi)\to \Delta^1\otimes (Y,\varphi)
\end{equation*}
sont dans $\Lambda$.
\item (An2) Si $(X,\varphi)\hookrightarrow (Y,\psi)$ est dans $\Lambda$, alors
\begin{equation*}
\Delta^1\otimes (X,\varphi)\cup \partial(\Delta^1)\otimes (Y,\varphi)\to \Delta^1\otimes (Y,\varphi)
\end{equation*}
est dans $\Lambda$.
\end{itemize}
\end{defin}

\begin{prop}\label{DefExtensionsAnodinesSSP}
Soit $\A$ l'ensemble des inclusions de cornets admissibles $\Lambda^{\varphi}_k\hookrightarrow \Delta^{\varphi}$ dans $\sS_P$. La classe $\Lambda=l(r(\A))$ est une classe d'extensions anodines.
\end{prop}

\begin{proof}
La définition de $\Lambda$ assure que la condition (An0) est vérifiée. La condition (An1) découle du lemme \ref{SatureeABC} et la condition (An2) est un cas particulier du lemme \ref{LemmeAn2}, avec $\fil{Z}\to\fil{W} =F(\partial(\Delta^1))\to F(\Delta^1)$.
\end{proof}

\begin{lemme}\label{SatureeABC}
Soit $\B$ l'ensemble des morphismes de $\sS_P$ de la forme 
\begin{equation*}
\Delta^1\otimes \partial(\Delta^{\varphi})\cup\{\epsilon\}\otimes \Delta^{\varphi}\to \Delta^1\otimes \Delta^{\varphi}
\end{equation*}
où $\Delta^{\varphi}=(\Delta^N,\varphi)$ est un simplexe filtré, et $\epsilon= 0$ ou $1$ est un sommet de $\Delta^1$.
Soit $\C$ la classe des morphismes de $\sS_P$ de la forme
\begin{equation*}
\Delta^1\otimes\fil{X}\cup\{\epsilon\}\otimes\fil{Y}\to\Delta^1\otimes\fil{Y}
\end{equation*}
où $\fil{X}\to\fil{Y}$ est une inclusion d'ensembles simpliciaux filtrés, et $\epsilon=0$ ou $1$ est un sommet de $\Delta^1$.
On a $\Sat(\A)=\Sat(\B)=\Sat(\C)$.
\end{lemme}

\begin{proof}
($\B\subset \Sat(\A)$) Soit $\Delta^{\varphi}=(\Delta^N,\varphi)$ un simplexe filtré. On notera $\Delta^N=[e_0,\dots,e_N]$ où les $e_i$ sont les sommets de $\Delta^N$ pour $0\leq i\leq N$. On se place dans le cas $\epsilon=1$, le cas $\epsilon=0$ étant similaire. On considère le morphisme de $\B$
\begin{equation*}
f\colon\Delta^1\otimes \partial(\Delta^{\varphi})\cup\{1\}\otimes \Delta^{\varphi}\to \Delta^1\otimes \Delta^{\varphi}.
\end{equation*}
On a $\Delta^1\otimes\Delta^{\varphi}=(\Delta^1\times\Delta^N,\varphi\circ\pr_{\Delta^N})$. On décompose $\Delta^1\times\Delta^N$ comme suit :
Soit $0\leq i\leq N$ un entier, on note $C_i$ le $n+1$-simplexe non-dégénéré de $\Delta^1\times\Delta^N$ défini comme :
\begin{equation*}
C_i=[(0,e_0),\dots,(0,e_i),(1,e_i),\dots,(1,e_N)]
\end{equation*}
Et on note $D_{-1}=\Delta^1\times\partial(\Delta^N)\cup\{1\}\times\Delta^N$, et $D_i=D_{i-1}\cup C_i$. On remarque que $D_N=\Delta^1\times\Delta^N$. En particulier, $f$ se factorise comme 
\begin{equation*}
\Delta^1\otimes \partial(\Delta^{\varphi})\cup\{\epsilon\}\otimes \Delta^{\varphi}\to (D_{0},(\varphi\circ\pr_{\Delta^N})_{|D_0})\to\dots\to (D_{N},\varphi\circ pr_{\Delta^N})=\Delta^1\otimes \Delta^{\varphi}
\end{equation*}
Il suffit donc de montrer que pour tout $0\leq i\leq N$, l'inclusion $(D_{i-1},\varphi\circ\pr_{\Delta^N})\to(D_{i},\varphi\circ\pr_{\Delta^N})$ est dans $\Sat(\A)$. Considérons le carré cocartésien suivant.
\begin{equation*}
\begin{tikzcd}
C_i\cap D_{i-1}
\arrow{r}
\arrow{d}
&D_{i-1}
\arrow{d}
\\
C_i
\arrow{r}
& D_i
\end{tikzcd}
\end{equation*}
On va montrer que l'inclusion $C_i\cap D_{i-1}\to C_i$, munie de la filtration induite, est dans $\Sat(\A)$, ce qui impliquera le résultat. On sait que $C_i$ est le $N+1$ simplexe $[(0,e_0),\dots,(0,e_i),(1,e_i),\dots,(1,e_N)]$. De plus, pour $0\leq l<i$ et pour $i<m\leq N$, les $N$ simplexes de la forme
\begin{equation*}
\begin{array}{c}
[(0,e_0),\dots,\widehat{(0,e_l)},\dots,(0,e_i),(1,e_i),\dots,(1,e_N)]\\
\text{et} \\{}
[(0,e_0),\dots,(0,e_i),(1,e_i),\dots,\widehat{(1,e_m)},\dots,(1,e_N)]
\end{array}
\end{equation*}
appartiennent à $\Delta^1\times \partial(\Delta^N)\subseteq D_{-1}\subseteq D_{i-1}$.
Ensuite, on a :
\begin{align*}
[(0,e_0),\dots,\widehat{(0,e_i)},(1,e_i),\dots,(1,e_N)]&\subseteq [(0,e_0),\dots,(0,e_{i-1}),(1,e_{i-1}),(1,e_i),\dots,(1,e_N)]\\
&=C_{i-1}\\
&\subseteq D_{i-1}.
\end{align*}
Et, on observe que 
\begin{equation*}
[(0,e_0),\dots,(0,e_i),\widehat{(1,e_i)},\dots,(1,e_N)]\not \subset D_{i-1}
\end{equation*}
On en déduit que l'inclusion $C_i\cap D_{i-1}\to C_i$ est de la forme $\Lambda^{N+1}_i\to \Delta^{N+1}$. On note $\psi=(\varphi\circ\pr_{\Delta^N})_{|C_i}\colon C_i\simeq\Delta^{N+1}\to N(P)$, on calcule $\psi(0,e_i)=\varphi(e_i)=\psi(1,e_i)$, et donc l'inclusion du cornet $\Lambda^{\psi}_i\to \Delta^\psi$ est admissible. En particulier, $(C_i\cap D_{i-1},\psi_{|\C_i\cap D_{i-1}})\to (C_i,\psi)$ est dans $\Sat(\A)$. Finalement, on a bien $f\in \Sat(\A)$ et donc $\B\subset \Sat(\A)$.

($\C\subset\Sat(\B)$) Soit $\fil{X}\to\fil{Y}$ une inclusion d'ensembles simpliciaux filtrés. On note $\sk^N(Y)$ le $N$-squelette de $Y$, c'est à dire le sous-ensemble simplicial de $Y$ généré par les simplexes de dimensions inférieures ou égales à $N$, avec la convention $\sk^{-1}(Y)=\emptyset$. De même, on note $\sk^N\fil{Y}=(\sk^N(Y),(\varphi_Y)_{|\sk^N(Y)})$. Alors, on a
\begin{align*}
\Delta^1\otimes\fil{X}\cup\{\epsilon\}\otimes\fil{Y}
\simeq \Delta^1\otimes\left(\fil{X}\cup\sk^{-1\phantom{|}}\fil{Y}\right)\cup\{\epsilon\}\otimes\fil{Y}\\
\text{et}\\
\lim_{\to}\left(\Delta^1\otimes\left(\fil{X}\cup\sk^{N}\fil{Y}\right)\cup\{\epsilon\}\otimes\fil{Y}\right)\simeq \Delta^1\otimes \fil{Y}
\end{align*}
Et $f$ est l'inclusion canonique du terme $N=-1$ dans la limite inductive. En particulier, il suffit de montrer pour tout $N\geq -1$ que le morphisme 
\begin{equation*}
\Delta^1\otimes\left(\fil{X}\cup\sk^{N}\fil{Y}\right)\cup\{\epsilon\}\otimes\fil{Y}\to \Delta^1\otimes\left(\fil{X}\cup\sk^{N+1}\fil{Y}\right)\cup\{\epsilon\}\otimes\fil{Y}
\end{equation*}
est dans $\Sat(\B)$ pour obtenir que $f$ est dans $\Sat(\B)$. Notons $\Sigma^N$ l'ensemble des $N$ simplexes non dégénérés de $Y$ qui ne sont pas contenus dans $X$, et considérons le diagramme commutatif suivant.
\begin{equation}\label{CarreCocartesienAnodines2}
\begin{tikzcd}
\coprod\limits_{\sigma\in \Sigma^{N+1}}\left(\Delta^1\otimes\partial(\Delta^{\varphi_Y\circ \sigma})\cup \{\epsilon\}\otimes\Delta^{\varphi_Y\circ \sigma}\right)
\arrow{r}{\coprod \Id\otimes\sigma}
\arrow{d}
&\Delta^1\otimes\left(\fil{X}\cup\sk^{N}\fil{Y}\right)\cup\{\epsilon\}\otimes\fil{Y}
\arrow{d}{f_N}
\\
\coprod\limits_{\sigma\in \Sigma^{N+1}}\left(\Delta^1\otimes\Delta^{\varphi_Y\circ \sigma}\right)
\arrow{r}{\coprod \Id\otimes\sigma}
&\Delta^1\otimes\left(\fil{X}\cup\sk^{N+1}\fil{Y}\right)\cup\{\epsilon\}\otimes\fil{Y}
\end{tikzcd}
\end{equation}
On doit montrer que $f_N$ est dans $\Sat(\B)$, et on sait que le morphisme de gauche est dans $\Sat(\B)$. Il suffit donc de montrer que ce carré est cocartésien. Pour ça, on utilise le fait que le carré suivant est cocartésien, par définition du squelette :
\begin{equation}\label{CarreCocartesienAnodines}
\begin{tikzcd}
\coprod\limits_{\sigma\in \Sigma^{N+1}}\partial(\Delta^{\varphi_Y\circ \sigma})
\arrow{r}{\sigma}
\arrow{d}
&\fil{X}\cup\sk^{N}\fil{Y}
\arrow{d}
\\
\coprod\limits_{\sigma\in \Sigma^{N+1}}\Delta^{\varphi_Y\circ \sigma}
\arrow{r}{\coprod \sigma}
&\fil{X}\cup\sk^{N+1}\fil{Y}
\end{tikzcd}
\end{equation}
Puis, on considère le diagramme commutatif suivant. Pour alléger les notations, on omet les filtrations et on note $Z^N=X\cup\sk^N(Y)$, et $\Delta=\Delta^{\varphi_Y\circ\sigma}$.
\begin{equation*}
\begin{tikzcd}[column sep=30]
\coprod \Delta^1\otimes\partial(\Delta)
\arrow{r}
\arrow{d}
&\Delta^1\otimes Z^N
\arrow{d}
&\phantom{X}
\\
\coprod\Delta^1\otimes \partial(\Delta)\cup\{\epsilon\}\otimes\Delta
\arrow{r}
\arrow{d}
&\Delta^1\otimes Z^N\cup\{\epsilon\}\otimes Z^{N+1}
\arrow{d}
\arrow{r}
&\Delta^1\otimes Z^N\cup \{\epsilon\}\otimes Y
\arrow{d}
\\
\coprod\Delta^1\otimes \Delta
\arrow{r}
&\Delta^1\otimes Z^{N+1}
\arrow{r}
&\Delta^1\otimes Z^{N+1}\cup\{\epsilon\}\otimes Y
\end{tikzcd}
\end{equation*}
Dans ce diagramme, le carré de droite est cocartésien. De plus, le rectangle de gauche est cocartésien, car c'est l'image du carré \ref{CarreCocartesienAnodines} par le foncteur $\Delta^1\otimes -$ qui préserve les colimites. Le rectangle du bas est le carré \ref{CarreCocartesienAnodines2}, dont on veut montrer qu'il est cocartésien. Finalement, il suffit de montrer que le carré en haut à gauche est cocartésien pour obtenir le résultat voulu. Pour ce faire, on considère d'abord le diagramme suivant :
\begin{equation*}
\begin{tikzcd}
\coprod\{\epsilon\}\otimes\partial(\Delta^{\varphi_Y\circ\sigma})
\arrow{r}
\arrow{d}
&\{\epsilon\}\otimes\fil{Z^N}
\arrow{d}
\arrow{r}
&\Delta^1\otimes\fil{Z^N}
\arrow{d}
\\
\coprod\{\epsilon\}\otimes\Delta^{\varphi_Y\circ\sigma}
\arrow{r}
&\{\epsilon\}\otimes\fil{Z^{N+1}}
\arrow{r}
&\Delta^1\otimes\fil{Z^N}\cup \{\epsilon\}\otimes\fil{Z^{N+1}}
\end{tikzcd}
\end{equation*}
Le carré de gauche est cocartésien car c'est l'image du diagramme \ref{CarreCocartesienAnodines} par le foncteur $\{\epsilon\}\otimes-$, qui préserve les colimites. Le carré de droite est cocartésien par construction, et on en déduit que le rectangle est cocartésien.
Puis, on considère le diagramme 
\begin{equation*}
\begin{tikzcd}[column sep= 5pt]
\coprod\{\epsilon\}\otimes\partial(\Delta^{\varphi_Y\circ\sigma})
\arrow{r}
\arrow{d}
&\coprod\Delta^1\otimes\partial(\Delta^{\varphi_Y\circ\sigma})
\arrow{r}
\arrow{d}
&\Delta^1\otimes\fil{Z^N}
\arrow{d}
\\
\coprod\{\epsilon\}\otimes\Delta^{\varphi_Y\circ\sigma}
\arrow{r}
&\coprod\Delta^1\otimes\partial(\Delta^{\varphi_Y\circ\sigma})\cup\{\epsilon\}\otimes\Delta^{\varphi_Y\circ\sigma}
\arrow{r}
&\Delta^1\otimes\fil{Z^N}\cup \{\epsilon\}\otimes\fil{Z^{N+1}}
\end{tikzcd}
\end{equation*}
On vient de prouver que le rectangle extérieur était cocartésien et le carré de gauche est cocartésien par construction. On en déduit que le carré de droite est cocartésien, ce qui complète la preuve.

($\A\subset\Sat(\C)$). Soit $i\colon\Lambda^{\varphi}_k\to \Delta^\varphi$ une inclusion de cornet admissible. Comme précédemment, $\Delta^{\varphi}=(\Delta^N,\varphi)$ et on note $\Delta^N=[e_0,\dots,e_N]$.  Nous allons prouver que $i$ peut s'obtenir comme le rétracte d'une application de $\mathcal{C}$.
\begin{equation*}
\begin{tikzcd}
\Lambda^{\varphi}_k
\arrow{r}{j}
\arrow{d}{i}
&\Delta^1\otimes\Lambda^{\varphi}_k\cup\{\epsilon\}\otimes\Delta^{\varphi}
\arrow{r}{r}
\arrow{d}{i'}
&\Lambda^{\varphi}_k
\arrow{d}{i}
\\
\Delta^{\varphi}
\arrow{r}{j}
&\Delta^1\otimes\Delta^{\varphi}
\arrow{r}{r}
&\Delta^{\varphi}
\end{tikzcd}
\end{equation*}
Ici, le morphisme $i'$ est le morphisme d'inclusion induit par $i$, et donc est dans $\C$. En particulier, il suffit d'exhiber $j\colon \Delta^{\varphi}\to\Delta^1\otimes\Delta^{\varphi}$ et $r\colon\Delta^1\otimes\Delta^{\varphi}\to \Delta^{\varphi}$ tels que $r\circ j=\Id_{\Delta^{\varphi}}$, $j(\Lambda^{\varphi}_k\subseteq \Delta^1\otimes\Lambda^{\varphi}_k\cup\{\epsilon\}\otimes\Delta^{\varphi}$, et $r(\Delta^1\otimes\Lambda^{\varphi}_k\cup\{\epsilon\}\otimes\Delta^{\varphi})\subset \Lambda^{\varphi}_k$.
Comme $\Lambda^{\varphi}_k$ est un cornet admissible, par hypothèse, on a $\varphi(e_{k})=\varphi(e_{k-1})$ ou $\varphi(e_{k})=\varphi(e_{k+1})$. On se place dans le second cas, ce qui impose de choisir $\epsilon=0$ (Dans le premier cas, il suffit d'échanger les occurrences de $0$ et $1$, et de renverser les relations d'ordres apparaissant dans les définitions suivantes).
On prendra pour $j$ l'inclusion $\Delta^{\varphi}\simeq\{1\}\otimes\Delta^{\varphi}\hookrightarrow\Delta^1\otimes\Delta^{\varphi}$.
On définit $r$ sur les sommets comme suit.
\begin{align*}
\Delta^1\otimes\Delta^{\varphi}&\xrightarrow{r}\Delta^{\varphi}\\
(\mu,e_l)&\mapsto\left\{\begin{array}{cl}
e_k & \text{ si $\mu=0$, $\varphi(e_l)=\varphi(e_k)$, et $l\geq k$} \\
e_l & \text{ sinon}
\end{array}\right.
\end{align*}
On vérifie que $r$ induit bien une application simpliciale. Pour cela, on remarque que $r$ est compatible à la relation d'ordre sur les sommets. En effet, si $l\leq l'\in \{0,\dots,N\}$ et $\mu\leq \mu'\in\{0,1\}$, et qu'on note $r(\mu,e_l)=e_m$ et $r(\mu',e_{l'})=e_{m'}$ alors $m\leq m'$. De plus, $r$ est une application filtrée. Si on note $\psi=\varphi\circ\pr_{\Delta^{N}}$, on a $\psi(\mu,e_l)=\varphi(e_l)$ et 
\begin{equation*}
\varphi(r(\mu,e_l))=\left\{\begin{array}{cl}
\varphi(e_k) & \text{ si $\mu=0$, $\varphi(e_l)=\varphi(e_k)$, et $l\geq k$} \\
\varphi(e_l) & \text{ sinon}
\end{array}\right.=\varphi(e_l)=\psi(\mu,e_l)
\end{equation*}
On a immédiatement $r\circ j=\Id_{\Delta^{\varphi}}$ et $j(\Lambda^{\varphi}_k)\subseteq \Delta^1\otimes\Lambda^{\varphi}_k\cup\{\epsilon\}\otimes\Delta^{\varphi})$. Il reste donc à montrer que $r(\Delta^1\otimes\Lambda^{\varphi}_k\cup\{\epsilon\}\otimes\Delta^{\varphi})\subseteq \Lambda^{\varphi}_k$, ce qui revient à montrer que la face $d_k(\Delta^N)$ n'est pas atteinte par la restriction de $j$ à $\Delta^1\otimes\Lambda^{\varphi}_k\cup\{\epsilon\}\otimes\Delta^{\varphi}$. Comme tous les $N-1$ simplexes de $\Delta^1\otimes\Lambda^{\varphi}_k\cup\{\epsilon\}\otimes\Delta^{\varphi}$ sont contenus dans un $N$ simplexes, il suffit de considérer l'image de ces derniers. Soit $\sigma$ un $N$ simplexe de $\Delta^1\times\Lambda^N_k\cup \{0\}\times\Delta^N$. Alors, $\sigma\in \Delta^1\times\Lambda^N_k$ ou $\sigma \in \{0\}\times\Delta^N$. Dans le premier cas, $\sigma$ est d'une des deux formes suivantes :
\begin{align*}
\sigma=[(0,e_0),\dots,\widehat{(0,e_l)},\dots,(0,e_m),(1,e_m),\dots,(1,e_N)]\\
\text{ou}\\
\sigma=[(0,e_0),\dots,(0,e_m),(1,e_m),\dots,\widehat{(1,e_l)},\dots,(1,e_N)]
\end{align*}
avec $l\not=k,m$. En particulier, on a $r(\sigma)\subseteq d_l(\Delta^N)\subseteq \Lambda^N_k$.
Dans le second cas, $\sigma= [(0,e_0),\dots,(0,e_N)]$, et on a $r(\sigma)\subseteq d_{k+1}(\Delta^N)\subseteq \Lambda^N_k$, car $r(0,e_{k+1})=e_k$. Finalement, nous avons construit le rétracte souhaité ce qui complète la preuve du cas $\A\subset\Sat(\C)$, et du lemme \ref{SatureeABC}.
\end{proof}

\begin{remarque}
La preuve du lemme \ref{SatureeABC} fait apparaitre pourquoi la restriction aux inclusions de cornets admissibles est nécessaire. En effet, pour prouver qu'une inclusion de cornet $\Lambda^{\varphi}_k\to\Delta^{\varphi}$ est dans $\Sat(\C)$, il est nécessaire de supposer
qu'au moins un sommet de $\Lambda^N_k$ en plus du $k$-ième a la même image par la filtration que le $k$-ième sommet, ce qui revient à dire que l'inclusion de cornet est admissible. On remarque de plus que la preuve de l'inclusion $\B\subset\Sat(\A)$ est une adaptation directe de la preuve dans le cas non-filtré, et qu'elle ne fait apparaitre que des cornets admissibles. Ce qui donne un argument quant à la pertinence de cette définition. Finalement, on observe que la classe $C$ est incluse dans toute classe d'extensions anodines (relativement au cylindre $\Delta^1\otimes-$), en vertu de l'axiome (An1). En conséquence de cette observation et du lemme \ref{SatureeABC}, on obtient que les inclusions de cornets admissibles sont incluses dans toute classe d'extensions anodines.
\end{remarque}

\begin{lemme}\label{LemmeAn2}
Soit $\fil{X}\to\fil{Y}$ un morphisme de $\Lambda$. et $\fil{Z}\to\fil{W}$ une inclusion d'ensembles simpliciaux filtrés. Alors, le morphisme
\begin{equation*}
\fil{W}\times_{N(P)}\fil{X}\cup\fil{Z}\times_{N(P)}\fil{Y}\to\fil{W}\times_{N(P)}\fil{Y}
\end{equation*}
est dans $\Lambda$.
\end{lemme}

\begin{proof}
Soit $\D$ la classe de morphismes de $\sS_P$ contenant les morphismes $f\colon\fil{X}\to\fil{Y}$ tels que 
\begin{equation*}
\fil{W}\times_{N(P)}\fil{X}\cup\fil{Z}\times_{N(P)}\fil{Y}\to\fil{W}\times_{N(P)}\fil{Y}
\end{equation*}
est dans $\Lambda$ pour toute inclusion d'ensembles simpliciaux filtrés $\fil{Z}\to\fil{W}$. Nous allons montrer que la classe $\D$ est saturée, et qu'elle contient $\C$, ce qui suffit à prouver le résultat souhaité à l'aide du lemme \ref{SatureeABC}. 
Soit $f\colon \fil{X}\to\fil{Y}$ dans $\D$. Considérons $f'$ son image directe :
\begin{equation}\label{DiagrammePreuveAn2}
\begin{tikzcd}
\fil{X}
\arrow{d}{f}
\arrow{r}
&\fil{X'}
\arrow{d}{f'}
\\
\fil{Y}
\arrow{r}
&\fil{Y'}
\end{tikzcd}
\end{equation}

et soit $\fil{Z}\to\fil{W}$ une inclusion d'ensembles simpliciaux filtrés. On a le diagramme commutatif suivant. Pour alléger les notations, on omet les filtrations. Attention cependant, $X\times Y$ désigne le produit filtré $\fil{X}\times_{N(P)}\fil{Y}$, et non le produit des ensembles simpliciaux sous-jacents.

\begin{equation*}
\begin{tikzcd}[column sep=30]
\phantom{X}
&X'\times Z
\arrow{rr}
\arrow{dd}
&\phantom{X}
&Y'\times Z
\arrow{dd}
&\phantom{X}
\\
X\times Z
\arrow[crossing over]{rr}
\arrow{dd}
\arrow{ur}
&\phantom{X}
&Y\times Z
\arrow{ur}
&\phantom{X}
&\phantom{X}
\\
\phantom{X}
&X'\times W
\arrow{rr}
&\phantom{X}
& X'\times W\cup  Y'\times Z
\arrow{dd}{i'}
&\phantom{X}
\\
X\times W
\arrow{ur}
\arrow{rr}
&\phantom{X}
& X\times W\cup Y\times Z
\arrow[leftarrow, crossing over]{uu}
\arrow{ur}
\arrow{dd}{i}\\
&&& Y'\times W
\\
\phantom{X}
&
& Y\times W
\arrow{ur}
&\phantom{X}
\end{tikzcd}
\end{equation*}
Par hypothèse, $i$ est dans $\Lambda$, montrons que $i'$ est dans $\Lambda$. Par construction, la face avant et la face arrière du cube apparaissant dans le diagramme précédent sont cocartésiennes.
De plus, par définition de $f'$, la face supérieure du cube est elle aussi cocartésienne. 
On en déduit que la face inférieure est cocartésienne. 
De plus, la composition de la face inférieure avec le carré contenant $i$ et $i'$ est aussi cocartésienne, car c'est l'image du diagramme \ref{DiagrammePreuveAn2} par le foncteur $-\times\fil{W}$ qui préserve les colimites. Finalement, on en déduit que le carré contenant $i$ et $i'$ est cocartésien, et donc que $i'$ est dans $\Lambda$, ce qui implique que $f'$ est dans $\D$ et donc que $\D$ est stable par image directe. De la même façon, on montre que $\D$ est stable par rétracte, somme directe et composition dénombrable, la classe $\D$ est donc saturée. Montrons que $\C\subset \D$.
Soit $f\colon\fil{X}\to\fil{Y}$ un morphisme de $\C$, et soit $\fil{Z}\to\fil{W}$ une inclusion d'ensembles simpliciaux filtrés. Le morphisme $f$ est de la forme 
\begin{equation*}
\fil{X}=\Delta^1\otimes\fil{A}\cup\{\epsilon\}\otimes\fil{B}\to\Delta^1\otimes\fil{B}=\fil{Y}
\end{equation*}
avec $\fil{A}\to\fil{B}$ une inclusion d'ensemble simpliciaux filtrés.
On calcule 
\begin{align*}
&\fil{W}\times\fil{X}\cup\fil{Z}\times\fil{Y}\\
&=\fil{W}\times(\Delta^1\otimes\fil{A}\cup\{\epsilon\}\otimes\fil{B}) 
\cup \fil{Z}\times(\Delta^1\otimes\fil{B})\\
&=\Delta^1\otimes\left(\fil{W}\times\fil{A}\cup\fil{Z}\times\fil{B}\right)\cup\{\epsilon\}\otimes \left(\fil{W}\times\fil{B}\right)
\end{align*}
On en déduit que l'inclusion 
\begin{equation*}
\fil{W}\times\fil{X}\cup\fil{Z}\times\fil{Y}\to\fil{W}\times\fil{Y}
\end{equation*}
est égale au morphisme de $\C$ correspondant à l'inclusion d'ensembles simpliciaux filtrés
\begin{equation*}
\fil{W}\times\fil{A}\cup\fil{Z}\times\fil{B}\to\fil{W}\times\fil{B}.
\end{equation*}
En particulier, $f$ est dans $\D$, ce qui implique $\C\subset \D$ et donc $\Lambda\subset \D$ par le lemme \ref{SatureeABC}.
\end{proof}

\subsection{Théorème d'existence d'une structure de modèle}

\begin{defin}\label{DefClassesSSP}
On définit les classes suivantes dans la catégorie $\sS_P$.
\begin{itemize}
\item Un morphisme $f\colon (X,\varphi_X)\to(Y,\varphi_Y)$ est une \textbf{cofibration} si $f\colon X\to Y$ est un monomorphisme.
\item Un morphisme est une \textbf{fibration triviale} s'il a la propriété de relèvement à droite par rapport aux cofibrations.
\item Un morphisme est une \textbf{extension anodine} s'il est dans $l(r(\A))$, où $\A$ est l'ensemble des inclusions de cornets admissibles.
\item Un morphisme est une \textbf{fibration naïve} s'il a la propriété de relèvement à droite par rapport aux inclusions de cornets admissibles.
\item Un ensemble simplicial filtré $\fil{X}$ est \textbf{fibrant} si le morphisme $\varphi_X\colon\fil{X}\to(N(P),\Id_{N(P)})$ est une fibration naïve.
\item Un morphisme $f\colon\fil{X}\to\fil{Y}$ est une \textbf{équivalence faible} si pour tout ensemble simplicial filtré fibrant, $\fil{Z}$, l'application entre ensembles de classes d'homotopies 
\begin{equation*}
f^{*}\colon [\fil{Y},\fil{Z}]\to[\fil{X},\fil{Z}]
\end{equation*}
est une bijection. Où $[\fil{Y},\fil{Z}]$ désigne l'ensemble des applications filtrées de $\fil{Y}$ vers $\fil{Z}$ à homotopie filtrée près.
\item Un morphisme est une \textbf{cofibration triviale} si c'est à la fois une cofibration et une équivalence faible.
\item Un morphisme est une \textbf{fibration} s'il a la propriété de relèvement à droite par rapport aux cofibrations triviales.
\end{itemize}
\end{defin}

\begin{theo}\label{ExistenceCMFCisinski}
La catégorie $\sS_P$ munie des classes de cofibrations, fibrations et équivalences faibles de la définition \ref{DefClassesSSP} est une catégorie modèle.
\end{theo}
\begin{proof}
C'est une application immédiate de \cite[Théorème 1.3.22]{Cisinski}
\end{proof}

\begin{remarque}
Le théorème de Cisinski ne garantit pas que la classe des fibrations naïves et celle des fibrations coïncident. Dualement, on ne sait pas a priori si les cofibrations triviales sont toutes des extensions anodines. On a en revanche les inclusions 
\begin{equation*}
\textbf{fibrations}\subseteq \textbf{fibrations naïves} \text{ et } \textbf{extensions anodines}\subseteq \textbf{cofibrations triviales}
\end{equation*}
Cependant, la description des fibrations naïves en terme de propriétés de relèvement par rapport aux inclusions de cornets admissibles est beaucoup plus maniable que la définition des fibrations. Aussi, on verra que dans le cas qui nous intéresse ici, la classe des fibrations et celle des fibrations naïves coïncident. Ce fait est difficile à démontrer mais permet ensuite de ramener un grand nombre de preuves à une étude de problème de relèvement impliquant des cornets admissibles. En particulier, on pourra montrer aisément en utilisant ce fait que la structure de modèle ici définie est simpliciale. 
\end{remarque}

\section{Caractérisation des fibrations}
\label{SectionCaracterisationFibrations}
\subsection{Subdivision filtrée}\label{SectionSubdivisionFiltree}

\begin{defin}
On définit la subdivision filtrée de $(N(P),\Id_{N(P)})$, $(\sd_P(N(P)),\varphi_P)$ comme suit. Les simplexes de $\sd_P(N(P))$ sont donnés par
\begin{equation*}
\sd_P(N(P))_k=\{\left( (\sigma_0,\dots,\sigma_k),(p_0,\dots,p_k)\right)\ |\ \sigma_0\subseteq\dots\subseteq\sigma_k\in N(P), \ p_0\leq \dots\leq p_k\in P, \ \{p_i\}\subseteq \sigma_0\ \forall i\}
\end{equation*}
les applications faces et dégénérescences sont données par :
\begin{align*}
\sd_P(N(P))_k&\xrightarrow{d_i}\sd_P(N(P))_{k-1}\\
\left( (\sigma_0,\dots,\sigma_k),(p_0,\dots,p_k)\right)&\mapsto \left( (\sigma_0,\dots,\widehat{\sigma_i},\dots,\sigma_k),(p_0,\dots,\widehat{p_i},\dots,p_k)\right)\\
\sd_P(N(P))_k&\xrightarrow{s_i}\sd_P(N(P))_{k+1}\\
\left( (\sigma_0,\dots,\sigma_k),(p_0,\dots,p_k)\right)&\mapsto \left( (\sigma_0,\dots,\sigma_i,\sigma_i,\dots,\sigma_k),(p_0,\dots,p_i,p_i,\dots,p_k)\right)\\
\end{align*}
En particulier, on a $\sd_P(N(P))\subseteq \sd(N(P))\times N(P)$. La composition $\sd_P(N(P))\to\sd(N(P))\times N(P)\xrightarrow{\pr_{N(P)}} N(P)$ fournit la filtration $\varphi$.
Soit $\fil{X}$ un ensemble simplicial filtré. On définit sa subdivision filtrée, $\sd_P\fil{X}$ comme suit. Comme ensemble simplicial, c'est le produit fibré
\begin{equation*}
\begin{tikzcd}
\sd_P(X)
\arrow{d}
\arrow{r}
&\sd(X)
\arrow{d}{\sd(\varphi_X)}
\\
\sd_P(N(P))
\arrow{r}
&\sd(N(P))
\end{tikzcd}
\end{equation*}
La filtration est donnée par la composition $\sd_P(X)\to\sd_P(N(P))\xrightarrow{\varphi} N(P)$. Par abus de notation, on note $(\sd_P(X),\sd_P(\varphi_X))$ pour la subdivision filtrée de $\fil{X}$, même si l'ensemble simplicial $\sd_P(X)$ dépend de la filtration $\varphi_X$.
Comme la subdivision $\sd$ est fonctorielle, cette définition induit un foncteur 
\begin{equation*}
\sd_P\colon\sS_P\to\sS_P
\end{equation*}
\end{defin}

\begin{prop}\label{SubdivisionSimplexe}
Soit $\Delta^{\varphi}=(\Delta^N,\varphi)$ un simplexe filtré. L'ensemble simplicial $\sd_P(\Delta^{\varphi})$ peut être décrit comme suit 
\begin{equation*}
\sd_P(\Delta^{\varphi})_k\simeq\{[(\sigma_0,p_0),\dots,(\sigma_k,p_k)]\ |\ \sigma_0\subseteq\dots\subseteq \sigma_k\in (\Delta^{N})_{\nd},p_0\leq\dots\leq p_k\in P, \{p_i\}\subseteq \varphi(\sigma_0)\ \forall i\},
\end{equation*}
où $(\Delta^{N})_{\nd}$ désigne l'ensemble des simplexes non-dégénérés de $\Delta^N$. Et la filtration $\psi\colon\sd_P(\Delta^{\varphi})\to N(P)$ est donné par $\psi[(\sigma_0,p_0),\dots,(\sigma_n,p_n)]=[p_0,\dots,p_n]$.
De plus, si $\fil{X}$ est un ensemble simplicial filtré, sa subdivision filtrée est isomorphe à la colimite suivante : 
\begin{equation*}
\sd_P(X)\simeq \colim\limits_{\sigma\in\Hom(\Delta^{\varphi},\fil{X})} \sd_P(\Delta^{\varphi})
\end{equation*}
Et la filtration $\sd_P(X)\to N(P)$ est obtenue par colimite des filtrations sur les $\sd_P(\Delta^{\varphi})$.
\end{prop}

\begin{proof}
Soit $\Delta^{\varphi}=(\Delta^N,\varphi)$ un simplexe filtré. Calculons $\sd_P(\Delta^{\varphi})$. Par définition la subdivision est obtenue comme le produit fibré 
\begin{equation*}
\begin{tikzcd}
\sd_P(\Delta^{\varphi})
\arrow{r}
\arrow{d}
&\sd(\Delta^N)
\arrow{d}{\sd(\varphi)}
\\
\sd_P(N(P))
\arrow{r}
&\sd(N(P))
\end{tikzcd}
\end{equation*}
On note temporairement $S(\Delta^{\varphi})$ l'ensemble simplicial filtré défini dans l'énoncé de la proposition \ref{SubdivisionSimplexe}. On a deux morphismes 
\begin{align*}
S(\Delta^{\varphi})&\to\sd(\Delta^N)\\ {} 
{[}(\sigma_0,p_0),\dots,(\sigma_n,p_n)]&\mapsto [\sigma_0,\dots,\sigma_n]\\
S(\Delta^{\varphi})&\to\sd_P(N(P))\\ {} 
{[}(\sigma_0,p_0),\dots,(\sigma_n,p_n)]&\mapsto [(\varphi(\sigma_0),p_0),\dots,(\varphi(\sigma_n),p_n)]
\end{align*}
De plus, on vérifie que les compositions avec $\sd(\varphi)$ et $\pr_{\sd(N(P))}$ respectivement donnent le même morphisme
\begin{align*}
S(\Delta^{\varphi})&\to\sd(N(P))\\ {} 
[(\sigma_0,p_0),\dots,(\sigma_n,p_n)]&\mapsto [\varphi(\sigma_0),\dots,\varphi(\sigma_n)]
\end{align*}
en particulier, par la propriété universelle du produit fibré, il existe un morphisme qui factorise les morphismes précédents.
\begin{align*}
S(\Delta^{\varphi})&\xrightarrow{f}\sd_P(\Delta^{\varphi})\\
[(\sigma_0,p_0),\dots,(\sigma_n,p_n)]&\mapsto (((\varphi(\sigma_0),\dots,\varphi(\sigma_n)),(p_0,\dots,p_n),(\sigma_0,\dots,\sigma_n))
\end{align*}
Montrons que $f$ est une bijection. Soient $\sigma=[\sigma_0,\dots,\sigma_n]$ un simplexe de $\sd(\Delta^N)$, et 
\begin{equation*}
\tau=(\tau_0,\dots,\tau_n),(p_0,\dots,p_n)
\end{equation*}
un simplexe de $\sd_P(N(P))$ tels que $\sd(\varphi)(\sigma)=\pr_{\sd(N(P))}(\tau)$. Alors, on a $\tau_i=\varphi(\sigma_i)$ pour tout $i$, et $\tau$ est de la forme $((\varphi(\sigma_0),\dots,\varphi(\sigma_n)),(p_0,\dots,p_n))$. On en déduit que $(\tau,\sigma)$ est dans l'image de $f$, et donc que $f$ est surjective.
De plus, on constate $f([(\sigma_0,p_0),\dots,(\sigma_n,p_n)])=f'([(\sigma'_0,p'_0),\dots,(\sigma'_n,p'_n)])$ implique $\sigma_i=\sigma'_i$ et $p_i=p'_i$ pour tout $i$, et donc que $f$ est injective. On en déduit que $f\colon S(\Delta^{\varphi})\xrightarrow{\simeq}\sd_P(\Delta^{\varphi})$ est un isomorphisme. On vérifie ensuite que la filtration sur $\sd_P(\Delta^{\varphi})$ correspond bien à la définition donnée dans la proposition \ref{SubdivisionSimplexe}, ce qui prouve le premier point.
Soit $\fil{X}$ un ensemble simplicial filtré. On rappelle que 
\begin{equation*}
\fil{X}\simeq \colim\limits_{\sigma\in \Hom(\Delta^{\varphi},\fil{X})}\Delta^{\varphi}
\end{equation*}
De plus, on sait sur la subdivision non filtrée $\sd$, que 
\begin{equation*}
\sd(X)\simeq\sd\left(\colim\limits_{\sigma\in\Hom(\Delta^N,X)}\Delta^N\right)\simeq\colim\limits_{\sigma\in\Hom(\Delta^N,X)}\sd(\Delta^N)
\end{equation*}
En particulier, $\sd_P(X)$ est définie comme le produit fibré 
\begin{equation*}
\begin{tikzcd}
\sd_P(X)
\arrow{d}
\arrow{r}
&\colim\limits_{\sigma\in\Hom(\Delta^N,X)}\sd(\Delta^N)
\arrow{d}{\sd(\varphi_X)}
\\
\sd_P(N(P))
\arrow{r}
&\sd(N(P))
\end{tikzcd}
\end{equation*}
On en déduit le résultat voulu.
\end{proof}

\begin{corollaire}\label{SimplexesSubdivisionEnsembleSimplicialFiltre}
Soit $\fil{X}$ un ensemble simplicial filtré. Tout simplexe $\sigma\colon \Delta^{\varphi}\to\sd_P\fil{X}$ de la subdivision de $\fil{X}$, peut se factoriser sous la forme
\begin{equation*}
\Delta^{\varphi}\xrightarrow{\widehat{\sigma}}\sd_P(\Delta^{\phi})\xrightarrow{\sd_P(\widetilde{\sigma})} \sd_P\fil{X}.
\end{equation*}
En particulier, tout simplexe de la subdivision de $\fil{X}$ peut s'écrire sous la forme 
\begin{equation*}
(\widetilde{\sigma},[(\sigma_0,p_0),\dots,(\sigma_n,p_n)])
\end{equation*}
où $\widetilde{\sigma}\colon \Delta^{\psi}\to X$ est un simplexe non dégénéré de $X$, et $[(\sigma_0,p_0),\dots,(\sigma_n,p_n)]$ est un simplexe de $\sd_P(\Delta^{\psi})$. Si on exige de plus que $\sigma_n=\Delta^{\psi}$, cette écriture est unique.
\end{corollaire}

\begin{proof}
L'existence d'une telle factorisation provient de la définition de $\sd_P\fil{X}$ comme une colimite de subdivsion de simplexes filtrés. Pour prouver l'unicité, il suffit de voir qu'il existe une unique face de $\Delta^{\psi}$ qui contient l'image de l'application $\Delta^{\varphi}\to\sd_P(\Delta^{\psi})$ et qui est maximale pour cette propriété. La restriction de la factorisation à cette face donne l'unique factorisation pour laquelle $\sigma_n=\Delta^{\psi}$.
\end{proof}

\begin{defin}\label{DernierSommetFiltre}
Soit $\Delta^{\varphi}=(\Delta^N,\varphi)$ un simplexe filtré. On définit l'application de dernier sommet, comme suit
\begin{align*}
\sd_P(\Delta^{\varphi})& \xrightarrow{\lv_P}\Delta^{\varphi}\\ {}
[(\sigma_0,p_0),\dots,(\sigma_n,p_n)]&\mapsto[(\lv_P(\sigma_0,p_0),\dots,\lv_P(\sigma_n,p_n)]
\end{align*}
où $\lv_P(\sigma,p)$ est donnée par 
\begin{equation*}
\lv_P(\sigma,p)=\max\{e_k\in (\sigma)_0\subseteq(\Delta^N)_0\ |\ \varphi(e_k)=p\}.
\end{equation*}
Soit $\fil{X}$ un ensemble simplicial filtré, on définit l'application de dernier sommet sur $\fil{X}$ comme la composition suivante
\begin{equation*}
\sd_P(X)\simeq\colim\limits_{\sigma\in\Hom(\Delta^{\varphi},\fil{X})} \sd_P(\Delta^{\varphi})\xrightarrow{\colim(\lv_P)} \colim\limits_{\sigma\in\Hom(\Delta^{\varphi},\fil{X})}\Delta^{\varphi}\simeq \fil{X}
\end{equation*}
Finalement, on obtient une transformation naturelle 
\begin{equation*}
\lv_P\colon \sd_P\to \Id
\end{equation*}
\end{defin}

\begin{prop}
Le foncteur $\sd_P$ admet un adjoint à droite $\Ex_P\colon\sS_P\to\sS_P$. 
\end{prop}

\begin{proof}
Soit $\fil{X}$ un ensemble simplicial filtré. Comme foncteur de $\Delta(P)^{\op}\to \Set$, $\Ex_P(X)$ est définit par la formule 
\begin{equation*}
\Ex_P(X)(\Delta^{\varphi})=\Hom_{\sS_P}(\sd_P(\Delta^{\varphi}),\fil{X})
\end{equation*}
Par construction, $\Ex_P$ est un adjoint à droite de $\sd_P$.
\end{proof}

\begin{defin}
On a la suite de transformation naturelle :
\begin{equation*}
\Id\xrightarrow{\iota}\Ex_P\xrightarrow{\Ex_P(\iota)}\Ex_P^{2}\to\dots
\end{equation*}
La colimite définit un foncteur 
\begin{equation*}
\Exi_P\colon\sS_P\to\sS_P
\end{equation*}
ainsi qu'une transformation naturelle
\begin{equation*}
j\colon\Id\to\Exi_P
\end{equation*}
\end{defin}

Le reste de ce chapitre a pour but de prouver que pour tout ensemble simplicial filtré $\fil{X}$, $\Exi_P(X)$ est fibrant, et $j_X\colon X\to\Exi_P(X)$ est une extension anodine. Une partie de la démonstration, essentiellement indépendante du reste de ce texte, utilise des méthodes générales sur les catégories de préfaisceaux et est reportée aux annexes \ref{ChapitreCaracterisationFibrationsAnnexe} et \ref{ChapitreCaracterisationMorphismeXExXAnnexe}.

\subsection{Présentations anodines filtrées}

\begin{defin}
Soit $m\colon\fil{X}\to\fil{Y}$ une inclusion d'ensembles simpliciaux filtrés. Une présentation anodine filtrée de $m$  est la donnée d'une suite croissante de sous-ensembles simpliciaux filtrés de $\fil{Y}$, $(\fil{X^n})_{n\in \mathbb{N}}$ tels que 
\begin{itemize}
\item $\fil{X^0}=\fil{X}$, $\cup_{n\in\mathbb{N}}\fil{X^n}=\fil{Y}$.
\item pour tout $n\in \mathbb{N}$, l'inclusion $\fil{X^n}\to\fil{X^{n+1}}$ est l'image directe d'une union disjointe d'inclusions de cornets admissibles.
\end{itemize}
Une extension anodine filtrée forte (ou simplement extension anodine forte) est une inclusion d'ensembles simpliciaux filtrés admettant une présentation anodine filtrée.
\end{defin}

\begin{prop}
Un morphisme $f\colon\fil{X}\to\fil{Y}$ de $\sS_P$ est une extension anodine si et seulement si c'est un rétracte d'une extension anodine filtrée forte.
\end{prop}

\begin{proof}
Si $f$ est le rétracte d'une extension anodine filtrée forte, c'est une extension anodine, par définition de la classe des extensions anodines comme classe saturée. Voir la définition \ref{DefClassesSSP} et la proposition \ref{SatureLifting}. Réciproquement, si $f$ est une extension anodine, alors $f$ peut s'écrire comme le rétracte d'une extension anodine filtrée forte, voir la preuve de la proposition \ref{SatureLifting}.
\end{proof}

\begin{remarque}
Soit $\fil{X}\to\fil{Y}$ une extension anodine filtrée forte, munie d'une présentation anodine filtrée. Soit $\sigma\in Y$ un simplexe non-dégénéré de $Y$. Comme $Y=\cup X^n$, il existe un entier $n$, minimal, tel que $\sigma\in X^n$. Si $n=0$, $\sigma\in X$. Si $n>0$, par construction $X^n$ diffère de $X^{n-1}$ par le remplissage de cornets admissibles. Aussi, les simplexes non dégénérés de $X^{n} _{\nd}\setminus X^{n-1}_{\nd}$ arrivent par paires. Ceux maximaux pour l'inclusion, correspondant au simplexe $\Delta^{\varphi}$ pour une inclusion de cornet admissible $\Lambda^{\varphi}_k\to\Delta^{\varphi}$, et ceux correspondant à $d_k(\Delta^{\varphi})$. Ainsi, deux cas sont possibles. Ou bien $\sigma$ est maximal dans $X^n_{\nd}\setminus X^{n-1}_{\nd}$ pour l'inclusion, ou alors $\sigma$ n'est pas maximal pour l'inclusion parmi les simplexes de $X^n_{\nd}\setminus X^{n-1}_{\nd}$ et il existe un unique simplexe maximal, $\tau\in X^{n}_{\nd}$, tel que $\sigma\subset\tau$. On note $\tau=\rho(\sigma)$. Dans le premier cas, on dis que $\sigma$ est de type I, dans le second cas, on dis que $\sigma$ est de type II. Et on note $Y_{\nd}\setminus X_{\nd}= Y_{I}\coprod Y_{II}$ pour la partition en simplexes de type I et II.
\end{remarque}

\begin{defin}\label{DefinitionOrdreAncestral}
Soit $\fil{X}\to\fil{Y}$ une extension anodine filtrée forte, munie d'une présentation anodine filtrée. On définit l'ordre ancestral $<_{\rho}$ sur l'ensemble des simplexes non dégénéré de $Y$, noté $Y_{\nd}$ comme la plus petite relation transitive vérifiant
\begin{itemize}
\item Si $\sigma,\tau\in Y_{\nd}\setminus X_{\nd}$ sont tels que $\sigma\subseteq \tau$ et $\tau\not=\rho(\sigma)$, alors $\sigma<_{\rho} \tau$.
\item Si $\sigma,\tau\in Y_{\nd}\setminus X_{\nd}$ sont tels que $\sigma\subset \rho(\tau)$ et $\sigma\not=\rho(\tau)$, alors $\sigma<_{\rho} \tau$.
\item Si $\sigma\in X_{\nd}$ et $\tau\in Y_{\nd}\setminus X_{\nd}$, alors $\sigma<_{\rho}\tau$.
\end{itemize}
\end{defin}

\begin{remarque}
On remarque que l'ordre ancestral ${<_{\rho}}$ est bien défini pour n'importe quel monomorphisme $\fil{X}\to\fil{Y}$, dès que l'on dispose d'une partition $Y_{\nd}\setminus X_{\nd}= Y_{I}\coprod Y_{II}$ et d'une bijection $\rho\colon Y_{II}\to Y_I$.
\end{remarque}

\begin{prop}
Soit $\fil{X}\to\fil{Y}$ une extension anodine filtrée forte, l'ordre ancestral $<_{\rho}$ est bien fondé.
\end{prop}

\begin{proof}
L'application
\begin{align*}
(Y_{\nd},<_{\rho})&\xrightarrow{R}(\mathbb{N},<)\\
\sigma&\mapsto \min\{n\ |\ \sigma\in X^{n}\}
\end{align*}
préserve l'ordre. En effet, les relations qui génèrent l'ordre $<_{\rho}$ vérifient toutes $\sigma<_{\rho} \tau\Rightarrow R(\sigma)<R(\tau)$. On en déduit que l'ordre $<_{\rho}$ est bien fondé.
\end{proof}

\begin{prop}\label{PresentationAnodinePartition}
Soit $f\colon\fil{X}\to\fil{Y}$ une inclusion d'ensembles simpliciaux filtrés. Alors $f$ est une extension anodine filtrée forte si et seulement s'il existe 
\begin{itemize}
\item Une partition $Y_{\nd}\setminus X_{\nd}=Y_I\coprod Y_{II}$,
\item une bijection $\rho\colon Y_{II}\to Y_{I}$,
\end{itemize}
telles que 
\begin{itemize}
\item pour tout $\sigma\colon \Delta^{\varphi}\to \fil{Y}$ dans $Y_{II}$, il existe un unique $k$ tel que $\sigma=d_k(\rho(\sigma))$, et l'inclusion de cornet $\Lambda^{\varphi}_k\to \Delta^{\varphi}$ est admissible,
\item l'ordre ancestral $<_{\rho}$ est bien fondé.
\end{itemize}
\end{prop}

\begin{proof}
Pour le sens direct, c'est une conséquence de la définition et des propositions précédentes. Pour la réciproque, étant donné une partition $Y_{I}\coprod Y_{II}$ et une bijection $\rho\colon Y_{II}\to Y_{I}$, on considère l'application  
$F\colon X_{\nd}\cup Y_{II}\to\mathbb{N}$ définie par la formule récursive suivante :
\begin{equation*}
F(\sigma)=\left\{\begin{array}{cl}
0& \text{si $\sigma\in X_{\nd}$}\\
\sup\{1\}\cup\{F(\tau)+1\ |\ \tau<_{\rho}\sigma, \tau\in X_{\nd}\cup Y_{II}\}& \text{si $\sigma\in Y_{II}$}
\end{array}\right.
\end{equation*}
Puis on définit $X^n$ par son ensemble de simplexes non-dégénérés :
\begin{equation*}
(X^n)_{\nd}=X_{\nd}\cup \{\sigma\in Y_{II}\ |\ F(\sigma)\leq n\}\cup\{\tau\in Y_{I}\ |\ \exists \sigma\in Y_{II},\ \rho(\sigma)= \tau, \ F(\sigma)\leq n\}
\end{equation*}
Il reste à montrer que pour tout $n\in \N$, l'inclusion $\fil{X^n}\to\fil{X^{n+1}}$ est l'image directe d'une union disjointe d'inclusions de cornets admissibles. Soit $n\geq 0$, notons $S_n$ l'ensemble des simplexes de type II de $X^{n+1}_{\nd}\setminus X^{n}_{\nd}$. De plus, si $\sigma$ est un simplexe de $S_n$, on note $k_{\sigma}$ l'entier tel que $\sigma=d_{k_{\sigma}}(\rho(\sigma))$. Alors, l'inclusion $X^n\to X^{n+1}$ peut être obtenue comme l'image directe suivante
\begin{equation*}
\begin{tikzcd}
\coprod\limits_{\sigma\in S_n}\Lambda^{\varphi\circ\rho(\sigma)}_{k_{\sigma}}
\arrow{r}{\coprod\rho(\sigma)_{|\Lambda}}
\arrow{d}
&\fil{X^n}
\arrow{d}
\\
\coprod\limits_{\sigma\in S_n}\Delta^{\varphi\circ\rho(\sigma)}
\arrow{r}{\coprod\rho(\sigma)}
&\fil{X^{n+1}}
\end{tikzcd}
\end{equation*}
\end{proof}

\subsection{La subdivision filtrée préserve les extensions anodines}

L'objet de cette section est de prouver la propriété suivante 

\begin{prop}\label{SDPreserveExtensionsAnodines}
Soit $f\colon \fil{X}\to\fil{Y}$ une extension anodine. Alors $\sd_P(f)\colon \sd_P\fil{X}\to\sd_P\fil{Y}$ est une extension anodine.
\end{prop}

\begin{proof}
Par fonctorialité, $\sd_P$ préserve les rétractes et les compositions. De plus, comme $\sd_P$ admet un adjoint à droite, il préserve les colimites. En particulier, il suffit de prouver que pour toute inclusion de cornet admissible $\Lambda^\varphi_k\to\Delta^\varphi$, l'inclusion $\sd_P(\Lambda^{\varphi}_k)\to\sd_P(\Delta^{\varphi})$ est une extension anodine. Fixons $\Delta^{\varphi}=(\Delta^N,\varphi)$ un simplexe filtré et $\Lambda^{\varphi}_k\to\Delta^\varphi$ une inclusion de cornet admissible. En utilisant la proposition \ref{SubdivisionSimplexe} pour décrire les simplexes de $\sd_P(\Delta^{\varphi})$, on définit les ensembles de simplexes non dégénérés suivants. Un simplexe non dégénéré de $\sd_P(\Delta^{\varphi})$ est dans un des ensembles $\ref{a}-\ref{h}$ s'il peut s'écrire sous la forme correspondante.
\begin{enumerate}[label=(\alph*)]
\item \label{a} $[(\xi_1,\alpha_1),\dots,(\xi_q,\alpha_q),(d_{k}(\Delta^N),\beta_0),\dots,(d_{k}(\Delta^N),\beta_r),(\Delta^N,\gamma_1),\dots,(\Delta^N,\gamma_s)]$, \\ où on a ($\beta_r\not=\gamma_1$, ou $s=0$).
\item \label{b} $[(\xi_1,\alpha_1),\dots,(\xi_q,\alpha_q),(d_{k}(\Delta^N),\beta_0),\dots,(d_{k}(\Delta^N),\beta_r),(\Delta^N,\beta_r),(\Delta^N,\gamma_1),\dots,(\Delta^N,\gamma_s)]$.
\item \label{c} $[(\mu_0,\epsilon_0),\dots,(\mu_t,\epsilon_0),(\mu_t,\epsilon_1),\dots,(\mu_t,\epsilon_u),(\sigma_1,\nu_1),\dots,(\sigma_v,\nu_v),(\Delta^N,\gamma_0),\dots,(\Delta^N,\gamma_s)]$, \\ où on a  ($\epsilon_u\not =\nu_1$ ou ($v=0$ et $\epsilon_u\not =\gamma_0$)).
\item \label{d} $[(\mu_0,\epsilon_0),\dots,(\mu_t,\epsilon_0),(\mu_t,\epsilon_1),\dots,(\mu_t,\epsilon_u),(\mu_t,\nu_1),(\sigma_1,\nu_1),\dots \\
\phantom{X} \hspace{200pt} \dots,(\sigma_v,\nu_v),(\Delta^N,\gamma_0),\dots ,(\Delta^N,\gamma_s)]$,\\ où $\nu_1=\gamma_0$ dans le cas où $v=0$.
\item \label{e} $[(\eta_0,\gamma_0),\dots,(\eta_w,\gamma_0),(\zeta_1,\gamma_0),\dots,(\zeta_x,\gamma_0),(\Delta^N,\gamma_0),\dots,(\Delta^N,\gamma_s)]$,
 où $\zeta_1\not = \eta_w\cup \{e_k\}$.
\item \label{f} $[(\eta_0,\gamma_0),\dots,(\eta_w,\gamma_0),(\eta_w\cup\{e_k\},\gamma_0),(\zeta_1,\gamma_0),\dots,(\zeta_x,\gamma_0),(\Delta^N,\gamma_0),\dots,(\Delta^N,\gamma_s)]$.
\item \label{g} $[(\theta_1,\gamma_0),\dots,(\theta_y,\gamma_0),(\kappa_1,\gamma_0),\dots,(\kappa_z,\gamma_0),(\Delta^N,\gamma_0),\dots,(\Delta^N,\gamma_s)]$, où $\kappa_1\not = \theta_y\cup \{e_{k'}\}$.
\item \label{h} $[(\theta_1,\gamma_0),\dots,(\theta_y,\gamma_0),(\kappa_1\setminus\{e_{k'}\},\gamma_0),(\kappa_1,\gamma_0),\dots,(\kappa_z,\gamma_0),(\Delta^N,\gamma_0),\dots,(\Delta^N,\gamma_s)]$.
\end{enumerate}
où :
\begin{itemize}
\item $e_k$ est le $k$-ème sommet de $\Delta^N$.
\item $e_{k'}\not=e_k$ est un sommet fixé de $\Delta^N$ tel que $\varphi(e_{k'})=\varphi(e_k)$ (Un tel sommet existe car l'inclusion de cornet $\Lambda^{\varphi}_k\to\Delta^{\varphi}$ est admissible)
\item $q,r,s,t,u,v,w,x,y,z\geq 0$
\item $(\xi_i)_{1\leq i \leq q},(\mu_i)_{1\leq i\leq t},(\sigma_i)_{1\leq i\leq v},(\eta_i)_{1\leq i\leq w},(\zeta_i)_{1\leq i\leq x},(\theta_i)_{1\leq i\leq y},(\kappa_i)_{1\leq i\leq z}$ sont des suites strictement croissantes de faces propres de $\Delta^N$ distinctes de $d_{k}(\Delta^N)$.
\item $(\alpha_i)_{1\leq i\leq q}$ et $(\nu_i)_{1\leq i\leq v}$ sont des chaines strictement croissantes d'éléments de $P$. Lorsque $q$ (resp. $v$) est nul, la chaine est vide.
\item $(\beta_i)_{1\leq i\leq r},(\epsilon_i)_{1\leq i\leq u}$ et $(\gamma_i)_{1\leq i\leq s}$ sont des chaines strictement croissantes et non vides d'éléments de $P$.
\item $e_k\not \in \eta_i$ pour tout $1\leq i\leq w$
\item $e_k\in \zeta_i$ pour tout $1\leq i\leq x$
\item $e_{k}\in \theta_i$ et $e_{k'}\not\in \theta_i$ pour tout $1\leq i\leq y$
\item $e_k,e_{k'}\in \kappa_i$ pour tout $1\leq i\leq z$.
\end{itemize}
On obtient ensuite une présentation anodine en prenant comme simplexes de type II l'union des ensembles \ref{a},\ref{c},\ref{e} et \ref{g}, comme simplexes de type I l'union des ensembles \ref{b}, \ref{d}, \ref{f} et \ref{h}, et comme bijection $\rho$ l'application suggérée par les notations, induisant les bijections $\ref{a}\simeq \ref{b}, \ref{c}\simeq \ref{d},\ref{e}\simeq \ref{f}$ et $\ref{g}\simeq \ref{h}$. Les lemmes \ref{DecompositionSimplexesSubdivision} et \ref{LemmeRhoBijection} permettent ensuite d'appliquer la proposition \ref{PresentationAnodinePartition} pour conclure. 
\end{proof}

\begin{lemme}\label{DecompositionSimplexesSubdivision}
L'ensemble des simplexes non dégénérés de $\sd_P(\Delta^{\varphi})$ est égal à l'union disjointe suivante :
\begin{equation*}
\sd_P(\Delta^{\varphi})_{\nd}=\ref{a}\coprod\ref{b}\coprod\ref{c}\coprod\ref{d}\coprod\ref{e}\coprod\ref{f}\coprod\ref{g}\coprod\ref{h}\coprod \sd_P(\Lambda^{\varphi}_k)_{\nd}
\end{equation*}
\end{lemme}

\begin{proof}
Soit $\tau=[(\tau_0,p_0),\dots,(\tau_n,p_n)]$ un simplexe non dégénéré de $\sd_P(\Delta^{\varphi})_{\nd}$. Procédons par disjonctions de cas successives.
\begin{itemize}
\item Si $\tau_n\not= \Delta^N, d_k(\Delta^N)$, alors $\tau$ est un simplexe non dégénéré de $\sd_P(\Lambda^{\varphi})_{\nd}$. De plus, $\tau$ n'appartient à aucun des ensembles de simplexes $\ref{a}-\ref{h}$.
\item Si $\tau_n=d_k(\Delta^N)$, alors $\tau$ est dans $\ref{a}$, et $\tau$ n'est dans aucun des autres ensembles. 
\item Si $\tau_n= \Delta^N$, alors $\tau$ n'est pas un simplexe de $\sd_P(\Lambda^{\varphi}_k)$ et on distingue plusieurs sous cas
\begin{itemize}
\item S'il existe $0\leq i\leq n$ tel que $\tau_i=d_k(\Delta^N)$, alors $\tau$ n'est dans aucun des ensembles $\ref{c}-\ref{h}$, notons $j=\max\{i\ |\ \tau_i=d_k(\Delta^N)\}$. 
\begin{itemize}
\item si $p_j=p_{j+1}$, $\tau$ est dans \ref{b} et n'est pas dans \ref{a}.
\item si $p_j< p_{j+1}$, $\tau$ est dans \ref{a} et n'est pas dans \ref{b}.
\end{itemize}
\item Si pour tout $0\leq i\leq n$, $\tau_i\not = d_k(\Delta^N)$, alors $\tau$ n'est ni dans \ref{a}, ni dans \ref{b}. Notons $j=\min\{i\ |\ \tau_i=\Delta^N\}$. On distingue plusieurs cas 
\begin{itemize}
\item si $p_0\not= p_j$, alors $\tau$ n'est dans aucun des ensembles $\ref{e}-\ref{h}$. On note $l=\max\{i\ |\ p_l=p_0\}$, et $m=\max\{i\ |\ \tau_i=\tau_l\}$.
\begin{itemize}
\item si $p_m=p_{m+1}$, $\tau$ est dans \ref{d}, et n'est pas dans \ref{c}.
\item si $p_m<p_{m+1}$, $\tau$ est dans \ref{c}, et n'est pas dans \ref{d}.
\end{itemize}
\end{itemize}
\end{itemize}
\end{itemize}

Dans les cas restants, on a $\tau_n=\Delta^N$, $\tau_i\not= d_k(\Delta^N)$ pour tout $0\leq i\leq n$, et en notant $j=\min\{i\ |\ \tau_i=\Delta^N\}$, $p_0=p_j$. En particulier, $\tau$ n'est pas un simplexe de $\sd_P(\Lambda^{\varphi}_k)$, et n'appartient à aucun des ensembles $\ref{a}-\ref{d}$. On distingue les cas restants :
\begin{itemize}
\item Si $e_k$ n'est pas dans $\tau_0$, alors $\tau$ n'appartient pas à \ref{g}, ni à \ref{h}. On note $l=\max\{i\ |\ e_k\not\in\tau_i\}$
\begin{itemize}
\item Si $\tau_{i+1}=\tau_{i}\cup\{e_k\}$, $\tau$ est dans \ref{f}, et n'est pas dans \ref{e}.
\item Si $\tau_{i+1}\not=\tau_{i}\cup\{e_k\}$, $\tau$ est dans \ref{e}, et n'est pas dans \ref{f}.
\end{itemize}
\item Si $e_k\in \tau_0$, $\tau$ n'est pas dans \ref{e}, ni dans \ref{f}, on note $l=\min\{i\ |\ e_{k'}\in \tau_i\}$ ($l$ est bien défini car $e_{k'}\in \Delta^N=\tau_n$.)
\begin{itemize}
\item Si $l=0$, $\tau$ est dans \ref{g} et n'est pas dans \ref{h}.
\item Si $l\geq 1$ et $\tau_l=\tau_{l-1}\cup\{e_{k'}\}$, $\tau$ est dans \ref{h} et n'est pas dans \ref{g}.
\item Si $l\geq 1$ et $\tau_l\not=\tau_{l-1}\cup\{e_{k'}\}$, $\tau$ est dans \ref{g} et n'est pas dans \ref{h}.
\end{itemize}
\end{itemize}
Finalement, nous avons montré que tout simplexe non dégénéré de $\sd_P(\Delta^{\varphi})$ est contenu dans un unique ensemble de la décomposition du lemme \ref{DecompositionSimplexesSubdivision}.
\end{proof}

\begin{lemme}\label{LemmeRhoBijection}
L'application 
\begin{equation*}
\rho\colon \ref{a}\coprod\ref{c}\coprod\ref{e}\coprod\ref{g}\to\ref{b}\coprod\ref{d}\coprod\ref{f}\coprod\ref{h}
\end{equation*}
est une bijection.
\end{lemme}

\begin{proof}
Par définition des ensembles \ref{b}, \ref{d}, \ref{f} et \ref{h}, l'application $\rho$ est surjective. La preuve du lemme \ref{DecompositionSimplexesSubdivision} implique que les ensembles \ref{b}, \ref{d}, \ref{f} et \ref{h} sont disjoints. Il suffit donc de montrer que chacune des restrictions de $\rho$ aux ensembles \ref{a}, \ref{c}, \ref{e} et \ref{g} est injective. Les arguments étant similaires, on donne la preuve pour la restriction à \ref{a}. Soient $\tau$ et $\tau'$ deux simplexes dans \ref{a} tels que $\rho(\tau)=\rho(\tau')$. Notons 
\begin{equation*}
\rho(\tau)=[(\tau_0,p_0),\dots,(\tau_n,p_n)].
\end{equation*}
On pose $l=\min\{i\ |\ \tau_l=\Delta^N\}$.
Par construction de $\rho$, on a $\tau=d_l(\rho(\tau))=d_l(\rho(\tau'))=\tau'$.
\end{proof}

\begin{lemme}\label{RhoPreserveStratification}
Soit $\tau$ un simplexe dans \ref{a},\ref{c},\ref{e} ou \ref{g}. Alors, en notant $\rho(\tau)\colon \Delta^{\psi}\to \sd_P(\Delta^{\varphi})$ son image par $\rho$, il existe un unique entier $l$ tel que $\tau=d_l\rho(\tau)$. De plus, l'inclusion de cornet $\Lambda^{\psi}_l\to\Delta^{\psi}$ est admissible.
\end{lemme}

\begin{proof}
La preuve est similaire pour chacun des ensembles \ref{a},\ref{c}\ref{e} et \ref{g}, on l'écrit donc seulement pour les simplexes de \ref{a}. Soit 
\begin{equation*}
\tau=[(\xi_1,\alpha_1),\dots,(\xi_q,\alpha_q),(d_{k}(\Delta^N),\beta_0),\dots,(d_{k}(\Delta^N),\beta_r),(\Delta^N,\gamma_1),\dots,(\Delta^N,\gamma_s)]
\end{equation*}
un simplexe dans \ref{a}. Alors, ou bien $s=0$, ou bien $\beta_r\not=\gamma_1$. On a, 
\begin{equation*}
\rho(\tau)=[(\xi_1,\alpha_1),\dots,(\xi_q,\alpha_q),(d_{k}(\Delta^N),\beta_0),\dots,(d_{k}(\Delta^N),\beta_r),(\Delta^N,\beta_r),(\Delta^N,\gamma_1),\dots,(\Delta^N,\gamma_s)].
\end{equation*}
en particulier, on a 
\begin{align*}
\tau&=[(\xi_1,\alpha_1),\dots,(\xi_q,\alpha_q),(d_{k}(\Delta^N),\beta_0),\dots,(d_{k}(\Delta^N),\beta_r),(\Delta^N,\gamma_1),\dots,(\Delta^N,\gamma_s)]\\
&=[(\xi_1,\alpha_1),\dots,(\xi_q,\alpha_q),(d_{k}(\Delta^N),\beta_0),\dots,(d_{k}(\Delta^N),\beta_r),\widehat{(\Delta^N,\beta_r)},(\Delta^N,\gamma_1),\dots,(\Delta^N,\gamma_s)]\\
&=d_l(\rho(\tau))
\end{align*}
pour un certain $l$. L'entier $l$ est unique car, par hypothèse, le sommet $(\Delta^N,\beta_r)$ n'est pas dans $\tau$. On note $\rho(\tau)\colon \Delta^{\psi}\to \sd_P(\Delta^{\varphi})$. On a $\psi(e_l)=\beta_r=\psi(e_{l-1})$, par définition de $\rho(\tau))$, et donc l'inclusion de cornet $\Lambda^{\psi}_l\to \Delta^{\psi}$ est admissible.
\end{proof}

\begin{lemme}
L'ordre ancestral $<_{\rho}$ est bien fondé.
\end{lemme}

\begin{proof}
Il suffit de montrer qu'il n'existe pas de chaine infinie strictement décroissante de simplexes non dégénérés. On commence par remarquer que si $\tau$ est un simplexe non dégénéré, et que $\tau'<_{\rho}\tau$ par une des relations élémentaire de la définition \ref{DefinitionOrdreAncestral},  alors $\dim(\tau')\leq \dim(\tau)$, ou $\tau'\in \sd_P(\Lambda^{\varphi}_k)$. De plus, si $\tau$ est de type I, on a $\dim(\tau')<\dim(\tau)$ ou $\tau'\in \sd_P(\Lambda^{\varphi}_k)$. Comme la dimension est bornée inférieurement, il suffit de montrer qu'il n'existe aucune chaine infinie strictement décroissante de simplexes de type II à dimension fixée. Pour ce faire, nous allons montrer que, à dimension fixée,
\begin{itemize}
\item il n'existe de chaine infinie strictement décroissante dans aucun des ensembles \ref{a}, \ref{c}, \ref{e} et \ref{g},
\item toute suite strictement décroissante dans $\ref{a}\coprod \ref{c}\coprod\ref{e}\coprod\ref{g}$ ne peut changer d'ensemble qu'un nombre fini de fois.
\end{itemize}
Dans la suite de cette preuve, on adopte le vocabulaire suivant. Si $\tau$ est un simplexe de type II, et $\tau'\not=\tau$ est un simplexe non dégénéré de $\sd_P(\Delta^{\varphi})$ avec $\dim(\tau)=\dim(\tau')$, on dit que $\tau'$ est un ancêtre direct de $\tau$ si $\tau'<_{\rho}\tau$ par une des relations élémentaires de la définition \ref{DefinitionOrdreAncestral}. En particulier, cela revient à dire que $\tau'=d_l(\rho(\tau))$ pour un certain $l$. Il suffit alors de montrer qu'il n'existe pas de chaine infinie de simplexes de type II de $\sd_P(\Delta^{\varphi})$, $(\tau_i)$ avec $\tau_{i+1}$ un ancêtre direct de $\tau_i$ pour tout $i\geq 0$.

On étudie les ancêtres directs possibles pour des simplexes  de chacun des ensembles \ref{a},\ref{c},\ref{e} et \ref{g}. L'idée de la preuve peut être résumée par le graphe dirigé de la figure \ref{FigureDirectAncestors}.

soit 
\begin{equation*}
\tau=[(\xi_1,\alpha_1),\dots,(\xi_q,\alpha_q),(d_{k}(\Delta^N),\beta_0),\dots,(d_{k}(\Delta^N),\beta_r),(\Delta^N,\gamma_1),\dots,(\Delta^N,\gamma_s)]
\end{equation*}
un simplexe de \ref{a}. On a 
\begin{equation*}
\rho(\tau)=[(\xi_1,\alpha_1),\dots,(\xi_q,\alpha_q),(d_{k}(\Delta^N),\beta_0),\dots,(d_{k}(\Delta^N),\beta_r),(\Delta^N,\beta_r),(\Delta^N,\gamma_1),\dots,(\Delta^N,\gamma_s)]
\end{equation*}
et les ancêtres directs de $\tau$ sont d'une des formes suivantes 
\begin{align*}
&[(\xi_1,\alpha_1),\dots,\widehat{(\xi_i,\alpha_i)},\dots,(\xi_q,\alpha_q),(d_{k}(\Delta^J),\beta_0),\dots\\
&\hspace{180pt} \dots,(d_{k}(\Delta^J),\beta_r),(\Delta^J,\beta_r),(\Delta^J,\gamma_1),\dots,(\Delta^J,\gamma_s)]\\
&\hspace{360pt}1\leq i\leq q\\
&[(\xi_1,\alpha_1),\dots,(\xi_q,\alpha_q),(d_{k}(\Delta^J),\beta_0),\dots,\widehat{(d_{k}(\Delta^J),\beta_i)},\dots\\
&\hspace{180pt}\dots,(d_{k}(\Delta^J),\beta_r),(\Delta^J,\beta_r),(\Delta^J,\gamma_1),\dots,(\Delta^J,\gamma_s)]\\
&\hspace{360pt}0\leq i \leq r\\
&[(\xi_1,\alpha_1),\dots,(\xi_q,\alpha_q),(d_{k}(\Delta^J),\beta_0),\dots,(d_{k}(\Delta^J),\beta_r),(\Delta^J,\beta_r),(\Delta^J,\gamma_1),\dots \\
&\hspace{286pt}\dots,\widehat{(\Delta^J,\gamma_i)},\dots,(\Delta^J,\gamma_s)]\\
&\hspace{360pt}1\leq i\leq s
\end{align*}

Dans le premier et dernier cas, on obtient un simplexe de \ref{b}. Dans le second cas, plusieurs sous cas sont possibles :
\begin{itemize}
\item si $i<r$ on obtient un simplexe de \ref{b},
\item si $i=r>0$, on obtient un simplexe de \ref{a} pour lequel $r$ est strictement plus petit,
\item si $i=r=0$, on obtient un simplexe de \ref{c},\ref{e} ou \ref{g}.
\end{itemize}
De plus, dans tous les sous-cas où on obtient un simplexe de type II, $\gamma_0$ est conservé.
En particulier, il n'existe pas de suite infinie d'ancêtres directs de simplexes de \ref{a} puisque chaque passage à un ancêtre direct dans \ref{a} diminue l'entier positif $r$. 

Soit 
\begin{equation*}
[(\mu_0,\epsilon_0),\dots,(\mu_t,\epsilon_0),(\mu_t,\epsilon_1),\dots,(\mu_t,\epsilon_u),(\sigma_1,\nu_1),\dots,(\sigma_v,\nu_v),(\Delta^N,\gamma_0),\dots,(\Delta^N,\gamma_s)]
\end{equation*}
son image par $\rho$ est 
\begin{align*}
&[(\mu_0,\epsilon_0),\dots,(\mu_t,\epsilon_0),(\mu_t,\epsilon_1),\dots,(\mu_t,\epsilon_u),(\mu_t,\nu_1),(\sigma_1,\nu_1),\dots \\
&\phantom{X} \hspace{200pt} \dots,(\sigma_v,\nu_v),(\Delta^N,\gamma_0),\dots ,(\Delta^N,\gamma_s)],
\end{align*}
et ses ancêtres directs sont d'une des formes suivantes : 
\begin{align*}
&[(\mu_0,\epsilon_0),\dots,\widehat{(\mu_i,\epsilon_0)},\dots,(\mu_t,\epsilon_0),(\mu_t,\epsilon_1),\dots,(\mu_t,\epsilon_u),(\mu_t,\nu_1),(\sigma_1,\nu_1),\dots \\
&\phantom{X} \hspace{200pt} \dots,(\sigma_v,\nu_v),(\Delta^N,\gamma_0),\dots ,(\Delta^N,\gamma_s)]
\\
&[(\mu_0,\epsilon_0),\dots,(\mu_t,\epsilon_0),(\mu_t,\epsilon_1),\dots,\widehat{(\mu_t,\epsilon_i)},\dots,(\mu_t,\epsilon_u),(\mu_t,\nu_1),(\sigma_1,\nu_1),\dots \\
&\phantom{X} \hspace{200pt} \dots,(\sigma_v,\nu_v),(\Delta^N,\gamma_0),\dots ,(\Delta^N,\gamma_s)]
\\
&[(\mu_0,\epsilon_0),\dots,(\mu_t,\epsilon_0),(\mu_t,\epsilon_1),\dots,(\mu_t,\epsilon_u),(\mu_t,\nu_1),(\sigma_1,\nu_1),\dots,\widehat{(\sigma_i,\nu_i)},\dots \\
&\phantom{X} \hspace{200pt} \dots,(\sigma_v,\nu_v),(\Delta^N,\gamma_0),\dots ,(\Delta^N,\gamma_s)]
\\
&[(\mu_0,\epsilon_0),\dots,(\mu_t,\epsilon_0),(\mu_t,\epsilon_1),\dots,(\mu_t,\epsilon_u),(\mu_t,\nu_1),(\sigma_1,\nu_1),\dots \\
&\phantom{X} \hspace{160pt} \dots,(\sigma_v,\nu_v),(\Delta^N,\gamma_0),\dots,\widehat{(\Delta^N,\gamma_i)} ,(\Delta^N,\gamma_s)]
\end{align*}
Dans le premier cas on a les sous-cas suivants. 
\begin{itemize}
\item Si $i<t$, on obtient un simplexe de \ref{d}.
\item Si $i=t>0$, on obtient un simplexe de \ref{c} de la forme suivante,
\begin{align*}
&[(\mu_0,\epsilon_0),\dots,(\mu_{t-1},\epsilon_0),(\mu_t,\epsilon_1),\dots,(\mu_t,\epsilon_u),(\mu_t,\nu_1),(\sigma_1,\nu_1),\dots \\
&\phantom{X} \hspace{200pt} \dots,(\sigma_v,\nu_v),(\Delta^N,\gamma_0),\dots ,(\Delta^N,\gamma_s)].
\end{align*}
En le réécrivant sous la forme \ref{c}, on obtient le même $\epsilon_0$ et un $t$ strictement plus petit.
\item Si $i=t=0$, on obtient un simplexe de la forme suivante 
\begin{align*}
&[(\mu_t,\epsilon_1),\dots,(\mu_t,\epsilon_u),(\mu_t,\nu_1),(\sigma_1,\nu_1),\dots \\
&\phantom{X} \hspace{200pt} \dots,(\sigma_v,\nu_v),(\Delta^N,\gamma_0),\dots ,(\Delta^N,\gamma_s)]
\end{align*}
plusieurs cas sont possibles :
\begin{itemize}
\item on obtient un simplexe de \ref{c} avec $\epsilon_0$ strictement plus grand, 
\item on obtient un simplexe de \ref{d},
\item on obtient un simplexe de \ref{e},\ref{f},\ref{g} ou \ref{h}.
\end{itemize}
\end{itemize}
Dans le deuxième cas, on obtient un simplexe de \ref{d}.
Dans le troisième cas, deux sous-cas sont possibles.
\begin{itemize}
\item Si $i=1$, on obtient un simplexe de la forme 
\begin{align*}
&[(\mu_0,\epsilon_0),\dots,(\mu_t,\epsilon_0),(\mu_t,\epsilon_1),\dots,(\mu_t,\epsilon_u),(\mu_t,\nu_1),\widehat{(\sigma_1,\nu_1)},\dots \\
&\phantom{X} \hspace{200pt} \dots,(\sigma_v,\nu_v),(\Delta^N,\gamma_0),\dots ,(\Delta^N,\gamma_s)]
\end{align*}
c'est un simplexe de \ref{c}, pour lequel $u$ est strictement plus grand, et avec les mêmes $\epsilon_0$ et $t$, ou un simplexe de \ref{d} si (($v\geq 2$ et $\mu_2=\mu_1$) ou ($v=1$ et $\mu_1=\gamma_0$)).
\item Si $i>1$, on obtient un simplexe de \ref{d}.
\end{itemize}
Dans le dernier cas, on obtient un simplexe de $\sd_P(\Lambda^{\varphi}_k)$ ou de \ref{d}.
En particulier, dans tous les sous-cas où on obtient un simplexe de type II, $\gamma_0$ est conservé. Comme précédemment, on conclut qu'il n'existe pas de suite infinie de simplexes de \ref{c} puisque tout passage à un ancêtre direct dans \ref{c}
\begin{itemize}
\item augmente $\epsilon_0$, qui est borné par $\gamma_0$,
\item ou diminue $t$, borné par $0$, en gardant $\epsilon_0$ constant,
\item ou augmente $u$, borné par la dimension du simplexe, en gardant $\epsilon_0$ et $t$ constants.
\end{itemize}
De plus, on constate que tout passage à un ancêtre direct de type II garde $\gamma_0$ constant.

Soit 
\begin{equation*}
[(\eta_0,\gamma_0),\dots,(\eta_w,\gamma_0),(\zeta_1,\gamma_0),\dots,(\zeta_x,\gamma_0),(\Delta^N,\gamma_0),\dots,(\Delta^N,\gamma_s)]
\end{equation*}
un simplexe de \ref{e}. Son image par $\rho$ est 
\begin{equation*}
[(\eta_0,\gamma_0),\dots,(\eta_w,\gamma_0),(\eta_w\cup\{e_k\},\gamma_0),(\zeta_1,\gamma_0),\dots,(\zeta_x,\gamma_0),(\Delta^N,\gamma_0),\dots,(\Delta^N,\gamma_s)]
\end{equation*}
et ses ancêtres directs sont de la forme
\begin{align*}
&[(\eta_0,\gamma_0),\dots,\widehat{(\eta_i,\gamma_0)},\dots,(\eta_w,\gamma_0),(\eta_w\cup\{e_k\},\gamma_0),(\zeta_1,\gamma_0),\dots,(\zeta_x,\gamma_0),(\Delta^N,\gamma_0),\dots,(\Delta^N,\gamma_s)]\\
&[(\eta_0,\gamma_0),\dots,(\eta_w,\gamma_0),(\eta_w\cup\{e_k\},\gamma_0),(\zeta_1,\gamma_0),\dots,\widehat{(\zeta_i,\gamma_0)},\dots,(\zeta_x,\gamma_0),(\Delta^N,\gamma_0),\dots,(\Delta^N,\gamma_s)]\\
&[(\eta_0,\gamma_0),\dots,(\eta_w,\gamma_0),(\eta_w\cup\{e_k\},\gamma_0),(\zeta_1,\gamma_0),\dots,(\zeta_x,\gamma_0),(\Delta^N,\gamma_0),\dots,\widehat{(\Delta^N,\gamma_i)},\dots,(\Delta^N,\gamma_s)]
\end{align*}
Dans le premier cas, trois sous cas sont possibles.
\begin{itemize}
\item Si $i<w$, on obtient un simplexe de \ref{f}.
\item Si $i=w>0$, on obtient un simplexe de \ref{e} avec un $w$ strictement plus petit.
\item Si $i=w=0$, on obtient un simplexe de la forme
\begin{equation*}
[(\eta_w\cup\{e_k\},\gamma_0),(\zeta_1,\gamma_0),\dots,(\zeta_x,\gamma_0),(\Delta^N,\gamma_0),\dots,(\Delta^N,\gamma_s)]
\end{equation*}
c'est un simplexe de \ref{g} ou \ref{h}.
\end{itemize}
Dans le second cas, on obtient un simplexe de \ref{f}. Dans le troisième cas, trois sous cas sont possibles.
\begin{itemize}
\item Si $i>0$, on obtient un simplexe de \ref{f}.
\item Si $i=s=0$, on obtient un simplexe de $\sd_P(\Lambda^{\varphi}_k)$.
\item Si $i=0<s$, on obtient un simplexe de la forme 
\begin{equation*}
[(\eta_0,\gamma_0),\dots,(\eta_w,\gamma_0),(\eta_w\cup\{e_k\},\gamma_0),(\zeta_1,\gamma_0),\dots,(\zeta_x,\gamma_0),(\Delta^N,\gamma_1),\dots,(\Delta^N,\gamma_s)]
\end{equation*}
qui est dans \ref{c} ou \ref{d}, et $\gamma_0$ est strictement plus grand.
\end{itemize}
On conclut qu'il n'existe pas de suite infinie de simplexes de \ref{e} puisque tout passage à un ancêtre direct dans \ref{e} diminue $w$ qui est borné par $0$. De plus, on constate que le passage à un ancêtre direct depuis \ref{e} vers \ref{c} augmente strictement $\gamma_0$, qui est borné par $\gamma_s$, et que le passage à un ancêtre direct dans \ref{e} ou \ref{g} garde $\gamma_0$ constant.

Soit 
\begin{equation*}
[(\theta_1,\gamma_0),\dots,(\theta_y,\gamma_0),(\kappa_1,\gamma_0),\dots,(\kappa_z,\gamma_0),(\Delta^N,\gamma_0),\dots,(\Delta^N,\gamma_s)]
\end{equation*}
un simplexe de \ref{g}. Son image par $\rho$ est 
\begin{equation*}
[(\theta_1,\gamma_0),\dots,(\theta_y,\gamma_0),(\kappa_1\setminus\{e_{k'}\},\gamma_0),(\kappa_1,\gamma_0),\dots,(\kappa_z,\gamma_0),(\Delta^N,\gamma_0),\dots,(\Delta^N,\gamma_s)]
\end{equation*}
et ses ancêtres directs sont d'une des formes suivantes :
\begin{align*}
&[(\theta_1,\gamma_0),\dots,\widehat{(\theta_i,\gamma_0)},\dots,(\theta_y,\gamma_0),(\kappa_1\setminus\{e_{k'}\},\gamma_0),(\kappa_1,\gamma_0),\dots,(\kappa_z,\gamma_0),(\Delta^N,\gamma_0),\dots,(\Delta^N,\gamma_s)]\\
&[(\theta_1,\gamma_0),\dots,(\theta_y,\gamma_0),(\kappa_1\setminus\{e_{k'}\},\gamma_0),(\kappa_1,\gamma_0),\dots,\widehat{(\kappa_i,\gamma_0)},\dots,(\kappa_z,\gamma_0),(\Delta^N,\gamma_0),\dots,(\Delta^N,\gamma_s)]\\
&[(\theta_1,\gamma_0),\dots,(\theta_y,\gamma_0),(\kappa_1\setminus\{e_{k'}\},\gamma_0),(\kappa_1,\gamma_0),\dots,(\kappa_z,\gamma_0),(\Delta^N,\gamma_0),\dots,\widehat{(\Delta^N,\gamma_i)},\dots,(\Delta^N,\gamma_s)].
\end{align*}
Dans le premier cas, on obtient un simplexe de \ref{h}.
Dans le deuxième cas, deux-sous cas sont possibles :
\begin{itemize}
\item si $i>1$, on obtient un simplexe de \ref{h},
\item si $i=1$, on obtient un simplexe de la forme
\begin{equation*}
[(\theta_1,\gamma_0),\dots,(\theta_y,\gamma_0),(\kappa_1\setminus\{e_{k'}\},\gamma_0),(\kappa_2,\gamma_0),\dots,(\kappa_z,\gamma_0),(\Delta^N,\gamma_0),\dots,(\Delta^N,\gamma_s)].
\end{equation*}
C'est un simplexe de \ref{g}, où $y$ est strictement plus grand.
\end{itemize}
Dans le troisième cas, trois sous-cas sont possibles :
\begin{itemize}
\item si $i>0$, on obtient un simplexe de \ref{h},
\item si $i=s=0$, on obtient un simplexe de $\sd_P(\Lambda^{\varphi}_k)$,
\item si $i=0<s$, on obtient un simplexe de \ref{c} ou \ref{d}, avec $\gamma_0$ strictement plus grand.
\end{itemize}
On conclut comme précédemment qu'il n'existe pas de suite infinie de simplexes de \ref{g} puisque tout passage à un ancêtre direct dans \ref{g} augmente $y$ qui est borné par la dimension du simplexe. De plus, on observe que le passage à un ancêtre direct de \ref{g} vers \ref{c} augmente $\gamma_0$ qui est borné par $\gamma_s$. Par ailleurs, tout passage à un ancêtre direct de \ref{g} vers \ref{g} garde $\gamma_0$ constant.

Finalement, nous avons montré qu'un passage à un ancêtre direct pour un simplexe de type II pouvait donner une des choses suivantes :
\begin{itemize}
\item un simplexe de type I ou de $\sd_P(\Lambda^{\varphi}_k)$,
\item un simplexe du même ensemble \ref{a},\ref{c},\ref{e},\ref{g}, en fixant $\gamma_0$, et ceci ne peut arriver qu'un nombre fini de fois,
\item un simplexe d'un autre ensemble, en respectant l'ordre \ref{a}>\ref{c}>\ref{e}>\ref{g}, en fixant la quantité $\gamma_0$,
\item un simplexe d'un autre ensemble, en ne respectant pas l'ordre précédent, en augmentant strictement la quantité $\gamma_0$.
\end{itemize}
Comme $\gamma_0$ est borné par le $\gamma_s$ d'un simplexe initial, on en déduit que toute chaine de simplexes de type II, strictement décroissante pour l'ordre ancestral, est finie.
\end{proof}

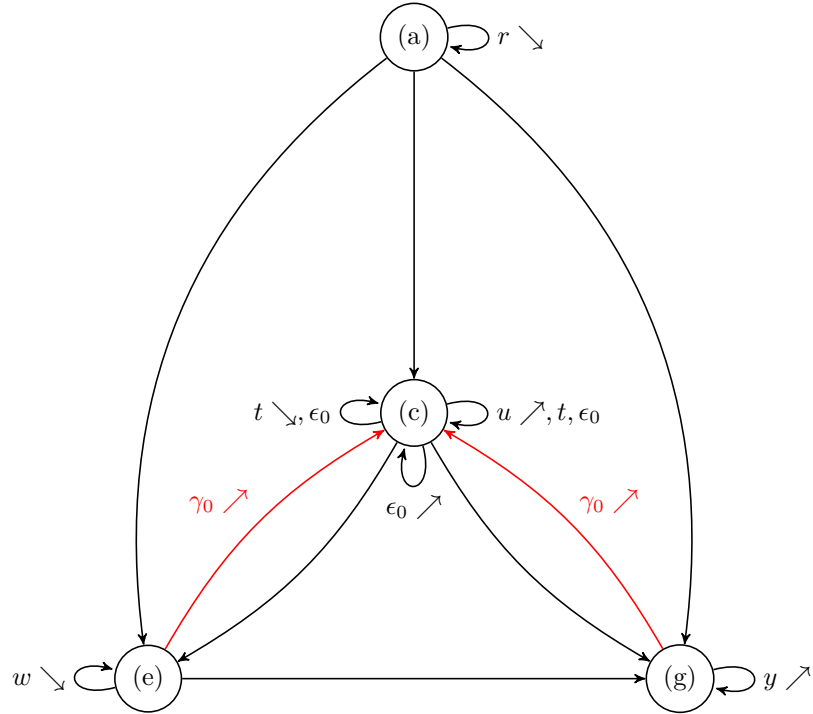
\begin{figure}[h]
\centering
\begin{tikzpicture}[->,>=stealth',shorten >=0pt,auto,node distance=5cm,semithick]
\node[state](A){\ref{a}};
\node[state](C)[below of =A]{\ref{c}};
\node[state](E)[below left of=C]{\ref{e}};
\node[state](G)[below right of =C]{\ref{g}};

\path (A) edge [loop right] node{$r\searrow$} (A)
	edge node{} (C)
	edge [bend right] node{} (E)
	edge [bend left] node{} (G)

(C) edge [loop right]  node{$u\nearrow,t,\epsilon_0$} (C)
	edge [loop left] node{$t\searrow,\epsilon_0$} (C)
	edge [loop below] node{$\epsilon_0\nearrow$} (C)
	edge [bend right=15] node{} (G)
	edge [bend left=15] node{} (E)
(E) edge [loop left] node{$w\searrow$} (E)
	edge node{} (G)
	edge [bend left=15,red] node{$\gamma_0\nearrow$} (C)
(G) edge [loop right] node{$y\nearrow$} (G)
	edge [bend right=15,red] node[above right]{$\gamma_0\nearrow$} (C); 	
\end{tikzpicture}
\caption{Le graphe dirigé représentant les ancêtres directs possibles pour chaque ensemble de simplexes. Les labels le long des flèches indiquent quelles quantités caractéristiques du simplexe augmentent ($\nearrow$), diminuent ($\searrow$) ou restent constantes lors du passage à un ancêtre direct.}
\label{FigureDirectAncestors}
\end{figure}

\clearpage

\subsection{$\Ex_P^{\infty}(X)$ est fibrant}

L'objet de cette sous section est de prouver le lemme suivant 
\begin{lemme}\label{ExInfiniFibrant}
Soit $\fil{X}$ un ensemble simplicial filtré, alors $\Exi_P(X)$ est fibrant.
\end{lemme}

Pour prouver ce résultat, on se servira de l'application de dernier sommet filtré (Définition \ref{DernierSommetFiltre}) ainsi que de deux autres applications similaires, dont on donne la définition.

\begin{defin}
Soit $\Delta^{\varphi}=(\Delta^N,\varphi)$ un simplexe filtré, avec $\Delta^N=[e_0,\dots,e_N]$, et soit $\tau=([e_{i_0},\dots,e_{i_k}],p)$ un sommet de la subdivision filtrée $\sd_P(\Delta^{\varphi})$. On définit le dernier sommet de $\tau$ comme $\lv(\tau)=e_{i_k}\in\Delta^{\varphi}$.
%
Cette définition s'étend en une application \textbf{non-filtrée}
\begin{align*}
\lv\colon \sd_P\fil{X}&\to X\\
(\sigma,[(\sigma_0,p_0),\dots,(\sigma_n,p_n))])&\mapsto \sigma([\lv(\sigma_0,p_0),\dots,\lv(\sigma_n,p_n)]).
\end{align*} 
Ici, $\sigma([\lv(\sigma_0),\dots,\lv(\sigma_n)])$ désigne la composition non filtrée $\Delta^{n}\to\Delta^N\xrightarrow{\sigma}X$, où $\sigma\colon \Delta^{\psi}\to X$ est un simplexe non dégénéré, avec $\Delta^{\psi}=(\Delta^N,\psi)$. Voir Corollaire \ref{SimplexesSubdivisionEnsembleSimplicialFiltre}.
Si $\tau=([e_{i_0},\dots,e_{i_k}],p)$ contient au moins un sommet tel que $\varphi(e_{i_l})=q$, on définit le dernier sommet de couleur $q$ de $\tau$ comme $\lv_q(\tau)=\max\{e\ |\ e\in \tau, \varphi(e)=q\}$, dans le cas contraire, l'application $\lv_q$ n'est pas définie. De même, cette définition s'étend aux sommets de $\sd_P\fil{X}$.
\end{defin}

\begin{remarque}
L'application de dernier sommet filtrée de la définition \ref{DernierSommetFiltre} peut se réécrire sous la forme
\begin{equation*}
\lv_P[(\sigma_0,p_0),\dots,(\sigma_n,p_n)]=[\lv_{p_0}(\sigma_0),\dots,\lv_{p_n}(\sigma_n)]
\end{equation*}
\end{remarque}

\begin{proof}[Démonstration du Lemme \ref{ExInfiniFibrant}]
Soit $\fil{X}$ un ensemble simplicial filtré. Considérons le problème de relèvement suivant :
\begin{equation*}
\begin{tikzcd}
\Lambda^{\varphi}_k
\arrow{d}
\arrow{r}
&\Exi_P(X)
\\
\Delta^{\varphi}
\arrow[dashed]{ur}
\end{tikzcd}
\end{equation*}
où $\Lambda^{\varphi}_k\to \Delta^{\varphi}$ est une inclusion de cornet admissible. Comme $\Lambda^{\varphi}_k$ est compact, l'image de $\Lambda^{\varphi}_k$ dans $\Exi_P$ se factorise par une inclusion $\Ex^k_P(X)\to\Exi_P(X)$ pour un certain $k\geq 0$. Quitte à remplacer $X$ par $\Ex^k_P(X)$, on peut supposer que l'on est dans la situation suivante
\begin{equation*}
\begin{tikzcd}
\Lambda^{\varphi}_k
\arrow{r}{\lambda}
\arrow{d}
&\Ex_P(X)
\arrow{d}
\\
\Delta^{\varphi}
\arrow[dashed]{r}
&\Ex^{3}_P(X)
\end{tikzcd}
\end{equation*}
Il suffit donc de trouver une application $\Delta^{\varphi}\to \Ex^3_P(X)$ faisant commuter le diagramme. En remarquant l'égalité entre les composées $\Lambda^{\varphi}_k\xrightarrow{\lambda}\Ex_P(X)\to\Ex^3_P(X)$ et $\Lambda^{\varphi}_k\to\Ex^2_P(\Lambda^{\varphi}_k)\xrightarrow{\Ex^2_P(\lambda)}\Ex^3_P(X)$ et en utilisant l'adjonction, on obtient à partir du diagramme précédent le diagramme suivant :
\begin{equation}\label{DiagrammeExInfiniFibrant}
\begin{tikzcd}
\sd^2_P(\Lambda^{\varphi}_k)
\arrow{r}{\lv^2_P}
\arrow{d}
&\Lambda^{\varphi}_k
\arrow{d}
\arrow{dr}{\lambda}
\\
\sd^2_P(\Delta^{\varphi})
\arrow[dashed, swap]{r}{h}
&\Ex_P(\sd_P(\Lambda^{\varphi}_k))
\arrow[swap]{r}{\Ex_P(\widehat{\lambda})}
&\Ex_P(X)
\end{tikzcd}
\end{equation}
Il suffit donc d'exhiber une application $h\colon \sd^2_P(\Delta^{\varphi})\to \Ex_P(\sd_P(\Delta^{\varphi}_k))$ rendant le diagramme commutatif pour prouver le résultat souhaité.
Soit $\sigma=[(\sigma_0,q_0),\dots,(\sigma_n,q_n)]$ un simplexe de la double subdivision de $\Delta^{\varphi}$. (On rappelle que chacun des $\sigma_l=[(\tau_0^l,r_0^l),\dots,(\tau_{m_l}^l,r_{m_l}^l)]$ est un simplexe de la subdivision filtrée $\sd_P(\Delta^{\varphi})$). On note $\Delta^{\psi}=(\Delta^n,\psi)$ avec $\Delta^n=[d_0,\dots,d_n]$ et $\psi(\Delta^n)=[q_0,\dots,q_n]$ et $\Delta^{\varphi}=(\Delta^N,\varphi)$, avec $\Delta^N=[e_0,\dots,e_N]$. On définit l'application $f_{\sigma}\colon \Delta^{\psi}_0\to\Delta^{\varphi}_0$ entre les sommets des simplexes filtrés. (Attention, $f_{\sigma}$ ne peut en général pas être étendue en une application simpliciale $\Delta^{\psi}\to\Delta^{\varphi}$.) Notons $p=\varphi(e_k)$.
\begin{align*}
f_{\sigma}\colon (\Delta^{\psi})_0&\to (\Delta^{\varphi})_0\\
d_l&\mapsto\left\{\begin{array}{cl}
\lv_{q_l}(\lv_{q_l}(\sigma_l))& \text{ si $p>q_l$, ou $\sigma_l$ ne contient pas de sommet de couleur $p$,}\\
\phantom{X}& \text{ ou $\lv(\sigma_l)\not = d_k(\Delta^{\varphi}),\Delta^{\varphi}$}\\
\lv_{q_l}(\lv_p(\sigma_l)) &\text{ si $p<q_l$, et $\sigma_l$ contient un sommet de couleur $p$}\\
\phantom{X} &\text{ et $\lv(\sigma_l)=d_k(\Delta^{\varphi})$ ou $\Delta^{\varphi}$.}\\
e_k &\text{ si $p=q_l$, et $\lv(\sigma_l)=d_k(\Delta^{\varphi})$ ou $\Delta^{\varphi}$.}
\end{array}
\right.
\end{align*}
Par le lemme \ref{FSigmaBienDef}, $f_{\sigma}\colon (\Delta^{\psi})_0\to(\Delta^{\varphi})_0$ est bien définie, et préserve la filtration. Pour obtenir une application $h(\sigma)\colon \sd_P(\Delta^{\psi})\to\sd_P(\Delta^{\varphi})$. On définit d'abord $h(\sigma)$ sur les sommets de $\sd_P(\Delta^{\psi})$ comme suit. Si $(\mu,s)$ est un sommet de $\sd_P(\Delta^{\psi})$, on définit $h(\sigma)(\mu,s)=(\nu,s)\in \sd_P(\Delta^{\varphi})$, où $\nu$ est la face de $\Delta^{\varphi}$ contenant tous les sommets $f_{\sigma}(d)$, $d\in\mu$. Cette définition fait sens car $f_{\sigma}$ préserve la filtration, et donc, si $\mu$ contient un sommet de couleur $s$, $\nu$ aussi. On étend ensuite la définition de $h(\sigma)$ en une application simpliciale par la définition suivante :
\begin{align*}
h(\sigma)\colon \sd_P(\Delta^{\psi})&\to\sd_P(\Delta^{\varphi})\\
[(\mu_0,s_0),\dots,(\mu_{m'},s_{m'})]&\mapsto[h(\sigma)(\mu_0,s_0),\dots,h(\sigma)(\mu_{m'},s_{m'})]
\end{align*}
Par construction, $h(\sigma)$ est une application simpliciale filtrée.
De plus, l'application $\sigma\mapsto h(\sigma)$ est simpliciale filtrée car, par construction de $f_{\sigma}$, on a $f_{d_i(\sigma)}=f_{\sigma}\circ D_i\colon (d_i(\Delta^{\psi}))_0\to(\Delta^{\psi})_0\to (\Delta^{\varphi})_0$, et de même pour les dégénérescences.
Il reste à montrer que pour tout $\sigma$, $h(\sigma)\in \Ex_P(sd_P(\Lambda^{\varphi}_k)$ (a priori, on sait seulement que $h(\sigma)\in \Ex_P(\Delta^{\varphi})$), et que $h$ rend bien le diagramme précédent commutatif. C'est le contenu des lemmes \ref{HSigmaDansLambda} et \ref{HFaitCommuterLeDiagramme}.
\end{proof}

\begin{lemme}\label{FSigmaBienDef}
L'application $f_{\sigma}\colon (\Delta^{\psi})_0\to(\Delta^{\varphi})_0$ est bien définie, et, si $d\in (\Delta^{\psi})_0$ est un sommet de $\Delta^{\psi}$, on a $\varphi(f_{\sigma}(d))=\psi(d)$.
\end{lemme}

\begin{proof}
 Si on fixe un sommet $d_l\in \Delta^{\psi}_0$, une seule des conditions sur $(\sigma_l,q_l)$ est vérifiée, il suffit donc de montrer que les applications de dernier sommet intervenant dans le calcul sont toutes bien définies.
On rappelle que $(\sigma_l,q_l)$ est un sommet de $\sd^2_P(\Delta^{\varphi})$. Dans le premier cas, par construction de la subdivision filtrée, $\sigma_l$ doit contenir au moins un sommet de couleur $q_l$, et donc $\lv_{q_l}(\sigma_l)$ est bien défini. On a donc, $\lv_{q_l}(\sigma_l)=(\tau^l_m,q_l)$ pour un certain $m\geq 0$. De même, $\tau^l_m$ doit contenir un sommet de couleur $q_l$, et donc $\lv_{q_l}(\tau^l_m)$ est bien défini.
 De plus, par construction, $\varphi(\lv_{q_l}(\tau^l_m))=q_l=\psi(d_l)$. Dans le second cas, $\lv_p(\sigma_l)=(\tau^l_m,p)$ pour un certain $m\geq 0$. Une fois encore, $\tau^l_m$ doit contenir un point de couleur $q_l$, car $[(\tau^l_0,q_0),\dots,(\tau^l_{m_l},q_{m_l})]$ est un simplexe de $\sd(\Delta^{\varphi})$, donc $\lv_{q_l}\tau^l_m$ est bien défini, et on a $\varphi(\lv_{q_l}(\tau^l_m)=q_l=\psi(d_l)$. Finalement, dans le troisième cas, on a $\varphi(e_k)=p=q_l=\psi(d_l)$. On en déduit que $f_{\sigma}$ est bien définie et préserve la filtration sur les sommets.
\end{proof}

\begin{lemme}\label{HSigmaDansLambda}
Pour tout simplexe $\sigma\colon\Delta^{\psi}\to \sd^2_P(\Delta^{\varphi})$, $h(\sigma)(\sd_P(\Delta^{\psi})) \subseteq\sd_P(\Lambda^{\varphi}_k)$.
\end{lemme}

\begin{proof}
Soit $\sigma\colon\Delta^{\psi}\to \sd^2_P(\Delta^{\varphi})$ un simplexe de $\sd^2_P(\Delta^{\varphi})$. On écrit $\sigma=[(\sigma_0,q_0),\dots,(\sigma_n,q_n)]$, 
avec $\sigma_l=[(\tau_0^l,r_0^l),\dots,(\tau_{m_l}^l,r_{m_l}^l)]$. On raisonne par contradiction. Supposons qu'il existe un sommet $(\mu,s)\in \sd_P(\Delta^{\psi})$ tel que $h(\sigma)(\mu,s)\not\in\sd_P(\Lambda^{\varphi})$. 
Dans ce cas, $h(\sigma)(\mu,s)=(\Delta^{\varphi},s)$ ou $(d_k(\Delta^{\varphi}),s)$. En particulier, $f_{\sigma}$ doit atteindre tous les sommets de $d_k(\Delta^{\varphi})$.
Par construction de $f_{\sigma}$, cela implique que tous les sommets de $d_k(\Delta^{\varphi})$ appartiennent à un des $\tau^l_i$. Mais, par construction de la subdivision filtrée, on doit avoir $\cup\tau^l_i\subseteq \tau_{m_n}^n$, et donc $d_k(\Delta^{\varphi})\subseteq \tau^n_{m_n}$. Soit $j=\min\{i\ |\ d_k(\Delta^{\varphi})\subseteq\tau^n_i\}$. Si $j=0$, alors comme pour tout $(i,l)$, $\tau^l_i$ est dans la liste $(\tau^n_0\subseteq\dots\subseteq\tau^n_{m_n})$, on a $\tau^l_i=d_k(\Delta^{\varphi})$ ou $\Delta^{\varphi}$. En particulier, pour tout $l$ tel que $q_l=p$, on a $f_{\sigma}(d_l)=e_k$. Or, $\Lambda^{\varphi}_k$ est admissible, donc $\varphi(e_k)=\varphi(e_{k'})$ avec $k'=k+1$ ou $k-1$. En particulier $e_{k'}$ n'est pas atteint par $f_{\sigma}$.
Si $j>0$, on a $d_k(\Delta^{\varphi})\not\subseteq\tau^n_{j-1}$. En particulier, il existe $k'\not=k$ tel que $\tau^n_{j-1}\subseteq d_{k'}(\Delta^{\varphi})$. De plus, par construction de la subdivision filtrée, tous les $\tau^l_i$ apparaissent dans la suite $(\tau^n_0\subseteq\dots\subseteq\tau^n_{m_n})$. En particulier, on a   $\tau^l_i\subseteq \tau^{n}_{j-1}\subseteq d_{k'}(\Delta^{\varphi})$ ou $d_k(\Delta^{\varphi})\subseteq\tau^n_{j}\subseteq\tau^l_i$. Considérons le sommet $e_{k'}$ de $\Delta^{\varphi}$, et notons $p'=\varphi(e_{k'})$. Par hypothèse, il existe $l'$ tel que $f_{\sigma}(d_{l'})=e_{k'}$. Par construction de $f_{\sigma}$, il existe $i'$ tel que $\tau^{l'}_{i'}$ contient $e_{k'}$, mais alors, par les remarques précédentes, $d_k(\Delta^{\varphi})\subseteq\tau^{l'}_{i'}$. En particulier, $d_k\Delta^{\varphi}\subseteq\lv(\sigma_{l'})$. De même, pour tout $l\geq l'$, on a aussi  $d_k\Delta^{\varphi}\subseteq\lv(\sigma_{l})$. On distingue maintenant plusieurs cas, suivant la valeur de $p'$.
\begin{itemize}
\item Si $p'<p$, alors comme $\Lambda^{\varphi}_k$ est admissible, $\varphi(e_k)=\varphi(e_{k+1})$ ou $\varphi(e_{k-1})$. L'argument étant identique dans les deux cas, on suppose $\varphi(e_{k+1})=\varphi(e_k)$. Par hypothèse, il existe $l$ tel que $f_{\sigma}(d_l)=e_{k+1}$. Comme $p>p'$, cela impose que $l>l'$. En particulier, $d_k(\Delta^{\varphi})\subseteq\lv(\sigma_l)$, et donc $f_{\sigma}(d_l)=e_k$ ce qui donne la contradiction souhaitée.
\item Si $p=p'$, alors $f_{\sigma}(d_{l'})=e_{k}\not=e_{k'}$, ce qui donne la contradiction souhaitée.
\item Si $p<p'$, comme $\Lambda^{\varphi}_k$ est admissible, $p=\varphi(e_k)=\varphi(e_{k+1})$ ou $\varphi(e_{k-1})$. Par hypothèse, tous les sommets de $d_k(\Delta^{\varphi})$ sont atteints donc en particulier, $f_{\sigma}$ atteint un sommet de couleur $p$. Comme $f_{\sigma}$ préserve la filtration, et par définition de la subdivision, on en déduit que tous les $\sigma_l$ contiennent un sommet de couleur $p$. De plus  comme $e_{k'}\in \lv(\sigma_{l'})$, on a $d_k(\Delta^{\varphi})\subseteq\lv(\sigma_{l'})$, et donc
$\lv(\sigma_{l'})=d_k(\Delta^{\varphi})$ ou $\Delta^{\varphi}$.
On en déduit que $f_{\sigma}(d_{l'})=\lv_{q_{l'}}(\lv_p(\sigma_{l'}))=e_{k'}$.
En particulier, $\lv_p(\sigma_{l'})$ doit contenir $e_{k'}$, et on en déduit que $d_k(\Delta^{\varphi})\subseteq\lv_p(\sigma_{l'})$. 
Comme $p'>p$ et par construction de $\sd_P$, ceci implique que pour tout $l$, $d_k(\Delta^{\varphi})\subseteq\lv_p(\sigma_{l'})\subseteq\lv_{p'}(\sigma_l)\subseteq\lv(\sigma_l)$. Comme $\Lambda^{\varphi}_k$ est admissible, et par symétrie, on suppose que $\varphi(e_{k+1})=\varphi(e_k)$. Par hypothèse, il existe $l$ tel que $f_{\sigma}(d_l)=e_{k+1}$, mais on a $d_k(\Delta^{\varphi})\subseteq\lv(\sigma_l)$, et donc $f_{\sigma}(d_l)=e_k$, ce qui donne la contradiction voulue.
%
\end{itemize}
\end{proof}

\begin{lemme}\label{HFaitCommuterLeDiagramme}
L'application $h$ comme définie dans la preuve de lemme \ref{ExInfiniFibrant} fait commuter le diagramme \ref{DiagrammeExInfiniFibrant}.
\end{lemme}

\begin{proof}
Soit $\sigma\colon \Delta^{\psi}\to\sd^2_P(\Lambda^{\varphi}_k)$ un simplexe. On note encore $\sigma=[(\sigma_0,q_0),\dots,(\sigma_n,q_n)]$. 
Comme $\sigma$ est dans la double subdivision de $\sd_P(\Lambda^{\varphi}_k)$, on a en particulier que $\lv(\sigma_l)\not=d_k(\Delta^{\varphi}),\Delta^{\varphi}$ pour tout $l$.
On en déduit que pour tout $l$, $f_{\sigma}(d_l)=\lv_{q_l}(\lv_{q_l})(\sigma_l)$. On remarque que l'on peut réécrire le second terme comme $\lv_P(\lv_P(\sigma_l,q_l))$. Calculons maintenant la composition $\sd^2_P(\Lambda^{\varphi}_k)\to\Lambda^{\varphi}_k\to\Ex_P(\sd_P(\Lambda^{\varphi}_k))$.
Le morphisme $\Lambda^{\varphi}_k\to\Ex_P(\sd_P(\Lambda^{\varphi}_k))$ envoie un simplexe $\Lambda^{\varphi}_k$, $\gamma\colon \Delta^{\psi}\to\Lambda^{\varphi}_k$ sur le simplexe de $\Ex_P(\sd_P(\Lambda^{\varphi}_k))$, $\sd_P(\gamma)\colon\sd_P(\Delta^{\psi})\to\sd_P(\Lambda^{\varphi}_k)$.
En particulier, l'application $\sd_P(\gamma)$ est obtenue à partir de $\gamma_{|(\Delta^{\psi})_0} \colon (\Delta^{\psi})_0 \to  (\Lambda^{\varphi}_k)_0$ de la même façon que $h(\sigma)$ est obtenue à partir de $f_{\sigma}$.
D'après le calcul précédent, on a bien $\lv^2_P(\sigma)_{|(\Delta^{\psi})_0}=f_{\sigma}$.
\end{proof}

\subsection{Une description explicite de la structure de modèle}

Dans cette section, on montre le théorème suivant

\begin{theo}\label{TheoDescriptionExpliciteSSetP}
La catégorie des ensembles simpliciaux filtrés $\sS_P$, munie des classes de flèches suivantes, est une catégorie modèle propre et engendrée de façon cofibrante.
\begin{itemize}
\item Un morphisme $f\colon (X,\varphi_X)\to(Y,\varphi_Y)$ est une \textbf{cofibration} si $f\colon X\to Y$ est un monomorphisme.
\item Un morphisme est une \textbf{fibration triviale} s'il a la propriété de relèvement à droite par rapport aux cofibrations.
\item Un morphisme est une \textbf{cofibration triviale} s'il est dans $l(r(\A))$, où $\A$ est l'ensemble des inclusions de cornets admissibles.
\item un morphisme est une \textbf{fibration} s'il a la propriété de relèvement à droite par rapport aux inclusions de cornets admissibles,
\item un morphisme $f\colon\fil{X}\to\fil{Y}$ est une \textbf{équivalence faible} si $\Exi_P(f)\colon\Exi_P\fil{X}\to\Exi_P\fil{Y}$ est une \textbf{équivalence d'homotopie filtrée}.
\end{itemize}
De plus, les classes de morphismes ainsi obtenues coïncident avec celles de la définition \ref{DefClassesSSP} et du théorème \ref{ExistenceCMFCisinski}.
\end{theo}

\begin{proof}
Nous allons d'abord montrer que les classes de la définition \ref{DefClassesSSP} coïncident avec celles définies ici. Cela revient à montrer que la classe des fibrations et celle des fibrations naïves coïncident, ou de façon équivalente, que la classe des extensions anodines et celle des cofibrations triviales coïncident. Ceci est une conséquence du théorème \ref{IdentificationFibrationsNaives} appliqué à la catégorie $\sS_P=\widehat{\Delta(P)}$, avec les foncteurs $S=\sd_P$ et $E=\Ex_P$, la transformation naturelle $\lv_P$, ainsi que le cylindre correspondant à la définition \ref{DefHomotopiesFiltrees}, et la classe d'extension anodine $\An=\Lambda$, munie de l'ensemble des inclusions de cornets admissibles comme ensemble générateur. (voir proposition \ref{DefExtensionsAnodinesSSP}). Nous avons déjà montré que :
\begin{itemize}
\item Le foncteur $\sd_P$ préserve les cofibrations et les extensions anodines (Proposition \ref{SDPreserveExtensionsAnodines}).
\item Pour tout ensemble simplicial filtré, $\fil{X}$, $\Exi_P\fil{X}$ est fibrant (Lemme \ref{ExInfiniFibrant}).
\end{itemize}
Il reste à montrer que $\Delta(P)$ est une catégorie de Eilenberg-Zilber - c'est une conséquence de \cite[Example 1.3.3]{Cisinski2} et du fait que $\Delta$ est une catégorie d'Eilenberg-Zilber - et que pour tout $\Delta^{\varphi}$, $\lv_P\colon \sd_P(\Delta^{\varphi})\to\Delta^{\varphi}$ est une équivalence faible absolue, ce qui est une conséquence du lemme \ref{SDSimplexeEquivalenceFaibleAbsolue}. Ceci prouve la caractérisation des fibrations et cofibrations (triviales). Soit $f\colon\fil{X}\to\fil{Y}$ un morphisme. On a le diagramme commutatif suivant
\begin{equation*}
\begin{tikzcd}
\fil{X}
\arrow{r}{f}
\arrow{d}
&\fil{Y}
\arrow{d}
\\
\Exi_P(X)
\arrow{r}{\Exi_P(f)}
&\Exi_P(Y)
\end{tikzcd}
\end{equation*}
Par le théorème \ref{TheoremeExXFaiblementEquivalentaX} les morphismes $\fil{X}\to\Exi_P(X)$ et $\fil{Y}\to\Exi_P(Y)$ sont des équivalences faibles. Ainsi, par deux sur trois, $f$ est une équivalence faible si et seulement si $\Exi_P(f)$ est une équivalence faible. Mais, par le lemme \ref{ExInfiniFibrant}, $\Exi_P(X)$ et $\Exi_P(Y)$ sont fibrants. En particulier, $\Exi_P(f)$ est une équivalence faible si et seulement si c'est une équivalence d'homotopie. Finalement, la propreté de la structure de modèle est le contenu du lemme \ref{PropreteSSetP}.
\end{proof}

\begin{lemme}\label{SDSimplexeEquivalenceFaibleAbsolue}
Soit $\Delta^{\psi}$ un simplexe filtré de $N(P)$ et $\overline{\Delta^{\psi}}$ l'unique simplexe non dégénéré dont $\Delta^{\psi}$ est une dégénérescence. Alors la composition
\begin{equation*}
\sd_P(\Delta^{\psi})\xrightarrow{\lv_P}\Delta^{\psi}\to\overline{\Delta^{\psi}}
\end{equation*}
est une équivalence d'homotopie dont l'inverse est donné par  l'extension anodine 
\begin{align*}
\overline{\Delta^{\psi}}&\to\sd_P(\Delta^{\psi})\\
[p_0,\dots,p_n]&\mapsto[(\Delta^{\psi},p_0),\dots,(\Delta^{\psi},p_n)]
\end{align*}
En particulier, $\sd_P(\Delta^{\psi})\to\Delta^{\psi}$ est une équivalence faible absolue au sens de \cite[Définition 1.3.55]{Cisinski}. 
\end{lemme}

\begin{proof}
Comme $\overline{\Delta^{\psi}}\subseteq N(P)$, et que les morphismes considérés sont filtrés, il est clair que la composition $\overline{\Delta^{\psi}}\to\sd_P(\Delta^{\psi})\to\overline{\Delta^{\psi}}$ est égale à l'identité. Il suffit donc de construire une homotopie 
\begin{equation*}
H\colon\Delta^1\otimes\sd_P(\Delta^{\psi})\to\sd_P(\Delta^{\psi})
\end{equation*}
telle que $H_{|\{0\}\otimes\sd_P(\Delta^{\psi})}=\Id_{\sd_P(\Delta^{\psi})}$, et telle que $H_{|\{1\}\otimes\sd_P(\Delta^{\psi})}$ est égale à la composée $\sd_P(\Delta^{\psi})\to\overline{\Delta^{\psi}}\to\sd_P(\Delta^{\psi})$. Comme l'ensemble simplicial sous jacent à $\Delta^1\otimes\sd_P(\Delta^{\psi})$ provient d'un complexe simplicial orienté, il suffit de définir $H$ sur les sommets et de vérifier que $H$ envoie tous les simplexes de $\Delta^{1}\otimes\sd_P(\Delta^{\psi})$ sur des simplexes de $\sd_P(\Delta^{\psi})$. On pose :
\begin{align*}
H\colon (\Delta^1\otimes\sd_P(\Delta^{\psi}))_0&\to (\sd_P(\Delta^{\psi}))_0\\
(\epsilon, (\sigma,p))&\mapsto\left\{\begin{array}{cl}
(\sigma,p) & \text{ si $\epsilon=0$}\\
(\Delta^{\psi},p) & \text{ si $\epsilon=1$}
\end{array}\right.
\end{align*}
Comme les restrictions de $H$ en $0$ et $1$ sont des applications simpliciales, il suffit de vérifier que $H([(0,(\sigma_0,p_0)),\dots,(0,(\sigma_k,p_k)),(1,(\sigma_{k+1},p_{k+1})),\dots,(1,(\sigma_n,p_n)])$ est un simplexe de $\sd_P(\Delta^{\psi})$ pour tout simplexe de $\sd_P(\Delta^{\psi})$, $\sigma=[(\sigma_0,p_0),\dots,(\sigma_n,p_n)]$ et pour tout $0<k<n$. On calcule
\begin{align*}
&H([(0,(\sigma_0,p_0)),\dots,(0,(\sigma_k,p_k)),(1,(\sigma_{k+1},p_{k+1})),\dots,(1,(\sigma_n,p_n)])\\
&= [(\sigma_0,p_0),\dots,(\sigma_k,p_k),(\Delta^{\psi},p_{k+1}),\dots,(\Delta^{\psi},p_n)]
\end{align*}
Comme par hypothèse, $\sigma$ est un simplexe de $\sd_P(\Delta^{\psi})$, on a $\sigma_0\subseteq \dots\subseteq \sigma_k\subseteq \Delta^{\psi}$. De plus,  $\sigma_0$ contient un sommet de couleur $p_i$ pour tout $0\leq i\leq n$, donc $H(\sigma)\in \sd_P(\Delta^{\psi})$. Et on a l'homotopie souhaitée. 

En particulier, le morphisme $\overline{\Delta^{\psi}}\to\sd_P(\Delta^{\psi})$ est un rétracte par déformation fort (Voir \cite[Définition 1.3.25]{Cisinski}), et donc une extension anodine \cite[Lemme 1.3.38]{Cisinski}. 
On en déduit que c'est en particulier une équivalence faible absolue \cite[Définition 1.3.55]{Cisinski}. 
Par deux sur trois le morphisme $\sd_P(\Delta^{\psi})\to\overline{\Delta^{\psi}}$ est aussi une équivalence faible absolue. 
Montrons que $\Delta^{\psi}\to\overline{\Delta^{\psi}}$ est une équivalence faible absolue. 
Notons $\Delta^{\psi}=(\Delta^N,\psi)$ avec 
\begin{equation*}
\Delta^N= [e^{p_0}_0,\dots,e^{p_0}_{j_{p_0}},e^{p_1}_0,
\dots ,e^{p_1}_{j_{p_1}},\dots,e^{p_n}_{j_{p_n}}] 
\end{equation*}
 et $\psi(e^{p}_k)=p$. 
Alors, on a la section suivante
\begin{align*}
\overline{\Delta^{\psi}}&\to\Delta^{\psi}\\  
{[}p_0,\dots,p_n]&\mapsto [e^{p_0}_0,e^{p_1}_0,\dots,e^{p_n}_0].
\end{align*}
Par construction, la composition $\overline{\Delta^{\psi}}\to\Delta^{\psi}\to\overline{\Delta^{\psi}}$ est l'identité de $\overline{\Delta^{\psi}}$, et on construit une homotopie $\Delta^1\otimes\Delta^{\psi}\to \Delta^{\psi}$ entre $\Id_{\Delta^{\psi}}$ et la composition inverse aisément. On en déduit que $\overline{\Delta^{\psi}}\to\Delta^{\psi}$ est une extension anodine, donc une équivalence faible absolue. 
Considérons maintenant le diagramme commutatif suivant
\begin{equation*}
\begin{tikzcd}
\sd_P(\overline{\Delta^{\psi}})
\arrow{r}{\lv_P}
\arrow{d}
&
\overline{\Delta^{\psi}}
\arrow{d}
\\
\sd_P(\Delta^{\psi})
\arrow{r}{\lv_P}
&\Delta^{\psi}
\end{tikzcd}
\end{equation*}
On a montré que le morphisme du haut et celui de droite sont des équivalences faibles absolues. De plus, par le lemme \ref{SDPreserveExtensionsAnodines}, le morphisme de gauche est une extension anodine, donc une équivalence faible absolue. On en déduit que le morphisme $\lv_P\colon \sd_P(\Delta^{\psi})\to\Delta^{\psi}$ est une équivalence faible absolue.
\end{proof}

\begin{lemme}\label{PropreteSSetP}
La structure de modèle sur $\sS_P$ des théorèmes \ref{ExistenceCMFCisinski} et \ref{TheoDescriptionExpliciteSSetP} est propre.
\end{lemme}

\begin{proof}
Nous avons défini la classe des extensions anodines comme la classe de flèche générée par $\mathcal{A}$, l'ensemble des inclusions de cornets admissibles. (i.e., la classe des extensions anodines est définie comme $l(r(\mathcal{A}))$.)
Nous allons montrer que la classe des extensions anodines est aussi générée par $\Lambda_{F(\Delta^1)}(\emptyset,\mathcal{M})$ au sens de la construction \cite[1.3.12]{Cisinski}. On pourra ensuite appliquer \cite[Corollaire 1.5.7]{Cisinski} grâce à \cite[Corollaire 1.4.18]{Cisinski} (voir aussi  \cite[Proposition 1.4.2]{Cisinski}) pour obtenir la propreté. 

Notons $\Lambda=\Lambda_{F(\Delta^1)}(\emptyset,\mathcal{M})$ pour simplifier les notations. Montrons d'abord que $\mathcal{A}\subseteq l(r(\Lambda))$. Comme 
\begin{equation*}
\{\partial(\Delta^{\varphi})\to\Delta^{\varphi}\ |\ \Delta^{\varphi}\in \Delta(P)\}
\end{equation*}
est un modèle cellulaire de $\sS_P$, par \cite[1.3.12]{Cisinski}, on a $\Lambda=\cup_{i\in \N}\Lambda^i$, avec 
\begin{equation*}
\Lambda^0=(\Delta^1\otimes\partial(\Delta^{\varphi})\cup\{\epsilon\}\otimes\Delta^{\varphi}\to\Delta^1\otimes\Delta^{\varphi}\ |\ \Delta^{\varphi}\in \Delta(P), \epsilon=0,1\}
\end{equation*}
et
\begin{equation*}
\Lambda^{i+1}=\{\Delta^1\otimes \fil{X}\cup \partial(\Delta^1)\otimes \fil{Y}\ |\ \fil{X}\to\fil{Y}\in \Lambda^i\}
\end{equation*}
On remarque que $\Lambda^0=\B$ où $\B$ est la classe de morphismes définie dans le lemme \ref{SatureeABC}. De plus, par la proposition \ref{SatureLifting}, on a $l(r(\B))=\Sat(\B)$. On déduit donc du lemme \ref{SatureeABC} que $\A\subseteq l(r(\B))=\l(r(\Lambda^0))\subseteq l(r(\Lambda))$. Réciproquement, par le lemme \ref{SatureeABC}, on obtient que $l(r(\B))=l(r(\A))$ est une classe d'extensions anodines. En particulier, par l'axiome (An2) de la définition \ref{DefinitionExtensionAnodine}, pour tout $\fil{X}\to\fil{Y} \in l(r(\B))$, le morphisme $\Delta^1\otimes\fil{X}\cup\{\epsilon\}\otimes\fil{Y}\to\Delta^1\otimes\fil{Y}$ est dans $l(r(\B))$. Comme $\B=\Lambda^0$, on en déduit que $\Lambda\subseteq l(r(\B))=l(r(\A))$.
\end{proof}

\begin{remarque}\label{RemarqueSSUPUniverselle}
La preuve du lemme précédent repose sur le constat suivant : la classe des équivalences faibles filtrées est le plus petit localisateur de $\sS_P$ contenant les équivalences d'homotopie filtrées (on renvoie à \cite[Définition 1.4.1]{Cisinski} pour la définition d'un localisateur). Ceci implique en particulier que toute structure de modèle sur $\sS_P$ ayant les monomorphismes pour cofibrations, et telle que les équivalences d'homotopie filtrées y sont des équivalences faibles est une localisation de la structure de modèle du théorème \ref{TheoDescriptionExpliciteSSetP}. En ce sens, la structure de modèle qu'on considère ici est universelle parmi les structures de modèle sur $\sS_P$. Lorsque plusieurs structures de modèle apparaitront simultanément, on notera $\sSU_P$ la structure provenant du théorème \ref{TheoDescriptionExpliciteSSetP}.
\end{remarque}
\chapter{Diagrammes et groupes d'homotopie filtrés}
\label{ChapitreStructureSimplicialeSSetP}
On a vu au chapitre précédent que la catégorie des ensembles simpliciaux filtrés $\sS_P$ pouvait être munie d'une structure modèle. On voit dans ce chapitre que cette catégorie modèle est simpliciale, et que les équivalences faibles entres objets fibrants y sont caractérisées par des invariants algébriques : les groupes d'homotopie filtrés. 

Dans la première section de ce chapitre, on définit une structure de catégorie simpliciale sur $\sS_P$ et on montre qu'elle fait de $\sS_P$ une catégorie modèle simpliciale. 

Dans la section \ref{SectionCategorieModeleDiagrammesSimpliciaux}, on définit la catégorie modèle des diagrammes simpliciaux. Il s'agit d'une catégorie de foncteurs, nécessaire pour la définition des groupes d'homotopie filtrés, et qui sera utile tout au long de ce texte. 
On définit finalement les groupes d'homotopies filtrés dans la section \ref{SectionGroupesHomotopieFiltres} (Définitions \ref{DefinSPi0Simplicial} et \ref{DefinitionGroupesHomotopieFiltree}), où on montre aussi qu'ils caractérisent les équivalences faibles entre objets fibrants. C'est le contenu du Théorème \ref{EquivalenceFaibleIsoGroupeHomotopie}. On montre ensuite comment simplifier le calcul des groupes d'homotopie filtrés.

Enfin, dans la section \ref{SectionRaffinementStratification}, on introduit une méthode pour raffiner la filtration d'un ensemble simplicial filtré de façon à ce que toutes ses strates soient connexes.

\section{La catégorie $\sS_P$ est une catégorie modèle simpliciale}

On rappelle la définition \ref{TenseurSimplicial}.

\begin{defin}
Soit $\fil{X}$ un ensemble simplicial filtré, $K$ un ensemble simplicial. On définit $K\otimes \fil{X}$ comme le produit filtré $F(K)\times_{N(P)}\fil{X}$. Cette définition donne lieu à un bifoncteur :
\begin{equation*}
-\otimes-\colon \sS\times \sS_P\to \sS_P
\end{equation*}
\end{defin}

\begin{defin}
On définit le bifoncteur $(-)^{-}\colon \sS_P\times\sS\to\sS_P$ comme suit 
\begin{align*}
\sS_P\times\sS^{\op}&\to\sS_P\\
(\fil{Y},K)&\mapsto \fil{Y}^K\left\{\begin{array}{rcl}
\Delta(P)^{\op}&\to &\Set\\
\Delta^{\varphi}&\mapsto&\Hom_{\sS_P}(K\otimes\Delta^{\varphi},\fil{Y})
\end{array}\right.
\end{align*}
\end{defin}
Par construction, la propriété suivante est vérifiée. 

\begin{prop}\label{AdjonctionPreuveCatSimpliciale}
Soit $K$ un ensemble simplicial. Le foncteur $(-)^K\colon \sS_P\to \sS_P$ est adjoint à droite au foncteur $K\otimes-\colon \sS_P\to \sS_P$.
\end{prop}

\begin{proof}
Soient $\fil{X}$ et $\fil{Y}$ deux ensembles simpliciaux filtrés, et $K$ un ensemble simplicial. On construit les morphismes d'adjonctions 
\begin{align*}
\Hom(K\otimes\fil{X},\fil{Y})&\to\Hom(\fil{X},\fil{Y}^K)\\
\left( f\colon K\otimes \fil{X}\to \fil{Y}\right) &\mapsto \left\{\begin{array}{ccc}
X&\to &\fil{Y}^K\\
(\sigma\colon \Delta^{\varphi}\to \fil{X}) &\mapsto& f\circ (\Id_K\otimes\sigma)\colon K\otimes\Delta^{\varphi}\to\fil{Y} 
\end{array}\right.
\end{align*}
et
\begin{align*}
\Hom(\fil{X},\fil{Y}^K)&\to\Hom(K\otimes\fil{X},\fil{Y})\\
g\colon \fil{X}\to\fil{Y}^K&\mapsto \left\{\begin{array}{ccc}
K\otimes \fil{X}&\to&\fil{Y}\\
(\sigma\colon \Delta^{\varphi}\to K\otimes \fil{X})&\mapsto & g\circ \pr_X (\sigma)(\pr_{F(K)}\times \Id_{\Delta^{\varphi}})
\end{array}\right.
\end{align*}
Ce sont des isomorphismes inverses l'un de l'autre, ce qui complète la preuve.
\end{proof}

\begin{defin}
On définit le bifoncteur $\Map(-,-)\colon \sS_P^{\op}\times\sS_P\to\sS$ comme suit :
\begin{align*}
\sS_P^{\op}\times\sS_P&\to\sS\\
(\fil{X},\fil{Y})&\mapsto \Map(\fil{X},\fil{Y}) \left\{\begin{array}{rcl}
\Delta^{\op}&\to &\Set\\
\Delta^{n}&\mapsto&\Hom_{\sS_P}(\Delta^n\otimes \fil{X},\fil{Y})
\end{array}\right.
\end{align*}
\end{defin}

Alors, on a la proposition suivante

\begin{prop}
Les foncteurs $\Map$, $(-)^{-}$ et $-\otimes-$ munissent $\sS_P$ d'une structure de catégorie simpliciale.
\end{prop}

\begin{proof}
D'après \cite[Lemme II.2.4]{GoerssJardine}, il suffit de vérifier les conditions suivantes 
\begin{itemize}
\item Pour tout ensemble simplicial $K$, $K\otimes-$ est adjoint à gauche de $(-)^K$. C'est le contenu de la proposition \ref{AdjonctionPreuveCatSimpliciale}.
\item Pour tout ensemble simplicial filtré $\fil{X}$, le foncteur $-\otimes \fil{X}$ commute avec les colimites et vérifie $\Delta^0\otimes \fil{X}\simeq X$. La première affirmation est vraie car $-\otimes \fil{X}=(-\times X,\varphi_X\circ\pr_X)$, et $-\times X$ commute avec les colimites. La seconde affirmation provient du fait que $F(\Delta^0)\simeq N(P)$ et de la définition de $\Delta^0\otimes \fil{X}$ comme $F(\Delta^0)\times_{N(P)} \fil{X}\simeq \fil{X}$.
\item Pour tous ensembles simpliciaux $K,L$ et tout ensemble simplicial filtré $\fil{X}$, il existe un isomorphisme naturel entre $K\otimes(L\otimes \fil{X})$ et $(K\times L)\otimes X$. Cet isomorphisme provient de la suite d'isomorphismes suivante 
\begin{align*}
K\otimes (L\otimes \fil{X})& \simeq F(K)\times_{N(P)}(L\otimes \fil{X})\\
&\simeq F(K)\times_{N(P)}(F(L)\times_{N(P)}\fil{X})\\
&\simeq (F(K)\times_{N(P)}F(L))\times_{N(P)}\fil{X}\\
&\simeq (F(K\times L))\times_{N(P)}\fil{X}\\
&\simeq (K\times L)\otimes \fil{X}
\end{align*}
où le quatrième isomorphisme provient du fait que 
\begin{align*}
F(K)\times_{N(P)}F(L)&\simeq (K\times N(P),\pr_{N(P)})\times_{N(P)}(L\times N(P),\pr_{N(P)})\\
&\simeq (K\times L\times N(P),\pr_{N(P)})\\
&\simeq F(K\times L).
\end{align*}
\end{itemize}
\end{proof}

\begin{theo}\label{CategorieModelSimpliciale}
La catégorie modèle $\sS_P$, munie des foncteurs $\Map$, $-\otimes-$ et $(-)^{-}$ est une catégorie modèle simpliciale.
\end{theo}
On renvoie à \cite[Section II.3]{GoerssJardine} pour la définition d'une catégorie modèle simpliciale.

\begin{proof}
D'après \cite[Corollaire II.3.12]{GoerssJardine}, il suffit de prouver les affirmations suivantes :
\begin{itemize}
\item Soit $f\colon \fil{X}\to\fil{Y}$ une cofibration entre ensembles simpliciaux filtrés, et $n\geq 0$. Alors l'application 
\begin{equation*}
\Delta^n\otimes \fil{X}\cup_{\partial(\Delta^n)\otimes\fil{X}}\partial(\Delta^n)\otimes \fil{Y}\to\Delta^n\otimes\fil{Y}
\end{equation*}
est une cofibration, qui est triviale si $f$ est une cofibration triviale.
\item Soit $f\colon \fil{X}\to\fil{Y}$ une cofibration entre ensembles simpliciaux filtrés, alors les applications
\begin{equation*}
\Delta^1\otimes \fil{X}\cup_{\{\epsilon\}\otimes \fil{X}}\{\epsilon\}\otimes\fil{Y}\to\Delta^1\otimes\fil{Y},
\end{equation*}
où $\epsilon=0$ ou $1$ est un sommet de $\Delta^1$, sont des cofibrations triviales.
\end{itemize}
La première affirmation, dans le cas où $f$ n'est pas triviale vient du fait que les cofibrations de $\sS_P$ sont les monomorphismes, et que les opérations impliquées préservent les monomorphismes. Dans le cas où $f$ est triviale c'est une extension anodine, par le théorème \ref{TheoDescriptionExpliciteSSetP}, et on obtient le résultat en appliquant le lemme \ref{LemmeAn2} avec $\fil{Z}\to\fil{W}=F(\partial(\Delta^n))\to F(\Delta^n)$. Par l'axiome (An1) de la définition des classes d'extensions anodines (définition \ref{DefinitionExtensionAnodine}), tous les morphismes apparaissant dans la deuxième affirmation sont des extensions anodines, et donc des cofibrations triviales.
\end{proof}

\section{Diagrammes d'ensembles simpliciaux}
\label{SectionCategorieModeleDiagrammesSimpliciaux}
\begin{defin}\label{DefCategorieDiagrammesSimpliciaux}
On définit la catégorie des diagrammes simpliciaux 
\begin{equation*}
\Diag_P=\Fun(\Delta(P)^{\op},\sS)
\end{equation*}
\end{defin}

\begin{prop}[{\cite[Théorème 11.6.1]{Hirschhorn}}]\label{CategorieModeleDiagramme}
La catégorie $\Diag_P$, est une catégorie modèle où :
\begin{itemize}
\item un morphisme $f\colon F\to G$ est une fibration si $f_{\Delta^{\varphi}}\colon F(\Delta^{\varphi})\to G(\Delta^{\varphi})$ est une fibration de Kan, pour tout $\Delta^{\varphi}\in \Delta(P)$,
\item un morphisme $f\colon F\to G$ est une équivalence faible si $f_{\Delta^{\varphi}}\colon F(\Delta^{\varphi})\to G(\Delta^{\varphi})$ est une équivalence faible pour la structure de Kan sur $\sS$, pour tout $\Delta^{\varphi}\in \Delta(P)$.
\end{itemize}
\end{prop}

\begin{defin}\label{DefFoncteurDiagrammeSimplicial}
On définit le foncteur
\begin{align*}
D\colon \sS_P&\to  \Diag_P\\
\fil{X}&\mapsto (\Delta^{\varphi}\mapsto \Map(\Delta^{\varphi},\fil{X})
\end{align*}
Par abus de notation, on notera $D(X)$ pour l'image par $D$ de $\fil{X}$ lorsqu'il n'y aura pas d'ambiguïté sur la filtration choisie sur $X$.
\end{defin}

\begin{prop}\label{DiagrammeCaracteriseFibrationTriviale}
Soit $f\colon \fil{X}\to\fil{Y}$ une fibration (triviale). Alors, $D(f)$ est une fibration (triviale). De plus, si $\fil{X}$ et $\fil{Y}$ sont fibrants et $D(f)$ est triviale, alors $f$ est triviale.
\end{prop}

La preuve utilise le lemme suivant .

\begin{lemme}\label{MapBordsSimplexes}
Soit $f\colon \fil{X}\to\fil{Y}$ une fibration. Alors $f$ est une fibration triviale si et seulement si, pour tout $\Delta^{\varphi}\in \Delta(P)$, les applications induites par $f$ 
\begin{align*}
\Map(\partial(\Delta^{\varphi}),\fil{X})&\to\Map(\partial(\Delta^{\varphi}),\fil{Y})\\
\Map(\Delta^{\varphi},\fil{X})&\to \Map(\Delta^{\varphi},\fil{Y})
\end{align*}
sont des équivalences faibles pour la structure de Kan sur $\sS$.
\end{lemme}

\begin{proof}
Par définition d'une structure de modèle simpliciale, si $f$ est une fibration triviale, les applications simpliciales du lemme \ref{MapBordsSimplexes} sont des fibrations triviales dans $\sS$ pour la structure de Kan.
Pour la réciproque, fixons $\Delta^{\varphi}$ un simplexe filtré et considérons le diagramme commutatif suivant, dont le carré central est cocartésien.
\begin{equation*}
\begin{tikzcd}
\Map(\Delta^{\varphi},\fil{X})
\arrow{dr}{(4)}
\arrow[swap, bend right = 18]{ddr}{(2)}
\arrow[bend left = 6]{drr}
&\phantom{X}
&\phantom{X}
\\
\phantom{X}
&\Map(\Delta^{\varphi},\fil{Y})\times
\Map(\partial(\Delta^{\varphi}),\fil{X})
\arrow{r}
\arrow{d}{(3)}
&\Map(\partial(\Delta^{\varphi}),\fil{X})
\arrow{d}{(1)}
\\
\phantom{X}
&\Map(\Delta^{\varphi},\fil{Y})
\arrow{r}
&\Map(\partial(\Delta^{\varphi}),\fil{Y})
\end{tikzcd}
\end{equation*}
Comme $\sS_P$ est une catégorie modèle simpliciale, on sait que les morphismes $(1)$ et $(2)$ sont des fibrations. Comme par hypothèse ce sont aussi des équivalences faibles, ce sont des fibrations triviales. Comme $(3)$ est l'image de $(1)$ par un produit fibré, $(3)$ est aussi une fibration triviale. On en déduit que $(4)$ est une équivalence faible, par l'axiome de 2 sur 3. C'est de plus une fibration, car $\sS_P$ est une catégorie modèle simpliciale. On en déduit que $(4)$ est en particulier une surjection.
Soit maintenant un problème de relèvement pour $f$ 
\begin{equation*}
\begin{tikzcd}
\partial(\Delta^{\varphi})
\arrow{r}
\arrow{d}
&\fil{X}
\arrow{d}{f}
\\
\Delta^{\varphi}
\arrow{r}
\arrow[dashed]{ur}{g}
&\fil{Y}
\end{tikzcd}
\end{equation*}
Un tel diagramme correspond à un $0$-simplexe de $\Map(\Delta^{\varphi},\fil{Y})\times
_{\Map(\partial(\Delta^{\varphi}),\fil{Y})}
\Map(\partial(\Delta^{\varphi}),\fil{X})$. Une pré-image de
ce $0$-simplexe par l'application $(4)$ fournit un relèvement $g$. On en déduit que $f$ admet la propriété de relèvement à droite par rapport à toutes les inclusions $\partial(\Delta^{\varphi})\to\Delta^{\varphi}$, et donc que $f$ est une fibration triviale.
\end{proof}

\begin{proof}[Démonstration de la proposition \ref{DiagrammeCaracteriseFibrationTriviale}]
Soit $f\colon \fil{X}\to\fil{Y}$ une fibration.
Pour tout $\Delta^{\varphi}\in \Delta(P)$, $D(f)(\Delta^{\varphi})\colon \Map(\Delta^{\varphi},\fil{X})\to\Map(\Delta^{\varphi},\fil{Y})$ est une fibration car $\sS_P$ est une catégorie modèle simpliciale. On en déduit que $D(f)$ est une fibration de $\Diag_P$. De même, $D(f)$ est une fibration triviale si $f$ est une fibration triviale. Supposons maintenant que $D(f)$ est une fibration triviale. Alors, on sait que les applications $\Map(\Delta^{\varphi},\fil{X})\to\Map(\Delta^{\varphi},\fil{Y})$ sont des équivalences faibles pour tout $\Delta^{\varphi}\in \Delta(P)$, et, par le lemme \ref{MapBordsSimplexes} il suffit de montrer que les applications $\Map(\partial(\Delta^{\varphi}),\fil{X})\to\Map(\partial(\Delta^{\varphi}),\fil{Y})$ sont des équivalences faibles pour tout $\Delta^{\varphi}\in \Delta(P)$.
Fixons un simplexe filtré $\Delta^{\varphi}\in \Delta(P)$. Comme pour tout ensemble simplicial filtré, on peut écrire $\partial(\Delta^{\varphi})$ comme la colimite 
\begin{equation*}
\partial(\Delta^{\varphi})\simeq \colim_{\Delta^{\psi}\subset \partial(\Delta^{\varphi})}\Delta^{\psi}
\end{equation*}
De plus, comme toutes les faces de simplexes non dégénérés de $\Delta^{\varphi}$ sont des simplexes non dégénérés, on peut restreindre cette colimite aux inclusions de simplexes non dégénérés. En appliquant $\Map(-,\fil{X})$, il vient
\begin{align*}
\Map(\partial(\Delta^{\varphi}),\fil{X})&\simeq \Map\left(\colim_{\Delta^{\psi}\subset\partial(\Delta^{\varphi})}\Delta^{\psi},\fil{X}\right)\\
&\simeq \lim_{\Delta^{\psi}\subset \partial(\Delta^{\varphi})}\Map(\Delta^{\psi},\fil{X})
\end{align*}
Finalement, si $\Delta^{\psi_1}\subset\Delta^{\psi_2}\subset\Delta^{\varphi}$ sont des inclusions de simplexes filtrés, le morphisme
$\Map(\Delta^{\psi_2},\fil{X})\to \Map(\Delta^{\psi_1},\fil{X})$ est une fibration, car $\Delta^{\psi_1}\to\Delta^{\psi_2}$ est une cofibration et $\fil{X}$ est fibrant. En particulier, la limite précédente est une limite de fibration. Finalement, on a
\begin{align*}
\Map(\partial(\Delta^{\varphi}),\fil{X})&\simeq \lim\Map(\Delta^{\psi},\fil{X})\\
&\sim \holim\Map(\Delta^{\psi},\fil{X})\\
&\sim \holim\Map(\Delta^{\psi},\fil{Y})\\
&\sim \lim\Map(\Delta^{\psi},\fil{Y})\\
&\simeq \Map(\partial(\Delta^{\varphi}),\fil{Y})
\end{align*}
et le lemme \ref{MapBordsSimplexes} permet de conclure.
\end{proof}

\begin{corollaire}\label{DiagAdjonctionQuillen}
Le foncteur $D$ est un foncteur de Quillen à droite.
\end{corollaire}

\begin{proof}
Il suffit de montrer que $D$ admet un adjoint à gauche. En effet, la première partie de la proposition \ref{DiagrammeCaracteriseFibrationTriviale} garantit que si une paire d'adjoints existe, il s'agira d'une paire de Quillen.
On définit l'adjoint à gauche $\Colim\colon \Diag_P\to \sS_P$ comme suit :
\begin{align*}
\Diag_P&\xrightarrow{\Colim} \sS_P\\
F&\mapsto \colim_{\Delta^{\psi}\to\Delta^{\varphi}} F(\Delta^\varphi)\otimes\Delta^{\psi}.
\end{align*}
Ici, la colimite est prise sur la petite catégorie des morphismes de $\Delta(P)$. 
On vérifie qu'on a bien une adjonction. Soit $F$ un objet de $\Diag_P$, et $\fil{X}$ un ensemble simplicial filtré. On a la bijection
\begin{align*}
\Hom_{\Diag_P}(F,D(X))&\to\Hom_{\sS_P}(\Colim(F),\fil{X})\\
(f_{\varphi}\colon F(\Delta^{\varphi})\to\Map(\Delta^{\varphi},\fil{X})_{\Delta^{\varphi}\in \Delta(P)}&\mapsto \colim f_{\varphi,\psi}\colon \colim F(\Delta^{\varphi})\otimes\Delta^{\psi}\to\fil{X},
\end{align*}
où  $f_{\varphi,\psi}$ est défini comme le morphisme adjoint à la composition 
\begin{equation*}
F(\Delta^{\varphi})\xrightarrow{f_{\varphi}} \Map(\Delta^{\varphi},\fil{X})\to\Map(\delta^{\psi},\fil{X})
\end{equation*}
par l'adjonction $(\Map(\Delta^{\psi},-),-\otimes\Delta^{\psi})$. Et la bijection inverse 
\begin{align*}
\Hom_{\sS_P}(\Colim(F),\fil{X})&\to \Hom_{\Diag_P}(F,D(X))\\
g\colon \Colim(F)\to \fil{X}&\mapsto (g_{\varphi}\colon F(\Delta^{\varphi})\to \Map(\Delta^{\varphi},\fil{X}))_{\Delta^{\varphi}\in \Delta(P)}
\end{align*}
où $g_{\varphi}$ est défini comme l'adjoint de la composition 
\begin{equation*}
F(\Delta^{\varphi})\otimes\Delta^{\varphi}\to\Colim(F)\xrightarrow{g}\to \fil{X}.
\end{equation*}
\end{proof}

\begin{prop}\label{fEquivalenceFaibleSSIDfEquivalenceFaible}
Soit $f\colon \fil{X}\to \fil{Y}$ un morphisme entre objets fibrants de $\sS_P$. $f$ est une équivalence faible si et seulement si $D(f)$ est une équivalence faible.
\end{prop}

\begin{proof}
D'après le corollaire \ref{DiagAdjonctionQuillen}, $D$ est un foncteur de Quillen à droite. En particulier, $D$ préserve les équivalences faibles entre objets fibrants. Pour montrer la réciproque, soit $f\colon \fil{X}\to \fil{Y}$ un morphisme entre objets fibrants, tel que $D(f)$ est une équivalence faible. Alors, soit $f=pi$ une factorisation de $f$ en une cofibration triviale suivie d'une fibration. En appliquant $D$ à cette factorisation, on obtient 
\begin{equation*}
\begin{tikzcd}
D(X)
\arrow{r}{D(f)}
\arrow[swap]{d}{D(i)}
&D(Y)
\\
D(Z)
\arrow[swap]{ur}{D(p)}
\end{tikzcd}
\end{equation*}
Comme $\fil{Y}$ est fibrant et que $p\colon\fil{Z}\to \fil{Y}$ est une fibration, on en déduit que $Z$ est fibrant. Mais alors, $D(i)$ est l'image d'une équivalence faible entre objets fibrants, et c'est donc une équivalence faible. Par 2 sur 3, $D(p)$ est donc une fibration triviale. Comme $\fil{Z}$ et $\fil{Y}$ sont fibrants, on peut appliquer la proposition \ref{DiagrammeCaracteriseFibrationTriviale} pour conclure que $p$ est une fibration triviale, et donc que $f$ est une équivalence faible, car composition d'une fibration triviale et d'une cofibration triviale.
\end{proof}

\section{Groupes d'homotopie filtrés}
\label{SectionGroupesHomotopieFiltres}
\subsection{Pointage et composantes connexes filtrées}

\begin{defin}\label{DefinSPi0Simplicial}
Soit $\fil{X}$ un ensemble simplicial filtré fibrant. On définit son ensemble de composantes connexes filtrées comme l'ensemble simplicial filtré suivant 
\begin{align*}
s\pi_0(\fil{X})\colon \Delta(P)^{\op}&\to \Set\\
\Delta^{\varphi} &\mapsto \pi_0(\Map(\Delta^{\varphi},\fil{X}))
\end{align*}
\end{defin}

\begin{remarque}\label{SimplexesDeSPi0}
Soit $\fil{X}$ un ensemble simplicial filtré fibrant. Un simplexe de $s\pi_0\fil{X}$, $\sigma\colon \Delta^{\varphi}\to s\pi_0\fil{X}$ est une classe d'homotopie $[\Delta^{\varphi},\fil{X}]$. En effet, par construction l'ensemble $\pi_0(\Map(\Delta^{\varphi},\fil{X})$ est l'ensemble des applications filtrées $\Hom(\Delta^{\varphi},\fil{X})$ quotienté par la relation d'équivalence induite par l'ensemble $\Hom(\Delta^{1}\otimes\Delta^{\varphi},\fil{X})$. C'est donc l'ensemble des applications filtrées à homotopie filtrée près $[\Delta^{\varphi},\fil{X}]$.
\end{remarque}

\begin{defin}
Soit $\fil{X}$ un ensemble simplicial filtré fibrant. Un pointage de $\fil{X}$ est la donnée d'un sous ensemble simplicial filtré $\fil{V}\subseteq (N(P),\Id)$, et d'une application simpliciale filtrée
\begin{equation*}
\phi\colon \fil{V}\to \fil{X}.
\end{equation*}
\end{defin}

\begin{remarque}\label{PointageInduitSimplicial}
Soit $\fil{X}$ un ensemble simplicial filtré fibrant, et $\phi\colon \fil{V}\to\fil{X}$ un pointage de $X$. Si $\fil{V'}\subseteq \fil{V}$ est un sous ensemble simplicial filtré, alors $\phi$ induit un pointage 
\begin{equation*}
\phi_{|V'}\colon \fil{V'}\to\fil{X}
\end{equation*}
\end{remarque}

\begin{defin}
Soit $\fil{X}$ un ensemble simplicial filtré fibrant, et $(\phi_i ,\fil{V_i})_{i\in I}$ un ensemble de pointages de $\fil{X}$. On dit que l'ensemble de pointages $(\phi_i,\fil{V_i})_{i\in I}$ est complet si pour tout simplexe filtré non dégénéré $\sigma\colon \Delta^{\varphi}\to s\pi_0(\fil{X})$, il existe $i\in I$, tel que $\Delta^{\varphi}\subseteq V_i$ et $[(\phi_i)_{|\Delta^{\varphi}}]=\sigma$.
\end{defin}

\begin{prop}\label{PointagesCompletsExistent}
Soit $\fil{X}$ un ensemble simplicial filtré fibrant, il existe un ensemble complet de pointages de $\fil{X}$.
\end{prop}

\begin{proof}
Soit $\sigma\colon \Delta^{\varphi}\to s\pi_0\fil{X}$ un simplexe filtré non dégénéré. D'après la remarque \ref{SimplexesDeSPi0} $\sigma$ correspond à un élément dans l'ensemble des classes d'homotopies $[\Delta^{\varphi},\fil{X}]$. Un représentant de cette classe d'homotopie, $\phi_{\sigma}\colon \Delta^{\varphi}\to \fil{X}$ fournit un pointage de $\fil{X}$ vérifiant $[\phi_{\sigma}]=\sigma$. L'ensemble $(\Delta^{\varphi_{\sigma}},\phi_{\sigma})_{\sigma\in s\pi_0\fil{X}_{\nd}}$ donne un ensemble de pointages complet pour $\fil{X}$.
\end{proof}

\begin{remarque}
L'ensemble de pointages construit ici contient plus de pointages que nécessaire. En effet, il suffit de considérer les pointages de la forme $\phi_{\sigma}$, où $\sigma$ est un simplexe maximal de $s\pi_0\fil{X}$.
\end{remarque}

\subsection{Caractérisation des équivalences faibles}

\begin{defin}\label{DefinitionGroupesHomotopieFiltree}
Soient $\fil{X}$ un ensemble simplicial filtré fibrant, $n\geq 1$ un entier et $\phi\colon\fil{V}\to\fil{X}$ un pointage de $\fil{X}$. On définit le $n$-ième groupe d'homotopie filtré de $\fil{X}$, comme l'ensemble simplicial filtré suivant
\begin{align*}
s\pi_n(\fil{X},\phi)\colon \Delta(P)^{\op}&\to\Set\\
\Delta^{\varphi}&\mapsto \left\{\begin{array}{cl}
\pi_n(\Map(\Delta^{\varphi},\fil{X}),\phi_{|\Delta^{\varphi}})&\text{ si $\Delta^{\varphi}\subseteq \fil{V}$}\\
\emptyset &\text{ si $\Delta^{\varphi}\not\subseteq \fil{V}$} 
\end{array}\right.
\end{align*}
\end{defin}

\begin{prop}\label{HomotopiePointagesGroupesHomotopie}
Soient $\phi,\psi\colon \fil{V}\to\fil{X}$ deux pointages d'un ensemble simplicial filtré fibrant. Si $\phi$ et $\psi$ sont homotopes par une homotopie filtrée $H$, il existe un isomorphisme naturel
\begin{equation*}
s\pi_n(H)\colon s\pi_n(\fil{X},\phi)\to s\pi_n(\fil{X},\psi)
\end{equation*}
De plus, si $f\colon\fil{X}\to\fil{Y}$ est une application filtrée entre ensembles simpliciaux fibrants, les homotopies $H$ et $f\circ H$ induisent le diagramme commutatif suivant :
\begin{equation*}
\begin{tikzcd}
s\pi_n(\fil{X},\phi)
\arrow{d}{s\pi_n(H)}
\arrow{r}{s\pi_n(f)}
&s\pi_n(\fil{Y},f\circ\phi)
\arrow{d}{s\pi_n(f\circ H)}
\\
s\pi_n(\fil{X},\psi)
\arrow{r}{s\pi_n(f)}
&s\pi_n(\fil{Y},f\circ\psi)
\end{tikzcd}
\end{equation*}
\end{prop}

\begin{proof}
Soient $H\colon \Delta^1\otimes\fil{V}\to \fil{X}$ l'homotopie filtrée entre $\phi$ et $\psi$, et soit $\Delta^{\varphi}\subseteq \fil{V}$ un simplexe filtré. On a le diagramme commutatif de morphismes d'ensembles simpliciaux pointés suivant. 
\begin{equation*}
\begin{tikzcd}[column sep=60]
(\Map(\Delta^{\varphi},\fil{X}),\phi_{|\Delta^{\varphi}})
\arrow{r}{\Map(\Delta^{\varphi},f)}
&(\Map(\Delta^{\varphi},\fil{Y}),f\circ\phi_{|\Delta^{\varphi}})
\\
(\Map(\Delta^1\otimes\Delta^{\varphi},\fil{X}),H_{|\Delta^1\otimes\Delta^{\varphi}})
\arrow{u}{\Map(\iota_0,\fil{X})}
\arrow{r}{\Map(\Delta^1\otimes\Delta^{\varphi},f)}
\arrow[swap]{d}{\Map(\iota_1,\fil{X})}
&(\Map(\Delta^1\otimes\Delta^{\varphi},\fil{Y}),f\circ H_{|\Delta^1\otimes\Delta^{\varphi}})
\arrow[swap]{u}{\Map(\iota_0,\fil{Y})}
\arrow{d}{\Map(\iota_1,\fil{Y})}
\\
(\Map(\Delta^{\varphi},\fil{X}),\psi_{|\Delta^{\varphi}})
\arrow{r}{\Map(\Delta^{\varphi},f)}
&(\Map(\Delta^{\varphi},\fil{Y}),f\circ\psi_{|\Delta^{\varphi}})
\end{tikzcd}
\end{equation*}
On remarque que les applications $\iota_i\colon\Delta^0\otimes\Delta^{\varphi}\to \Delta^1\otimes\Delta^{\varphi}$ sont des cofibrations triviales. En particulier, comme $\sS_P$ est une catégorie modèle simpliciale, et que $\fil{X}$ et $\fil{Y}$ sont fibrants, en appliquant le foncteur $\pi_n$ au diagramme précédent, tous les morphismes verticaux deviennent des isomorphismes. On
obtient ainsi l'isomorphisme souhaité en posant
\begin{equation*}
s\pi_n(H)_{\Delta^{\varphi}}=\pi_n(\iota_1)\circ (\pi_n(\iota_0))^{-1}
\end{equation*}
\end{proof}

\begin{remarque}
La proposition \ref{HomotopiePointagesGroupesHomotopie} implique que pour tout ensemble simplicial filtré $\fil{X}$, et tout pointage $\phi\colon\fil{V}\to\fil{X}$, $s\pi_1(\fil{X},\phi)$ agit naturellement sur $s\pi_n(\fil{X},\phi)$ pour tout $n\geq 1$. Cette action est à comprendre dans le sens suivant. Soit $\sigma\colon \Delta^{\varphi}\to s\pi_1(\fil{X},\phi)$ un simplexe, alors $\sigma$ correspond à une classe d'homotopie dans $[(\Delta^1,\partial(\Delta^1)),(\Map(\Delta^{\varphi},\fil{X}),\phi_{|\Delta^{\varphi}})]$. En particulier, $\sigma$ correspond à une homotopie entre $\phi$ et $\phi$. Par la proposition \ref{HomotopiePointagesGroupesHomotopie}, cette homotopie induit un isomorphisme $s\pi_n(\fil{X},\phi_{|\Delta^{\varphi}})\to s\pi_n(\fil{X},\phi_{|\Delta^{\varphi}})$. De plus, par construction, l'ensemble $s\pi_n(\fil{X},\phi)(\Delta^{\varphi})$ est un groupe dès qu'il est non-vide, et on obtient que la collection d'isomorphismes induits par les simplexes de $s\pi_1(\fil{X},\phi)(\Delta^{\varphi})$ sur $s\pi_n(\fil{X},\phi)(\Delta^{\varphi})$ est une action de groupe.
\end{remarque}

\begin{prop}\label{HomotopieMorphismesGroupesHomotopie}
Soient $f,g\colon \fil{X}\to\fil{Y}$ deux applications filtrées entre ensembles simpliciaux filtrés fibrants, $\phi\colon\fil{V}\to\fil{X}$ un pointage de $X$ et $H\colon \Delta^1\otimes\fil{X}\to \fil{Y}$ une homotopie filtrée entre $f$ et $g$. Alors, on a le diagramme commutatif suivant 
\begin{equation*}
\begin{tikzcd}
s\pi_n(\fil{X},\phi)
\arrow{r}{s\pi_n(f)}
\arrow[swap]{dr}{s\pi_n(g)}
&s\pi_n(\fil{Y},f\circ\phi)
\arrow{d}{s\pi_n(H\circ \phi)}
\\
\phantom{X}
&s\pi_n(\fil{Y},g\circ\phi)
\end{tikzcd}
\end{equation*} 
En particulier, si $f\circ \phi=g\circ \phi$, et $H\circ (\Delta^1\otimes\phi)$ est égal à la composition $\Delta^1\otimes\fil{V}\xrightarrow{\pr_V}\fil{V}\xrightarrow{f\circ\phi}\fil{Y}$, alors $s\pi_n(f)=s\pi_n(g)$.
\end{prop}

\begin{theo}\label{EquivalenceFaibleIsoGroupeHomotopie}
Soit $f\colon\fil{X}\to\fil{Y}$ un morphisme entre ensembles simpliciaux filtrés fibrants. Alors, $f$ est une équivalence faible si et seulement si $s\pi_0(f)\colon s\pi_0(\fil{X})\to s\pi_0(\fil{Y})$ est un isomorphisme et pour tout pointage $\phi\colon\fil{V}\to\fil{X}$, et pour tout entier $n\geq 1$, le morphisme
\begin{equation*}
s\pi_n(f)\colon s\pi_n(\fil{X},\phi)\to s\pi_n(\fil{Y},f\circ\phi)
\end{equation*}
est un isomorphisme.
\end{theo}

\begin{proof}
Comme $\fil{X}$ et $\fil{Y}$ sont fibrants, on sait par la proposition \ref{fEquivalenceFaibleSSIDfEquivalenceFaible} que $f$ est une équivalence faible si et seulement si $D(f)$ est une équivalence faible. De plus, par définition, $D(f)$ est une équivalence faible si et seulement si, pour tout $\Delta^{\varphi}\in \Delta(P)$, l'application
\begin{equation*}
D(f)(\Delta^{\varphi})\colon \Map(\Delta^{\varphi},\fil{X})\to\Map(\Delta^{\varphi},\fil{Y})
\end{equation*}
est une équivalence faible. Cette dernière est une équivalence faible si et seulement si, pour tout $n\geq 0$ et pour tout $*\in [\Delta^{\varphi},\fil{X}]=\pi_0(\Map(\Delta^{\varphi},\fil{X}))$, l'application
\begin{equation*}
\pi_n(D(f)(\Delta^{\varphi}))\colon \pi_n(\Map(\Delta^{\varphi},\fil{X}),*)\to\pi_n(\Map(\Delta^{\varphi},\fil{Y}),*)
\end{equation*}
est un isomorphisme. Cette dernière condition est satisfaite si et seulement si, pour tout pointage $\phi\colon \fil{V}\to\fil{X}$ et pour tout $n\geq 0$, l'application 
\begin{equation*}
s\pi_n(f)\colon s\pi_n(\fil{X},\phi)\to s\pi_n(\fil{Y},f\circ \phi)
\end{equation*}
est un isomorphisme.
\end{proof}

\begin{corollaire}\label{CorollaireEquivalenceFaiblePointageComplet}
Sous les hypothèses du théorème \ref{EquivalenceFaibleIsoGroupeHomotopie}, $f$ est une équivalence faible si et seulement si, pour $(\fil{V_i},\phi_i)_{i\in I}$ un ensemble complet de pointages de $\fil{X}$, le morphisme 
\begin{equation*}
s\pi_0(f)\colon s\pi_0(\fil{X})\to s\pi_0(\fil{Y})
\end{equation*}
ainsi que les morphismes
\begin{equation*}
s\pi_n(f)\colon s\pi_n(\fil{X},\phi_i)\to s\pi_n(\fil{Y},f\circ\phi_{i})
\end{equation*}
sont des isomorphismes pour $i\in I$ et $n\geq 1$.
\end{corollaire}

\begin{proof}
D'après la preuve du théorème \ref{EquivalenceFaibleIsoGroupeHomotopie}, il suffit de vérifier les hypothèses du théorème pour les pointages de la forme $\phi\colon \Delta^{\varphi}\to\fil{X}$. Soit $\phi$ un tel pointage, alors, d'après la remarque \ref{SimplexesDeSPi0}, $[\phi]$ correspond à un simplexe de $s\pi_0\fil{X}$. Comme l'ensemble de pointage $(\fil{V_i},\phi_i)_{i\in I}$ est supposé complet, il existe $i\in I$ tel que $[(\phi_i)_{|\Delta^{\varphi}}]=[\phi]$. En particulier, $(\phi_i)_{|\Delta^{\varphi}}$ et $\phi$ sont homotopes par une homotopie filtrée. Mais alors, par la proposition \ref{HomotopiePointagesGroupesHomotopie}, on a le diagramme commutatif suivant
\begin{equation*}
\begin{tikzcd}
s\pi_n(\fil{X},\phi)
\arrow{d}{\simeq}
\arrow{r}{s\pi_n(f)}
&s\pi_n(\fil{Y},f\circ\phi)
\arrow{d}{\simeq}
\\
s\pi_n(\fil{X},(\phi_i)_{|\Delta^{\varphi}})
\arrow{r}{s\pi_n(f)}
&s\pi_n(\fil{Y},f\circ (\phi_i)_{|\Delta^{\varphi}})
\end{tikzcd}
\end{equation*}
où les flèches verticales sont des isomorphismes. On en déduit qu'on a un isomorphisme pour le pointage $\phi$ si et seulement si on a un isomorphisme pour le pointage $(\phi_i)_{|\Delta^{\varphi}}$. 
\end{proof}

\subsection{Calcul des groupes d'homotopie filtrés}

\begin{prop}\label{SimplexesNonDegeneresGroupesHomotopiesFiltres}
Soient $\fil{X}$ un ensemble simplicial filtré fibrant, $\phi\colon\fil{V}\to\fil{X}$ un pointage et $n\geq 0$ un entier. Les simplexes non dégénérés de $s\pi_n(\fil{X},\phi)$ sont tous les simplexes de la forme
\begin{equation*}
\Delta^{\varphi}\to s\pi_n(\fil{X},\phi)
\end{equation*}
avec $\Delta^{\varphi}$ un simplexe non dégénéré de $N(P)$.
\end{prop}

\begin{proof}
Soit $\Delta^{\varphi}=(\Delta^n,\varphi)$ un simplexe non-dégénéré de $N(P)$, et soit $\Delta^{\varphi\circ S}=(\Delta^{n'},\varphi\circ S)$ une dégénérescence de $\Delta^{\varphi}$. Soit  $D\colon \Delta^{\varphi}\to\Delta^{\varphi\circ S}$ une section de $S$. Par construction, c'est une cofibration triviale. En particulier, en appliquant $\Map(-,\fil{X})$, on obtient une équivalence faible
\begin{equation*}
\Map(\Delta^{\varphi\circ S},\fil{X})\to \Map(\Delta^{\varphi},\fil{X}).
\end{equation*}
Par deux sur trois, on en déduit que l'application induite par $S$
\begin{equation*}
\Map(\Delta^{\varphi},\fil{X})\to \Map(\Delta^{\varphi\circ S},\fil{X}).
\end{equation*}
est aussi une équivalence faible.
En appliquant le foncteur $\pi_n$, on obtient donc un isomorphisme, ce qui implique que tous les simplexes de la forme $\Delta^{\varphi\circ S}\to s\pi_n(\fil{X},\phi)$ sont des dégénérescences de simplexes de la forme $\Delta^{\varphi}\to s\pi_n(\fil{X},\phi)$. 
\end{proof}

\begin{corollaire}\label{CorollaireSimplexesNonDegeneressPin}
Soient $\fil{X}$ et $\fil{Y}$ deux ensembles simpliciaux filtrés fibrants, $\phi\colon \fil{V}\to\fil{X}$ un pointage de $\fil{X}$, $f\colon \fil{X}\to\fil{Y}$ une application filtrée et $n\geq 0$ un entier. Alors $f$ induit un isomorphisme
\begin{equation*}
s\pi_n(f)\colon s\pi_n(\fil{X},\phi)\to s\pi_n(\fil{Y},f\circ\phi)
\end{equation*}
si et seulement si, pour tout simplexe non dégénéré $\Delta^\varphi\subseteq V$, $f$ induit des isomorphismes
\begin{equation*}
\pi_n(\Map(\Delta^{\varphi},\fil{X}),\phi_{|\Delta^{\varphi}})\to\Map(\Delta^{\varphi},\fil{Y}),f\circ\phi_{|\Delta^{\varphi}})
\end{equation*}
\end{corollaire}

\begin{corollaire}
Soit $\fil{X}$ un ensemble simplicial filtré fibrant. S'il existe un entier $d\geq 0$ tel que pour toute suite strictement croissante d'éléments de $P$, $p_0<p_1<\dots <p_k$, on a $n\leq d$, alors pour tout pointage $\phi\colon \fil{V}\to \fil{X}$, et pour tout entier $n\geq 0$,
l'ensemble simplicial sous-jacent à $s\pi_n(\fil{X},\phi)$ est de dimension inférieure à $d$.
\end{corollaire}

\begin{corollaire}
Soit $X$ un complexe de Kan, $P=\{*\}$ le poset trivial, $\varphi_X\colon X\to \{*\}$ la filtration triviale et $\phi\colon \{*\}\to X$ un pointage. Alors, $s\pi_n(\fil{X},\phi)_0\simeq \pi_n(X,\phi(*))$ pour tout $n\geq 0$, et $s\pi_n(\fil{X})$ est un ensemble simplicial de dimension $0$.
\end{corollaire}

La proposition \ref{SimplexesNonDegeneresGroupesHomotopiesFiltres} permet de calculer les valeur du foncteur $s\pi_n(\fil{X},\phi)$ seulement sur la sous catégorie de $\Delta(P)$ contenant les simplexes non dégénérés. Pour exploiter ce résultat, on utilisera la catégorie suivante.

\begin{defin}\label{DefinitionRP}
On définit la sous catégorie pleine $R(P)\subset\Delta(P)$ dont les objets sont les monomorphismes
\begin{equation*}
\varphi\colon\Delta^n\to N(P).
\end{equation*}
Autrement dit, $R(P)$ est la catégorie des simplexes non dégénérés de $N(P)$, où les morphismes sont les inclusions. Si $\Delta^{\varphi}\in \Delta(P)$ est un simplexe de $N(P)$, il existe un unique simplexe non dégénéré de $N(P)$ dont $\Delta^{\varphi}$ est une dégénérescence. Dans la suite de ce texte, on notera $\Delta^{\bar{\varphi}}$ le simplexe de $R(P)$ correspondant à $\Delta^{\varphi}$ de cette façon.
\end{defin}

On remarque que l'inclusion de sous catégorie $j\colon R(P)\to\Delta(P)$ induit un foncteur de restriction
\begin{equation*}
j^*\colon\Fun(\Delta(P)^{\op},\mathcal{C})\to\Fun(R(P)^{\op},\mathcal{C}),
\end{equation*}
pour toute catégorie $\mathcal{C}$. En particulier, si $\fil{X}$ est un ensemble simplicial filtré fibrant et $\phi\colon \fil{V}\to\fil{X}$ est un pointage, alors l'image de $s\pi_n(\fil{X},\phi)$ par $j^*$ fournit un foncteur
\begin{equation*}
j^*s\pi_n(\fil{X},\phi)\colon R(P)^{\op}\to\Set.
\end{equation*}
Avec ces notations, on peut donc réécrire la proposition \ref{SimplexesNonDegeneresGroupesHomotopiesFiltres} sous la forme suivante.

\begin{prop}
Soit $\fil{X}$ un ensemble simplicial filtré fibrant, et $\phi\colon \fil{V}\to\fil{X}$ un pointage. Alors le foncteur 
\begin{equation*}
s\pi_n(\fil{X},\phi)\colon \Delta(P)^{\op}\to\Set
\end{equation*}
est entièrement détérminé par sa restriction à $R(P)^{\op}$.
\begin{equation*}
j^*s\pi_n(\fil{X},\phi)\colon R(P)^{\op}\to\Set.
\end{equation*}
\end{prop}

\begin{remarque}\label{GroupesHomotopiesFiltresReduits}
En vertu de la proposition précédente, on peut identifier les groupes d'homotopie filtrés avec leur restriction à $R(P)$ sans risque de confusion. Pour cette raison, lorsqu'on calculera des groupes d'homotopie filtrés, on calculera simplement leur restriction à $R(P)$. Cependant, il est parfois utile de voir les groupes d'homotopie filtrés comme des foncteurs depuis $\Delta(P)^{\op}$ comme le montre la proposition \ref{PropositionRaffinementStratification}.
\end{remarque}

\section{Raffinement d'une stratification}
\label{SectionRaffinementStratification}
\begin{defin}
Soit $\fil{X}$ un ensemble simplicial filtré fibrant. On définit la transformation naturelle $\lambda\colon \Id\to s\pi_0$ comme suit.
\begin{align*}
\lambda_{\fil{X}}\colon\fil{X}&\to s\pi_0(\fil{X})\\
\sigma\colon \Delta^{\varphi}\to \fil{X}&\mapsto [\sigma]\in [\Delta^{\varphi},\fil{X}].
\end{align*}
Notons $P_{\fil{X}}=(s\pi_0\fil{X}_0,\leq)$ l'ordre partiel sur les sommets de $s\pi_0\fil{X}$, où la relation $\leq$ est engendrée par $\sigma\leq \tau$ s'il existe $\mu\colon \Delta^1\to s\pi_0(\fil{X})$ tel que $\sigma=\mu\circ D_1$ et $\tau=\mu\circ D_0$, où $D_0$ et $D_1$ correspondent aux deux inclusions de face $\Delta^0\to \Delta^1$.
Alors, on a la suite d'application suivante, dont la composition est égale à $\varphi_X$.
\begin{equation*}
\fil{X}\xrightarrow{\lambda}s\pi_0(\fil{X})\to N(P_{\fil{X}})\to N(P)
\end{equation*}
\end{defin}

\begin{prop}\label{PropositionRaffinementStratification}
Soit $\fil{X}$ un ensemble simplicial filtré fibrant, notons $\widetilde{\varphi}_X\colon X\to N(P_{\fil{X}})$ l'application filtrée de la définition précédente. Alors, pour tout $p\in P_{\fil{X}}$, $\widetilde{\varphi}_X^{-1}(p)$ est un ensemble simplicial connexe. De plus, $(X,\widetilde{\varphi}_X)$ est un objet fibrant de $\sS_{P_{\fil{X}}}$.
\end{prop}

\begin{remarque}
Ainsi, à tout ensemble simplicial filtré (fibrant) sur $P$, on peut associer un ensemble simplicial filtré (fibrant) sur un autre ordre partiel, tel que toutes ses strates soient connexes.
\end{remarque}

\chapter{Groupes d'homotopie filtrés et espaces filtrés}\label{ChapitreGroupesHomotopiesFiltresEspaces}

L'objet de ce texte est d'étudier la théorie de l'homotopie des espaces stratifiés. On a vu dans les chapitres précédents que les groupes d'homotopie filtrée permettaient de caractériser les équivalences faibles entre ensembles simpliciaux filtrés. Aussi, il est naturel d'étendre la définition des groupes d'homotopie filtrés aux espaces filtrés et d'expliciter les propriétés homotopiques des espaces filtrés qui sont codées par les groupes d'homotopie filtrés.

Dans la première section de ce chapitre, on définit la catégorie des espaces filtrés, et on étudie ses propriétés. Elle est liée par une adjonction à la catégorie des ensembles simpliciaux filtrés, (section \ref{SectionAdjonctionTopPsSetP}), c'est une catégorie simpliciale (section \ref{SectionTopPCategorieSimpliciale}) et on y définit une notion d'homotopie filtrée compatible avec les définitions précédentes section (\ref{SectionHomotopiesFiltreesTopP}).

Dans la section \ref{SectionGroupesHomotopieFiltresTopP}, on définit les groupes d'homotopie filtrés pour les espaces filtrés (Définitions \ref{DefinitionSPi0TopP} et \ref{DefinitionGroupesHomotopieFiltresTopP}), et on étudie leurs propriétés. On montre que pour une classe d'espaces filtrés appropriée, les groupes d'homotopie filtrés d'un espace filtré coincident avec ceux de l'ensemble simplicial filtré correspondant (Proposition \ref{IsomorphismeGroupesHomotopieSingP}).

Finalement, dans la section \ref{SectionTheoremesWhitehead}, on présente deux généralisations du théorème de Whitehead aux espaces filtrés. Ces théorèmes sont vrais pour une certaine classe d'espaces filtrés PL, décrite dans la section \ref{SubsectionObjetsFibrantsNaifTopP}, incluant notamment les pseudo-variétés PL. 
Pour deux tels objets et un morphisme entre eux $f\colon \fil{X}\to\fil{Y}$, on a équivalence entre les assertions suivantes
\begin{itemize}
\item Le morphisme $f$ est une équivalence d'homotopie filtrée.
\item Le morphisme $f$ induit des isomorphismes entre tout les groupes d'homotopie filtrés.
\item Le morphisme $f$ induit des équivalences faibles sur toutes les strates et tous les entrelacs homotopiques.
\end{itemize}
L'équivalence entre les deux premières assertions est le contenu du théorème \ref{PremierTheoremeWhitehead}, et l'équivalence avec la dernière assertion provient du théorème \ref{DeuxiemeTheoremeWhitehead}.

\begin{remarque}\label{RemarqueCategorieDeltaEngendre}
Dans ce chapitre et dans le reste de ce texte, on travaille avec la catégorie des espaces topologiques $\Delta$-engendrés. Un espace topologique $X$ est $\Delta$-engendré si pour tout $U\subset X$, $U$ est ouvert si et seulement si pour toute application continue $f\colon \Delta^n\to X$, $f^{-1}(U)$ est un ouvert du simplexe standard $\Delta^n$. Par \cite{Dugger} et \cite{ConvenientCategory}, la catégorie des espaces $\Delta$-engendré est cartésienne et localement présentable. On notera $\Top$ cette catégorie, et on notera $-\times-\colon \Top\times \Top\to \Top$ son produit. On réserve la notation $\otimes$ pour les produits entre objets filtrés et non-filtrés.
\end{remarque}

\section{La catégorie des espaces filtrés}
\subsection{Espaces et morphismes filtrés}

\begin{defin}
Soit $P$ un ensemble ordonné. On considère $P$ comme un espace topologique muni de la topologie d'Alexandroff définie comme suit. Un ensemble $U\subset P$ est ouvert si et seulement si $\forall x\in U$, $\forall y\in P$, $y\geq x \Rightarrow y\in U$. En particulier, une base d'ouvert de $P$ est donnée par $(U_p)_{p\in P}$ où $U_p$ est défini comme 
\begin{equation*}
U_p=\{q\in P\ |\ q\geq p\}
\end{equation*}
\end{defin}
Dans tout le reste de ce texte, on identifiera l'ensemble ordonné $P$ avec l'espace topologique $P$ muni de la topologie d'Alexandroff.

\begin{defin}\label{DefinitionTopP}
Un espace topologique filtré au dessus de $P$ est la donnée
\begin{itemize}
\item d'un espace topologique $A$
\item d'une application continue $\varphi_A\colon A\to P$. 
\end{itemize}
Une application filtrée $f\colon \fil{A}\to\fil{B}$ est la donnée d'une application continue $f\colon A\to B$ telle que le triangle suivant commute
\begin{equation*}
\begin{tikzcd}
A
\arrow{rr}{f}
\arrow[swap]{dr}{\varphi_A}
&\phantom{X}
&B
\arrow{dl}{\varphi_B}
\\
\phantom{X}
&P
\end{tikzcd}
\end{equation*}
En considérant la composition usuelle, on obtient la catégorie des espaces filtrés au dessus de $P$, $\Top_P$.
\end{defin}

\subsection{Adjonction avec les ensembles simpliciaux filtrés}
\label{SectionAdjonctionTopPsSetP}
\begin{defin}\label{DefinitionVarphiP}
Soit $P$ un ensemble ordonné, la réalisation de son classifiant, $\Real{N(P)}$, possède une filtration canonique $\varphi_P\colon \Real{N(P)}\to P$ définie par 
\begin{align*}
\varphi_P\colon\Real{N(P)}&\to P\\
([p_0,\dots,p_n],t\in \mathring{\Delta}^n)&\mapsto p_n 
\end{align*}
où $\mathring{\Delta}^n$ désigne l'intérieur du $n$-simplexe canonique $\Delta^n$ si $n\geq 1$, et $\Delta^0$ si $n=0$.
\end{defin}

\begin{defin}\label{DefinitionRealP}
Soit $\fil{X}$ un ensemble simplicial filtré. On définit sa réalisation filtrée $\RealP{\fil{X}}$, comme l'espace filtré $(\Real{X},\varphi_P\circ\Real{\varphi_X})$. Ceci définit un foncteur 
\begin{equation*}
\RealP{-}\colon\sS_P\to \Top_P
\end{equation*}
\end{defin}
On notera par abus de notation $\RealP{X}$ pour désigner l'espace filtré $\RealP{\fil{X}}$.

\begin{defin}
Soit $\fil{A}$ un ensemble simplicial filtré, on définit l'ensemble simplicial filtré $\Sing_P\fil{A}$ comme suit
\begin{align*}
\Sing_P\fil{A}\colon \Delta(P)^{\op}&\to \Set\\
\Delta^{\varphi}&\mapsto \Hom_{\Top_P}(\RealP{\Delta^{\varphi}},\fil{A})
\end{align*}
Ceci définit un foncteur $\Sing_P\colon \Top_P\to\sS_P$.
\end{defin}
Par abus de notation, on notera $\Sing_P(A)$ pour désigner l'image de $\fil{A}$ par $\Sing_P$.

\begin{prop}\label{AdjonctionSingReal}
Le foncteur $\RealP{-}$ est un adjoint à gauche du foncteur $\Sing_P$.
\end{prop}

\begin{proof}
Soient $\fil{A}$ un espace filtré et $\fil{X}$ un ensemble simplicial filtré. On définit les bijections suivantes 
\begin{align*}
\Hom_{\Top_P}(\RealP{X},\fil{A})&\to\Hom_{\sS_P}(\fil{X},\Sing_P(A))\\
(f\colon \RealP{X}\to\fil{A})&\mapsto \left\{\begin{array}{ccc}
\fil{X}&\to &\Sing_P(A)\\
(\sigma\colon\Delta^{\varphi}\to\fil{X})&\mapsto & f\circ\RealP{\sigma}\in \Sing_P(A)(\Delta^{\varphi})
\end{array}\right.
\end{align*}
\begin{align*}
\Hom_{\sS_P}(\fil{X},\Sing_P(A))&\to\Hom_{\Top_P}(\RealP{X},\fil{A})\\
(g\colon \fil{X}\to\Sing_P(A))&\mapsto \left(\begin{array}{ccc}
\RealP{X}&\to&\fil{A}\\
(\sigma,t)&\mapsto & g(\sigma)(t) 
\end{array}
\right)
\end{align*}
\end{proof}

\begin{prop}\label{SingPPullback}
Soit $\fil{X}$ un espace filtré. Le carré suivant est cartésien
\begin{equation}\label{CarreCocartesienSingP}
\begin{tikzcd}
\Sing_P\fil{X}
\arrow{d}
\arrow{r}
&\Sing(X)
\arrow{d}{\Sing(\varphi_X)}
\\
N(P)
\arrow{r}{i}
&\Sing(P)
\end{tikzcd}
\end{equation}
où $i\colon N(P)\to\Sing(P)$ est l'adjoint du morphisme canonique $\varphi_P\colon\Real{N(P)}\to P$
\end{prop}

\begin{proof}
Il est clair que le diagramme \ref{CarreCocartesienSingP} est commutatif, par définition de $\Sing_P$. De plus, par construction, le morphisme $\Sing_P\fil{X}\to\Sing(X)$ est injectif, de même que le morphisme $N(P)\to\Sing(P)$. Il suffit donc de vérifier que l'image de $\Sing_P\fil{X}$ dans $\Sing(X)$ est égale à la préimage de $N(P)$ par $\Sing_P(\varphi_X)$. Soit $\sigma\colon \Real{\Delta^n}\to X$ un simplexe de $\Sing(X)$ tel que la composition  
\begin{equation*}\varphi_X\circ \sigma\colon\Real{\Delta^n}\to X\to P
\end{equation*}
est un simplexe de $i(N(P))$.
Autrement dit, tel qu'il existe $\varphi\colon \Delta^{n}\to N(P)$,
vérifiant
\begin{equation*}
\varphi_X\circ\sigma=\varphi_P\circ\Real{\varphi}\colon \Real{\Delta^n}\to P.
\end{equation*}
Alors, on a le diagramme commutatif suivant
\begin{equation*}
\begin{tikzcd}
\Real{\Delta^n}
\arrow{rr}{\sigma}
\arrow[swap]{dr}{\varphi_X\circ\sigma=\varphi_P\circ\Real{\varphi}}
&\phantom{X}
&X
\arrow{dl}{\varphi_X}
\\
\phantom{X}
&P
\end{tikzcd}
\end{equation*}
En particulier, $\sigma$ est l'image d'un simplexe filtré $\RealP{\Delta^{\varphi}}\to \fil{X}$.
\end{proof}
\subsection{Structure simpliciale}
\label{SectionTopPCategorieSimpliciale}

\begin{defin}\label{DefinTenseurTopPSimplicial}
Soit $\fil{X}$ un espace filtré et $K$ un ensemble simplicial. On définit l'espace filtré $K\otimes \fil{X}$ comme $(\Real{K}\times X, \varphi_X\circ\pr_{X})$. Ceci définit un bifoncteur 
\begin{equation*}
-\otimes-\colon\sS\times \Top_P\to\Top_p
\end{equation*}
On en déduit un foncteur 
\begin{align*}
\Map\colon\Top_P^{\op}\times\Top_P&\to\sS\\
(\fil{X},\fil{Y})&\mapsto \Map(X,Y)\colon \left\{\begin{array}{ccl}
\Delta^{\op}&\to &\Set\\
\Delta^n&\mapsto &\Hom_{\Top_P}(\Delta^n\otimes\fil{X},\fil{Y})
\end{array}\right.
\end{align*}
\end{defin}

\begin{defin}\label{DefinitionEspacesDeCheminsStratifies}
Soit $K$ un ensemble simplicial et $\fil{Y}$ un espace filtré. On définit l'espace filtré $\fil{Y}^{K}$ comme suit. L'espace topologique sous jacent est obtenu comme l'union :
\begin{equation*}
\fil{Y}^{K}=\coprod_{p\in P}\Hom_{\Top_P}(\RealP{K_p},\fil{Y})\subset \Hom_{\Top}(\Real{K},Y)
\end{equation*}
où $K_p$ est l'ensemble simplicial filtré $K\to\{p\}\subset N(P)$.
La topologie est obtenue en considérant $\fil{Y}^K$ comme un sous objet de $\Hom_{\Top}(\Real{K},Y)$. La filtration est donnée par $\varphi(\Hom_{\Top_P}(\RealP{K_p},\fil{Y}))=p$.
\end{defin}

\begin{prop}\label{AdjointFoncteurTenseurTop}
La définition précédente s'étend en un foncteur 
\begin{equation*}
(-)^{-}\colon\Top_P\times \sS^{\op}\to \Top_P
\end{equation*}
De plus, si $K$ est un ensemble simplicial fixé, le foncteur $K\otimes -$ est un adjoint à gauche du foncteur $(-)^{K}$.
\end{prop}

\begin{proof}
Soit $K$ un ensemble simplicial et $\fil{X}$ et $\fil{Y}$ deux espaces filtrés. On a les applications suivantes
\begin{align*}
\Hom(\fil{X},\fil{Y}^K)&\to\Hom(K\otimes\fil{X},\fil{Y})\\
(f\colon X\to Y^K)&\mapsto \left\{\begin{array}{ccc}
\Real{K}\times X&\to &Y\\
((\sigma,t),x)&\mapsto & f(x)(\sigma,t)
\end{array}\right.
\end{align*}
et
\begin{align*}
\Hom(K\otimes\fil{X},\fil{Y})&\to \Hom(\fil{X},\fil{Y}^K)\\
(g\colon \Real{K}\times X\to Y)&\mapsto\left\{\begin{array}{ccc}
X&\to &\fil{Y}^K\\
x&\mapsto &(\sigma,t)\mapsto g((\sigma,t),x)
\end{array}\right.
\end{align*}
Il suffit de vérifier que ces applications sont bien définies, elles seront, par construction, inverses l'une de l'autre. Si $f\colon \fil{X}\to\fil{Y}^K$ est une application filtrée, alors, pour tout $x\in X$ et tout $(\sigma,t)\in \Real{K}$, $\varphi_{K\otimes X}((\sigma,t),x))=\varphi_X(x)=\varphi_{Y^K}(f(x))=\varphi_Y(f(x)(\sigma,t))$. En particulier, la première application est bien définie. Et, si $g\colon K\otimes \fil{X}\to \fil{Y}$ est une application filtrée,  pour tout $x$ et pour tout $(\sigma,t)$, on a $\varphi_X(x)=\varphi_{K\otimes X}((\sigma,t),x)=\varphi_Y(g((\sigma,t),x))$. On en déduit que la deuxième application est bien définie.
\end{proof}

\begin{prop}\label{PropositionTopPCategorieSimpliciale}
La catégorie $\Top_P$ munie des foncteurs $\Map$, $-\otimes - $, et $(-)^{-}$ est une catégorie simpliciale.
\end{prop}

\begin{proof}
D'après \cite[Lemme II.2.4]{GoerssJardine}, il suffit de vérifier les points suivants. 
\begin{itemize}
\item Pour tout ensemble simplicial $K$, le foncteur $K\otimes -$ est un adjoint à gauche du foncteur $(-)^K$, c'est la proposition \ref{AdjointFoncteurTenseurTop}.
\item Pour tout espace topologique filtré $\fil{X}$, le foncteur $-\otimes \fil{X}$ préserve les colimites et vérifie $\Delta^{0}\otimes \fil{X}\simeq \fil{X}$. La première affirmation provient du fait que $-\times X$ préserve les colimites dans $\Top$ (C'est notamment pour avoir cette propriété qu'on travaille avec la catégorie des espaces $\Delta$-engendrés). La deuxième provient du fait que pour tout espace topologique $X$, $X\times \Delta^{0}\simeq X$.
\item Il existe un isomorphisme $L\otimes(K\otimes \fil{X})\simeq (L\times K)\otimes \fil{X}$ naturel en $K,L\in \sS$ , et en $\fil{X}\in \Top_P$.
Cet isomorphisme provient de l'isomorphisme naturel 
$\Real{K}\times(\Real{L}\times X)\simeq \Real{K\times L}\times X$.  
\end{itemize}
\end{proof}

\begin{prop}\label{AdjonctionSimplicialeSingReal}
L'adjonction $\RealP{-}\colon\sS_P\leftrightarrow\Top_P\colon\Sing_P$ de la proposition \ref{AdjonctionSingReal} est une adjonction simpliciale.
\end{prop}

\begin{proof}
Soient $\fil{X}$ un ensemble simplicial filtré et $\fil{Y}$ un espace topologique filtré. On a les bijections canoniques suivantes
\begin{align*}
\Map(\RealP{X},\fil{Y})_n&= \Hom(\Delta^n\otimes\RealP{X},\fil{Y}\\
&\simeq \Hom(\RealP{\Delta^n\otimes\fil{X}},\fil{Y})\\
&\simeq \Hom(\Delta^n\otimes\fil{X},\Sing_P(Y))\\
&= \Map(\fil{X},\Sing_P(Y))_n
\end{align*}
La première bijection provient du fait que si $K$ et $L$ sont des ensembles simpliciaux, avec $K$ ayant un nombre fini de simplexes non dégénérés, on a un isomorphisme canonique $\Real{K\times L}\simeq \Real{K}\times\Real{L}$ \cite{MilnorGeometricRealization}. La seconde bijection provient de l'adjonction.
\end{proof}

\subsection{Homotopies filtrées}
\label{SectionHomotopiesFiltreesTopP}

\begin{defin}
Soit $f,g\colon \fil{X}\to\fil{Y}$ deux applications filtrées entre espaces filtrés. On dit que $f$ et $g$ sont homotopes au sens filtré s'il existe une application filtrée $H\colon \Delta^1\otimes\fil{X}\to\fil{Y}$ telle que $H\circ (\iota_0\otimes\Id_X)=f$ et $H\circ (\iota_1\otimes \Id_X)=g$. On dira aussi que $f$ et $g$ sont homotopes par l'homotopie filtrée $H$.
\end{defin}

\begin{prop}
Soient $\fil{X}$ et $\fil{Y}$ deux espaces filtrés et $f,g\colon \fil{X}\to\fil{Y}$ deux applications filtrées homotopes au sens filtré. Alors $\Sing_P(f)$ et $ \Sing_P(g)$ sont homotopes au sens filtré.
\end{prop}

\begin{proof}
Il suffit d'exhiber une application $h$ faisant commuter le diagramme suivant 
\begin{equation*}
\begin{tikzcd}[column sep = large]
\phantom{X}
&\Sing_P(X)
\arrow[bend left =12]{drr}{\Sing_P(f)}
\arrow[swap]{dl}{\iota_0\otimes\Id}
\arrow{dr}{\Sing_P(\iota_0\otimes\Id)}
\\
\Delta^1\otimes\Sing_P(X)
\arrow{rr}{h}
&\phantom{X}
&\Sing_P(\Delta^1\otimes\fil{X})
\arrow{r}{\Sing_P(H)}
&\Sing_P(Y)
\\
\phantom{X}
&\Sing_P(X)
\arrow[swap]{ur}{\Sing_P(\iota_1\otimes\Id)}
\arrow{ul}{\iota_1\otimes\Id}
\arrow[swap,bend right= 12]{urr}{\Sing_P(g)}
\end{tikzcd}
\end{equation*}
On définit $h$ comme suit 
\begin{align*}
h\colon \Delta^1\otimes\Sing_P(X)&\to \Sing_P(\Delta^1\otimes \fil{X})\\
(\sigma\colon \Delta^{\varphi}\to \Delta^1,\tau\colon \RealP{\Delta^{\varphi}}\to \fil{X})&\mapsto \Real{\sigma}\times\tau\colon \RealP{\Delta^{\varphi}}\to \Delta^1\otimes\fil{X}
\end{align*}
\end{proof}

\begin{prop}
Soient $\fil{X}$ et $\fil{Y}$ deux ensembles simpliciaux filtrés et $f,g\colon \fil{X}\to\fil{Y}$ deux applications  homotopes au sens filtré. Alors, $\RealP{f}$ et $\RealP{g}$ sont homotopes au sens filtré.
\end{prop}

\begin{proof}
Soit $H\colon \Delta^1\otimes\fil{X}\to\fil{Y}$ une homotopie entre $f$ et $g$. On a $\RealP{\Delta^1\otimes\fil{X}}\simeq\Delta^1\otimes\RealP{\fil{X}}$. En particulier, $\RealP{H}$ fournit une homotopie entre $\RealP{f}$ et $\RealP{g}$.
\end{proof}
\section{Groupes d'homotopie filtrés pour les espaces filtrés}
\label{SectionGroupesHomotopieFiltresTopP}
\subsection{Espaces de fonctions}

\begin{defin}
Soient $\fil{X}$ et $\fil{Y}$ deux espaces filtrés. On définit une topologie sur $\Hom_{\Top_P}(\fil{X},\fil{Y})$ via l'inclusion
\begin{equation*}
\Hom_{\Top_P}(\fil{X},\fil{Y})\subseteq C^0(X,Y),
\end{equation*}
et on note $C^0_P(\fil{X},\fil{Y})$ l'espace topologique correspondant (ou $C^0_P(X,Y)$ lorsqu'il n'y a pas d'ambiguïté possible sur les filtrations considérées). Ceci définit un foncteur
\begin{equation*}
C^0_P\colon \Top_P^{\op}\times\Top_P\to\Top
\end{equation*}
\end{defin}

\begin{remarque}
Ici, la topologie de $C^0(X,Y)$ est donné par la structure cartésienne sur la catégorie des espaces $\Delta$-engendrés. Voir la remarque \ref{RemarqueCategorieDeltaEngendre} et \cite{Dugger}.
\end{remarque}

\begin{defin}
Soient $A$ un espace topologique et $\fil{X}$ un espace filtré. On définit le produit $A\otimes \fil{X}$ comme l'espace filtré $A\times X\xrightarrow{\varphi_X\circ\pr_X} P$. Ceci définit un foncteur
\begin{equation*}
\Top\times\Top_P\to\Top_P
\end{equation*}
\end{defin}

La notation choisie pour ce foncteur est la même que celle du foncteur de la définition \ref{DefinTenseurTopPSimplicial}. La proposition suivante garantit que cela n'introduira pas de conflits dans les notations.
\begin{prop}
Soit $\fil{X}$ un simplexe filtré. La composition de foncteurs
\begin{equation*}
\sS\xrightarrow{\Real{-}}\Top\xrightarrow{-\otimes \fil{X}}\Top_P
\end{equation*}
est égale au foncteur $-\otimes\fil{X}\colon \sS\to \Top_P$. En particulier, pour tout ensemble simplicial $K$, on a $K\otimes\fil{X}=\Real{K}\otimes\fil{X}$.
\end{prop}

\begin{proof}
Par définition (voir définition \ref{DefinTenseurTopPSimplicial}), $K\otimes\fil{X}$ est défini comme $\Real{K}\otimes \fil{X}$.
\end{proof}

\begin{prop}\label{AdjonctionTenseurTopologique}
Soit $\fil{X}$ un espace filtré. Le foncteur $-\otimes\fil{X}$ est un adjoint à gauche du foncteur $C^0_P(X,-)$. De plus, cette adjonction est simpliciale.
\end{prop}

\begin{proof}
Soient $A$ un espace topologique et $\fil{Y}$ un espace filtré. On a les bijections suivantes
\begin{align*}
\Hom_{\Top_P}(A\otimes\fil{X},\fil{Y})&\to\Hom_{\Top}(A,C^0_P(X,Y))\\
f\colon A\times X\to Y&\mapsto\left(a\mapsto \left\{\begin{array}{ccc}
X&\to&Y\\
x&\mapsto & f(a,x)
\end{array}\right.\right)
\end{align*}
et
\begin{align*}
\Hom_{\Top}(A,C^0_P(X,Y))&\to \Hom_{\Top_P}(A\otimes\fil{X},\fil{Y})\\
g\colon A\to C^0_P(X,Y)&\mapsto \left\{\begin{array}{ccc}
A\times X&\to & Y\\
(a,x)&\mapsto & g(a)(x)
\end{array}\right.
\end{align*}
Pour montrer l'adjonction, il suffit de montrer que ces applications sont bien définies. En effet, elles seront inverses l'une de l'autre par construction.
Soit $f\colon A\otimes\fil{X}\to\fil{Y}$ une application filtrée. En particulier $f\colon A\times X\to Y$ est une application continue et l'application $a\mapsto f(a,-)\colon A\to C^0(X,Y)$ est continue. De plus, on a $\varphi_X(x)=\varphi_{A\otimes X}(a,x)=\varphi_Y(f(a,x))$. On en déduit que la première application est bien définie.
Soient $g\colon A\to C^0_P(X,Y)$ une application continue, et $(a,x)\in A\times X$. Alors, $\varphi_{A\otimes X}(a,x)=\varphi_X(x)=\varphi_Y(g(a)(x))$. On en déduit que la deuxième application est bien définie.
Montrons que l'adjonction est simpliciale. On calcule 
\begin{align*}
\Map_{\Top_P}(A\otimes\fil{X},\fil{Y})_n&\simeq \Hom_{\Top_P}(\Delta^n\otimes(A\otimes\fil{X}),\fil{Y})\\
&\simeq \Hom_{\Top_P}((\Real{\Delta^n}\times A)\otimes\fil{X},\fil{Y})\\
&\simeq \Hom_{\Top}(\Real{\Delta^n}\times A,C^0_P(X,Y))\\
&\simeq \Map_{\Top}(A,C^0_P(X,Y))_n
\end{align*}
Et toutes les bijections sont naturelles par rapport à $\Delta^n$.
\end{proof}
\subsection{Foncteurs diagrammes}

On rappelle la définition de la catégorie de diagrammes simpliciaux sur $P$ (définition \ref{DefCategorieDiagrammesSimpliciaux}), $\Diag_P=\Fun(\Delta(P)^{\op},\sS)$. Seulement pour cette sous-section, de façon à éviter toute ambiguïté, on notera $D^{\sS}$ pour désigner le foncteur de la définition \ref{DefFoncteurDiagrammeSimplicial}.

\begin{defin}
On définit la catégorie des diagrammes topologiqes sur $P$ comme
\begin{equation*}
\Diag^{\Top}_P=\Fun(\Delta(P)^{\op},\Top)
\end{equation*}
Et on note $\Sing\colon \Diag^{\Top}_P\to\Diag_P$ le foncteur
\begin{align*}
\Sing\colon \Diag^{\Top}&\to\Diag\\
\left(F\colon \Delta(P)^{\op}\to \Top\right)&\mapsto \left(\Sing\circ F\colon \Delta(P)^{\op}\to \sS\right)
\end{align*}
\end{defin}

\begin{defin}\label{DefinitionFoncteurDTopologique}
On définit les foncteurs
\begin{align*}
D\colon \Top_P &\to \Diag_P\\
\fil{X}&\mapsto (\Delta^{\varphi}\mapsto \Map_{\Top_P}(\RealP{\Delta^{\varphi}},\fil{X})
\end{align*}
et
\begin{align*}
D^{\Top}\colon \Top_P &\to \Diag^{\Top}_P\\
\fil{X}&\mapsto (\Delta^{\varphi}\mapsto C^0_P(\RealP{\Delta^{\varphi}},\fil{X})
\end{align*}
\end{defin}

\begin{prop}\label{DiagrammesTopologiquesIsomorphes}
Les foncteurs $D, D^{\sS}\circ \Sing_P$ et $\Sing\circ D^{\Top}$ sont isomorphes.
\end{prop}

\begin{proof}
Soient $\fil{X}$ un espace filtré et $\Delta^{\varphi}\in \Delta(P)$. Par la proposition \ref{AdjonctionSimplicialeSingReal}, on a les isomorphismes suivants naturels en $\fil{X}$ et en $\Delta^{\varphi}$.
\begin{align*}
D(X)(\Delta^{\varphi})&=\Map_{\Top_P}(\RealP{\Delta^{\varphi}},\fil{X})\\
&\simeq \Map_{\sS_P}(\Delta^{\varphi},\Sing_P(X))\\
&\simeq D^{\sS}(\Sing_P(X))(\Delta^{\varphi})
\end{align*}
D'autre part, pour $n\geq 0$, d'après la proposition \ref{AdjonctionTenseurTopologique}, on a les bijections suivantes, naturelles en $\fil{X}$, et en $\Delta^{\varphi}$.
\begin{align*}
(D(X)(\Delta^{\varphi}))_n&=\Map(\RealP{\Delta^{\varphi}},\fil{X})_n\\
&\simeq \Hom(\Delta^n\otimes\RealP{\Delta^{\varphi}},\fil{X})\\
&\simeq \Hom(\Real{\Delta^n}\otimes\RealP{\Delta^{\varphi}},\fil{X})\\
&\simeq \Hom(\Real{\Delta^n},C^0_P(\RealP{\Delta^{\varphi}},  \fil{X}))\\
&\simeq \Sing(C^0_P(\RealP{\Delta^{\varphi}},  \fil{X})_n\\
&\simeq \left(\Sing(D^{\Top}(X))(\Delta^{\varphi})\right)_n
\end{align*}
Comme les bijections respectent les applications faces et dégénérescences $\Delta^n\to \Delta^m$, on en déduit l'isomorphisme naturel en $X$, $D(X)\simeq \Sing\circ D^{\Top}(X)$.
\end{proof}

\begin{remarque}
En vertu de la proposition précédente, on écrira $D$ pour désigner indifféremment un des trois foncteurs $D, D^{\sS}\circ \Sing_P$ et $\Sing\circ D^{\Top}$.
\end{remarque}

\subsection{Pointages et composantes connexes}

\begin{defin}\label{DefinitionSPi0TopP}
Soit $\fil{X}$ un espace filtré. On définit son ensemble de composantes connexes filtrées comme l'ensemble simplicial filtré suivant 
\begin{align*}
s\pi_0\fil{X}\colon \Delta(P)^{\op}&\to \Set\\
\Delta^{\varphi}&\mapsto \pi_0(\Map(\RealP{\Delta^{\varphi}},\fil{X})
\end{align*}
\end{defin}

\begin{remarque}\label{DefinirSPi0DepuisDiagrammes}
De façon équivalente, on peut définir $s\pi_0$ comme la composée du foncteur $D\colon \Top_P\to\Diag_P$ avec le foncteur $\pi_0\colon \Fun(\Delta(P)^{\op},\sS)\to\Fun(\Delta(P)^{\op},\Set)$. En particulier, par la remarque \ref{DiagrammesTopologiquesIsomorphes}, on peut aussi calculer $s\pi_0\fil{X}(\Delta^{\varphi})$ comme $\pi_0(C^0_P(\RealP{\Delta^{\varphi}},\fil{X}))$ ou comme $\pi_0(\Map(\Delta^{\varphi},\Sing_P(X))$. En particulier, si $\Sing_P\fil{X}$ est un objet fibrant de $\sS_P$, on a $s\pi_0\fil{X}\simeq s\pi_0(\Sing_P(\fil{X})$ (voir définition \ref{DefinSPi0Simplicial}). On verra dans la section \ref{SubsectionObjetsFibrantsNaifTopP} un critère pour assurer que $\Sing_P\fil{X}$ est fibrant.
\end{remarque}

\begin{defin}\label{DefinitionPointageEspaceFiltre}
Soit $\fil{X}$ un espace filtré. Un pointage de $ \fil{X}$ est la donnée d'un sous ensemble simplicial filtré $\fil{V}\subseteq (N(P),\Id)$ et d'une application filtrée
\begin{equation*}
\phi\colon \RealP{V}\to\fil{X}
\end{equation*}
\end{defin}

\begin{remarque}
Soit $\fil{X}$ un espace filtré et $\phi\colon \RealP{V}\to\fil{X}$ un pointage de $X$. De même que dans le cas simplicial (voir remarque \ref{PointageInduitSimplicial}), si $\fil{V'}\subseteq\fil{V}$ est un sous ensemble simplicial filtré, on a un pointage induit
\begin{equation*}
\phi_{|V'}\colon \RealP{V'}\to\fil{X}
\end{equation*}
\end{remarque}

\begin{remarque}
Soit $\fil{X}$ un espace filtré. Un simplexe de $s\pi_0\fil{X}$, $\sigma\colon \Delta^{\varphi}\to s\pi_0\fil{X}$ correspond à une classe d'homotopie filtrée d'applications $\RealP{\Delta^{\varphi}}\to \fil{X}$. En effet, en vertu de la remarque \ref{DefinirSPi0DepuisDiagrammes}, un tel simplexe est un élément de $\pi_0(C^0_P(\RealP{\Delta^{\varphi}},\fil{X}))$, autrement dit une classe d'homotopie d'applications filtrées. Réciproquement étant donné un pointage de la forme $\phi\colon \RealP{\Delta^{\varphi}}\to\fil{X}$, on obtient un simplexe de $s\pi_0\fil{X}$ en considérant sa classe d'homotopie filtrée 
\begin{equation*}
[\phi]\in \pi_0(C^0_P(\RealP{\Delta^{\varphi}},\fil{X}))
\end{equation*}
\end{remarque}

\begin{defin}
Soit $\fil{X}$ un espace filtré et $(\phi_i,\fil{V_i})_{i\in I}$ un ensemble de pointages de $\fil{X}$. On dit que l'ensemble de pointages $(\phi_i,\fil{V_i})_{i\in I}$ est complet, si pour tout simplexe non dégénéré $\sigma\colon \Delta^{\varphi}\to s\pi_0(\fil{X})$, il existe $i\in I$ tel que $\Delta^{\varphi}\subseteq V_i$ et $[(\phi_i)_{|\Delta^{\varphi}}]=\sigma$.
\end{defin}

\begin{prop}
Tout espace filtré admet un ensemble complet de pointages.
\end{prop}

\begin{proof}
On définit l'ensemble de pointages suivant $(\phi_{\sigma},\Delta^{\varphi_{\sigma}})_{\sigma\in s\pi_0(\fil{X})_{\nd}}$ comme suit. Pour tout simplexe non dégénéré $\sigma\colon \Delta^{\varphi_{\sigma}}\to s\pi_0(\fil{X})$, on choisit un pointage 
\begin{equation*}
\phi_{\sigma}\colon \RealP{\Delta^{\varphi_{\sigma}}}\to\fil{X} 
\end{equation*}
tel que $[\phi_{\sigma}]=\sigma$. Par la remarque précédente, un tel pointage existe toujours. Alors, l'ensemble obtenu est un ensemble de pointages complet.
\end{proof}

\begin{remarque}\label{PointageCompletTopSimplicial}
Soit $\fil{X}$ un espace filtré tel que $\Sing_P(X)$ est un objet fibrant de $\sS_P$. Alors, tout pointage de $\fil{X}$ correspond à un unique pointage de $\Sing_P(X)$ par adjonction, et un ensemble complet de pointages de $\fil{X}$ correspond à un ensemble complet de pointages de $\Sing_P(X)$.
\end{remarque}

\subsection{Groupes d'homotopie filtrés supérieurs}

\begin{defin}\label{DefinitionGroupesHomotopieFiltresTopP}
Soient $\fil{X}$ un espace filtré, $n\geq 1$ un entier et $\phi\colon \RealP{V}\to\fil{X}$ un pointage de $\fil{X}$. On définit le $n$-ième groupe d'homotopie filtré de $\fil{X}$ comme l'ensemble simplicial filtré suivant
\begin{align*}
s\pi_n(\fil{X},\phi)\colon \Delta(P)^{\op}&\to \Set\\
\Delta^{\varphi}&\mapsto \left\{\begin{array}{cl}
\pi_n(\Map(\RealP{\Delta^{\varphi}},\fil{X}),\phi_{|\Delta^{\varphi}}) & \text{ si $\Delta^{\varphi}\subseteq \fil{V}$}\\
\emptyset & \text{ si $\Delta^{\varphi}\not\subseteq \fil{V}$}
\end{array}\right.
\end{align*}
\end{defin}

\begin{remarque}
Encore une fois, quitte à modifier les catégories impliquées pour tenir compte des notions de pointage, on aurait pu définir $s\pi_n$ comme la composée du foncteur $D\colon \Top_P\to\Diag_P$ avec le foncteur $\pi_n$. 
En particulier, en utilisant la proposition \ref{DiagrammesTopologiquesIsomorphes}, on peut calculer $s\pi_n(\fil{X},\phi)(\Delta^{\varphi})$ 
comme $\pi_n(\Map(\Delta^{\varphi},\Sing_P(X)),\phi_{|\Delta^{\varphi}})$ 
ou comme $\pi_n(C^0_P\RealP{\Delta^{\varphi}},\fil{X}),\phi_{|\Delta^{\varphi}})$.
\end{remarque}

\begin{prop}\label{IsomorphismeGroupesHomotopieSingP}
Soient $\fil{X}$ un espace filtré et $\phi\colon \RealP{V}\to\fil{X}$ un pointage de $\fil{X}$. Si $\Sing_P(X)$ est un objet fibrant de $\sS_P$, on a $s\pi_n(\fil{X},\phi)\simeq s\pi_n(\Sing_P(X),\widehat{\phi})$, où $\widehat{\phi}$ est obtenu à partir de $\phi$ par adjonction.
\end{prop}

\begin{proof}
On a d'une part
\begin{align*}
s\pi_n(\fil{X},\phi)(\Delta^{\varphi})&\simeq \pi_n(\Map(\RealP{\Delta^{\varphi}},\fil{X}),\phi_{|\Delta^{\varphi}})\\
&\simeq \pi_n(\Map(\Delta^{\varphi},\Sing_P(X)),\widehat{\phi}_{|\Delta^{\varphi}}).
\end{align*}
Et d'autre part, si $\Sing_P(X)$ est fibrant, on a 
\begin{equation*}
s\pi_n(\Sing_P(X),\widehat{\phi})(\Delta^{\varphi})=\pi_n(\Map(\Delta^{\varphi},\Sing_P(X)),\widehat{\phi}_{|\Delta^{\varphi}}).
\end{equation*}
Comme toutes ces bijections sont naturelles, on obtient l'isomorphisme voulu.
\end{proof}

\begin{remarque}\label{RemarqueSingPPasFibrant}
Si $\Sing_P(X)$ n'est pas fibrant, ses groupes d'homotopies filtrés ne sont (a priori) pas définis, et l'affirmation précédente n'a plus de sens. La raison pour laquelle les groupes d'homotopie filtrés peuvent être définis pour tout espace filtré provient de la proposition \ref{DiagrammesTopologiquesIsomorphes}. En effet, pour tout espace filtré $\fil{X}$ et pour tout simplexe filtré $\Delta^{\varphi}\in \Delta(P)$, on a l'isomorphisme d'ensembles simpliciaux
\begin{equation*}
\Map(\RealP{\Delta^{\varphi}},\fil{X})\simeq\Map(\Delta^{\varphi},\fil{X})\simeq \Sing(C^P_0(\RealP{\Delta^{\varphi}},\fil{X}))
\end{equation*}
en particulier, ce sont tous des complexes de Kan, car ils sont isomorphes à l'image d'un espace topologique par $\Sing$. On peut donc calculer leur groupes d'homotopies sans avoir besoin de calculer un remplacement fibrant.
\end{remarque}

Comme dans le cas des ensembles simpliciaux filtrés, les groupes d'homotopies filtrés sont des invariants du type d'homotopie filtré. Cette affirmation est rendue précise par les propositions suivantes  (voir les preuves des propositions \ref{HomotopiePointagesGroupesHomotopie} et \ref{HomotopieMorphismesGroupesHomotopie} pour les preuves dans le cas simplicial).

\begin{prop}\label{PointagesHomotopesIsomorphismes}
Soient $\phi,\psi\colon \RealP{V}\to\fil{X}$ deux pointages d'un espace filtré. Si $\phi$ et $\psi$ sont homotopes par une homotopie filtrée $H$, il existe un isomorphisme naturel
\begin{equation*}
s\pi_n(H)\colon s\pi_n(\fil{X},\phi)\to s\pi_n(\fil{X},\psi)
\end{equation*}
pour tout $n\geq 1$.
De plus, si $f\colon\fil{X}\to\fil{Y}$ est une application filtrée, les homotopies $H$ et $f\circ H$ induisent le diagramme commutatif suivant :
\begin{equation*}
\begin{tikzcd}
s\pi_n(\fil{X},\phi)
\arrow{d}{s\pi_n(H)}
\arrow{r}{s\pi_n(f)}
&s\pi_n(\fil{Y},f\circ\phi)
\arrow{d}{s\pi_n(f\circ H)}
\\
s\pi_n(\fil{X},\psi)
\arrow{r}{s\pi_n(f)}
&s\pi_n(\fil{Y},f\circ\psi)
\end{tikzcd}
\end{equation*}
\end{prop}

\begin{prop}\label{ApplicationsHomotopesMorphismesEgaux}
Soient $f,g\colon \fil{X}\to\fil{Y}$ deux applications filtrées entre espaces filtrés, $\phi\colon\RealP{V}\to\fil{X}$ un pointage de $X$ et $H\colon \Delta^1\otimes\fil{X}\to \fil{Y}$ une homotopie filtrée entre $f$ et $g$. Alors, pour tout $n\geq 0$, on a le diagramme commutatif suivant 
\begin{equation*}
\begin{tikzcd}
s\pi_n(\fil{X},\phi)
\arrow{r}{s\pi_n(f)}
\arrow[swap]{dr}{s\pi_n(g)}
&s\pi_n(\fil{Y},f\circ\phi)
\arrow{d}{s\pi_n(H\circ \phi)}
\\
\phantom{X}
&s\pi_n(\fil{Y},g\circ\phi)
\end{tikzcd}
\end{equation*} 
En particulier, si $f\circ \phi=g\circ \phi$, et $H\circ (\Delta^1\otimes\phi)$ est égal à la composition $\Delta^1\otimes\RealP{V}\xrightarrow{\pr_{\RealP{V}}}\RealP{V}\xrightarrow{f\circ\phi}\fil{Y}$, alors $s\pi_n(f)=s\pi_n(g)$.
\end{prop} 

\begin{corollaire}\label{EquivalenceHomotopieIsoTop}
Soient $f\colon \fil{X}\to\fil{Y}$ une équivalence d'homotopie filtrée et $\phi\colon \RealP{V}\to\fil{X}$ un pointage de $\fil{X}$. Alors, on a un isomorphisme
\begin{equation*}
s\pi_0(f)\colon s\pi_0\fil{X}\to s\pi_0\fil{Y}
\end{equation*}
et pour tout $n\geq 1$, on a un isomorphisme
\begin{equation*}
s\pi_n(f)\colon s\pi_n(\fil{X},\phi)\to s\pi_n(\fil{Y},f\circ\phi)
\end{equation*}
\end{corollaire}

\section{Théorèmes de Whitehead filtrés}
\label{SectionTheoremesWhitehead}
\subsection{Objets fibrants}
\label{SubsectionObjetsFibrantsNaifTopP}
On rappelle que si $\fil{X}$ est un espace filtré tel que $\Sing_P(X)$ est un objet fibrant de $\sS_P$, on a pour tout pointage $\phi$, $s\pi_n(\fil{X},\phi)\simeq s\pi_n(\Sing_P(X),\widehat{\phi})$ (voir Proposition \ref{IsomorphismeGroupesHomotopieSingP}). On souhaite donc identifier un critère permettant de garantir que l'ensemble simplicial filtré $\Sing_P(X)$ sera fibrant, afin d'appliquer le théorème \ref{EquivalenceFaibleIsoGroupeHomotopie} pour caractériser les équivalences d'homotopies filtrées entre espaces filtrés. Dans cette sous-section, on dira qu'un espace filtré $\fil{X}$ est fibrant si $\Sing_P\fil{X}$ est fibrant. Attention, cette notion ne correspond a priori pas à la définition d'une classe d'objets fibrants dans une catégorie modèle.

\begin{exemple}\label{ExempleEspaceNonFibrant}
Soit $P=\{p_0<p_1\}$ un poset linéaire à $2$ éléments. On considère le $2$-simplexe filtré $[p_0,p_0,p_1]=\Delta^{\varphi}$, et $\Lambda^{\varphi}_1$ le cornet obtenu à partir de $\Delta^{\varphi}$ en enlevant la face $d_1(\Delta^{\varphi})$ ainsi que le simplexe de dimension maximale. On observe que $\Lambda^{\varphi}_1\to\Delta^{\varphi}$ est une inclusion de cornet admissible. Par ailleurs, l'espace filtré $\RealP{\Lambda^{\varphi}_1}$ n'est pas fibrant. Plus particulièrement il n'existe pas de solution au problème de relèvement suivant :
\begin{equation*}
\begin{tikzcd}
\Lambda^{\varphi}_1
\arrow{r}{\widehat{\Id}}
\arrow{d}
&\Sing_P(\RealP{\Lambda^{\varphi}_1})
\\
\Delta^{\varphi}
\arrow[dashed]{ur}
\end{tikzcd}
\end{equation*}
où $\widehat{\Id}$ est obtenu par adjonction à partir de $\Id\colon \RealP{\Lambda^{\varphi}_1}\to\RealP{\Lambda^{\varphi}_1}$. En effet, par adjonction, ce problème de relèvement est équivalent au problème suivant :
\begin{equation*}
\begin{tikzcd}
\RealP{\Lambda^{\varphi}_1}
\arrow{r}{\Id}
\arrow{d}
&\RealP{\Lambda^{\varphi}_1}
\\
\RealP{\Delta^{\varphi}}
\arrow[dashed]{ur}{f}
\end{tikzcd}
\end{equation*}
Une solution à ce problème est une section de l'inclusion $\RealP{\Lambda^{\varphi}_1}\to\RealP{\Delta^{\varphi}}$, cependant, il n'existe pas de telle section filtrée. En effet, supposons qu'il existe une telle section $f\colon \RealP{\Delta^{\varphi}}\to\RealP{\Lambda^{\varphi}_1}$. Soit $U\subset \RealP{\Lambda^{\varphi}_1}$ un ouvert inclus dans l'intérieur de $\RealP{d_2(\Delta^{\varphi})}$. Alors, par continuité, $f^{-1}(U)$ est un ouvert de $\RealP{\Delta^{\varphi}}$, mais comme $f$ est filtré, on doit avoir $\varphi(f^{-1}(U))=\varphi(U)=\{p_0\}$. On en déduit que $f^{-1}(U)\subseteq \varphi^{-1}(\{p_0\})=\RealP{d_2(\Delta^{\varphi})}$ ce qui contredit le fait que $f^{-1}(U)$ est ouvert. Cependant, on verra comme une conséquence des résultats de cette section que $\RealP{\Delta^{\varphi}}$ est fibrant. La figure \ref{FigureCornetNonFibrant} illustre la situation.
\end{exemple}

\begin{figure}[h]
\centering
\begin{tikzpicture}[scale = 1.5]
\node (f) at (4,1) {$f$};
\node  (U) at (5.3,0.5) {$U$};
\draw[blue, thick,opacity=0.5](0,1)--(1,1);
\draw[red, thick](0,0)--(0,1.014);
\draw[->] (1.1,0.5)--(1.8,0.5);
\filldraw[blue,blue,opacity=0.5](2,1)--(3,1)--(2,0)--(2,1);
\draw[red,thick](2,0)--(2,1.005);
\draw[->, dashed] (3.1,0.5)--(4.8,0.5);
\draw[blue, thick,opacity=0.5](5,1)--(6,1);
\draw[red, thick](5,0)--(5,1.014);
\draw[black] (4.9,0.25)--(5.1,0.25);
\draw[black] (4.9,0.75)--(5.1,0.75);

\end{tikzpicture}
\caption{L'inclusion de cornet $\RealP{\Lambda^{\varphi}_1}\to \RealP{\Delta^{\varphi}}=\RealP{[p_0,p_0,p_1]}$, ainsi que la "section" $f$.}
\label{FigureCornetNonFibrant}
\end{figure}

\begin{defin}\label{DefinitionCone}
Soit $(L,\varphi_L\colon L\to Q)$ un espace filtré sur le poset $Q$. On note $c(Q)$ le poset obtenu en ajoutant à $Q$ un élément minimal $-\infty$, et on définit le cône de $\fil{L}$ comme l'espace filtré sur le poset $c(Q)$ obtenu comme
\begin{equation*}
c(L)=L\times [0,1[/(L\times\{0\})
\end{equation*}
où la filtration est donnée par
\begin{align*}
\varphi_{c(L)}\colon c(L)&\to c(Q)\\
(x,s)&\mapsto \left\{\begin{array}{cl}
\varphi_L(x)& \text{ si $s\not =0$}\\
-\infty &\text{ si $s=0$}
\end{array}\right.
\end{align*}
\end{defin} 

\begin{defin}[{\cite[Definition A.4.10]{HigherAlgebra}}]\label{DefinitionConiquementStratifie}
Soient $\fil{X}$ un espace filtré, et $x\in X$ un point de $X$. Notons $\varphi_X(x)=p$ et $P_{>p}\subset P$ pour le sous ensemble ordonné contenant tous les éléments de $P$ strictement supérieurs à $P$. On dit que $\fil{X}$ est coniquement stratifié en $x$ s'il existe 
\begin{itemize}
\item un espace filtré $(L,\varphi_L\colon L\to P_{>p})$,
\item un espace topologique $V$ 
\item un voisinage de $x$, $\fil{U}$ 
\item un homéomorphisme filtré $\fil{U}\simeq V\otimes c(\fil{L})$.
\end{itemize} 
Ici, on voit $c(\fil{L})$ comme un espace filtré sur $P$ via l'inclusion $c(P_{>p})\to P$ envoyant $-\infty$ sur $p$.
L'espace filtré $\fil{X}$ est coniquement stratifié s'il est coniquement stratifié en tout point.
\end{defin}

\begin{exemple}
Les diverses notions de pseudo-variétés considérées au chapitre \ref{ChapterEspacesStratifies} fournissent des exemples d'espaces coniquement stratifiés.
\end{exemple}

L'objet de cette sous-section est de prouver le résultat suivant. 

\begin{prop}\label{ConiquementStratifieImpliqueFibrant}
Soit $\fil{X}$ un espace filtré coniquement stratifié. Alors, $\Sing_P(X)$ est un objet fibrant de $\sS_P$.
\end{prop}

La preuve repose sur le théorème \cite[Theorem A.6.4]{HigherAlgebra}, que l'on rappelle ici.

\begin{theo}\label{LurieSingQuasiCat}
Soit $\fil{X}$ un espace filtré coniquement stratifié. Alors $\Sing_P(X)$ est une quasi-catégorie et $\Sing_P(X)\to N(P)$ est une fibration intérieure.
\end{theo}

On en déduit une preuve de la proposition \ref{ConiquementStratifieImpliqueFibrant}.

\begin{proof}
Soient $\fil{X}$ un espace filtré coniquement stratifié et $\Lambda^{\varphi}_k\to \Delta^{\varphi}=(\Delta^n,\varphi)$ une inclusion de cornet admissible. On considère le problème de relèvement suivant.
\begin{equation*}
\begin{tikzcd}
\Lambda^{\varphi}_k
\arrow{d}
\arrow{r}
&\Sing_P(X)
\\
\Delta^{\varphi}
\arrow[dashed]{ur}
\end{tikzcd}
\end{equation*}
Si $0<k<n$, il existe un relèvement par le Théorème \ref{LurieSingQuasiCat}. Si $n=1$, $\Lambda^{\varphi}_k$ est un $0$-simplexe, et comme $\Lambda^{\varphi}_k\to\Delta^{\varphi}$ est admissible, il existe une section $\Delta^{\varphi}\to \Lambda^{\varphi}_k$. Ceci fournit une solution au problème de relèvement. Si $0=k<n$, par adjonction, on considère le problème de relèvement équivalent
\begin{equation*}
\begin{tikzcd}
\RealP{\Lambda^{\varphi}_0}
\arrow{d}
\arrow{r}
&\fil{X}
\\
\RealP{\Delta^{\varphi}}
\arrow[dashed]{ur}
\end{tikzcd}
\end{equation*}
Comme $\Lambda^{\varphi}_0\to\Delta^{\varphi}$ est admissible, on a $\varphi(e_0)=\varphi(e_1)$. En particulier, l'homéomorphisme affine défini par
\begin{align*}
f\colon \RealP{\Delta^{\varphi}}&\to\RealP{\Delta^{\varphi}}\\
e_i&\mapsto\left\{\begin{array}{cl}
e_0 &\text{ si $i=1$}\\
e_1 &\text{ si $i=0$}\\
e_i &\text{ sinon}
\end{array}\right.
\end{align*}
fournit un homéomorphisme filtré.
La restriction de $f$ à $\RealP{\Lambda^{\varphi}_0}$ fournit un homéomorphisme 
\begin{equation*}
f\colon \RealP{\Lambda^{\varphi}_0}\xrightarrow{\simeq}\RealP{\Lambda^{\varphi}_1}
\end{equation*}
On remarque par ailleurs que $f=f^{-1}$. Finalement, on a la situation suivante :
\begin{equation*}
\begin{tikzcd}
\RealP{\Lambda^{\varphi}_1}
\arrow{r}{f}
\arrow{d}
&\RealP{\Lambda^{\varphi}_0}
\arrow{d}
\arrow{r}
&\fil{X}
\\
\RealP{\Delta^{\varphi}}
\arrow[swap]{r}{f}
\arrow[dashed, crossing over,near start]{urr}{h}
&\RealP{\Delta^{\varphi}}
\end{tikzcd}
\end{equation*}
Finalement, comme $0<1<n$, on peut appliquer le Théorème \ref{LurieSingQuasiCat} pour obtenir un relèvement $h$. (L'adjoint de) la composée $h\circ f$ fournit une solution au problème de relèvement initial. Le cas $0<k=n$ est similaire.
\end{proof}

\subsection{Premier théorème de Whitehead filtré}
L'objet de cette sous-section est de prouver le théorème suivant.

\begin{theo}\label{PremierTheoremeWhitehead}
Soient $\fil{X},\fil{Y}$ deux espaces filtrés et $f\colon \fil{X}\to\fil{Y}$ une application filtrée. On suppose que 
\begin{itemize}
\item Il existe $\fil{A},\fil{B}$ deux ensembles simpliciaux filtrés tels que $\fil{X}\simeq\RealP{\fil{A}}$ et $\fil{Y}\simeq\RealP{\fil{B}}$. (On ne suppose pas que $f$ est la réalisation d'une application simpliciale).
\item Les ensembles simpliciaux filtrés $\Sing_P(X)$ et $\Sing_P(Y)$ sont fibrants.
\end{itemize}
Alors, $f$ est une équivalence d'homotopie filtrée si et seulement si 
\begin{equation*}
s\pi_0(f)\colon s\pi_0\fil{X}\to s\pi_0\fil{Y}
\end{equation*}
est un isomorphisme, et, pour tout pointage $\phi\colon\RealP{V}\to\fil{X}$ et pour tout $n\geq 1$, les morphismes
\begin{equation*}
s\pi_n(f)\colon s\pi_n(\fil{X},\phi)\to s\pi_n(\fil{Y},f\circ \phi)
\end{equation*}
sont des isomorphismes.
\end{theo}

\begin{remarque}
Comme dans le cas simplicial (voir Corollaire \ref{CorollaireEquivalenceFaiblePointageComplet}) il suffit de tester les morphismes $s\pi_n(f)$ sur un ensemble complet de pointages de $\fil{X}$. 
\end{remarque}

\begin{proof}
Le sens direct est le corollaire \ref{EquivalenceHomotopieIsoTop}. 
La preuve de la réciproque repose sur la construction du diagramme suivant
\begin{equation}\label{DiagrammePreuveWhitehead1}
\begin{tikzcd}[column sep = huge]
\phantom{X}
&\RealP{\Sing_P(Y)}
\arrow{r}{\RealP{\widetilde{g}}}
&\RealP{\Sing_P{X}}
\arrow{r}{\RealP{\Sing_P(f)}}
\arrow{d}{\ev_X}
&\RealP{\Sing_P{Y}}
\arrow{d}{\ev_Y}
\\
\fil{X}
\arrow{d}{i_X}
\arrow{r}{h}
&\fil{Y}
\arrow{r}{g}
\arrow{u}{i_Y}
&\fil{X}
\arrow{r}{f}
&\fil{Y}
\\
\RealP{\Sing_P{X}}
\arrow{r}{\RealP{\widetilde{h}}}
&\RealP{\Sing_P{Y}}
\arrow{u}{\ev_Y}
\arrow{r}{\RealP{\Sing_P(g)}}
&\RealP{\Sing_P{X}}
\arrow{u}{\ev_X}
\end{tikzcd}
\end{equation}
Soit $(\phi_i,\fil{V_i})_{i\in I}$ un ensemble complet de pointages de $\fil{X}$. On suppose que $f\colon \fil{X}\to\fil{Y}$ induit des isomorphismes
\begin{equation*}
s\pi_0(f)\colon s\pi_0\fil{X}\to s\pi_0\fil{Y}
\end{equation*}
 et
\begin{equation*}
s\pi_n(f)\colon s\pi_n(\fil{X},\phi_i)\to s\pi_n(\fil{Y},f\circ \phi_i)
\end{equation*}
pour tout $n\geq 1$. Alors, comme les ensembles simpliciaux filtrés $\Sing_P(X)$ et $\Sing_P(Y)$ sont fibrants,  par la proposition \ref{IsomorphismeGroupesHomotopieSingP} on a aussi les isomorphismes
\begin{equation*}
s\pi_0(\Sing_P(f))\colon s\pi_0(\Sing_P(X))\to s\pi_0(\Sing_P(Y))
\end{equation*}
 et
\begin{equation*}
s\pi_n(\Sing_P(f))\colon s\pi_n(\Sing_P(f),\widehat{\phi_i})\to s\pi_n(\Sing_P(Y),\Sing_P(f)\circ \widehat{\phi_i})
\end{equation*}
pour tout $n\geq 1$. En particulier, par le corollaire \ref{CorollaireEquivalenceFaiblePointageComplet}, et la remarque \ref{PointageCompletTopSimplicial}, $\Sing_P(f)$ est une équivalence faible. Comme c'est une équivalence faible entre objets fibrants et cofibrants, c'est une équivalence d'homotopie de $\sS_P$. On en déduit qu'il existe une application filtrée $\widetilde{g}\colon \Sing_P(Y)\to\Sing_P(X)$ telle que $\Sing_P(f)\circ\widetilde{g}$ est homotope à $\Id$ par une homotopie filtrée $\widetilde{H}\colon \Delta^1\otimes\Sing_P(Y)\to\Sing_P(Y)$. 
Notons $\ev\colon \RealP{\Sing_P(-)}\to\Id$ l'unité de l'adjonction $(\RealP{-},\Sing_P)$ et $i\colon \RealP{-}\to\RealP{\Sing_P(\RealP{-})}$ l'image par $\RealP{-}$ de la co-unité de cette même adjonction. Alors, on définit l'application $g\colon \fil{Y}\to\fil{X}$ comme la composée $g=\ev_X\circ\RealP{\widetilde{g}}\circ i_Y$ ($i_Y$ est bien définie car par hypothèse $\fil{Y}\simeq \RealP{\fil{B}}$). On a construit les deux carrés supérieurs du diagramme \ref{DiagrammePreuveWhitehead1}. Montrons que $f\circ g$ est homotope par une homotopie filtrée à $\Id_Y$. On définit $H\colon \Delta^1\otimes \fil{Y}\to\fil{Y}$ comme la composée
\begin{equation*}
\Delta^1\otimes\fil{Y}\xrightarrow{\Delta^1\otimes i_Y} \Delta^1\otimes\RealP{\Sing_P(Y)}\xrightarrow{\RealP{\widetilde{H}}}\RealP{\Sing_P(Y)}\xrightarrow{\ev_Y}\fil{Y}.
\end{equation*}
Par hypothèse, on a $\widetilde{H}_{|\{0\}\otimes\Sing_P(Y)}=\widetilde{g}\circ\Sing_P(f)$. Par commutativité du diagramme, on en déduit que 
\begin{align*}
H_{|\{0\}\otimes\fil{Y}}&=\ev_Y\circ \widetilde{H}_{|\{0\}\otimes\Sing_P(Y)}\circ(\{0\}\otimes i_Y)\\
&=\ev_Y\circ\RealP{\Sing_P(f)}\circ\RealP{\widetilde{g}}\circ i_Y\\
&= f\circ\ev_X\circ\RealP{\widetilde{g}}\circ i_Y\\
&= f\circ g
\end{align*}
D'autre part, on calcule 
\begin{align*}
H_{|\{1\}\otimes\fil{Y}}&=\ev_Y\circ \widetilde{H}_{|\{1\}\otimes\Sing_P(Y)}\circ(\{0\}\otimes i_Y)\\
&=\ev_Y\circ\Id\circ i_Y\\
&= \Id_Y
\end{align*}
On en déduit que $g\circ f$ est homotope par une homotopie filtrée à $\Id_Y$. En particulier, par application de la proposition \ref{ApplicationsHomotopesMorphismesEgaux}, on en déduit qu'on a l'isomorphisme
\begin{equation*}
s\pi_0(g)\colon s\pi_0(\fil{Y}\to s\pi_0(\fil{X}),
\end{equation*}
et que $s\pi_0(g)=(s\pi_0(f))^{-1}$. Par ailleurs, si $(\phi_i,\fil{V_i})_{i\in I}$ est un ensemble de pointages complet de $X$, alors $(f\circ\phi_i,\fil{V_i})_{i\in I}$ est un ensemble de pointages complet de $\fil{Y}$. On en déduit, par les propositions \ref{ApplicationsHomotopesMorphismesEgaux} et \ref{PointagesHomotopesIsomorphismes}, que pour tout $n\geq 0$ et pour tout $i\in I$, on a les isomorphismes
\begin{equation*}
s\pi_n(g)\colon s\pi_n(\fil{Y},f\circ\phi_i)\to s\pi_n(\fil{X},g\circ f\circ \phi_i).
\end{equation*} 
En particulier, $g$ vérifie les mêmes hypothèses que $f$. On peut donc itérer les constructions précédentes pour obtenir $\widetilde{h}\colon \Sing_P(X)\to \Sing_P(Y)$ tel que $\Sing_P(g)\circ\widetilde{h}$ est homotope à $\Id$ par une homotopie filtrée $\widetilde{G}\colon \Delta^1\otimes\Sing_P(X)\to\Sing_P(X)$. Puis on définit $h=\ev_Y\circ\RealP{\widetilde{h}}\circ i_X$, et on note $G$ l'homotopie entre $g\circ h$ et $\Id$ obtenue à partir de $\widetilde{G}$. On obtient finalement le diagramme commutatif \ref{DiagrammePreuveWhitehead1}. Montrons maintenant que $g\circ f$ est homotope à $\Id_X$ par une homotopie filtrée. On a la suite d'homotopie suivante
\begin{align*}
g\circ f&\sim_{g\circ f\circ G} g\circ f\circ g\circ h\\
&\sim_{g\circ H\circ (\Delta^1\otimes h)} g\circ h\\
&\sim_{G} \Id_X
\end{align*}
où on note $f\sim_{H} g$ pour expliciter que $f$ est homotope à $g$ par l'homotopie filtrée $H$. Finalement, $g$ est un inverse à gauche et à droite de $f$ à homotopie filtrée près, $f$ est donc une équivalence d'homotopie filtrée.
\end{proof}

\begin{remarque}\label{PremierTheoremeWhiteheadAHomotopiePres}
L'hypothèse $\fil{X}\simeq \RealP{\fil{A}}$ et $\fil{Y}\simeq \RealP{\fil{B}}$ du théorème \ref{PremierTheoremeWhitehead} peut être affaiblie. En effet, il suffit d'imposer l'existence d'équivalences d'homotopies filtrées entre $\fil{X}$ et $\RealP{\fil{A}}$ et entre $\fil{Y}$ et $\RealP{\fil{B}}$. On remarque que dans la preuve précédente, cette hypothèse est utilisée pour définir l'application $i_Y$ comme la composition
\begin{equation*}
\fil{Y}\simeq \RealP{\fil{B}}\to \RealP{\Sing_P(\RealP{\fil{B}})}\simeq \RealP{\Sing_P(Y)}
\end{equation*}
où l'application 
\begin{equation*}
\RealP{\fil{B}}\to \RealP{\Sing_P(\RealP{\fil{B}})}
\end{equation*}
est la réalisation de l'unité de l'adjonction $\RealP{-},\Sing_P$. Si on suppose seulement l'existence d'une équivalence d'homotopie filtrée
\begin{equation*}
\alpha\colon \fil{Y}\to \RealP{\fil{A}}
\end{equation*}
On peut définir $i_Y$ comme la composée
\begin{equation*}
\fil{Y}\xrightarrow{\alpha} \RealP{\fil{B}}\to \RealP{\Sing_P(\RealP{\fil{B}})}\xrightarrow{\RealP{\Sing_P(\alpha^{-1})}} \RealP{\Sing_P(Y)}
\end{equation*}
Où $\alpha^{-1}$ est un inverse à homotopie filtré près de $\alpha$.
En composant l'homotopie entre $\alpha^{-1}\circ \alpha$ et l'identité de $\fil{Y}$ avec les homotopies apparaissant dans la preuve, on obtient toujours que $g$ est un inverse à gauche de $f$. En procédant de même pour $\fil{X}$ et $h$, on arrive au résultat.
\end{remarque}

\subsection{Squelette des groupes d'homotopies filtrés et deuxième théorème de Whitehead filtré}

\begin{defin}
Soient $\fil{X}$ un espace filtré et $p\in P$. On définit la $p$-strate de $\fil{X}$ comme 
\begin{equation*}
X_p=\varphi_{X}^{-1}({p})\simeq C^0_P(\RealP{[p]},\fil{X})
\end{equation*}
De plus, si $[p_0,p_1]$ est un simplexe non dégénéré de $N(P)$, (c'est à dire $p_0<p_1\in P$), alors on définit le $[p_0,p_1]$-entrelac homotopique (noté $\Hol$ pour "Homotopy link") de $\fil{X}$ comme 
\begin{equation*}
\Hol_{[p_0,p_1]}(\fil{X})=C^0_P(\RealP{[p_0,p_1]},\fil{X})
\end{equation*}
\end{defin}

Dans cette sous section, on montre le théorème suivant 

\begin{theo}\label{DeuxiemeTheoremeWhitehead}
Soient $\fil{X},\fil{Y}$ deux espaces filtrés et $f\colon\fil{X}\to\fil{Y}$ une application filtrée. On suppose que
\begin{itemize}
\item il existe $\fil{A},\fil{B}$ deux ensembles simpliciaux filtrés tels que $\fil{X}\simeq\RealP{\fil{A}}$ et $\fil{Y}\simeq\RealP{\fil{B}}$, (on ne suppose pas que $f$ provient d'une application simpliciale),
\item les espaces filtrés $\fil{X}$ et $\fil{Y}$ sont coniquement stratifiés.
\end{itemize}
Alors, $f$ est une équivalence d'homotopie filtrée si et seulement si $f$ induit une équivalence faible sur chacune des strates et sur chacun des entrelacs homotopiques.
\end{theo}

Pour prouver le théorème \ref{DeuxiemeTheoremeWhitehead}, nous aurons besoin des lemmes suivants :

\begin{lemme}\label{EquivalenceFaibleSimplexeNonDegenere}
Soient $\fil{X}$ un espace filtré, et $\Delta^{\varphi}$ un simplexe de $N(P)$. Notons $\overline{\Delta^{\varphi}}$ l'unique simplexe non dégénéré de $N(P)$ tel que $\Delta^{\varphi}$ est une dégénérescence de $\overline{\Delta^{\varphi}}$. Alors, la dégénérescence $\Delta^{\varphi}\to\overline{\Delta^{\varphi}}$ induit une équivalence faible 
\begin{equation*}
\Map(\RealP{\overline{\Delta^{\varphi}}},\fil{X})\to \Map(\RealP{\Delta^{\varphi}},\fil{X})
\end{equation*}
\end{lemme}

\begin{proof}
Si $\Sing_P(X)$ est fibrant, c'est une application directe du corollaire \ref{CorollaireSimplexesNonDegeneressPin} et de la proposition \ref{AdjonctionSimplicialeSingReal}. Dans le cas général, il suffit de voir que le morphisme $\Delta^{\varphi}\to\overline{\Delta^{\varphi}}$ est une équivalence d'homotopie filtrée, et que $\RealP{-}$ et $\Map(-,\fil{X})$ préserve les équivalences d'homtopies (filtrées).
\end{proof}

\begin{defin}
Soit $\Delta^{\varphi}=(\Delta^n,\varphi)=[p_0,\dots,p_n]$ un simplexe non dégénéré de $N(P)$. On définit son squelette extérieur $\OSk(\Delta^{\varphi})\subseteq\Delta^{\varphi}$ comme le sous ensemble simplicial filtré de $\Delta^{\varphi}$ contenant les sommets de $\Delta^{\varphi}$ et les $1$-simplexes de la forme $[p_i,p_{i+1}]$. En particulier, $\OSk(\Delta^{\varphi})=(\OSk(\Delta^n),\varphi_{|\OSk(\Delta^n)})$.
\end{defin}

\begin{lemme}\label{EquivalenceFaibleSqueletteExterieur}
Soient $\fil{X}$ un espace filtré coniquement stratifié, et $\Delta^{\varphi}$ un simplexe non dégénéré de $N(P)$. Alors, l'inclusion $\OSk(\Delta^{\varphi})\to\Delta^{\varphi}$ induit une équivalence faible
\begin{equation*}
\Map(\RealP{\Delta^{\varphi}},\fil{X})\to\Map(\RealP{\OSk(\Delta^{\varphi})},\fil{X}).
\end{equation*}
\end{lemme}

\begin{proof}
Par la proposition \ref{AdjonctionSimplicialeSingReal}, il suffit de montrer que 
\begin{equation*}
\Map(\Delta^{\varphi},\Sing_P(X))\to\Map(\OSk(\Delta^{\varphi}),\Sing_P(X))
\end{equation*}
est une équivalence faible. Par la proposition \ref{ConiquementStratifieImpliqueFibrant}, $\Sing_P(X)$ est fibrant et par définition $\OSk(\Delta^{\varphi})\to\Delta^{\varphi}$ est une cofibration. On déduit du théorème \ref{CategorieModelSimpliciale} que ce morphisme est une fibration, il suffit donc de montrer que c'est une fibration triviale. Considérons le problème de relèvement suivant
\begin{equation*}
\begin{tikzcd}
\partial(\Delta^n)
\arrow{r}{\alpha}
\arrow{d}
&\Map(\Delta^{\varphi},\Sing_P(X))
\arrow{d}
\\
\Delta^n
\arrow[swap]{r}{\beta}
\arrow[dashed]{ur}
&\Map(\OSk(\Delta^{\varphi}),\Sing_P(X))
\end{tikzcd}
\end{equation*}
Par adjonction, la donnée d'un tel diagramme correspond à la donnée d'un morphisme 
\begin{equation*}
\partial(\Delta^n)\otimes\Delta^{\varphi}\cup_{\partial(\Delta^n)\otimes\OSk(\Delta^{\varphi})}\Delta^n\otimes\OSk(\Delta^{\varphi})\xrightarrow{\widehat{\alpha}\cup\widehat{\beta}}\fil{X}
\end{equation*}
Et une solution au problème de relèvement précédent correspond à un relèvement dans le diagramme suivant
\begin{equation*}
\begin{tikzcd}
\partial(\Delta^n)\otimes\Delta^{\varphi}\cup_{\partial(\Delta^n)\otimes\OSk(\Delta^{\varphi})}\Delta^n\otimes\OSk(\Delta^{\varphi})
\arrow[swap]{d}{j}
\arrow{r}{\widehat{\alpha}\cup\widehat{\beta}}
&\fil{X}
\\
\Delta^n\otimes\Delta^{\varphi}
\arrow[dashed, swap]{ur}{h}
\end{tikzcd}
\end{equation*}
On remarque par ailleurs que l'inclusion $j\colon\partial(\Delta^n)\otimes\Delta^{\varphi}\cup_{\partial(\Delta^n)\otimes\OSk(\Delta^{\varphi})}\Delta^n\otimes\OSk(\Delta^{\varphi})\to \Delta^n\otimes\Delta^{\varphi}$ est une surjection sur les sommets. En particulier, toute application simpliciale $h\colon \Delta^n\otimes\Delta^{\varphi}\to X$ faisant commuter le diagramme précédent sera nécessairement filtrée.
Par \cite[Theorem A.6.4]{HigherAlgebra} (voir Théorème \ref{ConiquementStratifieImpliqueFibrant}), on sait que $\Sing_P(X)$ est un objet fibrant pour la structure de Joyal sur $\sS$. En particulier, il suffit de montrer que l'inclusion $j$ est une cofibration triviale pour la structure de Joyal sur $\sS$. Considérons maintenant le diagramme commutatif suivant
\begin{equation*}
\begin{tikzcd}
\partial(\Delta^n)\otimes\OSk(\Delta^{\varphi})
\arrow{d}
\arrow{r}
&
\partial(\Delta^n)\otimes\Delta^{\varphi}
\arrow{d}
\arrow[bend left = 20]{ddr}
&\phantom{X}
\\
\Delta^n\otimes\OSk(\Delta^{\varphi})
\arrow{r}
\arrow[bend right = 12]{drr}
&\partial(\Delta^n)\otimes\Delta^{\varphi}\cup_{\partial(\Delta^n)\otimes\OSk(\Delta^{\varphi})}\Delta^n\otimes\OSk(\Delta^{\varphi})
\arrow{dr}{j}
&\phantom{X}
\\
\phantom{X}
&\phantom{X}
&\Delta^n\otimes\Delta^{\varphi}
\end{tikzcd}
\end{equation*} 
Le carré est cocartésien, par construction.
Notons $\Delta^{\varphi}=(\Delta^N,\varphi)$, l'inclusion $\partial(\Delta^n)\times\OSk(\Delta^N))\to\partial(\Delta^n)\times\Delta^N$ est une cofibration triviale dans la structure de Joyal, par \cite[Proposition 2.13]{Joyal} et \cite[Theorem 6.12]{Joyal}. De même, l'inclusion $\partial(\Delta^n)\times\OSk(\Delta^N)\to\Delta^n\times\OSk(\Delta^N)$ est une cofibration. On déduit du théorème \cite[Theorem 6.12]{Joyal} que le morphisme $j$ est une cofibration triviale dans la structure de Joyal.
En particulier, il existe une solution $h$ au problème de relèvement précédent et donc le morphisme 
\begin{equation*}
\Map(\Delta^{\varphi},\Sing_P(X))\to\Map(\OSk(\Delta^{\varphi}),\Sing_P(X))
\end{equation*}
est une fibration triviale.
\end{proof}

\begin{lemme}\label{LemmeEquivalenceFaible1Squelette}
Soient $f\colon \fil{A}\to\fil{B}$ une application filtrée entre deux ensembles simpliciaux filtrés fibrants, et $\Delta^{\varphi}=[p_0,\dots,p_n]\in N(P)$ un simplexe non dégénéré. 
Si, pour tout $0\leq i\leq n$  et pour tout $0\leq j\leq n-1$, $f$ induit des équivalence faibles
\begin{equation*}
\Map([p_j,p_{j+1}],\fil{A})\to\Map([p_j,p_{j+1}],\fil{B})
\end{equation*}
et
\begin{equation*}
\Map([p_i],\fil{A})\to\Map([p_i],\fil{B})
\end{equation*}
Alors, $f$ induit une équivalence faible 
\begin{equation*}
\Map(\OSk(\Delta^{\varphi}),\fil{A})\to\Map(\OSk(\Delta^{\varphi}),\fil{B})
\end{equation*}
\end{lemme}

\begin{proof}
Comme pour tout ensemble simplicial filtré, on a 
\begin{equation*}
\OSk(\Delta^{\varphi})=\colim_{\sigma\colon \Delta^{\psi}\to\OSk(\Delta^{\varphi})}\Delta^{\psi}
\end{equation*}
De plus, on sait que $\OSk(\Delta^{\varphi})$ est engendré par les simplexes non dégénérés $[p_j,p_{j+1}]$ et $[p_i]$. En particulier, en notant $S$ la catégorie des simplexes de la forme $[p_i]$ ou $[p_j,p_{j+1}]$, on a 
\begin{equation*}
\OSk(\Delta^{\varphi})=\colim_{\Delta^{\psi}\in S}\Delta^{\psi}
\end{equation*} 
On calcule maintenant 
\begin{align*}
\Map(\OSk(\Delta^{\varphi}),\fil{A})&\simeq\Map(\colim_{S}\Delta^{\psi},\fil{A})\\
&\simeq\lim_{S}\Map(\Delta^{\psi},\fil{A})
\end{align*}
De plus, les seuls morphismes non triviaux de $S$ sont de la forme $[p_i]\subseteq [p_j,p_{j+1}]$ avec $i=j$ ou $j+1$ (Cela vient du fait que $\Delta^{\varphi}$ est non dégénéré, et donc $p_j<p_{j+1}$ pour tout $j$). En particulier, ce sont des cofibrations. Comme $\sS_P$ est une catégorie modèle simpliciale et que $\fil{A}$ est fibrant, on en déduit que tous les morphismes apparaissant dans la limite
\begin{equation*}
\lim_{S}\Map(\Delta^{\psi},\fil{A})
\end{equation*}
sont des fibrations, et tous les objets sont fibrants.
Finalement, on a
\begin{equation*}
\Map(\OSk(\Delta^{\varphi},\fil{A})\simeq\lim_{S}\Map(\Delta^{\psi},\fil{A})\sim \holim_{S}\Map(\Delta^{\psi},\fil{A})
\end{equation*}
et de même pour $\fil{B}$. Par ailleurs, on sait par hypothèse que pour tout $\Delta^{\psi}\in S$, $f$ induit une équivalence faible
\begin{equation*}
\Map(\Delta^{\psi},\fil{A})\to\Map(\Delta^{\psi},\fil{B})
\end{equation*}
On en déduit que $f$ induit une équivalence faible 
\begin{equation*}
\holim_{S}\Map(\Delta^{\psi},\fil{A})\to\holim_{S}\Map(\Delta^{\psi},\fil{B}),
\end{equation*}
d'où le résultat voulu.
\end{proof}

\begin{proof}[Démonstration du Théorème \ref{DeuxiemeTheoremeWhitehead}]
Soient $\fil{X},\fil{Y}$ et $f\colon\fil{X}\to\fil{Y}$ deux espaces filtrés et une application filtrée vérifiant les hypothèses du théorème \ref{DeuxiemeTheoremeWhitehead}. Supposons d'abord que $f$ est une équivalence d'homotopie filtrée. Alors, il existe $g\colon\fil{Y}\to\fil{X}$ une application filtrée, et des homotopies filtrées $H\colon\Delta^1\otimes\fil{X}\to\fil{X}$ et $G\colon\Delta^1\otimes\fil{Y}\to\fil{Y}$ respectivement entre $g\circ f$ et $\Id_X$ et entre $f\circ g$ et $\Id_Y$. En particulier, les restrictions de $H$ et $G$ à $X_p$ et $Y_p$, montrent que les restrictions de $f$ et $g$ induisent des équivalences d'homotopies entre les strates de $X$ et de $Y$. De même, étant donné $[p_0,p_1]$ un simplexe non dégénéré de $N(P)$, en composant les homotopies induites par $H$ et $G$ avec les inclusions naturelles
\begin{equation*}
\Delta^1\otimes C^0_P(\RealP{[p_0,p_1]},\fil{X})\to C^0_P(\RealP{[p_0,p_1]},\Delta^1\otimes\fil{X})
\end{equation*}
et
\begin{equation*}
\Delta^1\otimes C^0_P(\RealP{[p_0,p_1]},\fil{Y})\to C^0_P(\RealP{[p_0,p_1]},\Delta^1\otimes\fil{Y})
\end{equation*}
on obtient que $f$ induit des équivalences d'homotopies entre les entrelacs homotopiques de $X$ et ceux de $Y$.

Supposons maintenant que $f$ induit des équivalences faibles entre chacune des strates et chacun des entrelacs homotopiques de $X$ et ceux de $Y$. Montrons que $f$ induit un isomorphisme 
\begin{equation*}
s\pi_0(f)\colon s\pi_0\fil{X}\to s\pi_0\fil{Y}
\end{equation*}
et que pour tout pointage de $\fil{X}$, $\phi\colon \RealP{V}\to\fil{X}$ et pout tout $n\geq 1$, $f$ induit un isomorphisme 
\begin{equation*}
s\pi_n(\fil{X},\phi)\to s\pi_n(\fil{Y},\phi).
\end{equation*}
Par définition de $s\pi_n$, il suffit de montrer que pour tout simplexe $\Delta^{\varphi}\in N(P)$, $f$ induit des équivalences faibles
\begin{equation*}
\Map(\RealP{\Delta^{\varphi}},\fil{X})\to\Map(\RealP{\Delta^{\varphi}},\fil{Y}).
\end{equation*}
Par le lemme \ref{EquivalenceFaibleSimplexeNonDegenere}, il suffit de le vérifier pour $\Delta^{\varphi}$ un simplexe non dégénéré. Par le lemme \ref{EquivalenceFaibleSqueletteExterieur}, comme $\fil{X}$ et $\fil{Y}$ sont coniquements stratifiés par hypothèse, il suffit de montrer que $f$ induit des équivalences faibles
\begin{equation*}
\Map(\RealP{\OSk(\Delta^{\varphi})},\fil{X})\to\Map(\RealP{\OSk(\Delta^{\varphi})},\fil{Y}).
\end{equation*}
Pour tout simplexe non dégénéré $\Delta^{\varphi}\in N(P)$.
En utilisant l'adjonction simpliciale (voir proposition \ref{AdjonctionSimplicialeSingReal}), on en déduit qu'il suffit de montrer que pour tout simplexe non dégénéré $\Delta^{\varphi}\in N(P)$, $f$ induit une équivalence faible
\begin{equation*}
\Map(\OSk(\Delta^{\varphi}),\Sing_P(X))\to\Map(\OSk(\Delta^{\varphi}),\Sing_P(Y)).
\end{equation*}
On sait par le théorème \ref{ConiquementStratifieImpliqueFibrant} que $\Sing_P(X)$ et $\Sing_P(Y)$ sont des objets fibrants de $\sS_P$. Par le lemme \ref{LemmeEquivalenceFaible1Squelette} il suffit donc de montrer que pour tout $[p]\in N(P)$ et pour tout simplexe non dégénéré $[p_0,p_1]\in N(P)$, $f$ induit des équivalences faibles
\begin{equation*}
\Map([p],\Sing_P(X))\to\Map([p],\Sing_P(Y)),
\end{equation*}
et
\begin{equation*}
\Map([p_0,p_1],\Sing_P(X))\to\Map([p_0,p_1],\Sing_P(Y)).
\end{equation*}
Par la proposition \ref{DiagrammesTopologiquesIsomorphes}, cela revient à montrer que $f$ induit des équivalences faibles
\begin{equation*}
X_p\simeq C^0_P(\RealP{[p]},\fil{X})\to\C^0_P(\Real{[p]},\fil{Y})\simeq Y_p
\end{equation*}
et 
\begin{equation*}
\Hol_{[p_0,p_1]}(X)= C^0_P(\RealP{[p_0,p_1]},\fil{X})\to C^0_P(\RealP{[p_0,p_1]},\fil{Y})=\Hol_{[p_0,p_1]}(Y).
\end{equation*}
Or, par hypothèse, ces applications sont des équivalences faibles. On en déduit que $f$ est une équivalence d'homotopie filtrée par application du théorème \ref{PremierTheoremeWhitehead}.
\end{proof}

\begin{remarque}
Soient $K$ un ensemble simplicial et $n\geq 0$ un entier. On appelle $n$-squelette de $K$ le sous ensemble simplicial $\sk_n(K)\subseteq K$ engendré par les simplexes non-dégénéré de $K$ de dimension inférieure à $n$. On peut étendre cette définition aux ensembles simpliciaux filtrés en posant $\sk_n(\fil{A})=(\sk_n(A),\varphi_{|\sk_n(A)})$. Avec ces notations, on peut reformuler le théorème \ref{DeuxiemeTheoremeWhitehead} sous la forme suivante. Soient $\fil{X}$, $\fil{Y}$ et $f\colon \fil{X}\to\fil{Y}$ vérifiant les hypothèses du théorème \ref{DeuxiemeTheoremeWhitehead}. Alors, $f$ induit des isomorphismes
\begin{equation*}
s\pi_0(f)\colon s\pi_0\fil{X}\to s\pi_0\fil{Y}
\end{equation*}
et
\begin{equation*}
s\pi_n(f)\colon s\pi_0(\fil{X},\phi)\to s\pi_0(\fil{Y},f\circ\phi)
\end{equation*}
pour tout pointage $\phi\colon\RealP{V}\to\fil{X}$ et pour tout entier $n\geq 1$, si et seulement si $f$ induit des isomorphismes
\begin{equation*}
s\pi_0(f)\colon \sk_1(s\pi_0\fil{X})\to \sk_1(s\pi_0\fil{Y})
\end{equation*}
et
\begin{equation*}
s\pi_n(f)\colon \sk_1(s\pi_0(\fil{X},\phi))\to \sk_1(s\pi_0(\fil{Y},f\circ\phi))
\end{equation*}
pour tout pointage $\phi\colon\RealP{V}\to\fil{X}$ et pour tout entier $n\geq 1$.
En particulier, étant donné une telle application $f$, il suffit de calculer les morphismes induits sur le $1$-squelette des groupes d'homotopies filtrés pour montrer que c'est une équivalence faible. Cependant, l'existence d'un isomorphisme (abstrait) $\Sk_1(s\pi_n\fil{X})\simeq \sk_1(s\pi_n(\fil{Y}))$ n'implique pas l'existence d'un isomorphisme (abstrait) $s\pi_n\fil{X}\simeq s\pi_n\fil{Y}$. En particulier, on peut trouver des obstructions à l'existence d'une équivalence d'homotopie $f\colon \fil{X}\to\fil{Y}$ en examinant les simplexes de dimensions supérieures à $1$ de $s\pi_n\fil{X}$ et $s\pi_n\fil{Y}$.
\end{remarque}

\begin{remarque}\label{RemarqueNandLalWhitehead}
Le théorème \ref{DeuxiemeTheoremeWhitehead} est très similaire au résultat obtenu par Miller \cite[Theorem 6.3]{Miller}. En effet, il montre que si $\fil{X}$ et $\fil{Y}$ sont des espaces homotopiquement stratifiés, et $f\colon \fil{X}\to\fil{Y}$ est une application filtrée, alors $f$ est une équivalence d'homotopie si et seulement si $f$ induit des équivalences faibles sur toutes les strates et tous les entrelacs homotopiques. Cependant, les notions d'espaces homotopiquement stratifiés et d'espaces coniquement stratifiés ne coïncident pas exactement. De plus, les méthodes employées pour prouver \cite[Theorem 6.3]{Miller} et le théorème \ref{DeuxiemeTheoremeWhitehead} sont très différentes.
\end{remarque}
\begin{remarque}
Dans \cite{NandLal}, Nand-Lal montre que si $\fil{X}$ est un espace métrique, homotopiquement stratifié et de stratification finie, $\Sing_P\fil{X}$ est une quasi-catégorie. On remarque que la preuve de la proposition \ref{ConiquementStratifieImpliqueFibrant} s'adapte pour prouver que $\Sing_P\fil{X}$ est fibrant dans $\sS_P$ dès que l'ensemble simplicial sous-jacent est une quasi-catégorie. Ceci permet de généraliser le théorème \ref{DeuxiemeTheoremeWhitehead} pour inclure cette classe d'objets. On obtient ainsi une version plus proche du résultat de Miller \cite[Theorem 6.3]{Miller}
\end{remarque}

\chapter[Exemples et applications des théorèmes de Whitehead filtrés]{Exemples et applications des théorèmes de Whitehead filtrés}\label{ChapitreExemples}
\chaptermark{Exemples}
L'objet de ce chapitre est de comprendre, à travers des exemples, comment utiliser les groupes d'homotopie filtrés pour différentier des types d'homotopie filtrés, et quel type d'information est contenu dans le type d'homotopie filtré. 

Dans la première section de ce chapitre, on s'intéresse au cas (simple) où $P=\{p_0<p_1\}$. Dans la section \ref{ConstructionFibre}, on étudie une construction générale permettant d'obtenir un espace coniquement stratifié (Voir la définition \ref{DefinitionConiquementStratifie}) à partir d'un fibré localement trivial. Cette construction a pour vertu de permettre de calculer les groupes d'homotopie filtrés explicitement (Voir les propositions \ref{GroupesHomotopiesConeOuvertFibre} et \ref{GroupesHomotopiesFibre}). De plus, sous de bonnes hypothèses sur le fibré de départ, l'espace obtenu est une pseudo-variété (Propostions \ref{FibrePseudoVariete} et \ref{GroupesHomotopiesFibre}). Dans la section \ref{SectionPseudoVarietesFibres}, on étudie de tels exemples. Dans la section \ref{SectionEilenbergMaclane} on explore ensuite la possibilité de construire des espaces d'Eilenberg-Mac Lane à partir de cette construction. On obtient un résultat partiel sous la forme de la proposition \ref{PropositionCaracterisationEilenbergMaclane}.

Dans la section \ref{SectionPlongementEspaceFiltre}, on s'intéresse à des filtrations induites par des plongements. Étant donnés un espace topologique $X$ et une collection de sous-espaces $X_i\subseteq X$, on dispose de l'ensemble ordonné des intersections des $X_i$, $P_X$. On obtient alors une filtration naturelle $X\to P_X$. La proposition \ref{PropositionPlongementPseudoVariete} permet de garantir que l'espace filtré ainsi obtenu est une pseudo-variété (et donc vérifie les hypothèses du théorème \ref{PremierTheoremeWhitehead}). On étudie ensuite le cas particulier des noeuds, et on montre que les résultats de \cite{Waldhausen} et de \cite{GordonLuecke} impliquent que le premier groupe d'homotopie filtré fournit un invariant complet des nœuds (Théorème \ref{TheoremeGroupesHomotopieFiltresNoeuds}).

Dans la section \ref{SectionFibresFiltres}, on reprend la construction de la section \ref{ConstructionFibre}, permettant de construire un espace filtré à partir d'un fibré, pour montrer qu'elle a encore un sens lorsque la fibre est elle même filtrée. Ceci permet de fournir un contexte plus général dans lequel les groupes d'homotopie filtrés peuvent encore être exprimés de façon "élémentaire" (voir la proposition \ref{PropositionGroupesHomotopiesFiltreFibreFiltre}).

On commence ce chapitre par un exemple élémentaire permettant d'illustrer les constructions à venir.

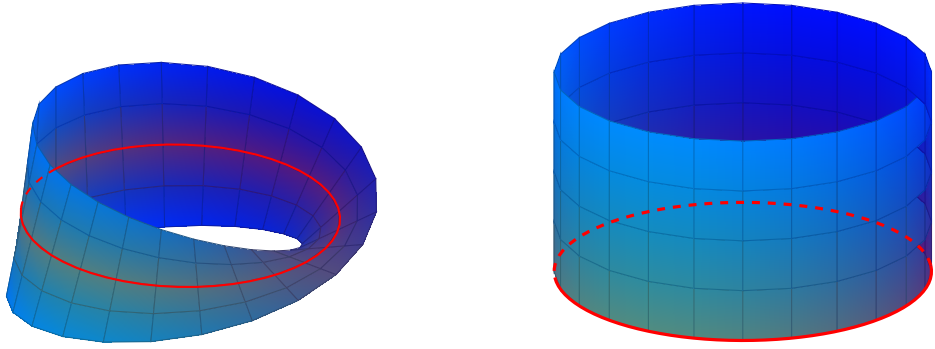
\begin{figure}
\centering
\begin{tikzpicture}
\begin{axis}[
    hide axis,
    view={40}{40}
]
\addplot3 [
    surf, shader=faceted interp,
    mesh/color input=explicit mathparse,
    samples y=5,
    domain=0:360,
    y domain=-0.5:0.5,
    point meta={symbolic={
    		1.3*(0.5-2*abs(z))*(0.5-2*abs(z)), 
	   		0.3-0.2*y, 
    		0.5+2*abs(z) 
    	}
    }
] (
    {(1+0.5*y*cos(x/2)))*cos(x)},
    {(1+0.5*y*cos(x/2)))*sin(x)},
    {0.5*y*sin(x/2)});
\addplot3 [
    samples=50,
    domain=-145:180, 
    samples y=0,
    thick, red
] (
    {cos(x)},
    {sin(x)},
    {0});
\addplot3 [
	dashed,
    samples=20,
    domain=-180:-145, 
    samples y=0,
    thick, red
] (
    {cos(x)},
    {sin(x)},
    {0});
\end{axis}

\end{tikzpicture}
\hspace{60 pt}
\begin{tikzpicture}
\begin{axis}[
    hide axis,
    view={30}{30}
]
\addplot3 [
    surf, shader=faceted interp,
    mesh/color input=explicit mathparse,
    samples y=5,
    domain=0:360,
    y domain=0:0.5,
    point meta={symbolic={
    		1.3*(0.5-2*abs(z))*(0.5-2*abs(z)), 
	   		0.3-0.2*y, 
    		0.5+2*abs(z) 
    	}
    }
] (
    {1.25*cos(x)},
    {1.25*sin(x)},
    {0.5*y});
\addplot3 [
    samples=50,
    domain=-145:35, 
    samples y=0,
    very thick, red
] (
    {1.25*cos(x)},
    {1.25*sin(x)},
    {0});
\addplot3 [
	dashed,
    samples=50,
    domain=35:215, 
    samples y=0,
    very thick, red
] (
    {1.25*cos(x)},
    {1.25*sin(x)},
    {0});
\end{axis}
\end{tikzpicture}
\caption{Les espaces filtrés $M$ et $C$ au dessus de l'ensemble ordonné 
$P=\{\textcolor{red}{p_0}<\textcolor{blue}{p_1}\}$
}\label{FigureMobiusCylindre}
\end{figure}

\begin{exemple}\label{ExempleMobiusCylindre}
Soit $S^1\subset \mathbb{C}$ le cercle unité. On considère le cylindre $C=S^1\times [0,1]$ et le ruban de Möbius $M=S^1\times [-1,1]/\sim$, avec $(x,t)\sim (-x,-t)$. On munit $C$ et $M$ des filtrations données par
\begin{align*}
\varphi_C\colon S^1\times [0,1]&\to \{p_0<p_1\}\\
(x,t)&\mapsto \left\{\begin{array}{cl}
p_0& \text{ si $t=0$}\\
p_1& \text{ si $t>0$}
\end{array}\right.
\end{align*}
et
\begin{align*}
\varphi_M\colon S^1\times [-1,1]&\to \{p_0<p_1\}\\
(x,t)&\mapsto \left\{\begin{array}{cl}
p_0& \text{ si $t=0$}\\
p_1& \text{ si $t\not=0$}
\end{array}\right.
\end{align*}
On vérifie facilement que $\varphi_M$ passe au quotient pour induire une filtration $\varphi_M\colon M\to P$. Les espaces filtrés $M$ et $C$ sont représentés Figure \ref{FigureMobiusCylindre}. 
On calcule leurs groupes d'homotopie filtrés. Par la remarque \ref{GroupesHomotopiesFiltresReduits}, il suffit de calculer $s\pi_n(-)(\Delta^{\varphi})$, pour $\Delta^{\varphi}=[p_0],[p_1]$ et $[p_0,p_1]$.
On constate d'abord qu'on a les homéomorphismes suivants
\begin{eqnarray*}
C^0_P(\RealP{[p_0]},M)  \simeq & S^1& \simeq  C^0_P(\RealP{[p_0]},C)\\
C^0_P(\RealP{[p_1]},M)\simeq & S^1\times ]0,1]&\simeq  C^0_P(\RealP{[p_1]},C)\\
\end{eqnarray*}
Puis, en utilisant la définition de $M$ comme un quotient de $S^1\times[-1,1]$, on calcule $C^0_P(\RealP{[p_0,p_1]},M)$.
\begin{equation*}
\{f\ |\ f\in C^0([0,1], S^1\times [-1,1]),\  f(0)\in S^1\times\{0\}, \text{ et } f(t)\in S^1\times([-1,0[\cup]0,1]), \forall t >0\}/\sim
\end{equation*}
Par symétrie, cet espace de fonction est homéomorphe à 
\begin{align*}
\{f\ |\ f\in C^0([0,1],S^1\times[0,1]),\ f(0)\in S^1\times\{0\}, \text{ et } f(t)\in S^1\times]0,1], \forall t >0\}\\
\simeq C^0_P(\RealP{[p_0,p_1]},C)
\end{align*}
On remarque de plus que cet espace est homéomorphe au produit
\begin{equation}\label{HolinkProduit}
C^0([0,1],S^1)\times \{g\ |\ g\in C^0([0,1],[0,1]), \ g(0)=0\text{ et } g(t)>0,\ \forall t>0\}
\end{equation}
On en déduit que $s\pi_0(M)\simeq s\pi_0(C)\simeq N(P)$. En particulier, pour chacun de ces espaces filtrés, il existe un unique pointage maximal à homotopie filtrée près. Ils sont connexes au sens filtré. (Pour un exemple d'espaces non-connexes au sens filtré, voir figure \ref{FigureExempleSPi0}). En utilisant l'identification $\RealP{N(P)}\simeq [0,1]$, on fixe les pointages suivants.
L'application
\begin{align*}
\widetilde{\phi_M}\colon \RealP{N(P)}&\to S^1\times[-1,1]\\
t&\mapsto (0,t)
\end{align*}
passe au quotient pour donner un pointage de $M$
\begin{equation*}
\phi_M\colon \RealP{N(P)}\to M.
\end{equation*}
Et on définit 
\begin{align*}
\phi_C\colon \RealP{N(P)}&\to C\\
t&\mapsto(0,t).
\end{align*} D'après les calculs précédents, et en utilisant les conventions de la remarque \ref{ConventionsExemples}, on a les groupes d'homotopies filtrés de $M$ et $C$ :
\begin{equation*}
s\pi_1(M,\phi_M)= 
\begin{tikzcd}
&\mathbb{Z}
\arrow[swap]{dl}{f_M}
\arrow{dr}{g_M}
\\
\mathbb{Z}
&&
\mathbb{Z}
\end{tikzcd}
\end{equation*}
et 
\begin{equation*}
s\pi_1(C,\phi_C)= 
\begin{tikzcd}
&\mathbb{Z}
\arrow[swap]{dl}{f_C}
\arrow{dr}{g_C}
\\
\mathbb{Z}
&&
\mathbb{Z}
\end{tikzcd}
\end{equation*}
De plus, pour $n\geq 2$, on a
\begin{equation*}
s\pi_n(M,\phi_M)= 
\begin{tikzcd}
&0
\arrow{dl}
\arrow{dr}
\\
0
&&
0
\end{tikzcd}
=s\pi_n(C,\phi_C)
\end{equation*}
Il nous reste donc à calculer les applications $f_M$ et $f_C$ induites par l'inclusion $[p_0]\to [p_0,p_1]$ ainsi que les applications $g_M$ et $g_C$ induites par l'inclusion $[p_1]\to [p_0,p_1]$.
On commence par $C$. On observe que comme $C$ est le produit $S^1\otimes\RealP{N(P)}$, les restrictions
\begin{equation*}
C^0_P(\RealP{[p_0,p_1]},C)\to C^0_P(\RealP{[p_i]},C)
\end{equation*}
où $i=0,1$, se factorisent comme
\begin{equation*}
C^0_P(\RealP{[p_0,p_1]},C)\to C^0([0,1],S^1)\xrightarrow{\ev_i} S^1
\end{equation*}
où la première application est la projection sur le premier terme du produit (\ref{HolinkProduit}), et la seconde est l'évaluation en $i$. En particulier, on obtient que $f_C=g_C=\Id_{\mathbb{Z}}$. Dans le cas de $M$, on considère l'application suivante 
\begin{align*}
\widetilde{\gamma}\colon S^1\times [0,1]&\to S^1\times [-1,1]\\
(s,t)&\mapsto (s,t).
\end{align*}
Cette application passe au quotient et induit une application
\begin{equation*}
\gamma\colon S^1\to C^0_P(\RealP{[p_0,p_1]},M).
\end{equation*}
De plus, en reprenant la série d'homéomorphismes aboutissant à (\ref{HolinkProduit}), on constate que (la classe de) $\gamma$ génère $\pi_1(C^0_P(\RealP{[p_0,p_1]},M),\phi_M)$. Il suffit donc de calculer $f_M(\gamma)$ et $g_M(\gamma)$. On obtient 
\begin{align*}
f_M(\gamma)\colon S^1&\to S^1\times\{0\}/\sim\\
s&\mapsto [(s,0)].
\end{align*}
En particulier, $f_M(\gamma)$ est le double d'un générateur de $\pi_1(S^1\times\{0\}/\sim,[(0,0)])$. 
D'autre part en utilisant l'identification $S^1\times([-1,0[\cup]0,1])/{\sim} \simeq S^1\times ]0,1]$, on a
\begin{align*}
g_M(\gamma)\colon S^1&\to S^1\times]0,1]\\
s&\mapsto (s,1)
\end{align*}
et on en déduit que $g_M(\gamma)$ est un générateur de $\pi_1(C^0_P(\RealP{[p_1]},M),[(0,1)])$. Finalement, on a

\begin{equation*}
s\pi_1(M,\phi_M)= 
\begin{tikzcd}
&\mathbb{Z}
\arrow[swap]{dl}{\times 2}
\arrow{dr}{\Id_{\Z}}
\\
\mathbb{Z}
&&
\mathbb{Z}
\end{tikzcd}
\hspace{30pt}
s\pi_1(C,\phi_C)= 
\begin{tikzcd}
&\mathbb{Z}
\arrow[swap]{dl}{\Id_{\Z}}
\arrow{dr}{\Id_{\Z}}
\\
\mathbb{Z}
&&
\mathbb{Z}
\end{tikzcd}
\end{equation*}
En particulier, il n'existe pas d'équivalence d'homotopie filtrée entre $M$ et $C$. Cependant, on vérifie aisément que en tant qu'espace topologique, $M$ et $C$ sont homotopes. Ainsi, le type d'homotopie filtré contient plus d'information que le type d'homotopie non filtré, et cette information est détectée par les groupes d'homotopies filtrés.
\end{exemple}

\begin{figure}[h]
\centering
\includegraphics[width= 200pt]{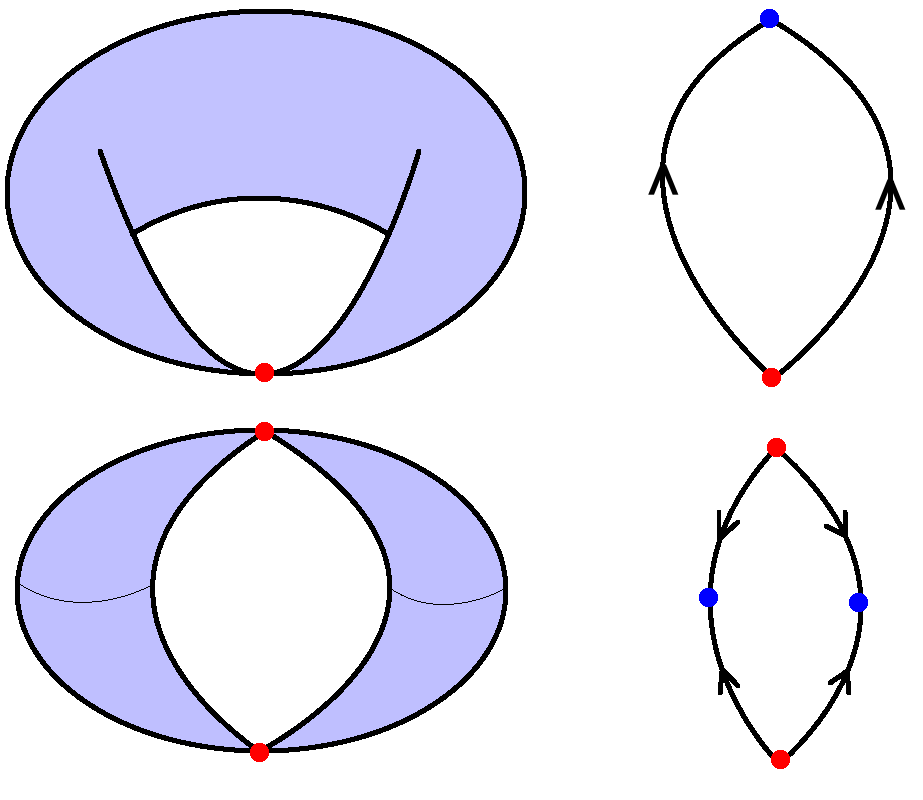}
\caption{Deux espaces filtrés non connexes au sens filtré, et leur ensemble de composantes connexes filtrés}
\label{FigureExempleSPi0}
\end{figure}

\begin{remarque}\label{ConventionsExemples}
Dans l'exemple précédent, et dans la première section de ce chapitre, on s'intéressera particulièrement au cas où $P=\{p_0,p_1\}$. En vertu de la remarque \ref{GroupesHomotopiesFiltresReduits}, pour tout espace filtré $\fil{X}$ et tout pointage $\phi\colon \RealP{N(P)}\to \fil{X}$, il suffit de calculer la restriction à $R(P)$ $j^*s\pi_n(\fil{X},\phi)$ pour déterminer entièrement $s\pi_n(\fil{X},\phi)$. Dans le cas où $P=\{p_0<p_1\}$, la catégorie $R(P)$ (voir la définition \ref{DefinitionRP}) a pour objets $[p_0],[p_1]$ et $[p_0,p_1]$ et pour morphismes non identiques :
\begin{equation*}
\begin{tikzcd}
& {[}p_0,p_1]
\\
{[}p_0]
\arrow{ur}
&&
{[}p_1]
\arrow{ul}
\end{tikzcd}
\end{equation*}
Pour cette raison, on identifiera $s\pi_n(\fil{X},\phi)$ et $j^*s\pi_n(\fil{X},\phi)$ et on écrira
\begin{equation*}
s\pi_n(\fil{X},\phi)= 
\begin{tikzcd}
&\pi_n(\C^0_P(\RealP{[p_0,p_1]},\fil{X}),\phi)
\arrow{dl}
\arrow{dr}
\\
\pi_n(X_{p_0},\phi(p_0))
&&
\pi_n(X_{p_1},\phi(p_1))
\end{tikzcd}
\end{equation*}
avec $X_{p_i}=C^0_P(\RealP{[p_i]},\fil{X})$, $i=0,1$. De plus, on laissera parfois implicite le calcul des morphismes apparaissant dans le diagramme bien que ceux ci soient essentiels à la distinction des groupes d'homotopie filtrés.
Finalement, lorsque ce sera pertinent, on fera les identifications 
\begin{equation*}
\RealP{N(P)}\simeq \RealP{[p_0,p_1]}\simeq \textcolor{red}{\{0\}}\cup\textcolor{blue}{]0,1]}\simeq [0,1].
\end{equation*}
\end{remarque}

\section{D'un fibré localement trivial à un espace filtré}

\subsection{Une construction générale}\label{ConstructionFibre}
\label{SectionConstructionFibreNonFiltre}
Une façon élémentaire de produire un espace filtré à partir d'un espace topologique non filtré $X$ est de prendre son cône $c(X)$. (voir Définition \ref{DefinitionCone}). Si $X$ est une variété topologique de dimension $n$, on obtient ainsi une pseudo variété de dimension $n+1$ à deux strates : la strate singulière est le sommet du cône, et la strate régulière est homéomorphe à $X\times \mathbb{R}$. Si $X$ est une pseudo variété, on obtiendra toujours une pseudo variété $c(X)$, ayant pour strate la plus singulière le sommet du cône et ayant une strate de la forme $S\times \mathbb{R}$ pour chaque strate $S\subseteq X$. Ces constructions donnent des espaces filtrés dont on sait calculer les groupes d'homotopies filtrés. Ou plus exactement, on sait exprimer les groupes d'homotopies filtrés de $c(X)$ en fonction des groupes d'homotopies (filtrés) de $X$ (voir les propositions \ref{GroupesHomotopiesConeOuvertFibre} et \ref{PropositionGroupesHomotopiesFiltreFibreFiltre}, le cas des cônes correspond au cas particulier où $B=\{*\}$). Cependant, toutes les pseudo variétés que l'on peut obtenir ainsi contiennent une strate de dimension $0$. Pour étendre cette classe d'exemple, on considère la construction suivante.

Fixons un fibré localement trivial
\begin{equation}\label{DefinitionFibre}
F\hookrightarrow E\xrightarrow{\pi}B
\end{equation}
Nous allons construire un fibré localement trivial 
\begin{equation}\label{DefinitionConeFibre}
c(F)\hookrightarrow c_{\pi}(E)\xrightarrow{c_{\pi}(\pi)}B
\end{equation}
tel que l'inclusion du sommet dans $c(F)$ induise une section $B\hookrightarrow c_{\pi}(E)$. En particulier, on obtiendra une filtration $\varphi_{\pi}\colon c_{\pi}(E)\to P$ en posant $\varphi_{\pi}^{-1}(p_0)=i_{B}(B)$. Commençons par considérer la somme amalgamée suivante
\begin{equation*}
\begin{tikzcd}
E
\arrow{r}{\pi}
\arrow{d}{i_0}
&B
\arrow[swap]{d}{i_B}
\arrow[bend left=18]{ddr}{\Id_B}
\\
E\times [0,1[
\arrow{r}
\arrow[swap, bend right=18]{drr}{\pi\circ\pr_E}
&c_{\pi}(E)
\arrow[swap, near start]{dr}{c_{\pi}(\pi)}
\\
&&B
\end{tikzcd}
\end{equation*}
Ceci définit le fibré (\ref{DefinitionConeFibre}). Montrons qu'il s'agit bien d'un fibré localement trivial de fibre $c(F)$. Soit $U\subset B$ un ouvert trivialisant pour $\pi$.  Alors, $c_{\pi}(\pi)^{-1}(U)$ est donné comme la somme amalgamée
\begin{equation*}
\begin{tikzcd}
\pi^{-1}(U)
\arrow{r}{\pi}
\arrow{d}{i_0}
&U
\arrow[swap]{d}{i_B}
\\
\pi^{-1}(U)\times [0,1[
\arrow{r}
&c_{\pi}(\pi)^{-1}(U)
\end{tikzcd}
\end{equation*}
En utilisant la trivialisation $\pi^{-1}(U)\simeq F\times U$, il vient :
\begin{equation*}
\begin{tikzcd}
U\times F
\arrow{r}{\pr_U}
\arrow{d}{i_0}
&U
\arrow[swap]{d}{i_B}
\arrow[bend left=18]{ddr}{i_0}
\\
U\times F\times [0,1[
\arrow{r}
\arrow[bend right=18]{drr}
&c_{\pi}(\pi)^{-1}(U)
\arrow{dr}{{\simeq}}
\\
&&U\times c(F)
\end{tikzcd}
\end{equation*}
On en déduit que le fibré (\ref{DefinitionConeFibre}) est un fibré localement trivial, et que $i_B$ induit une section $i_B\colon B\to c_{\pi}(E)$. On remarque aussi que par construction, on a un morphisme $h\colon c_{\pi}(E)\to [0,1[$, induit par la projection $F\times [0,1[\to [0,1[$. Ce morphisme vérifie $h^{-1}(0)=i_B(B)$ et permet de définir une filtration $\varphi_{\pi}\colon c_{\pi}(E)\to P$ comme la composition
\begin{equation*}
c_{\pi}(E)\xrightarrow{h} [0,1[\xrightarrow{\varphi_P}P
\end{equation*}
Avant de procéder au calcul des groupes d'homotopies filtrés de $c_{\pi}(\pi)$, on fait les observations suivantes.

\begin{prop}\label{FibrePseudoVariete}
L'espace filtré $(c_{\pi}(E),\varphi_{\pi})$ est coniquement stratifié. De plus, si $B$ et $F$ sont des variétés topologiques de dimensions $m$ et $n$ respectivement, alors $c_{\pi}(E)$ est une pseudo variété de dimension $m+n+1$.
\end{prop}

\begin{proof}
Montrons d'abord que $c_{\pi}(E)$ est coniquement stratifié. Soit $x\in c_{\pi}(E)$. Si $\varphi_{\pi}(x)=p_1$, il n'y a rien à montrer. Si $\varphi_{\pi}(x)=p_0$, soit $U\subset B$ un ouvert trivialisant de $\pi$ contenant $c_{\pi}(\pi)(x)$. Alors, $c_{\pi}(\pi)^{-1}(U)$ contient $x$ et est filtré homéomorphe à $U\otimes c(F)$.

Supposons maintenant que $B$ et $F$ sont des variétés de dimension $m$ et $n$ respectivement. Ceci implique que $E$ est une variété de dimension $m$+$n$. Soit $x\in c_{\pi}(E)$. Si $\varphi_{\pi}(x)=p_1$, alors $x\in c_{\pi}(\pi)^{-1}(p_1)\simeq E\times]0,1[$. En particulier $x$ possède un voisinage homéomorphe à $\R^{m+n+1}$. Si $\varphi_{\pi}(x)=p_0$, soit $U\subset B$ un ouvert trivialisant de $\pi$ contenant $c_{\pi}(\pi)(x)$. Quitte à remplacer $U$ par un voisinage de $c_{\pi}(\pi)(x)$ plus petit, on peut supposer que $U\simeq \R^{m}$. Finalement, on a un homéomorphisme filtré $c_{\pi}(\pi)^{-1}(U)\simeq \R^m\otimes c(L)$, avec $L$ une variété topologique. On en déduit que $c_{\pi}(E)$ est une pseudo-variété. 
\end{proof}

\begin{prop}\label{GroupesHomotopiesConeOuvertFibre}
On a 
\begin{equation*}
s\pi_0(c_{\pi}(E))\simeq 
\begin{tikzcd}
&\pi_0(E)
\arrow[swap]{dl}{\pi_*}
\arrow{dr}{\Id}
\\
\pi_0(B)
&&
\pi_0(E)
\end{tikzcd}
\end{equation*}
Et, pour tout pointage $\phi\colon \RealP{N(P)}\to c_{\pi}(E)$, on a
\begin{equation*}
s\pi_n(c_{\pi}(E),\phi)\simeq 
\begin{tikzcd}
&\pi_n(E,\phi(1))
\arrow[swap]{dl}{\pi_*}
\arrow{dr}{\Id}
\\
\pi_n(B,\phi(0))
&&
\pi_n(E,\phi(1))
\end{tikzcd}
\end{equation*}
\end{prop}

\begin{proof}
Par construction, $\varphi_{\pi}^{-1}(p_0)\simeq B$ et $\varphi_{\pi}^{-1}(p_1)\simeq E\times ]0,1[$. On en déduit les termes de gauche et de droite de $s\pi_n$. Montrons que la restriction :
\begin{equation*}
C^0_P(\RealP{\ [p_0,p_1]\ },c_{\pi}(E))\to C^0_P(\RealP{\ [p_1]\ },c_{\pi}(E))
\end{equation*}
induit une équivalence d'homotopie.
On commence par constater qu'on a des homéomorphismes
\begin{equation*}
C^0_P(\RealP{\ [p_1]\ },c_{\pi}(E))\simeq \varphi_{\pi}^{-1}(p_1)\simeq E\times ]0,1[
\end{equation*}
et la projection $\pr_E\colon E\times ]0,1[\to E$ est une équivalence d'homotopie.
On exhibe une section pour la composition
\begin{equation*}
\pr\circ\ev_1\colon C^0_P(\RealP{\ [p_0,p_1]\ },c_{\pi}(E))\to E
\end{equation*}
donnée comme suit :
\begin{align*}
i_E\colon E&\to C^0_P(\RealP{\ [p_0,p_1]\ },c_{\pi}(E))\\
x&\mapsto \left\{\begin{array}{ccc}
[0,1]&\to &c_{\pi}(E)\\
t&\mapsto &\left\{ \begin{array}{cl}
(x,t/2) & \text{si $t>0$}\\
\pi(x) &\text{si $t=0$}
\end{array}\right.
\end{array}\right.
\end{align*}
où on utilise l'identification $c_{\pi}(E)=B\cup E\times]0,1[$. Par construction, on a $\pr\circ\ev_1\circ i_E=\Id_E$. On définit une homotopie entre $i_E\circ \pr\circ\ev_1$ et $\Id$ en deux étapes. Considérons d'abord l'homotopie suivante
\begin{align*}C^0_P(\RealP{\ [p_0,p_1]\ },c_{\pi}(E))\times [0,1] &\to C^0_P(\RealP{\ [p_0,p_1]\ },c_{\pi}(E))\\
(f,s)&\mapsto \left\{ \begin{array}{ccc}
[0,1]&\to & c_{\pi}(E)\\
t&\mapsto & \left\{
\begin{array}{cl}
(f_E(t),(1-s) f_{]0,1[}(t)+s t/2)& \text{ si $t>0$}\\
f(0) &\text{ si $t=0$}
\end{array}\right.
\end{array}\right.
\end{align*}
où $f(t)=(f_{E}(t),f_{]0,1[}(t)\in E\times ]0,1[$ pour $t>0$. C'est une homotopie entre $\Id$ et une application $g$. On a maintenant l'homotopie suivante entre $i_E\circ\pr\circ\ev_1$ et $g$ donnée par la formule suivante :
\begin{align*}
C^0_P(\RealP{\ [p_0,p_1]\ },c_{\pi}(E))\times [0,1] &\to C^0_P(\RealP{\ [p_0,p_1]\ },c_{\pi}(E))\\
(f,s)&\mapsto \left\{ \begin{array}{ccc}
[0,1]&\to & c_{\pi}(E)\\
t&\mapsto & \left\{
\begin{array}{cl}
(f_{E}(t),t/2)& \text{ si $t\geq s, t>0$}\\
(f_{E}(s),t/2)& \text{ si $0<t\leq s$}\\
\pi(f_{E}(s))& \text{ si $0=t$}
\end{array}\right.
\end{array}\right.
\end{align*}
On en déduit que $\pr\circ\ev_1$ est une équivalence d'homotopie. Par deux sur trois, $\ev_1$ est une équivalence d'homotopie, et on obtient le terme du haut de $s\pi_n$ ainsi que le morphisme de droite ($\Id$). En composant $c_{\pi}(\pi)$ avec les homotopies précédentes, on obtient que le morphisme de gauche est $\pi_*$.
\end{proof}
Les pseudo variétés obtenues par application de la proposition \ref{FibrePseudoVariete} ne peuvent pas être compacts. En particulier, ce sont des cônes ouverts dans le cas où $B$ est un point. De façon à obtenir des pseudo variétés compactes, on remplace les cônes ouverts de la construction précédente par des cônes fermés. On rappelle que si $F$ est un espace topologique, le cône fermé de $F$ est le quotient
\begin{equation*}
\bar{c}(F)=F\times [0,1]/F\times\{0\}.
\end{equation*}
En particulier, on a $\partial(\bar{c}(F))\simeq F\times\{1\}$. Ainsi, en effectuant ce remplacement, on disposera d'un bord le long duquel recoller un autre espace topologique de façon à obtenir une pseudo variété compacte. Cette considération donne lieu à la définition suivante.

\begin{defin}
On définit l'espace filtré $\bar{c}_{\pi}(E)$ comme la somme amalgamée
\begin{equation*}
\begin{tikzcd}
E
\arrow{r}{\pi}
\arrow{d}{i_0}
&B
\arrow[swap]{d}{i_B}
\arrow[bend left=18]{ddr}{\Id_B}
\\
E\times [0,1]
\arrow{r}
\arrow[swap, bend right=18]{drr}{\pi\circ\pr_E}
&\bar{c}_{\pi}(E)
\arrow[swap, near start]{dr}{\bar{c}_{\pi}(\pi)}
\\
&&B
\end{tikzcd}
\end{equation*}
Où la filtration $\varphi_{\pi}\colon \bar{c}_{\pi}(E)$ est construite de la même façon que pour $c_{\pi}(E)$.
Soit $f\colon E\to E'$ une application continue. On définit l'espace filtré $X(\pi,f)$ comme la somme amalgamée
\begin{equation*}
\begin{tikzcd}
E
\arrow{d}{i_1}
\arrow{r}{f}
&E'
\arrow{d}
\\
\bar{c}_{\pi}(E)
\arrow{r}
&X(\pi,f)
\end{tikzcd}
\end{equation*}
où $i_1$ est induit par l'inclusion en $1$, $E\hookrightarrow E\times [0,1]$.
\end{defin}

\begin{prop}\label{GroupesHomotopiesFibre}
L'espace filtré $X(\pi,f)$ est coniquement stratifié. De plus, si $B,F, E$ sont des variétés, $E'$ est une variété à bord et $f\colon E\to E'$ est l'inclusion du bord de $E$, alors $X(\pi,f)$ est une pseudo variété. 
Dans tous les cas, on a
\begin{equation*}
s\pi_0(X(\pi,f))\simeq 
\begin{tikzcd}
&\pi_0(E)
\arrow[swap]{dl}{\pi_*}
\arrow{dr}{f}
\\
\pi_0(B)
&&
\pi_0(E')
\end{tikzcd}
\end{equation*}
Et, pour tout pointage $\phi\colon \RealP{N(P)}\to X(\pi,f)$, on a
\begin{equation*}
s\pi_n(X(\pi,f),\phi)\simeq 
\begin{tikzcd}
&\pi_n(E,\phi(1))
\arrow[swap]{dl}{\pi_*}
\arrow{dr}{f_*}
\\
\pi_n(B,\phi(0))
&&
\pi_n(E',f\circ\phi(1))
\end{tikzcd}
\end{equation*}
\end{prop}

\begin{proof}
La preuve de la proposition \ref{FibrePseudoVariete} s'adapte directement pour montrer que $X(\pi,f)$ est coniquement stratifié, et une pseudo variété sous les hypothèses de la proposition. Pour calculer les groupes d'homotopies, on constate d'abord qu'on a
\begin{equation*}
\varphi_{\pi}^{-1}(p_1)\simeq E\times]0,1]\cup_{E\times\{1\}} E'\sim E'
\end{equation*}
Ce qui donne le terme de droite. Ensuite, en appliquant \cite[Proposition A.7.9]{HigherAlgebra} à l'ouvert $U=c_{\pi}(E)\subset X(\pi,f)$, on obtient que le morphisme induit par l'inclusion
\begin{equation*}
C^0_P(\RealP{\ [p_0,p_1]\ },c_{\pi}(E))\to C^0_P(\RealP{\ [p_0,p_1]\ },X(\pi,f))
\end{equation*}
est une équivalence faible. On conclue de la même façon que dans la preuve de la proposition \ref{GroupesHomotopiesConeOuvertFibre}.
\end{proof}

\subsection{Quelques pseudo variétés}
\label{SectionPseudoVarietesFibres}
\begin{exemple}
Considérons les deux fibrés localement triviaux
\begin{align*}
\pi\colon S^1&\to S^1 \\
e^{i\theta}&\mapsto e^{2i\theta},
\end{align*}
et 
\begin{equation*}
\Id\colon S^1\to S^1.
\end{equation*}
En appliquant la construction de la section précédente, on a $X(\pi,\Id)=\fil{M}$ et $X(\Id,\Id)=\fil{C}$, les espaces filtrés de l'exemple \ref{ExempleMobiusCylindre}. En particulier, on retrouve les résultats des calculs fait dans l'exemple \ref{ExempleMobiusCylindre} en appliquant la proposition \ref{GroupesHomotopiesFibre}.
\end{exemple}

\begin{exemple}
Considérons maintenant les deux fibrés suivants
\begin{align*}
\pi\colon S^1\times S^1&\to S^1 \\
(e^{i\theta},e^{i\mu})&\mapsto e^{2i\theta},
\end{align*}
et
\begin{align*}
\pr_1\colon S^1\times S^1&\to S^1 \\
(e^{i\theta},e^{i\mu})&\mapsto e^{i\theta},
\end{align*}
et notons $f\colon S^1\times S^1\to S^1\times D^2$ l'inclusion du tore comme le bord du tore plein. Alors, par la proposition \ref{GroupesHomotopiesFibre}, les espaces filtrés $X(\pi,f)$ et $X(\pr_1,f)$ sont des pseudo-variétés dont on sait calculer les groupes d'homotopies filtrés. Elles sont connexes au sens filtré. fixons $\phi_1\colon \RealP{\ N(P)\ }\to X(\pr_1,f)$ et $\phi_2\colon \RealP{\ N(P)\ }\to X(\pi,f)$ deux pointages, on a 
\begin{equation*}
s\pi_1(X(\pr_1,f),\phi_1)= 
\begin{tikzcd}
&\mathbb{Z}\oplus\Z
\arrow[swap]{dl}{\pr_1}
\arrow{dr}{\pr_1}
\\
\mathbb{Z}
&&
\mathbb{Z}
\end{tikzcd}
\hspace{30pt}
s\pi_1(X(\pi,f),\phi_2)= 
\begin{tikzcd}
&\mathbb{Z}\oplus\Z
\arrow[swap]{dl}{(\times 2,0)}
\arrow{dr}{\pr_1}
\\
\mathbb{Z}
&&
\mathbb{Z}
\end{tikzcd}
\end{equation*}
et
\begin{equation*}
s\pi_n(X(\pr_1,f),\phi_1)= 
\begin{tikzcd}
&0
\arrow[swap]{dl}
\arrow{dr}
\\
0
&&
0
\end{tikzcd}
=
s\pi_n(X(\pi,f),\phi_2) 
\end{equation*}
pour tout $n\geq 2$. En particulier, on en déduit que $X(\pi,f)$ et $X(\pr_1,f)$ ne sont pas homotopes au sens filtré. Cependant, $X(\pi,f)$ et $X(\pr_1,f)$ ne sont pas distingué par l'homologie d'intersection. Calculons leur homologie d'intersection respectives. On commence par constater que $X(\pr_1,f)$ est homéomorphe au produit $S^1\times S^2$. En effet, comme le fibré $\pr_1\colon S^1\times S^1\to S^1$ est trivial, on a:
\begin{equation*}
c_{\pr_1}(S^1\times S^1)\simeq S^1\times c(S^1)\simeq S^1\times D^2
\end{equation*}
Puis, 
\begin{equation*}
X(\pr_1,f)\simeq S^1\times D^2\cup_{S^1\times S^1}S^1\times D^2\simeq S^1\times S^2
\end{equation*}
En particulier, $X(\pr_1,f)$ est une variété, et la stratification sur $X(\pr_1,f)$ est celle d'une pseudo variété. On en déduit que pour toute perversité (de Goresky et MacPherson), l'homologie d'intersection de $X(\pr_1,f)$ est isomorphe à son homologie singulière. Dans notre contexte, il y a une unique strate singulière qui est de codimension $2$. En particulier la seule perversité de Goresky et MacPherson est celle qui associe à la strate singulière la valeur $0$. On a donc
\begin{equation*}
I^0H_k(X(\pr_1,f))\simeq H_k(S^1\times S^2)\simeq \left\{\begin{array}{cl}
\Z &\text{ si $0\leq k\leq 3$}\\
0 &\text{ si $k<0$ ou $k>3$}
\end{array}\right.
\end{equation*}
D'autre part, pour calculer l'homologie d'intersection de $X(\pi,f)$, on considère l'union 
\begin{equation*}
X(\pi,f)\simeq \bar{c}_{\pi}(S^1\times S^1)\cup_{S^1\times S^1}S^1\times D^2
\end{equation*}
On pose $U=c_{\pi}(S^1\times S^1)$ et $V=S^1\times S^1\times]1/2,1]\cup_{S^1\times S^1}S^1\times D^2$. Alors, on a
\begin{equation*}
X(\pi,f)=U\cup V
\end{equation*}
et
\begin{equation*}
U\cap V\simeq S^1\times S^1\times ]1/2,1[.
\end{equation*}
On peut donc appliquer le théorème de Mayer-Vietoris en homologie d'intersection pour calculer l'homologie d'intersection de $X(\pi,f)$. On a la suite exacte longue
\begin{equation*}
I^0H_n(U\cap V)\to I^0H_n(U)\oplus I^0H_n(V)\to I^0H_n(X(\pi,f)) \to I^0H_{n-1}(U\cap V)\to I^0H_{n-1}(U)\oplus I^0H_{n-1}(V)
\end{equation*}
On calcule
\begin{equation*}
I^0H_n(U\cap V)\simeq H_n(S^1\times S^1\times ]1/2,1[)\simeq \left\{\begin{array}{cl}
\Z &\text{si $n=0,2$}\\
\Z\oplus\Z &\text{si $n=1$}\\
0 &\text{si $n<0$ ou $n>2$}
\end{array}
\right.
\end{equation*}
car $U\cap V$ est inclus dans la strate régulière. De même 
\begin{equation*}
I^0H_n(V)\simeq H_n(V)\simeq H_n(S^1\times D^2)\simeq \left\{\begin{array}{cl}
\Z &\text{si $n=0,1$}\\
0 &\text{si $n<0$ ou $n>1$}
\end{array}
\right.
\end{equation*}
D'autre part, le morphisme $I^0H_0(U\cap V)\to I^0H_0(V)$ est l'identité, et le morphisme $I^0H_1(U\cap V)\to I^0H_1(V)$ correspond à la projection $\Z\oplus \Z\xrightarrow{\pr_1} \Z$ induite par l'inclusion $S^1\times S^1\to S^1\times D^2$.
Il reste à calculer $I^0H_n(U)$ et l'application $I^0H_n(U\cap V)\to I^0H_n(U)$. On identifie $S^1$ à $\R\mod 2\pi$, et on constate que $S^1\simeq U_1\cup U_2$ avec $U_1=]-\pi,\pi[$ et $U_2=]0,2\pi[$. De plus, $U_1$ et $U_2$ sont des ouverts trivialisants pour $c_{\pi}(\pi)$. En particulier, on a
\begin{equation*}
c_{\pi}(S^1\times S^1)\simeq ]-\pi,\pi[\times c(S^1\coprod S^1)\cup ]0,2\pi[\times c(S^1\coprod S^1)
\end{equation*}
Où les deux composantes sont recollées par l'identité le long de $]0,\pi[$ et par l'homéomorphisme échangeant les composantes de $c(S^1\coprod S^1)$ le long de $]\pi,2\pi[$. 
On calcule 
\begin{equation*}
I^0H_k(U_1)\simeq I^0H_k(U_2)\simeq I^0H_k(c(S^1\coprod S^1))\simeq \left\{\begin{array}{cl}
\Z\oplus \Z &\text{ si $k=0$}\\
0 &\text{ si $k\not= 0$}
\end{array}\right.
\end{equation*}
De même, 
\begin{equation*}
I^0H_k(U_1\cap U_2)\simeq I^0H_k(c(S^1\coprod S^1)\coprod c(S^1\coprod S^1))\simeq \left\{\begin{array}{cl}
(\Z\oplus \Z)\oplus (\Z\oplus \Z) &\text{ si $k=0$}\\
0 &\text{ si $k\not= 0$}
\end{array}\right.
\end{equation*}
En appliquant Mayer Vietoris à $U_1,U_2$, on a la suite exacte longue
\begin{equation*}
0\to I^0H_1(U)\to I^0H_0(U_1\cap U_2)\xrightarrow{\alpha} I^0H_0(U_1)\oplus I^0H_0(U_2)\to I^0H_0(U)\to 0
\end{equation*}
En particulier, on a $I^0H_1(U)=\ker(\alpha)$ et $I^0H_0(U)=\coker(\alpha)$.
On calcule
\begin{align*}
(\Z\oplus \Z)\oplus (\Z\oplus \Z)&\xrightarrow{\alpha}(\Z\oplus \Z)\oplus (\Z\oplus \Z)\\
((a,b),(c,d))&\mapsto ((a+c,b+d),(a+d,b+c))
\end{align*}
En particulier, on a $\coker(\alpha)\simeq\ker(\alpha)\simeq \Z$.
On en déduit
\begin{equation*}
I^0H_n(U)\simeq \left\{\begin{array}{cl}
\Z &\text{si $n=0,1$}\\
0 &\text{si $n<0$ ou $n>1$}
\end{array}\right.
\end{equation*}
De plus, on voit que le $I^0H_1(U)$ provient de l'homologie de la base du fibré $c_{\pi}(\pi)$. En particulier, le morphisme
\begin{equation*}
I^0H_1(U\cap V)\to I^0H_1(U)
\end{equation*}
correspond à la projection $\Z\oplus \Z\xrightarrow{\pr_1}\Z$. On en déduit l'homologie d'intersection de $X(\pi,f)$.
On a $I^0H_n(X(\pi,f))=0$ pour $n>3$ ou $n<0$. Puis, on a 
\begin{equation*}
I^0H_3(X(\pi,f))\simeq I^0H_2(U\cap V)\simeq \Z
\end{equation*}
\begin{equation*}
I^0H_2(X(\pi,f))\simeq \ker\left(I^0H_1(U\cap V)\to I^0H_1(U)\oplus I^0H_1(V)\right)\simeq 0\oplus \Z\simeq \Z
\end{equation*}
\begin{equation*}
I^0H_1(X(\pi,f))\simeq \coker\left(I^0H_1(U\cap V)\to I^0H_1(U)\oplus I^0H_1(V)\right)\simeq \Z
\end{equation*}
et
\begin{equation*}
I^0H_1(X(\pi,f))\simeq \coker\left(I^0H_0(U\cap V)\to I^0H_0(U)\oplus I^0H_0(V)\right)\simeq \Z
\end{equation*}
En particulier, les groupes d'homotopies filtrés sont un invariants des pseudo variétés plus fin que l'homologie d'intersection. Et dans des contextes appropriés tels que celui-ci, ils sont même plus facile à calculer.
\end{exemple}

\subsection{Espaces D'Eilenberg-Mac Lane Filtrés}
\label{SectionEilenbergMaclane}
On peut faire la définition suivante d'espaces D'Eilenberg-Mac Lane dans le cas filtré.

\begin{defin}
Soient $\fil{X}$ un espace filtré sur $P=\{p_0,p_1\}$, 
\begin{equation*}
\bar{G}= 
\begin{tikzcd}
&G_{[p_0,p_1]}
\arrow[swap]{dl}
\arrow{dr}
\\
G_{p_0}
&&
G_{p_1}
\end{tikzcd}
\end{equation*}
un diagramme de groupe et $n\geq 1$ un entier. Alors, $\fil{X}$ est un $(\bar{G},n)$ espace d'Eilenberg-Mac Lane au sens filtré si $\fil{X}$ est connexe au sens filtré, et pour tout pointage $\phi\colon\RealP{N(P)}\to \fil{X}$, on a
\begin{equation*}
s\pi_k(\fil{X},\phi)= 
\begin{tikzcd}
&0
\arrow[swap]{dl}
\arrow{dr}
\\
0
&&
0
\end{tikzcd}
\end{equation*}
pour $k\not =n$ et
\begin{equation*}
s\pi_n(\fil{X},\phi)\simeq \bar{G}=
\begin{tikzcd}
&G_{[p_0,p_1]}
\arrow[swap]{dl}
\arrow{dr}
\\
G_{p_0}
&&
G_{p_1}
\end{tikzcd}
\end{equation*}
\end{defin} 
Cependant, on ne sait pas à priori si il existe un $(\bar{G},n)$ espace d'Eilenberg-Mac Lane pour tout $\bar{G}$ et $n$. On ne sait pas non plus si ils sont uniques lorsqu'ils existent.
D'autre part, dans le cas non filtré, on sait que $K(G,1)\simeq B(G)$ correspond au classifiant du groupe $G$ et permet de classifier les $G$-fibrés. Qu'en est-il dans le cas des espaces filtrés?

L'exemple suivant permet de répondre partiellement à la première question, en construisant des $(\bar{G},1)$ espaces D'Eilenberg-Mac Lane pour une large famille de diagrammes de groupes $\bar{G}$.

\begin{exemple}
Soient $G$ un groupe et $\pi\colon H\hookrightarrow G$ une inclusion de sous groupe. A partir de $\pi$, on construit un fibré localement trivial
\begin{equation*}
G/H\hookrightarrow EG\times_{G}G/H\xrightarrow{\pi} BG
\end{equation*}
où, $BG$ est un classifiant de $G$ et $EG$ est le revêtement universel de $BG$. De plus, on a
\begin{equation*}
EG\times_{G}G/H\simeq (EG\times_{G}G)/H\simeq (EG/H)\simeq BH
\end{equation*}
On obtient donc un fibré localement trivial de la forme
\begin{equation*}
G/H\hookrightarrow BH\xrightarrow{\pi} BG.
\end{equation*}
Soit $f\colon H\to H'$ un morphisme de groupe, on choisit un représentant de $f$ :
\begin{equation*}
f\colon BH\to BH'.
\end{equation*}
Alors, en appliquant la proposition \ref{GroupesHomotopiesFibre}, on obtient que l'espace $X(\pi,f)$ est coniquement stratifié, et connexe au sens filtré. De plus, pour tout pointage $\phi\colon \RealP{N(P)}\to X(\pi,f)$, on a
\begin{equation*}
s\pi_1(X(\pi,f),\phi)=
\begin{tikzcd}
&H
\arrow[swap]{dl}{\pi}
\arrow{dr}{f}
\\
G
&&
H'
\end{tikzcd}
\end{equation*}
et
\begin{equation*}
s\pi_n(X(\pi,f),\phi)=
\begin{tikzcd}
&0
\arrow[swap]{dl}
\arrow{dr}
\\
0
&&
0
\end{tikzcd}
\end{equation*}
pour tout $n\geq 2$.
En particulier, en notant
\begin{equation*}
\bar{G}=\begin{tikzcd}
&H
\arrow[swap]{dl}{\pi}
\arrow{dr}{f}
\\
G
&&
H'
\end{tikzcd}
\end{equation*}
L'espace filtré $X(\pi,f)$ est un $(\bar{G},1)$-espace d'Eilenberg-Mac Lane.
\end{exemple}

\begin{prop}\label{PropositionCaracterisationEilenbergMaclane}
Soient 
\begin{equation*}
\bar{G_1}=\begin{tikzcd}
&H_1
\arrow[swap]{dl}{\pi_1}
\arrow{dr}{f_1}
\\
G_1
&&
H'_1
\end{tikzcd}
\hfill
\bar{G_2}=\begin{tikzcd}
&H_2
\arrow[swap]{dl}{\pi_2}
\arrow{dr}{f_2}
\\
G_2
&&
H'_2
\end{tikzcd}
\end{equation*}
deux diagrammes de groupes avec $\pi_i\colon H_i\to G_i$ des monomorphismes. Si $X(\pi_1,f_1)$ et $X(\pi_2,f_2)$ sont homotopiquement équivalents au sens filtré, alors il existe un isomorphisme de diagramme entre $\bar{G_1}$ et $\bar{G_2}$. C'est à dire trois isomorphismes $h_{G},h_{H}$ et $h_{H'}$ tel que le diagramme suivant commute.
\begin{equation*}
\begin{tikzcd}
G_1
\arrow{d}{h_G}
&H_1
\arrow[swap]{l}{\pi_1}
\arrow{r}{f_1}
\arrow{d}{h_H}
&H'_1
\arrow{d}{h_{H'}}
\\
G_2
&H_2
\arrow{l}{\pi_2}
\arrow[swap]{r}{f_2}
&H'_2
\end{tikzcd}
\end{equation*}
Réciproquement, si il existe un tel isomorphisme de diagrammes et deux ensembles simpliciaux filtrés $\fil{A}$ et $\fil{B}$ tels que  $X(\pi_1,f_1)$ et $X(\pi_2,f_2)$ soit homotopiquement équivalents au sens filtré à $\RealP{\fil{A}}$ et $\RealP{\fil{B}}$ respectivement, alors $X(\pi_1,f_1)$ et $X(\pi_2,f_2)$ sont homotopiquements équivalents au sens filtré.
\end{prop}

\begin{proof}
Le sens direct provient des calculs précédents, et de l'invariance des groupes d'homotopies filtrés par équivalences d'homotopies filtrées. Réciproquement, par la propriété universelle des $BG$, on obtient des morphismes définis à homotopie près
\begin{equation*}
h_G\colon BG_1\to BG_2,
\end{equation*}
\begin{equation*}
h_H\colon EG_1\times_{G_1}G_1/H_1\to EG_2\times_{G_2}G_2/H_2,
\end{equation*}
et
\begin{equation*}
h_{H'}\colon BH'_1\to BH'_2.
\end{equation*}
Ceux ci induisent une application filtrée
\begin{equation*}
h\colon X(\pi_1,f_1)\to X(\pi_2,f_2).
\end{equation*}
Par construction, $s\pi_1(h)$ est l'isomorphisme de diagramme dont on est parti. Par application du théorème \ref{PremierTheoremeWhitehead} et de la remarque \ref{PremierTheoremeWhiteheadAHomotopiePres}, on déduit que $h$ est une équivalence d'homotopie filtrée.
\end{proof}

\begin{remarque}
La proposition précédente ne garantit pas l'unicité des espaces filtrés ayant la propriété d'Eilenberg-Mac Lane. Elle garantit seulement que deux espaces ayant cette propriété et issus de la construction \ref{GroupesHomotopiesFibre}, sont homotopiquement équivalents au sens filtré.
\end{remarque}

\section{Un plongement est un espace filtré}
\label{SectionPlongementEspaceFiltre}
\subsection{Construction générale}

Soit $X$ un espace topologique et $\{X_i\}_{i\in S}$ un ensemble localement fini de sous espaces fermés de $X$. Alors, on peut munir $X$ d'une filtration
\begin{equation*}
\varphi_S\colon X\to \mathcal{P}(S)
\end{equation*}
où $\mathcal{P}(S)$ est l'ensemble des parties de $S$ ordonné par l'inverse de l'inclusion. C'est à dire 
\begin{equation*}
I\subseteq J\subseteq S\Rightarrow I\geq J.
\end{equation*}
On définit la filtration comme suit. Soit $x\in X$, on pose
\begin{equation*}
\varphi_X(x)=\{X_i\ |\ x\in X_i\}.
\end{equation*}
On vérifie que la filtration est continue. Une base d'ouvert de $\mathcal{P}(S)$ est donnée par les
\begin{equation*}
U_{I}=\{J\ |\ J\subseteq I\}
\end{equation*}
Soit $I\subseteq S$, on a 
\begin{equation*}
\varphi^{-1}(U_{I})=(\cup_{i\in S\setminus I}X_i)^{c}
\end{equation*}
Comme l'ensemble des $X_i$ est localement fini, l'union
\begin{equation*}
\cup_{i\in S\setminus I}X_i
\end{equation*}
est fermée, et donc $\varphi^{-1}(U_I)$ est ouvert. 

\begin{prop}\label{PropositionPlongementPseudoVariete}
Soit $X$ une variété de dimension $n$ et $\{X_i\}_{i\in S}$ un ensemble localement fini de sous variétés de dimensions respectives $n_i$. On suppose que pour tout $I\subseteq S$, et pour tout $x\in\cap_{I}X_i$, il existe un ouvert $U\subseteq X$ contenant $x$ et des homéomorphismes
\begin{equation*}
h\colon U\to \R^n
\end{equation*}
et 
\begin{equation*}
h_i\colon U\cap X_i\to \R^{n_i}
\end{equation*}
pour $i\in I$, ainsi que des applications linéaires
\begin{equation*}
\alpha_i\colon \R^{n_i}\to \R^n
\end{equation*}
tels que, pour tout $i\in I$, la restriction $h_{|U\cap X_i}$ est égale à la composition
\begin{equation*}
U\cap X_i\xrightarrow{h_i}\R^{n_i}\xrightarrow{\alpha_i}\R^n.
\end{equation*}
Alors, l'espace filtré $(X,\varphi_S)$ est une pseudo variété.
\end{prop}

\begin{proof}
Soit $x\in X$. Notons $I=\varphi_S(x)$. Si $I=\emptyset$, $x$ appartient à la sous variété ouverte 
\begin{equation*}
X\setminus (\cup_{i\in S}X_i)
\end{equation*}
et donc, il existe un voisinage de $x$, $U$ tel que $U\simeq \R^n$, avec $\varphi_S(U)=\emptyset\in \mathcal{P}(S)$. Si $I\not=\emptyset$, soit $U$ un voisinage de $x$ vérifiant les hypothèses de la proposition. Notons $V_x$ le sous espace vectoriel de $\R^n$ obtenu comme
\begin{equation*}
V_x=\cap_{i\in I}\alpha_i(\R^{n_i})
\end{equation*}
Et notons $W_x$ son complément orthogonal dans $\R^n$.
On a un isomorphisme canonique
\begin{equation*}
V_x\times W_x\simeq \R^n
\end{equation*}
On a une filtration naturelle sur $W_x$ donnée par
\begin{align*}
\varphi_x\colon W_x &\to \mathcal{P}(S)_{\geq I}\\
y&\mapsto \{i\in I\ |\ y\in \alpha_i(\R^{n_i})\cap W_x\}.
\end{align*}
De plus, cette filtration vérifie $\varphi_x^{-1}(I)=\{0\}$ et pour tout $y\in W_x$, et pour tout $\lambda\in \R\setminus\{0\}$, $\varphi_x(y)=\varphi_x(\lambda y)\in \mathcal{P}(S)_{>I}$. On en déduit que si la sphère unité $S_x\subseteq W_x$ est munie de la filtration induite
\begin{equation*}
\varphi_x\colon S_x\to \mathcal{P}(S)_{>I},
\end{equation*}
on a un homéomorphisme filtré canonique
\begin{equation*}
c(S_x)\simeq W_x
\end{equation*}
Finalement, on obtient que l'homéomorphisme $h$ est un homéomorphisme filtré de la forme
\begin{equation*}
h\colon U\to V_x\times c(S_x)
\end{equation*}
avec $V_x\simeq \R^{n_I}$ pour un certain $n_I$ et $S_x$ est une pseudo-variété, par récurrence sur le cardinal de $I$.
\end{proof}

\begin{exemple}
Soit $X=\mathbb{C}^n$ et $\{X_i\}_{i\in S}$ une collection d'hyperplan de $\mathbb{C}^n$. Alors $(X,\varphi_S)$ est une pseudo-variété. En effet, soit $x\in \mathbb{C}^n$, on peut définir $h$ comme
\begin{align*}
h\colon B(x,\epsilon)&\to \mathbb{C}^n\\
z&\mapsto \frac{z-x}{\epsilon-|z-x|}
\end{align*}
pour $\epsilon$ suffisamment petit. En considérant le sous ensemble ordonné $\varphi_S(X)\subseteq \mathcal{P}(S)$, on obtient ainsi l'ensemble ordonné des intersections de l'arrangement d'hyperplan. 
\end{exemple}

On rappelle que si $N$ est une variété topologique de dimension $n$ et $M\subseteq N$ est une sous variété de dimension $m$, $M$ est localement plate dans $N$ si pour tout point $x\in M$, il existe un voisinage de $x$, $U\subseteq N$ , et un homéomorphisme de paire
\begin{equation*}
h\colon (U,U\cap M)\to (\R^n,\R^m)
\end{equation*}
où $\R^m$ est vu comme le sous espace $\R^m\times\{0_{\R^{n-m}}\}\subseteq \R^n$. Avec cette définition, on a l'exemple suivant.

\begin{exemple}
Si $X$ est une variété topologique et $\{X_i\}_{i\in S}$ est un ensemble localement fini de sous variétés localement plates de $X$ deux à deux disjointes, alors $(X,\varphi_S)$ est une pseudo variété. On s'intéresse dans la section suivante au cas où $X=S^3$ et où les $X_i$ sont homéomorphes à $S^1$. On obtient ainsi les noeuds et les entrelacs comme des structures de pseudo variété sur $S^3$.
\end{exemple}

\subsection{Nœuds et entrelacs}

Soit $\gamma=\coprod_{i\in S}\gamma_i\colon \coprod_{i\in S} S^1\to S^3$ un entrelacs lisse. Notons $X_i=\gamma_i(S^1)\subseteq S^3$, et considérons $(X,\varphi_S)$ la pseudo variété obtenue en appliquant la construction précédente. On remarque que le sous ensemble ordonné $\varphi_S(X)\subseteq \mathcal{P}(S)$ est de la forme
\begin{equation*}
P=\varphi_S(X)=\{\{i\}\ | i\in S\}\cup\{\emptyset\}
\end{equation*}
Et les seuls relations sont de la forme $\{i\}\leq \emptyset$. Pour simplifier les notations à venir, notons $n=|S|$ et identifions $S$ à l'ensemble $\{-n+1,\dots,0\}$. On a alors :
\begin{equation*}
P=\{p_i\ | -n+1\leq i\leq 1\}
\end{equation*}
avec pour seuls relations $p_i\leq p_1$ pour tout $-n+1\leq i\leq 1$.
Chaque composante de l'entrelacs admet un voisinage tubulaire homéomorphe à un tore plein. Autrement dit, on a des voisinages de la forme
\begin{equation*}
c(\pr_1)(S^1\times S^1)\to S^1.
\end{equation*}
On fixe de tels voisinages disjoints deux à deux, (quitte à ajouter des hypothèses sur l'entrelacs pour que de tels voisinages existent). En travaillant composante par composante, on se retrouve dans la situation décrite à la section \ref{SectionConstructionFibreNonFiltre}, et on peut donc appliquer les résultats de la proposition \ref{GroupesHomotopiesFibre} pour calculer les groupes d'homotopies filtrés de $(X,\varphi_S)$. (Voir aussi la proposition \ref{PropositionGroupesHomotopiesFiltreFibreFiltre} pour un calcul plus direct.)
La catégorie $R(P)^{\op}$ est de la forme
\begin{equation*}
\begin{tikzcd}
&p_1
\\
{[}p_{-n+1},p_1]
\arrow{d}
\arrow{ur}
&\dots
&{[}p_{-1},p_1]
\arrow{d}
\arrow{ul}
&{[}p_0,p_1]
\arrow{d}
\arrow{ull}
\\
p_{-n+1}
&
\dots
&
p_{-1}
&
p_0
\end{tikzcd}
\end{equation*} 
et $(X,\varphi_S)$ est connexe au sens filtré. On choisit un pointage $\phi\colon \RealP{N(P)}\to (X,\varphi_S)$. (c'est à dire une collection d'applications continues $[\phi_i\colon [0,1]\to X$, vérifiant $\phi_i(0)\in X_i$, et $\phi_i(t)\in X\setminus(\cup X_i)$ pour tout $t>0$ et $\phi_i(1)=\phi_j(1)$ pour tout $-n+1\leq i,j\leq 0$). On obtient
\begin{equation}\label{Pi1Entrelacs}
s\pi_1((X,\varphi_S),\phi)=\hspace{15pt}
\begin{tikzcd}
&G_{\gamma}
\\
\Z\oplus \Z
\arrow{d}{\pr_1}
\arrow{ur}{\alpha_{-n+1}}
&\dots
&\Z\oplus \Z
\arrow{d}{\pr_1}
\arrow{ul}{\alpha_{-1}}
&\Z\oplus\Z
\arrow{d}{\pr_1}
\arrow[swap]{ull}{\alpha_0}
\\
\Z
&
\dots
&
\Z
&
\Z
\end{tikzcd}
\end{equation}
où $G_{\gamma}=\pi_1(S^3\setminus \cup_{i\in S}X_i,*)$ est le groupe de l'entrelacs (ou du nœud si $n=1$). On a aussi
\begin{equation*}
s\pi_n((X,\varphi_S),\phi)\simeq\hspace{15pt}
\begin{tikzcd}
&G_{\gamma,n}
\\
0
\arrow{d}
\arrow{ur}
&\dots
&0
\arrow{d}
\arrow{ul}
&0
\arrow{d}
\arrow{ull}
\\
0
&
\dots
&
0
&
0
\end{tikzcd}
\end{equation*}
pour tout $n\geq 2$, où $G_{\gamma,n}=\pi_n(S^3\setminus\cup(X_i),*)$. Pour tout $-n+1\leq i\leq 0$, l'application $\alpha_i\colon \Z\oplus \Z\to G_{\gamma}$ est induite par l'inclusion du bord d'un voisinage tubulaire de $X_i$ dans $S^3\setminus (\cup X_i)$. On note que si il n'existe pas de plongement $S^2\to S^3$ n'intersectant aucun des $X_i$ et séparant l'ensemble des $X_i$ en deux sous ensembles non vides, alors $G_{\gamma,n}=0$ \cite{LinkComplementAspherical}, en particulier, c'est le cas si $\gamma$ est un nœud.

On rappelle la notions usuelle d'équivalences entre entrelacs.
\begin{defin}
Deux entrelacs $\gamma,\gamma'\colon \coprod S^1\to S^3$ sont équivalents si il existe un homéomorphisme préservant l'orientation
\begin{equation*}
f\colon S^3\to S^3
\end{equation*}
tel que $\gamma'=f\circ \gamma$.
\end{defin}

On en déduit le résultat suivant

\begin{prop}
Soit $\gamma,\gamma'\colon \coprod S^1\to S^3$ deux entrelacs lisses. On note $(X,\varphi_S)$ et $(X',\varphi'_S)$ les pseudo-variétés associées. Si $\gamma$ et $\gamma'$ sont équivalents, alors $(X,\varphi_S)$ et $(X',\varphi'_S)$ sont homotopiquement équivalents au sens filtrés. En particulier, pour tout pointages $\phi\RealP{N(P)}\to (X,\varphi_S)$ et $\phi'\colon \RealP{N(P)}\to (X',\varphi'_S)$ et pour tout $n\geq 1$, il existe des isomorphismes de diagrammes
\begin{equation*}
s\pi_n((X,\varphi_S),\phi)\simeq s\pi_n((X',\varphi'_S),\phi')
\end{equation*}
\end{prop}

On en déduit que le type d'homotopie filtré de $(X,\varphi_S)$, est un invariant de l'entrelacs.

\begin{proof}
Un homéomorphisme $f\colon S^3\to S^3$ réalisant l'équivalence entre $\gamma$ et $\gamma'$ est en particulier une équivalence d'homotopie filtrée. Comme les groupes d'homotopies filtrés sont invariants par équivalences d'homotopies filtrés, on obtient le résultat voulu. 
\end{proof}

Dans le cas des nœuds, on a la réciproque 

\begin{theo}[\cite{Waldhausen}]
\label{TheoremeGroupesHomotopieFiltresNoeuds}
Soit $\gamma,\gamma'\colon S^1\to S^3$ deux nœuds lisses. Les assertions suivantes sont équivalentes
\begin{enumerate}
\item Il existe un homéomorphisme $f\colon S^3\to S^3$ tel que $f(\gamma(S^1))=\gamma'(S^1)$
\item Il existe un isomorphisme de diagrammes 
\begin{equation*}
g\colon s\pi_1((X,\varphi_S),\phi)\to s\pi_1((X',\varphi'_S),\phi')
\end{equation*}
pour des pointages $\phi$ et $\phi'$.
\item Il existe une équivalence d'homotopie filtrée
\begin{equation*}
h\colon (X,\varphi_S)\to (X',\varphi'_S).
\end{equation*}
\end{enumerate}
\end{theo}

\begin{proof}
Par la proposition précédente, on a $1\Rightarrow 3\Rightarrow 2$. Pour montrer l'implication $2\Rightarrow 1$, on considère $Y$ et $Y'$ les sous variétés à bord de $S^3$ obtenue en enlevant des voisinages tubulaire de $\gamma(S^1)$ et $\gamma'(S^1)$ respectivement. Alors, l'isomorphisme $g$ fournit un isomorphisme $\pi_1(Y,*)\to \pi_1(Y',*)$ envoyant $\pi_1(\partial(Y),*)$ sur $\pi_1(\partial(Y'),*)$. Par
\cite{Waldhausen}, il existe un homéomorphisme $\widetilde{f}\colon Y\to Y'$. Par \cite{GordonLuecke} ceci implique qu'il existe un homéomorphisme $f\colon S^3\to S^3$ tel que $f(\gamma(S^1))=\gamma'(S^1)$.
\end{proof}

\begin{remarque}
L'assertion $1$ du théorème précédent n'implique pas que les nœuds $\gamma$ et $\gamma'$ sont équivalents. En effet, pour cela, il faudrait supposer que $f$ préserve l'orientation de $S^3$ ainsi que celle de $S^1$. Cependant, on remarque que la donnée de ces deux orientations peut se lire sur l'identification (\ref{Pi1Entrelacs}). En effet, le choix du générateur de $\pi_1(\gamma(S^1),*)$ correspond à une orientation de $S^1$. L'orientation de $S^3$ impose des contions supplémentaires sur l'élément de $\pi_1(S^1\times S^1,*)$ envoyé sur $(0,1)\in \Z\oplus \Z$. Finalement, on obtient une équivalence entre les assertions
\begin{itemize}
\item Les noeuds $\gamma$ et $\gamma'$ sont équivalents.
\item Il existe un isomorphisme de diagrammes entre $s\pi_1((X,\varphi_s),*)$ et $s\pi_1((X',\varphi'_S),*)$, se factorisant par un isomorphisme de la forme :
\begin{equation*}
\begin{tikzcd}
\Z
\arrow[swap]{d}{\Id}
&\Z\oplus\Z
\arrow[swap]{l}{\pr_1}
\arrow{d}{g_1}
\arrow{r}{\alpha}
&G_{\gamma}
\arrow{d}{g_2}
\\
\Z
&\Z\oplus\Z
\arrow{l}{\pr_1}
\arrow[swap]{r}{\alpha'}
&G_{\gamma'}
\end{tikzcd}
\end{equation*}
\end{itemize}
où $g_1(0,1)=(0,k)$ avec $k>1$.
après deux identifications (\ref{Pi1Entrelacs}) respectant l'orientation.
\end{remarque}

\begin{remarque}
Nous avons vu qu'à chaque nœud ou entrelacs (lisse) correspondait une structure de pseudo variété sur $S^3$. Réciproquement, toute structure de pseudo variété sur $S^3$ ne contenant que des strates de codimension $2$ provient d'un entrelacs. En effet, dans ce cas, les strates sont des sous variétés plongées de dimension $1$, c'est à dire des copies de $S^1$. Il n'est donc pas surprenant que le type d'homotopie de la pseudo variété associée contienne de l'information sur l'entrelacs.
\end{remarque}

\section{Des fibrés filtrés}
\label{SectionFibresFiltres}

La construction de la section \ref{ConstructionFibre} met en jeu des objets non-filtrés. Pour cette raison, elle ne permet d'obtenir que des objets filtrés simples. Cependant, elle s'adapte sans grande difficulté au cas d'un fibré localement trivial de fibre filtré. On reprend cette construction pour inclure ce cas-ci.
Soit
\begin{equation}\label{DefinitionFibreFiltre}
F\hookrightarrow E\xrightarrow{\pi} B
\end{equation}
Un fibré localement trivial, avec $\varphi_F\colon  F\to P$ une filtration de $F$. Soient $U$ et $V$ deux ouverts trivialisant pour $\pi$. Comme $\pi$ est un fibré localement trivial, on a un diagramme commutatif de la forme suivante
\begin{equation*}
\begin{tikzcd}[column sep=large]
U\cap V\otimes \fil{F}
\arrow{r}{(h_U)_{|U\cap V}}
\arrow[swap]{dr}{\pr_{U\cap V}}
&\pi^{-1}(U\cap V)
\arrow{d}{\pi}
&U\cap V\otimes \fil{F}
\arrow[swap]{l}{(h_V)_{|U\cap V}}
\arrow{dl}{\pr_{U\cap V}}
\\
&U\cap V
\end{tikzcd}
\end{equation*}
où $h_{U}$ et $h_{V}$ sont des homéomorphismes. On suppose que pour tout ouvert trivialisant $U,V$, la composée $((h_V)_{|U\cap V})^{-1}\circ (h_U)_{|U\cap V}$ est un homéomorphisme filtré. Autrement dit, on suppose que les fonctions de transitions préservent la filtration de $F$. Dans ce cas, on a une filtration sur $E$, définie localement par 
\begin{equation*}
(\varphi_E)_{|\pi^{-1}(U)}\colon \pi^{-1}(U)\xrightarrow{h_U^{-1}} U\otimes \fil{F}\xrightarrow{\varphi_F}P
\end{equation*}
Finalement, de même que dans la section \ref{ConstructionFibre} on peut définir le cône de $E$ le long de $\pi$ comme l'espace filtré obtenu comme la somme amalgammée suivante
\begin{equation*}
\begin{tikzcd}
E
\arrow{r}{\pi}
\arrow{d}{i_0}
&B
\arrow[swap]{d}{i_B}
\arrow[bend left=18]{ddr}{\Id_B}
\\
E\times [0,1[
\arrow{r}
\arrow[swap, bend right=18]{drr}{\pi\circ\pr_E}
&c_{\pi}(E)
\arrow[swap, near start]{dr}{c_{\pi}(\pi)}
\\
&&B
\end{tikzcd}
\end{equation*}
On rappelle que $c(P)$ désigne l'ensemble ordonné obtenu à partir de $P$ en ajoutant un élément minimal $-\infty$ (Définition \ref{DefinitionCone}). On considère la filtration :
\begin{align*}
c_{\pi}(\varphi_E)\colon E\times [0,1[&\to c(P)\\
(x,t)&\mapsto \left\{\begin{array}{cl}
\varphi(x) &\text{ si $t>0$}\\
-\infty &\text{ si $t=0$}
\end{array}\right.
\end{align*}
Ainsi que la filtration triviale sur $B$
\begin{equation*}
\varphi_B\colon B\to \{-\infty\}\to c(P)
\end{equation*}
Par propriété universelle de la somme amalgamée, ceci définit une fibration
\begin{equation*}
c_{\pi}(\varphi_E)\colon c_{\pi}(E)\to P
\end{equation*}
On a maintenant le résultat suivant, dont la proposition \ref{FibrePseudoVariete} est un corollaire immédiat dans le cas où $P=\{*\}$.
\begin{prop}
Si l'espace filtré $\fil{F}$ est coniquement stratifié, alors $(c_{\pi}(E),c_{\pi}(\varphi_E))$ l'est aussi. De plus, si $B$ est une variété (de dimension $m$) et $\fil{F}$ est une pseudo variété (de dimension $n$), alors $(c_{\pi}(E),c_{\pi}(\varphi_E))$ est une pseudo variété de dimension $m+n+1$.
\end{prop}

\begin{proof}
Montrons que $(c_{\pi}(E),c_{\pi}(\varphi_E))$ est coniquement stratifié. Soit $x\in c_{\pi}(E)$. Si $c_{\pi}(\varphi_E)(x)\in P\subset c(P)$, alors $x\in E\times ]0,1[$. Soit $U$ un ouvert trivialisant pour $\pi$ contenant $c_{\pi}(\pi)(x)$. Alors, on a un voisinage de $x$ de la forme $U\times ]0,1[\times F\subset E\times ]0,1[$. Comme $F$ est coniquement stratifié, il existe un espace filtré $\fil{L}$ et un voisinage de $\pr_F\circ h_U(x)$ homéomorphe au sens filtré à $V\otimes c\fil{L}$. Mais alors, on a un voisinage de $x$ homéomorphe au sens filtré à 
\begin{equation*}
(U\times ]0,1[\times V)\otimes c\fil{L}.
\end{equation*}
Si $c_{\pi}(\varphi_E)(x)=-\infty$, soit $U\subset B$ un ouvert trivialisant pour $\pi$ contenant $c_{\pi}(\pi)(x)$. Alors, $x$ a un voisinage de la forme
\begin{equation*}
U\otimes c\fil{F}.
\end{equation*}
Dans le cas où $B$ est une variété, quitte à restreindre $U$, on peut supposer que $U\simeq \R^m$. Si de plus $F$ est une pseudo variété, on obtient ainsi que $(c_{\pi}(E),c_{\pi}(\varphi_E))$ est une pseudo variété.
\end{proof}

On souhaite calculer les groupes d'homotopie filtrés de $(c_{\pi}(E),c_{\pi}(\varphi_E))$. Dans ce but, explicitons d'abord la catégorie $R(c(P))$ des simplexes non dégénérés de $N(c(P))$ (voir la définition \ref{DefinitionRP}). Les objets de $R(c(P))$ sont d'une des deux formes suivantes :
\begin{itemize}
\item $[p_0,\dots,p_n]$, où $p_0<\dots<p_n\in P$ est une suite strictement croissante, non vide, d'éléments de $P$
\item $[-\infty, p_1,\dots,p_n]$ où $p_1<\dots<p_n\in P$ est une suite strictement croissante, possiblement vide, d'éléments de $P$.
\end{itemize}
Les ensemble de morphismes de $R(c(P))$ contiennent au plus un élément, et on a
\begin{align*}
\Hom_{R(c(P))}([p_0,\dots,p_n],[q_0,\dots,q_m])\not= \emptyset &\Leftrightarrow \{p_0,\dots,p_n\}\subseteq \{q_0,\dots,q_m\}\\
\Hom_{R(c(P))}([p_0,\dots,p_n],[-\infty,q_1,\dots,q_m])\not= \emptyset &\Leftrightarrow \{p_0,\dots,p_n\}\subseteq \{q_1,\dots,q_m\}\\
\Hom_{R(c(P))}([-\infty,p_1,\dots,p_n],[-\infty,q_1,\dots,q_m])\not= \emptyset &\Leftrightarrow \{p_1,\dots,p_n\}\subseteq \{q_1,\dots,q_m\}\\
\Hom_{R(c(P))}([-\infty,p_1,\dots,p_n],[q_0,\dots,q_m])= \emptyset
\end{align*}
On remarque que l'inclusion canonique $P\subset c(P)$ induit un foncteur pleinement fidèle
\begin{align*}
R(P)&\to R(c(P))\\
[p_0,\dots,p_n]&\mapsto [p_0,\dots,p_n].
\end{align*}
Ce qui permet de considérer $R(P)$ comme une sous catégorie pleine de $R(c(P))$. De plus, on dispose d'un second foncteur pleinement fidèle
\begin{align*}
R(P)&\to R(c(P))\\
[p_0,\dots,p_n]&\mapsto [-\infty,p_0,\dots,p_n].
\end{align*}
Ce second foncteur fournit une seconde copie de $R(P)$ dans $R(c(P))$, disjointe de la première. On note cette seconde copie $[-\infty]R(P)$, car elle est obtenue en adjoignant $-\infty$ aux objets de $R(P)$. Finalement, on observe que à l'exception de $[-\infty]$, tout les objets de $R(c(P))$ sont contenus dans une des deux sous catégories $R(P)$ ou $[-\infty]R(P)$. On se tourne ensuite vers les morphismes. Par construction, il n'existe aucun morphisme depuis les objets de $[-\infty]R(P)$ vers les objets de $R(P)$ ou vers $[-\infty]$. Réciproquement, si $[-\infty,p_0,\dots,p_n]$ est un objet de $[-\infty]R(P)$, on a les morphismes suivants dans $R(c(P))$.
\begin{equation*}
[-\infty]\to[-\infty,p_0,\dots,p_n]
\end{equation*}
et
\begin{equation*}
[p_0,\dots,p_n]\to [-\infty,p_0,\dots,p_n]
\end{equation*}
Finalement, on peut représenter symboliquement la catégorie $R(c(P))$ comme suit.
\begin{equation*}
\begin{tikzcd}
\{-\infty\}
\arrow{r}
&{[}-\infty]R(P)
&R(P)
\arrow{l}
\end{tikzcd}
\end{equation*}
Cette représentation est justifiée par la proposition suivante
\begin{prop}\label{PropositionFoncteurRcP}
Soit $\C$ une catégorie. Un foncteur $F\colon R(c(P))^{\op}\to \C$ correspond de façon unique à la donnée de
\begin{itemize}
\item deux foncteurs $F_{[-\infty]R(P)},F_{R(P)}\colon R(P)^{\op}\to\C$,
\item un objet $F_{[-\infty]}\in \C$,
\item deux transformations naturelles $\alpha\colon F_{[-\infty]R(P)}\to F_{R(P)}$ et $\beta\colon F_{[-\infty]R(P)}\to F_{[-\infty]}$, où l'objet $F_{[-\infty]}$ est identifié avec le foncteur constant.
\end{itemize}
\end{prop}

\begin{proof}
Soit $F\colon R(c(P))^{\op}\to \C$ un foncteur. Comme les notations le suggèrent, on pose
\begin{align*}
F_{[-\infty]R(P)}\colon R(P)^{\op}&\to\C\\
[p_0,\dots,p_n]&\mapsto F([-\infty,p_0,\dots,p_n]),
\end{align*}
\begin{align*}
F_{R(P)}\colon R(P)^{op}&\to\C\\
[p_0,\dots,p_n]&\mapsto F([p_0,\dots,p_n]),
\end{align*}
et
\begin{equation*}
F_{[-\infty]}=F([-\infty])\in\C.
\end{equation*}
De plus, pour $[p_0,\dots,p_n]$ un objet de $R(P)$, l'image par $F$ du morphisme
\begin{equation*}
[p_0,\dots,p_n]\to [-\infty,p_0,\dots,p_n]
\end{equation*}
fournit le morphisme 
\begin{equation*}
\alpha_{[p_0,\dots,p_n]}\colon F_{[-\infty]R(P)}([p_0,\dots,p_n])\to F_{R(P)}([p_0,\dots,p_n])
\end{equation*}
De même, l'image par $F$ du morphisme
\begin{equation*}
[-\infty]\to [-\infty,p_0,\dots,p_n]
\end{equation*}
fournit le morphisme
\begin{equation*}
\beta_{[p_0,\dots,p_n]}\colon F_{[-\infty]R(P)}([p_0,\dots,p_n])\to F_{[-\infty]}.
\end{equation*}
Comme $F$ est un foncteur, $\alpha$ et $\beta$ vérifient les conditions de naturalités nécessaires et sont donc des transformations naturelles. Réciproquement, la donnée de foncteurs $F_{[-\infty]R(P)}$ et $F_{R(P)}$, d'un objet $F_{[-\infty]}$ et de transformations naturelles $\alpha$ et $\beta$ permet de définir un foncteur $F\colon R(c(P))^{\op}\to\C$.
\end{proof}

\begin{remarque}\label{RemarqueExpressionSPiKFibreFiltre}
En vertu de la proposition \ref{PropositionFoncteurRcP}, les groupes d'homotopie filtrés de $(c_{\pi}(E),c_{\pi}(\varphi_E))$ peuvent donc être exprimés comme la donnée de deux foncteurs, un ensemble et deux transformations naturelles. Par commodité, et pour mettre en lumière les similarités avec la section \ref{ConstructionFibre}, on représentera cette donnée sous la forme suivante.
\begin{equation*}
\begin{tikzcd}[column sep = -20pt]
&s\pi_k((c_{\pi}(E),c_{\pi}(\varphi_E)),\phi)_{[-\infty]R(P)}
\arrow[swap]{dl}{\beta}
\arrow{dr}{\alpha}
\\
s\pi_k((c_{\pi}(E),c_{\pi}(\varphi_E)),\phi)_{[-\infty]}
&&
s\pi_k((c_{\pi}(E),c_{\pi}(\varphi_E)),\phi)_{R(P)}
\end{tikzcd}
\end{equation*}
\end{remarque}

\begin{remarque}\label{RemarqueMorphismeInduitFibreFiltre}
On remarque que pour tout objet de $R(P)$, $[p_0,\dots,p_n]$  le morphisme $\pi$ induit un morphisme
\begin{equation*}
\pi_*\colon s\pi_0\fil{E}([p_0,\dots,p_n])\to\pi_0(B)
\end{equation*}
En effet, un élément de $s\pi_0\fil{E}$ est une classe d'homotopie filtrée d'applications filtrées de la forme
\begin{equation*}
\phi\colon\RealP{[p_0,\dots,p_n]}\to \fil{E}.
\end{equation*}
Comme $\RealP{[p_0,\dots,p_n]}$ est connexe $\pi(\phi(\RealP{[p_0,\dots,p_n]}))$ est un sous espace connexe de $B$, et correspond donc à une composante connexe de $B$. Cette composante connexe ne dépend pas du choix de représentant dans la classe d'homotopie filtrée, et l'application $\pi_*$ est donc bien définie. Fixons maintenant un pointage 
\begin{equation*}
\phi\colon\RealP{[p_0,\dots,p_n]}\to \fil{E}
\end{equation*}
tel que la composition $\pi\circ\phi\colon \Real{[p_0,\dots,p_n]}\to B$ soit constante, et notons $\phi_B\in B$ son image. On a alors comme précédemment un morphisme bien défini
\begin{equation*}
\pi_*\colon s\pi_n(\fil{E},\phi)([p_0,\dots,p_n])\to\pi_n(B)
\end{equation*}
\end{remarque}

avec les remarques précédentes, on peut finalement procéder au calcul des groupes d'homotopie filtrés de $(c_{\pi}(E),c_{\pi}(\varphi_E))$.

\begin{prop}\label{PropositionGroupesHomotopiesFiltreFibreFiltre}
Avec les conventions de la remarque \ref{RemarqueExpressionSPiKFibreFiltre}, on a
\begin{equation}\label{Pi0FibreFiltre}
s\pi_0(c_{\pi}(E),c_{\pi}(\varphi_E))\simeq 
\begin{tikzcd}
&s\pi_0\fil{E}
\arrow[swap]{dl}{\pi_*}
\arrow{dr}{\Id}
\\
\pi_0(B)
&&
s\pi_0\fil{E}
\end{tikzcd}
\end{equation}
De plus, si $\phi\colon \RealP{V}\to (c_{\pi}(E),c_{\pi}(\varphi_E))$ est un pointage tel que la composition $c_{\pi}(\pi)\circ\phi\colon \RealP{V}\to B$ soit constante, on a, pour tout $k\geq 1$,
\begin{equation}\label{PiKFibreFiltre}
s\pi_k((c_{\pi}(E),c_{\pi}(\varphi_E)),\phi)\simeq 
\begin{tikzcd}
&s\pi_k(\fil{E},\phi_E)
\arrow[swap]{dl}{\pi_*}
\arrow{dr}{\Id}
\\
\pi_k(B,\phi_B)
&&
s\pi_k(\fil{E},\phi_E)
\end{tikzcd}
\end{equation}
où $\phi_B$ est l'image de $c_{\pi}(\pi)\circ\phi$, et $\phi_E$ est la restriction de $\phi$ à $\Real{V}\cap \Real{N(P)}\subset \Real{N(c(P))}$.
\end{prop}

\begin{proof}
Par construction, $c_{\pi}(\varphi)^{-1}(-\infty)=B$. On en déduit les termes de la forme $\pi_k(B,\phi_B)$ dans les diagrammes (\ref{Pi0FibreFiltre}) et (\ref{PiKFibreFiltre}). Soit $[p_0,\dots,p_n]\in R(P)$. Par construction de $c_{\pi}(E)$, on a
\begin{align*}
C^0_{c(P)}(\Real{[p_0,\dots,p_n]}_{c(P)},(c_{\pi}(E),c_{\pi}(\varphi_E)))&\simeq C^0_P(\RealP{[p_0,\dots,p_n]},]0,1[\otimes\fil{E})\\
&\sim C^0_P(\RealP{[p_0,\dots,p_n]},\fil{E})
\end{align*}
On en déduit les termes de droites dans les diagrammes (\ref{Pi0FibreFiltre}) et (\ref{PiKFibreFiltre}). 
Par ailleurs, on a une application de restriction bien définie
\begin{align*}
\text{r}\colon C^0_{c(P)}(\Real{[-\infty,p_0,\dots,p_n]}_{c(P)},(c_{\pi}(E),c_{\pi}(\varphi_E)))&\to C^0_P(\RealP{[p_0,\dots,p_n]},]0,1[\otimes\fil{E})\\
&\sim C^0_P(\RealP{[p_0,\dots,p_n]},\fil{E})
\end{align*}
Montrons que c'est une équivalence d'homotopie. Si $n=0$ la preuve de la proposition \ref{GroupesHomotopiesConeOuvertFibre} s'applique. Si $n\geq 1$, on constate qu'on a un homéomorphisme filtré $\Real{[-\infty,p_0,\dots,p_n}_{c(P)}\simeq \bar{c}(\RealP{[p_0,\dots,p_n]})$.
De plus, on fait aussi l'identification $\RealP{[p_0,\dots,p_n]}\simeq \bar{c}(\Real{[p_1,\dots,p_n]}_{P_{>p_0}})$. En particulier, les points de $\Real{[-\infty,p_0,\dots,p_n]}$ sont représentés sous la forme $((y,s),t)\in \bar{c}(\bar{c}(\Real{ [p_1,\dots,p_n]}))$, avec $y\in \Real{[p_1,\dots,p_n]}$, $s,t\in [0,1]$. Avec ces notations, on définit une section de $r$.
\begin{align*}
s\colon C^0_P(\RealP{[p_0,\dots,p_n]},\fil{E})&\to  C^0_{c(P)}(\Real{[-\infty,p_0,\dots,p_n]}_{c(P)},(c_{\pi}(E),c_{\pi}(\varphi_E)))\\
\sigma\colon \RealP{[p_0,\dots,p_n]}\to\fil{E}&\mapsto 
\left\{
\begin{array}{ccc}
\Real{[-\infty,p_0,\dots,p_n]}&\to&(c_{\pi}(E),c_{\pi}(\varphi_E)))\\
((y,s),t)&\mapsto& \left\{\begin{array}{cl}
(\sigma(y,st),t)\in  E\times ]0,1[ & \text{ si $t>0$}\\
\pi(\sigma(y,0)) \in B &\text{ si $t=0$}
\end{array}\right.
\end{array}\right.
\end{align*}
On vérifie que $r\circ s=\Id$. Soit $\sigma\colon \RealP{[p_0,\dots,p_n]}\to \fil{E}$ un élément de $C^0_P(\RealP{[p_0,\dots,p_n]},\fil{E})$. On a 
\begin{equation*}
s(\sigma)((y,s),t)=\left\{\begin{array}{cl}
(\sigma(y,st),t) & \text{ si $t>0$}\\
\pi(\sigma(y,0)) &\text{ si $t=0$}
\end{array}\right.
\end{equation*}
De plus, on a $r(s(\sigma))=\pr_E\circ s(\sigma)_{|\RealP{[p_0,\dots,p_n]}}=\pr_E\circ s(\sigma)_{|t=1}$.
On en déduit que $r\circ s(\sigma)(y,s)=\sigma(y,s)$ pour tout $(y,s)\in \Real{[p_0,\dots,p_n]}$ et donc que $r\circ s= \Id$. On construit l'homotopie entre $s\circ r$ et $\Id$ en deux étapes. Pour $\tau\colon \Real{[-\infty,p_0,\dots,p_n]}_{c(P)}\to (c_{\pi}(E),c_{\pi}(\varphi_E))$, on a la restriction 
\begin{align*}
\Real{[-\infty,p_0,\dots,p_n]}_{c(P)}\setminus\{-\infty\}&\to (c_{\pi}(E),c_{\pi}(\varphi_E))\setminus B\simeq (E\times ]0,1[,\varphi_E\circ \pr_E)\\
((y,s),t)&\mapsto (\tau_E((y,s),t),\tau_{]0,1[}((y,s)t))
\end{align*}
Avec ces notations, et en notant $C=C^0_{c(P)}(\Real{[-\infty,p_0,\dots,p_n]}_{c(P)},(c_{\pi}(E),c_{\pi}(\varphi_E)))$, on considère l'homotopie
\begin{align*}
H_1\colon C\times [0,1]&\to C\\
(\tau,u)&\mapsto \left\{\begin{array}{ccc}
\Real{[-\infty,p_0,\dots,p_n]}_{c(P)}&\to &(c_{\pi}(E),c_{\pi}(\varphi_E))\\
((y,s),t)&\mapsto &\left\{\begin{array}{cl}
(\tau_E((y,s),t),u\tau_{]0,1[}((y,s),t)+(1-u)t) &\text{ si $t>0$}\\
\tau((y,s),0) &\text{ si $t=0$}
\end{array}\right.
\end{array}\right.
\end{align*}
Ainsi que l'homotopie
\begin{align*}
H_2\colon C\times [0,1]&\to C\\
(\tau,u)&\mapsto \left\{\begin{array}{ccc}
\Real{[-\infty,p_0,\dots,p_n]}_{c(P)}&\to &(c_{\pi}(E),c_{\pi}(\varphi_E))\\
((y,s),t)&\mapsto &\left\{\begin{array}{cl}
\tau((y,s),t) &\text{ si $t\geq u$}\\
(\tau_E((y,st/u),u),t) &\text{ si $0<t\leq u$}\\
\pi\tau_E((y,0),u) &\text{ si $0=t<u$}
\end{array}\right.
\end{array}\right.
\end{align*}
Alors, en composant l'homotopie $H_1$ et $H_2$, on obtient l'homotopie souhaitée entre $s\circ r$ et $\Id$. Finalement, on en déduit les terme du milieu dans les diagrammes (\ref{Pi0FibreFiltre}) et (\ref{PiKFibreFiltre}). Par les calculs précédents, la transformation naturelle de droite est $\Id$,  et celle de gauche est donnée par les morphismes $\pi_*$ considérés dans la remarque \ref{RemarqueMorphismeInduitFibreFiltre}.
\end{proof}

\begin{remarque}\label{RemarqueGroupesHomotopiesFiltreFibreFiltre}
De même que dans le cas de la proposition \ref{GroupesHomotopiesFibre}, on peut considérer $X(\pi,f)$ obtenu en considérant la somme amalgamée
\begin{equation*}
\begin{tikzcd}
\fil{E}
\arrow{d}{i_1}
\arrow{r}{f}
&\fil{E'}
\arrow{d}
\\
(\bar{c}_{\pi}(E),\bar{c}_{\pi}(\varphi_E))
\arrow{r}
&X(\pi,f)
\end{tikzcd}
\end{equation*}
On obtient alors, un résultat similaire à celui de la proposition \ref{GroupesHomotopiesFibre}. Avec les pointages appropriés, on obtient
\begin{equation*}
s\pi_0(X(\pi,f))\simeq 
\begin{tikzcd}
&s\pi_0(\fil{E}
\arrow[swap]{dl}{\pi_*}
\arrow{dr}{f_*}
\\
\pi_0(B)
&&
s\pi_0(\fil{E'}
\end{tikzcd}
\end{equation*}
et,
\begin{equation*}
s\pi_n(X(\pi,f),\phi)\simeq 
\begin{tikzcd}
&s\pi_n(\fil{E},\phi_{E})
\arrow[swap]{dl}{\pi_*}
\arrow{dr}{f_*}
\\
\pi_n(B,\phi_B)
&&
s\pi_n(\fil{E'},f\circ\phi_{E})
\end{tikzcd}
\end{equation*}
\end{remarque}

\begin{remarque}
Les propositions \ref{GroupesHomotopiesConeOuvertFibre} et \ref{GroupesHomotopiesFibre} sont des cas particuliers de la proposition \ref{PropositionGroupesHomotopiesFiltreFibreFiltre} et de la remarque \ref{RemarqueGroupesHomotopiesFiltreFibreFiltre} dans le cas où $P=\{*\}$, et donc $c(P)\simeq \{p_0<p_1\}$.
\end{remarque}

\chapter{Une catégorie modèle pour les espaces filtrés}
\label{ChapitreCMFTopP}
L'objet de ce chapitre est de construire et d'étudier une structure de modèle sur la catégorie des espaces stratifiés. D'après les résultats du chapitre \ref{ChapitreGroupesHomotopiesFiltresEspaces} et les exemples du chapitre \ref{ChapitreExemples}, il parait raisonnable de demander que les équivalences faibles soient caractérisées par les groupes d'homotopie filtrés. D'autre part, l'approche naturelle est de transporter la structure de modèle obtenue sur $\sS_P$ le long de l'adjonction $\RealP{-}\colon \sS_P\to \Top_P\colon \Sing_P$. Cependant, plusieurs difficultés se présentent immédiatement. Si $\fil{X}$ est un espace filtré, on s'attendrait à ce que $\Sing_P\fil{X}$ soit un ensemble simplicial filtré fibrant, comme dans le cas non filtré. Cependant, comme on l'a vu dans l'exemple \ref{ExempleEspaceNonFibrant}, ce n'est pas toujours le cas. D'autre part, les constructions classiques permettant de factoriser tout morphisme $f\colon X\to Y$ en une cofibration (triviale) $i\colon X\to Z$ suivie d'une fibration (triviale) $p\colon Z\to Y$, se généralisent mal. En effet, dans le cas filtré, $\Sing_P(p)$ n'est en général pas une fibration. 

Toutefois, on a vu avec la proposition \ref{DiagrammesTopologiquesIsomorphes} que les foncteurs $D, D^{\sS}\circ\Sing_P$ et $\Sing\circ D^{\Top}$ sont isomorphes comme foncteurs $\Top_P\to \Diag_P$. En particulier, ceci implique que pour tout espace filtré $\fil{X}$, le diagramme $D\fil{X}$ est un objet fibrant de $\Diag_P$ pour la structure de modèle de la proposition \ref{CategorieModeleDiagramme}. De plus, par définition des groupes d'homotopies filtrés, une application filtrée $f\colon \fil{X}\to\fil{Y}$ induit un isomorphisme sur tous les groupes d'homotopie filtrés si et seulement si $D(f)$ est une équivalence faible de $\Diag_P$. Ceci suggère de construire une structure de modèle sur $\Top_P$ par transport de la structure de modèle sur $\Diag_P$. Cependant, on a vu avec la remarque \ref{GroupesHomotopiesFiltresReduits} que le foncteur $s\pi_n$ était entièrement déterminé par sa restriction à la sous catégorie pleine $R(P)\subset\Delta(P)$ (voir la définition \ref{DefinitionRP}). Ainsi, on choisit plutôt de travailler avec la catégorie $\DiagR_P$ des diagrammes simpliciaux réduits.

Dans la première section de ce chapitre, on définit la catégorie des diagrammes réduits $\DiagR_P$. On décrit ensuite une adjonction entre les catégories $\DiagR_P$ et $\Top_P$ (Proposition \ref{PropositionColim}), et on montre que la catégorie $\DiagR_P$ est une catégorie modèle à engendrement cofibrant (Proposition \ref{CategorieModeleDiagrammesReduits}).

Dans la section \ref{SectionCMFTopP}, on applique \cite{Hess} pour construire une structure de modèle sur la catégorie $\Top_P$ par transport de la structure sur $\DiagR_P$. C'est le théorème \ref{CategorieModeleTopP}.

Finalement, dans la section \ref{SectionTopNP}, on définit la catégorie des espaces fortement filtrés $\Top_{N(P)}$. La catégorie des espaces fortement filtrés est similaire à celle des espaces filtrés, mais les filtrations sont de la forme $\varphi_X\colon X\to \Real{N(P)}$. Dans la section \ref{SectionCMFTopNP}, après avoir défini cette catégorie, on montre qu'elle vérifie des propriétés similaires à celle des espaces filtrés. C'est notamment une catégorie modèle (Voir théorème \ref{CategorieModeleTopNP}). Dans la section \ref{SectionQETopNPTopP}, on montre que les catégories modèles $\Top_P$ et $\Top_{N(P)}$ sont Quillen équivalentes (Théorème \ref{EquivalenceQuillenTopPTopNP}). 
Enfin, dans la section \ref{SectionQETopNPDiagRP}, on montre qu'on a une équivalence de Quillen entre $\Top_{N(P)}$ et $\DiagR_P$, (Théorème \ref{EquivalenceQuillenTopNPDiagRP}) ce qui implique immédiatement que les catégories de modèle $\Top_{P}$ et $\DiagR_P$ sont Quillen équivalentes.

\section{La catégorie des diagrammes simpliciaux réduits}

\subsection{Adjonction avec les espaces filtrés}
On a vu dans la remarque \ref{GroupesHomotopiesFiltresReduits} que les groupes d'homotopie filtrés étaient complètement déterminés par leur restriction à la sous catégorie des simplexes non dégénérés $R(P)\subset \Delta(P)$. Ceci suggère de travailler avec la catégorie des diagrammes sur $R(P)$, et non sur $\Delta(P)$, pour définir une structure de modèle sur $\Top_P$.

\begin{defin}
On note $\DiagR_P$ la catégorie des diagrammes simpliciaux réduits. Elle est définie comme
\begin{equation*}
\DiagR_P=\Fun(R(P)^{\op},\sS)
\end{equation*}
L'inclusion $R(P)\to \Delta(P)$ induit un foncteur de restriction
\begin{equation*}
\Diag_P\to\DiagR_P.
\end{equation*}
Dans ce chapitre, on considérera tous les diagrammes sous leur forme réduite, c'est à dire qu'on identifiera les objets de $\Diag_P$ avec leurs images dans $\DiagR_P$. En particulier, on notera aussi $D$ le foncteur diagramme réduit
\begin{equation*}
D\colon\Top_{P}\to \DiagR_P,
\end{equation*}
obtenu en composant le foncteur $D$ de la définition \ref{DefinitionFoncteurDTopologique} avec la restriction.
\end{defin}


\begin{prop}\label{PropositionColim}
Le foncteur $D$ admet un adjoint à gauche
\begin{equation*}
\Colim\colon \DiagR_P\to\Top_P.
\end{equation*}
\end{prop}

\begin{proof}
On considère la sous catégorie pleine $\mathcal{C}\subset R(P)^{\op}\times R(P)$ contenant les objets de la forme $(\Delta^{\varphi},\Delta^{\psi})$ avec $\Delta^{\psi}\subseteq \Delta^{\varphi}$. 
On définit un foncteur
\begin{align*}
-\otimes R(P)\colon\DiagR_P&\to\Fun(\mathcal{C},\sS_P)\\
F&\mapsto \left\{\begin{array}{ccc}
C&\to&\sS_P\\
(\Delta^{\varphi},\Delta^{\psi})&\mapsto &F(\Delta^{\varphi})\otimes\Delta^{\psi}
\end{array}\right.
\end{align*}
Ceci permet de définir le foncteur $\Colim$ comme 
\begin{align*}
\Colim\colon \DiagR_P&\to \Top_P\\
F&\mapsto \RealP{\colim_{\mathcal{C}}F\otimes R(P)}=\RealP{\colim_{(\Delta^{\varphi},\Delta^{\psi})\in \mathcal{C}}F(\Delta^{\varphi})\otimes\Delta^{\psi}}
\end{align*}

Montrons que $\Colim$ est bien un adjoint à gauche pour $D$. Soient $F\colon R(P)^{\op}\to \sS$ un foncteur, et $\fil{X}$ un espace filtré. Un élément de $\Hom_{\DiagR_P}(F,D(\fil{X}))$ est une collection de morphismes $\{f_{\Delta^{\varphi}}\}_{\Delta^{\varphi}\in R(P)}$ où pour tout $\Delta^{\varphi}\in R(P)$
\begin{equation*}
f_{\Delta^{\varphi}}\colon F(\Delta^{\varphi})\to\Map(\RealP{\Delta^{\varphi}},\fil{X})
\end{equation*}
est un morphisme de $\sS$, et telle que pour tout morphisme de $R(P)$, $i\colon\Delta^{\psi}\to \Delta^{\varphi}$, on a un diagramme commutatif
\begin{equation*}
\begin{tikzcd}
F(\Delta^{\varphi})
\arrow{r}{f_{\Delta^{\varphi}}}
\arrow{d}{F(i)}
&\Map(\RealP{\Delta^{\varphi}},\fil{X})
\arrow{d}{i^*}
\\
F(\Delta^{\psi})
\arrow{r}{f_{\Delta^{\psi}}}
&\Map(\RealP{\Delta^{\psi}},\fil{X})
\end{tikzcd}
\end{equation*}
D'autre part, par propriété universelle de la colimite, un élément de $\Hom_{\Top_P}(\Colim F,\fil{X})$ correspond à une collection de morphismes $\{g_{(\Delta^{\varphi},\Delta^{\psi})}\}_{(\Delta^{\varphi},\Delta^{\psi})\in \mathcal{C}}$, où pour tout $(\Delta^{\varphi},\Delta^{\psi})\in \mathcal{C}$, 
\begin{equation*}
g_{(\Delta^{\varphi},\Delta^{\psi})}\colon F(\Delta^{\varphi})\otimes\RealP{\Delta^{\psi}}\to \fil{X}
\end{equation*}
est un morphisme de $\Top_P$, telle que pour tout diagramme commutatif dans $R(P)$ 
\begin{equation*}
\begin{tikzcd}
\Delta^{\psi_1}
\arrow{r}
\arrow{d}
&\Delta^{\psi_2}
\arrow{d}
\\
\Delta^{\varphi_1}
&\Delta^{\varphi_2}
\arrow{l}
\end{tikzcd}
\end{equation*}
on a un diagramme commutatif 
\begin{equation*}
\begin{tikzcd}[column sep=huge]
 F(\Delta^{\varphi_1})\otimes\RealP{\Delta^{\psi_1}}
\arrow{r}{g_{(\Delta^{\varphi_1},\Delta^{\psi_1})}}
\arrow{d}
&\fil{X}
\arrow{d}{\Id}
\\
 F(\Delta^{\varphi_2})\otimes\RealP{\Delta^{\psi_2}}
\arrow{r}{g_{(\Delta^{\varphi_2},\Delta^{\psi_2})}}
&\fil{X}
\end{tikzcd}
\end{equation*}
Avec ces observations, on a la bijection suivante.
\begin{align*}
\Hom_{\DiagR_P}(F,D\fil{X})&\to \Hom_{\Top_P}(\Colim F,\fil{X})\\
\{\widehat{f}_{\Delta^{\varphi}}\}&\mapsto \{\widehat{f}_{(\Delta^{\varphi},\Delta^{\psi}})\}
\end{align*}
où on définit $\widehat{f}_{(\Delta^{\varphi},\Delta^{\psi})}$ comme la composition
\begin{equation*}
F(\Delta^{\varphi})\otimes\RealP{\Delta^{\psi}}\to F(\Delta^{\psi})\otimes\RealP{\Delta^{\psi}}\xrightarrow{\widehat{f_{\Delta^{\psi}}}}\fil{X}
\end{equation*}
où $\widehat{f_{\Delta^{\psi}}}$ est obtenu à partir de $f_{\Delta^{\psi}}$ par l'adjonction entre $\Map(\RealP{\Delta^{\psi}},-)$ et $-\otimes\RealP{\Delta^{\psi}}$, provenant de la structure simpliciale de $\Top_P$. (Voir proposition \ref{PropositionTopPCategorieSimpliciale})
La bijection réciproque est donnée par
\begin{align*}
\Hom_{\Top_P}(\Colim F,\fil{X})&\to \Hom_{\DiagR_P}(F,D\fil{X})\\
\{g_{(\Delta^{\varphi},\Delta^{\psi})}\}&\mapsto \{(g_{(\Delta^{\varphi},\Delta^{\varphi})})^{\#}\}\end{align*}
où $(g_{(\Delta^{\varphi},\Delta^{\varphi})})^{\#}$ est obtenu à partir de $g_{(\Delta^{\varphi},\Delta^{\varphi})}$ par l'adjonction entre $-\otimes\RealP{\Delta^{\varphi}}$ et $\Map(\RealP{\Delta^{\varphi}},-)$.
On en déduit que $\Colim$ est bien un adjoint à gauche de $D$.
\end{proof}

\subsection{Structure de modèle sur la catégorie des diagrammes réduits}
On décrit maintenant une structure de modèle sur la catégorie $\DiagR_P$. Pour ce faire, on utilisera la définition suivante. Soit $K$ un ensemble simplicial, et $\Delta^{\varphi}\in R(P)$ un simplexe non dégénéré. On considère le diagramme réduit 
\begin{align*}
K^{\Delta^{\varphi}}\colon R(P)^{\op}&\to \sS\\
\Delta^{\psi}&\mapsto \left\{\begin{array}{cl}
K &\text{ si $\Delta^{\psi}\subseteq \Delta^{\varphi}$}\\
\emptyset &\text{ sinon}
\end{array}\right.
\end{align*}
où, pour tout $\Delta^{\psi_1}\subseteq \Delta^{\psi_2}\subseteq \Delta^{\varphi}$ on a
\begin{equation*}
K^{\Delta^{\varphi}}(\Delta^{\psi_1}\to\Delta^{\psi_2})=\Id_K
\end{equation*}
On remarque que $(-)^{\Delta^{\varphi}}$ induit un foncteur $(-)^{\Delta^{\varphi}}\colon \sS\to \DiagR_P$. On note aussi la propriété du foncteur $(-)^{\Delta^{\varphi}}$ suivante. 
\begin{lemme}\label{ColimCofibrationsGeneratrices}
Soient $\Delta^{\varphi}\in R(P)$ un simplexe non dégénéré et $K\in \sS$ un ensemble simplicial. Alors, on a un isomorphisme canonique 
\begin{equation*}
\Colim (K^{\Delta^{\varphi}})\simeq K\otimes\RealP{\Delta^{\varphi}}
\end{equation*}
\end{lemme} 

\begin{proof}
On reprend les calculs de la preuve de la proposition \ref{PropositionColim}. On a 
\begin{align*}
K^{\Delta^{\varphi}}\otimes R(P)\colon \mathcal{C}&\to \Top_P\\
(\Delta^{\psi_1},\Delta^{\psi_2})&\mapsto \left\{\begin{array}{cl}
K\otimes \RealP{\Delta^{\psi_2}} &\text{ si $\Delta^{\psi_1}\subseteq \Delta^{\varphi}$}\\
\emptyset &\text{ sinon}
\end{array}\right.
\end{align*}
On en déduit 
\begin{equation*}
\Colim\left(K^{\Delta^{\varphi}}\right)=\colim \left(\left(K^{\Delta^{\varphi}}\right)\otimes R(P)\right)\simeq \colim_{\Delta^{\psi_2}\subseteq \Delta^{\varphi}}K\otimes\RealP{\Delta^{\psi_2}}\simeq K\otimes \RealP{\Delta^{\varphi}}
\end{equation*}
\end{proof}
Nous pouvons maintenant énoncer le résultat suivant, qui est une application de \cite[Théorème 11.6.1]{Hirschhorn}
(voir aussi la proposition \ref{CategorieModeleDiagramme})

\begin{prop}\label{CategorieModeleDiagrammesReduits}
La catégorie $\DiagR_P$, munie des classes de morphismes suivantes, est une catégorie modèle à engendrement cofibrant :
\begin{itemize}
\item un morphisme $f\colon F\to G$ est une fibration si pour tout $\Delta^{\varphi}\in R(P)$, $f_{\Delta^{\varphi}}\colon F(\Delta^{\varphi})\to G(\Delta^{\varphi})$ est une fibration de Kan,
\item un morphisme $f\colon F\to G$ est une équivalence faible si pour tout $\Delta^{\varphi}\in R(P)$, $f_{\Delta^{\varphi}}\colon F(\Delta^{\varphi})\to G(\Delta^{\varphi})$ est une équivalence faible pour la structure de Kan Quillen sur $\sS$.
\end{itemize}
Les cofibrations (triviales) génératrices sont données par les ensembles :
\begin{itemize}
\item $I=\{(\partial(\Delta^n)\to\Delta^n)^{\Delta^{\varphi}}\ |\ n\geq 0, \ \Delta^{\varphi}\in R(P) \}$,
\item $J=\{(\Lambda^n_k\to \Delta^n)^{\Delta^{\varphi}} \ |\ n\geq 0,\ 0\leq k\leq n,\ \Delta^\varphi\in R(P)\}$.
\end{itemize}
\end{prop}

\section{Catégorie modèle des espaces filtrés}
\label{SectionCMFTopP}
L'objet de cette section est de prouver le théorème suivant
\begin{theo}\label{CategorieModeleTopP}
Il existe une structure de modèle sur $\Top_P$ obtenue par transport à partir de la structure de modèle sur $\DiagR_P$. En particulier, les fibrations et les équivalences faibles sont définies comme suit :
\begin{itemize}
\item une application filtrée $f\colon \fil{X}\to\fil{Y}$ est une fibration si et seulement si 
\begin{equation*}
D(f)\colon D(\fil{X})\to D(\fil{Y})
\end{equation*}
est une fibration de $\DiagR_P$,
\item une application filtrée $f\colon \fil{X}\to\fil{Y}$ est une équivalence faible si et seulement si 
\begin{equation*}
D(f)\colon D(\fil{X})\to D(\fil{Y})
\end{equation*}
est une équivalence faible de $\DiagR_P$.
\end{itemize}
De plus, cette structure de modèle est engendrée de façon cofibrante par les ensembles de cofibrations (triviales) génératrices suivants 
\begin{itemize}
\item $I=\{\partial(\Delta^n)\otimes\RealP{\Delta^{\varphi}}\to\Delta^n\otimes\RealP{\Delta^{\varphi}}\ |\ n\geq 0,\ \Delta^{\varphi}\in R(P)\}$
\item $J=\{\Lambda^n_k\otimes \RealP{\Delta^{\varphi}}\to\Delta^n\otimes \RealP{\Delta^{\varphi}}\ |\ n\geq 0,\ 0\leq k\leq n,\ \Delta^{\varphi}\in R(P)\}$
\end{itemize}
\end{theo}
En particulier, dans cette structure de modèle, les équivalences faibles sont exactement les morphismes induisant des isomorphismes sur tous les groupes d'homotopie filtrés.

\begin{proof}
Nous allons prouver le théorème \ref{CategorieModeleTopP} en appliquant \cite[Corollary 3.3.4]{Hess}. Par le lemme \ref{AccessibleLocalementPresentable} les hypothèses sur $\Top_P$ et $\DiagR_P$ sont vérifiées. Il suffit donc de montrer que tout morphisme admettant la propriété de relèvement à droite par rapport aux fibrations définies dans le théorème \ref{CategorieModeleTopP} est aussi une équivalence faible. Soit $f\colon \fil{X}\to\fil{Y}$ un tel morphisme. Par le lemme \ref{LemmeFactorisationFibration}, $f$ peut être factorisé sous la forme $f=q\circ i$, avec $i$ une équivalence d'homotopie filtrée, et $q$ une fibration. Considérons maintenant le problème de relèvement suivant :
\begin{equation*}
\begin{tikzcd}
\fil{X}
\arrow{r}{i}
\arrow[swap]{d}{f}
& \fil{Z}
\arrow{d}{q}
\\
\fil{Y}
\arrow[swap]{r}{\Id_Y}
\arrow[dashrightarrow]{ur}{h}
&\fil{Y}
\end{tikzcd}
\end{equation*}
Par hypothèse, $f$ admet la propriété de relèvement à gauche par rapport aux fibrations, et $q$ est une fibration. Il existe donc un relèvement $h$ faisant commuter le diagramme. 
On fixe un pointage de $\fil{X}$, $\phi\colon \RealP{V}\to\fil{X}$, et on calcule les groupes d'homotopie filtrés. On a le diagramme commutatif suivant :
\begin{equation*}
\begin{tikzcd}
s\pi_n(\fil{X},\phi)
\arrow{r}{s\pi_n(i)}
\arrow[swap]{d}{s\pi_n(f)}
& s\pi_n(\fil{Z},i\circ\phi)
\arrow{d}{s\pi_n(q)}
\\
s\pi_n(\fil{Y},f\circ\phi)
\arrow[swap]{r}{\Id}
\arrow{ur}{s\pi_n(h)}
&s\pi_n(\fil{Y},f\circ\phi)
\end{tikzcd}
\end{equation*}
Par construction, les deux morphismes horizontaux sont des isomorphismes. On en déduit que $s\pi_n(h)$ admet un inverse à gauche ($s\pi_n(f)\circ s\pi_n(i)^{-1}$), et à droite ($s\pi_n(q)$), c'est donc un isomorphisme. Finalement, par deux sur trois, $s\pi_n(f)$ est un isomorphisme. On en déduit que $f$ est une équivalence faible.
La génération cofibrante est maintenant une conséquence de la définition des classes de fibrations et de fibrations triviales, et du lemme \ref{ColimCofibrationsGeneratrices}.
\end{proof}

\begin{lemme}\label{AccessibleLocalementPresentable}
La catégorie $\DiagR_P$ est une catégorie modèle accessible, et la catégorie $\Top_P$ est localement présentable.
\end{lemme}

\begin{proof}
La catégorie $\DiagR_P$ est une catégorie modèle à engendrement cofibrant (Proposition \ref{CategorieModeleDiagrammesReduits}), et elle est localement présentable (c'est une catégorie de foncteurs depuis une petite catégorie vers une la catégorie localement présentable $\sS$) elle est donc accessible (voir l'introduction de \cite{Hess}). On travaille avec la catégorie $\Top$ des espaces topologiques $\Delta$ engendrés (voir la remarque \ref{RemarqueCategorieDeltaEngendre}). Cette catégorie est localement présentable \cite{Dugger}. Comme $\Top_P=\Top/P$ est la catégorie des espaces au dessus de $P$, cette dernière est aussi localement présentable.
\end{proof}

\begin{lemme}\label{LemmeFactorisationFibration}
Soit $f\colon\fil{X}\to\fil{Y}$ une application filtrée. Il existe un espace filtré $\fil{Z}$, une fibration $q\colon \fil{Z}\to\fil{Y}$ et l'inclusion d'un rétracte par déformation filtré $i\colon \fil{X}\to\fil{Z}$ tels que $f=q\circ i$.
\end{lemme}

\begin{proof}
Soit $f\colon \fil{X}\to \fil{Y}$ une application filtrée. On définit l'espace topologique $Z$ comme suit :
\begin{equation*}
Z=\{(x,\gamma)\in X\times Y^{[0,1]}\ |\ \gamma(0)=f(x),\ \varphi_Y(\gamma(t))=\varphi_X(x)\ \forall t\}\subseteq X\times Y^{[0,1]}
\end{equation*}
La filtration $\varphi_Z\colon Z\to P$ est induite par l'inclusion $Z\subseteq \fil{X}\times_P\fil{Y}^{\Delta^1}$ (voir La définition \ref{DefinitionEspacesDeCheminsStratifies} pour la définition de $Y^{\Delta^1}$). Plus explicitement, $\varphi_Z$ est donnée par la composition
\begin{equation*}
Z\to X\times Y^{[0,1]}\xrightarrow{\pr_X} X\xrightarrow{\varphi_X}P
\end{equation*}
On définit les applications filtrées suivantes 
\begin{align*}
i\colon \fil{X}&\to\fil{Z}\\
x &\mapsto (x,(t\mapsto f(x)))
\end{align*}
\begin{align*}
q\colon \fil{Z}&\to\fil{Y}\\
(x,\gamma)&\mapsto \gamma(1)
\end{align*}
et
\begin{align*}
r\colon \fil{Z}&\to\fil{X}\\
(x,\gamma)&\mapsto x
\end{align*}
Par construction, $f=q\circ i$ et $r\circ i=\Id_X$. De plus, l'application filtrée
\begin{align*}
[0,1]\otimes\fil{Z}&\to\fil{Z}\\
(s,(x,\gamma))&\mapsto (x,(t\mapsto \gamma(st)))
\end{align*}
fournit une homotopie entre $i\circ r$ et $\Id_Z$.
Il suffit donc de vérifier que $q$ est une fibration. Cela revient à vérifier que pour toute inclusion de cornet $\Lambda^n_k\to \Delta^n$, tout simplexe non dégénéré $\Delta^{\varphi}\in R(P)$, et tout problème de relèvement de la forme
\begin{equation*}
\begin{tikzcd}
\Real{\Lambda^n_k}\otimes\RealP{\Delta^{\varphi}}
\arrow{r}{g}
\arrow{d}{j}
&\fil{Z}
\arrow{d}{q}
\\
\Real{\Delta^n}\otimes\RealP{\Delta^{\varphi}}
\arrow[swap]{r}{G}
\arrow[dashrightarrow]{ur}{\widetilde{G}}
&\fil{Y}
\end{tikzcd}
\end{equation*}
il existe une application filtrée $\widetilde{G}$ faisant commuter le diagramme. On remarque que l'inclusion $j\colon\Real{\Lambda^n_k}\to \Real{\Delta^n}$ admet un rétracte $s\colon\Real{\Delta^n}\to \Real{\Lambda^n_k}$ tel qu'il existe une homotopie $H\colon \Real{\Delta^n}\times [0,1]\to \Real{\Delta^n}$ entre $j\circ s$ et $\Id_{\Real{\Delta^n}}$ fixant $\Real{\Lambda^n_k}$. On considère de plus une fonction continue $h\colon \Real{\Delta^n}\to \R_{+}$ telle que $h^{-1}(0)=\Real{\Lambda^n_k}$ (la fonction distance à $\Real{\Lambda^n_k}$ convient). Toutes ces applications induisent des applications filtrées, 
\begin{equation*}
j\colon \Real{\Lambda^n_k}\otimes\RealP{\Delta^{\varphi}}\to
\Real{\Delta^n}\otimes\RealP{\Delta^{\varphi}},
\end{equation*}
\begin{equation*}
s\colon \Real{\Delta^n}\otimes\RealP{\Delta^{\varphi}}\to 
\Real{\Lambda^n_k}\otimes\RealP{\Delta^{\varphi}},
\end{equation*}
\begin{equation*}
H\colon [0,1]\otimes(\Real{\Delta^n}\otimes\RealP{\Delta^{\varphi}})\to \Real{\Delta^n}\otimes\RealP{\Delta^{\varphi}}
\end{equation*}
et
\begin{equation*}
h\colon \Real{\Delta^n}\otimes\RealP{\Delta^{\varphi}}\to \R_{+}.
\end{equation*}
Et on a $s\circ j=\Id$, $H$ est une homotopie filtrée entre $j\circ s$ et $\Id$ fixant $\Real{\Lambda^n_k}\otimes\RealP{\Delta^{\varphi}}$, et $h^{-1}(0)=\Real{\Lambda^n_k}\otimes\RealP{\Delta^{\varphi}}$. Soit $a\in \Real{\Delta^n}\otimes\RealP{\Delta^{\varphi}}$, on note $g(a)=(g_X(a),g_Y(a))\in Z\subset X\times Y^{[0,1]}$. Avec ces notations, on définit explicitement le relèvement $\widetilde{G}=(\widetilde{G}_X,\widetilde{G}_Y)$.
\begin{align*}
\widetilde{G}_X\colon \Real{\Delta^n}\otimes\RealP{\Delta^{\varphi}}&\to \fil{X}\\
a&\mapsto g_X(r(a))
\end{align*}
\begin{align*}
\widetilde{G}_Y\colon \Real{\Delta^n}\otimes\RealP{\Delta^{\varphi}}&\to \fil{Y}^{[0,1]}\\
a&\mapsto \left\{\begin{array}{ccc}
[0,1]&\to& Y\\
t&\mapsto& \left\{\begin{array}{cl}
g_Y(r(a))(t(1+h(a)) & 0\leq t\leq \frac{1}{1+h(a)}\\
G\left(H\left(a,\frac{1+h(a)}{h(a)}(t-\frac{1}{1+h(a)})\right)\right) & \frac{1}{1+h(a)}<t\leq 1
\end{array}\right.
\end{array}\right.
\end{align*}
Vérifions que $\widetilde{G}_Y$ est bien définie. Soit $a\in \Real{\Delta^n}\otimes\RealP{\Delta^{\varphi}}$. Si $a\in \Real{\Lambda^n_k}\otimes\RealP{\Delta^{\varphi}}$, alors $h(a)=0$, et $\widetilde{G}_Y(a)=g_Y(r(a)$ donc $\widetilde{G}_Y(a)$ est continue. (Ceci montre aussi que le triangle supérieur commute) Sinon, pour $t=\frac{1}{1+h(a)}<1$, on a
\begin{align*}
\widetilde{G}_Y(a)(t)&=g_Y(r(a))(1)\\
&=q\circ g(r(a))\\
&=G\circ j(r(a))\\
&= G(H(a,0))
\end{align*}
On en déduit que $\widetilde{G}_Y(a)$ est continue pour tout $a\in \Real{\Delta^n}\otimes\RealP{\Delta^{\varphi}}$, et donc $\widetilde{G}_Y$ est bien définie.
Vérifions que $\widetilde{G}$ est bien définie. Soit $a\in \Real{\Delta^n}\otimes\RealP{\Delta^{\varphi}}$. On a 
\begin{equation*}
\widetilde{G}_Y(a)(0)=g_Y(r(a))(0)=f(g_X(r(a))=f(\widetilde{G}_X(a))
\end{equation*} 
Donc, $\widetilde{G}$ est bien définie. Comme $H$ fixe $\Real{\Lambda^n_k}\otimes\RealP{\Delta^{\varphi}}$, $\widetilde{G}$ est continue. Par construction, $\widetilde{G}$ fait commuter le diagramme, et donc est une solution au problème de relèvement. On en déduit que $q$ est une fibration.
\end{proof}

\section{Espaces fortement filtrés}
\label{SectionTopNP}

La notion de pointage d'un espace filtré utilisé jusqu'ici peut paraitre surprenante. En effet, le plus souvent, un pointage d'un objet $X$ dans une catégorie $\mathcal{C}$ est la donnée d'un morphisme $*\to X$ où $*$ est l'objet terminal de $\mathcal{C}$. Une généralisation immédiate est de définir des pointages comme des morphismes $V\to X$ où $V\subseteq *$ est un sous-objet de l'objet terminal. Cependant, la définition de pointage que nous considérons ici (Définition \ref{DefinitionPointageEspaceFiltre}) n'est pas de ce type là. En effet, on définit les pointages comme des morphismes depuis des sous objets de $\RealP{N(P)}$ qui n'est pas l'objet terminal de $\Top_P$, et non pas comme des morphismes depuis (les sous objets de) $P$. La raison derrière cette définition vient du fait que presque toujours, il n'existe aucun morphisme filtré $\phi\colon P\to \fil{X}$
lorsque $\fil{X}$ est une pseudo variété (en fait, dès que $X$ est séparable). En effet, dès que $P$ contient au moins deux éléments comparables $p_0<p_1\in P$, tout voisinage de $\phi(p_0)$ devrait contenir $\phi(p_1)$. Et si cette application est filtrée, on a nécessairement que $\phi(p_0)\not=\phi(p_1)$ ce qui implique que $X$ ne peut pas être séparé. Pour cette raison, et pour simplifier les preuves à venir, on considère maintenant la catégorie des espaces fortement filtrés, pour laquelle la notion de pointage redevient naturelle.

\subsection{La catégorie modèle des espaces fortement filtrés}
\label{SectionCMFTopNP}
\begin{defin}
Un espace fortement filtré au dessus de $P$ est la donnée
\begin{itemize}
\item d'un espace topologique $X$,
\item d'une application continue $\varphi_X\colon X\to \Real{N(P)}$.
\end{itemize}
Une application fortement filtrée $f\colon \fil{X}\to\fil{Y}$ est la donnée d'une application continue $f\colon X\to Y$ telle que le triangle suivant commute
\begin{equation*}
\begin{tikzcd}
X
\arrow{rr}{f}
\arrow[swap]{dr}{\varphi_X}
&&
Y
\arrow{dl}{\varphi_Y}
\\
&\Real{N(P)}
\end{tikzcd}
\end{equation*}
On note $\Top_{N(P)}$ la catégorie des espaces fortement filtrés au dessus de $P$.
\end{defin}

On rappelle que $\varphi_P$ désigne l'application continue $\varphi_P\colon \Real{N(P)}\to P$ de la définition \ref{DefinitionVarphiP}
\begin{prop}
Les foncteurs
\begin{align*}
\varphi_P\circ -\colon \Top_{N(P)}&\to \Top_P\\
\fil{X}&\mapsto (X,\varphi_P\circ\varphi_X)
\end{align*}
et
\begin{align*}
-\times_{P}\Real{N(P)}\colon \Top_P &\to\Top_{N(P)}\\
\fil{Y}&\mapsto (Y\times_{P}\Real{N(P)},\pr_{\Real{N(P)}})
\end{align*}
forment une paire de foncteurs adjoints, où $\varphi_P\circ -$ est l'adjoint à gauche.
\end{prop}

\begin{proof}
C'est une conséquence de la propriété universelle du produit fibré.
\end{proof}

\begin{defin}
On définit le foncteur de réalisation fortement filtrée
\begin{align*}
\Real{-}_{N(P)}\colon \sS_P&\to \Top_{N(P)}\\
\fil{X}&\mapsto (\Real{X},\Real{\varphi_X}).
\end{align*}
Il admet un adjoint à droite 
\begin{align*}
\Sing_{N(P)}\colon \Top_{N(P)}&\to \sS_P\\
\fil{X}&\mapsto \left\{\begin{array}{ccc}
\Delta(P)^{\op}&\to &\Set\\
\Delta^{\varphi}&\mapsto &\Hom_{\Top_{N(P)}}(\Real{\Delta^{\varphi}}_{N(P)},\fil{X})
\end{array}\right.
\end{align*}
\end{defin}

\begin{prop}
On a les isomorphismes de foncteurs
\begin{equation*}
\RealP{-}\simeq (\varphi_P\circ-)\circ\Real{-}_{N(P)}
\end{equation*}
et
\begin{equation*}
\Sing_P\simeq \Sing_{N(P)}\circ (-\times_{P}\Real{N(P)}).
\end{equation*}
\end{prop}

\begin{proof}
Le premier isomorphisme provient de la définition de $\RealP{-}$, voir la définition \ref{DefinitionRealP}. Pour le second isomorphisme, soient $\Delta^{\varphi}\in \Delta(P)$ et $\fil{X}\in \Top_P$. On calcule
\begin{align*}
\Sing_P\fil{X}(\Delta^{\varphi})&=\Hom(\RealP{\Delta^{\varphi}},\fil{X})\\
&\simeq \Hom((\varphi\circ-)\Real{\Delta^{\varphi}}_{N(P)},\fil{X})\\
&\simeq \Hom(\Real{\Delta^{\varphi}}_{N(P)},\fil{X}\times_{P}\Real{N(P)})\\
&\simeq \Sing_{N(P)}\circ (-\times_P\Real{N(P)})(\fil{X})(\Delta^{\varphi})
\end{align*}
Et tous les isomorphismes proviennent d'adjonctions. Ils sont donc naturels en $\fil{X}$ et en $\Delta^{\varphi}$.
\end{proof}

\begin{defin}
Comme pour les espaces filtrés, on définit le foncteur diagramme réduit 
\begin{align*}
D_{N(P)}\colon \Top_{N(P)}&\to\DiagR_P\\
\fil{X}&\mapsto \left\{\begin{array}{ccc}
R(P)^{\op}&\to& \sS\\
\Delta^{\varphi}&\mapsto & \Map(\Delta^{\varphi},\Sing_{N(P)}(\fil{X}))
\end{array}\right.
\end{align*}
\end{defin}
Par la proposition précédente, ainsi que la proposition \ref{DiagrammesTopologiquesIsomorphes}, on a immédiatement le résultat suivant.
\begin{prop}\label{PropositionDiagrammesTopNP}
Les foncteurs $D\colon \Top_P\to \DiagR_P$ et $D_{N(P)}\circ (-\times_P\RealP{(N(P)})\colon \Top_P\to \DiagR_P$ sont isomorphes.
\end{prop}
Pour cette raison, on notera aussi $D$ pour le foncteur $D_{N(P)}\colon \Top_{N(P)}\to\DiagR_P$.

La preuve du théorème \ref{CategorieModeleTopP} s'adapte directement, pour donner une preuve du résultat suivant :

\begin{theo}\label{CategorieModeleTopNP}
Il existe une structure de modèle sur $\Top_{N(P)}$ obtenue par transport à partir de la structure de modèle sur $\DiagR_P$. En particulier, les fibrations et les équivalences faibles sont définies comme suit :
\begin{itemize}
\item une application fortement filtrée $f\colon \fil{X}\to\fil{Y}$ est une fibration si et seulement si 
\begin{equation*}
D(f)\colon D(\fil{X})\to D(\fil{Y})
\end{equation*}
est une fibration de $\DiagR_P$,
\item une application fortement filtrée $f\colon \fil{X}\to\fil{Y}$ est une équivalence faible si et seulement si 
\begin{equation*}
D(f)\colon D(\fil{X})\to D(\fil{Y})
\end{equation*}
est une équivalence faible de $\DiagR_P$.
\end{itemize}
De plus, cette structure de modèle est engendrée de façon cofibrante par les ensembles de cofibrations (triviales) génératrices suivants 
\begin{itemize}
\item $I=\{\partial(\Delta^n)\otimes\RealNP{\Delta^{\varphi}}\to\Delta^n\otimes\RealNP{\Delta^{\varphi}}\ |\ n\geq 0,\ \Delta^{\varphi}\in R(P)\}$
\item $J=\{\Lambda^n_k\otimes \RealNP{\Delta^{\varphi}}\to\Delta^n\otimes \RealNP{\Delta^{\varphi}}\ |\ n\geq 0,\ 0\leq k\leq n,\ \Delta^{\varphi}\in R(P)\}$
\end{itemize}
\end{theo}

\subsection{Equivalence de Quillen avec la catégorie des espaces filtrés}
\label{SectionQETopNPTopP}
\begin{theo}\label{EquivalenceQuillenTopPTopNP}
L'adjonction $(\varphi_P\circ -, -\times_P \RealP{N(P)})$ induit une équivalence de Quillen entre les catégories de modèle $\Top_{N(P)}$ et $\Top_P$.
\end{theo}

La preuve repose sur l'observation suivante

\begin{lemme}\label{LemmeVarphiPFibrationTriviale}
L'application filtrée $\varphi_P\colon \RealP{N(P)}\to P$ est une fibration triviale de $\Top_P$.
\end{lemme}

\begin{proof}
Par construction de la structure de modèle sur $\Top_P$, il suffit de montrer que pour tout $\Delta^{\varphi}\in R(P)$, le morphisme induit par $\varphi_P$
\begin{equation*}
\Map(\RealP{\Delta^{\varphi}},\RealP{N(P)})\to\Map(\RealP{\Delta^{\varphi}},P)
\end{equation*}
est une fibration de Kan triviale. On observe immédiatement que 
\begin{equation*}
\Map(\RealP{\Delta^{\varphi}},P)\simeq \{*\}.
\end{equation*}
Aussi, comme $\Map(\RealP{\Delta^{\varphi}},\RealP{N(P)})\simeq \Sing(C^0_P(\RealP{\Delta^{\varphi}},\RealP{N(P)})$ est un complexe de Kan, le morphisme induit par $\varphi_P$ est nécessairement une fibration de Kan. Montrons que c'est aussi une équivalence faible. Notons $\Delta^{\varphi}=[p_0,\dots,p_n]\subseteq N(P)$. Alors, par définition de $\varphi_P$, tout morphisme filtré $\Delta^n\otimes\RealP{\Delta^{\varphi}}\to \RealP{N(P)}$ se factorise par $\varphi_P^{-1}(\{p_0,\dots,p_n\})\subseteq \RealP{N(P)}$. En particulier, l'inclusion induit un isomorphisme :
\begin{equation*}
\Map(\RealP{\Delta^{\varphi}},\varphi_P^{-1}(\{p_0,\dots,p_n\}))\simeq \Map(\RealP{\Delta^{\varphi}},\RealP{N(P)}.
\end{equation*}
D'autre part, on observe que tout point de $\varphi_P^{-1}(\{p_0,\dots,p_n\})$ est de la forme $((t_0,\dots,t_m),\Delta^{\psi})$ avec $\psi\colon \Delta^m\to N(P)$, et pour $i=\sup\{j\ |\ t_j\not =0\}$, $\psi(e_i)\in \{p_0,\dots,p_n\}$. (Voir la Définition \ref{DefinitionVarphiP} pour la définition de $\varphi_P$). Soit $\psi\colon \Delta^m\to N(P)$ un simplexe de $N(P)$, on note $I=\{i\ |\ i\in \{0,\dots,m\},\  \psi(e_i)\in \{p_0,\dots,p_n\}\}$. On définit l'homotopie suivante
\begin{align*}
H^{\psi}\colon\RealP{\Delta^{\psi}}\cap \varphi_P^{-1}(\{p_0,\dots,p_n\})\times [0,1]&\to \RealP{\Delta^{\psi}}\cap \varphi_P^{-1}(\{p_0,\dots,p_n\})\\
((t_0,\dots,t_m),s)&\mapsto (H_0^{\psi}(t_0,\dots,t_m,s),\dots,H_m^{\psi}(t_0,\dots,t_m,s))
\end{align*}
où $H_i^{\psi}$ est définie comme
\begin{align*}
H_i^{\psi}\colon \RealP{\Delta^{\psi}}\cap \varphi_P^{-1}(\{p_0,\dots,p_n\})\times [0,1]&\to [0,1]\\
((t_0,\dots,t_m),s)&\mapsto \left\{\begin{array}{cl}
(1-s)t_i &\text{ si $i\not\in I$}\\
t_i\left(1+s\frac{\sum_{j\not\in I}t_j}{\sum_{j\in I}t_j}\right) &\text{ si $i\in I$}
\end{array}\right.
\end{align*}
Les applications $H_i^{\psi}$ sont bien définies car, si $(t_0,\dots,t_n)\in \RealP{\Delta^{\psi}}\cap \varphi_P^{-1}(\{p_0,\dots,p_n\})$, alors $\sum_{i\in I}t_i\not = 0$. On vérifie que $H^{\psi}$ est bien définie. Soient $(t_0,\dots,t_m)\in \RealP{\Delta^{\psi}}\cap \varphi_P^{-1}(\{p_0,\dots,p_n\})$ et $s\in [0,1]$. 
\begin{align*}
\sum_i H^{\psi}_i(t_0,\dots,t_n,s)&=\sum_{i\not\in I}(1-s)t_i +\sum_{i\in I}t_i\left(1+s\frac{\sum_{j\not\in I}t_j}{\sum_{j\in I}t_j}\right)\\
&= \sum_{i} t_i-s\sum_{i\not\in I}t_i+s(\sum_{i\in I}t_i)\frac{\sum_{i\not\in I}t_i}{\sum_{i\in I}t_i}\\
&= 1-s\sum_{i\not\in I}t_i+s(\sum_{i\not\in I}t_i)\\
&= 1
\end{align*}
De plus, par construction, pour tout $(t_0,\dots,t_n)\in \RealP{\Delta^{\psi}}\cap \varphi_P^{-1}(\{p_0,\dots,p_n\})$, si $i=\sup\{j\ |\ t_j\not = 0\}$, alors $\psi(e_i)\in \{p_0,\dots,p_n\}$ et donc $i\in I$. Ainsi, pour tout $s\in [0,1]$, 
\begin{equation*}
\varphi_P(H^{\psi}(t_0,\dots,t_m),s))=\varphi_P(t_0,\dots,t_m).
\end{equation*}
On en déduit que $H^{\psi}$ est une application filtrée. Finalement, en considérant les restrictions en $0$ et $1$, on a que $H^{\psi}$ est une homotopie filtrée entre l'identité de $\Delta^{\psi}\cap \varphi_P^{-1}(\{p_0,\dots,p_n\})$ et une rétraction 
\begin{equation*}
\RealP{\Delta^{\psi}}\cap \varphi_P^{-1}(\{p_0,\dots,p_n\})\to \RealP{\Delta^{\psi}}\cap\RealP{\Delta^{\varphi}}\subseteq \RealP{\Delta^{\psi}}\cap \varphi_P^{-1}(\{p_0,\dots,p_n\})
\end{equation*}
Finalement, on obtient un ensemble d'homotopies filtrées $H^{\psi}$ pour tous les simplexes de $N(P)$, dont on vérifie facilement qu'elles sont compatibles aux faces. On obtient ainsi une homotopie filtrée
\begin{equation*}
H\colon \varphi_P^{-1}(\{p_0,\dots,p_n\})\times[0,1]\to
\varphi_P^{-1}(\{p_0,\dots,p_n\})
\end{equation*}
entre l'identité, et une rétraction
\begin{equation*}
\varphi_P^{-1}(\{p_0,\dots,p_n\})\to \RealP{\Delta^{\varphi}}\subseteq \varphi_P^{-1}(\{p_0,\dots,p_n\})
\end{equation*}
On en déduit que l'inclusion $\RealP{\Delta^{\varphi}}\to \varphi_P^{-1}(\{p_0,\dots,p_n\})$ est une équivalence d'homotopie filtrée. Elle induit donc une équivalence faible
\begin{equation*}
\Map(\RealP{\Delta^{\varphi}},\RealP{\Delta^{\varphi}})\to\Map(\RealP{\Delta^{\varphi}},\RealP{N(P)}).
\end{equation*}
Montrons maintenant que $\Map(\RealP{\Delta^{\varphi}},\RealP{\Delta^{\varphi}})$ est contractile. Soit $\sigma\colon \Real{\Delta^m}\otimes\RealP{\Delta^{\varphi}}\to\RealP{\Delta^{\varphi}}$ un simplexe de $\Map(\RealP{\Delta^{\varphi}},\RealP{\Delta^{\varphi}})$. On définit l'application
\begin{align*}
G(\sigma)\colon \Delta^1\otimes(\Real{\Delta^m}\otimes\RealP{\Delta^{\varphi}})&\to\RealP{\Delta^{\varphi}}\\
(s,u,t)&\mapsto \sigma(u,t)s+t(1-s)
\end{align*}
On peut maintenant définir l'homotopie contractante 
\begin{align*}
\Delta^1\times\Map(\RealP{\Delta^{\varphi}},\RealP{\Delta^{\varphi}})&\to\Map(\RealP{\Delta^{\varphi}},\RealP{\Delta^{\varphi}})\\
(\tau\colon \Delta^n\to \Delta^1,\sigma)&\mapsto G(\sigma)\circ\left(\Real{\tau}\times\Id_{\Real{\Delta^n}}\times\Id_{\RealP{\Delta^{\varphi}}}\right)\circ(\delta_{\Real{\Delta^n}}\times\Id_{\RealP{\Delta^{\varphi}}})
\end{align*}
où $\delta_{\Real{\Delta^n}}\colon \Real{\Delta^n}\to\Real{\Delta^n}\times\Real{\Delta^n}$ est l'application diagonale. On conclut que $\Map(\RealP{\Delta^{\varphi}},\RealP{\Delta^{\varphi}})$ est contractile, ce qui implique que $\varphi_P\colon \RealP{N(P)}\to P$ est une fibration triviale de $\Top_P$.
\end{proof}

On peut maintenant prouver le théorème \ref{EquivalenceQuillenTopPTopNP}

\begin{proof}
Soit $f\colon \fil{X}\to\fil{Y}$ un morphisme de $\Top_P$. C'est une fibration (triviale) si et seulement si 
\begin{equation*}
D(f)\colon D(\fil{X})\to D(\fil{Y})
\end{equation*}
est une fibration (triviale) de $\DiagR_P$. Par la proposition \ref{PropositionDiagrammesTopNP}, cette dernière condition est vérifiée si et seulement si
\begin{equation*}
D(f\times_P\Real{N(P)})\colon D(\fil{X}\times_P\Real{N(P)})\to D(\fil{Y}\times_P\Real{N(P)})
\end{equation*}
c'est à dire si et seulement si $f\times_P\Real{N(P)}$ est une fibration (triviale) de $\Top_{N(P)}$. En particulier, on en déduit que $-\times_P\Real{N(P)}$ préserve les fibrations et les fibrations triviales. C'est donc un foncteur de Quillen à droite. 

Soient $\fil{X}\in \Top_{N(P)}$ un espace fortement filtré, $\fil{Y}\in \Top_P$ un espace filtré et $f\colon \fil{X}\to\fil{Y}\times_P\Real{N(P)}$ une application fortement filtrée. Montrons que $f$ est une équivalence faible de $\Top_{N(P)}$ si et seulement si son image par l'adjonction
\begin{equation*}
\widehat{f}\colon(X,\varphi_P\circ\varphi_X)\to \fil{Y}
\end{equation*}
est une équivalence faible de $\Top_P$. On remarque que $f$ peut se factoriser sous la forme
\begin{equation*}
\fil{X}\xrightarrow{\epsilon_X}\fil{X}\times_P\Real{N(P)}\xrightarrow{\widehat{f}\times_P\Real{N(P)}}Y\times_P\Real{N(P)}
\end{equation*}
où $\epsilon_X$ est l'unité de l'adjonction $(\varphi_P\circ-,-\times_P\Real{N(P)})$. Mais par définition $\widehat{f}$ est une équivalence faible de $\Top_P$ si et seulement si $\widehat{f}\times_P\Real{N(P)}$ est une équivalence faible de $\Top_{N(P)}$. Il suffit donc de prouver que $\epsilon_X$ est une équivalence faible pour tout espace fortement filtré $\fil{X}$. On considère le diagramme commutatif suivant
\begin{equation*}
\begin{tikzcd}[column sep = huge]
X
\arrow[bend left= 18]{drr}{\Id_X}
\arrow{dr}{\epsilon_X}
\arrow[swap, bend right = 18]{dddr}{\varphi_X}
\\
&X\times_P\Real{N(P)}
\arrow{r}{\pr_X}
\arrow{d}{\varphi_X\times_P\Real{N(P)}}
&X
\arrow{d}{\varphi_X}
\\
&\Real{N(P)}\times_P\Real{N(P)}
\arrow[swap]{r}{\varphi_P\times_P\Real{N(P)}}
\arrow{d}
&\Real{N(P)}
\arrow{d}{\varphi_P}
\\
&\Real{N(P)}
\arrow[swap]{r}{\varphi_P}
&P
\end{tikzcd}
\end{equation*}
Dans ce diagramme, les trois carrés sont cartésiens. De plus, par le lemme \ref{LemmeVarphiPFibrationTriviale}, $\varphi_P$ est une fibration triviale de $\Top_P$. Ceci implique que $\varphi_P\times_P\Real{N(P)}$ est une fibration triviale de $\Top_N(P)$, car $-\times_P\Real{N(P)}$ préserve les fibrations triviales. Finalement, $\pr_X\colon (X,\varphi_P\circ\varphi_X)\times_P\Real{N(P)}\to \fil{X}$ est une fibration triviale de $\Top_{N(P)}$ car c'est l'image d'une fibration triviale par un produit fibré. Comme $\Id_X=\pr_X\circ\epsilon_X$, on en déduit que $\epsilon_X$ est une équivalence faible de $\Top_{N(P)}$ par deux sur trois.
\end{proof}
%

\subsection{Equivalence de Quillen avec la catégorie des diagrammes simpliciaux réduits}
\label{SectionQETopNPDiagRP}
L'objet de cette sous section est de montrer le résultat suivant.

\begin{theo}\label{EquivalenceQuillenTopNPDiagRP}
L'adjonction $\Colim\colon \DiagR_P\leftrightarrow \Top_{N(P)}\colon D$ induit une équivalence de Quillen.
\end{theo}

On utilisera les lemmes suivants.

%
%
%

\begin{lemme}\label{LemmeDiagrammesCofibrants}
Soit $F\colon R(P)^{\op}\to \sS$ un diagramme simplicial réduit cofibrant. Alors, pour tout $\Delta^{\psi}\to\Delta^{\varphi}\in R(P)$, le morphisme 
\begin{equation*}
F(\Delta^{\psi}\to\Delta^{\varphi})\colon F(\Delta^{\varphi})\to F(\Delta^{\psi})
\end{equation*}
est un monomorphisme.
\end{lemme}

\begin{proof}
Soient $F$ un objet cofibrant de $\DiagR_P$ et $\Delta^{\psi}\to\Delta^{\varphi}$ un morphisme de $R(P)$. Montrons que $f\colon F(\Delta^{\varphi})\to F(\Delta^{\psi})$ est un monomorphisme. Par définition de la structure de Kan Quillen sur $\sS$, cela revient à montrer que $f$ admet la propriété de relèvement à gauche par rapport à toutes les fibrations de Kan triviales.
Soit $p\colon X\to Y$ une fibration de Kan triviale, on considère le problème de relèvement suivant.
\begin{equation}\label{ProblemeCofibrationsMonomorphismes}
\begin{tikzcd}
F(\Delta^{\varphi})
\arrow{r}{\alpha}
\arrow[swap]{d}{f}
&X
\arrow{d}{p}
\\
F(\Delta^{\psi})
\arrow[swap]{r}{\beta}
\arrow[dashrightarrow]{ur}{h}
&Y
\end{tikzcd}
\end{equation}
Pour construire le relèvement $h$, on considère d'abord les diagrammes simpliciaux réduits suivants.
\begin{align*}
H\colon R(P)^{\op}&\to \sS\\
\Delta^{\mu}&\mapsto \left\{\begin{array}{cl}
X &\text{ si $\Delta^{\varphi}\subseteq \Delta^{\mu}$}\\
Y &\text{ si $\Delta^{\psi}\subseteq\Delta^{\mu}$ et $\Delta^{\varphi}\not\subseteq \Delta^{\mu}$}\\
* &\text{ si $\Delta^{\psi}\not\subseteq \Delta^{\mu}$}
\end{array}\right.
\end{align*}
Et, pour $\Delta^{\mu_1}\to \Delta^{\mu_2}$, \begin{equation*}
H(\Delta^{\mu_2})\to H(\Delta^{\mu_1})=\left\{\begin{array}{cl}
\Id_X & \text{si $\Delta^{\varphi}\subseteq \Delta^{\mu_1}$}\\
\Id_Y &\text{ si $\Delta^{\varphi}\not\subseteq \Delta^{\mu_2}$ et $\Delta^{\psi}\subseteq \Delta^{\mu_1}$}\\
p & \text{ si $\Delta^{\varphi}\subseteq \Delta^{\mu_2}$, $\Delta^{\varphi}\not \subseteq \Delta^{\mu_1}$ et $\Delta^{\psi}\subseteq\Delta^{\mu_1}$}\\
* &\text{ sinon}
\end{array}\right. .
\end{equation*} 
Où $*$ est l'unique application vers l'ensemble simplicial terminal $*$. Le diagramme est bien défini car si $\Delta^{\mu_1}\subseteq \Delta^{\mu_2}$, alors $\Delta^{\varphi}\subseteq \Delta^{\mu_1}\Rightarrow \Delta^{\varphi}\subseteq \Delta^{\mu_2}$. 
\begin{align*}
G\colon R(P)^{\op}&\to \sS\\
\Delta^{\mu}&\mapsto \left\{\begin{array}{cl}
X &\text{ si $\Delta^{\psi}\subseteq \Delta^{\mu}$}\\
* &\text{ sinon}
\end{array}\right.
\end{align*}
On définit un morphisme $g\colon G\to H$ par \begin{equation*}
G(\Delta^{\mu})\to H(\Delta^{\mu})= \left\{\begin{array}{cl}p &\text{ si $H(\Delta^{\mu})=Y$}\\
\Id_X &\text{ si $H(\Delta^{\mu})=X$}\\
* &\text{ si $H(\Delta^{\mu})=*$}
\end{array}\right. 
\end{equation*}
Le morphisme $g$ est une fibration triviale de $\DiagR_P$. En effet, objet par objet, il est de la forme $p,\Id_X$ ou $*$, qui sont tous des fibrations triviales.
D'autre part, $\alpha\colon F(\Delta^{\varphi})\to X$ et $\beta\colon F(\Delta^{\psi})\to Y$ induisent un morphisme de diagramme 
\begin{equation*}
\gamma\colon F\to H
\end{equation*}
défini par 
\begin{equation*}
\gamma_{\Delta^{\mu}}=\left\{\begin{array}{cl}
\alpha\circ F(\Delta^{\varphi}\to \Delta^{\mu}) &\text{ si $\Delta^{\varphi}\subseteq \Delta^{\mu}$}\\
\beta \circ F(\Delta^{\psi}\to\Delta^{\mu}) &\text{ si $\Delta^{\varphi}\not \subseteq \Delta^{\mu}$ et $\Delta^{\psi}\subseteq \Delta^{\mu}$}\\
* &\text{ sinon}
\end{array}\right.
\end{equation*}
En particulier, on a le problème de relèvement suivant
\begin{equation*}
\begin{tikzcd}
&G
\arrow{d}{g}
\\
F
\arrow{r}{\gamma}
\arrow[dashrightarrow]{ur}{\delta}
&H
\end{tikzcd}
\end{equation*}
Par hypothèse, $F$ est un objet cofibrant, et par construction $g$ est une fibration triviale de $\DiagR_P$. On en déduit qu'il existe un relèvement $\delta$. Finalement, $h=\delta_{\Delta^{\psi}}\colon F(\Delta^{\psi})\to G(\Delta^{\psi})=X$ fournit un relèvement au problème initial (Diagramme \ref{ProblemeCofibrationsMonomorphismes}). 
\end{proof}
On rappelle que si $F\colon R(P)^{\op}\to \sS$ un diagramme simplicial réduit, le foncteur $F\otimes R(P)$ est défini par
\begin{align*}
F\otimes R(P)\colon \mathcal{C}&\to \sS\\
(\Delta^{\varphi},\Delta^{\psi})&\mapsto F(\Delta^{\varphi})\otimes\Delta^{\psi}
\end{align*}
où $\mathcal{C}$ est la sous catégorie pleine de $R(P)^{\op}\times R(P)$ contenant les objets de la forme $(\Delta^{\varphi},\Delta^{\psi})$ avec $\Delta^{\psi}\subseteq \Delta^{\varphi}$.

\begin{lemme}\label{LemmeMonomorphismeColimite}
Soit $F\colon R(P)^{\op}\to \sS$ un diagramme simplicial réduit. Si pour tout morphisme de $R(P)$, $\Delta^{\psi}\to\Delta^{\varphi}$, le morphisme
\begin{equation*}
F(\Delta^{\varphi})\to F(\Delta^{\psi})
\end{equation*}
est un monomorphisme.
Alors, pour tout $\Delta^{\varphi}\in R(P)$, le morphisme canonique
\begin{equation*}
F(\Delta^{\varphi})\otimes\Delta^{\varphi}\to \colim F\otimes R(P)
\end{equation*}
est un monomorphisme.
\end{lemme}

\begin{proof}
L'ensemble des $N$-simplexes de $\colim F\otimes R(P)$ est donné par
\begin{equation*}
\left(\colim F\otimes R(P)\right)_N=
\frac{\left(\coprod_{(\Delta^{\varphi},\Delta^{\psi})}\left(F(\Delta^{\varphi})\otimes \Delta^{\psi}\right)_N\right)}{\sim}
\end{equation*}
où $\sim$ est la relation d'équivalence engendrée par 
\begin{equation*}
(x,(\Delta^{\varphi_1},\Delta^{\psi_1}))\sim (F\otimes R(P)(\alpha)(x),(\Delta^{\varphi_2},\Delta^{\psi_2}))
\end{equation*}
pour tout $x\in \left(F(\Delta^{\varphi_1})\otimes\Delta^{\psi_1}\right)_N$ et tout morphisme de $\mathcal{C}$ $\alpha\colon (\Delta^{\varphi_1},\Delta^{\psi_1})\to (\Delta^{\varphi_2},\Delta^{\psi_2})$. Il suffit de montrer que si $x,x'\in\left(F(\Delta^{\varphi})\otimes\Delta^{\psi}\right)_N$ sont deux simplexes tels que 
\begin{equation*}
(x,(\Delta^{\varphi},\Delta^{\psi}))\sim (x',(\Delta^{\varphi},\Delta^{\psi})),
\end{equation*}
alors $x=x'$. Soient $x$ et $x'$ deux tels simplexes. Alors, l'équivalence entre $x$ et $x'$ provient d'un zigzag de morphismes dans $\mathcal{C}$. Quitte à composer des morphismes composables, un tel zigzag est de l'une des formes suivantes
\begin{equation}\label{ZigZag1}
\begin{tikzcd}
&(\Delta^{\varphi_1},\Delta^{\psi_1})
\arrow{dl}{\alpha_0}
\arrow[swap]{dr}{\alpha_1}
&\dots
&(\Delta^{\varphi_{n-1}},\Delta^{\psi_{n-1}})
\arrow{dl}{\alpha_{n-2}}
\arrow[swap]{dr}{\alpha_{n-1}}
\\
(\Delta^{\varphi_0},\Delta^{\psi_0})
&
&\phantom{X}\dots\phantom{X}
&&(\Delta^{\varphi_n},\Delta^{\psi_n})
\end{tikzcd}
\end{equation}
\begin{equation}\label{ZigZag2}
\begin{tikzcd}
&(\Delta^{\varphi_1},\Delta^{\psi_1})
\arrow{dl}{\beta_0}
\arrow[swap]{dr}{\beta_1}
&\dots
&(\Delta^{\varphi_n},\Delta^{\psi_n})
\arrow{dl}{\beta_{n-1}}
\\
(\Delta^{\varphi_0},\Delta^{\psi_0})
&
&\phantom{X}\dots\phantom{X}
\end{tikzcd}
\end{equation}
\begin{equation}\label{ZigZag3}
\begin{tikzcd}
(\Delta^{\varphi_0},\Delta^{\psi_0})
\arrow[swap]{dr}{\gamma_0}
&
&(\Delta^{\varphi_2},\Delta^{\psi_2})
\arrow{dl}{\gamma_1}
\arrow[swap]{dr}{\gamma_2}
&\dots
&(\Delta^{\varphi_n},\Delta^{\psi_n})
\arrow{dl}{\gamma_{n-1}}
\\
&
(\Delta^{\varphi_1},\Delta^{\psi_1})
&
&\phantom{X}\dots\phantom{X}
\end{tikzcd}
\end{equation}
Avec $(\Delta^{\varphi_0},\Delta^{\psi_0})=(\Delta^{\varphi_n},\Delta^{\psi_n})=(\Delta^{\varphi},\Delta^{\psi})$. De plus, on peut toujours se ramener à la situation \ref{ZigZag1}. En effet, dans le cas \ref{ZigZag2}, en remplaçant $\beta_0$ par $\beta_{n-1}\circ \beta_0$, on obtient un zigzag de la forme \ref{ZigZag1} commençant et terminant en $(\Delta^{\varphi_{n-1}},\Delta^{\psi_{n-1}})$. Comme par construction $F\otimes R(P)(\beta_{n-1})$ est un monomorphisme, on a \begin{equation*}
x=x'\Leftrightarrow F\otimes R(P)(\beta_{n-1})(x)=F\otimes\Delta(P)(\beta_{n-1})(x').
\end{equation*}
De la même façon, en posant $\widetilde{\alpha}_i=\gamma_{i-1}$ pour $1\leq i\leq n-1$ et $\widetilde{\alpha}_0=\gamma_{n-1}$, on obtient un zigzag de la forme \ref{ZigZag1} à partir d'un zigzag de la forme \ref{ZigZag3}. On remarque par ailleurs que pour un tel zigzag, $n=2k$ est pair. On va prouver que $x=x'$ par récurrence sur $k$. Le cas $k=0$ est vide. Le cas $k=1$ correspond à 
\begin{equation*}
\begin{tikzcd}
&(\Delta^{\varphi_1},\Delta^{\psi_1})
\arrow{dl}{\alpha_0}
\arrow{dr}{\alpha_1}
\\
(\Delta^{\varphi},\Delta^{\psi})
&&(\Delta^{\varphi},\Delta^{\psi})
\end{tikzcd}
\end{equation*}
de plus, par construction de $\mathcal{C}$ on a nécessairement $\alpha_0=\alpha_1$. Finalement, il existe $y\in \left((F(\Delta^{\varphi})\otimes (\Delta^{\psi})\right)_N$ tel que $x'=\alpha_1(y)=\alpha_0(y)=x$. Si $k=2$, on a la situation suivante
\begin{equation}\label{ZigZagKEgal2}
\begin{tikzcd}
&(\Delta^{\varphi_1},\Delta^{\psi_1})
\arrow{dl}{\alpha_0}
\arrow[swap]{dr}{\alpha_1}
&
&(\Delta^{\varphi_{3}},\Delta^{\psi_{3}})
\arrow{dl}{\alpha_2}
\arrow[swap]{dr}{\alpha_3}
\\
(\Delta^{\varphi},\Delta^{\psi})
&
&(\Delta^{\varphi_2},\Delta^{\psi_2})
&&(\Delta^{\varphi},\Delta^{\psi})
\end{tikzcd}
\end{equation}
Soient $\Delta^{\varphi},\Delta^{\varphi'}$ deux objets de $R(P)$. S'il existe $\Delta^{\mu}\in R(P)$ tel que
$\Delta^{\varphi}\subseteq \Delta^{\mu}$ et $\Delta^{\varphi'}\subseteq \Delta^{\mu}$, alors on note $\Delta^{\varphi'}\cup \Delta^{\varphi}$ pour le plus petit objet de $R(P)$ ayant cette propriété. (Si on identifie les éléments de $R(P)$ avec leurs ensembles de sommets, l'existence d'un tel $\mu$ revient à imposer que l'ensemble $\Delta^{\varphi}\cup \Delta^{\varphi'}$ est totalement ordonné, dans ce cas leur union appartient bien à $R(P)$).
De même soient $\Delta^{\psi}, \Delta^{\psi'}$ deux objets de $R(P)$. S'il existe  $\Delta^{\nu}$ tel que $\Delta^{\nu}\subseteq \Delta^{\psi}$ et $\Delta^{\nu}\subseteq \Delta^{\psi'}$, on note $\Delta^{\psi}\cap\Delta^{\psi'}$ pour le plus grand objet de $R(P)$ ayant cette propriété. Avec ces notations, on constate que le diagramme \ref{ZigZagKEgal2} peut se factoriser comme suit

\begin{equation*}
\begin{tikzcd}[column sep= 15pt]
&(\Delta^{\varphi_1},\Delta^{\psi_1})
\arrow{dl}{\alpha_0}
\arrow{dd}{\delta_1}
\arrow[swap]{dr}{\alpha_1}
&
&(\Delta^{\varphi_{3}},\Delta^{\psi_{3}})
\arrow{dl}{\alpha_2}
\arrow[swap]{dr}{\alpha_3}
\arrow{dd}{\delta_3}
\\
(\Delta^{\varphi},\Delta^{\psi})
&
&(\Delta^{\varphi_2},\Delta^{\psi_2})
&&(\Delta^{\varphi},\Delta^{\psi})
\\
&(\Delta^{\varphi}\cup\Delta^{\varphi_2},\Delta^{\psi}\cap\Delta^{\psi_2})
\arrow{ul}{\alpha '_0}
\arrow[swap]{ur}{\alpha '_1}
\arrow[swap, leftrightarrow]{rr}{\Id}
&
&(\Delta^{\varphi_{2}}\cup\Delta^{\varphi},\Delta^{\psi_{2}}\cap\Delta^{\psi})
\arrow{ul}{\alpha '_2}
\arrow[swap]{ur}{\alpha '_3}
\end{tikzcd}
\end{equation*}
Finalement, si $x_i\in \left(F(\Delta^{\varphi_i})\otimes\Delta^{\psi_i}\right)_N$ pour $i=1,2,3$, avec $\alpha_0(x_1)=x, \alpha_1(x_1)=x_2, \alpha_2(x_3)=x_2$ et $\alpha_3(x_3)=x'$, alors $\alpha'_1\circ\delta_1(x_1)=\alpha'_2\circ\delta_3(x_3)=x_2$. 
De plus, par construction de $\mathcal{C}$, $\alpha'_1=\alpha'_2$, et $\alpha'_1$ est un monomorphisme par hypothèse. 
On en déduit que $\delta_1(x_1)=\delta_3(x_3)$. 
Finalement, on s'est ramené au cas où $k=1$. Supposons maintenant $k\geq 3$. 
Soient $x_i\in \left(F(\Delta^{\varphi_i})\otimes\Delta^{\psi_i}\right)_N$, pour $0\leq i\leq 2k$, 
avec $\alpha_{2j}(x_{2j+1})=x_{2j}$ et $\alpha_{2j+1}(x_{2j+1})=x_{2j+2}$ 
pour $0\leq j\leq k-1$, et $x_0=x, x_{2k}=x'$. 
Pour $x_i\in (F(\Delta^{\varphi_i}\otimes\Delta^{\psi_i})_N$, on a $x_i=(x_i^F,x_i^P)$ avec $x_i^F\in F(\Delta^{\varphi_i})_N$ et $x_i^P\in (\Delta^{\psi})_N$. 
De plus, en identifiant $\Delta^{\psi_i}$ avec son image dans $N(P)$, on a pour tout $i$, $x_i^P\in (N(P))_N$. 
Comme les $\alpha_i$ induisent les inclusions naturelles sur les sous ensembles de simplexes de $N(P)$, on  a pour tout $i$, $x_i^P=x^P$. En particulier, on en déduit que le sous ensemble simplicial de $N(P)$,
\begin{equation*}
\Delta^{\bar{\psi}}=\cap_{i}\Delta^{\psi_i}
\end{equation*}
est non vide (car il contient $x^P$). De plus, on a
\begin{equation*}
\Delta^{\bar{\psi}}\subseteq\Delta^{\bar{\varphi}}=\cap_i\Delta^{\varphi_i}
\end{equation*}
D'où on déduit que ce dernier est aussi non vide.
Finalement, on a ramené le zigzag \ref{ZigZag1} à 
\begin{equation*}
\begin{tikzcd}[column sep = -5pt]
&&&(\Delta^{\bar{\varphi}},\Delta^{\bar{\psi}})
&\phantom{(\Delta^{\varphi_2},\Delta^{\bar{\psi}}}
\\
&(\Delta^{\varphi_1},\Delta^{\bar{\psi}})
\arrow[swap]{dl}{\alpha_0}
\arrow{dr}{\alpha_1}
\arrow{urr}{\epsilon_1}
&&(\Delta^{\varphi_3},\Delta^{\bar{\psi}})
\arrow{dr}{\alpha_2}
\arrow{dl}{\alpha_3}
\arrow{u}{\epsilon_3}
&\dots
&(\Delta^{\varphi_{n-1}},\Delta^{\bar{\psi}})
\arrow{dr}{\alpha_{n-2}}
\arrow{dl}{\alpha_{n-1}}
\arrow[swap]{ull}{\epsilon_{n-1}}
\\
(\Delta^{\varphi},\Delta^{\psi})
&&
(\Delta^{\varphi_2},\Delta^{\bar{\psi}})
&&
\dots
&&
(\Delta^{\varphi},\Delta^{\psi})
\end{tikzcd}
\end{equation*}
Et en posant $\bar{x}=\epsilon_1(x_1)=\dots=\epsilon_{n-1}(x_{n-1})$, on obtient un zigzag de la forme
\begin{equation*}
\begin{tikzcd}
&(\Delta^{\varphi_1},\Delta^{\bar{\psi}})
\arrow[swap]{dl}{\alpha_0}
\arrow{dr}{\epsilon_1}
&&
(\Delta^{\varphi_{n-1}},\Delta^{\bar{\psi}})
\arrow[swap]{dl}{\epsilon_{n-1}}
\arrow{dr}{\alpha_{n-1}}
\\
(\Delta^{\varphi},\Delta^{\psi})
&&
(\Delta^{\bar{\varphi}},\Delta^{\bar{\psi}})
&&
(\Delta^{\varphi},\Delta^{\bar{\psi}})
\end{tikzcd}
\end{equation*}
avec $\alpha_0(x_1)=x$, $\epsilon_1(x_1)=\bar{x}=\epsilon_{n-1}(x_{n-1})$ et $\alpha_{n-1}(x_{n-1})=x'$
Finalement, on s'est ramené au cas $n=2k=4$. On en déduit que pour tout $(\Delta^{\varphi},\Delta^{\psi})\in \mathcal{C}$, le morphisme canonique
\begin{equation*}
F(\Delta^{\varphi})\otimes\Delta^{\psi}\to \colim F\otimes R(P)
\end{equation*}
est un monomorphisme ce qui implique le résultat voulu.
\end{proof}

On peut maintenant prouver le théorème \ref{EquivalenceQuillenTopNPDiagRP}.

\begin{proof}[Démonstration du théorème \ref{EquivalenceQuillenTopNPDiagRP}]
Par construction de la structure de modèle sur $\Top_{N(P)}$, $D$ préserve les fibrations et le fibrations triviales. On en déduit que l'adjonction $(\Colim,D)$ est une adjonction de Quillen. Soient $F\in \DiagR_P$ un diagramme cofibrant et $\fil{X}\in \Top_N(P)$ un espace fortement filtré, et $f\colon F\to \fil{X}$ un morphisme de $\DiagR_P$. Il suffit de montrer que $f$ est une équivalence faible de $\DiagR_P$ si et seulement si son image par l'adjonction $\widehat{f}\colon \Colim(F)\to \fil{X}$ est une équivalence faible de $\Top_{N(P)}$. Or, par définition des équivalences faibles sur $\Top_{N(P)}$, $\widehat{f}$ est une équivalence faible si et seulement si $D(\widehat{f})\colon D(\Colim(F))\to D(X)$ est une équivalence faible de $\DiagR_P$. De plus, on constate que $f$ est égale à la composée
\begin{equation*}
F\xrightarrow{\eta_F}D(\Colim(F))\xrightarrow{D(\widehat{f})}D(X),
\end{equation*}
où $\eta$ est l'unité de l'adjonction $(\Colim,D)$.
Ainsi, par deux sur trois, il suffit de montrer que pour tout diagramme cofibrant $F\in \DiagR_P$, $\eta_F$ est une équivalence faible. Par définition de la structure de modèle sur $\DiagR_P$, cela revient à montrer que pour tout $\Delta^{\mu}\in R(P)$, le morphisme
\begin{equation}\label{MorphismeEquivalenceQuillenDiagramme}
F(\Delta^{\mu})\to\Map(\RealNP{\Delta^{\mu}},\Colim(F))
\end{equation}
est une équivalence faible pour la structure de Kan Quillen sur $\sS$. Montrons d'abord que lorsque $F$ est cofibrant, le morphisme canonique $\Real{F(\Delta^{\mu})}\otimes\RealNP{\Delta^{\mu}}\to \Colim(F)$ induit un isomorphisme d'ensembles simpliciaux
\begin{equation*}
\Map(\RealNP{\Delta^{\mu}},\Real{F(\Delta^{\mu})}\otimes\RealNP{\Delta^{\mu}})\xrightarrow{\simeq}\Map(\RealNP{\Delta^{\mu}},\Colim(F)).
\end{equation*}
Pour $(\Delta^{\varphi},\Delta^{\psi})\in \mathcal{C}$, on note $i_{(\Delta^{\varphi},\Delta^{\psi})}$ le morphisme canonique
\begin{equation*}
i_{(\Delta^{\varphi},\Delta^{\psi})}\colon \Real{F(\Delta^{\varphi})}\otimes\RealNP{\Delta^{\psi}}\to \Colim(F)
\end{equation*}
Alors, par construction de $\Colim(F)$, on a
\begin{equation*}
\Colim(F)=\bigcup_{(\Delta^{\varphi},\Delta^{\psi})\in \mathcal{C}}\Im(i_{(\Delta^{\varphi},\Delta^{\psi})})=\bigcup_{\Delta^{\varphi}\in R(P)}\Im(i_{(\Delta^{\varphi},\Delta^{\varphi})}).
\end{equation*}
Ceci implique que pour tout $\Delta^{\varphi}\in R(P)$, le carré suivant est cartésien
\begin{equation}\label{DiagrammePullbackInterieur}
\begin{tikzcd}
i_{(\Delta^{\varphi},\Delta^{\varphi})}(\Real{F(\Delta^{\varphi})}\otimes\Int(\RealNP{\Delta^{\varphi})})
\arrow[hookrightarrow]{r}
\arrow{d}
&\Colim(F)
\arrow{d}
\\
\Int(\RealNP{\Delta^{\varphi}})
\arrow{r}
&\Real{N(P)}
\end{tikzcd}
\end{equation}
où $\Int(\RealNP{\Delta^{\varphi}})$ désigne l'intérieur de la réalisation du simplexe. Plus précisément, $\Int(\RealNP{\Delta^{\varphi}})$ désigne le simplexe entier si ce dernier est de dimension $0$, et le sous espace $\{(t_0,\dots,t_m\ \sum_{i}t_i=1,0<t_i\leq 1\}\subseteq \RealNP{\Delta^{\varphi}}$ si $\Delta^{\varphi}$ est de dimension $m\geq 1$).  
Soit $\sigma\colon \Real{\Delta^n}\otimes\RealNP{\Delta^{\mu}}\to \Colim(F)$ un simplexe de $\Map(\RealNP{\Delta^{\mu}},\Colim(F))$. On note $\widetilde{\sigma}$ sa restriction à $\Real{\Delta^n}\otimes\Int(\RealNP{\Delta^{\mu}})$. Comme le carré \ref{DiagrammePullbackInterieur} est cocartésien, $\widetilde{\sigma}$ se factorise par l'inclusion
\begin{equation*}
i_{(\Delta^{\mu},\Delta^{\mu})}(\Real{F(\Delta^{\mu})}\otimes\Int(\RealNP{\Delta^{\mu})})
\hookrightarrow
\Colim(F)
\end{equation*}
De plus, comme l'image de $F(\Delta^{\mu})\otimes\Delta^{\mu}$ est un sous ensemble simplicial de $\colim F\otimes R(P)$, on a
\begin{equation*}
\overline{i_{(\Delta^{\mu},\Delta^{\mu})}(\Real{F(\Delta^{\mu})}\otimes\Int(\RealNP{\Delta^{\mu})})}
=i_{(\Delta^{\mu},\Delta^{\mu})}(\Real{F(\Delta^{\mu})}\otimes(\RealNP{\Delta^{\mu})})
\end{equation*}
Finalement, $\sigma$ se factorise par
\begin{equation*}
i_{(\Delta^{\mu},\Delta^{\mu})}(\Real{F(\Delta^{\mu})}\otimes(\RealNP{\Delta^{\mu})})\hookrightarrow \Colim(F).
\end{equation*}
D'autre part, comme $F$ est un objet cofibrant, le lemme \ref{LemmeDiagrammesCofibrants} nous permet d'appliquer le lemme \ref{LemmeMonomorphismeColimite}. On en déduit que le morphisme canonique
\begin{equation*}
F(\Delta^{\mu})\otimes\Delta^{\mu}\hookrightarrow \colim F\otimes R(P)
\end{equation*}
est un monomorphisme, et donc que
\begin{equation*}
i_{(\Delta^{\mu},\Delta^{\mu})}\colon \Real{F(\Delta^{\mu})}\otimes\RealNP{\Delta^{\mu}} \hookrightarrow \Colim(F)
\end{equation*}
est une cofibration. Ainsi on a une factorisation de $\sigma$ sous la forme
\begin{equation*}
\begin{tikzcd}
&\Real{F(\Delta^{\mu})}\otimes\RealNP{\Delta^{\mu}}
\arrow[hookrightarrow]{d}{i_{(\Delta^{\mu},\Delta^{\mu})}}
\\
\Real{\Delta^n}\otimes\RealNP{\Delta^{\mu}}
\arrow{r}{\sigma}
\arrow[dashrightarrow]{ur}{\tau}
&\Colim(F)
\end{tikzcd}
\end{equation*}
On en déduit que le morphisme canonique
\begin{equation*}
\Map(\RealNP{\Delta^{\mu}},\Real{F(\Delta^{\mu})}\otimes\RealNP{\Delta^{\mu}})\to \Map(\RealNP{\Delta^{\mu}},\Colim(F))
\end{equation*}
est une surjection. De plus, comme on vient de montrer que $\Real{F(\Delta^{\mu})}\otimes\RealNP{\Delta^{\mu}}$ est un sous espace de $\Colim(F)$, c'est aussi une injection. Il reste à montrer que le morphisme
\begin{equation*}
F(\Delta^{\mu})\to\Map(\RealNP{\Delta^{\mu}},\Real{F(\Delta^{\mu})}\otimes\RealNP{\Delta^{\mu}})
\end{equation*}
est une équivalence faible. Soit $\tau$ un simplexe de $\Map(\RealNP{\Delta^{\mu}},\Real{F(\Delta^{\mu})}\otimes\RealNP{\Delta^{\mu}})$ on décompose $\tau$ sous la forme
\begin{align*}
\tau\colon \Real{\Delta^n}\otimes\RealNP{\Delta^{\mu}}&\to \Real{F(\Delta^{\mu})}\otimes\RealNP{\Delta^{\mu}}\\
(x,y)&\mapsto (\tau_F(x,y),\tau_{\Delta^{\mu}}(x,y))
\end{align*}
Comme $\tau$ est filtré, on doit avoir $\tau_{\Delta^{\mu}}(x,y)=y$. Ainsi, $\tau$ est complétement détérminé par l'application non filtrée
\begin{equation*}
\tau_F\colon \Real{\Delta^n}\times\Real{\Delta^{\mu}}\to\Real{F(\Delta^{\mu})}
\end{equation*}
On en déduit l'isomorphisme
\begin{equation*}
\Map(\RealNP{\Delta^{\mu}},\Real{F(\Delta^{\mu})}\otimes\RealNP{\Delta^{\mu}})\simeq \Map_{\Top}(\Real{\Delta^{m}},\Real{F(\Delta^{\mu})})
\end{equation*}
où $\Delta^{\mu}=(\Delta^m,\mu)$. Finalement, on a les équivalences faibles suivantes :
\begin{align*}
F(\Delta^{\mu})&\simeq  \Map(\Delta^0,F(\Delta^{\mu}))\\
&\sim \Map(\Delta^0,\Sing(\Real{F(\Delta^{\mu})})\\
& \sim \Map(\Delta^m,\Sing(\Real{F(\Delta^{\mu})})\\
& \simeq \Map(\Real{\Delta^m},\Real{F(\Delta^{\mu})})\\
& \simeq \Map(\RealNP{\Delta^{\mu}},\Real{F(\Delta^{\mu})}\otimes\RealNP{\Delta^{\mu}})\\
& \simeq \Map(\RealNP{\Delta^{\mu}},\Colim(F))
\end{align*}
dont la composée est le morphisme \ref{MorphismeEquivalenceQuillenDiagramme}. 
\end{proof}

\begin{remarque}\label{RemarqueInteretsTopNP}
Cette preuve fait apparaitre deux avantages majeurs des espaces fortement filtrés. D'une part, pour tout $\Delta^{\varphi}\in R(P)$, on a 
\begin{equation*}
\Map(\RealNP{\Delta^{\varphi}},\RealP{N(P)})\simeq \Map(\RealNP{\Delta^{\varphi}},\Real{\Delta^{\varphi}})\simeq \{*\},
\end{equation*}
ce qui n'est pas le cas pour les espaces filtrés (voir la preuve du lemme \ref{LemmeVarphiPFibrationTriviale} où on montre que $\Map(\RealP{\Delta^{\varphi}},\RealP{N(P)})$ est contractile).
D'autre part, on a le carré cartésien \ref{DiagrammePullbackInterieur}, que l'on peut généraliser comme suit. Soient $\fil{X}$ un ensemble simplicial filtré et $\Delta^{\varphi}\in R(P)$ un simplexe non dégénéré. Notons $X_{\Delta^{\varphi}}$ l'ensemble des simplexes de $X$ de la forme $\sigma\colon\Delta^{\epsilon}\to \fil{X}$, avec $\Delta^{\epsilon}$ une dégénérescence de $\Delta^{\varphi}$. Alors, on a un carré cartésien
\begin{equation*}
\begin{tikzcd}
\bigcup_{\sigma\in X_{\Delta^{\varphi}}}\sigma(\Int(\RealNP{\Delta^{\epsilon}})
\arrow[hookrightarrow]{r}
\arrow{d}
&\RealNP{\fil{X}}
\arrow{d}{\Real{\varphi_X}}
\\
\Int(\RealNP{\Delta^{\varphi}})
\arrow{r}
&\Real{N(P)}
\end{tikzcd}
\end{equation*}
Dans le cas où $\Delta^{\varphi}=[p]$, la pré-image ainsi obtenue est la $p$-ième strate de $\RealNP{\fil{X}}$. Si $\Delta^{\varphi}=[p_0,p_1]$, on obtient un voisinage tubulaire de la strate $p_0$ dans la strate $p_1$ (privé de la strate $p_0$). En général, on obtient une généralisation de la notions de voisinage tubulaire. Le fait que ces "voisinages tubulaires" apparaissent naturellement permet de trivialiser un certain nombre de preuves pour les espaces fortement filtrés. (Voir par exemple la proposition \ref{GroupesHomotopiesFibre} dont la preuve repose sur le résultat technique \cite[Proposition A.7.9]{HigherAlgebra}).
\end{remarque}

%

 \chapter{Adjonction de Kan-Quillen filtrée}
\label{ChapitreKanQuillen}
Dans le chapitre \ref{ConstructionCMFSSetP}, nous avons muni $\sS_P$ d'une structure de modèle. Ensuite, au chapitre \ref{ChapitreGroupesHomotopiesFiltresEspaces} on a utilisé la catégorie modèle $\sS_P$  pour obtenir des résultats sur les espaces filtrés. D'autre part, dans le chapitre précédent nous avons construit une structure modèle sur les catégories des espaces filtrés et fortement filtrés. L'objet de se chapitre est d'éclaircir le lien entre ces différentes catégories modèles.

On commence ce chapitre par identifier une classe d'objets cofibrants de la catégorie modèle $\Top_{N(P)}$. Dans la section \ref{SectionRealisationSubdivisionCofibrant}, on montre que la réalisation de la subdivision d'un ensemble simplicial filtré est cofibrante (c'est la proposition \ref{PropositionSubdivisionCofibrante}). Puis, dans la section \ref{SectionLvPEquivalenceFaible} on montre que pour tout ensemble simplicial filtré $\fil{A}$, l'application de dernier sommet se réalise en une équivalence faible de $\Top_{N(P)}$
\begin{equation*}
\RealNP{\lv_P}\colon\RealNP{\sd_P\fil{A}}\to\RealNP{\fil{A}},
\end{equation*}
c'est la proposition \ref{LastVertexEquivalenceFaible}. On en déduit ensuite dans la section \ref{SectionRemplacementsCofibrantsPL}, une méthode pour calculer de façon élémentaire des remplacements cofibrants pour les espaces filtrés PL (voir Proposition \ref{PropositionRemplacementCofibrantEfficace}). 
Finalement, dans la section \ref{SectionRemplacementsCofibrantsFibre}, on modifie la construction de la section \ref{ConstructionFibre} pour obtenir des espaces fortement filtrés, et on explicite
des remplacements cofibrants pour les espaces issus de cette construction. Ceci permet en particulier d'identifier la sous catégorie pleine de $\Ho(\Top_{N(P)})$ contenant ces espaces fortement filtrés avec la catégorie homotopique naïve correspondante. (voir la proposition \ref{PropositionClasseHomotopieNaiveFibre}). On conjecture que la construction généralisée de la section \ref{SectionFibresFiltres} permet d'étendre ce résultat à la classe des espaces coniquement stratifiés PL. Voir la remarque \ref{RemarqueWhiteheadFibre} ainsi que la conjecture \ref{ConjectureSousCategorieFibre}.

Dans la section \ref{SectionSSTopP}, on considère une seconde structure modèle sur la catégorie $\sS_P$, transportée depuis $\Top_{N(P)}$ le long de l'adjonction $(\RealNP{-},\Sing_{N(P)})$. On la note $\sSTop_P$ pour la distinguer de la structure $\sSU_P$ du théorème \ref{TheoDescriptionExpliciteSSetP}. Dans la section \ref{SectionConstructionCMFSSTopP}, on montre qu'une telle structure transportée est bien définie, c'est le théorème \ref{TheoremeCategorieModeleSSTop}. La preuve repose essentiellement sur les propositions \ref{PropositionSubdivisionCofibrante} et \ref{LastVertexEquivalenceFaible} de la section \ref{SectionEspaceFiltreCofibrants}. Dans la section \ref{SectionAdjonctionQuillenSSUSSTop}, on montre que la catégorie modèle est liée aux catégories modèles $\sSU_P$ et $\Top_{N(P)}$ par deux adjonctions de Quillen. Ceci implique que les catégories $\sSU_P$ et $\Top_{N(P)}$ sont reliés par l'adjonction de Quillen composée. C'est le corollaire \ref{CorollaireAdjonctionKanQuillen}.

Dans la section \ref{SectionKanQuillen}, on tente de déduire des constructions précédentes qu'il existe une équivalence de Quillen entre $\sSU_P$ et $\Top_{N(P)}$. Comme l'adjonction de Quillen entre ces deux catégories est la composée de deux adjonctions de Quillen, on les examine séparément. Dans la section \ref{SectionRealNPSingNP}, on étudie l'adjonction entre $\sSTop_P$ et $\Top_{N(P)}$. On conjecture que c'est une équivalence de Quillen (voir Conjecture \ref{ConjectureRealNPSingNP}), et on ramène cette conjecture à l'existence d'homotopies non filtrées satisfaisant certaines propriétés (voir les conditions \ref{EnumerateHXHomotopie}, \ref{EnumerateHPHomotopie}, \ref{EnumerateHPRestrictsToProjection} et \ref{EnumerateHomotopiesCommute}).
Dans la section \ref{SectionSdPExP} on s'intéresse ensuite à l'adjonction entre $\sSU_P$ et $\sSTop_P$. On montre notamment que l'unité et la counité de ces adjonctions sont des équivalences faibles objet par objet (proposition \ref{PropositionUniteCouniteEquivalenceFaible}), mais ceci ne suffit pas à conclure que l'adjonction est une équivalence de Quillen (Conjecture \ref{ConjectureSdPExP}). On montre cependant qu'il suffit de prouver que les cofibrations triviales de $\sSTop_P$ sont des cofibrations triviales de $\sSU_P$ pour en déduire la conjecture \ref{ConjectureSdPExP}.
Finalement, dans la section \ref{SectionSingNPRealNPPreserventEquivalences}, on montre que bien que l'adjonction $(\RealNP{-},\Sing_{N(P)})$ entre $\sSU_P$ et $\Top_{N(P)}$ n'est pas une adjonction de Quillen, les foncteurs $\RealNP{-}$ et $\Sing_{N(P)}$ préservent les équivalences faibles. C'est le théorème \ref{RealisationSingPreserventEquivalencesFaibles}.

\section{Espaces (fortement) filtrés cofibrants}
\label{SectionEspaceFiltreCofibrants}
Dans cette section et pour le reste de ce chapitre, on travaille avec la catégorie des espaces fortement filtrés, notamment pour les raisons mentionnées dans la remarque \ref{RemarqueInteretsTopNP}. Cependant, par construction des structures modèles sur $\Top_P$ et $\Top_{N(P)}$ et en vertu du théorème \ref{EquivalenceQuillenTopPTopNP}, les conclusions qu'on obtient ici sont aussi valables pour $\Top_P$.

\subsection{Réalisation d'un ensemble simplicial filtré}
\label{SectionRealisationSubdivisionCofibrant}
Les foncteurs adjoints $\RealNP{-}$ et $\Sing_{N(P)}$ ne forment pas une adjonction de Quillen entre les structures de modèles des théorèmes \ref{TheoDescriptionExpliciteSSetP} et \ref{CategorieModeleTopNP}. En effet, comme on l'a vu dans l'exemple \ref{ExempleEspaceNonFibrant}, si $\fil{X}$ est un espace (fortement) filtré, $\Sing_{N(P)}\fil{X}$ n'est pas fibrant en général. Ceci implique en particulier que si $i\colon\fil{A}\to\fil{B}$ est un monomorphisme entre ensembles simpliciaux filtrés, $\RealNP{i}\colon \RealNP{A}\to\RealNP{B}$ n'est pas toujours une cofibration de $\Top_{N(P)}$. Par exemple, on vérifie facilement que la réalisation de l'inclusion de cornet de l'exemple \ref{ExempleEspaceNonFibrant} est une équivalence faible. Comme celle ci n'a pas la propriété de relèvement par rapport à tous les morphismes de la forme $\fil{X}\to \Real{N(P)}$ (qui sont tous des fibrations de $\Top_{N(P)}$), ceci implique qu'elle ne peut pas être une cofibration triviale. En particulier ce n'est pas une cofibration. Cependant, le résultat suivant nous permettra d'obtenir une vaste classe d'ensembles simpliciaux filtrés dont la réalisation fortement filtrée est cofibrante.

\begin{prop}\label{PropositionSubdivisionCofibrante}
Soit $\fil{A}$ un ensemble simplicial filtré, $\RealNP{\sd_P\fil{A}}$ est un espace fortement filtré cofibrant.
\end{prop}

\begin{proof}
On rappelle que $\sk_n(A)$ désigne le $n$-squelette de $A$, c'est à dire le sous ensemble simplicial de $A$ généré par les simplexes de dimension inférieure à $n$. En notant $\Sigma_n$ l'ensemble des simplexes non dégénérés de $A$, on a les carrés cocartésiens suivants pour $n\geq 0$ (avec $\sk_{-1}(A)=\emptyset$),
\begin{equation*}
\begin{tikzcd}
\Coprod\limits_{\sigma\in \Sigma_n}\partial(\Delta^n)
\arrow{r}{\coprod\sigma}
\arrow{d}
&
\sk_{n-1}(A)
\arrow{d}
\\
\Coprod\limits_{\sigma\in \Sigma_n}\Delta^n
\arrow{r}{\coprod\sigma}
&\sk_n(A)
\end{tikzcd}
\end{equation*}
et on peut écrire $A$ comme la colimite
\begin{equation*}
A\simeq \colim_{n\in \N}\sk_{n}(A).
\end{equation*}
De même, on a les carrés cocartésiens dans $\sS_P$
\begin{equation*}
\begin{tikzcd}
\Coprod\limits_{\sigma\in \Sigma_n}(\partial(\Delta^n),\varphi_A\circ\sigma)
\arrow{r}{\coprod\sigma}
\arrow{d}
&
\sk_{n-1}\fil{A}
\arrow{d}
\\
\Coprod\limits_{\sigma\in \Sigma_n}(\Delta^n,\varphi_A\circ\sigma)
\arrow{r}{\coprod\sigma}
&\sk_n\fil{A}
\end{tikzcd}
\end{equation*}
et
\begin{equation*}
\fil{A}\simeq \colim_{n\in \N}\sk_{n}\fil{A}.
\end{equation*}
Comme les foncteurs de subdivisions filtrées et de réalisations fortement filtrées sont des adjoints à gauche, ils préservent les colimites. On a donc aussi les carrés cocartésiens dans $\Top_{N(P)}$,
\begin{equation*}
\begin{tikzcd}[column sep = huge]
\Coprod\limits_{\sigma\in \Sigma_n}\RealNP{sd_P(\partial(\Delta^n),\varphi_A\circ\sigma)}
\arrow{r}{\coprod\RealNP{\sd_P(\sigma)}}
\arrow{d}
&
\RealNP{\sd_P(\sk_{n-1}\fil{A})}
\arrow{d}
\\
\Coprod\limits_{\sigma\in \Sigma_n}\RealNP{\sd_P(\Delta^n,\varphi_A\circ\sigma)}
\arrow{r}{\coprod\RealNP{\sd_P(\sigma)}}
&\RealNP{\sd_P(\sk_n\fil{A})}
\end{tikzcd}
\end{equation*}
et
\begin{equation*}
\RealNP{\sd_P\fil{A})}\simeq \colim_{n\in \N}\RealNP{\sd_P(\sk_{n}\fil{A})}.
\end{equation*}
Comme la classe des cofibrations est stable par unions disjointes, sommes amalgamées et compositions transfinies, il suffit de montrer que pour tout $\Delta^{\varphi}\in \Delta(P)$ l'inclusion
\begin{equation*}
\RealNP{\sd_P(\partial(\Delta^{\varphi}))}\to \RealNP{\sd_P(\Delta^{\varphi})}
\end{equation*}
est une cofibration de $\Top_{N(P)}$. Fixons $\Delta^{\varphi}=(\Delta^n,\varphi)\in \Delta(P)$. On note $\Delta^{\bar{\varphi}}\in R(P)$ le simplexe non dégénéré de $N(P)$ dont $\Delta^{\varphi}$ est une dégénérescence. Pour $\Delta^{\psi}\subseteq \Delta^{\bar{\varphi}}\in R(P)$, on note $X^{\psi}$ le sous ensemble simplicial filtré de  $\sd_P(\Delta^{\varphi})$ engendré par les simplexes filtrés de la forme
\begin{equation*}
\sigma\colon \Delta^{\mu}\to\sd_P(\Delta^{\varphi})
\end{equation*} 
avec $\Delta^{\bar{\mu}}=\Delta^{\psi}$. D'autre part, on pose 
\begin{equation*}
X^{n+1}=\sd_P(\partial(\Delta^{\varphi}))\subset \sd_P(\Delta^{\varphi}),
\end{equation*}
et pour tout $0\leq k\leq n$, on définit $X^k$ comme la somme amalgamée suivante
\begin{equation*}
\begin{tikzcd}
\Coprod\limits_{\substack{\Delta^{\psi}\subseteq\Delta^{\bar{\varphi}}\\
\dim(\Delta^{\psi})=k}}X^{\psi}\cap X^{k+1}
\arrow{r}
\arrow{d}
&X^{k+1}
\arrow{d}
\\
\Coprod\limits_{\substack{\Delta^{\psi}\subseteq\Delta^{\bar{\varphi}}\\
\dim(\Delta^{\psi})=k}}X^{\psi}\phantom{\cap X^{\mu+1}}
\arrow{r}
&X^k
\end{tikzcd}
\end{equation*}
Montrons qu'on a un isomorphisme $X^0\simeq \sd_P(\Delta^{\varphi})$. Par construction, on a un morphisme surjectif $X^0\to \sd_P(\Delta^{\varphi})$ induit par les inclusions $X^{\psi}\subset \sd_P(\Delta^{\varphi})$. Montrons que c'est une injection. Soient $\Delta^{\psi_1}\not=\Delta^{\psi_2}$ deux simplexes non dégénérés de $N(P)$ avec $\dim(\Delta^{\psi_1})=\dim(\Delta^{\psi_2})=k$ et $\Delta^{\psi_1},\Delta^{\psi_2}\subseteq \Delta^{\bar{\varphi}}$, et soient $\sigma\in X^{\psi_1}, \tau\in X^{\psi_2}$ tels que $\sigma\cap \tau\not =\emptyset$. On note
\begin{equation*}
\sigma=[(\sigma_0,p_0),\dots,(\sigma_m,p_m)]
\end{equation*}
et
\begin{equation*}
\tau=[(\tau_0,q_0),\dots,(\tau_{m'},q_{m'})].
\end{equation*}
(On renvoie à la proposition \ref{SubdivisionSimplexe} pour une description des simplexes de $\sd_P(\Delta^{\varphi})$.)
Notons leur intersection
\begin{equation*}
\nu=[(\nu_0,r_0),\dots,(\nu_l,r_l)]
\end{equation*}
on remarque que l'ensemble $Q=\{p_0,\dots,p_m\}\cup\{q_0,\dots,q_{m'}\}\subset P$ est complètement ordonné (car c'est un sous ensemble de $\Delta^{\varphi}_0$). On énumère ses éléments comme suit. On note $s^{-1}_i$ le $i+1$-ème plus grand élément de $Q$ vérifiant $r_0\leq s^{-1}_i<r_1$, on note $s^{l}_i$ le $i+1$-ème plus grand élément de $Q$ vérifiant $r_l\leq s^l_i$ et pour $0\leq j<l$, on note $s^j_i$ le $i+1$-ème plus grand élément de $Q$ vérifiant $r_i\leq s^j_i<r_{i+1}$. On définit alors $\nu'\in \sd_P(\Delta^{\varphi})$ comme suit :
\begin{equation}\label{EquationRajouterCouleurs}
\nu'=[(\nu_0,s^{-1}_0),(\nu_0,s^{-1}_1),\dots, (\nu_0,s_0)(\nu_0,s^{0}_0),(\nu_0,s^0_1),\dots,(\nu_1,s_1),(\nu_1,s^1_0),\dots,(\nu_l,s_l),(\nu_l,s^l_0),\dots]
\end{equation}
Le simplexe $\nu'$ est bien dans $\sd_P(\Delta^{\varphi})$ car par hypothèse $\sigma_0\subseteq \nu_0$ et $\tau_0\subseteq\nu_0$, et par construction $\nu'$ contient $\nu$. D'autre part, $\nu'$ est de la forme $\Delta^{\mu}\to \sd_P(\Delta^{\varphi})$ avec 
\begin{equation*}
\Delta^{\mu}=[s^{-1}_0,s^{-1}_1,\dots, s_0,s^{0}_0,s^0_1,\dots,s_1,s^1_0\dots,s_l,s^l_0,\dots]
\end{equation*}
et par construction $\Delta^{\psi_1}\not=\Delta^{\psi_2}\subseteq \Delta^{\bar{\mu}}$. On en déduit que $\dim(\Delta^{\bar{\mu}})>k$, et donc, $\sigma\cap\tau\in X^{k+1}$, et finalement on a $X^0\simeq \sd_P(\Delta^{\varphi})$.
Comme les cofibrations sont stables par compositions transfinies, unions disjointes et sommes amalgamées, il suffit de montrer que pour tout $\Delta^{\psi}\in R(P)$ avec $\dim(\Delta^{\psi})=k$, le morphisme
\begin{equation}\label{MorphismeRealXPsi}
\RealNP{X^{\psi}\cap X^{k+1}}\to \RealNP{X^{\psi}}
\end{equation}
est une cofibration de $\Top_{N(P)}$. Dans ce but, on décrit explicitement les simplexes de $X^{\psi}$. 
Soit 
\begin{equation*}
\sigma= [(\sigma_0,p_0),\dots,(\sigma_n,p_n)]
\end{equation*} 
un simplexe de $\sd_P(\Delta^{\varphi})$. On a $\sigma_0\subseteq\dots\subseteq\sigma_n\subseteq \Delta^{\varphi}$, $p_0\leq\dots\leq p_n\in P$ et $p_0,\dots,p_n\in \varphi(\sigma_0)$. De plus, $\sigma$ est dans $X^{\psi}$ si et seulement si, on a aussi $\Delta^{\psi}\subseteq\varphi(\sigma_0)$. En effet, si $\sigma$ est dans $X^{\psi}$, il existe $\tau\colon \Delta^{\mu}\to \sd_P(\Delta^{\varphi})$ tel que $\sigma\subseteq \tau$ et $\Delta^{\bar{\mu}}=\Delta^{\psi}$. Mais alors, en notant 
\begin{equation*}
\tau=[(\tau_0,q_0),\dots,(\tau_m,q_m)],
\end{equation*}
on a $\Delta^{\psi}\subseteq {[q_0,\dots,q_n]}\subseteq\varphi(\tau_0)\subseteq\varphi(\sigma_0)$. Réciproquement, si $\Delta^{\psi}\subseteq\varphi(\sigma_0)$, on exhibe un $\tau\colon \Delta^{\mu}\to \sd_P(\Delta^{\varphi})$ par une construction similaire à la construction \ref{EquationRajouterCouleurs}. Finalement, on observe que $X^{\psi}$ est de la forme $L^{\psi}\otimes\Delta^{\psi}$ où $L^{\psi}$ est le sous ensemble simplicial de $\sd(\Delta^{\varphi})$ donné par
\begin{equation*}
L^{\psi}=\{[\sigma_0,\dots,\sigma_n]\ |\ \Delta^{\psi}\subseteq\varphi(\sigma_0)\}\subseteq\sd(\Delta^{\psi}).
\end{equation*}
D'autre part, un simplexe $\sigma\in \sd_P(\Delta^{\varphi})$ est dans $X^{\psi}\cap X^{k+1}$ si et seulement si il existe $\Delta^{\psi'}\in R(P)$ avec $\Delta^{\psi}\subsetneq\Delta^{\psi'}$ tel que $\sigma\in X^{\psi'}$, ou si $\sigma\in \sd_P(\partial(\Delta^{\varphi}))$. On en déduit que $X^{\psi}\cap X^{k+1}\simeq K^{\psi}\otimes\Delta^{\psi}$ avec 
\begin{equation*}
K^{\psi}=\{[\sigma_0,\dots,\sigma_n]\ |\ \Delta^{\psi'}\subseteq\varphi(\sigma_0) \text{ pour un certain $\Delta^{\psi}\subsetneq\Delta^{\psi'}$, ou } \Delta^{\psi}\subseteq \varphi(\sigma_0) \text{ et $\sigma_n\not=\Delta^{\varphi}$}
\end{equation*}
Finalement, le morphisme \ref{MorphismeRealXPsi} est de la forme
\begin{equation*}
\RealNP{K^{\psi}\otimes\Delta^{\psi}}\to\RealNP{L^{\psi}\otimes\Delta^{\psi}}.
\end{equation*}
En particulier, c'est l'image d'une cofibration de $\DiagR_P$. En effet, par le lemme \ref{ColimCofibrationsGeneratrices} $\RealNP{K^{\psi}\otimes\Delta^{\psi}}\simeq \Colim((K^{\psi})^{\Delta^{\psi}})$, et par construction de la strucuture de modèle sur $\DiagR_P$, comme $K^{\psi}\to L^{\psi}$ est un monomorphisme, $(K^{\psi})^{\Delta^{\psi}}\to (L^{\psi})^{\Delta^{\psi}}$ est une cofibration de $\DiagR_P$ (voir Proposition \ref{CategorieModeleDiagrammesReduits}). On en déduit que le morphisme \ref{MorphismeRealXPsi} est une cofibration de $\Top_{N(P)}$, ce qui conclut la preuve.
\end{proof}

\begin{remarque}\label{RemarqueSdPEnvoieMonoSurCofibration}
En reprenant la preuve de la proposition \ref{PropositionSubdivisionCofibrante}, on obtient aussi que pour tout monomorphisme d'ensembles simpliciaux filtrés $f\colon\fil{X}\to\fil{Y}$, le morphisme 
\begin{equation*}
\RealNP{\sd_P(f)}\colon\RealNP{\sd_P\fil{X}}\to\RealNP{\sd_P\fil{Y}}
\end{equation*} est une cofibration de $\Top_{N(P)}$.
\end{remarque}

\subsection{Remplacement cofibrant}
\label{SectionLvPEquivalenceFaible}
On rappelle que pour $\fil{A}$ un ensemble simplicial filtré, $\lv_P\colon \sd_P\fil{A}\to \fil{A}$ désigne le morphisme de dernier sommet filtré. Voir la définition \ref{DernierSommetFiltre}. L'objet de cette sous-section est de montrer que $\lv_P\colon \sd_P\fil{A}\to\fil{A}$ correspond à un foncteur de remplacement cofibrant. Ceci revient à montrer la proposition suivante.

\begin{prop}\label{LastVertexEquivalenceFaible}
Soit $\fil{A}$ un ensemble simplicial filtré. L'application $\lv_P\colon \sd_P\fil{A}\to\fil{A}$ induit une équivalence faible dans $\Top_{N(P)}$ :
\begin{equation*}
\RealNP{\lv_P}\colon \RealNP{\sd_P\fil{A}}\to\RealNP{\fil{A}}.
\end{equation*}
\end{prop}

\begin{defin}
Soit $\Delta^{\varphi}\in R(P)$ un simplexe non dégénéré et $\fil{A}$ un ensemble simplicial filtré. On définit la $\varphi$-troncature de $\fil{A}$, $\tr_{\varphi}\fil{A}$, comme le sous ensemble simplicial filtré de $\fil{A}$ engendré par les simplexes de la forme $\sigma\colon \Delta^{\psi}\to \fil{A}$ avec $\Delta^{\bar{\psi}}=\Delta^{\varphi}$.
\end{defin}

%
%
%
%

Par construction, la $\varphi$-troncature vérifie la propriété suivante :

\begin{lemme}\label{LemmeMapTroncature}
Soient $\fil{A}$ un ensemble simplicial filtré, $\Delta^{\varphi}\in R(P)$ et $K$ un ensemble simplicial. Alors, l'inclusion $\tr_{\varphi}\fil{A}\to\fil{A}$ induit un isomorphisme
\begin{equation*}
\Map(K\otimes\Delta^{\varphi},\tr_{\varphi}\fil{A})\simeq \Map(K\otimes\Delta^{\varphi},\fil{A}).
\end{equation*}
De même, pour tout espace topologique $X$, l'inclusion induit un isomorphisme
\begin{equation*}
\Map(X\otimes\RealNP{\Delta^{\varphi}},\RealNP{\tr_{\varphi}\fil{A}})\simeq \Map(X\otimes\RealNP{\Delta^{\varphi}},\RealNP{\fil{A}})
\end{equation*}
\end{lemme}

\begin{proof}
Le cas simplicial est une conséquence directe de la définition. Dans le cas topologique, comme $\tr_{\varphi}\fil{A}$ est un sous ensemble simplicial (filtré), le morphisme considéré est nécessairement un monomorphisme. Montrons que c'est un épimorphisme. Soient $f\colon (\Real{\Delta^n}\times X)\otimes\RealNP{\Delta^{\varphi}}\to \RealNP{\fil{A}}$ un simplexe de $\Map(X\otimes\RealNP{\Delta^{\varphi}},\RealNP{\fil{A}})$ et $t\in \Int(\RealNP{\Delta^{\varphi}})$. Alors, comme $f$ est une application filtrée, pour tout $x\in (\Real{\Delta^n}\times X)$, on a $\RealNP{\varphi_A}(f(x,t))=t\in \Int(\Real{\Delta^{\varphi}})\subset \Real{N(P)}$. Ceci implique que $f(x,t)$ est de la forme $(\sigma,(y_0,\dots,y_m))\in A_m\times \Delta^m$ avec $\sigma\colon \Delta^{\psi}\to \fil{A}$, $\Delta^{\bar{\psi}}=\Delta^{\varphi}$, et $y_i>0$ pour tout $i$. En particulier, $f(x,t)\in \RealNP{\tr_{\varphi}\fil{A}}$. Comme $\RealNP{\tr_{\varphi}\fil{A}}$ est fermé dans $\RealNP{\fil{A}}$, on en déduit que pour tout $(x,t)$, $f(x,t)\in \RealNP{\tr_{\varphi}\fil{A}}$, il existe donc $g\colon (\Real{\Delta^n}\times X)\otimes\RealNP{\Delta^{\varphi}}\to \RealNP{\tr_{\varphi}\fil{A}}$ tel que $f$ est l'image de $g$.
\end{proof}

\begin{lemme}\label{LemmeFactorisationLastVertex}
Soient $\fil{A}$ un ensemble simplicial filtré et $\Delta^{\varphi}\in R(P)$. Il existe une factorisation de l'application $\tr_{\varphi}(\lv_P)\colon \tr_{\varphi}(\sd_P\fil{A})\to\tr_{\varphi}\fil{A}$ sous la forme
\begin{equation*}
\tr_{\varphi}(\sd_P\fil{A})\xrightarrow{r}\tr_{\varphi}(\sd_P(\tr_{\varphi}\fil{A}))\xrightarrow{\lv_P}\tr_{\varphi}\fil{A}
\end{equation*}
où $r$ est un rétracte de l'inclusion $\tr_{\varphi}(\sd_P(\tr_{\varphi}\fil{A}))\subset \tr_{\varphi}(\sd_P\fil{A})$ et une équivalence d'homotopie filtrée, et $\lv_P$ est la restriction à $\tr_{\varphi}(\sd_P(\tr_{\varphi}\fil{A}))$ de l'application de dernier sommet filtré appliquée à $\tr_{\varphi}\fil{A}$ :
\begin{equation*}
\lv_P\colon \sd_P(\tr_{\varphi}\fil{A})\to\tr_{\varphi}\fil{A}.
\end{equation*}
\end{lemme}

\begin{proof}
On rappelle que la subdivision filtrée $\sd_P\fil{A}$ est obtenue comme la colimite
\begin{equation*}
\sd_P\fil{A}=\colim_{\sigma\colon \Delta^{\psi}\to\fil{A}}\sd_P(\Delta^{\psi})
\end{equation*}
Ainsi, un simplexe $\tau\colon\Delta^{\mu}\to \sd_P\fil{A}$ est décrit par une paire $(\sigma,[(\tau_0,p_0),\dots,(\tau_n,p_n)])$, où $\sigma\colon \Delta^{\psi}\to \fil{A}$ est un simplexe de $\fil{A}$ et $[(\tau_0,p_0),\dots,(\tau_n,p_n)]$ est un simplexe de $\sd_P(\Delta^{\psi})$. On a dans ce cas $\Delta^{\mu}=[p_0,\dots,p_n]$. Un tel simplexe $\tau$ est dans la $\varphi$-troncature si et seulement si il existe $\tau'\colon \Delta^{\mu'}\to\sd_P\fil{A}$ tel que $\Delta^{\mu}\subseteq\Delta^{\mu'}$ et $\tau$ est égal à la composition
\begin{equation*}
\Delta^{\mu}\to\Delta^{\mu'}\xrightarrow{\tau'}\sd_P\fil{A}.
\end{equation*}
Dans ce cas, on doit avoir 
\begin{equation*}
\tau'=(\sigma',[(\tau'_0,p'_0),\dots,(\tau'_m,p'_m)]),
\end{equation*}
avec $\sigma'\colon \Delta^{\psi'}\to\fil{A}$, tel que $\Delta^{\varphi}\subseteq\Delta^{\bar{\psi'}}$. Finalement, on voit qu'il suffit de définir $r_{\sigma}\colon \tr_{\varphi}(\sd_P(\Delta^{\psi}))\to\tr_{\varphi}(\sd_P(\tr_{\varphi}(\Delta^{\psi})))$ pour $\sigma\colon \Delta^{\psi}\to \fil{A}$ avec $\Delta^{\varphi}\subseteq\Delta^{\bar{\psi}}$.
On pose :
\begin{align*}
r_{\sigma}\colon \tr_{\varphi}(\sd_P(\Delta^{\psi}))&\to\tr_{\varphi}(\sd_P(\tr_{\varphi}(\Delta^{\psi})))\\
[(\tau_0,p_0),\dots,(\tau_n,p_n)]&\mapsto [(\tau_0\cap\tr_{\varphi}(\Delta^{\psi}),p_0),\dots,(\tau_n\cap\tr_{\varphi}(\Delta^{\psi}),p_n)].
\end{align*}
Montrons qu'elle est bien définie. Si $[(\tau_0,p_0),\dots,(\tau_n,p_n)]$ est un simplexe de $\tr_{\varphi}(\sd_P(\Delta^{\psi}))$, on doit avoir $p_0,\dots,p_n\in \Delta^{\varphi}$ et $\Delta^{\varphi}\subseteq \psi(\tau_0)$. Comme par hypothèse, $\Delta^{\varphi}\subseteq\Delta^{\bar{\psi}}$, on a $\Delta^{\varphi}\subseteq\psi(\tau_0\cap \tr_{\varphi}(\Delta^{\psi}))$. On en déduit que $[(\tau_0\cap\tr_{\varphi}(\Delta^{\psi}),p_0),\dots,(\tau_n\cap\tr_{\varphi}(\Delta^{\psi}),p_n)]$ est bien un simplexe de $\tr_{\varphi}(\sd_P(\tr_{\varphi}(\Delta^{\psi})))$. De plus, $r_{\sigma}$ est clairement compatible aux faces et aux dégénérescences, donc c'est une application simpliciale. Les applications $r_{\sigma}$ se recollent en une application
\begin{equation*}
r\colon \tr_{\varphi}(\sd_P\fil{A})\to\tr_{\varphi}(\sd_P(\tr_{\varphi}\fil{A})).
\end{equation*}
Par construction cette application est une section de l'inclusion 
\begin{equation*}
i\colon\tr_{\varphi}(\sd_P(\tr_{\varphi}\fil{A}))\to \tr_{\varphi}(\sd_P\fil{A}).
\end{equation*}
De plus, comme $\lv_P(\tau_0\cap\tr_{\varphi}(\Delta^{\psi}),p_0)=\lv_P(\tau_0,p_0)$, on a la factorisation voulue. Il reste à montrer que $r$ est une équivalence d'homotopie filtrée. Soit $\sigma\colon \Delta^{\psi}\to\fil{A}$ un simplexe avec $\Delta^{\varphi}\subseteq\Delta^{\bar{\psi}}$, on définit l'homotopie
\begin{align*}
H_{\sigma}\colon \Delta^1\otimes\tr_{\varphi}(\sd_P(\Delta^{\psi}))&\to \tr_{\varphi}(\sd_P(\Delta^{\psi}))\\
([0,\dots,0,1,\dots,1],[(\tau_0,p_0),\dots,(\tau_n,p_n)])&\mapsto [(\tau_0\cap\tr_{\varphi}(\Delta^{\psi}),p_0),\dots,(\tau_k\cap\tr_{\varphi}(\Delta^{\psi}),p_k),(\tau_{k+1},p_{k+1}),\dots,(\tau_n,p_n)]
\end{align*}
où le premier $1$ apparait en $k+1$-ème position. $H_{\sigma}$ fournit une homotopie entre $i\circ r_{\sigma}$ et $\Id$, et les $H_{\sigma}$ se recollent pour fournir une homotopie entre $i\circ r$ et $\Id$.
\end{proof}

\begin{lemme}\label{LemmeTroncatureSubdivisionEgalProduit}
Soient $\fil{A}$ un ensemble simplicial filtré et $\Delta^{\varphi}\in R(P)$ un simplexe non dégénéré de $N(P)$. Il existe un isomorphisme naturel
\begin{equation*}
\tr_{\varphi}(\sd_P(\tr_{\varphi}\fil{A}))\simeq K\otimes\Delta^{\varphi},
\end{equation*}
où $K$ est l'ensemble simplicial $K=\sd(\varphi_A)^{-1}([\Delta^{\varphi}])$.
\end{lemme}

\begin{proof}
Soient $\Delta^{\varphi}=[p_0,\dots,p_n]$ un simplexe non dégénéré de $R(P)$, et $\fil{A}$ un ensemble simplicial filtré. On définit l'application suivante :
\begin{align*}
\mu\colon\Delta^{\varphi}&\mapsto \sd_P(N(P))\\
p_i&\mapsto (\Delta^{\varphi},p_i)
\end{align*}
On a alors le diagramme commutatif suivant dans la catégorie $\sS$. 
\begin{equation}\label{DiagrammeSubdivisonProduit}
\begin{tikzcd}
\tr_{\varphi}(\sd_P(\tr_{\varphi}\fil{A}))
\arrow{d}
\arrow{r}
&\sd_P\fil{A}
\arrow{d}{\sd_P(\varphi_A)}
\arrow{r}
&\sd(A)
\arrow{d}{\sd(\varphi_A)}
\\
\Delta^{\varphi}
\arrow{r}{\mu}
&\sd_P(N(P))
\arrow{r}
&\sd(N(P))
\end{tikzcd}
\end{equation}
Par définition de la subdivision filtrée $\sd_P$, le carré de droite est cocartésien. Montrons que le carré de gauche l'est aussi. Notons $X$ le produit fibré
\begin{equation*}
\begin{tikzcd}
X
\arrow{d}
\arrow{r}
&\sd_P\fil{A}
\arrow{d}
\\
\Delta^{\varphi}
\arrow{r}{\mu}
&\sd_P(N(P))
\end{tikzcd}
\end{equation*}
Comme le morphisme
\begin{equation*}
\tr_{\varphi}(\sd_P(\tr_{\varphi}\fil{A}))\to\sd_P\fil{A}
\end{equation*}
est une inclusion, il suffit de montrer que le morphisme canonique
\begin{equation*}
\tr_{\varphi}(\sd_P(\tr_{\varphi}\fil{A}))\to X
\end{equation*}
est une surjection. Considérons un simplexe de $X$. Un tel simplexe est de la forme
\begin{equation*}
((\sigma,[(\sigma_0,q_0),\dots,(\sigma_m,q_m)]),\tau)
\end{equation*}
avec 
\begin{itemize}
\item $\sigma\colon \Delta^{\psi}\to \fil{A}$ est un simplexe filtré
\item $\sigma_0\subseteq\dots\subseteq\sigma_m$ sont des faces de $\Delta^{\psi}$
\item $q_0,\dots,q_m\in\psi(\sigma_0)$
\item $\tau\colon \Delta^l\to\Delta^{\varphi}$ est un simplexe (non-filtré) de $\Delta^{\varphi}$.
\item  $ \mu\circ\tau=\sd_P(\varphi_A)(\sigma,[(\sigma_0,q_0),\dots,(\sigma_m,q_m)])
$
\end{itemize}
Cette dernière condition implique que
\begin{itemize}
\item $\Delta^{\psi}=(\Delta^l,\psi)$, en particulier, la donnée de $\tau$ est redondante.
\item $\sd_P(\varphi_A)(\sigma,[(\sigma_0,q_0),\dots,(\sigma_m,q_m)])=[(\Delta^{\varphi},q_0),\dots,(\Delta^{\varphi},q_m)]$. En particulier, pour tout $0\leq i\leq m$, $\Delta^{\bar{\psi}}=\psi(\sigma_i)=\Delta^{\varphi}$.
\end{itemize}
Finalement, on a qu'un simplexe de la somme amalgamée est un simplexe de $\sd_P\fil{A}$
\begin{equation*}
(\sigma,[(\sigma_0,q_0),\dots,(\sigma_m,q_m)])
\end{equation*}
tel que
\begin{itemize}
\item $\sigma\colon \Delta^{\psi}\to \fil{A}$ avec $\Delta^{\bar{\psi}}=\Delta^{\varphi}$
\item $\sigma_0\colon \Delta^{\psi'}\to\Delta^{\psi}$ avec $\Delta^{\bar{\psi'}}=\Delta^{\varphi}$.
\end{itemize}
Ces conditions sont vérifiées par les simplexes de $\tr_{\varphi}(\sd_P(\tr_{\varphi}\fil{A}))$. On en déduit que le carré de gauche du diagramme \ref{DiagrammeSubdivisonProduit} est cocartésien, et donc que le carré extérieur est cocartésien. Or la composition
\begin{equation*}
\Delta^{\varphi}\xrightarrow{\mu}\sd_P(N(P))\to\sd(N(P))
\end{equation*}
a pour image le sous ensemble simplicial de dimension $0$, engendré par $[\Delta^{\varphi}]\in\sd(N(P))_0$. On en déduit que 
\begin{equation*}
\tr_{\varphi}(\sd_P(\tr_{\varphi}\fil{A}))\simeq\sd(\varphi_A)^{-1}([\Delta^{\varphi}])\otimes\Delta^{\varphi}.
\end{equation*}
\end{proof}

On rappelle que si $\Delta^{\varphi}=(\Delta^n,\varphi)\in R(P)$ est un simplexe filtré, son intérieur $\Real{\Int(\Delta^{\varphi})}$ est le sous espace topologique 
\begin{equation*}
\Int(\Real{\Delta^{\varphi}})=\{(t_0,\dots,t_n)\ |\ \sum t_i=1,\  0<t_i\leq 1 \forall i\}\subseteq\Real{\Delta^{\varphi}}
\end{equation*}

\begin{lemme}\label{HomeomorphismeTroncatureSubdivision}
Soient $\fil{A}$ un ensemble simplicial filtré et $\Delta^{\varphi}$ un simplexe de $N(P)$ non dégénéré. Il existe une application continue filtrée surjective
\begin{equation*}
g\colon \RealNP{\tr_{\varphi}(\sd_P(\tr_{\varphi}\fil{A}))}\to\RealNP{\tr_{\varphi}\fil{A}}
\end{equation*}
homotope au sens filtré à (la réalisation de) l'application de dernier sommet filtré $\lv_P$ et
telle que sa restriction au-dessus de l'intérieur de $\Delta^{\varphi}$ 
\begin{equation*}
\begin{tikzcd}
\Int(\RealNP{\tr_{\varphi}(\sd_P(\tr_{\varphi}\fil{A}))})
\arrow{dr}
\arrow{rr}{g_{\Int}}
&&\Int(\RealNP{\tr_{\varphi}\fil{A}})
\arrow{dl}
\\
&\Int\Real{\Delta^{\varphi}}
\end{tikzcd}
\end{equation*}
est un homéomorphisme filtré. 
\end{lemme}

\begin{remarque}
Un tel homéomorphisme ne peut pas être naturel en $\fil{A}$, car pour $\fil{A}$ trivialement filtré (ou pour $\Delta^{\varphi}$ de dimension $0$), $g$ fournit un homéomorphisme (non filtré)
$\Real{\sd(A)}\to\Real{A}$, et on sait qu'un tel homéomorphisme ne peut pas être naturel (voir par exemple \cite[Section 4.6]{CellularStructure}).
\end{remarque}

La preuve de ce lemme est une adaptation de la preuve de \cite[Theorem 4.6.4]{CellularStructure}. De façon à adapter cette stratégie de preuve, nous aurons besoin de la définition ad-hoc suivante ainsi que du lemme \ref{LemmeAdHocProjection}.

\begin{defin}
Soient $\Delta^{\psi}=(\Delta^m,\psi)$ un simplexe filtré, et $p\in \psi(\Delta^{m})$ un élément de $p$. En numérotant de $0$ à $m$ les sommets de $\Delta^m$, on note $I_p=\{i\ |\ \psi(i)=p\}$. On définit la projection sur la $p$-ème strate dans $\Delta^{\psi}$ comme l'application non filtrée (partiellement définie) suivante
\begin{align*}
\pr_p^{\psi}\colon \Real{\Delta^{\psi}}&\to\Real{\Delta^{\psi}}\\
(t_0,\dots,t_m)&\mapsto \frac{1}{\sum\limits_{i\in I_p}{t_i}}\sum\limits_{i\in I_p}(0,\dots,0,t_i,0,\dots,0)
\end{align*}
Cette application n'est bien définie que s'il existe $i\in I_p$ tel que $t_i\not =0$.
\end{defin}

\begin{lemme}\label{LemmeAdHocProjection}
Soient $\alpha\colon \Delta^{\psi}\to\Delta^{\psi'}$ une application simpliciale filtrée et $p\in P$. Alors les compositions $\pr^{\psi'}_p\circ \Real{\alpha}$ et $\Real{\alpha}\circ\pr^{\psi}_p$ sont égales dès qu'elles sont définies, et elles sont définies sur le même sous espace de $\Real{\Delta^{\psi}}$.
\end{lemme}

\begin{proof}
Soient $\alpha\colon \Delta^{\psi}\to\Delta^{\psi'}$ et $p\in P$.
Soit $(t_0,\dots,t_m)\in \Real{\Delta^{\psi}}$, on a $\Real{\alpha}(t_0,\dots,t_m)=(t'_0,\dots,t'_{m'})$ avec $t'_{j}=\sum_{i\in \alpha^{-1}(j)}t_i$. Comme $\alpha$ est filtré on doit avoir $\alpha^{-1}(I'_p)=I_p$. En particulier, on a $t_i=0$, pour tout $i\in I_p$, si et seulement si $t'_j=0$ pour tout $j\in I'_p$, et dans ce cas aucune des deux compositions n'est définie. Sinon, les compositions sont bien définies, et on calcule
\begin{align*}
\Real{\alpha}\circ\pr^{\psi}_p(t_0,\dots,t_m)&=\Real{\alpha}\left(\frac{1}{\sum\limits_{i\in I_p}{t_i}}\sum\limits_{i\in I_p}(0,\dots,0,t_i,0,\dots,0)\right)\\
&=\frac{1}{\sum\limits_{i\in I_p}{t_i}}\sum\limits_{j\in I'_p}(0,\dots,0,t'_j,0,\dots,0)\\
&=\frac{1}{\sum\limits_{j\in I'_p}{t'_j}}\sum\limits_{j\in I'_p}(0,\dots,0,t'_j,0,\dots,0)\\
&=\pr^{\psi'}_p(t'_0,\dots,t'_m)\\
&=\pr^{\psi'}_p\circ\Real{\alpha}(t_0,\dots,t_m)
\end{align*} 
\end{proof}

\begin{proof}[Démonstration du lemme \ref{HomeomorphismeTroncatureSubdivision}]
On va définir $g$ simplexe par simplexe. Soit $\sigma\colon \Delta^{\psi}\to \fil{A}$ un simplexe filtré, avec $\Delta^{\bar{\psi}}=\Delta^{\varphi}$. Un point dans la réalisation $\Real\tr_{\varphi}{\sd_P(\Delta^{\psi})}$ est décrit (pas de façon unique) par
la donnée de 
\begin{itemize}
\item un simplexe non dégénéré $[(\mu_0,q_0),\dots,(\mu_m,q_m)]\in \sd_P(\Delta^{\varphi})$, avec $\Delta^{\varphi}=\psi(\mu_0)$, et $m=\dim(\Delta^{\psi})$, (l'hypothèse de non dégénérescence implique que $(\mu_i,q_i)\not=(\mu_{i+1},q_{i+1})$ pour tout $0\leq i\leq m-1$, et la maximalité de la dimension implique que pour tout $0\leq i\leq m-1$, ou bien $\mu_i=\mu_{i+1}$ ou $q_i=q_{i+1}$),
\item un point $(u_0,\dots,u_m)\in \Real{[q_0,\dots,q_m]}$, c'est à dire un $m+1$-uplet de réels vérifiant $0\leq u_i\leq 1$ et $\sum u_i=1$.
\end{itemize}
Par ailleurs, dans la preuve de \cite[Theorem 4.6.4]{CellularStructure}, Fritsch et Piccinini construisent une application (non-filtrée) 
\begin{equation*}
h_{\sigma}\colon \Real{\sd(\Delta^{\psi})}\to\Real{\Delta^{\psi}}
\end{equation*}
définie par
\begin{align*}
h_{\sigma}\colon \Real{\sd(\Delta^{\psi})}&\to\Real{\Delta^{\psi}}\\
([\mu_0,\dots,\mu_m],(u_0,\dots,u_m))&\mapsto \sum\limits_{k=0}^{m}\sum\limits_{j=0}^k u_ku_jb_{k,k}+\sum\limits_{k=0}^m\sum\limits_{j=0}^{k-1}u_ku_jb_{k,j}
\end{align*}
où les $b_{k,j}$ sont des barycentres de faces dans $\Delta^{\psi}$ détérminés par les $\mu_i$.
(Ces faces ne sont pas nécessairement égales à l'un des $\mu_i$). 
On remarque que si $\mu_0$ est une face de $\Delta^{\psi}$ vérifiant $\psi(\mu_0)=\Delta^{\varphi}$, alors pour tout $p\in \Delta^{\varphi}$, $\pr_p(b_{k,j})$ est bien défini. Ceci provient de la construction des $b_{k,j}$ et du fait que si $\Delta^{\epsilon}$ est un simplexe filtré avec $\Delta^{\bar{\epsilon}}=\Delta^{\varphi}$, tout $\Delta^{\epsilon'}$ qui est une dégénérescence de $\Delta^{\epsilon}$ vérifie $\Delta^{\bar{\epsilon'}}=\Delta^{\varphi}$. Ainsi, pour $p\in \Delta^{\varphi}$, on a les applications (non filtrées) bien définies suivantes :
\begin{align*}
h_{\sigma}^p\colon \Real{\tr_{\varphi}(\sd_P(\Delta^{\psi}))}&\to\Real{\Delta^{\psi}}\\
([(\mu_0,q_0),\dots,(\mu_m,q_m)],(u_0,\dots,u_m))&\mapsto \sum\limits_{k=0}^{m}\sum\limits_{j=0}^k u_ku_j\pr_p(b_{k,k})+\sum\limits_{k=0}^m\sum\limits_{j=0}^{k-1}u_ku_j\pr_p(b_{k,j})
\end{align*}
Pour $[(\mu_0,q_0),\dots,(\mu_m,q_m)],(u_0,\dots,u_m))$ un point de $\Real{\tr_{\varphi}(\sd_P(\Delta^{\psi}))}$, et $p\in \Delta^{\varphi}$, on pose $u_p=\sum\limits_{q_i=p}u_i$. On définit ainsi
\begin{align*}
g_{\sigma}\colon \RealNP{\tr_{\varphi}\sd_P(\Delta^{\psi})}&\to \RealNP{\Delta^{\psi}}\\
([(\mu_0,q_0),\dots,(\mu_m,q_m)],(u_0,\dots,u_m))&\mapsto \sum_{p}u_ph_{\sigma}^p([(\mu_0,q_0),\dots,(\mu_m,q_m)],(u_0,\dots,u_m))
\end{align*}
On remarque que par construction, l'application $g_{\sigma}$ est filtrée. En effet, l'image par la filtration d'un point $([(\mu_0,q_0),\dots,(\mu_m,q_m)],(u_0,\dots,u_m))\in \Real{\sd_P(\Delta^{\psi})}$ est 
\begin{equation*}
([p_0,\dots,p_n],(u_{p_0},\dots,u_{p_n}))=\Real{\psi}\left(\sum_{p\in P}u_ph_{\sigma}^p([(\mu_0,q_0),\dots,(\mu_m,q_m)],(u_0,\dots,u_m))\right)
\end{equation*}
car $\psi(h_\sigma^p)=p$, par construction. Pour voir que les $g_{\sigma}$ sont compatibles aux faces et aux dégénérescences, il suffit de vérifier que pour toute application simpliciale filtrée $\alpha\colon \Delta^{\psi}\to\Delta^{\psi'}$, et tout $[(\mu_0,q_0),\dots,(\mu_m,q_m)]\in \tr_{\varphi}(\sd_P(\Delta^{\psi}))$, l'application $\Real{\alpha}$ envoie les $\pr_p(b_{k,j})$ associé aux $\mu_i$ sur les $\pr_p(b'_{k,j})$ associés aux $\alpha(\mu_i)$. Dans la preuve de \cite[Theorem 4.6.4]{CellularStructure} , les auteurs prouvent que $\Real{\alpha}$ envoie les $b_{k,j}$ sur les $b'_{k,j}$, par construction des $b_{k,j}$, et le résultat voulu est donc une conséquence du lemme \ref{LemmeAdHocProjection}. On peut donc recoller les $g_{\sigma}$ pour obtenir une application filtrée
\begin{equation*}
g\colon \RealNP{\tr_{\varphi}(\sd_P(\tr_{\varphi}\fil{A}))}\to \RealNP{\fil{A}}
\end{equation*}
Montrons qu'elle vérifie les hypothèses du lemme \ref{HomeomorphismeTroncatureSubdivision}. 
Pour $\sigma\colon \Delta^{\psi}\to\fil{A}$ comme plus tôt, on définit
\begin{align*}
H_{\sigma}\colon \Delta^1\otimes\Real{\sd_P(\Delta^{\psi})}&\to\Real{\Delta^{\psi}}\\
(t,([\mu,q],u))&\mapsto t\RealNP{\lv_P}([\mu,q],u)+(1-t)g_{\sigma}([\mu,q],u)
\end{align*}
Comme les applications $\RealNP{\lv_P}$ et $g_{\sigma}$ sont compatibles aux faces et aux dégénérescences ($\RealNP{\lv_P}$ est la réalisation d'une application simpliciale), $H_{\sigma}$ l'est aussi, et on peut donc recoller les $H_{\sigma}$ pour obtenir une homotopie
\begin{equation*}
H\colon \Delta^1\otimes\RealNP{\tr_{\varphi}(\sd_P(\tr_{\varphi}\fil{A}))}\to\RealNP{\fil{A}}.
\end{equation*}
Il reste à montrer que pour tout simplexe filtré non dégénéré $\sigma\colon \Delta^{\psi}\to \fil{A}$, avec $\Delta^{\bar{\psi}}=\Delta^{\varphi}$, la restriction de $g_{\sigma}$ aux intérieurs
\begin{equation}\label{EquationHomeoInterieur}
g_{\sigma}\colon\mathring{\RealNP{\tr_{\varphi}(\sd_P(\Delta^{\psi}))}}\to\mathring{\RealNP{\Delta^{\psi}}}
\end{equation}
est bijective.
Par construction, un point 
\begin{equation*}
([\mu,q],u)=([(\mu_0,q_0),\dots,(\mu_m,q_m)],(u_0,\dots,u_m))\in \RealNP{\sd_P(\Delta^{\psi})}
\end{equation*}
 est dans l'intérieur de $\RealNP{\tr_{\varphi}(\sd_P(\Delta^{\psi}))}$ s'il existe $0\leq i\leq m$ tel que $u_i\not =0$ et $\mu_i=\Delta^{\psi}$, et pour tout $p\in \Delta^{\varphi}$, $u_p\geq 0$.
Soit $([\mu,q],u)$ un tel point. Notons $g_{\sigma}([\mu,q],u)=(t_0,\dots,t_m)\in \Real{\Delta^{\psi}}$ et montrons que $t_j\geq 0$ pour tout $0\leq j\leq m$. Fixons un tel $j$. On doit avoir $b_{i,i}$ est le barycentre de $\Real{\Delta^{\psi}}$ car $\mu_i=\Delta^{\psi}$ et $\sigma$ est non dégénéré. On a donc $(b_{i,i})_j\geq 0$, et en notant $p=\psi(e_i)$ on a aussi $(\pr_p(b_{i,i}))_j\geq 0$, mais alors $t_j=(g_{\sigma}([\mu,q],u))_j=\left(u_p h^p_{\sigma}([\mu,q],u)\right)_j\geq u_pu_ju_j(b_{i,i})_j>0$. Ainsi, $g_{\sigma}$ envoie des points intérieurs sur des points intérieurs, et la restriction \ref{EquationHomeoInterieur} est bien définie.

Montrons maintenant que l'application $g_{\sigma}$ est injective. Soient 
\begin{align*}
([\mu,q],u)&=([(\mu_0,q_0),\dots,(\mu_m,q_m)],(u_0,\dots,u_m))\\
([\mu',q'],u')&=([(\mu'_0,q'_0),\dots,(\mu'_{m},q'_{m})],(u'_0,\dots,u_{m}))
\end{align*}
 deux points de $\RealNP{\tr_{\varphi}(\sd_P(\Delta^{\psi}))}$ tels que 
$g_{\sigma}([\mu,q],u)=g_{\sigma}([\mu',q'],u')$. Alors, on a 
\begin{equation*}
\sum_{p\in P}u_ph^p_{\sigma}([\mu,q],u)=\sum_{p\in P}u'_ph^p_{\sigma}([\mu',q'],u'))
\end{equation*}
En raisonnant sur la filtration, ceci implique immédiatement que $h^p_{\sigma}([\mu,q],u)=h_{\sigma}^p([\mu',q']u')$ et que $u_p=u'_p$ pour tout $p\in P$. Si $\dim{\Delta^{\psi}}=\dim{\Delta^{\varphi}}=n$, on a nécessairement $[(\mu_0,q_0),\dots,(\mu_n,q_n)]=[(\Delta^{\psi},p_0),\dots,(\Delta^{\psi},p_n)]$, et $g_{\sigma}([\mu,q],u)$ est entièrement déterminé par les $u_i=u_{p_i}=u'_{p_i}=u'_i$, ce qui implique que $([\mu,q],u)=([\mu',q'],u')$.
Sinon, notons $(\tau,v)=([\tau_0,\dots,\tau_{m-n}],(v_0,\dots,v_{m-n})$ le point de $\Real{\sd(\Delta^{\psi}}$ obtenu à partir de $([(\mu_0,q_0),\dots,(\mu_m,q_m)],(u_0,\dots,u_m)]$ en oubliant les $q_i$. On obtient les $\tau_i$ en supprimant les répétitions dans les $\mu_i$ (il y en a nécessairement $\dim(\Delta^{\psi})-\dim(\Delta^{\varphi})$), et les $v_i$ sont obtenus en regroupant les $u_i$ de façon correspondante. De même pour $([\mu',q'],u')$ on définit $(\tau',v')\in \Real{\sd(\Delta^{\psi})}$. Avec ces notations, on peut réécrire les $h^p_{\sigma}$ (sans changer leur définition) comme 
\begin{equation*}
h^p_{\sigma}([\mu,q],u)=\sum\limits_{k=0}^{m-n}\sum\limits_{j=0}^k v_kv_j\pr_p(b_{k,k})+\sum\limits_{k=0}^{m-n}\sum\limits_{j=0}^{k-1}v_kv_j\pr_p(b_{k,j})=h^p_{\sigma}(\tau,v)
\end{equation*}
où les $b_{k,j}$ sont déterminés à partir des $\tau_j$.
Nous allons montrer par récurrence qu'on a pour tout $0\leq i\leq m$, $u_i=u'_i$ et, si $u_i\not =0$, $\mu_i=\mu'_i$. Par construction on a $\tau_{m-n}=\Delta^{\psi}$ et $\tau_{m-n-1}=d_i\Delta^{\psi}$ pour un certain $0\leq i\leq m$. Ceci implique que le seul $b_{k,j}$ avec $(b_{k,j})_i>0$ est $b_{m-n,m-n}$. En notant $p=\psi(e_i)$, ceci implique que
\begin{equation*}
\left(h^p_{\sigma}(\tau,v)\right)_i=v_{m-n}(\pr_p(b_{m-n,m-n}))_i=\frac{v_{m-n}}{\dim(\psi^{-1}(p))+1},
\end{equation*} 
et donc,
\begin{equation*}
\left(g_{\sigma}([\mu,q],u)\right)_i=u_p\left(h^p_{\sigma}(\tau,v)\right)_i=u_p\frac{v_{m-n}}{\dim(\psi^{-1}(p))+1},
\end{equation*} 
d'autre part, comme $e_i\in \tau'_{m-n}=\Delta^{\psi}$, on doit aussi avoir
\begin{equation*}
\left(h^p_{\sigma}(\tau',v')\right)_i\geq \frac{v'_{m-n}}{\dim(\psi^{-1}(p))+1}
\end{equation*}
On a donc
\begin{equation*}
\left(g_{\sigma}([\mu',q'],u')\right)_i=u'_p\left(h^p_{\sigma}(\tau',v')\right)_i=u'_p\frac{v'_{m-n}}{\dim(\psi^{-1}(p))+1},
\end{equation*} 
Et donc, on obtient
\begin{equation*}
u_p\frac{v_{m-n}}{\dim(\psi^{-1}(p))+1}\geq u'_p\frac{v'_{m-n}}{\dim(\psi^{-1}(p))+1}
\end{equation*}
On a vu précédemment que $u_p=u'_p$, et comme $([\mu,q],u)$ est un point intérieur, on a $u_p>0$, on peut donc simplifier l'inégalité en $v_{m-n}\geq v'_{m-n}$. Puis, par symétrie de l'argument (on peut le répéter pour $\tau'_{m-n-1}=d_{i'}\Delta^{\psi}$), on a l'égalité $v_{m-n}=v'_{m-n}$.
Or, par construction, on a
\begin{equation*}
v_{m-n}=\sum\limits_{\mu_j=\Delta^{\psi}}u_j=\sum\limits_{j=m-l}^mu_j
\end{equation*}
pour un certain $l\geq 0$. On a aussi, $q_j<q_{j+1}$ pour tout $m-l\leq j<m$. Pour $m-l<j\leq m$, ceci implique que $u_j=u_{q_j}=u'_{q_j}=u'_j$. D'autre part, on a 
\begin{equation*}
u_{m-l}=v_{m-n}-\sum\limits_{j=m-l+1}^mu_j=u'_{m-l}
\end{equation*}
Finalement, on a trouvé $l\geq 0$ tel que :
\begin{itemize}
\item pour tout $m-l\leq j\leq m$, $u_j=u_j'$,
\item pour tout $m-l\leq j\leq m$, si $u_j>0$, $\mu_j=\mu_j'$,
\item $m-l-1<0$, ou ($\mu_{m-l-1}\not=\mu_{m-l}$ et  $\mu'_{m-l-1}\not=\mu'_{m-l}$).
\end{itemize}
Ceci complète l'initialisation de la récurrence. Soit $l\geq 0$ vérifiant les conditions si dessus avec $m-l-1\geq 0$, on a alors un $s\in\{0,\dots,m-n\}$, correspondant à $m-l$ tel que pour tout $s\leq j\leq m-n$, $\tau_j=\tau'_j$ et $v_s=v'_s$. On considère alors la combinaison linéaire
\begin{equation*}
(\widetilde{t}_0,\dots,\widetilde{t}_m)=\sum\limits_p u_p\left(\sum\limits_{k=0}^{s-1}\sum\limits_{j=0}^{k}v_kv_j\pr_p(b_{k,k})+\sum\limits_{j=0}^{s-1}\sum\limits_{k=j+1}^{m-n}v_jv_k\pr_p(b_{k,j})\right)\in \R^{m+1}
\end{equation*}
On remarque que
\begin{equation*}
(\widetilde{t}_0,\dots,\widetilde{t}_m)=g_{\sigma}([\mu,q],u)-\sum\limits_p u_p\left(\sum\limits_{k=s}^{m-n}\sum\limits_{j=0}^{k}v_kv_j\pr_p\left(b_{k,k}\right)+\sum\limits_{j=s}^{m-n}\sum\limits_{k=j+1}^{m-n}v_jv_k\pr_p\left(b_{k,j}\right)\right),
\end{equation*}
et que
\begin{equation*}
\sum_{j=0}^{k}v_j=1-\sum_{j=k+1}^{m-n}v_j.
\end{equation*}
Par l'hypothèse de récurrence, on a donc
\begin{equation*}
(\widetilde{t}_0,\dots,\widetilde{t}_m)=(\widetilde{t}'_0,\dots,\widetilde{t}'_m).
\end{equation*}
Supposons que $\tau_{s-1}\not=\tau'_{s-1}$. Comme $\dim{\tau_{s-1}}=\dim{\tau'_{s-1}}$, il existe $e_i\in \tau_{s-1}$ tel que $e_i\not\in \tau'_{s-1}$. On fixe un tel $i$, et on note $p=\psi(e_i)$. On a alors
\begin{equation*}
u_p\frac{v_{s-1}}{\dim(\psi^{-1}(p)+1}=u_pv_{s-1}(\pr_p(b_{m-n,s-1}))_i\leq \widetilde{t}_i=\widetilde{t}'_i=0
\end{equation*}
Comme $u_p\not =0$ (car $([\mu,q],u)$ est un point intérieur), on en déduit que $v_{s-1}=0$. Par symétrie de l'argument, on a aussi $v'_{s-1}=0$. On définit alors $l'$ tel que pour tout $m-l'\leq j\leq m-l-1$, $\mu_j=\tau_{s-1}$ et $m-l'-1<0$ ou $\mu_{m-l'-1}\not=\tau_{s-1}$. Alors, on a
\begin{equation*}
v_{s-1}=\sum\limits_{j=m-l'}^{m-l-1}u_j=0.
\end{equation*} 
On en déduit que pour tout $m-l'\leq j\leq m-l-1$, $u_j=u_j'=0$, ce qui complète le pas de récurrence dans le cas où $\tau_{s-1}\not=\tau'_{s-1}$. Supposons maintenant que $\tau'_{s-1}=\tau_{s-1}$. On considère la combinaison linéaire
\begin{equation*}
(\widehat{t}_0,\dots,\widehat{t}_m)=\sum\limits_pu_p\left(
\sum_{j=0}^{s-1}v_j\pr_p(b_{s-1,s-1})+\sum\limits_{k=s}^{m-n}v_k\pr_p(b_{k,s-1})
\right)\in \R^{m+1}
\end{equation*}
Comme $\tau_j=\tau'_j$ pour tout $s-1\leq j\leq m-n$, on a aussi que $b_{k,j}=b'_{k,j}$ pour tout $s-1\leq j\leq k\leq m-n$. On a donc l'égalité
\begin{equation*}
(\widehat{t}_0,\dots,\widehat{t}_m)=(\widehat{t}'_0,\dots,\widehat{t}'_m).
\end{equation*}
De plus, si $s-1>0$, on doit avoir $e_{i}\in \tau_{s-1}$ tel que $e_{i}\not\in \tau_{s-2}$. On fixe un tel $i$, et on note $p=\psi(e_i)$. On a alors
\begin{equation*}
v_{s-1}\widehat{t}_i=\widetilde{t}_i=\widetilde{t}'_i\geq v'_{s-1}\widehat{t}'_i=v'_{s-1}\widehat{t}_i.
\end{equation*}
Comme $\widehat{t}_i\geq u_pv_{m-n}\left(\pr_p(b_{m-n,s-1})\right)_i=u_p\frac{v_{m-n}}{\dim(\tau_{s-1}+1)}>0$, on en déduit que $v_{s-1}\geq v'_{s-1}$. Une fois encore, par symétrie, on obtient $v_{s-1}=v'_{s-1}$. Par les mêmes arguments que précédemment, on en déduit l'existence d'un $l'>l$ tel que pour tout $m-l'\leq j\leq m-l-1$, $\mu_j=\mu'_j$ et $u_j=u'_j$ et tel que $\mu_{m-l'-1}\not=\mu_{m-l-1}$. Ceci conclut le pas de récurrence dans le cas où $\tau_{s-1}=\tau'_{s-1}$ et $s-1>0$. Supposons maintenant que $s-1=0$ et $\tau_{0}=\tau'_0$. On a alors 
\begin{equation*}
v_0(\widehat{t}_0,\dots,\widehat{t}_m)=(\widetilde{t}_0,\dots,\widetilde{t}_m)=(\widetilde{t}'_0,\dots,\widetilde{t}'_m)=v'_0(\widehat{t}'_0,\dots,\widehat{t}'_m)=v'_0(\widehat{t}_0,\dots,\widehat{t}_m)
\end{equation*}
On en déduit que $v_0=v'_0$, et on conclut, comme précédemment que $u_j=u'_j$ pour tout $0\leq j\leq m$ et $\mu_j=\mu'_j$ dès que $u_j>0$, ce qui conclue la preuve de l'injectivité de $g_{\sigma}$ sur l'intérieur de $\Real{\tr_{\varphi}(\sd_P(\Delta^{\psi}))}$.

Finalement, il reste à montrer que la restriction de $g_{\sigma}$ à l'intérieur de $\Real{\tr_{\varphi}(\sd_P(\Delta^{\psi}))}$ est surjective sur l'intérieur de $\Real{\Delta^{\psi}}$ lorsque $\sigma$ est non dégénéré. Comme la restriction de $g_{\sigma}$ est injective, on a que l'image de la restriction de $g_{\sigma}$ 
\begin{equation*}
g_{\sigma}(\mathring{\Real{\tr_{\varphi}(\sd_P(\Delta^{\psi}))}})\subseteq \mathring{\Real{\Delta^{\psi}}}
\end{equation*}
est ouverte dans $\mathring{\Real{\Delta^{\psi}}}$ (voir par exemple \cite[Theorem A.9.6]{CellularStructure}). Mais on a aussi 
\begin{equation*}
g_{\sigma}(\mathring{\Real{\tr_{\varphi}(\sd_P(\Delta^{\psi}))}})= \mathring{\Real{\Delta^{\psi}}}\cap g_{\sigma}(\Real{\tr_{\varphi}(\sd_P(\Delta^{\psi}))})
\end{equation*}
Et l'image est donc fermée. Puisqu'elle est non vide, elle est égale à $\mathring{\Real{\Delta^{\psi}}}$, et la restriction de $g_{\sigma}$ à l'intérieur de $\Real{\tr_{\varphi}(\sd_P(\Delta^{\psi}))}$ est surjective.
\end{proof}

On peut maintenant passer à la preuve de la proposition \ref{LastVertexEquivalenceFaible}.

\begin{proof}[Démonstration de la proposition \ref{LastVertexEquivalenceFaible}]
Soient $\fil{A}$ un ensemble simplicial filtré et $\Delta^{\varphi}\in R(P)$ un simplexe non dégénéré. Montrons que $\lv_P$ induit une équivalence faible 
\begin{equation*}
\Map(\RealNP{\Delta^{\varphi}},\RealNP{\sd_P\fil{A}})\to\Map(\RealNP{\Delta^{\varphi}},\RealNP{\fil{A}})
\end{equation*}
Par le lemme \ref{LemmeMapTroncature}, il suffit de considérer la restriction
\begin{equation*}
\Map(\RealNP{\Delta^{\varphi}},\RealNP{\tr_{\varphi}(\sd_P\fil{A})})\to\Map(\RealNP{\Delta^{\varphi}},\RealNP{\tr_{\varphi}\fil{A}})
\end{equation*}
Puis, par le lemme \ref{LemmeFactorisationLastVertex}, on se ramène à la restriction
\begin{equation*}
\Map(\RealNP{\Delta^{\varphi}},\RealNP{\tr_{\varphi}(\sd_P(\tr_{\varphi}\fil{A}))})
\to\Map(\RealNP{\Delta^{\varphi}},\RealNP{\tr_{\varphi}\fil{A}}).
\end{equation*}
Cette application est homotope à $g_*$, l'application induite par le morphisme $g$, construit au lemme \ref{HomeomorphismeTroncatureSubdivision}. Nous allons montrer que $g_*$ est une équivalence d'homotopie. Par le lemme \ref{LemmeTroncatureSubdivisionEgalProduit}, la source de $g_*$ est isomorphe à 
\begin{equation*}
\Map(\RealNP{\Delta^{\varphi}},\RealNP{K\otimes\Delta^{\varphi}})\simeq\Map(\RealNP{\Delta^{\varphi}},\Real{K}\otimes\RealNP{\Delta^{\varphi}})
\end{equation*}
Pour un certain ensemble simplicial $K$, notons $\Delta^{\varphi}=(\Delta^n,\varphi)$. On remarque qu'on a un isomorphisme
\begin{equation*}
\Map(\RealNP{\Delta^{\varphi}},\Real{K}\otimes\RealNP{\Delta^{\varphi}})\simeq \Map(\Real{\Delta^n},\Real{K}).
\end{equation*}
 En effet, pour toute application filtrée $\Delta^m\otimes\RealNP{\Delta^{\varphi}}\to \Real{K}\otimes\RealNP{\Delta^{\varphi}}$ la composition 
\begin{equation*}
\Delta^m\otimes\RealNP{\Delta^{\varphi}}\to \Real{K}\otimes\RealNP{\Delta^{\varphi}}\xrightarrow{\pr_{\RealNP{\Delta^{\varphi}}}}\RealNP{\Delta^{\varphi}}
\end{equation*}
doit être égale à la projection sur $\RealNP{\Delta^{\varphi}}$. De plus, par le lemme \ref{HomeomorphismeTroncatureSubdivision}, $g$ induit un isomorphisme
\begin{equation*}
\Map(\Int(\Real{\Delta^n}),\Real{K})\simeq \Map(\Int(\RealNP{\Delta^{\varphi}}),\RealNP{\tr_{\varphi}\fil{A}})
\end{equation*}
Finalement, l'inclusion $\Int(\RealNP{\Delta^{\varphi}})\to\RealNP{\Delta^{\varphi}}$ induit des monomorphismes, et on a le diagramme commutatif suivant
\begin{equation*}
\begin{tikzcd}
\Map(\Real{\Delta^n},\Real{K})
\arrow[hookrightarrow]{r}{i_k}
\arrow[swap]{d}{g_*}
&
\Map(\Int(\Real{\Delta^n}),\Real{K})
\arrow{d}{\simeq}
\\
\Map(\RealNP{\Delta^{\varphi}},\RealNP{\tr_{\varphi}\fil{A}})
\arrow[hookrightarrow,swap]{r}{i_A}
&\Map(\Int(\RealNP{\Delta^{\varphi}}),\RealNP{\tr_{\varphi}\fil{A}})
\end{tikzcd}
\end{equation*}
l'isomorphisme de droite est induit par $g$. Pour éviter les confusions entre les différents morphismes induits par $g$, on note $g_*$ le morphisme apparaissant à gauche du diagramme, et $h$ l'isomorphisme apparaissant à droite. Pour montrer que $g_*$ est une équivalence d'homotopie, il suffit de montrer que $i_K$ et $i_A$ sont des équivalences d'homotopies. On remarque qu'il existe un homéomorphisme (non filtré)
\begin{equation*}
\Real{\Delta^n}\simeq \bar{c}(\Real{\partial(\Delta^n)})=\frac{\Real{\partial(\Delta^n)}\times[0,1]}{\Real{\partial(\Delta^n)}\times\{0\}}
\end{equation*}
dont la restriction à l'intérieur de $\Real{\Delta^n}$ :
\begin{equation*}
\Int(\Real{\Delta^n})\simeq c(\Real{\partial(\Delta^n)})=\frac{\Real{\partial(\Delta^n)}\times[0,1[}{\Real{\partial(\Delta^n)}\times\{0\}}
\end{equation*}
est encore un homéomorphisme. En fixant un tel homéomorphisme, on peut identifier les points de $\Real{\Delta^n}$ avec les points $(y,t)\in \bar{c}(\Real{\partial(\Delta^n)})$. Avec ces notations, on définit la rétraction 
\begin{align*}
r_K\colon \Map(\Int(\Real{\Delta^n}),\Real{K})&\to\Map(\Real{\Delta^n},\Real{K})\\
(\sigma\colon \Int(\Real{\Delta^n})\times\Real{\Delta^k}\to\Real{K})&\mapsto \left\{\begin{array}{ccc}
\Real{\Delta^n}\times\Real{\Delta^k}&\to &\Real{K}\\
((y,t),s)&\mapsto &\sigma((y,0),s)
\end{array}\right.
\end{align*}
Et on considère l'homotopie
\begin{align*}
H\colon \Delta^1\times\Map(\Int(\Real{\Delta^n}),\Real{K})&\to\Map(\Int(\Real{\Delta^n}),\Real{K})\\
(\tau\colon \Delta^k\to\Delta^1,\sigma\colon \Int(\Real{\Delta^n})\times\Real{\Delta^k}\to\Real{K})&\mapsto \left\{\begin{array}{ccc}
\Int(\Real{\Delta^n})\times\Real{\Delta^k}&\to &\Real{K}\\
((y,t),s)&\mapsto &\sigma((y,t\Real{\tau}(s)),s)
\end{array}\right.
\end{align*}
où $\Real{\tau}(s)\in \Real{\Delta^1}$ est vu comme un élément de $[0,1]$ via l'identification $\Real{\Delta^1}\simeq [0,1]$. L'application $H$ fournit une homotopie entre $r_K\circ i_K$ et l'identité. D'autre part, on remarque que la restriction 
\begin{equation*}
H_{| \Delta^1\times\Map(\Real{\Delta^n},\Real{K})}\colon \Delta^1\times\Map(\Real{\Delta^n},\Real{K})\to\Map(\Real{\Delta^n},\Real{K})
\end{equation*}
est bien définie et fournit une homotopie entre $i_K\circ r_K$ et l'identité. On en déduit que $i_K$ est une équivalence d'homotopie. On considère maintenant l'homotopie $H_A=h\circ H\circ (\Id_{\Delta^1}\times h^{-1})$ 
\begin{equation*}
H_A\colon \Delta^1\times\Map(\Int(\RealNP{\Delta^{\varphi}}),\RealNP{\tr_{\varphi}\fil{A}})\to \Map(\Int(\RealNP{\Delta^{\varphi}}),\RealNP{\tr_{\varphi}\fil{A}})
\end{equation*}
Par construction de $H$, la restriction de $H_A$ 
\begin{equation*}
\widetilde{H}_A\colon \Delta^1\times\Map(\RealNP{\Delta^{\varphi}},\RealNP{\tr_{\varphi}\fil{A}})\to \Map(\RealNP{\Delta^{\varphi}},\RealNP{\tr_{\varphi}\fil{A}})
\end{equation*}
est bien définie. On calcule alors
\begin{align*}
i_A\circ(\widetilde{H}_A)_{|\{0\}}&=(H_A)_{|\{0\}}\circ i_A\\
&= h\circ i_K\circ r_K\circ h^{-1}\circ i_A\\
&= i_A\circ g_*\circ r_K\circ h^{-1}\circ i_A
\end{align*}
Comme $i_A$ est un monomorphisme, on en déduit que 
$(\widetilde{H}_A)_{|\{0\}}=g_*\circ r_K\circ h^{-1}\circ i_A$. D'autre part, $(\widetilde{H}_A)_{|\{1\}}$ est l'identité, on en déduit que $g_*$ est une équivalence d'homotopie, d'inverse $r_K\circ h^{-1}\circ i_A$. 
\end{proof}

\subsection{Remplacements cofibrants pour les espaces fortement filtrés PL}
\label{SectionRemplacementsCofibrantsPL}
Une conséquence presque immédiate des propositions \ref{PropositionSubdivisionCofibrante} et \ref{LastVertexEquivalenceFaible} est la proposition \ref{PropositionRemplacementCofibrantEfficace} suivante, qui permet de calculer efficacement un remplacement cofibrant pour les espaces fortement filtrés PL. On rappelle qu'on a un morphisme canonique $\varphi\colon \sd_P(N(P))\to\sd(N(P))$ donné par $\varphi(\Delta^{\psi},p)=\Delta^{\psi}$ (voir la section \ref{SectionSubdivisionFiltree}). D'autre part, le théorème \cite[Theorem 4.6.4]{CellularStructure} fournit un homéomorphisme $h_{N(P)}\colon\Real{\sd(N(P))}\to\Real{N(P)}$. 
\begin{prop}\label{PropositionRemplacementCofibrantEfficace}
Soit $\fil{X}$ un espace fortement filtré. On considère le produit fibré
\begin{equation*}
\begin{tikzcd}[column sep = large]
X^{\text{cof}}
\arrow{r}{r_X}
\arrow{d}
&X
\arrow{d}{\varphi_X}
\\
\Real{\sd_P(N(P))}
\arrow[swap]{r}{h_{N(P)}\circ \varphi}
&\Real{N(P)}
\end{tikzcd}
\end{equation*}
et on munit $X^{\text{cof}}$ de la filtration donnée par la composition
\begin{equation*}
\varphi_{X^{\text{cof}}}\colon X^{\text{cof}}\to\Real{\sd_P(N(P))}\xrightarrow{\RealNP{\lv_P}}\Real{N(P)}.
\end{equation*}
S'il existe $\fil{A}\in \sS$ tel que $\fil{X}\simeq \RealNP{\fil{A}}$, alors $\fil{X^{\text{cof}}}$ est cofibrant dans $\Top_{N(P)}$, et faiblement équivalent à $\fil{X}$.
\end{prop}

\begin{remarque}
Attention, l'équivalence faible de la proposition \ref{PropositionRemplacementCofibrantEfficace} n'est pas fournie par $r_X$ (l'application $r_X$ n'est en général pas filtrée). Si l'on veut une équivalence faible explicite entre un remplacement cofibrant de $\fil{X}$ et $\fil{X}$, il faudra travailler directement avec l'ensemble simplicial filtré $\fil{A}$.
\end{remarque}

\begin{proof}
Par les propositions \ref{PropositionSubdivisionCofibrante} et \ref{LastVertexEquivalenceFaible}, l'espace fortement filtré $\RealNP{\sd_P\fil{A}}$ est cofibrant dans $\Top_{N(P)}$ et faiblement équivalent à $\RealNP{\fil{A}}$. Par ailleurs, comme $N(P)$ est un ensemble simplicial non singulier ($N(P)$ provient d'un complexe simplicial) par \cite[Theorem 2.3.2]{SpacesOfPLManifolds}, il existe un homéomorphisme $h^{\varphi_A}\colon \Real{\sd(A)}\to\Real{A}$ tel que le diagramme suivant commute
\begin{equation*}
\begin{tikzcd}[column sep= large]
\Real{\sd(A)}
\arrow{r}{h^{\varphi_A}}
\arrow[swap]{d}{\Real{\sd(\varphi_A)}}
&\Real{A}
\arrow{d}{\Real{\varphi_A}}
\\
\Real{\sd(N(P))}
\arrow[swap]{r}{h_{N(P)}}
&\Real{N(P)}
\end{tikzcd}
\end{equation*}
D'autre part, comme $\Real{-}$ commute avec les limites finies, le carré
\begin{equation*}
\begin{tikzcd}
\Real{\sd_P\fil{A}}
\arrow{d}
\arrow{r}
&\Real{\sd(A)}
\arrow{d}{\Real{\sd(\varphi_A)}}
\\
\Real{\sd_P(N(P))}
\arrow{r}{\Real{\varphi}}
&\Real{\sd(N(P))}
\end{tikzcd}
\end{equation*}
est cartésien. Finalement, en identifiant $\Real{A}$ et $X$, et $\Real{\varphi_A}$ et $\varphi_X$, on obtient le diagramme commutatif suivant
\begin{equation*}
\begin{tikzcd}[column sep = large]
&X^{\text{cof}}
\arrow{rr}
\arrow{dd}
&& X
\arrow{dd}{\varphi_X}
\\
\Real{\sd_P\fil{A}}
\arrow{dd}
\arrow[crossing over]{rr}
\arrow{ur}{\simeq}
&&\Real{\sd(A)}
\arrow{ur}{h^{\varphi_A}}
\\
&\Real{\sd_P{N(P)}}
\arrow[near start]{rr}{h_{N(P)}\circ\Real{\varphi}}
&&\Real{N(P)}
\\
\Real{\sd_P(N(P))}
\arrow{rr}{\Real{\varphi}}
\arrow{ur}{\Id}
&&\Real{\sd(N(P))}
\arrow{ur}{h_{N(P)}}
\arrow[leftarrow, crossing over, near start]{uu}{\Real{\sd(\varphi_A)}}
\end{tikzcd}
\end{equation*}
où les faces avant et arrière sont cartésiennes. On en déduit que le morphisme $\Real{\sd_P\fil{A}}\to X^{\text{cof}}$ est un homéomorphisme. De plus, par construction de $\varphi_{X^{\text{cof}}}$, cet homéomorphisme est filtré, et on en déduit que $\fil{X^{\text{cof}}}$ est un remplacement cofibrant de $X$.
\end{proof}

\subsection{Application aux espaces coniquement stratifiés}
\label{SectionRemplacementsCofibrantsFibre}
Dans la section \ref{ConstructionFibre}, nous avons vu une construction permettant d'obtenir un espace filtré à partir d'un fibré localement trivial. On généralise immédiatement cette construction pour obtenir un espace \textbf{fortement}
filtré. On explicite cette construction dans le cas où $P=\{p_0<p_1\}$. Soit
\begin{equation*}
F\hookrightarrow E\xrightarrow{\pi} B
\end{equation*}
un fibré localement trivial.
Dans la section \ref{ConstructionFibre} on a construit un fibré localement trivial 
\begin{equation*}
\bar{c}(F)\hookrightarrow \bar{c}_{\pi}(E)\xrightarrow{\bar{c}_{\pi}(\pi)}B.
\end{equation*}
On considère la filtration $\varphi_{\pi}\colon \bar{c}_{\pi}(E)\to \Real{N(P)}$ donnée par la composition
\begin{equation*}
\bar{c}_{\pi}(E)\xrightarrow{\pi}[0,1]\xrightarrow{\varphi_{N(P)}} [0,1]\simeq \Real{N(P)}
\end{equation*}
où l'application continue $\varphi_{N(P)}$ est définie comme
\begin{equation*}
\varphi_{N(P)}(t)=\left\{\begin{array}{cl}
\frac{3t}{2} &\text{ si $0\leq t\leq 2/3$}\\
1 &\text{ si $2/3<t$}
\end{array}\right.
\end{equation*}
On a immédiatement que la composée
\begin{equation*}
c_{\pi}(E)\to\Real{N(P)}\to P
\end{equation*}
est égale à la filtration $\varphi_{\pi}$ de la section \ref{ConstructionFibre}. On considère par ailleurs la seconde filtration $\widetilde{\varphi}_{\pi}\colon \bar{c}_{\pi}(E)\to \Real{N(P)}$ définie comme la composition
\begin{equation*}
\bar{c}_{\pi}(E)\xrightarrow{\pi}[0,1]\xrightarrow{\widetilde{\varphi_{N(P)}}} [0,1]\simeq \Real{N(P)},
\end{equation*}
où $\widetilde{\varphi_{N(P)}}$ est définie comme
\begin{equation*}
\widetilde{\varphi_{N(P)}}(t)=\left\{\begin{array}{cl}
0 &\text{ si $0\leq t<1/3$}\\
3t-1 &\text{ si $1/3\leq t\leq 2/3$}\\
1 &\text{ si $2/3<t$}
\end{array}\right.
\end{equation*}
Soit $f\colon E\to E'$ une application continue, on considère la somme amalgamée
\begin{equation*}
\begin{tikzcd}
E
\arrow{r}{f}
\arrow{d}{i_1}
&E'
\arrow{d}
\\
\bar{c}_{\pi}(E)
\arrow{r}
&X(\pi,f)
\end{tikzcd}
\end{equation*}

Alors, $\varphi_{\pi}$ et $\widetilde{\varphi}$ induisent deux filtrations $X(\pi,f)\to \Real{N(P)}$ qu'on notera respectivement $\varphi_X$ et $\widetilde{\varphi}_X$. On peut maintenant énoncer la propriété suivante.

\begin{prop}\label{PropositionRemplacementCofibrantFibre}
Soit $(X(\pi,f),\varphi_X)$ un espace fortement filtré obtenu par la construction précédente. Alors son remplacement cofibrant obtenu par application de la proposition \ref{PropositionRemplacementCofibrantEfficace} est homéomorphe au sens filtré à $(X(\pi,f),\widetilde{\varphi}_X)$. De plus, si $E, B$ et $F$ sont des variétés, et $E\to E'$ est l'inclusion du bord dans une variété à bord, l'espace fortement filtré $(X(\pi,f),\varphi_X)$ est une pseudo variété. 
\end{prop}

\begin{proof}
Par la proposition \ref{FibrePseudoVariete} l'espace (fortement) filtré $(X(\pi,f),\varphi_X)$ est une pseudo variété. Pour montrer que $(X(\pi,f),\widetilde{\varphi}_X)$ est homéomorphe à $(X(\pi,f)^{\cof},\varphi^{\cof}_X)$, il suffit d'exhiber un carré cartésien
\begin{equation}\label{CarreCartesienRemplacementCofibrantFibre}
\begin{tikzcd}
X(\pi,f)
\arrow{r}{\alpha}
\arrow{d}{\beta}
&X(\pi,f)
\arrow{d}{\varphi_X}
\\
\Real{\sd_P(N(P))}
\arrow{r}{h_{N(P)}\circ\varphi}
&\Real{N(P)}
\end{tikzcd}
\end{equation}
tel que la composition 
\begin{equation*}
X(\pi,f)\xrightarrow{\beta}\Real{\sd_P(N(P))}\xrightarrow{\Real{\varphi_{\sd_P(N(P))}}}\Real{N(P)}
\end{equation*}
est égale à $\widetilde{\varphi}_X$.
On définit un homéomorphisme $g\colon\Real{\sd_P(N(P))}\to [0,1]$ en posant $g([p_0],p_0)=0$, $g([p_0,p_1],p_0)=1/3$, $g([p_0,p_1],p_1)=2/3$ et $g([p_1],p_1)=1$ et en étendant la définition par linéarité. Ainsi le morphisme $\pr_{[0,1]}\colon E\times[0,1]\to [0,1]$ induit un morphisme 
\begin{equation*}
\beta\colon X(\pi,f)\to [0,1]\xrightarrow{g^{-1}}\Real{\sd_P(N(P))}.
\end{equation*}
 De plus, par construction, la filtration $\widetilde{\varphi}_X\colon X(\pi,f)\to \Real{N(P)}$ est égale à la composition $\Real{\varphi_{\sd_P(N(P))}}\circ \beta$.
Soit $U\subseteq B$ un ouvert trivialisant pour le fibré $c_{\pi}(E)\xrightarrow{c_{\pi}(\pi)}B$. Alors, on a 
\begin{equation*}
V=U\times c(L)\simeq c_{\pi}(\pi)^{-1}(U)\subseteq c_{\pi}(E)\subseteq X(\pi,f)
\end{equation*}
On définit $\alpha_{|V}$, la restriction de $\alpha$ à l'ouvert $V$ comme
\begin{align*}
\alpha_{|V}\colon U\times c(L) &\to U\times c(L)\\
(b,(y,t))&\mapsto \left\{\begin{array}{cl}
(b,(y,2t-1)) &\text{ si $2/3\leq t\leq 1$}\\
(b,(y,1/3)) &\text{ si $1/3\leq t\leq 2/3$}\\
(b,(y,t)) &\text{ si $0\leq t\leq 1/3$} 
\end{array}\right.
\end{align*}
Par construction de $c_{\pi}(E)$, les $\alpha_V$ se recollent et fournissent une application continue
\begin{equation*}
\alpha_{|\bar{c}_{\pi}(E)}\colon \bar{c}_{\pi}(E)\to\bar{c}_{\pi}(E)
\end{equation*}
De plus, par construction, la restriction de $\alpha$ à $\partial(\bar{c}_{\pi}(E))=E$ est égale à l'identité de $E$, ainsi, on peut prolonger $\alpha$ en un morphisme
\begin{equation*}
\alpha\colon X(\pi,f)\to X(\pi,f)
\end{equation*}
en imposant $\alpha_{|E'}=\Id_{E'}$.
De plus, avec les identifications précédentes, on peut expliciter le morphisme $h_{N(P)}\circ\varphi$. On a 
\begin{align*}
h_{N(P)}\circ\varphi\colon [0,1]&\to [0,1]\\
t&\mapsto \left\{\begin{array}{cl}
\frac{3t}{2} &\text{ si $0\leq t\leq 1/3$}\\
1/2 &\text{ si $1/3\leq t\leq 2/3$}\\
3t-2  &\text{ si $2/3\leq t\leq 1$}
\end{array}\right.
\end{align*}
Finalement, avec les morphismes $\alpha$ et $\beta$ définis précédemment, le carré \ref{CarreCartesienRemplacementCofibrantFibre} est commutatif. Montrons qu'il est cartésien. On remarque d'abord que la restriction de $\alpha$ à $\alpha^{-1}(\varphi_{X}^{-1}([0,1/2[\cup]1/2,1]))$ est un homéomorphisme sur son image. De même, la restriction de $h_{N(P)}\circ\varphi$ à $(h_{N(P)}\circ\varphi)^{-1}([0,1/2[,]1/2,1])$ est un homéomorphisme sur son image. D'autre part, on a
\begin{equation*}
\alpha^{-1}(\varphi_X^{-1}(\{1/2\})\simeq E\times [1/3,2/3]=(h_{N(P)}\circ\varphi)^{-1}(\{1/2\})\times \varphi_X^{-1}(\{1/2\})
\end{equation*}
On en déduit que le carré est cartésien.
\end{proof}

De façon à exploiter ce résultat, on définit la construction suivante. Soit $g\colon (X(\pi_Y,f_Y),\widetilde{\varphi}_Y)\to (X(\pi_Z,f_Z),\widetilde{\varphi}_Z)$ une application filtrée entre deux espaces issus de la construction précédente. Soit $U\subseteq B_{Y}$ un ouvert trivialisant pour le fibré $\pi_{Y}$. On note $V=\pi_Y^{-1}(U)\simeq U\times c(L_Y)$, quitte à restreindre $U$, on peut supposer que $g(V)$ est inclus dans un ouvert trivialisant pour $\pi_Z$ de la forme $W\times c(L_Z)$. On définit alors l'application filtrée $c(g)$ en restriction à $V$ par
\begin{align*}
c(g)_{|V}\colon U\times c(L_Y)&\to W\times c(L_Z)\\
(b,(l,t))&\mapsto \left\{\begin{array}{cl}
g(b,(l,t)) &\text{ si $1/3\leq t\leq 1$}\\
(\pr_{B_Z}(g(b,(l,1/3))),(\pr_{L_Z}(g(b,(l,1/3))),t) &\text{si $0\leq t\leq 1/3$}
\end{array}\right.
\end{align*}
par construction, les $c(g)_{|V}$ se recollent en une application $c(g)\colon c_{\pi_Y}(E_Y)\to c_{\pi_Z}(E_Z)$ que l'on peut étendre en une application $c(g)\colon X(\pi_Y,f_Y)\to X(\pi_Z,f_Z)$ en posant $c(g)_{|E'}=g_{|E'}$. On remarque que par construction, $c(g)$ est filtrée pour les filtrations $\varphi_Y$ et $\varphi_Z$. (Et donc aussi pour les filtrations $\widetilde{\varphi}_Y$ et $\widetilde{\varphi}_Z$.)

\begin{lemme}\label{LemmeRedresserApplicationsFibre}
Soit $g\colon (X(\pi_Y,f_Y),\widetilde{\varphi}_Y)\to (X(\pi_Z,f_Z),\widetilde{\varphi}_Z)$ une application  filtrée entre deux espaces issus de la construction précédente. L'application filtrée 
\begin{equation*}
c(g)\colon (X(\pi_Y,f_Y),\widetilde{\varphi}_Y)\to (X(\pi_Z,f_Z),\widetilde{\varphi}_Z)
\end{equation*}
est homotope par une homotopie filtrée $H$ à $g$. De plus, si $g$ était aussi filtrée pour $\varphi_Y$ et $\varphi_Z$, on peut supposer que $H$ est une homotopie filtrée pour $\varphi_Y$ et $\varphi_Z$.
\end{lemme}

\begin{proof}
Soit $U\subseteq B_{Y}$ un ouvert trivialisant pour le fibré $\pi_{Y}$. On note $V=\pi_Y^{-1}(U)\simeq U\times c(L_Y)$. Quitte à restreindre $U$, on peut supposer que $g(V)$ est inclus dans un ouvert trivialisant pour $\pi_Z$ de la forme $W\times c(L_Z)$. Par commodité, pour $(b,(l,t))\in U\times c(L_Y)$, on écrit 
\begin{equation*}
g_{|V}(b,(l,t))=(g_{B_Z}(b,(l,t)),g_{L_Z}(b,(l,t)),g_{[0,1]}(b,(l,t)))\in W\times c(L_Z).
\end{equation*}
  On définit alors l'homotopie $H$ en restriction à $V$ comme
\begin{align*}
U\times c(L_Y)\times [0,1]&\xrightarrow{H_{|V} } W\times c(L_Z)\\
(b,(l,t),s)&\mapsto \left\{\begin{array}{cl}
g(b,(l,t)) &\text{ si $1/3\leq t\leq 1$}\\
(g_{B_Z}(b,(l,1/3)),g_{L_Z}(b,(l,1/3)),t) &\text{si $\frac{1-s}{3}\leq t\leq 1/3$}\\
(g_{B_Z}(b,(l,s/3+(1-s)t)),g_{L_Z}(b,l,\frac{t}{1-s}),(1-s)g_{[0,1]}(b,l,\frac{t}{1-s})) &\text{ si $0\leq t<\frac{1-s}{3}$}
\end{array}\right.
\end{align*}
Pour montrer que $H_{|V}$ est continue, il suffit de vérifier composante par composante, on a alors le résultat immédiatement.
Les $H_{|V}$ se recollent, et on peut étendre $H$ à $E'\times[0,1]$ en prenant l'homotopie constante égale à $g_{|E'}$. Par construction, $H$ est filtrée pour $\widetilde{\varphi}_Y$ et $\widetilde{\varphi}_Z$, et on a bien $H_{|s=0}=c(g)$ et $H_{|s=1}=g$. Par ailleurs, si $g$ est filtrée pour les filtrations $\varphi_{Y}$ et $\varphi_Z$, on a pour tout$(b,(l,t))\in U\times c(L_Y)$, $g_{[0,1]}(b,(l,t))=t$. En particulier, dès que $0\leq t<\frac{1-s}{3}$ on a
\begin{equation*}
(1-s)g_{[0,1]}(b,(l,\frac{t}{1-s}))=(1-s)\frac{t}{1-s}=t
\end{equation*} 
Dans ce cas, $H$ est filtrée pour $\varphi_Y$ et $\varphi_Z$.
\end{proof}

Pour deux espaces fortement filtrés $\fil{Y}$ et $\fil{Z}$, on note $[\fil{Y},\fil{Z}]$ l'ensemble des classes d'homotopies filtrées d'applications filtrées de $\fil{Y}$ vers $\fil{Z}$.

\begin{prop}\label{PropositionClasseHomotopieNaiveFibre}
Soient $\fil{Y}=(X(\pi_Y,f_Y),\varphi_Y)$ et $\fil{Z}=(X(\pi_Z,f_Z),\varphi_Z)$ deux espaces issus de la construction précédente tels qu'il existe $\fil{A}$ et $\fil{B}$ deux ensembles simpliciaux filtrés vérifiant $\fil{Y}\simeq\RealNP{\fil{A}}$ et $\fil{Z}\simeq\RealNP{\fil{B}}$. Alors le morphisme naturel
\begin{equation*}
[\fil{Y},\fil{Z}]\to \Hom_{\Ho\Top_{N(P)}}(\fil{Y},\fil{Z})
\end{equation*}
est une bijection. 
\end{prop}

\begin{proof}
Par la proposition \ref{PropositionRemplacementCofibrantFibre}, il existe des équivalences faibles filtrés $(Y,\widetilde{\varphi}_Y)\to \fil{Y}$ et $(Z,\widetilde{\varphi}_Z)\to \fil{Z}$. Ainsi, on a une bijection
\begin{equation*}
\Hom_{\Ho\Top_{N(P)}}(\fil{Y},\fil{Z})\simeq \Hom_{\Ho\Top_{N(P)}}((Y,\widetilde{\varphi}_Y),(Z,\widetilde{\varphi}_Z))
\end{equation*}
De plus par les proposition \ref{PropositionRemplacementCofibrantEfficace} et \ref{PropositionRemplacementCofibrantFibre}, les espaces filtrés $(Y,\widetilde{\varphi}_Y)$ et $(Z,\widetilde{\varphi}_Z)$ sont cofibrants. Comme tout les objets de $\Top_{N(P)}$ sont fibrants, on a
\begin{equation*}
\Hom_{\Ho\Top_{N(P)}}((Y,\widetilde{\varphi}_Y),(Z,\widetilde{\varphi}_Z))=[(Y,\widetilde{\varphi}_Y),(Z,\widetilde{\varphi}_Z)]
\end{equation*}
Il reste à montrer qu'on a un isomorphisme
\begin{equation*}
[\fil{Y},\fil{Z}]\simeq [(Y,\widetilde{\varphi}_Y),(Z,\widetilde{\varphi}_Z)].
\end{equation*}
Si $g\colon \fil{Y}\to\fil{Z}$ est une application filtrée, alors $g$ est aussi une application filtrée pour les filtrations $\widetilde{\varphi}_Y$ et $\widetilde{\varphi}_Z$. De plus, comme le cylindre $Y\otimes \Delta^1$ peut être obtenu en appliquant la construction précédente à 
\begin{equation*}
F_Y\hookrightarrow E_Y\times\Delta^1\xrightarrow{\pi_Y\times\Delta^1}B_Y\times \Delta^1,
\end{equation*}
on en déduit qu'il en va de même pour toute homotopie $H\colon \Delta^1\otimes\fil{Y}\to\fil{Z}$. On a donc une application bien définie
\begin{align*}
\alpha\colon[\fil{Y},\fil{Z}]&\to [(Y,\widetilde{\varphi}_Y),(Z,\widetilde{\varphi}_Z)]\\
[g]&\mapsto [g]
\end{align*}
Par le lemme \ref{LemmeRedresserApplicationsFibre}, pour tout application $h\colon (Y,\widetilde{\varphi}_Y)\to (Z,\widetilde{\varphi}_Z)$, il existe une application $c(h)\colon \fil{Y}\to\fil{Z}$ telle que $c(h)$ et $h$ sont homotopes par une homotopie filtrée pour les filtrations $\widetilde{\varphi}_Y$ et $\widetilde{\varphi}_Z$. On en déduit que le morphisme $\alpha$ est surjectif. Soient maintenant $g_1,g_2\colon \fil{Y}\to\fil{Z}$ telles qu'il existe une homotopie filtrée $H\colon \Delta^1\otimes(Y,\widetilde{\varphi}_Y)\to (Z,\widetilde{\varphi}_Z)$ entre $g_1$ et $g_2$. Alors, d'après le lemme \ref{LemmeRedresserApplicationsFibre}, $c(H)$ fournit une homotopie entre $c(g_1)$ et $c(g_2)$, filtrée pour $\varphi_Y$ et $\varphi_Z$. Toujours par le lemme \ref{LemmeRedresserApplicationsFibre}, $c(g_1)$ et $c(g_2)$ sont respectivement homotopes à $g_1$ et $g_2$ par des homotopies filtrées pour $\varphi_Y$ et $\varphi_Z$. Finalement, $g_1$ et $g_2$ sont homotopes par une homotopie filtrée pour les filtrations $\varphi_Y$ et $\varphi_Z$. On en déduit que $\alpha$ est injective. 
\end{proof}

\begin{corollaire}\label{CorollaireWhiteheadFibre}
Soient $\fil{Y}=(X(\pi_Y,f_Y),\varphi_Y)$ et $\fil{Z}=(X(\pi_Z,f_Z),\varphi_Z)$ deux espaces fortement filtrés provenant de la construction précédente, et $g\colon \fil{Y}\to\fil{Z}$ une application filtrée entre eux. On suppose qu'il existe des ensembles simpliciaux filtrés $\fil{A}$ et $\fil{B}$ tels que $\fil{Y}\simeq \RealNP{\fil{A}}$ et $\fil{Z}\simeq \RealNP{\fil{B}}$.
Alors, $g$ est une équivalence d'homotopie (fortement) filtrée si et seulement si $g$ induit des isomorphismes sur tous les groupes d'homotopie filtrés.
\end{corollaire}

\begin{proof}
Le sens direct provient de la définition des groupes d'homotopie filtrés. Pour la réciproque, soit $g\colon \fil{Y}\to\fil{Z}$ induisant des isomorphismes sur tous les groupes d'homotopie filtrés. Alors, $g$ est un isomorphisme de la catégorie $\Ho\Top_{N(P)}$. Il existe donc $h\in \Hom_{\Ho\Top_{N(P)}}(\fil{Y},\fil{Z})$ un inverse de $g$ dans $\Ho\Top_{N(P)}$. Mais, par la proposition \ref{PropositionClasseHomotopieNaiveFibre}, on a
\begin{equation*}
[\fil{Z},\fil{Y}]\simeq \Hom_{\Ho\Top_{N(P)}}(\fil{Z},\fil{Y}).
\end{equation*}
Il existe donc une application filtrée $h'\colon\fil{Z}\to\fil{Y}$ égale à $h$ dans $\Ho\Top_{N(P)}$. Mais alors, $h'\circ g$ est égal à $\Id_Y$ dans $\Ho\Top_{N(P)}$, et en appliquant de nouveau la proposition \ref{PropositionClasseHomotopieNaiveFibre}, on a que $h'\circ g$ est homotope à $\Id_Y$ par une homotopie filtrée. De même, $g\circ h'$ est homotope au sens filtré à $\Id_Z$, et on en déduit que $h'$ est une équivalence d'homotopie filtrée inverse à $g$.
\end{proof}

\begin{remarque}\label{RemarqueWhiteheadFibre}
Le corollaire \ref{CorollaireWhiteheadFibre} est une reformulation du théorème \ref{PremierTheoremeWhitehead} dans le cas particulier où les espaces coniquement stratifiés considérés sont issus de la construction de ce chapitre. Cependant, on constate que la preuve du corollaire \ref{CorollaireWhiteheadFibre} utilise des techniques différentes. En particulier, on n'a pas besoin ici d'utiliser le théorème \cite[Theorem A.6.4]{HigherAlgebra} impliquant que pour tout espace coniquement stratifié $\fil{X}$, $\Sing_P\fil{X}$ est un objet fibrant de $\sS_P$. Finalement, la preuve du corollaire \ref{CorollaireWhiteheadFibre} permet aussi de comprendre le théorème de Whitehead stratifié comme un résultat sur les objets cofibrants de la catégorie modèle $\Top_{N(P)}$. 
On conjecture que la proposition \ref{PropositionClasseHomotopieNaiveFibre} et son corollaire \ref{CorollaireWhiteheadFibre} sont vrais dans un cas beaucoup plus général. En effet, la construction de la section \ref{ConstructionFibre} est généralisable à un fibré de fibre (fortement) filtré. Les constructions de cette sections peuvent ensuite être généralisées à ces objets là en les appliquant de façon inductive. D'autre part les résultats de D. Stone sur les polyèdres stratifiés \cite{Stone} suggèrent que tout espace coniquement stratifié PL (a fortiori, toute pseudo-variété PL) peut être obtenu par applications successives de la construction vue plus tôt. Si ces affirmations sont correctes, on obtient ainsi le résultat suivant.
\end{remarque}

Soit $\mathcal{C}$ la sous catégorie pleine de $\Top_{N(P)}$ formée des espaces PL coniquement stratifiés. Notons $\mathcal{C}/{\sim}$ la catégorie ayant les mêmes objets que $\mathcal{C}$ et vérifiant
\begin{equation*}
\Hom_{\mathcal{C}/{\sim}}(\fil{X},\fil{Y})=[\fil{X},\fil{Y}]
\end{equation*}

\begin{conjecture}\label{ConjectureSousCategorieFibre}
Le foncteur induit par l'inclusion $C\subset\Top_{N(P)}$
\begin{equation*}
\mathcal{C}/{\sim}\to \Ho\Top_{N(P)}
\end{equation*}
est pleinement fidèle. 
\end{conjecture}

La validité de cette conjecture impliquerait immédiatement le théorème de Whitehead stratifié.
\begin{corollaire}[Théorème \ref{PremierTheoremeWhitehead}]
Soient $\fil{X},\fil{Y}$ deux espaces filtrés et $f\colon \fil{X}\to\fil{Y}$ une application filtrée. On suppose que 
\begin{itemize}
\item il existe $\fil{A},\fil{B}$ deux ensembles simpliciaux filtrés tels que $\fil{X}\simeq\RealP{\fil{A}}$ et $\fil{Y}\simeq\RealP{\fil{B}}$ (on ne suppose pas que $f$ provient d'une application simpliciale),
\item les ensembles simpliciaux filtrés $\Sing_P(X)$ et $\Sing_P(Y)$ sont fibrants.
\end{itemize}
Alors, $f$ est une équivalence d'homotopie filtrée si et seulement si 
\begin{equation*}
s\pi_0(f)\colon s\pi_0\fil{X}\to s\pi_0\fil{Y}
\end{equation*}
est un isomorphisme, et, pour tout pointage $\phi\colon\RealP{V}\to\fil{X}$ et pour tout $n\geq 1$, les morphismes
\begin{equation*}
s\pi_n(f)\colon s\pi_n(\fil{X},\phi)\to s\pi_n(\fil{Y},f\circ \phi)
\end{equation*}
sont des isomorphismes.
\end{corollaire}

\section{Une autre structure de modèle sur $\sS_P$}
\label{SectionSSTopP}
Dans cette section, on considère une deuxième structure de modèle sur la catégorie $\sS_P$. Pour éviter tout confusion on notera $\sSU_P$ la structure de modèle du théorème \ref{TheoDescriptionExpliciteSSetP}, et $\sSTop_P$ la nouvelle structure de modèle. 
On a vu précédement que l'adjonction $\RealNP{-},\Sing_{N(P)}$ n'était pas une adjonction de Quillen entre $\sSU_P$ et $\Top_{N(P)}$. De façon à exhiber une telle adjonction de Quillen, on construit une catégorie modèle intermédiaire, $\sSTop_P$, obtenue par transport de la structure sur $\Top_P$ le long de l'adjonction $\RealNP{-},\Sing_{N(P)}$. On comparera ensuite les catégories de modèles $\sSU_P$ et $\sSTop_P$ de façon à avoir une comparaison entre $\sSU_P$ et $\Top_{N(P)}$.
\subsection{Une structure de modèle transporté depuis $\Top_{N(P)}$}
\label{SectionConstructionCMFSSTopP}

L'objet de cette sous-section est de prouver le résultat suivant.
\begin{theo}\label{TheoremeCategorieModeleSSTop}
Il existe une structure de modèle sur $\sS_P$ telle que 
\begin{itemize}
\item un morphisme $f\colon \fil{A}\to\fil{B}$ est une cofibration si et seulement si 
\begin{equation*}
\RealNP{f}\colon \RealNP{\fil{A}}\to\RealNP{\fil{B}}
\end{equation*}
est une cofibration de $\Top_{N(P)}$,
\item un morphisme $f\colon \fil{A}\to\fil{B}$ est une équivalence faible si et seulement si 
\begin{equation*}
\RealNP{f}\colon \RealNP{\fil{A}}\to\RealNP{\fil{B}}
\end{equation*}
est une équivalence faible de $\Top_{N(P)}$.
\end{itemize}
\end{theo}

\begin{proof}
Comme la catégorie modèle $\Top_{N(P)}$ est engendrée de façon cofibrante, et que $\sS_P$ est localement présentable, par \cite[Corollary 3.3.4]{Hess}, il suffit de montrer que tout morphisme $f\colon \fil{X}\to\fil{Y}\in \sS_P$ ayant la propriété de relèvement à droite par rapport aux cofibrations est une équivalence faible. Soit $f$ un tel morphisme. On suppose dans un premier temps que $\fil{X}$ et $\fil{Y}$ sont cofibrants (c'est à dire que $\RealNP{\fil{X}}$ et $\RealNP{\fil{Y}}$ sont cofibrants dans $\Top_{N(P)}$).
On considère le cylindre de $f$, $\fil{C(f)}$, défini comme la somme amalgamée suivante
\begin{equation}\label{DiagrammeCylindreFiltre}
\begin{tikzcd}[column sep = huge]
\fil{X}\coprod\fil{X}\coprod\fil{Y}
\arrow[swap]{d}{i_0\coprod i_1\coprod \Id_Y}
\arrow{r}{\Id_X\coprod f\coprod\Id_Y}
&\fil{X}\coprod\fil{Y}
\arrow[swap]{d}{i_C}
\arrow{ddr}{f\coprod \Id_Y}
\\
\left(\Delta^1\otimes \fil{X}\right)\coprod\fil{Y}
\arrow{r}
\arrow[swap]{drr}{f\circ\pr_X\coprod\Id_Y}
&\fil{C(f)}
\arrow{dr}{q}
\\
&&\fil{Y}
\end{tikzcd}
\end{equation}
Comme dans le cas non filtré, le morphisme $q\colon \fil{C(f)}\to \fil{Y}$ est une équivalence d'homotopie filtrée, d'inverse $\fil{Y}\xrightarrow{i_Y}\fil{X}\coprod\fil{Y}\xrightarrow{i_C}C(f)$.
D'autre part, comme $\fil{X}$ est cofibrant par hypothèse, par le lemme \ref{LemmeCofibrationOtimesCofibration}, le morphisme 
\begin{equation*}
\left(\partial(\Delta^1)\otimes \fil{X}\right)\coprod\fil{Y} \simeq
\fil{X}\coprod\fil{X}\coprod\fil{Y}\xrightarrow{i_0\coprod i_1\coprod \Id_Y} \left(\Delta^1\otimes\fil{X}\right)\coprod\fil{Y}
\end{equation*}
est une cofibration. Comme $\RealNP{-}$ préserve les colimites et que la classe des cofibrations dans $\Top_{N(P)}$ est stable par somme amalgamée, on en déduit que $i_C$ est une cofibration. Finalement, comme $\fil{Y}$ est cofibrant, l'inclusion $i_X\colon\fil{X}\to\fil{X}\coprod\fil{Y}$ est une cofibration, et donc la composition $i$ 
\begin{equation*}
\fil{X}\xrightarrow{i_X}\fil{X}\coprod\fil{Y}\xrightarrow{i_C}\fil{C(f)}
\end{equation*}
est une cofibration. Finalement, on a le diagramme commutatif suivant
\begin{equation*}
\begin{tikzcd}
\fil{X}
\arrow{d}{i}
\arrow{r}{\Id_X}
&\fil{X}
\arrow{d}{f}
\\
\fil{C(f)}
\arrow{r}{q}
\arrow[dashrightarrow]{ur}{g}
&\fil{Y}
\end{tikzcd}
\end{equation*}
Comme $i$ est une cofibration et $f$ a la propriété de relèvement à droite par rapport aux cofibrations, on en déduit qu'il doit exister un morphisme $g\colon \fil{C(f)}\to\fil{X}$ faisant commuter le diagramme. Fixons $\phi\colon V\to \RealNP{\fil{X}}$ un pointage de $\RealNP{\fil{X}}$, $n\geq 0$ et $\Delta^{\varphi}\subseteq V$ un simplexe non dégénéré. On calcule les groupes d'homotopies filtrés :
\begin{equation*}
\begin{tikzcd}
s\pi_n(\RealNP{\fil{X}},\phi)(\Delta^{\varphi})
\arrow{d}{i_*}
\arrow{r}{\Id}
&s\pi_n(\RealNP{\fil{X}},\phi)(\Delta^{\varphi})
\arrow{d}{f_*}
\\
s\pi_n(\RealNP{\fil{C(f)}},i\circ \phi)(\Delta^{\varphi})
\arrow{r}{q_*}
\arrow{ur}{g_*}
&s\pi_n(\RealNP{\fil{Y}},f\circ \phi)(\Delta^{\varphi})
\end{tikzcd}
\end{equation*}
Comme $q$ est une équivalence d'homotopie filtrée, $\RealNP{q}$ aussi, et $q_*$ est donc un isomorphisme. Ainsi, $g_*$ admet un inverse à gauche et à droite, c'est donc aussi un isomorphisme. Mais alors on en déduit que $f_*$ est un isomorphisme. Finalement $\RealNP{f}$ induit des isomorphismes sur tous les groupes d'homotopie filtrés, c'est donc une équivalence faible de $\Top_{N(P)}$, d'où $f$ est une équivalence faible de $\sSTop_P$. 

Si $\fil{X}$ et $\fil{Y}$ ne sont pas cofibrants, on considère le morphisme 
\begin{equation*}
\sd_P(f)\colon \sd_P\fil{X}\to\sd_P\fil{Y}
\end{equation*}
En appliquant la construction \ref{DiagrammeCylindreFiltre} à $\sd_P(f)$ on obtient le diagramme commutatif suivant
\begin{equation*}
\begin{tikzcd}
\sd_P\fil{X}
\arrow{d}{i}
\arrow{r}{\Id}
&\sd_P\fil{X}
\arrow{r}{\lv_P}
&\fil{X}
\arrow{d}{f}
\\
\fil{C(\sd_P(f))}
\arrow[dashrightarrow, near start]{urr}{g}
\arrow[swap]{r}{q}
&\sd_P\fil{Y}
\arrow[leftarrow, crossing over, swap, near start]{u}{\sd_P(f)}
\arrow[swap]{r}{\lv_P}
&\fil{Y}
\end{tikzcd}
\end{equation*}
Or, par la proposition \ref{PropositionSubdivisionCofibrante}, $\sd_P\fil{X}$ et $\sd_P\fil{Y}$ sont cofibrants, le morphisme $i$ est donc une cofibration. Ainsi, il existe un morphisme $g$ faisant commuter le diagramme. Par ailleurs, par la proposition \ref{LastVertexEquivalenceFaible} les morphismes $\lv_P\colon\sd_P\fil{X}\to\fil{X}$ et $\lv_P\colon \sd_P\fil{Y}\to\fil{Y}$ sont des équivalences faibles de $\sSTop_P$. Ainsi, par les mêmes arguments que précédemment, on en déduit que $g$ est une équivalence faible de $\sSTop_P$ et donc que $f$ aussi.
\end{proof}

\begin{lemme}\label{LemmeCofibrationOtimesCofibration}
Soient $i\colon\fil{X}\to\fil{Y}$ une cofibration de $\Top_{N(P)}$ et $j\colon K\to L$ une cofibration de $\sS$. Le morphisme
\begin{equation*}
L\otimes\fil{X}\cup K\otimes\fil{Y}\to L\otimes\fil{Y}
\end{equation*}
est une cofibration qui est triviale dès que $i$ ou $j$ l'est.
\end{lemme}

\begin{proof}
On rappelle qu'un ensemble de cofibrations génératrices de $\Top_{N(P)}$ est donné par 
\begin{equation*}
I=\{\RealNP{\partial(\Delta^n)\otimes\Delta^\varphi\to\Delta^n\otimes\Delta^{\varphi}}\ |\ n\geq 0,\  \Delta^{\varphi}\in R(P)\}.
\end{equation*}
Comme pour tout $n\geq 0$, $\Delta^{\varphi}\in R(P)$, $\RealNP{\partial(\Delta^n)\otimes\Delta^{\varphi}}$ est compact, par l'argument du petit objet $i\colon \fil{X}\to\fil{Y}$ est le rétracte d' une application 
\begin{equation}\label{EquationPetitObjet}
\fil{Z_0}\to\colim_{n\in \N}\fil{Z_n}=\fil{Z_{\infty}}
\end{equation}
où, pour tout $k\geq 0$, $\fil{Z_k}\to \fil{Z_{k+1}}$ est obtenu comme la somme amalgamée
\begin{equation*}
\begin{tikzcd}
\coprod_{\alpha\in S_k}\RealNP{\partial(\Delta^n)\otimes\Delta^{\varphi}}
\arrow{r}{\coprod\alpha}
\arrow{d}
&\fil{Z_k}
\arrow{d}
\\
\coprod_{\alpha\in S_k}\RealNP{\Delta^n\otimes\Delta^{\varphi}}
\arrow{r}
&\fil{Z_{k+1}}
\end{tikzcd}
\end{equation*}
où $S_k$ est un certain ensemble.
Comme la classe des cofibrations est stable par rétractes, on peut supposer que $X=Z_0$, $Y=Z_{\infty}$ et $i$ est égale à la cofibration \ref{EquationPetitObjet}. On pose $\fil{V_0}=\fil{W_0}=L\otimes\fil{X}$, et $f_0=\Id\colon V_0\to W_0$. Puis, pour tout $k\geq 0$, on définit $\fil{V_{k+1}}$ et $\fil{W_{k+1}}$ comme les sommes amalgamées suivantes
\begin{equation*}
\begin{tikzcd}
K\otimes\left(\coprod_{\alpha\in S_k}\RealNP{\partial(\Delta^n)\otimes\Delta^{\varphi}}\right)
\arrow{r}{\Id_K\otimes\left(\coprod\alpha\right)}
\arrow{d}
&\fil{V_k}
\arrow{d}
\\
K\otimes\left(\coprod_{\alpha\in S_k}\RealNP{\Delta^n\otimes\Delta^{\varphi}}\right)
\arrow{r}
&\fil{V_{k+1}}
\end{tikzcd}
\end{equation*}
et
\begin{equation*}
\begin{tikzcd}
L\otimes\left(\coprod_{\alpha\in S_k}\RealNP{\partial(\Delta^n)\otimes\Delta^{\varphi}}\right)
\arrow{r}{\Id_L\otimes\left(\coprod\alpha\right)}
\arrow{d}
&\fil{W_k}
\arrow{d}
\\
L\otimes\left(\coprod_{\alpha\in S_k}\RealNP{\Delta^n\otimes\Delta^{\varphi}}\right)
\arrow{r}
&\fil{W_{k+1}}
\end{tikzcd}
\end{equation*}
et on définit $f_{k+1}\colon \fil{V_{k+1}}\to\fil{W_{k+1}}$ comme l'application induite par $f_{k}$ et l'inclusion $K\to L$. 
On note alors
\begin{equation*}
V_{\infty}=\colim_{k\in \N}V_k, W_{\infty}=\colim_{k\in \N}W_k \text{ et } f=\colim_{k\in \N}f_k\colon V_{\infty}\to W_{\infty}
\end{equation*}
Alors par construction, on a $V_{\infty}\simeq L\otimes\fil{X}\cup K\otimes\fil{Y}$ et $W_{\infty}\simeq L\otimes\fil{Y}$, et $f$ est l'inclusion. Par \cite[Lemme 1.1.10]{Cisinski} il suffit de montrer que pour tout $k\geq 0$, le morphisme
\begin{equation*}
\fil{V_{k+1}}\cup_{\fil{V_k}}\fil{W_k}\to\fil{W_{k+1}}
\end{equation*}
est une cofibration.
On considère le diagramme commutatif suivant (ou les filtrations et le foncteur $\RealNP{-}$ sont omis par soucis de lisibilité)
\begin{equation*}
\begin{tikzcd}
&\coprod \left(K\times\partial(\Delta^n)\right)\otimes\Delta^{\varphi}
\arrow{rr}
\arrow{dd} 
\arrow{dl}
&&V_k
\arrow{dd}
\arrow{dl}
\\
\coprod\left(L\times \partial(\Delta^n)\right)\otimes\Delta^{\varphi}
\arrow{rr}
\arrow[crossing over]{dd}
&&W_{k}
\arrow{dd}
\\
&\coprod\left(K\times\Delta^n\right)\otimes\Delta^{\varphi}
\arrow{rr}
\arrow{dl}
&&
V_{k+1}
\arrow{dl}
\\
\coprod\left(L\times\partial(\Delta^n)\cup K\times \Delta^n\right)\otimes\Delta^{\varphi}
\arrow{rr}
\arrow{d}
&&V_{k+1}\cup W_k
\arrow{d}
\\
\coprod\left(L\times\Delta^n\right)\otimes\Delta^{\varphi}
\arrow{rr}
&& W_{k+1}
\end{tikzcd}
\end{equation*}
Par définition de $\fil{V_{k+1}}$, la face arrière de ce cube est cartésienne. De plus, les faces de gauche et de droite sont cartésiennes, et donc la face avant est cartésienne. Comme le rectangle avant est cartésien, on en déduit que le carré en bas du diagramme est cartésien. Finalement, comme les cofibrations sont stables par unions disjointes et somme amalgamée, il suffit de montrer que les morphismes de la forme
\begin{equation*}
\Real{L\otimes\partial(\Delta^n)\cup K\times\Delta^n}\otimes\RealNP{\Delta^{\varphi}}\to \Real{L\times\Delta^n}\otimes\RealNP{\Delta^{\varphi}}
\end{equation*}
sont des cofibrations de $\Top_{N(P)}$. Ces morphismes sont images par le foncteur $\Colim$ des morphisme de $\DiagR_P$
\begin{equation*}
(L\times\partial(\Delta^n)\cup K\times\Delta^n)^{\Delta^{\varphi}}\to (L\times \Delta^n)^{\Delta^{\varphi}}
\end{equation*}
Et ces derniers sont des cofibrations de $\DiagR_P$ car pour tout $n\geq 0$, le morphisme
\begin{equation}\label{EquationPreuveCofibrationOtimesCofibration}
L\times\partial(\Delta^n)\cup K\times\Delta^n\to L\times \Delta^n
\end{equation}
est une cofibration. On en déduit le résultat voulu.
Si $j$ est une cofibration triviale, le morphisme \ref{EquationPreuveCofibrationOtimesCofibration} est une cofibration triviale et on en déduit le résultat voulu. Si $i$ est une cofibration triviale, on remplace $I$ par 
\begin{equation*}
J=\{\RealNP{\Lambda^n_k\otimes\Delta^\varphi\to\Delta^n\otimes\Delta^{\varphi}}\ |\ n\geq k\geq 0,\  \Delta^{\varphi}\in R(P)\}.
\end{equation*}
pour obtenir le résultat voulu.
\end{proof}

\subsection{Adjonction de Quillen entre les catégories de modèles $\sSU_P$ et $\sSTop_P$.}
\label{SectionAdjonctionQuillenSSUSSTop}

On a construit la catégorie modèle $\sSTop_P$ comme une catégorie modèle intermédiaire pour comparer $\Top_{N(P)}$ et $\sSU_P$. Par construction de $\sSTop_P$, on a une adjonction de Quillen 
\begin{equation*}
\RealNP{-}\colon\sSTop_P\leftrightarrow \Top_{N(P)}\colon \Sing_{N(P)}
\end{equation*}
Celle-ci fournit la première étape en direction d'une comparaison.
La proposition suivante fournit la seconde étape.

\begin{prop}\label{PropositionAdjonctionQuillenSSUSSTop}
On a une adjonction de Quillen 
\begin{equation*}
\sd_P\colon \sSU_P\leftrightarrow\sSTop_P\colon\Ex_P
\end{equation*}
\end{prop} 

\begin{proof}
Par \cite[Proposition 2.4.40]{Cisinski2}, il suffit de vérifier que $\sd_P$ envoie les monomorphismes sur des cofibrations de $\sSTop_P$, et que pour toute inclusion de cornet admissible, $\Lambda^{\varphi}_k\to\Delta^{\varphi}$, le morphisme $\sd_P(\Lambda^{\varphi}_k)\to\sd_P(\Delta^{\varphi})$ est une cofibration triviale de $\sSTop_P$. Par la remarque \ref{RemarqueSdPEnvoieMonoSurCofibration}, $\sd_P$ envoie les monomorphismes sur des cofibrations de $\sSTop_P$. Soit $\Lambda^{\varphi}_k\to\Delta^{\varphi}$ un cornet admissible. On a le diagramme commutatif suivant
\begin{equation*}
\begin{tikzcd}
\sd_P(\Lambda^{\varphi}_k)
\arrow{r}
\arrow[swap]{d}{\lv_P}
&\sd_P(\Delta^{\varphi})
\arrow{d}{\lv_P}
\\
\Lambda^{\varphi}_k
\arrow{r}
&\Delta^{\varphi}
\end{tikzcd}
\end{equation*}
Par la proposition \ref{LastVertexEquivalenceFaible}, les morphismes verticaux sont des équivalences faibles de $\sSTop_P$. De plus, par la proposition \ref{PropCornetAdmissible}, le morphisme $\Lambda^{\varphi}_k\to \Delta^{\varphi}$ est une équivalence d'homotopie filtrée. C'est donc une équivalence faible de $\sSTop_P$ car $\RealNP{-}$ préserve les équivalences d'homotopies filtrées.  Finalement, par deux sur trois on en déduit que $\sd_P(\Lambda^{\varphi}_k)\to\sd_P(\Delta^{\varphi})$ est une équivalence faible. Comme $\sd_P$ envoie monomorphisme sur cofibrations de $\sSTop_P$, c'est une cofibration triviale.
\end{proof}

\begin{corollaire}\label{CorollaireAdjonctionKanQuillen}
On a une adjonction de Quillen
\begin{equation*}
\RealNP{\sd_P(-)}\colon \sSU_P\leftrightarrow\Top_{N(P)}\colon \Ex_P\Sing_{N(P)}
\end{equation*}
\end{corollaire}

\section{Vers une équivalence de Kan-Quillen filtrée}
\label{SectionKanQuillen}
Après les résultats de la section précédentes, on aimerait montrer que l'adjonction de Quillen du corollaire \ref{CorollaireAdjonctionKanQuillen} entre $\Top_{N(P)}$ et $\sSU_P$ est une équivalence de Quillen. Ceci fournirait un contexte pour l'étude de la théorie de l'homotopie des espaces filtrés analogue à la situation non filtrée. Pour montrer ce résultat, il est naturel d'essayer de montrer que les adjonctions de Quillen de la proposition \ref{PropositionAdjonctionQuillenSSUSSTop} et du théorème \ref{TheoremeCategorieModeleSSTop} sont toutes les deux des équivalences de Quillen. On conjecture que c'est le cas. Dans la fin de ce chapitre, on présente certains arguments motivant ces conjectures.

\subsection{L'adjonction $(\RealNP{-},\Sing_{N(P)})$}
\label{SectionRealNPSingNP}
Par définition, un morphisme entre ensemble simpliciaux filtrés $f\colon\fil{A}\to\fil{B}$ est une équivalence faible de $\sSTop_P$ si et seulement si sa réalisation filtrée 
\begin{equation*}
\RealNP{f}\colon\RealNP{\fil{A}}\to\RealNP{\fil{B}}
\end{equation*}
est une équivalence faible de $\Top_{N(P)}$. Ainsi, pour montrer que l'adjonction 
\begin{equation*}
\RealNP{-}\colon\sSTop_P\leftrightarrow\Top_{N(P)}\colon\Sing_{N(P)}
\end{equation*}
est une équivalence de Quillen, il suffit de montrer que pour tout espace fortement filtré cofibrant $\fil{X}$, la counité de l'adjonction
\begin{equation*}
\epsilon_X\colon\RealNP{\Sing_{N(P)}\fil{X}}\to\fil{X}
\end{equation*}
est une équivalence faible de $\Top_{N(P)}$. 
Fixons un espace fortement filtré fibrant $\fil{X}$. On considère les morphismes
\begin{equation*}
\Sing_{N(P)}(\epsilon_X)\colon \Sing_{(N(P)}(\RealNP{\Sing_{N(P)}\fil{X}})\to\Sing_{N(P)}\fil{X}
\end{equation*}
et
\begin{equation*}
\eta_{\Sing_{N(P)}(X)}\colon\Sing_{N(P)}\fil{X}\to \Sing_{(N(P)}(\RealNP{\Sing_{N(P)}\fil{X}}).
\end{equation*}
Alors, par définition de l'unité et de la counité d'une adjonction, on a $\Sing_{N(P)}(\epsilon_X)\circ\eta_{\Sing_{N(P)}(X)}=\Id$. Ainsi, il suffit d'exhiber une homotopie filtrée
\begin{equation}\label{EquationHomotopieFiltreeSingRealSing}
H\colon \Delta^1\otimes \Sing_{(N(P)}(\RealNP{\Sing_{N(P)}\fil{X}})\to \Sing_{(N(P)}(\RealNP{\Sing_{N(P)}\fil{X}})
\end{equation}
entre $\eta_{\Sing_{N(P)}(X)}\circ\Sing_{N(P)}(\epsilon_X)$ et $\Id$ pour montrer que $\Sing_{N(P)}(\epsilon_X)$ est une équivalence d'homotopie filtrée.  De plus, si $\Sing_{N(P)}(\epsilon_X)$ est une équivalence d'homotopie filtrée, $D(\Sing_P(\epsilon_X))$ est une équivalence faible de $\DiagR_P$, et donc $\epsilon_X$ est une équivalence faible de $\Top_{N(P)}$. 
Par ailleurs, on remarque que le carré 
\begin{equation}\label{EquationSingNPPullback}
\begin{tikzcd}
\Sing_{N(P)}\fil{X}
\arrow{r}
\arrow{d}
&\Sing(X)
\arrow{d}{\Sing(\varphi_X)}
\\
N(P)
\arrow{r}{i}
&\Sing(\Real{N(P)})
\end{tikzcd}
\end{equation}
est cartésien (voir la proposition \ref{SingPPullback}, la preuve pour $\Sing_{N(P)}$ est similaire). On considère ensuite le diagramme commutatif suivant.
\begin{equation*}
\begin{tikzcd}
\Sing_{N(P)}(\RealNP{\Sing_{N(P)}\fil{X}})
\arrow{r}
\arrow{d}
&\Sing(\RealNP{\Sing_{N(P)}\fil{X}})
\arrow{d}
\arrow{r}
&\Sing\left(\Real{\Sing(X)}\right)
\arrow[swap]{d}{\Sing\left(\Real{\Sing(\varphi_X)}\right)}
\\
N(P)
\arrow{r}{i}
&\Sing(\Real{N(P)})
\arrow[swap]{r}{\Sing(\Real{i})}
&\Sing(\Real{\Sing(\Real{N(P)})})
\end{tikzcd}
\end{equation*}
Le carré de gauche est le carré \ref{EquationSingNPPullback} appliqué à $\RealNP{\Sing_{N(P)}\fil{X}}$, il est donc cartésien. Le carré de droite est l'image du carré \ref{EquationSingNPPullback} par le foncteur $\Sing(\Real{-})$. Comme les foncteurs $\Real{-}$ et $\Sing$ préservent les limites finies, il est lui aussi cartésien. On en déduit que le rectangle extérieur est cartésien.
On notera aussi $\eta$ l'unité, et $\epsilon$ la counité, de l'adjonction $(\Real{-},\Sing)$.
Supposons maintenant qu'il existe deux homotopies 
\begin{equation*}
H_X\colon \Sing(\Real{\Sing(X)})\times\Delta^1\to \Sing(\Real{\Sing(X)})
\end{equation*}
et
\begin{equation*}
H_P\colon \Sing(\Real{\Sing(\Real{N(P)})})\times\Delta^1\to \Sing(\Real{\Sing(\Real{N(P)})})
\end{equation*}
telles que
\begin{enumerate}
\item \label{EnumerateHXHomotopie} $H_X$ est une homotopie entre $\eta_{\Sing(X)}\circ \Sing(\epsilon_X)$ et $\Id$.
\item \label{EnumerateHPHomotopie} $H_P$ est une homotopie entre $\eta_{\Sing(\Real{N(P)})}\circ \Sing(\epsilon_{\Real{N(P)}})$ et $\Id$.
\item \label{EnumerateHPRestrictsToProjection} $(H_P)_{|N(P)}=\pr_{N(P)}\colon N(P)\times\Delta^1\to N(P)$
\item \label{EnumerateHomotopiesCommute} On a un diagramme commutatif
\begin{equation*}
\begin{tikzcd}
\Sing(\Real{\Sing(X)})\times\Delta^1
\arrow{r}{H_X}
\arrow[swap]{d}{\Sing\left(\Real{\Sing(\varphi_X)}\right)\times\Delta^1}
&\Sing(\Real{\Sing(X)})
\arrow{d}{\Sing\left(\Real{\Sing(\varphi_X)}\right)}
\\
\Sing(\Real{\Sing(\Real{N(P)})})\times\Delta^1
\arrow{r}{H_P} 
&\Sing(\Real{\Sing(\Real{N(P)})})
\end{tikzcd}
\end{equation*}
\end{enumerate}

A partir des morphismes $H_X$, $H_P$ et $\pr_{N(P)}$ on peut définir l'homotopie filtrée \ref{EquationHomotopieFiltreeSingRealSing}, en considérant le diagramme commutatif suivant
\begin{equation*}
\begin{tikzcd}[column sep = -30pt]
&\Sing_{N(P)}(\RealNP{\Sing_{N(P)}\fil{X}})
\arrow{rr}
\arrow{dd}
&&\Sing\left(\Real{\Sing(X)}\right)
\arrow[swap]{dd}
\\
\Sing_{N(P)}(\RealNP{\Sing_{N(P)}\fil{X}})\times\Delta^1
\arrow{rr}
\arrow[dashrightarrow]{ur}{H}
\arrow{dd}
&&\Sing\left(\Real{\Sing(X)}\right)\times\Delta^1
\arrow{ur}{H_X}
\arrow{dd}
\\
&N(P)
\arrow{rr}
&&\Sing(\Real{\Sing(\Real{N(P)})})
\\
N(P)\times\Delta^1
\arrow{rr}
\arrow{ur}{\pr_{N(P)}}
&&\Sing(\Real{\Sing(\Real{N(P)})})\times{\Delta^1}
\arrow{ur}{H_P}
\end{tikzcd}
\end{equation*}
En effet, les conditions \ref{EnumerateHPRestrictsToProjection} et \ref{EnumerateHomotopiesCommute} garantissent que le diagramme commute, et comme les faces avant et arrière de ce cubes sont cartésiennes, il existe une unique application $H$ faisant commuter le diagramme. Les conditions \ref{EnumerateHXHomotopie} et \ref{EnumerateHPHomotopie} garantissent ensuite que le morphisme $H$ ainsi obtenu est une homotopie entre $\eta_{\Sing_{N(P)}(X)}\circ\Sing_{N(P)}(\epsilon_X)$ et $\Id$.

Finalement, il suffit d'exhiber $H_X$ et $H_P$ vérifiant les conditions précédentes. D'après \cite[Proposition 4.5.29]{CellularStructure} il existe $H_X$ et $H_P$ vérifiant les hypothèses \ref{EnumerateHXHomotopie}, \ref{EnumerateHPHomotopie} et \ref{EnumerateHPRestrictsToProjection}. Il suffirait de montrer qu'on peut choisir $H_X$ de façon à aussi avoir \ref{EnumerateHomotopiesCommute} pour conclure, et obtenir le résultat suivant :

\begin{conjecture}\label{ConjectureRealNPSingNP}
L'adjonction de Quillen
\begin{equation*}
\RealNP{-}\colon\sSTop_P\leftrightarrow\Top_{N(P)}\colon\Sing_{N(P)}
\end{equation*}
est une équivalence de Quillen.
\end{conjecture}

\begin{reduction}
Par la discussion précédente, il suffit d'exhiber des homotopies non filtrées $H_X$ et $H_P$ vérifiant les propriétés \ref{EnumerateHXHomotopie}, \ref{EnumerateHPHomotopie}, \ref{EnumerateHPRestrictsToProjection} et \ref{EnumerateHomotopiesCommute}.
\end{reduction}

\subsection{L'adjonction $(\sd_P,\Ex_P)$}
\label{SectionSdPExP}
De façon à distinguer les équivalences faibles de $\sSTop_P$ et $\sSU_P$, on parlera respectivement de $\Top$-équivalence faible et de $\U$-équivalence faible. On rappelle le résultat suivant :

\begin{prop}\label{PropositionUEquivalenceFaibleLVBeta}
Soit $\fil{X}\in \sS_P$, les morphismes
\begin{equation*}
\lv_P\colon \sd_P\fil{X}\to \fil{X}
\end{equation*}
et
\begin{equation*}
\beta_X\colon \fil{X}\to\Ex_P\fil{X}
\end{equation*}
sont des $\U$-équivalences faibles.
\end{prop}

\begin{proof}
Par les lemmes \ref{LemmeLastVertexUEquivalenceFaible} et \ref{SDSimplexeEquivalenceFaibleAbsolue}, $\lv_P$ est une équivalence faible, et par le théorème \ref{TheoremeExXFaiblementEquivalentaX} $\beta_X$ est une équivalence faible.
\end{proof}

On a aussi les propositions suivantes, qui sont des conséquences de l'adjonction $(\sd_P,\Ex_P)$ :

\begin{prop}\label{PropositionUequivalenceTopEquivalence}
Toute $\U$-équivalence faible est une $\Top$-équivalence faible.
\end{prop}

\begin{proof}
Soit $f\colon \fil{X}\to\fil{Y}$ une $\U$-équivalence faible.
On a le diagramme commutatif suivant
\begin{equation*}
\begin{tikzcd}
\sd_P\fil{X}
\arrow{r}{\sd_P(f)}
\arrow[swap]{d}{\lv_P}
&\sd_P\fil{Y}
\arrow{d}{\lv_P}
\\
\fil{X}
\arrow[swap]{r}{f}
&\fil{Y}
\end{tikzcd}
\end{equation*}
Par la proposition \ref{LastVertexEquivalenceFaible}, les morphismes de gauche et de droite sont des $\Top$-équivalences faibles. Comme l'adjonction $(\sd_P,\Ex_P)$ est une adjonction de Quillen, $\sd_P$ envoie les $\U$-équivalences faibles entre objets cofibrants sur des $\Top$-équivalences faibles. On en déduit que $\sd_P(f)$ est une $\Top$-équivalence faible. Par deux sur trois, on a donc que $f$ est une $\Top$-équivalence faible.
\end{proof}

\begin{prop}
Une $\Top$-équivalence faible entre objets fibrants de $\sSTop_P$ est une $\U$-équivalence faible.
\end{prop}

\begin{proof}
Soient $\fil{X}$ et $\fil{Y}$ deux objets fibrants de $\sSTop_P$ et $f\colon \fil{X}\to\fil{Y}$ une $\Top$-équivalence faible. Comme $\Ex_P$ est un foncteur de Quillen  à droite, il préserve les équivalences faibles entre fibrants. Ainsi, $\Ex_P(f)\colon \Ex_P\fil{X}\to\Ex_P\fil{Y}$ est une $\U$-équivalence faible. Or, on a le diagramme commutatif
\begin{equation*}
\begin{tikzcd}
\fil{X}
\arrow{r}{f}
\arrow[swap]{d}{\beta_X}
&\fil{Y}
\arrow{d}{\beta_Y}
\\
\Ex_P\fil{X}
\arrow[swap]{r}{\Ex_P(f)}
&\Ex_P\fil{Y}
\end{tikzcd}
\end{equation*}
Comme $\beta_X$ et $\beta_Y$ sont des $\U$-équivalences faibles, on en déduit par deux sur trois que $f$ est une $\U$-équivalence faible.
\end{proof}
Notons $\eta$ et $\epsilon$ l'unité et la counité de l'adjonction $(\sd_P,\Ex_P)$. Les observations précédentes nous permettent de montrer la proposition suivante.
\begin{prop}\label{PropositionUniteCouniteEquivalenceFaible}
Soient $\fil{X}$ et $\fil{Y}$ deux ensembles simpliciaux filtrés. Les morphismes
\begin{equation*}
\eta_X\colon\fil{X}\to\Ex_P(\sd_P\fil{X})
\end{equation*}
et
\begin{equation*}
\epsilon_Y\colon \sd_P(\Ex_P\fil{Y})\to\fil{Y}
\end{equation*}
sont des $\U$-équivalences faibles.
\end{prop}

\begin{proof}
Soit $\fil{X}$ un ensemble simplicial filtré. On considère le diagramme commutatif 
\begin{equation*}
\begin{tikzcd}
\fil{X}
\arrow{rr}{\beta_X}
\arrow[swap]{dr}{\eta_X}
&&\Ex_P\fil{X}
\\
&\Ex_P(\sd_P\fil{X})
\arrow[swap]{ur}{\Ex_P(\lv_P)}
\end{tikzcd}
\end{equation*}
Par la proposition \ref{PropositionUEquivalenceFaibleLVBeta}, les morphismes $\beta_X$ et $\lv_P$ sont des $\U$-équivalences faibles. De plus, comme $\Ex_P$-préserve les $\U$-équivalences faibles on en déduit que $\Ex_P(\lv_P)$ est une $\U$-équivalence faible, et donc que $\eta_X$ est une $\U$-équivalence faible par deux sur trois.
Soit $\fil{Y}$ un ensemble simplicial filtré, on considère le diagramme commutatif
\begin{equation*}
\begin{tikzcd}
\sd_P\fil{Y}
\arrow{rr}{\lv_P}
\arrow[swap]{dr}{\sd_P(\beta_Y)}
&&\fil{Y}
\\
&\sd_P(\Ex_P\fil{Y})
\arrow[swap]{ur}{\epsilon_Y}
\end{tikzcd}
\end{equation*}
Par la proposition \ref{PropositionUEquivalenceFaibleLVBeta}, les morphismes $\beta_Y$ et $\lv_P$ sont des $\U$-équivalences faibles. Comme $\lv_P$ est une $\U$-équivalence faible, par deux sur trois $\sd_P(\beta_Y)$ est une $\U$-équivalence faible.
on en déduit par deux sur trois que $\epsilon_Y$ est une $\U$-équivalence faible.
\end{proof}

On a montré que l'unité et la counité de l'adjonction $(\sd_P,\Ex_P)$ étaient des équivalences faibles. Cependant, cela ne permet pas de déduire directement que l'adjonction $(\sd_P,\Ex_P)$ est une équivalence de Quillen (Il faudrait montrer le même résultat pour l'unité et la counité de l'adjonction dérivée). Examinons cependant ce que la proposition \ref{PropositionUniteCouniteEquivalenceFaible} permet de dire quant à l'adjonction entre $\sSU_P$ et $\sSTop_P$.
Pour montrer que $(\sd_P,\Ex_P)$ est une équivalence de Quillen entre $\sSU_P$ et $\sSTop_P$, il suffit de montrer que pour toute paire  d'ensembles simpliciaux filtrés $(\fil{X},\fil{Y})$ où $\fil{Y}$ est un objet fibrant de $\sSTop_P$, et tout morphisme $f\colon \fil{X}\to\Ex_P\fil{Y}$, $f$ est une $\U$ équivalence faible si et seulement si son image par l'adjonction $\widehat{f}\colon \sd_P\fil{X}\to\fil{Y}$ est une $\Top$-équivalence faible.
Le sens direct se montre sans difficulté. Si $f$ est une $\U$-équivalence faible,  on considère le diagramme commutatif suivant
\begin{equation*}
\begin{tikzcd}
\sd_P\fil{X}
\arrow{rr}{\widehat{f}}
\arrow[swap]{dr}{\sd_P(f)}
&&\fil{Y}
\\
&\sd_P(\Ex_P\fil{Y})
\arrow[swap]{ur}{\epsilon_Y}
\end{tikzcd}
\end{equation*}
Par la proposition \ref{PropositionUniteCouniteEquivalenceFaible} $\epsilon_Y$ est une $\U$-équivalence faible. C'est donc une $\Top$-équivalence faible par la proposition \ref{PropositionUequivalenceTopEquivalence}. D'autre part, $\sd_P(f)$ est l'image d'une $\U$-équivalence faible par le foncteur de Quillen à gauche $\sd_P$, c'est donc une $\Top$-équivalence faible. Par deux sur trois, on en déduit que $\widehat{f}$ 
est une $\Top$-équivalence faible.

Réciproquement, supposons que $\widehat{f}$ est une $\Top$-équivalence faible. On a le diagramme commutatif
\begin{equation}\label{DiagrammeAdjointUEquivalence}
\begin{tikzcd}
\fil{X}
\arrow{rr}{f}
\arrow[swap]{dr}{\eta_X}
&&\Ex_P\fil{Y}
\\
&\Ex_P(\sd_P\fil{X})
\arrow[swap]{ur}{\Ex_P(\widehat{f})}
\end{tikzcd}
\end{equation}
Par la proposition \ref{PropositionUniteCouniteEquivalenceFaible} $\eta_X$ est une $\U$-équivalence faible. Par deux sur trois, il suffit donc de montrer que $\Ex_P(\widehat{f})$ est une $\U$-équivalence faible pour obtenir que $f$ est une $\U$-équivalence faible. Or, par le corollaire \ref{CorollaireBetaInfiniEquivalenceFaible}, $\Ex_P(\widehat{f})$ est une $\U$-équivalence faible si et seulement si 
\begin{equation*}
\Exi_P(\Ex_P(\widehat{f}))\colon\Exi_P(\Ex_P(\sd_P\fil{X}))\to\Exi_P(\Ex_P\fil{Y})
\end{equation*}
est une $\U$-équivalence faible, et on remarque que $\Exi_P\Ex_P=\Exi_P=\Ex_P\Exi_P$. En particulier, $\Exi_P(\Ex_P(\widehat{f}))=\Ex_P(\Exi_P(\widehat{f}))$. De plus, par le corollaire \ref{CorollaireBetaInfiniEquivalenceFaible} $\beta^{\infty}$ est une $\U$-équivalence faible objet par objet, c'est aussi une $\Top$-équivalence objet par objet. On considère le diagramme commutatif suivant.
\begin{equation*}
\begin{tikzcd}
\Ex_P(\sd_P\fil{X})
\arrow{r}{\Ex_P(\widehat{f})}
\arrow[swap]{d}{\beta_{\Ex_P(\sd_P(X))}}
&\Ex_P\fil{Y}
\arrow{d}{\beta_Y}
\\
\Exi_P(\sd_P\fil{X})
\arrow[swap]{r}{\Exi_P(\widehat{f})}
&\Exi_P\fil{Y}
\end{tikzcd}
\end{equation*}
Par deux sur trois, on a que $\Exi_P(\widehat{f})$ est une $\Top$-équivalence faible. Finalement, il reste à montrer que $\Ex_P(\Exi_P(\widehat{f}))$ est une $\U$-équivalence faible. Or on sait que $\Ex_P$ envoie $\Top$-équivalences faibles sur $\U$-équivalences faibles à condition que les domaines et codomaines soient fibrants dans $\sSTop_P$. Ainsi, il suffit de montrer que pour tout $X$, $\Exi_P\fil{X}$ est fibrant dans $\sSTop_P$. Comme on sait que $\Exi_P\fil{X}$ est fibrant dans $\sSU_P$, il suffirait en particulier de montrer que tout objet fibrant dans $\sSU_P$ est fibrant dans $\sSTop_P$. Cette dernière assertion est équivalente à ce que la classe des cofibrations triviales de $\sSTop_P$ soit contenue dans la classe des cofibrations triviales de $\sSU_P$. On note que c'est le cas des cofibrations triviales provenant de $\DiagR_P$. En effet, par le théorème \ref{CategorieModelSimpliciale} la catégorie $\sSU_P$ est simpliciale. Ceci implique, par \cite[Proposition II.3.4]{GoerssJardine} que $\Colim$ envoie les cofibrations triviales génératrices de $\DiagR_P$ sur des cofibrations triviales de $\sSU_P$. Cependant, on ne sait pas si la classe des cofibrations triviales de $\DiagR_P$ génère les cofibrations triviales de $\sSTop_P$. Si c'est le cas, on a immédiatement le résultat suivant :

\begin{conjecture}\label{ConjectureSdPExP}
L'adjonction
\begin{equation*}
\sd_P\colon\sSU_P\leftrightarrow\sSTop_P\colon\Ex_P
\end{equation*}
est une équivalence de Quillen.
\end{conjecture}

\begin{reduction}
Par la discussion précédente, il suffit de montrer que toute  cofibration triviale de $\sSTop_P$ est une cofibration triviale de $\sSU_P$. 
\end{reduction}

\subsection{Comparaison entre $\sSU_P$ et $\Top_{N(P)}$}
\label{SectionSingNPRealNPPreserventEquivalences}
Bien qu'on ne sache pas (encore) montrer qu'il existe une équivalence de Quillen entre $\Top_{N(P)}$ et $\sSU_P$, on sait tout de même comparer les catégories homotopiques correspondantes. En effet, si un morphisme entre ensembles simpliciaux filtrés
\begin{equation*}
f\colon\fil{A}\to\fil{B}
\end{equation*}
est une $\U$-équivalence faible, alors, par la proposition \ref{PropositionUequivalenceTopEquivalence}, c'est aussi une $\Top$-équivalence faible. Mais par définition des équivalences faibles de $\sSTop_P$ ceci implique que la réalisation fortement filtrée
\begin{equation*}
\RealNP{f}\colon\RealNP{\fil{A}}\to\RealNP{\fil{B}},
\end{equation*}
est une équivalence faible de $\Top_{N(P)}$. Dans l'autre sens, si une application filtrée
\begin{equation*}
g\colon\fil{X}\to\fil{Y}
\end{equation*}
est une équivalence faible dans $\Top_{N(P)}$, alors le morphisme d'ensembles simpliciaux filtrés
\begin{equation*}
\Sing_{N(P)}(g)\colon\Sing_{N(P)}\fil{X}\to\Sing_{N(P)}\fil{Y}
\end{equation*}
est une équivalence faible de $\sSU_P$.
En effet, par construction de la catégorie modèle $\sSTop_P$ , $\Sing_{N(P)}$ est un foncteur de Quillen à droite (voir le théorème \ref{TheoremeCategorieModeleSSTop}). Comme tout les objets de $\Top_{N(P)}$ sont fibrants, ceci implique que  $\Sing_{N(P)}(g)$ est une équivalence faible entre objets fibrants de $\sSTop_P$. Ainsi, $\Ex_P(\Sing_{N(P)}(g))$ est une $\U$-équivalence faible. Or on a le diagramme commutatif suivant
\begin{equation*}
\begin{tikzcd}[column sep= 65pt]
\Sing_{N(P)}\fil{X}
\arrow[swap]{d}{\beta_{\Sing_{N(P)}(X)}}
\arrow{r}{\Sing_{N(P)}(g)}
&\Sing_{N(P)}\fil{Y}
\arrow{d}{\beta_{\Sing_{N(P)}(Y)}}
\\
\Ex_P(\Sing_{N(P)}\fil{X})
\arrow[swap]{r}{\Ex_P(\Sing_{N(P)}(g))}
&\Ex_P(\Sing_{N(P)}\fil{Y})
\end{tikzcd}
\end{equation*}
où les morphismes verticaux sont des $\U$-équivalences faibles. On en déduit que $\Sing_{N(P)}(g)$ est une $\U$-équivalence faible. Finalement, on a

\begin{theo}\label{RealisationSingPreserventEquivalencesFaibles}
L'adjonction
\begin{equation*}
\RealNP{-}\colon\sSU_P\leftrightarrow\Top_{N(P)}\colon\Sing_{N(P)}
\end{equation*}
préserve les équivalences faibles.
\end{theo}
 \chapter{Retour aux espaces stratifiés}
\label{ChapitreCMFStrat}
Jusqu'ici, nous avons fixé un ensemble ordonné $P$, et considéré des objets filtrés au-dessus de $P$. L'objet de ce chapitre est d'étudier comment décrire une catégorie modèle d'objets stratifiés à l'aide des résultats obtenus sur les catégories modèles d'objets filtrés. 

Dans la première section de ce chapitre, on examine comment la donnée d'un morphisme d'ensembles ordonnés $\alpha\colon P\to Q$ permet de comparer les catégories d'objets (fortement) filtrés au dessus de $P$ et de $Q$. On montre que $\alpha$ permet de définir une adjonction de Quillen $(\alpha_*,\alpha^*)$ entre $\sSU_P$ et $\sSU_Q$, (Proposition \ref{AlphaAdjonctionQuillen}), ainsi que entre $\Top_{N(P)}$ et $\Top_{N(Q)}$ et entre $\Top_P$ et $\Top_Q$ (Proposition \ref{PropositionAlphaAdjonctionTop} et Remarque \ref{RemarqueAlphaStarTopP}). Finalement, dans la section \ref{SectionAlphaSingReal}, on examine les relations de compatibilités entre les foncteurs induits par $\alpha$ et les foncteurs $\Sing_{N(P)}$ et $\RealNP{-}$.

Dans la section \ref{SectionStrat}, on montre que la donnée des catégories modèles $\Top_P$, $\Top_{N(P)}$ et $\sS_P$ permet de définir des catégories modèles d'objets stratifiés : $\Strat$, $\FStrat$ et $\sStrat$. C'est le contenu du théorème \ref{TheoremeCMFStrat} qui est une application de \cite[Theorem 4.2]{Cagne}. On montre aussi (proposition \ref{PropositionAdjonctionRealStratSingStrat}) que les relations de compatibilités de la proposition \ref{PropositionCompatibiliteSingAlpha} permettent d'obtenir une adjonction :
\begin{equation*}
\Real{-}_{\FStrat}\colon\sStrat\leftrightarrow \FStrat\colon\Sing_{\FStrat}.
\end{equation*}
Bien que ce ne soit pas une adjonction de Quillen, on montre que cette adjonction préserve les équivalences faibles, c'est le théorème \ref{TheoremeRealStratSingStratPreserventEquivalenceFaible}.

Finalement, dans la section \ref{SectionAutresCMF}, on présente deux autres structures de modèles pour les espaces stratifiés. Dans la section \ref{SectionNandLal}, on considère la structure de modèle sur $\Strat$ introduit par Nand-Lal dans \cite{NandLal}, et dans la section \ref{SectionHaine}, on s'intéresse à la catégorie modèle $\sSJK_P$ de \cite{Haine}.
\section{Changer l'ensemble ordonné}
Pour toute cette section, on fixe un morphisme d'ensembles ordonnés $\alpha\colon P\to Q$. Demander que $\alpha$ soit un morphisme d'ensembles ordonnés revient à imposer que \begin{equation*}
p_0\leq p_1\in P\Rightarrow \alpha(p_0)\leq \alpha(p_1)\in Q.
\end{equation*}
Nous allons étudier comment $\alpha$ permet de comparer les catégories de modèle d'objets filtrés sur $P$ et sur $Q$.

\subsection{Les ensembles simpliciaux filtrés}

On remarque que si $(X,\varphi_X\colon X\to N(P))$ est un ensemble simplicial filtré sur $P$, $(X,N(\alpha)\circ\varphi_X\colon X\to N(Q))$ est un ensemble simplicial filtré sur $Q$. Ceci permet de définir un foncteur.
\begin{defin}
On définit le foncteur
\begin{align*}
\alpha_*\colon \sSU_P&\to\sSU_Q\\
\fil{X}&\mapsto (X,N(\alpha)\circ\varphi_X)\\
f\colon\fil{X}\to\fil{Y}&\mapsto f
\end{align*}
\end{defin}
De même, si $(Y,\varphi_Y\colon Y\to N(Q))$ est un ensemble simplicial filtré sur $Q$, alors en considérant le produit fibré
\begin{equation*}
\begin{tikzcd}
Y\times_{N(Q)}N(P)
\arrow{r}
\arrow{d}{\alpha^*\varphi_Y}
&
Y
\arrow{d}{\varphi_Y}
\\
N(P)
\arrow{r}{N(\alpha)}
&N(Q)
\end{tikzcd}
\end{equation*}
On obtient un ensemble simplicial filtré sur $N(P)$ $(Y\times_{N(Q)}N(P),\alpha^*\varphi_Y\colon Y\times_{N(Q)}N(P)\to N(P))$.
On obtient ainsi un foncteur
\begin{defin}\label{DefinitionAlphaUpperStarSimplicial}
On définit le foncteur
\begin{align*}
\alpha^*\colon \sSU_Q&\to\sSU_P\\
\fil{Y}&\mapsto (Y\times_{N(Q)}N(P),\alpha^*\varphi_Y)
\end{align*}
en définissant $\alpha^*$ sur les morphismes par la propriété universelle du produit fibré.
\end{defin}
On a immédiatement la propriété suivante
\begin{prop}\label{AlphaAdjonction}
La paire $(\alpha_*,\alpha^*)$ est une paire de foncteurs adjoints.
\end{prop}

\begin{proof}
Soient $\fil{X}$ un ensemble simplicial filtré sur $P$ et $\fil{Y}$ un ensemble simplicial filtré sur $Q$. Un morphisme $f\colon\alpha_*\fil{X}\to\fil{Y}$ correspond à la donnée d'un diagramme commutatif
\begin{equation}\label{DiagrammeAlphaEtoile}
\begin{tikzcd}
X
\arrow{r}{f}
\arrow[swap]{d}{\varphi_X}
&Y
\arrow{d}{\varphi_Y}
\\
N(P)
\arrow[swap]{r}{N(\alpha)}
&N(Q)
\end{tikzcd}
\end{equation}
Par la propriété universelle du produit fibré, les carrés commutatifs de la forme \ref{DiagrammeAlphaEtoile} sont en bijections avec les morphismes $\fil{X}\to\alpha^*\fil{Y}$.
\end{proof}

\begin{prop}\label{AlphaAdjonctionQuillen}
L'adjonction $(\alpha_*,\alpha^*)$ est une adjonction de Quillen. De plus, c'est une équivalence de Quillen si et seulement si $\alpha$ est un isomorphisme d'ensemble ordonné.
\end{prop}

\begin{proof}
Le foncteur $\alpha_*$ préserve les monomorphismes par constructions, il préserve donc les cofibrations. Soit 
$f\colon\fil{Y}\to\fil{Z}$ une fibration de $\sSU_Q$.
Montrons que $\alpha^*f$ est une fibration de $\sSU_P$. Soit $j\colon \Lambda^{\varphi}_k\to\Delta^{\varphi}$ une inclusion de cornet admissible de $\sS_P$. On considère le problème de relèvement
\begin{equation*}
\begin{tikzcd}
\Lambda^{\varphi}_k
\arrow[swap]{d}{j}
\arrow{r}
&\alpha^*\fil{Y}
\arrow{d}{\alpha^*f}
\\
\Delta^{\varphi}
\arrow{r}
\arrow[dashrightarrow]{ur}{h}
&\alpha^*\fil{Z}
\end{tikzcd}
\end{equation*}
Par adjonction, ce problème de relèvement est équivalent à
\begin{equation*}
\begin{tikzcd}
\alpha_*\Lambda^{\varphi}_k
\arrow[swap]{d}{\alpha_*j}
\arrow{r}
&\fil{Y}
\arrow{d}{f}
\\
\alpha_*\Delta^{\varphi}
\arrow{r}
\arrow[dashrightarrow]{ur}{\widehat{h}}
&\fil{Z}
\end{tikzcd}
\end{equation*}
Or, $\alpha_*\Delta^{\varphi}=\Delta^{N(\alpha)\circ\varphi}$. De plus, comme l'inclusion de cornet $j$ est admissible, on a $\varphi(e_k)=\varphi(e_{k-1})$ ou $\varphi(e_{k+1})$. En particulier, $N(\alpha)\circ\varphi(e_k)=N(\alpha)\circ\varphi(e_{k-1})$ ou $N(\alpha)\circ\varphi(e_{k+1})$ et l'inclusion de cornet $\alpha_*j$ est admissible. Comme $f$ est une fibration de $\sSU_Q$, on en déduit qu'il existe un relèvement $\widehat{h}$, et donc par adjonction, il existe un relèvement $h$. Ainsi, $\alpha_*$ préserve les cofibrations et $\alpha^*$ préserve les fibrations, l'adjonction $(\alpha_*,\alpha^*)$ est donc une adjonction de Quillen.

Si $\alpha$ est un isomorphisme, on a $\alpha_*\circ\alpha^*\simeq\alpha^*\circ\alpha_*\simeq\Id$ et l'adjonction de Quillen est une équivalence de Quillen. Réciproquement, supposons que $\alpha\colon P\to Q$ induit une équivalence de Quillen $(\alpha_*,\alpha^*)$. On considère $(N(P),\Id)$ l'ensemble simplicial filtré sur $P$ et $(N(Q),\Id)$ l'ensemble simplicial filtré sur $Q$. On remarque que la propriété universelle du produit fibré fournit un isomorphisme canonique $(N(P),\Id)\simeq \alpha^*(N(Q),\Id)$. Par adjonction, ce morphisme devient
\begin{equation*}
N(\alpha)\colon(N(P),N(\alpha))\to (N(Q),\Id).
\end{equation*}
Comme $(\alpha_*,\alpha^*)$ est une équivalence de Quillen par hypothèse, $N(\alpha)$ est une équivalence faible de $\sSU_Q$. Par la proposition \ref{RealisationSingPreserventEquivalencesFaibles}, le morphisme
\begin{equation*}
\Real{N(\alpha)}_{N(Q)}\colon \Real{(N(P),N(\alpha))}_{N(Q)}\to\Real{(N(Q),\Id)}_{N(Q)}
\end{equation*}
est une équivalence faible de $\Top_{N(Q)}$. En particulier, pour tout $q\in Q$, le morphisme
\begin{equation*}
s\pi_0(\Real{N(\alpha)}_{N(Q)})([q])\colon s\pi_0(\Real{(N(P),N(\alpha))}_{N(Q)})([q])\to s\pi_0(\Real{(N(Q),\Id)}_{N(Q)})([q]).
\end{equation*}
est un isomorphisme. Or $s\pi_0(\Real{(N(Q),\Id)}_{N(Q)})([q])=\{[q]\}$. On en déduit qu'il existe un unique $p\in P$ tel que $\alpha(p)=q$. Donc, $\alpha$ est une bijection. Soit maintenant $q_0<q_1\in Q$. comme précédement, on a un isomorphisme
\begin{equation*}
s\pi_0(\Real{N(\alpha)}_{N(Q)})([q_0,q_1])\colon s\pi_0(\Real{(N(P),N(\alpha))}_{N(Q)})([q_0,q_1])\to s\pi_0(\Real{(N(Q),\Id)}_{N(Q)})([q_0,q_1]).
\end{equation*}
On déduit de même que si $p_0$ et $p_1$ sont les préimages par $\alpha$ de $q_0$ et $q_1$ respectivement, on a $p_0<p_1$. Finalement, $\alpha$ est un isomorphisme d'ensembles ordonnés.
\end{proof}

\subsection{Les espaces fortement filtrés}
Comme précédemment, on peut définir les foncteurs $\alpha_*$ et $\alpha^*$ entre les catégories $\Top_{N(P)}$ et $\Top_{N(Q)}$.

\begin{defin}
on définit le foncteur
\begin{align*}
\alpha_*\colon \Top_{N(P)}&\to\Top_{N(Q)}\\
\fil{X}&\mapsto (X,\Real{N(\alpha)}\circ\varphi_X)\\
f\colon\fil{X}\to\fil{Y}&\mapsto f
\end{align*}
\end{defin}
\begin{defin}\label{DefinitionAlphaUpperStarTopologique}
On définit le foncteur
\begin{align*}
\alpha^*\colon \Top_{N(Q)}&\to\Top_{N(P)}\\
\fil{Y}&\mapsto (Y\times_{\Real{N(Q)}}\Real{N(P)},\alpha^*\varphi_Y)
\end{align*}
en définissant $\alpha^*$ sur les morphismes par la propriété universelle du produit fibré.
\end{defin}

On a encore le résultat suivant.
\begin{prop}\label{PropositionAlphaAdjonctionTop}
La paire $(\alpha_*,\alpha^*)$ forme une adjonction de Quillen. C'est une équivalence de Quillen si et seulement si $\alpha$ est un isomorphisme d'ensembles ordonnés.
\end{prop}

\begin{proof}
La paire $(\alpha_*,\alpha^*)$ est une paire de foncteurs adjoints par les mêmes arguments que pour la proposition \ref{AlphaAdjonction}. Montrons que c'est une adjonction de Quillen. Il suffit de vérifier que les cofibrations (triviales) génératrices de $\Top_{N(P)}$ sont envoyées sur des cofibrations (triviales) par $\alpha_*$. Les cofibrations triviales génératrices de $\Top_{N(P)}$ sont de la forme
\begin{equation*}
j\colon\Real{K}\otimes\RealNP{\Delta^{\varphi}}\to\Real{L}\otimes\RealNP{\Delta^{\varphi}}
\end{equation*}
où $K\to L$ est une cofibration (triviale) entre ensembles simpliciaux et $\Delta^{\varphi}\in R(P)$ est un simplexe non dégénéré de $N(P)$. On remarque que 
\begin{equation*}
\alpha_*\RealNP{\Delta^{\varphi}}\simeq\Real{\alpha_*\Delta^{\varphi}}_{N(Q)}\simeq\Real{\Delta^{N(\alpha)\circ\varphi}}
\end{equation*}
Ainsi, il suffit de montrer que pour tout $\Delta^{\psi}\in \Delta(Q)$, et pour toute cofibration (triviale) de $\sS$, $K\to L$ le morphisme
\begin{equation}\label{EquationMorphismeTopNQCofibrationTriviale}
\Real{K}\otimes\Real{\Delta^{\psi}}_{N(Q)}\to\Real{L}\otimes\Real{\Delta^{\psi}}_{N(Q)}
\end{equation}
est une cofibration (triviale) de $\Top_{N(Q)}$. Or, on remarque que pour tout $\Delta^{\psi}=(\Delta^m,\psi)\in \Delta(Q)$, $\Delta^{\psi}$ est un rétracte de $\Delta^m\otimes\Delta^{bar(\psi)}$.
Ainsi, le morphisme \ref{EquationMorphismeTopNQCofibrationTriviale} est un rétracte de
\begin{equation*}
\Real{K\times\Delta^m}\otimes\Real{\Delta^{\bar{\psi}}}_{N(Q)}\to\Real{L\times\Delta^m}\otimes\Real{\Delta^{\bar{\psi}}}_{N(Q)}
\end{equation*}
c'est donc une cofibration (triviale), et l'adjonction $(\alpha_*,\alpha^*)$ est donc une adjonction de Quillen. On prouve que l'adjonction de Quillen est une équivalence de Quillen si et seulement si $\alpha$ est un isomorphisme d'ensembles ordonnés par les mêmes arguments qu'à la proposition \ref{AlphaAdjonctionQuillen}
\end{proof}

\begin{remarque}\label{RemarqueAlphaStarTopP}
La preuve précédente s'adapte immédiatement pour montrer que $\alpha\colon P\to Q$ induit une adjonction de Quillen
\begin{equation*}
\alpha_*\colon \Top_P\leftrightarrow\Top_Q\colon\alpha^*.
\end{equation*}
\end{remarque}

\subsection{Compatibilité avec les foncteurs $\Sing_{N(P)}$ et $\RealNP{-}$}
\label{SectionAlphaSingReal}
On a le résultat de compatibilité suivant.

\begin{prop}\label{PropositionCompatibiliteSingAlpha}
On a des isomorphismes de foncteurs, naturels en $\alpha\colon P\to Q$
\begin{align*}
\Sing_{N(P)}\circ\alpha^*\simeq \alpha^*\circ\Sing_{N(Q)}&\colon \Top_{N(Q)}\to\sS_P\\
\alpha_*\RealNP{-}\simeq \Real{\alpha_*(-)}_{N(Q)}&\colon \sS_P\to\Top_{N(Q)}\\
\RealNP{\alpha^*(-)}\simeq \alpha^*\Real{-}_{N(Q)}&\colon \sS_Q\to\Top_{N(P)}
\end{align*}
\end{prop}

\begin{proof}
Soit $\fil{Y}$ un espace fortement filtré au-dessus de $Q$. On considère les deux diagrammes commutatifs suivants
\begin{equation}\label{Diagramme1CommutationAlphaSing}
\begin{tikzcd}
\alpha^*\Sing_{N(Q)}\fil{Y}
\arrow{d}
\arrow{r}
&\Sing_{N(Q)}\fil{Y}
\arrow{r}
\arrow{d}
&\Sing(Y)
\arrow{d}{\Sing(\varphi_Y)}
\\
N(P)
\arrow[swap]{r}{N(\alpha)}
&N(Q)
\arrow{r}
&\Sing(\Real{N(Q)})
\end{tikzcd}
\end{equation}
et
\begin{equation}\label{Diagramme2CommutationAlphaSing}
\begin{tikzcd}[column sep = 45pt]
\Sing_{N(P)}(\alpha^*\fil{Y})
\arrow{d}
\arrow{r}
&\Sing(\alpha^*\fil{Y})
\arrow{r}
\arrow{d}
&\Sing(Y)
\arrow{d}{\Sing(\varphi_Y)}
\\
N(P)
\arrow{r}
&\Sing(\Real{N(P)})
\arrow[swap]{r}{\Sing(\Real{N(\alpha)})}
&\Sing(\Real{N(Q)})
\end{tikzcd}
\end{equation}
Par définition de $\Sing_{N(Q)}$, le carré de droite du diagramme \ref{Diagramme1CommutationAlphaSing} est cartésien. De plus, le carré de gauche du même diagramme est cartésien par définition de $\alpha^*$. Ainsi, le carré extérieur du diagramme \ref{Diagramme1CommutationAlphaSing} est cartésien. D'autre part, comme $\Sing$ préserve les limites, et par définition de $\alpha^*$, le carré de droite du diagramme \ref{Diagramme2CommutationAlphaSing} est cartésien. De plus, par définition de $\Sing_{N(P)}$ le carré de gauche du même diagramme est cartésien, et donc le rectangle extérieur est cartésien. Finalement, par la propriété universelle du produit fibré, on en déduit qu'il existe un isomorphisme naturel 
\begin{equation*}
\alpha^*\Sing_{N(Q)}\fil{Y}\simeq\Sing_{N(P)}(\alpha^*\fil{Y})
\end{equation*}
Soit $\fil{A}$ un ensemble simplicial filtré sur $P$. Par définition de $\RealNP{-}$ et $\Real{-}_{N(Q)}$, $\Id_{\Real{A}}$ fournit un isomorphisme 
\begin{equation*}
\alpha_*\RealNP{\fil{A}}\simeq (\Real{A},\Real{N(\alpha)})\simeq\Real{\alpha_*\fil{A}}_{N(Q)}
\end{equation*}
Soit $\fil{B}$ un ensemble simplicial filtré sur $Q$. Comme le foncteur $\Real{-}$ préserve les limites finies, le carré
\begin{equation*}
\begin{tikzcd}
\Real{\alpha^*\fil{B}}
\arrow{d}
\arrow{r}
&\Real{B}
\arrow{d}{\Real{\varphi_B}}
\\
\Real{N(P)}
\arrow{r}{\Real{N(\alpha)}}
&\Real{N(Q)}
\end{tikzcd}
\end{equation*}
est cartésien.
Mais alors, par définition de $\alpha^*$ et par la propriété universelle du produit fibré, on a un isomorphisme naturel
\begin{equation*}
\RealNP{\alpha^*\fil{B}}\simeq \alpha^*\Real{\fil{B}}_{N(Q)}.
\end{equation*}
\end{proof}

\begin{remarque}
Attention, la quatrième relation de compatibilité n'est pas vérifiée. En effet considérons le cas où $P\not=\{*\}$ et $\alpha\colon P\to\{*\}$ est l'unique morphisme de $P$ vers l'ensemble ordonné terminal. Dans ce cas, on peut identifier $\sS_{\{*\}}$ et $\sS$. Soit $\fil{X}$ un espace fortement filtré au-dessus de $P$, alors, $\alpha_*(\Sing_{N(P)}\fil{X})$ est égal à l'ensemble simplicial sous jacent à $\Sing_{N(P)}\fil{X}$. Par ailleurs, $\Sing_{N(\{*\})}(\alpha_*\fil{X})$ est isomorphe à $\Sing(X)$. Ces deux ensembles simpliciaux ne sont pas isomorphes en général.
Cependant, si $\alpha\colon P\to Q$ est un monomorphisme (c'est une condition plus faible que de demander que $P$ soit un sous ensemble ordonné de $Q$) alors on a la quatrième relation de compatibilité. Soit $\fil{X}$ un espace fortement filtré sur $P$. On considère le diagramme commutatif suivant
\begin{equation*}
\begin{tikzcd}
\Sing_{N(P)}\fil{X}
\arrow{r}
\arrow{d}
&\Sing(X)
\arrow{d}
\\
N(P)
\arrow[swap]{d}{N(\alpha)}
\arrow{r}
&\Sing(\Real{N(P)})
\arrow{d}{\Sing(\Real{N(\alpha)})}
\\
N(Q)\arrow{r}
&\Sing(\Real{N(Q)})
\end{tikzcd}
\end{equation*}
Si $\alpha$ est un monomorphisme, le carré inférieur est cartésien. Dans tout les cas, par définition de $\Sing_{N(P)}$, le carré supérieur est cartésien. Ainsi, si $\alpha$ est un monomorphisme, le rectangle extérieur est cartésien, et par définition de $\Sing_{N(Q)}$, on a un isomorphisme naturel
\begin{equation*}
\Sing_{N(Q)}(\alpha_*\fil{X})\simeq\alpha_*\Sing_{N(P)}\fil{X}
\end{equation*}
\end{remarque}

\section{Les catégories d'ensembles simpliciaux et d'espaces stratifiés}
\label{SectionStrat}
Dans cette section, on montre comment construire des catégories de modèles d'objets stratifiés à partir des catégories de modèles d'objets filtrés étudiées précédemment. Pour ce faire, on appliquera \cite[Theorem 4.2]{Cagne} On commence par définir les différentes catégories d'objets stratifiés auxquels on s'intéressera.
\begin{defin}
Un espace stratifié est la donnée de
\begin{itemize}
\item un espace topologique $X\in \Top$,
\item un ensemble ordonné $P$,
\item une application continue $\varphi_X\colon X\to P$, où $P$ est muni de la topologie d'Alexandrov.
\end{itemize}
Un morphisme entre espaces stratifiés est la donnée d'une paire $(f\colon X\to Y,\alpha\colon P\to Q)$ telle que le diagramme suivant commute.
\begin{equation*}
\begin{tikzcd}
X
\arrow[swap]{d}{\varphi_X}
\arrow{r}{f}
&Y
\arrow{d}{\varphi_Y}
\\
P
\arrow{r}{\alpha}
&Q
\end{tikzcd}
\end{equation*}
La catégorie des espaces stratifiés est noté $\Strat$.
\end{defin}

\begin{defin}
Un espace fortement stratifié est la donnée de
\begin{itemize}
\item un espace topologique $X\in \Top$,
\item un ensemble ordonné $P$,
\item une application continue $\varphi_X\colon X\to \Real{N(P)}$, où $N(P)$ est le nerf de $P$.
\end{itemize}
Un morphisme entre espaces fortement stratifiés est la donnée d'une paire de morphismes $(f\colon X\to Y,\alpha\colon P\to Q)$ telle que le diagramme suivant commute.
\begin{equation*}
\begin{tikzcd}
X
\arrow[swap]{d}{\varphi_X}
\arrow{r}{f}
&Y
\arrow{d}{\varphi_Y}
\\
\Real{N(P)}
\arrow{r}{\Real{N(\alpha)}}
&\Real{N(Q)}
\end{tikzcd}
\end{equation*}
La catégorie des espaces fortement stratifiés est noté $\FStrat$.
\end{defin}

\begin{defin}
Un ensemble simplicial stratifié est la donnée de
\begin{itemize}
\item un ensemble simplicial $X\in \sS$,
\item un ensemble ordonné $P$,
\item une application simpliciale $\varphi_X\colon X\to N(P)$, où $N(P)$ est le nerf de $P$.
\end{itemize}
Un morphisme entre ensembles simpliciaux stratifiés est la donnée d'une paire de morphismes $(f\colon X\to Y,\alpha\colon P\to Q)$ telle que le diagramme suivant commute.
\begin{equation*}
\begin{tikzcd}
X
\arrow[swap]{d}{\varphi_X}
\arrow{r}{f}
&Y
\arrow{d}{\varphi_Y}
\\
N(P)
\arrow{r}{N(\alpha)}
&N(Q)
\end{tikzcd}
\end{equation*}
La catégorie des ensembles simpliciaux stratifiés est noté $\sStrat$.
\end{defin}

Dans la suite de cette section $\E$ désignera l'une des trois catégories, $\Strat$, $\FStrat$ ou $\sStrat$.
Par commodité, $\E(P)$ dénotera respectivement $P$ munie de la topologie d'alexandrov, $\Real{N(P)}$ ou $N(P)$. 
On considérera aussi la catégorie des ensembles ordonnés $\Poset$ dont les objets sont les ensembles ordonnés, et les morphismes sont les applications croissantes.
On a un foncteur 
\begin{align*}
F\colon \E&\to \Poset \\
(X,P,\varphi_X)&\mapsto P\\
(f,\alpha)&\mapsto \alpha
\end{align*}
Pour tout ensemble ordonné $P$, $\E_P$ désignera la sous catégorie de $\E$ dont les objets sont de la forme $(X,P,\varphi_X)$ et les morphismes sont de la forme $(f,\Id_P)\colon (X,P,\varphi_X)\to (Y,P,\varphi_Y)$. Elle correspondra respectivement à $\Top_P$, $\Top_{N(P)}$ ou $\sS_P$ suivant que $\E$ est la catégorie $\Strat$, $\FStrat$ ou $\sStrat$.
On a maintenant la proposition suivante (Voir \cite[Definition 2.1]{Cagne})
\begin{prop}
Le foncteur $F$ est une bifibration de Grothendieck.
\end{prop}

\begin{proof}
Soit $\filstrat{Y}{Q}\in \E$ un objet stratifié, et $\alpha\colon P\to Q$ un morphisme de $\Poset$. On considère l'objet stratifié $\alpha^*\filstrat{Y}{Q}=(\alpha^*Y,P,\alpha^*\varphi_Y)$. Alors, le morphisme 
\begin{equation*}
\begin{tikzcd}
\alpha^*Y
\arrow{r}
\arrow[swap]{d}{\alpha^*\varphi_Y}
&Y
\arrow{d}{\varphi_Y}
\\
\E(P)
\arrow{r}{\E(\alpha)}
&\E(Q)
\end{tikzcd}
\end{equation*}
est un morphisme cartésien de $\E$. (voir les définitions \ref{DefinitionAlphaUpperStarSimplicial}, \ref{DefinitionAlphaUpperStarTopologique} et la remarque \ref{RemarqueAlphaStarTopP})
De même, si $\filstrat{X}{P}\in \E$ est un objet stratifié et $\alpha\colon P\to Q$ un morphisme de $\Poset$, l'objet stratifié $(X,Q,\E(\alpha)\circ\varphi_X)$ fournit un morphisme cocartésien de $\E$ :
\begin{equation*}
\begin{tikzcd}
X
\arrow{r}{\Id}
\arrow[swap]{d}{\varphi_X}
&X
\arrow{d}{\E(\alpha)\circ\varphi_X}
\\
\E(P)
\arrow{r}{\E(\alpha)}
&\E(Q)
\end{tikzcd}
\end{equation*}
\end{proof}

\begin{defin}
Soit $(f,\alpha)\colon \filstrat{X}{P}\to\filstrat{Y}{Q}$ un morphisme stratifié. On peut factoriser le morphisme $f$ de deux façons différentes comme suit :
\begin{equation*}
\begin{tikzcd}
\filstrat{X}{P}
\arrow{r}
\arrow[swap]{d}{f^{\triangleleft}}
\arrow[swap, near start]{dr}{(f,\alpha)}
&(X,Q,\E(\alpha)\circ\varphi_X)
\arrow{d}{f_{\triangleright}}
\\
(\alpha^*Y,P,\alpha^*\varphi_Y)
\arrow{r}
&\filstrat{Y}{Q}
\end{tikzcd}
\end{equation*}
Le morphisme $f^{\triangleleft}$ est un morphisme de $\E_P$ et le morphisme $f_{\triangleright}$ est dans $\E_Q$.
\end{defin}

On peut maintenant définir les classes de cofibrations et de fibrations sur $\E$. On considérera sur $\sS_P$ la structure de modèle $\sSU_P$ du théorème \ref{TheoDescriptionExpliciteSSetP}. Les structures de modèles de $\Top_P$ et $\Top_{N(P)}$ proviennent des théorèmes \ref{CategorieModeleTopP} et \ref{CategorieModeleTopNP} respectivement.

\begin{defin}
Un morphisme de $\E$, $(f,\alpha)\colon \filstrat{X}{P}\to\filstrat{Y}{Q}$ est 
\begin{itemize}
\item une cofibration si $f_{\triangleright}$ est une cofibration de $\E_P$,
\item une fibration si $f^{\triangleleft}$ est une fibration de $\E_Q$,
\item une cofibration triviale si $\alpha$ est un isomorphisme et $f_{\triangleright}$ est une cofibration triviale de $\E_P$,
\item une fibration triviale si $\alpha$ est un isomorphisme et $f^{\triangleleft}$ est une fibration triviale de $\E_Q$.
\end{itemize}
\end{defin}

On a alors le théorème suivant \cite[Theorem 4.2]{Cagne}

\begin{theo}\label{TheoremeCMFStrat}
Il existe une structure de modèle sur les catégories $\Strat$, $\FStrat$ et $\sStrat$, où les classes de cofibrations et de fibrations (triviales) sont celles de la définition précédente.
\end{theo}

\begin{proof}
C'est une application de \cite[Theorem 4.2]{Cagne} à la bifibration de Grothendieck $F\colon \E\to\Poset$, où $\Poset$ est munie de la structure de modèle triviale : tout les morphismes sont des cofibrations et des fibrations, et les équivalences faibles sont les isomorphismes. Les hypothèses (hBC) et (hCon) sont automatiquement vérifiées, et l'hypothèse (Q) correspond aux propositions \ref{AlphaAdjonctionQuillen}, \ref{PropositionAlphaAdjonctionTop} et à la remarque \ref{RemarqueAlphaStarTopP}.
\end{proof}

\begin{remarque}\label{RemarqueEquivalenceFaibleFiberwise}
Pour ces structures de modèles, un morphismes entre objet stratifiés 
\begin{equation*}
(f,\alpha)\colon \filstrat{X}{P}\to\filstrat{Y}{Q}
\end{equation*}
est une équivalence faible si et seulement si $\alpha$ est un isomorphisme d'ensembles ordonnés et $f_{\triangleright}$ est une équivalence faible de $\E_P$. (Dans ce cas, on a aussi que $f^{\triangleleft}$ est une équivalence faible de $\E_Q$.) En effet, $(f,\alpha)$ est une équivalence faible si et seulement si c'est la composée d'une cofibration triviale $(i\beta)$ et d'une fibration triviale $(p,\gamma)$. Or, par définitions des classes de cofibrations triviales et de fibrations triviales, $\beta$ et $\gamma$ doivent être des isomorphismes. Ainsi, les paires $(\beta_*,\beta^*)$ et $(\gamma_*,\gamma^*)$ sont des équivalences de Quillen, et $f$ est une équivalence faible si et seulement si les morphismes $p_{\triangleright}$, $p^{\triangleleft}$, $i_{\triangleright}$ et $i^{\triangleleft}$ sont des équivalences faibles.
\end{remarque}

\begin{defin}
On définit un foncteur
$\Sing_{\FStrat}\colon \FStrat\to\sStrat$
sur les objets par, 
\begin{equation*}
\Sing_{\FStrat}\filstrat{X}{P}=\Sing_{N(P)}\fil{X}.
\end{equation*} 
Pour $(f,\alpha)\colon\filstrat{X}{P}\to\filstrat{Y}{Q}$, on définit $\Sing_{\FStrat}(f,\alpha)$ comme la composition
\begin{equation*}
\begin{tikzcd}
\Sing_{N(P)}\fil{X}
\arrow{d}
\arrow{r}{\Sing_{N(P)}(f^{\triangleleft})}
&\Sing_{N(P)}(\alpha^*\fil{Y})
\arrow{d}
\arrow{r}{\simeq}
&\alpha^*\Sing_{N(Q)}\fil{Y}
\arrow{r}
\arrow{d}
&\Sing_{N(Q)}\fil{Y}
\arrow{d}
\\
N(P)
\arrow[swap]{r}{\Id}
&N(P)
\arrow[swap]{r}{\Id}
&N(P)
\arrow[swap]{r}{N(\alpha)}
&N(Q)
\end{tikzcd}
\end{equation*}
\end{defin}

\begin{defin}
On définit un foncteur $\RealFStrat{-}\colon \sStrat\to\FStrat$ sur les objets par
\begin{equation*}
\RealFStrat{\filstrat{X}{P}}=\RealNP{\fil{X}}
\end{equation*}
Pour $(f,\alpha)\colon \filstrat{X}{P}\to\filstrat{Y}{Q}$, on définit $\RealFStrat{(f,\alpha)}$ comme la composition
\begin{equation*}
\begin{tikzcd}[column sep = 60pt]
\Real{X}
\arrow[swap]{d}{\Real{\varphi_X}}
\arrow{r}{\Id_{\Real{X}}}
&\Real{X}
\arrow{d}{\Real{N(\alpha)}\circ\Real{\varphi_X}}
\arrow{r}{\Real{f}}
&\Real{Y}
\arrow{d}{\Real{\varphi_Y}}
\\
\Real{N(P)}
\arrow[swap]{r}{\Real{N(\alpha)}}
&\Real{N(Q)}
\arrow[swap]{r}{\Id}
&\Real{N(Q)}
\end{tikzcd}
\end{equation*}
\end{defin}

\begin{prop}\label{PropositionAdjonctionRealStratSingStrat}
La paire de foncteurs $(\RealFStrat{-},\Sing_{\FStrat})$ forme une paire de foncteurs adjoints.
\end{prop}

\begin{proof}
Soient $\filstrat{A}{P}$ un ensemble simplicial stratifié et $\filstrat{Y}{Q}$ un espace fortement stratifié. Un morphisme 
\begin{equation*}
(f,\alpha)\colon \filstrat{A}{P}\to\Sing_{\FStrat}\filstrat{Y}{Q}
\end{equation*}
est entièrement déterminé par la factorisation
\begin{equation*}
\begin{tikzcd}
A
\arrow[swap]{d}{\varphi_A}
\arrow{r}{f^{\triangleleft}}
&\alpha^*\Sing_{N(Q)}\fil{Y}
\arrow{d}
\arrow{r}
&\Sing_{N(Q)}\fil{Y}
\arrow{d}
\\
N(P)
\arrow[swap]{r}{\Id}
&N(P)
\arrow[swap]{r}{N(\alpha)}
&N(Q)
\end{tikzcd}
\end{equation*}
Par la proposition \ref{PropositionCompatibiliteSingAlpha}, on a $\alpha^*\Sing_{N(Q)}\simeq \Sing_{N(P)}\alpha^*$. En utilisant l'adjonction $(\RealNP{-},\Sing_{N(P)})$, on obtient une factorisation
\begin{equation}\label{CompositionAdjonctionSingStrat}
\begin{tikzcd}[column sep= 40pt]
\Real{A}
\arrow[swap]{d}{\Real{\varphi_A}}
\arrow{r}{g^{\triangleleft}}
&\alpha^*Y
\arrow{d}{\varphi_Y}
\arrow{r}
&Y
\arrow{d}{\varphi_Y}
\\
\Real{N(P)}
\arrow[swap]{r}{\Id}
&\Real{N(P)}
\arrow[swap]{r}{\Real{N(\alpha)}}
&\Real{N(Q)}
\end{tikzcd}
\end{equation}
où $g^{\triangleleft}$ est l'image de $f^{\triangleleft}$ par l'adjonction $(\RealNP{-},\Sing_{N(P)})$. Notons $(g,\alpha)$ la composition \ref{CompositionAdjonctionSingStrat}. Alors, l'association $(f,\alpha)\mapsto (g,\alpha)$ définit la bijection voulue.
\end{proof}

\begin{theo}\label{TheoremeRealStratSingStratPreserventEquivalenceFaible}
Les foncteurs $\RealFStrat{-}$ et $\Sing_{\Strat}$ préservent les équivalences faibles.
\end{theo}

\begin{proof}
Soit $(f,\alpha)\colon \filstrat{A}{P}\to\filstrat{B}{Q}$ une équivalence faible de $\sStrat$. Par la remarque \ref{RemarqueEquivalenceFaibleFiberwise}, $\alpha$ est un isomorphisme d'ensembles ordonnés, on peut donc supposer que $P=Q$ et $\alpha=\Id$. Mais alors 
\begin{equation*}
\RealFStrat{(f,\alpha)}=\RealNP{f}\colon \RealNP{\fil{A}}\to\RealNP{\fil{B}}
\end{equation*}
est une équivalence faible de $\Top_{N(P)}$ par la proposition \ref{RealisationSingPreserventEquivalencesFaibles}. C'est donc une équivalence faible de $\FStrat$. De même, soit $(g,\beta)\colon \filstrat{X,P}\to\filstrat{Y}{Q}$ une équivalence faible de $\FStrat$. Par la remarque \ref{RemarqueEquivalenceFaibleFiberwise}, on peut supposer que $P=Q$ et $\alpha=\Id$. Mais alors 
\begin{equation*}
\Sing_{\Strat}(g,\beta)=\Sing_{N(P)}(g)\colon \Sing_{N(P)}\fil{X}\to\Sing_{N(P)}\fil{Y}
\end{equation*}
est une équivalence faible de $\sSU_P$ par la proposition \ref{RealisationSingPreserventEquivalencesFaibles}. C'est donc une équivalence faible de $\sStrat$.
\end{proof}

\section{D'autres structures de modèles pour les espaces stratifiés}
\label{SectionAutresCMF}
\subsection{Transport depuis la structure de Joyal}
\label{SectionNandLal}
Motivé par l'étude des espaces homotopiquement stratifiés introduits par Quinn \cite{Quinn}, Nand-Lal \cite{NandLal} a construit une structure de modèle sur une sous-catégorie de $\Strat$ d'objets fibrants. La stratégie suivie par Nand-Lal est d'exploiter une adjonction entre $\Strat$ et $\sS$ :
\begin{equation*}
\text{SS}\colon \Strat\leftrightarrow \sS \colon |-|
\end{equation*}
où le foncteur $\text{SS}$ (\cite[Definition 7.1.0.3]{NandLal}) est obtenu comme la composition
\begin{equation*}
\Strat\xrightarrow{\Sing_{\Strat}}\sStrat\xrightarrow{\text{Oubli}}\sS
\end{equation*}
Nand-Lal transporte ensuite la structure de Joyal sur $\sS$ le long de cette adjonction. En se restreignant à la sous catégorie des espaces stratifiés fibrants, c'est à dire les espaces stratifiés $\fil{X}$ tels que $\text{SS}\fil{X}$ est une quasi-catégorie, il obtient ainsi une structure de modèle sur la sous-catégorie pleine correspondante. Les équivalences faibles y sont les morphismes $f\colon\fil{X}\to\fil{Y}$ tels que le morphisme d'ensembles simpliciaux
\begin{equation}\label{MorphismeSS}
\text{SS}(f)\colon \text{SS}\fil{X}\to\text{SS}\fil{Y}
\end{equation}
est une équivalence faible dans la structure de Joyal, et les fibrations y sont les morphismes tels que \ref{MorphismeSS} est une fibration dans la structure de Joyal. Nand-Lal montre que les objets fibrants-cofibrants pour cette structure transportée sont tous des espaces homotopiquement stratifiés au sens de Quinn \cite{Quinn}. Comme réciproque partielle, il prouve que tout espace métrique, homotopiquement stratifié, de stratification finie est fibrant. D'autre part, on a le résultat suivant.
\begin{prop}\label{PropositionQuasiCategorieImpliqueFibrant}
Soit $(X,\varphi_X\colon X\to P)$ un espace filtré au dessus de $P$. Si $SS\fil{X}$ est une quasi-catégorie, alors $\Sing_P\fil{X}$ est un objet fibrant de $\sSU_P$.
\end{prop}

\begin{proof}
La preuve de la proposition \ref{ConiquementStratifieImpliqueFibrant} montre que si l'ensemble simplicial sous-jacent $\Sing_P\fil{X}$ est une quasi-catégorie, alors $\Sing_P\fil{X}$ est fibrant dans $\sSU_P$.
\end{proof}
\begin{remarque}
Si $\fil{A}$ est un ensemble simplicial filtré, $A$ peut être une quasi-catégorie sans que $\fil{A}$ soit un objet fibrant de $\sSU_P$. La proposition \ref{PropositionQuasiCategorieImpliqueFibrant} ne donne cette implication que dans le cas où $\fil{A}=\Sing_P\fil{X}$ pour un certain espace filtré $\fil{X}$.
\end{remarque}

Ainsi, la proposition \ref{PropositionQuasiCategorieImpliqueFibrant} permet d'étendre la classe des objets vérifiants les théorèmes de Whitehead filtré (Théorèmes \ref{PremierTheoremeWhitehead} et \ref{DeuxiemeTheoremeWhitehead}) pour inclure les espaces métriques, homotopiquement stratifiés, de stratification finie.

D'autre part, pour les espaces stratifiés fibrants pointés (c'est à dire munis d'un morphisme $\phi\colon\Real{N(P)}\to\fil{X}$), Nand-Lal définit des catégories homotopiques, $\pi_k(\fil{X},\phi)$. Les catégories homotopiques sont des invariants du type d'homotopie stratifiés, et sont un analogue stratifié des groupoïdes fondamentaux. Elles sont construites à partir des groupes d'homotopie des strates et des entrelacs homotopiques. Comme les groupes d'homotopie filtrés, les catégories homotopiques caractérisent les équivalences d'homotopies stratifiées entre espaces cofibrants-fibrants. Il apparait alors naturel de se poser les questions suivantes :

\begin{question}
Quelle est la relation entre les catégories homotopiques et les groupes d'homotopie filtrés? Est-il possible de calculer l'un à partir de l'autre?
\end{question}

\subsection{La structure de Joyal-Kan sur $\sS_P$}
\label{SectionHaine}
Dans \cite{Haine}, Haine construit et explore une structure de modèle sur $\sS_P$ différente de celles qu'on considère ici. La construction de la structure de Joyal-Kan fait intervenir la structure suivante.
\begin{defin}
Un morphisme de $\sS_P$, $f\colon\fil{X}\to\fil{Y}$ est 
\begin{itemize}
\item une fibration de $\sSJoyal_P$ si $f\colon X\to Y$ est une fibration dans la structure de Joyal sur $\sS$
\item une équivalence faible de $\sSJoyal_P$ si $f\colon X\to Y$ est une équivalence faible dans la structure de Joyal sur $\sS$.
\end{itemize}
\end{defin}

On définit ensuite l'ensemble de morphismes 
\begin{equation*}
E_P=\{d_i\Delta^{\varphi}\to\Delta^{\varphi}\ |\ i=0\text{ ou }1, \ \varphi\colon \Delta^1\to N(P) \text{ est constante}\}.
\end{equation*}
L'ensemble $E_P$ contient donc les inclusions $\Delta^0\subset\Delta^1$ où $\Delta^{1}$ est filtrée par une filtration constante.
Haine définit ensuite la structure de Joyal-Kan comme suit.
\begin{defin}
La catégorie modèle $\sSJK_P$ est la localisation de Bousfield à gauche de $\sSJoyal_P$ par rapport à l'ensemble $E_P$.
\end{defin}
La catégorie modèle $\sSJK_P$ simpliciale est un modèle pour la $\infty$-catégorie d'espaces $P$-stratifiés définit dans \cite[Definition 2.1]{Exodromy}
On a aussi la propriété suivante
\begin{prop}
La catégorie modèle $\sSJK_P$ est une localisation de la catégorie modèle $\sSU_P$.
\end{prop}

\begin{proof}
Les cofibrations des structures $\sSU_P$ et $\sSJK_P$ coïncident, ce sont les monomorphismes. Il suffit donc de montrer que les équivalences faibles de $\sSU_P$ sont des équivalences faibles de $\sSJK_P$. Mais, les équivalences d'homotopies filtrés sont des équivalences faibles de $\sSJK_P$. Comme la classe des équivalences faibles de $\sSU_P$ est minimale parmi celles qui contiennent les équivalences d'homotopie filtrés, on en déduit le résultat voulu (voir la remarque \ref{RemarqueSSUPUniverselle} sur l'universalité de $\sSU_P$).
\end{proof}
Ainsi, on a un zigzag de localisations
\begin{equation*}
\sSU_P\to\sSJK_P\leftarrow \sSJoyal_P.
\end{equation*}

\appendix
\chapter{Caractérisation des fibrations dans une catégorie de préfaisceaux}
\label{ChapitreCaracterisationFibrationsAnnexe}
\chaptermark{Caractérisation des fibrations}

L'objet de ce chapitre est de démontrer le théorème suivant. La preuve que l'on présente ici est une généralisation de la preuve de \cite[Proposition 2.1.41]{Cisinski}. Cette dernière est formulée pour le cas des ensembles simpliciaux, et on l'étend ici à un contexte légèrement plus général incluant notamment les catégories d'ensembles simpliciaux filtrés.

\begin{theo}\label{IdentificationFibrationsNaives}
Soient $A$ une petite catégorie de Eilenberg-Zilber, et $(I,\An)$ une donnée homotopique sur $\widehat{A}$. On considère la structure de modèle sur $\widehat{A}$ obtenue à partir de la donnée homotopique $(I,\An)$, et on suppose qu'il existe un ensemble de cofibration $\Lambda$ tel que $\An=l(r(\Lambda))$ tel que pour tout morphisme de $\Lambda$, $j\colon K\to L$, $K$ et $L$ sont des objets compacts de $\widehat{A}$.
On suppose de plus qu'on a 
\begin{itemize}
\item un foncteur $A\to \widehat{A}$, on note $S\colon \widehat{A}\to\widehat{A}$ son extension par colimites,
\item un adjoint de $S$ à droite $E\colon \widehat{A}\to\widehat{A}$,
\item une transformation naturelle $\alpha\colon S\to \Id$. 
\end{itemize}
On note $\beta\colon \Id\to E$ la transformation naturelle adjointe à $\alpha$, et on note $E^{\infty}$ le foncteur obtenu comme la colimite
\begin{equation*}
\Id\xrightarrow{\beta} E\xrightarrow{\beta_{E{-}}}E^2\xrightarrow{\beta_{E^2{-}}}\dots \to E^n\to E^{n+1}\to \dots
\end{equation*}  
Si de plus
\begin{itemize}
\item le foncteur $S$ préserve les monomorphismes et les extensions anodines,
\item pour tout préfaisceau représentable $a$ de $\widehat{A}$, $\alpha_X\colon S(a)\to a$ est une équivalence faible absolue,
\item pour tout préfaisceau $X$, $E^{\infty}(X)$ est fibrant,
\end{itemize}
alors, la classe des fibrations naïves et celle des fibrations coïncident et $\Lambda$ est un ensemble générateur des cofibrations triviales.
\end{theo}

\begin{lemme}\label{EquivalenceFaibleFoncteurEilenbergZilber}
Soient $A$ une catégorie d'Eilenberg-Zilber $F,G\colon A\to \mathcal{C}$ deux foncteurs à valeur dans une catégorie de modèle, et soient $F_{!}$ et $G_{!}$ leur extension par colimites. On suppose que $F_{!}$ et $G_{!}$ envoient les monomorphismes sur des cofibrations. Si une transformation naturelle $u\colon F\to G$ induit des équivalences faibles $F(a)\to G(a)$ pour tout préfaisceau représentable $a$, alors pour tout préfaisceau $X$ de $\widehat{A}$, le morphisme de $\mathcal{C}$
\begin{equation*}
u_X\colon F_{!}(X)\to G_{!}(X)
\end{equation*}
est une équivalence faible.
\end{lemme}

\begin{proof}
On considère la classe des objets $X$ de $\widehat{A}$ tel que $u_X$ est une équivalence faible. Par hypothèse, elle contient les préfaisceaux représentables. En appliquant \cite[Corollaires 2.3.16, 2.3.18 et 2.3.29]{Cisinski2}, on obtient que cette classe est saturée par monomorphisme \cite[Definition 1.3.9]{Cisinski2}. On en déduit par \cite[Corollaire 1.3.10]{Cisinski2} que cette classe contient tous les préfaisceaux de $\widehat{A}$.
\end{proof}
On déduit immédiatement de ce lemme le résultat suivant 

\begin{lemme}\label{LemmeLastVertexUEquivalenceFaible}
Pour tout préfaisceau $X$ de $\widehat{A}$, le morphisme $\alpha_X\colon S(X)\to X$ est une équivalence faible. De plus, $S$ préserve les cofibrations triviales.
\end{lemme}
\begin{proof}
La première proposition est une application du lemme précédent avec $u=\alpha$, $F_{!}=S$ et $G_{!}=\Id$. Pour la seconde proposition, on considère une cofibration triviale $f\colon X\to Y$. Alors, on a le diagramme commutatif suivant
\begin{equation*}
\begin{tikzcd}
S(X)
\arrow{r}{S(f)}
\arrow{d}{\alpha_X}
&S(Y)
\arrow{d}{\alpha_Y}
\\
X
\arrow{r}{f}
&Y
\end{tikzcd}
\end{equation*}
Comme $S$ préserve les cofibrations, $S(f)$ est une cofibration. De plus, par deux sur trois, $S(f)$ est une équivalence faible.
\end{proof}

\begin{defin}
Soit $Y$ un préfaisceau de $\widehat{A}$. On définit sur $\widehat{A}/Y$ le foncteur $S^{Y}\colon \widehat{A}/Y\to\widehat{A}/Y$ dont l'image sur les objets est
\begin{equation*}
S^{Y}(X,f\colon X\to Y)=(S(X),f\circ \alpha_X\colon S(X)\to Y)
\end{equation*}
On définit aussi la transformation naturelle $\alpha^{Y}\colon S^{Y}\to \Id$ par $\alpha^{Y}_{(X,f\colon X\to Y)}=\alpha_X\colon (S(X),f\circ \alpha_X)\to (X,f)$.
De même, on définit le foncteur $E^{Y}$ par 
\begin{equation*}
E^{Y}(X,f)=(Y\times_{E(Y)}E(X),\pr_Y)
\end{equation*}
Et la transformation naturelle $\beta^Y\colon \Id\to E^{Y}$, où le morphisme $\beta^{Y}_{(X,f)}$ est induit par le diagramme commutatif suivant
\begin{equation*}
\begin{tikzcd}
X
\arrow[bend left=20]{drr}{\beta_X}
\arrow[swap, bend right = 12]{ddr}{f}
\arrow{dr}{\beta^Y_{(X,f)}}
\\
\phantom{X}
&Y\times_{E(Y)}E(X)
\arrow{r}{\beta_X}
\arrow{d}{f}
&E(X)
\arrow{d}{E(f)}
\\
\phantom{X}
&Y
\arrow{r}{\beta_Y}
&E(Y)
\end{tikzcd}
\end{equation*}
\end{defin}

\begin{lemme}\label{ExExtensionAnodine}
Soit $Y$ un préfaisceau de $\widehat{A}$, alors $E^{Y}$ est un adjoint à droite de $S^Y$ et $\beta^Y$ est l'image de $\alpha^Y$ par l'adjonction. De plus, si $f\colon X\to Y$ est une fibration naïve, alors $\beta^{Y}_{X,f}\colon (X,f)\to E^{Y}(X,f)$ est une extension anodine.
\end{lemme}

\begin{proof}
La première partie du lemme provient de la construction des foncteurs $S^Y$ et $E^Y$. Pour la seconde partie, soit $f\colon X\to Y$ une fibration naïve de $\widehat{A}$. On veut appliquer le lemme \ref{AdjonctionEquivalenceFaible} pour $\C= \widehat{A}/Y$, $G=S^Y$, $D=E^Y$ et $\alpha=\alpha^Y$. On considère sur $\widehat{A}/Y= \widehat{A/Y}$ la structure de modèle induite par celle sur $\widehat{A}$. En particulier, la classe des extensions anodines est donnée par $\An/Y$, où $\An$ est la classe des extensions anodines considérée sur $\widehat{A}$. Comme $S$ respecte les cofibrations, c'est aussi le cas pour $S^Y$. Montrons que pour tout objet $(X,f\colon X\to Y)
$ de $\widehat{A}/Y$, $\alpha^Y_{(X,f)}$ est une équivalence faible de $\widehat{A}/Y$. Pour ça, il suffit d'appliquer le lemme \ref{EquivalenceFaibleFoncteurEilenbergZilber} à la catégorie $\widehat{A/Y}$. En effet, la petite catégorie $A/Y$ est elle aussi une catégorie d'Eilenberg-Zilber (voir par exemple \cite[Example 1.3.3]{Cisinski2}) et, par hypothèse, $\alpha^Y_{(a,f\colon a\to Y)}$ est une équivalence faible de $\widehat{A}/Y$ pour tout préfaisceau représentable $a$ de $\widehat{A}$ et pour tout morphisme $f\colon a\to Y$. On en déduit que pour tout objet $(X,f)$ de $\widehat{A}/Y$, $\alpha^Y_{(X,f)}$ est une équivalence faible de $\widehat{A}/Y$. Considérons maintenant une cofibration triviale dans $\widehat{A}/Y$, $g\colon (X,f)\to (X',f')$. Alors, $S^Y(g)$ est une cofibration, et on a le diagramme commutatif suivant 
\begin{equation*}
\begin{tikzcd}
S^Y(X,f)
\arrow{r}{S^Y(g)}
\arrow{d}{\alpha^Y_{(X,f)}}
&S^Y(X',f')
\arrow{d}{\alpha^Y_{(X',f')}}
\\
(X,f)
\arrow{r}{g}
&(X',f')
\end{tikzcd}
\end{equation*}
par hypothèse, $g$ est une équivalence faible et on a montré que les flèches verticales étaient des équivalences faibles, on en déduit que $S^Y(g)$ est une équivalence faible de $\widehat{A}/Y$. En particulier, $S^Y$ préserve les cofibrations triviales. On peut donc appliquer le lemme \ref{AdjonctionEquivalenceFaible}, et on déduit que le morphisme $\beta^Y_{(X,f)}$ est une équivalence faible. Comme c'est de plus un monomorphisme, c'est une cofibration triviale de $\widehat{A}/Y$. De plus, comme $S$ préserve les extensions anodines, $E$ préserve les fibrations naïves, et donc $E(f)$ est une fibration naïve. On en déduit que $Y\times_{E(Y)}E(X)\to Y$ est une fibration naïve, et donc que $E^Y(X,f)$ est un objet fibrant de $\widehat{A}/Y$. Finalement, $\beta^Y_{(X,f)}$ est une cofibration triviale entre objets fibrants, donc une extension anodine de $\widehat{A}/Y$ \cite[Lemme 1.3.39]{Cisinski}. Par définition de la classe des extensions anodines sur $\widehat{A}/Y$, on en déduit que $(f,\beta)\colon X\to Y\times_{E(Y)}E(X)$ est une extension anodine de $\widehat{A}$.
\end{proof}

\begin{lemme}\label{AdjonctionEquivalenceFaible}
Soient $\mathcal{C}$ une catégorie de modèles fermée, $G\colon  \C\to\C$ un foncteur respectant les cofibrations et cofibrations triviales et admettant un adjoint à droite $D$, et $\alpha\colon G\to \Id$ une transformation naturelle. On note $\beta\colon \Id\to D$ la transformation naturelle obtenue par adjonction. Si pour tout objet cofibrant $X$ de $\C$, le morphisme $\alpha_X\colon G(X)\to X$ est une équivalence faible, alors pour tout objet fibrant $X$ de $\C$, le morphisme $\beta_X\colon X\to D(X)$ est une équivalence faible.
\end{lemme}
\begin{proof}
C'est une application de \cite[Corollary 1.4.4 (b)]{Hovey} au cas où $F'=\Id$.
\end{proof}

\begin{lemme}\label{ExtensionAnodineExYInfini}
Soit $f\colon X\to Y$ une fibration naïve dans $\widehat{A}$. Alors, $X\to Y\times_{E^{\infty}(Y)}E^{\infty}(X)$ est une extension anodine.
\end{lemme}

\begin{proof}
Le morphisme $X\to Y\times_{E^{\infty}(Y)}E^{\infty}(X)$ est la composition transfinie 
\begin{equation*}
X\to Y\times_{E(Y)}E(X)\to\dots\to Y\times_{E^n(Y)}E^n(X)\xrightarrow{\Id\times_{\beta_{E^n(Y)}}\beta_{E^n(X)}} Y\times_{E^{n+1}(Y)}E^{n+1}(X)\to \dots
\end{equation*}
Il suffit donc de montrer que chacun des morphismes $\Id\times_{\beta_{E^n(Y)}}\beta_{E^n(X)}$ est une extension anodine. Comme $E$ respecte les fibrations naïves, $E^n(f)$ est une fibration naïve, et on a le carré cartésien suivant.
\begin{equation*}
\begin{tikzcd}
Y\times_{E^n(Y)}E^n(X)
\arrow{r}
\arrow{d}{f_n}
&E^n(X)
\arrow{d}{E^n(f)}
\\
Y
\arrow{r}
&E^n(Y)
\end{tikzcd}
\end{equation*}
La classe des fibrations naïves étant stable par produits fibrés, $f_n$ est une fibration naïve, et donc, on peut appliquer le lemme \ref{ExExtensionAnodine} à $f_n$. Il résulte que le morphisme 
\begin{equation*}
Y\times_{E^n(Y)}E^n(X)\to Y\times_{E(Y)}E(Y\times_{E^n(Y)}E^n(X))
\end{equation*}
est une extension anodine. On obtient le résultat voulu en composant cette extension anodine avec les isomorphismes canoniques suivants 
\begin{equation*}
Y\times_{E(Y)}E(Y\times_{E^n(Y)}E^n(X))\simeq Y\times_{E(Y)}E(Y)\times_{E^{n+1}(Y)}E^{n+1}(X)\simeq Y\times_{E^{n+1}(Y)}E^{n+1}(X)
\end{equation*}
où le premier isomorphisme provient de la commutation de $E$ avec les limites finies.
\end{proof}

\begin{lemme}
Toute limite inductive filtrante de fibration naïve de $\widehat{A}$ est une fibration naïve.
\end{lemme}

\begin{proof}
Soient $I$ une petite catégorie filtrante, $X,Y\colon I\to \widehat{A}$ deux foncteurs et $p\colon X\to Y$ une transformation naturelle telle que, pour tout $i\in I$, $p_i\colon X(i)\to Y(i)$ est une fibration naïve. Considérons $j\colon K\to L$ un élément de $\Lambda$, et un problème de relèvement 
\begin{equation*}
\begin{tikzcd}
K
\arrow{d}{j}
\arrow{r}{k}
&\lim\limits_{\to}{X}
\arrow{d}{\lim\limits_{\to}p}
\\
L
\arrow{r}{l}
&\lim\limits_{\to}Y
\end{tikzcd}
\end{equation*}
alors, par hypothèse, il existe des factorisations de $k$ et $l$, telles qu'on a un diagramme commutatif comme suit
\begin{equation*}
\begin{tikzcd}
K
\arrow{d}{j}
\arrow{r}
&X(i)
\arrow{r}
\arrow{d}{p_i}
&\lim\limits_{\to}{X}
\arrow{d}{\lim\limits_{\to}p}
\\
L
\arrow{r}
&Y(i)
\arrow{r}
&\lim\limits_{\to}Y
\end{tikzcd}
\end{equation*}
Le carré de gauche admet par hypothèse un relèvement, ce qui donne un relèvement dans le problème initial. On en déduit que $\lim\limits_{\to}p$ est une fibration naïve.
\end{proof}

\begin{lemme}
Pour toute fibration naïve $p\colon X\to Y$ de $\widehat{A}$, $E^{\infty}(p)\colon E^{\infty}(X)\to E^{\infty}(Y)$ est une fibration entre objets fibrants.
\end{lemme}

\begin{proof}
Comme $S$ préserve les extensions anodines, $E$ préserve les fibrations naïves, et donc $E^n$ préserve les fibrations naïves pour tout $n$. Par le lemme précédent, on en déduit que $E^{\infty}(p)$ est une fibration naïve. Comme $E^{infty}(X)$ et $E^{\infty}(Y)$ sont fibrants par hypothèse, $E^{\infty}(p)$ est une fibration naïve entre objets fibrants, et donc une fibration de $\widehat{A}$, par  \cite[Proposition 1.3.36]{Cisinski}
\end{proof}
\begin{proof}[Démonstration du théorème \ref{IdentificationFibrationsNaives}.]
Par \cite[Proposition 1.3.47]{Cisinski}, il suffit de montrer que toute fibration naïve $p$ admet une factorisation sous la forme $p=qj$ où $q$ est une fibration et $j$ est une extension anodine. On considère le carré cartésien
\begin{equation*}
\begin{tikzcd}
Y\times_{E^{infty}(Y)}E^{\infty}(X)
\arrow{r}
\arrow{d}{q}
&E^{\infty}(X)
\arrow{d}{E^{\infty}(p)}
\\
Y
\arrow{r}{\beta^{\infty}_Y}
&E^{\infty}(Y)
\end{tikzcd}
\end{equation*}
Par le lemme précédent, $E^{\infty}(p)$ est une fibration, et donc $q$ est une fibration. De plus, le morphisme $X\to Y\times_{E^{\infty}(Y)}E^{\infty}(X)$ est une extension anodine par le lemme \ref{ExtensionAnodineExYInfini}, et $p$ est égale à la composition 
\begin{equation*}
X\to Y\times_{E^{\infty}(Y)}E^{\infty}(X)\xrightarrow{q} Y
\end{equation*}
On en déduit que la classe des fibrations naïves et celle des fibrations coïncident. De plus, on a les égalités entre les classes suivantes 
\begin{equation*}
l(r(\Lambda))=\An=l(\text{FibN})=l(\text{Fib})=\text{Cof}\cap \text{W}
\end{equation*}
où $\text{FibN}$, $\text{Fib}$, $\text{Cof}$ et $\text{W}$ désignent respectivement les classes de fibrations naïves, fibrations, cofibrations et équivalences faibles.
\end{proof}

\chapter{Caractérisation du morphisme $X\to\Exi_P(X)$ à l'aide du foncteur diagonal}
\label{ChapitreCaracterisationMorphismeXExXAnnexe}
\chaptermark{Caractérisation du morphisme $X\to\Exi_P(X)$}
\section{Le foncteur diagonal}

Soit $A$ une catégorie de Eilenberg Zilber. Par \cite[Exemple 1.3.4]{Cisinski2} $A\times A$ est aussi une catégorie d'Eilenberg Zilber. On considère la catégorie $\widehat{A\times A}$ des préfaisceaux sur $A\times A$. Etant donné un préfaisceau $X\in \widehat{A\times A}$, on définit sa diagonale 
\begin{align*}
\diag(X)\colon A^{\op}&\to \Set\\
a&\mapsto X(a,a)\\
(f\colon a\to b)&\mapsto X(f,f)
\end{align*}

Ceci permet de définir un foncteur 
\begin{equation*}
\diag\colon \widehat{A\times A}\to \widehat{A}
\end{equation*}
De plus, si $a$ est un objet de $A$, et $X$ un préfaisceau de $\widehat{A\times A}$, on définit 

\begin{align*}
X^a\colon A^{\op}&\to \Set\\
b&\mapsto X(a,b)
\end{align*}

Ceci permet de définir $X^K$, pour tout préfaisceau $K\in \widehat{A}$, comme la limite inverse
\begin{equation*}
X^K=\lim_{a\to K}X^a
\end{equation*}
où $a$ parcourt les objets de $A$, et où on identifie $a\in A$ et le préfaisceau représentable par $a$.

De même, si $b$ est un objet de $A$ et $X$ est un préfaisceau de $\widehat{A\times A}$, on définit

\begin{align*}
X_b\colon A^{\op}&\to\Set\\
a&\mapsto X(a,b)
\end{align*}

Par ailleurs, on définit le bifoncteur $-\boxtimes -$ comme suit
\begin{align*}
-\boxtimes -\colon \widehat{A}\times\widehat{A}&\to \widehat{A\times A}\\
(K,L)&\mapsto \left\{
\begin{array}{ccc}
K\boxtimes L\colon A\times A &\to &\Set\\
(a,b)&\mapsto & K(a)\times L(b)
\end{array}\right.
\end{align*}

\begin{remarque}\label{DiagonaleBoxProduct}
Par construction, si $K,L\in \widehat{A}$ sont des préfaisceaux, on a un isomorphisme naturel $\diag(K\boxtimes L)\simeq K\times L$.
\end{remarque}

\begin{remarque}\label{BoxProductCoProduct}
Soient $K,L\in \widehat{A}$ des préfaisceaux et $a\in A$. On a 
\begin{equation*}
(K\boxtimes L)^{a}=\coprod_{K(a)}L.
\end{equation*}
En effet, pour $b\in A$, on calcule
\begin{align*}
(K\boxtimes L)^{a}(b)&=K(a)\times L(b)\\
&\simeq \coprod_{K(a)}L(b)\\
&\simeq (\coprod_{K(a)}L)(b).
\end{align*}
De plus, toutes les bijections apparaissant ici sont naturelles par rapport aux morphismes $f\colon b\to b'\in A$.
\end{remarque}

On s'intéresse maintenant au cas $A=\Delta(P)$, c'est à dire $\widehat{A}=\sS_P$. Pour simplifier les notations, on note $\bisS_P=\widehat{\Delta(P)\times\Delta(P)}$ la catégorie des préfaisceaux sur $\Delta(P)\times\Delta(P)$. L'objet de cette section est de prouver le théorème suivant :

\begin{theo}\label{TheoremeDiagonal}
Soit $f\colon X\to Y$ un morphisme entre préfaisceaux sur $\Delta(P)\times\Delta(P)$. Si pour tout $\Delta^{\varphi}\in \Delta(P)$ le morphisme induit 
\begin{equation*}
f^{\Delta^{\varphi}}\colon X^{\Delta{\varphi}}\to Y^{\Delta^{\varphi}}
\end{equation*}
est une équivalence faible de $\sS_P$, alors le morphisme
\begin{equation*}
\diag(f)\colon \diag(X)\to\diag(Y)
\end{equation*}
est aussi une équivalence faible de $\sS_P$.
\end{theo}

\begin{remarque}\label{TheoremeDiagonalRenverse}
En composant $\diag$ avec l'involution $\widehat{\Delta(P)\times \Delta(P)}\to\widehat{\Delta(P)\times \Delta(P)}$ inversant les facteurs, on obtient que la première hypothèse du théorème peut être remplacée par le fait que $f$ induit des équivalence faibles $f_{\Delta^{\varphi}}\colon X_{\Delta^{\varphi}}\to Y_{\Delta^{\varphi}}$ pour tout $\Delta^{\varphi}\in \Delta(P)$.
\end{remarque}

Pour prouver le Théorème \ref{TheoremeDiagonal}, on adapte la preuve de \cite[Theorem 3.1.16]{Cisinski2}. Celle-ci exploite l'existence d'une structure de modèle sur $\widehat{A\times A}$, dont on explicite certaines des propriétés à travers les lemmes suivants. 

\begin{lemme}\label{FibrationsTrivialesBiEnsemblesSimpliciauxFiltres}
L'ensemble de morphismes
\begin{equation*}
\{\partial(\Delta^{\varphi})\boxtimes\Delta^{\psi}\cup\Delta^{\varphi}\boxtimes\partial(\Delta^{\psi})\to\Delta^{\varphi}\boxtimes\Delta^{\psi}\ |\ \Delta^{\varphi},\Delta^{\psi}\in \Delta(P)\}
\end{equation*}
est un modèle cellulaire pour $\bisS_P$. En particulier, les fibrations triviales de $\bisS_P$ sont les morphismes $f\colon X\to Y$ tels que pour tout $\Delta^{\varphi}\in \Delta(P)$ le morphisme induit
\begin{equation*}
X^{\Delta^{\varphi}}\to X^{\partial(\Delta^{\varphi})}\times_{Y^{\partial(\Delta^{\varphi})}}Y^{\Delta^{\varphi}}
\end{equation*}
est une fibration triviale dans $\sS_P$.
\end{lemme}

\begin{proof}
Les préfaisceaux représentables de $\bisS_P$ sont les préfaisceaux de la forme $\Delta^{\varphi}\boxtimes\Delta^{\psi}$ où $\Delta^{\varphi},\Delta^{\psi}\in \Delta(P)$. Par \cite[Theorem 1.3.8]{Cisinski2}, l'ensemble du lemme \ref{FibrationsTrivialesBiEnsemblesSimpliciauxFiltres} est un modèle cellulaire, car il contient toutes les applications de la forme $\partial(h)\to h$ où $h$ est un préfaisceau représentable. Soit $f\colon X\to Y$ un morphisme de $\bisS_P$. Par définition, c'est une fibration triviale si et seulement il existe un relèvement dans tout diagramme de la forme suivante :
\begin{equation}\label{DiagrammeBiFibrationsTriviales}
\begin{tikzcd}
\partial(\Delta^{\varphi})\boxtimes\Delta^{\psi}\cup\Delta^{\varphi}\boxtimes\partial(\Delta^{\psi})
\arrow{r}{\alpha_1\cup\alpha_2}
\arrow{d}
&X
\arrow{d}{f}
\\
\Delta^{\varphi}\boxtimes\Delta^{\psi}
\arrow[swap]{r}{\beta}
\arrow[dashed,swap]{ur}{g}
&Y
\end{tikzcd}
\end{equation}
On remarque que par adjonction, les morphismes $\alpha_1,\alpha_2$ et $\beta$ correspondent respectivement à des morphismes
\begin{align*}
&\widehat{\alpha_1}\colon \Delta^{\psi}\to X^{\partial(\Delta^{\varphi})}\\
&\widehat{\alpha_2}\colon \partial(\Delta^{\psi})\to X^{\Delta^{\varphi}}\\
&\widehat{\beta}\colon \Delta^{\psi}\to Y^{\Delta^{\varphi}}.
\end{align*}
En particulier, on a le diagramme commutatif suivant
\begin{equation}\label{DiagrammeBiFibrationsTriviales2}
\begin{tikzcd}
\partial(\Delta^{\psi})
\arrow{r}{\widehat{\alpha_2}}
\arrow{d}
&X^{\Delta^{\varphi}}
\arrow{d}{\pr\times f^{\Delta^{\varphi}}}
\\
\Delta^{\psi}
\arrow[swap]{r}{\widehat{\alpha_1}\times \widehat{\beta}}
\arrow[dashed]{ur}{\widehat{g}}
&X^{\partial(\Delta^{\varphi})}\times_{Y^{\partial(\Delta^{\varphi})}}Y^{\Delta^{\varphi}}
\end{tikzcd}
\end{equation}
Et, par adjonction, il existe une solution au problème de relèvement \ref{DiagrammeBiFibrationsTriviales} si et seulement si il existe une solution au problème de relèvement \ref{DiagrammeBiFibrationsTriviales2}. Finalement, $f\colon X\to Y$ est une fibration triviale de $\bisS_P$ si et seulement si, pour tout $\Delta^{\varphi}$, le morphisme 
\begin{equation*}
X^{\Delta^{\varphi}}\to X^{\partial(\Delta^{\varphi})}\times_{Y^{\partial(\Delta^{\varphi})}}Y^{\Delta^{\varphi}}
\end{equation*}
est une fibration triviale de $\sS_P$.
\end{proof}

\begin{defin}
On note 
\begin{equation*}
\mathcal{M}=\{\partial(\Delta^{\varphi})\boxtimes\Delta^{\psi}\cup\Delta^{\varphi}\boxtimes\partial(\Delta^{\psi})\to\Delta^{\varphi}\boxtimes\Delta^{\psi}\ |\ \Delta^{\varphi},\Delta^{\psi}\in \Delta(P)\}
\end{equation*}
le modèle cellulaire du lemme \ref{FibrationsTrivialesBiEnsemblesSimpliciauxFiltres}, et 
\begin{equation*}
S=\{\Delta^{\varphi}\boxtimes\Lambda^{\psi}_k\to\Delta^{\varphi}\boxtimes\Delta^{\psi}\ |\ \Delta^{\varphi}\in \Delta(P), \text{ et $\Lambda^{\psi}_k\to\Delta^{\psi}$ est une inclusion de cornet admissible}\}
\end{equation*}
De plus, on définit le cylindre $J$ comme
\begin{equation*}
J= F(\Delta^0)\boxtimes F(\Delta^1)
\end{equation*}
avec les inclusions 
\begin{equation*}
\iota_{\epsilon}\colon \{\epsilon\}=F(\Delta^0)\boxtimes F(\{\epsilon\})\to F(\Delta^0)\boxtimes F(\Delta^1), \epsilon=0,1
\end{equation*}
où on identifie $\{\epsilon\}$ et $F(\Delta^0)\boxtimes F(\{\epsilon\})$.
Alors, la construction \cite[1.3.12]{Cisinski} fournit un ensemble générateur d'une classe d'extensions anodines pour $\bisS_P$ :
\begin{equation*}
\Lambda_J(S,\mathcal{M})
\end{equation*}
à partir de maintenant, on considère $\bisS_P$ comme la catégorie de modèle obtenue en appliquant \cite[Théorème 1.3.22]{Cisinski} à la catégorie $\bisS_P$ muni du cylindre $J$ et de la classe d'extension anodine $l(r(\Lambda_J(S,\mathcal{M})))$.
\end{defin}

\begin{prop}\label{FibrationsBisSetP}
Soit $f\colon X\to Y$ un morphisme de $\bisS_P$. Le morphisme $f$ est une fibration naïve de $\bisS_P$ si et seulement si, pour tout $\Delta^{\varphi}\in \Delta(P)$, les morphismes induits par $f$,
\begin{equation}
\label{MorphismeBiEnsemblesSimpliciauxFiltres1}
X^{\Delta^{\varphi}}\to Y^{\Delta^{\varphi}}
\end{equation}
et 
\begin{equation}\label{MorphismeBiEnsemblesSimpliciauxFiltres2}
X^{\Delta^{\varphi}}\to X^{\partial(\Delta^{\varphi})}\times_{Y^{\partial(\Delta^{\varphi})}}Y^{\Delta^{\varphi}}
\end{equation}
sont des fibrations (naïves) dans $\sS_P$.
\end{prop}

\begin{proof}
Soit $f$ une fibration naïve de $\bisS_P$. Par définition des fibrations naïves, $f$ a la propriété de relèvement à droite par rapport à $\Lambda_J(S,\mathcal{M})$. En particulier, $f$ a la propriété de relèvement à droite par rapport à $\Lambda^0_J(S,\mathcal{M})$. De plus, par le lemme \ref{CalculExtensionsAnodinesBiEnsemblesSimpliciauxFiltres}, on a
\begin{align*}
&\Lambda^0_J(S,\mathcal{M})=\{\Delta^{\varphi}\boxtimes\Lambda^{\psi}_k\to\Delta^{\varphi}\boxtimes\Delta^{\psi}\ |\ \Delta^{\varphi}\in \Delta(P), \Lambda^{\psi}_k\to \Delta^{\psi}\in \mathcal{A}\}\\
&\cup \{\partial(\Delta^{\varphi})\boxtimes(\Delta^1\otimes\Delta^{\psi})\cup\Delta^{\varphi}\boxtimes(\Delta^1\otimes\partial(\Delta^{\psi})\cup\{\epsilon\}\otimes\Delta^{\psi})\to \Delta^{\varphi}\boxtimes(\Delta^1\otimes\Delta^{\psi})\ |\ \Delta^{\varphi},\Delta^{\psi}\in \Delta(P)\}
\end{align*}
où $\epsilon=0,1$ et $\mathcal{A}$ est l'ensemble des inclusions de cornets admissibles. Soit $\Delta^{\varphi}\in \Delta(P)$ et $\Lambda^{\psi}_k\to\Delta^{\psi}\in \mathcal{A}$. Considérons le problème de relèvement suivant
\begin{equation}\label{DiagrammeFibrationsBiEnsembesSimpliciauxFiltres}
\begin{tikzcd}
\Lambda^{\psi}_k
\arrow{r}
\arrow{d}
&X^{\Delta^{\varphi}}
\arrow{d}{f^{\Delta^{\varphi}}}
\\
\Delta^{\psi}
\arrow[dashed]{ur}
\arrow{r}
&Y^{\Delta^{\varphi}}
\end{tikzcd}
\end{equation}
Par adjonction, il est équivalent au problème de relèvement suivant :
\begin{equation}\label{DiagrammeFibrationsBiEnsembesSimpliciauxFiltres2}
\begin{tikzcd}
\Delta^{\varphi}\boxtimes\Lambda^{\psi}_k
\arrow{r}
\arrow{d}
&X
\arrow{d}{f}
\\
\Delta^{\varphi}\boxtimes\Delta^{\psi}
\arrow{r}
\arrow[dashed]{ur}
&Y
\end{tikzcd}
\end{equation}
Or, comme $f$ est une fibration naïve de $\bisS_P$, il existe une solution au problème \ref{DiagrammeFibrationsBiEnsembesSimpliciauxFiltres2} et donc au problème \ref{DiagrammeFibrationsBiEnsembesSimpliciauxFiltres}. On en déduit que $f^{\Delta^{\varphi}}$ a la propriété de relèvement à droite par rapport à toutes les inclusions de cornets admissibles et donc que c'est une fibration de $\sS_P$. Considérons maintenant le problème de relèvement suivant :
\begin{equation}\label{DiagrammeFibrationsBiEnsembesSimpliciauxFiltres3}
\begin{tikzcd}
\Delta^1\otimes\partial(\Delta^{\psi})\cup \{\epsilon\}\otimes\Delta^{\psi}
\arrow{r}
\arrow{d}
&X^{\Delta^{\varphi}}
\arrow{d}{\pr\times f^{\Delta^{\varphi}}}
\\
\Delta^1\otimes\Delta^{\psi}
\arrow{r}
\arrow[dashed]{ur}
&X^{\partial(\Delta^{\varphi})}\times_{Y^{\partial(\Delta^{\varphi})}}Y^{\Delta^{\varphi}}
\end{tikzcd}
\end{equation}
Comme précédement, par adjonction, on obtient un problème de relèvement pour $f$ par rapport à un morphisme de $\Lambda^0_J(S,\mathcal{M})$, et on en déduit qu'il existe une solution au problème \ref{DiagrammeFibrationsBiEnsembesSimpliciauxFiltres3}. En particulier, le morphisme induit par $f$
\begin{equation*}
X^{\Delta^{\varphi}}\to X^{\partial(\Delta^{\varphi})}\times_{Y^{\partial(\Delta^{\varphi})}}Y^{\Delta^{\varphi}}
\end{equation*}
a la propriété de relèvement à droite par rapport à tous les morphismes de la forme 
\begin{equation*}
\Delta^1\otimes\partial(\Delta^{\psi})\cup \{\epsilon\}\otimes\Delta^{\psi}\to \Delta^1\otimes\Delta^{\psi}
\end{equation*}
Par le lemme \ref{SatureeABC}, on en déduit que le morphisme \ref{MorphismeBiEnsemblesSimpliciauxFiltres2} a la propriété de relèvement à droite par rapport aux inclusions de cornets admissibles. C'est donc une fibration de $\sS_P$.

Réciproquement, soit $f\colon X\to Y$ un morphisme de $\bisS_P$ tel que pour tout $\Delta^{\varphi}\in \Delta(P)$, les morphismes \ref{MorphismeBiEnsemblesSimpliciauxFiltres1} et \ref{MorphismeBiEnsemblesSimpliciauxFiltres2} sont des fibrations dans $\sS_P$. Alors, pour toute extension anodine de $\sS_P$, $\fil{Z}\to\fil{W}$, et pour tout $\Delta^{\varphi}\in \Delta(P)$, il existe des solutions pour les problèmes de relèvement
\begin{equation*}
\begin{tikzcd}
\fil{Z}
\arrow{r}
\arrow{d}
&X^{\Delta^{\varphi}}
\arrow{d}{f^{\Delta^{\varphi}}}
\\
\fil{W}
\arrow[dashed]{ur}
\arrow{r}
&Y^{\Delta^{\varphi}}
\end{tikzcd}
\end{equation*}
et
\begin{equation*}\begin{tikzcd}
\fil{Z}
\arrow{r}
\arrow{d}
&X^{\Delta^{\varphi}}
\arrow{d}{\pr\times f^{\Delta^{\varphi}}}
\\
\fil{W}
\arrow{r}
\arrow[dashed]{ur}
&X^{\partial(\Delta^{\varphi})}\times_{Y^{\partial(\Delta^{\varphi})}}Y^{\Delta^{\varphi}}
\end{tikzcd}
\end{equation*}
Par adjonction, on en déduit que $f$ a la propriété de relèvement à droite par rapport à tous les morphismes de $\Delta(P)\boxtimes\Lambda\cup\partial(\Delta(P))\boxtimes\Lambda$. Par le lemme \ref{CalculExtensionsAnodinesBiEnsemblesSimpliciauxFiltres}, on en déduit que $f$ a la propriété de relèvement à droite par rapport à $\Lambda_J(S,\mathcal{M})$, et donc que $f$ est une fibration naïve de $\sS_P$.
\end{proof}

\begin{corollaire}\label{ObjetsFibrantsBisSetP}
Soit $X$ un objet fibrant de $\bisS_P$. Alors, pour toute cofibration $\fil{Z}\to\fil{W}$ de $\sS_P$, le morphisme induit
\begin{equation*}
X^{\fil{Z}}\to X^{\fil{W}}
\end{equation*}
est une fibration de $\sS_P$.
\end{corollaire}

\begin{proof}
Soit $X$ un objet fibrant de $\bisS_P$, $\fil{Z}\to\fil{W}$ une cofibration de $\sS_P$ et $n\geq -1$ un entier. On rappelle que le $n$-squelette de $\fil{W}$ est défini comme $\sk_n\fil{W}=(\sk_n(W),(\varphi_W)_{|\sk_n(W)})$, avec $\sk_{-1}(W)=\emptyset$. En particulier, si on note $\Sigma^n$ l'ensemble des $n$ simplexes de $W$ n'étant pas dans l'image de $Z$, on a le carré cartésien suivant
\begin{equation*}
\begin{tikzcd}
\coprod\limits_{\sigma\in \Sigma^{n+1}}\partial(\Delta^{\varphi})
\arrow{d}
\arrow{r}{\coprod\sigma}
&\fil{Z}\cup\sk_n\fil{W}
\arrow{d}
\\
\coprod\limits_{\sigma\in\Sigma^{n+1}}\Delta^{\varphi}
\arrow[swap]{r}{\coprod\sigma}
&\fil{Z}\cup\sk_{n+1}\fil{W}
\end{tikzcd}
\end{equation*}
De plus, on a 
\begin{equation*}
\fil{W}\simeq\colim_{i\geq -1}\fil{Z}\cup\sk_i\fil{W}.
\end{equation*}
Or, par construction, le foncteur $X^{(-)}\colon \sS_P^{\op}\to\sS_P$ envoie les colimites sur des limites. En particulier, il envoie les carrés cocartésiens sur des carrés cartésiens et les coproduits sur des produits. Comme la classe des fibrations (naïves) de $\sS_P$ est stable par toutes ces opérations, on en déduit qu'il suffit de vérifier que pour tout $\Delta^{\varphi}\in \Delta(P)$, le morphisme
\begin{equation*}
X^{\Delta^{\varphi}}\to X^{\partial(\Delta^{\varphi})}
\end{equation*}
est une fibration (naïve) de $\sS_P$, ce qui découle de la proposition \ref{FibrationsBisSetP} appliquée à la fibration $X\to F(\Delta^0)\boxtimes F(\Delta^0)$.
\end{proof}

\begin{lemme}\label{CalculExtensionsAnodinesBiEnsemblesSimpliciauxFiltres}
Notons 
\begin{equation*}
\Delta(P)\boxtimes \Lambda=\{\Delta^{\varphi}\boxtimes \fil{Z}\to\Delta^{\varphi}\boxtimes \fil{W}\ |\ \Delta^{\varphi}\in \Delta(P),\text{ et } \fil{Z}\to\fil{W} \in \Lambda\}
\end{equation*}
et
\begin{equation*}
\partial(\Delta(P))\boxtimes\Lambda =\{\partial(\Delta^{\varphi})\boxtimes\fil{W}\cup\Delta^{\varphi}\boxtimes\fil{Z}\to\Delta^{\varphi}\boxtimes\fil{W}\ |\ \Delta^{\varphi}\in \Delta(P), \text{ et } \fil{Z}\to\fil{W}\in \Lambda\}
\end{equation*}
où $\Lambda$ est la classe des extensions anodines sur $\sS_P$. Alors, on a 
\begin{equation*}
\Lambda_J(S,\mathcal{M})\subseteq \Delta(P)\boxtimes\Lambda\cup\partial(\Delta(P))\boxtimes\Lambda
\end{equation*}
De plus, 
\begin{align*}
&\Lambda^0_J(S,\mathcal{M})=\{\Delta^{\varphi}\boxtimes\Lambda^{\psi}_k\to\Delta^{\varphi}\boxtimes\Delta^{\psi}\ |\ \Delta^{\varphi}\in \Delta(P), \Lambda^{\psi}_k\to \Delta^{\psi}\in \mathcal{A}\}\\
&\cup \{\partial(\Delta^{\varphi})\boxtimes(\Delta^1\otimes\Delta^{\psi})\cup\Delta^{\varphi}\boxtimes(\Delta^1\otimes\partial(\Delta^{\psi})\cup\{\epsilon\}\otimes\Delta^{\psi})\to \Delta^{\varphi}\boxtimes(\Delta^1\otimes\Delta^{\psi})\ |\ \Delta^{\varphi},\Delta^{\psi}\in \Delta(P)\}.
\end{align*}
\end{lemme}

\begin{proof}
On calcule $\Lambda_J(S,\mathcal{M})$. On a par définition
\begin{equation*}
\Lambda_J^0(S,\mathcal{M})=S\cup\{J\times X\cup \{\epsilon\}\times Y\to J\times Y\ |\ X\to Y\in \mathcal{M}\}
\end{equation*}
plus explicitement,  pour $\partial(\Delta^{\varphi})\boxtimes\Delta^{\psi}\cup\Delta^{\varphi}\boxtimes\partial(\Delta^{\psi})\to \Delta^{\varphi}\boxtimes\Delta^{\psi}$ un morphisme de $\mathcal{M}$, on a
\begin{align*}
&\phantom{=} J\times(\partial(\Delta^{\varphi})\boxtimes\Delta^{\psi}\cup\Delta^{\varphi}\boxtimes\partial(\Delta^{\psi}))\cup 
\{\epsilon\}\times (\Delta^{\varphi}\boxtimes\Delta^{\psi}) \\
&=(F(\Delta^0)\boxtimes F(\Delta^1))\times(\partial(\Delta^{\varphi})\boxtimes\Delta^{\psi}\cup\Delta^{\varphi}\boxtimes\partial(\Delta^{\psi}))\cup (F(\Delta^0)\boxtimes F(\{\epsilon\}))\times (\Delta^{\varphi}\boxtimes\Delta^{\psi})\\
&\simeq \partial(\Delta^{\varphi})\boxtimes(\Delta^1\otimes\Delta^{\psi})\cup \Delta^{\varphi}\boxtimes(\Delta^{1}\otimes\partial(\Delta^{\psi}))\cup \Delta^{\varphi}\boxtimes(\{\epsilon\}\otimes\Delta^{\psi})\\
&\simeq \partial(\Delta^{\varphi})\boxtimes(\Delta^1\otimes\Delta^{\psi})\cup\Delta^{\varphi}\boxtimes(\Delta^1\otimes\partial(\Delta^{\psi})\cup\{\epsilon\}\otimes\Delta^{\psi})
\end{align*}
Ici, on a utilisé le fait que $F(\Delta^0)\times\fil{Z}\simeq \fil{Z}$ et que $F(\Delta^1)\times \fil{Z}=\Delta^1\otimes\fil{Z}$ pour tout $\fil{Z}\in \sS_P$. On en déduit la description de $\Lambda^0_J(S,\mathcal{M})$. D'autre part, en notant $X=\partial(\Delta^{\varphi})\boxtimes\Delta^{\psi}\cup\Delta^{\varphi}\boxtimes\partial(\Delta^{\psi})$ et $Y=\Delta^{\varphi}\boxtimes\Delta^{\psi}$, on remarque que le morphisme 
\begin{equation}\label{MorphismeLambda0}
J\times X\cup \{\epsilon\}\times Y\to J\times Y
\end{equation}
est de la forme
\begin{equation*}
\partial(\Delta^{\varphi})\boxtimes \fil{W}\cup \Delta^{\varphi}\boxtimes \fil{Z}\to \Delta^{\varphi}\boxtimes\fil{W}
\end{equation*}
où $\fil{Z}=\Delta^1\otimes\partial(\Delta^{\psi})\cup\{\epsilon\}\otimes\Delta^{\psi}$ et $\fil{W}=\Delta^1\otimes\Delta^{\psi}$. En particulier, $\fil{Z}\to\fil{W}$ est une extension anodine de $\sS_P$, et le morphisme \ref{MorphismeLambda0} est dans $\partial(\Delta(P))\boxtimes\Lambda$. Par ailleurs, par définition de $S$, on a l'inclusion
\begin{equation*}
S\subseteq \Delta(P)\boxtimes\Lambda
\end{equation*}
On en déduit que 
\begin{equation*}
\Lambda^0_J(S,\mathcal{M})\subseteq \Delta(P)\boxtimes\Lambda\cup\partial(\Delta(P))\boxtimes\Lambda.
\end{equation*}
On remarque que $\Lambda_J(S,\mathcal{M})$ est le plus petit ensemble de morphisme contenant $\Lambda_J^0(S,\mathcal{M})$
et stable par l'opération $\Lambda_J$, où $\Lambda_J$ est l'opération définie par 
\begin{equation*}
\Lambda_J(\mathcal{C})=\{J\times X\cup \partial(J)\times Y\to J\times Y\ |\ X\to Y\in \mathcal{C}\}
\end{equation*}
pour $\mathcal{C}$ une classe de morphismes. Il suffit donc de montrer que $\Delta(P)\boxtimes\Lambda\cup\partial(\Delta(P))\boxtimes\Lambda$ est stable par l'opération $\Lambda_J$. Soient $\fil{Z}\to\fil{W}\in \Lambda$ et $\Delta^{\varphi}\in \Delta(P)$. On calcule
\begin{align*}
&J\times (\Delta^{\varphi}\boxtimes\fil{Z})\cup \partial(J)\times (\Delta^{\varphi}\boxtimes\fil{W})\\
&\simeq (F(\Delta^0)\boxtimes F(\Delta^1))\times(\Delta^{\varphi}\boxtimes\fil{Z})\cup (F(\Delta^0)\boxtimes F(\partial(\Delta^1)))\times(\Delta^{\varphi}\boxtimes\fil{W})\\
&\simeq \Delta^{\varphi}\boxtimes(\Delta^1\otimes\fil{Z})\cup \Delta^{\varphi}\boxtimes(\partial(\Delta^1)\otimes\fil{W})\\
&\simeq \Delta^{\varphi}\boxtimes(\Delta^1\otimes\fil{Z}\cup\partial(\Delta^1)\otimes\fil{W})
\end{align*}
Et l'inclusion 
\begin{equation*}
\Delta^1\otimes\fil{Z}\cup\partial(\Delta^1)\otimes\fil{W}\to\Delta^1\otimes \fil{W}
\end{equation*}
est une extension anodine de $\sS_P$ par l'axiome (An2). On en déduit que
\begin{equation*}
\Lambda_J(\Delta(P)\boxtimes\Lambda)\subseteq\Delta(P)\boxtimes\Lambda.
\end{equation*}
De même, on calcule
\begin{align*}
&J\times(\partial(\Delta^{\varphi})\boxtimes\fil{W}\cup\Delta^{\varphi}\boxtimes\fil{Z})\cup \partial(J)\times(\Delta^{\varphi}\boxtimes\fil{W})\\
&\simeq \partial(\Delta^{\varphi})\boxtimes(\Delta^1\otimes\fil{W})\cup\Delta^{\varphi}\boxtimes(\Delta^1\otimes\fil{Z})\cup \Delta^{\varphi}\boxtimes(\partial(\Delta^1)\otimes\fil{W}\\
&\simeq \partial(\Delta^{\varphi})\boxtimes(\Delta^1\otimes\fil{W})\cup \Delta^{\varphi}\boxtimes(\Delta^1\otimes\fil{Z}\cup\partial(\Delta^1)\otimes\fil{W}).
\end{align*}
et on en déduit que 
\begin{equation*}
\Lambda_J(\partial(\Delta(P))\boxtimes\Lambda)\subseteq\partial(\Delta(P))\boxtimes\Lambda.
\end{equation*}

\end{proof}

\begin{lemme}\label{diagPreserveEquivalencesFaibles}
Le foncteur $\diag\colon \bisS_P\to\sS_P$ est un foncteur de Quillen à gauche. En particulier, il préserve les équivalences faibles.
\end{lemme}

\begin{proof}
Par construction, le foncteur $\diag$ préserve les monomorphismes et les colimites. Par le lemme \cite[2.4.40]{Cisinski2}, il suffit donc de montrer que pour tout $f\colon X\to Y\in \Lambda_J(S,\mathcal{M})$, $\diag(f)$ est une cofibration triviale de $\sS_P$. Par le lemme \ref{CalculExtensionsAnodinesBiEnsemblesSimpliciauxFiltres}, il suffit de montrer ce résultat pour $f\in \Delta(P)\boxtimes\Lambda\cup\partial(\Delta(P))\boxtimes\Lambda$.
Soit $\Delta^{\varphi}\boxtimes\fil{Z}\to\Delta^{\varphi}\boxtimes\fil{W}\in \Delta(P)\boxtimes\Lambda$. On calcule en utilisant la remarque \ref{DiagonaleBoxProduct} :
\begin{align*}
&\diag(\Delta^{\varphi}\boxtimes\fil{Z}\to\Delta^{\varphi}\boxtimes\fil{W})\\
&\simeq \Delta^{\varphi}\times\fil{Z}\to\Delta^{\varphi}\times\fil{W}
\end{align*}
Par le lemme \ref{LemmeAn2} appliqué à $\emptyset\to \Delta^{\varphi}$ et $\fil{Z}\to\fil{W}$, on en déduit que ce morphisme est une extension anodine de $\sS_P$. De même, pour 
\begin{equation*}
\Delta^{\varphi}\boxtimes\fil{Z}\cup\partial(\Delta^{\varphi})\boxtimes\fil{W}\to\Delta^{\varphi}\boxtimes\fil{W}\in \partial(\Delta(P))\boxtimes\Lambda,
\end{equation*}
on calcule :
\begin{align*}
&\diag(\Delta^{\varphi}\boxtimes\fil{Z}\cup\partial(\Delta^{\varphi})\boxtimes\fil{W}\to\Delta^{\varphi}\boxtimes\fil{W})\\
&\simeq \Delta^{\varphi}\times\fil{Z}\cup\partial(\Delta^{\varphi})\times\fil{W}\to\Delta^{\varphi}\times\fil{W}
\end{align*}
Ce morphisme est une extension anodine de $\sS_P$ par le lemme \ref{LemmeAn2}. Ceci prouve le résultat voulu car toute extension anodine est une cofibration triviale.
\end{proof}

\begin{defin}
Soit $f\colon X\to Y$ morphisme de $\bisS_P$. Le morphisme $f$ est une équivalence faible niveau par niveau si tous les morphismes induits de la forme \ref{MorphismeBiEnsemblesSimpliciauxFiltres1} sont des équivalences faibles.
\end{defin}
On remarque que les équivalences faibles niveaux par niveaux vérifient l'axiome de deux sur trois.

\begin{lemme}\label{ExtensionsAnodinesBiSimplicialeEquivalenceFaibleNiveau}
Toute extension anodine de $\bisS_P$ est une équivalence faible niveau par niveau.
\end{lemme}

\begin{proof}
Pour tout $\Delta^{\varphi}\in \Delta(P)$, le foncteur $(-)^{\Delta^{\varphi}}\colon \bisS_P\to\sS_P$ préserve les sommes amalgamées. On en déduit que la classe des morphismes $f\colon X\to Y\in \bisS_P$ tels que $f^{\Delta^\varphi}$ est une extension anodine de $\sS_P$ est une classe saturée. Comme la classe des extensions anodines de $\bisS_P$ est la plus petite classe saturée contenant $\Lambda_J(S,\mathcal{M})$, il suffit de montrer que pour tout $f\in \Lambda_J(S,\mathcal{M})$, $f^{\Delta^{\varphi}}$ est une extension anodine de $\sS_P$. Par le lemme \ref{CalculExtensionsAnodinesBiEnsemblesSimpliciauxFiltres}, il suffit de montrer ce résultat pour $f$ de la forme
\begin{equation*}
\Delta^{\psi}\boxtimes\fil{Z}\to\Delta^{\psi}\boxtimes\fil{W}
\end{equation*}
ou de la forme
\begin{equation*}
\Delta^{\psi}\boxtimes\fil{Z}\cup\partial(\Delta^{\psi})\boxtimes\fil{W}\to\Delta^{\psi}\boxtimes\fil{W}
\end{equation*}
où $\Delta^{\psi}\in \Delta(P)$ et $\fil{Z}\to\fil{W}$ est une extension anodine de $\sS_P$. En utilisant le lemme \ref{BoxProductCoProduct}, on a
\begin{align*}
&(\Delta^{\psi}\boxtimes\fil{Z})^{\Delta^{\varphi}}\to(\Delta^{\psi}\boxtimes\fil{W})^{\Delta^{\varphi}}\\
&\simeq \coprod_{\Delta^{\varphi}\to\Delta^{\psi}}\fil{Z}\to\coprod_{\Delta^{\varphi}\to\Delta^{\psi}}\fil{W}.
\end{align*}
Comme les extensions anodines sont stables par unions disjointes, on obtient bien une extension anodine.
On a d'autre part
\begin{align*}
&(\Delta^{\psi}\boxtimes\fil{Z}\cup\partial(\Delta^{\psi})\boxtimes\fil{W})^{\Delta^{\varphi}}\to(\Delta^{\psi}\boxtimes\fil{W})^{\Delta^{\varphi}}\\
&\simeq \coprod_{\Delta^{\varphi}\to\Delta^{\psi}}\fil{Z}\cup\coprod_{\Delta^{\varphi}\to\partial(\Delta^{\psi})}\fil{W}\to \coprod_{\Delta^{\varphi}\to\Delta^{\psi}}\fil{W}.
\end{align*}
En notant $S=\Hom(\Delta^{\varphi},\partial(\Delta^{\psi}))$ , $T=\Hom(\Delta^{\varphi},\Delta^{\psi})$, et $T\setminus S$ pour l'ensemble des éléments de $T$ ne se factorisant pas par un élément de $S$, on peut réécrire le morphisme précédent sous la forme :
\begin{equation*}\coprod_{T\setminus S}\fil{Z}\coprod\coprod_{S}\fil{W}\to \coprod_{T}\fil{W}.
\end{equation*}
On obtient une union disjointe de morphismes $\fil{Z}\to\fil{W}$ et de $\Id_W$. En particulier, on obtient une extension anodine de $\sS_P$. Finalement, si $f$ est une extension anodine de $\bisS_P$, $f^{\Delta^{\varphi}}$ est une extension anodine de $\sS_P$ pour tout $\Delta^{\varphi}$. En particulier, c'est une équivalence faible, et donc $f$ est une équivalence faible niveau par niveau.
\end{proof}

\begin{lemme}\label{EquivalenceFaibleBiEnsemblesSimpliciauxFiltres}
Toute équivalence faible niveau par niveau est une équivalence faible de $\bisS_P$
\end{lemme}

\begin{proof}
Soit $f\colon X\to Y$ un morphisme de $\bisS_P$. Par l'argument du petit objet appliqué à $Y\to F(\Delta^0)\boxtimes F(\Delta^0)$, on peut construire une extension anodine de $\bisS_P$, $j\colon Y\to V$ avec $V$ un objet fibrant de $\bisS_P$. De même, en appliquant une seconde fois l'argument du petit objet à la composition $X\xrightarrow{f}Y\xrightarrow{j}V$, on obtient une factorisation 
\begin{equation*}
\begin{tikzcd}
X
\arrow{r}{i}
\arrow{d}{f}
&U
\arrow{d}{g}
\\
Y
\arrow{r}{j}
&V
\end{tikzcd}
\end{equation*}
avec $i$ une extension anodine et $g$ une fibration naïve entre objets fibrants. Par le lemme \ref{ExtensionsAnodinesBiSimplicialeEquivalenceFaibleNiveau}, $i$ et $j$ sont des équivalences faibles niveau par niveau. Par deux sur trois, si $f$ est une équivalence faible niveau par niveau, alors $g$ l'est aussi. D'autre part, $g$ induit une transformation naturelle entre les foncteurs
\begin{equation*}
U^{(-)}\colon \sS_P\to\sS_P
\end{equation*}
et 
\begin{equation*}
V^{(-)}\colon \sS_P\to\sS_P
\end{equation*}
Par le corollaire \ref{ObjetsFibrantsBisSetP}, chacun de ces foncteurs envoie les cofibrations sur des fibrations.
On peut donc appliquer le lemme \ref{EquivalenceFaibleFoncteurEilenbergZilber} avec $\mathcal{C}=\sS_P^{\op}$ pour en déduire que pour tout ensemble simplicial filtré $\fil{Z}$, on a une équivalence faible
\begin{equation*}
U^{\fil{Z}}\to V^{\fil{Z}}.
\end{equation*}
En particulier on considère le diagramme commutatif suivant
\begin{equation*}
\begin{tikzcd}
U^{\Delta^{\varphi}}
\arrow{dr}{h_1}
\arrow[bend left = 16]{drr}{g^{\Delta^{\varphi}}}
\arrow[bend right = 16,swap]{ddr}
\\
&U^{\partial(\Delta^{\varphi})}\times_{V^{\partial(\Delta^{\varphi})}}V^{\Delta^{\varphi}}
\arrow[swap]{r}{h_2}
\arrow{d}
&V^{\Delta^{\varphi}}
\arrow{d}{\pr_V}
\\
&U^{\partial(\Delta^{\varphi})}
\arrow[swap]{r}{g^{\partial(\Delta^{\varphi})}}
&V^{\partial(\Delta^{\varphi})}
\end{tikzcd}
\end{equation*}
où $\Delta^{\varphi}\in \Delta(P)$, et $ \pr_V$ est le morphisme induit par l'inclusion $\partial(\Delta^{\varphi})\to\Delta^{\varphi}$.
On a montré que $g^{\partial(\Delta^{\varphi})}$ et $g^{\Delta^{\varphi}}$ sont des équivalences faibles.
Par ailleurs, comme $U$ et $V$ sont fibrants, $\pr_V$ est une fibration de $\sS_P$ et, par le corollaire \ref{ObjetsFibrantsBisSetP}, $U^{\partial(\Delta^{\varphi})}, V^{\Delta^{\varphi}}$ et $V^{\partial(\Delta^{\varphi})}$ sont des objets fibrants de $\sS_P$. On en déduit que $h_2$ est une équivalence faible de $\sS_P$. 
Par deux sur trois, on a donc que $h_1$ est une équivalence faible de $sS_P$. 
Comme on a supposé que $g$ était une fibration naïve de $\bisS_P$, on sait que $h_1$ est une fibration de $\sS_P$. C'est donc une fibration triviale de $\sS_P$.
Par le lemme \ref{FibrationsTrivialesBiEnsemblesSimpliciauxFiltres}, on en déduit que $g$ est une fibration triviale de $\bisS_P$.
Par deux sur trois, $f$ est donc une équivalence faible de $\bisS_P$.
\end{proof}

\begin{proof}[Démonstration du Théorème \ref{TheoremeDiagonal}]
Soit $f\colon X\to Y\in \bisS_P$ tel que pour tout $\Delta^{\varphi}$ 
\begin{equation*}
f^{\Delta^{\varphi}}\colon X^{\Delta^{\varphi}}\to Y^{\Delta^{\varphi}}
\end{equation*}
est une équivalence faible de $\sS_P$. Alors, par le lemme \ref{EquivalenceFaibleBiEnsemblesSimpliciauxFiltres}, $f$ est une équivalence faible de $\bisS_P$. Mais alors, comme tous les objets de $\bisS_P$ sont cofibrants, par le lemme \ref{diagPreserveEquivalencesFaibles}, $\diag(f)$ est une équivalence faible de $\sS_P$.
\end{proof}

\section{Le foncteur $\Ex_P$}
Soient $X,Y$ deux objets de $\widehat{A}$. On définit le préfaisceau sur $A$ 
\begin{align*}
\HOM(X,Y)\colon A^{\op}&\to \Set\\
a&\mapsto \Hom(X\times \widetilde{a},Y)
\end{align*}
où $\widetilde{a}$ est le préfaisceau représenté par $a$.
Ceci définit un bi-foncteur
\begin{equation*}
\HOM\colon \widehat{A}^{\op}\times\widehat{A}\to\widehat{A}.
\end{equation*}
De plus, par construction, le foncteur $\HOM(Z,-)$ est un adjoint à droite du foncteur $-\times Z$ pour tout préfaisceau $Z$ de $\widehat{A}$.

On suppose à partir de maintenant qu'il existe un préfaisceau $I$ de $\widehat{A}$ tel que pour tout préfaisceau $X$ de $\widehat{A}$, le cylindre sur $X$ est donné par $I\times X$. Alors, on a le résultat suivant.

\begin{prop}\label{HomotopiesHomInterne}
Soient $f,g\colon X\to Y$ deux morphismes de préfaisceaux tels qu'il existe une homotopie $H\colon I\times X\to Y$ entre $f$ et $g$. Alors, pour tout préfaisceau $Z$, les applications induites par $f$ et $g$
\begin{equation*}
f_{*},g_{*}\colon \HOM(Y,Z)\to\HOM(X,Z)
\end{equation*}
sont homotopes.
\end{prop}

\begin{proof}
On applique le foncteur $\HOM(-,Z)$ à l'application $H\colon I\times X\to Y$. On obtient un morphisme $H_*\colon \HOM(Y,Z)\to\HOM(I\times X,Z)$. Par adjonction on a $\HOM(I\times X,Z)\simeq \HOM(I,\HOM(X,Z))$. Puis, en prenant l'image de $H_*$ par l'adjonction, on obtient 
\begin{equation*}
\widehat{H_*}\colon I\times \HOM(Y,Z)\to\HOM(X,Z)
\end{equation*}
qui donne l'homotopie voulue.
\end{proof}

Dans la suite de cette section, pour alléger les notations, on écrit $X$ pour désigner l'ensemble simplicial filtré $\fil{X}$, et la filtration $\varphi_X$ est laissée implicite. Alternativement, on peut considérer $X$ comme un foncteur $\Delta(P)^{\op}\to \Set$. De plus, les classes d'équivalences faibles, de cofibrations et de fibrations considérées sont celles du théorème \ref{ExistenceCMFCisinski} et ne coïncident pas avec celles
sur les ensembles simpliciaux sous-jacents.

\begin{theo}\label{TheoremeExXFaiblementEquivalentaX}
 Soit $X$ un ensemble simplicial filtré. Alors, le morphisme $\beta_X\colon X\to \Ex_P(X)$ est une cofibration triviale de $\sS_P$.
\end{theo}

\begin{proof}
Si $X$ est fibrant, c'est une conséquence du Lemme \ref{ExExtensionAnodine} appliqué à la fibration $X\to N(P)$. Si $X$ n'est pas fibrant, on sait par construction que le morphisme $X\to \Ex_P(X)$ est une cofibration et il suffit de montrer que c'est une équivalence faible. Pour ce faire, on factorise $X\to N(P)$ en une cofibration triviale suivie d'une fibration $X\xrightarrow{i} Y\to N(P)$. En particulier, on obtient un ensemble simplicial filtré fibrant et faiblement équivalent à $X$. Considérons maintenant le diagramme commutatif suivant
\begin{equation*}
\begin{tikzcd}
X
\arrow{r}{i}
\arrow{d}{\beta_X}
&Y
\arrow{d}{\beta_Y}
\\
\Ex_P(X)
\arrow{r}{\Ex_P(i)}
&\Ex_P(Y)
\end{tikzcd}
\end{equation*}
Par hypothèse $i$ est une équivalence faible, et $\beta_Y$ est une équivalence faible car $Y$ est fibrant. Par deux sur trois, il suffit de montrer que $\Ex_P(i)$ est une équivalence faible. En particulier, il suffit de montrer que le foncteur $\Ex_P$ préserve les équivalences faibles.
Pour ce faire, on considère les préfaisceaux $C$, $D$ et $E$ sur $\Delta(P)\times\Delta(P)$ définis comme suit 

\begin{equation*}
C(X)(\Delta^{\varphi},\Delta^{\psi})=\Hom(\Delta^{\varphi}\times\overline{\Delta^{\psi}},X)
\end{equation*}
\begin{equation*}
D(X)(\Delta^{\varphi},\Delta^{\psi})=\Hom(\Delta^{\varphi}\times\sd_P(\Delta^{\psi}),X).
\end{equation*}
et
\begin{equation*}
E(X)(\Delta^{\varphi},\Delta^{\psi})=\Hom(\overline{\Delta^{\varphi}}\times\sd_P(\Delta^{\psi}),X)
\end{equation*}
où $\overline{\Delta^{\varphi}}\in N(P)$ désigne l'unique simplexe non dégénéré dont $\Delta^{\varphi}$ est une dégénérescence, et où les produits sont fibrés au dessus de $N(P)$. Les morphismes canoniques
\begin{equation*}
\Delta^{\varphi}\times\overline{\Delta^{\psi}}\xleftarrow{}\Delta^{\varphi}\times\sd_P(\Delta^{\psi})\to\overline{\Delta^{\varphi}}\times\sd_P(\Delta^{\psi})
\end{equation*}
induisent les morphismes dans $\widehat{A\times A}$
\begin{equation*}
C(X)\to D(X)\xleftarrow{} E(X)
\end{equation*}
Soit $\Delta^{\psi}\in \Delta(P)$, on calcule 
\begin{align*}
C(X)_{\Delta^{\psi}}(\Delta^{\varphi})&=\Hom(\Delta^{\varphi}\times\overline{\Delta^{\psi}},X)\\
&\simeq \Hom(\Delta^{\varphi},\HOM(\overline{\Delta^{\psi}},X))\\
&\simeq \HOM(\overline{\Delta^{\psi}},X)(\Delta^{\varphi})
\end{align*}
D'autre part, on a 
\begin{align*}
D(X)_{\Delta^{\psi}}(\Delta^{\varphi})&=\Hom(\Delta^{\varphi}\times\sd_P(\Delta^{\psi}),X)\\
&\simeq \Hom(\Delta^{\varphi},\HOM(\sd_P(\Delta^{\psi}),X))\\
&\simeq\HOM(\sd_P(\Delta^{\psi}),X)(\Delta^{\varphi})
\end{align*}
En particulier, le morphisme $C(X)\to D(X)$ est de la forme :
\begin{equation*}
C(X)_{\Delta^{\psi}}\simeq\HOM(\overline{\Delta^{\psi}},X)\to\HOM(\sd_P(\Delta^{\psi}),X)\simeq D(X)_{\Delta^{\psi}}
\end{equation*}
Par le lemme \ref{SDSimplexeEquivalenceFaibleAbsolue}, le morphisme $\sd_P(\Delta^{\psi})\to\overline{\Delta^{\psi}}$ est une équivalence d'homotopie filtrée. On en déduit par le lemme \ref{HomotopiesHomInterne} que les morphismes $C(X)_{\Delta^{\psi}}\to D(X)_{\Delta^{\psi}}$ sont des équivalences faibles. En appliquant la remarque \ref{TheoremeDiagonalRenverse}, on obtient que le morphisme $\diag(C(X))\to\diag(D(X))$ est une équivalence faible. On calcule
\begin{align*}
\diag(C(X))(\Delta^{\varphi})&=\Hom(\Delta^{\varphi}\times\overline{\Delta^{\varphi}},X)\\
&\simeq \Hom(\Delta^{\varphi},X)\\
&\simeq X(\Delta^{\varphi})
\end{align*}
Et finalement, le morphisme $X\to \diag(D(X))$ est une équivalence faible.
De même, pour $\Delta^{\varphi}\in \Delta(P)$, on calcule :
\begin{align*}
E(X)^{\Delta^{\varphi}}(\Delta^{\psi})&=\Hom(\overline{\Delta^{\varphi}}\times\sd_P(\Delta^{\psi}),X)\\
&\simeq \Hom(\sd_P(\Delta^{\psi}),\HOM(\overline{\Delta^{\varphi}},X))\\
&\simeq \Hom(\Delta^{\psi},\Ex_P(\HOM(\overline{\Delta^{\varphi}},X)))\\
&\simeq \Ex_P(\HOM(\overline{\Delta^{\varphi}},X))(\Delta^{\psi})
\end{align*}
D'autre part, on a 
\begin{align*}
D(X)^{\Delta^{\varphi}}(\Delta^{\psi})&=\Hom(\Delta^{\varphi}\times\sd_P(\Delta^{\psi}),X)\\
&\simeq \Hom(\sd_P(\Delta^{\psi}),\HOM(\Delta^{\varphi},X))\\
&\simeq \Hom(\Delta^{\psi},\Ex_P(\HOM(\Delta^{\varphi},X)))\\
&\simeq \Ex_P(\HOM(\Delta^{\varphi},X))(\Delta^{\psi})
\end{align*}
On en déduit que le morphisme $E(X)^{\Delta^{\varphi}}\to D(X)^{\Delta^{\varphi}}$ est de la forme
\begin{equation*}
E(X)^{\Delta^{\varphi}}\simeq \Ex_P(\HOM(\overline{\Delta^{\varphi}},X))\to\Ex_P(\HOM(\Delta^{\varphi},X))\simeq D(X)^{\Delta^{\varphi}}
\end{equation*}
On souhaite montrer que c'est une équivalence faible.
On sait, par la preuve du lemme \ref{SDSimplexeEquivalenceFaibleAbsolue} que $\Delta^{\varphi}\to\overline{\Delta^{\varphi}}$ est une équivalence d'homotopie, on en déduit par le lemme \ref{HomotopiesHomInterne} que $\HOM(\overline{\Delta^{\varphi}},X)\to \HOM(\Delta^{\varphi},X)$ est une équivalence faible. Il suffit donc de montrer que $\Ex_P$ préserve les homotopies. Soit $H\colon \Delta^1\otimes Z\to W$ une homotopie entre $f$ et $g$. Comme $\Ex_P$ préserve les limites, et que $\Delta^1\otimes Z\simeq F(\Delta^1)\times Z$, la composition
\begin{equation*}
\Delta^1\otimes\Ex_P(Z)\xrightarrow{\beta_{F(\Delta^1)}\times \Id}\Ex_P(F(\Delta^1))\times Z\simeq \Ex_P(\Delta^1\otimes Z)\xrightarrow{\Ex_P(H)} \Ex_P(W)
\end{equation*}
donne une homotopie entre $\Ex_P(f)$ et $\Ex_P(g)$. En particulier, on en déduit que le morphisme $E(X)^{\Delta^{\varphi}}\to D(X)^{\Delta^{\varphi}}$ est une équivalence faible pour tout $\Delta^{\varphi}$. On en déduit que $\diag(E(X))\to\diag(D(X))$ est une équivalence faible.
On calcule
\begin{align*}
\diag(E(X))(\Delta^{\psi})&=\Hom(\overline{\Delta^{\psi}}\times\sd_P(\Delta^{\psi}),X)\\
&\simeq\Hom(\sd_P(\Delta^{\psi}),X)\\
&\simeq \Ex_P(X)(\Delta^{\psi})
\end{align*}
On en déduit que $\Ex_P(X)\simeq \diag(E(X))\to\diag(D(X))$ est une équivalence faible. Finalement, si $f\colon X\to Y$ est un morphisme d'ensemble simpliciaux filtrés, on a le diagramme commutatif suivant 
\begin{equation*}
\begin{tikzcd}
X\arrow{r}
\arrow{d}{f}
&\diag(D(X))
\arrow{d}{\diag(D(f))}
&\Ex_P(X)
\arrow{d}{\Ex_P(f)}
\arrow{l}
\\
Y
\arrow{r}
&\diag(D(Y))
&Ex_P(Y)
\arrow{l}
\end{tikzcd}
\end{equation*}
où les flèches horizontales sont des équivalences faibles. Par deux sur trois, on a équivalence entre les assertions
\begin{itemize}
\item $f$ est une équivalence faible,
\item $\diag(D(f))$ est une équivalence faible, 
\item $\Ex_P(f)$ est une équivalence faible.
\end{itemize}
On en déduit que pour tout $X$, $\beta_X\colon X\to \Ex_P(X)$ est une cofibration triviale.
\end{proof}

\begin{corollaire}\label{CorollaireBetaInfiniEquivalenceFaible}
Soit $X$ un ensemble simplicial filtré, $\beta^{\infty}_X \colon X\to \Exi_P(X)$ est une cofibration triviale.
\end{corollaire}

\begin{proof}
Par construction, $\beta_X^{\infty}$ est la composition transfinie des morphismes $\beta_{\Ex^n(X)}$ qui sont tous des cofibrations triviales par le théorème précédent. Comme la classe des cofibrations triviales est saturée, on en déduit le résultat voulu.
\end{proof}

\bibliographystyle{alpha}
\bibliography{biblio}

\newpage

\begin{center}\bf Stratified homotopy theory\end{center}
\textbf{Keywords} : algebraic topology, homotopical algebra, singular spaces, stratified spaces, stratifications, simplicial sets, homotopy groups.
\\~\\

A stratified space is a topological space cut into strata. In the case of pseudo-manifolds, these strata are manifolds, of varying dimension, corresponding to a partition of the space according to the type of singularity. Examples of such spaces appear everywhere in topology and geometry, as the natural generalization to manifolds.

The study of stratified spaces - and of stratifications - involves invariants such as intersection cohomology or the category of exit paths. Those are not homotopy invariants; in general, they depend on the stratification. Nevertheless, they are invariant by stratified homotopy, a stronger notion of homotopy that preserves the stratification.
\\~\\
In this thesis, we study the homotopy theory of stratified spaces with respect to stratified homotopy. To do so, we construct model categories for stratified spaces and we introduce new invariants to characterize them, the filtered homotopy groups.

A stratified space can be seen as a topological space $X$ together 
with a continuous map to a partially ordered set of strata, $X\to P$. We begin our study by restricting ourselves to the case where the poset of strata is fixed, this is the \textbf{filtered} context. Then we define the model category of filtered simplicial sets - the simplicial analogue to the category of filtered spaces. We show that it admits a description "à la Kan" and we characterize its weak-equivalences using a new invariant : the filtered homotopy groups. Then, using an adjunction between the categories of filtered simplicial sets and of filtered spaces, we prove a filtered version of Whitehead theorem.

We then construct a model category of filtered spaces. We show that it can be described similarly to the classical model category of topological spaces : The weak-equivalences are the morphisms that induce isomorphisms on all filtered homotopy groups, and the fibration satisfy a filtered version of Serre's lifting conditions. Lastly, we show that it is Quillen-equivalent to a category of diagrams of simplicial sets.

We then work toward a comparison between the model categories of filtered simplicial sets and of filtered spaces. They are connected by a Quillen-adjunction, similar to the classical Kan-Quillen adjunction. We conjecture that it is in fact a Quillen equivalence.

Lastly, working with the notion of Quillen bifibration, we show that there are model categories of stratified spaces and of stratified simplicial sets. We show that the two are related by an adjunction that preserve weak equivalences.
\thispagestyle{empty}
\newpage

\begin{center} {\bf Étude homotopique des espaces stratifiés}\end{center}
{\textbf{Mots-clés} : \textit{topologie algébrique, algèbre homotopique, espaces singuliers, espaces stratifiés, stratifications, ensembles simpliciaux, groupes d'homotopie.}
\\~\\

Un espace stratifié est un espace topologique découpé en strates. Dans les cas des pseudo-variétés, ces strates sont des variétés, de dimension variable, et correspondent à une partition de l'espace selon le type des singularités. De tels espaces apparaissent très largement en topologie et en géométrie, où ils constituent la généralisation naturelle de la notion de variété lisse.

L'étude des espaces stratifiés - et plus généralement l'étude des stratifications - passe par le calcul de plusieurs invariants tels que la cohomologie d'intersection, ou la catégorie des chemins sortants. Ceux-ci ne sont pas invariants par homotopies, et dépendent en général de la stratification. Néanmoins, ils sont invariants par homotopies stratifiées - une notion d'homotopie plus forte qui tient compte de la stratification.
\\~\\
Dans cette thèse, on étudie la théorie homotopique des espaces stratifiés relative à cette notion d'homotopie stratifiée. Ceci passe par la construction de catégories modèles pour les espaces stratifiés, et par leur caractérisation par des invariants d'homotopie stratifiée, les groupes d'homotopie filtrés.

Un espace stratifié est un espace topologique $X$ muni d'une application continue vers un ensemble ordonné de strates $X\to P$. On commence notre étude par travailler dans le cas où l'ensemble ordonné $P$ est fixé, c'est le contexte \textbf{filtré}. Dans ce contexte on définit la catégorie modèle des ensembles simpliciaux filtrés, l'analogue simplicial à la catégorie des espaces filtrés. On montre qu'elle admet une description "à la Kan" et on y caractérise les équivalences faibles via un nouvel invariant : les groupes d'homotopie filtrés. En exploitant une adjonction entre la catégorie des ensembles simpliciaux filtrés et celle des espaces filtrés, on montre une version filtrée du théorème de Whitehead.

Dans un second temps, on construit une catégorie modèle pour les espaces filtrés, admettant une description similaire à la catégorie modèle des espaces topologiques. Les équivalences faibles y sont les morphismes induisant des isomorphismes sur tous les groupes d'homotopie filtrés et les fibrations vérifient une version filtrée de la condition de Serre. On montre ensuite que celle-ci est Quillen-équivalente à une catégorie de diagrammes d'ensembles simpliciaux.

On entreprend ensuite de comparer les catégories modèles des ensembles simpliciaux filtrés et des espaces filtrés. Celles-ci sont liées par une adjonction de Quillen dont on conjecture qu'il s'agit d'une équivalence de Quillen.

Finalement, en exploitant la notion de bifibration de Quillen, on construit des structures de modèle sur les catégories des espaces stratifiés et des ensembles simpliciaux stratifiés. On montre que ces catégories sont liées par une adjonction qui préserve les équivalences faibles. 
\thispagestyle{empty}
\end{document}